%% file: article.tex
\documentclass[onefignum,onetabnum]{siamart190516}

\input{shared}
\input{header}

\ifpdf
\hypersetup{
  pdftitle={Rigidity for pseudo-range based frameworks},
  pdfauthor={C. Cros, P-O. Amblard, and C. Prieur}
}
\fi

\begin{document}

\maketitle

% REQUIRED
\begin{abstract}
	This paper introduces new structures called conic frameworks and their rigidity. They are composed by agents and a set of directed constraints between pairs of agents. When the structure cannot be flexed while preserving the constraints, it is said to be rigid. If only smooth deformations are considered a sufficient condition for rigidity is called infinitesimal rigidity. In conic frameworks, each agent $u$ has a spatial position $x_u$ and a clock offset represented by a bias $\beta_u$. If the constraint from Agent $u$ to Agent $w$ is in the framework, the pseudo-range from $u$ to $w$, defined as $\norm{x_u - x_w} + \beta_w - \beta_u$, is set. Pseudo-ranges appear when measuring inter-agent distances using a Time-of-Arrival method. This paper completely characterizes infinitesimal rigidity of conic frameworks whose agents are in general position. Two characterizations are introduced: one for unidimensional frameworks, the other for multidimensional frameworks. They both rely on the graph of constraints and use a decoupling between space and bias variables. In multidimensional cases, this new conic paradigm sharply reduces the minimal number of constraints required to maintain a formation with respect to classical Two-Way Ranging methods. 

\end{abstract}

% REQUIRED

% REQUIRED

\section{Introduction}

Consider a group of $n$ agents whose geometric formation has to be maintained, a fleet of Unmanned Aerial Vehicles (UAV) or autonomous cars for example. Maintaining the formation means preserving the relative positions between the agents as the group might move as a rigid entity from one point to another. A natural way to achieve this result is \begin{enumerate*}[label =(\roman*)]
	\item\label{it: choice distance} to select some distances between pairs of agents and,
	\item\label{it: command} to force the agents to preserve these distances using some command \cite{anderson_2008_uav, gazi_2003_stability}.
\end{enumerate*} The choice of the distances in Phase \ref{it: choice distance} must ensure that the formation cannot flex. For example, to maintain a square formation between four agents in the plane, constraining the sides and one diagonal of the square is sufficient. Another solution could be to impose directly the positions of the agents, but this would need an external positioning system, such as the satellite systems GPS or Galileo. The first option is considered here. It is free from external systems but supposes that distances between pairs of agents can be constrained and therefore measured. An efficient way to measure the distance between two agents is to send a signal from one to the other, measure the delay between the emission and the reception, and multiply that delay by the speed of the signal. If the agents' clocks are synchronized, this procedure provides the distance. In general however, this procedure only gives a \emph{pseudo-range} which is the sum of the actual distance between the agents and a bias reflecting the lack of synchronization. This bias is the difference between the clock offsets of the agents (to some absolute reference) multiplied by the speed of the signal. Concretely, consider a pair of agents $(e,r)$ whose positions are $(x_e, x_r)$ and clock offsets are $(\tau_e,\tau_r)$. Denote $\beta_z = c\tau_z$ for $z\in\set{e,r}$ the clock offset premultiplied by the speed of signal $c$. The pseudo-range from $e$ to $r$ is then defined by the asymmetrical map:
\begin{equation}\label{eq: Pseudo-range function}
	r\left(\begin{pmatrix}
		x_e \\ \beta_e
	\end{pmatrix}, \begin{pmatrix}
	x_r \\ \beta_r
\end{pmatrix}\right) = \norm{x_e - x_r} + \beta_r - \beta_e
\end{equation}
Note in particular that there is no absolute value on the bias: it is positive if the emitter's clock is ahead of the receiver's and negative otherwise. Pseudo-ranges were first introduced in the context of hyperbolic positioning systems, see \eg{} \cite[Chap.~2]{kaplan_2017_understanding} for an introduction with satellite systems.

To retrieve the distance between $x_e$ and $x_r$, the pseudo-ranges in both directions are usually summed to compensate for the biases. This technique is called symmetrical double-sided two-way ranging (SDS TWR) \cite{standard_802_15_4a_2007} and is quite used with Ultra WideBand (UWB) technology \cite{sahinoglu_2008_ultra}. To maintain a formation using SDS TWR, each constrained distance is measured using two pseudo-range measurements. This paper introduces a more efficient alternative method by constraining pseudo-ranges instead of distances. In this new scheme, to maintain the formation, in Phase \ref{it: choice distance} pseudo-ranges are chosen instead of distances and the control of Phase \ref{it: command} is adapted to preserve theses pseudo-ranges. This paper only focuses on the selection of the pseudo-ranges in Phase \ref{it: choice distance} and does not deal with the control.

The underlying questions are how many and which pseudo-ranges should be constrained to maintain the formation? These problems are closely associated with the notion of \emph{rigidity} and more precisely of \emph{infinitesimal rigidity}. This notion has been well-studied when considering distances instead of pseudo-ranges \cite{gluck_1975_almost, connelly_1982_rigidity, hendrickson_1992_conditions}.

In particular, the minimal number of distances required to preserve a formation of $n$ agents in $\R^d$ is known and denoted as $S_e(n,d)$ \cite{asimow_rigidity_1978}. Consequently, preserving the same formation using SDS TWR requires at least $2S_e(n,d)$ pseudo-ranges. In contrast, the new scheme requires only $S_e(n,d) + n - 1$ pseudo-ranges. As $S_e(n,d)$ is greater than $n-1$, the new scheme always reduces the minimal number of pseudo-ranges. Furthermore, $S_e(n,d) \sim_n dn$, therefore, when $n$ is large the minimal number of pseudo-range constraints is reduced by a fourth in $\R^2$ and a third in $\R^3$ compared to the SDS TWR method. This new scheme exploits the fact that constraining two symmetrical pseudo-ranges sets both the distance and the bias difference between the agents: the distance is the half of their sum and the bias difference the half of their difference. Therefore, with $2(n-1)$ well-chosen pseudo-range constraints, $n-1$ distances and the bias differences between any pair of agents can be set. Indeed, to set the $n(n-1)/2$ bias differences, it is sufficient to set $n-1$ among them, \eg{} setting $\beta_i - \beta_1$ for $i = 2, \dots, n$ is a solution. Then, $S_e(n,d) - (n-1)$ more distances must be set to preserve the formation. With the bias differences set, pseudo-range constraints are equivalent to distance constraints. Thus, in total $S_e(n,d) + n-1$ pseudo-range constraints are needed. In addition, we prove that the pseudo-ranges can also be chosen without any symmetrical pair.
Our main results state that when the agents are in general positions, \emph{infinitesimal rigidity} of the formation depends only on the underlying graph of pseudo-range constraints denoted $\Gamma$ (and on the ordering of the points in the unidimensional case). When constraining distances instead of pseudo-ranges, infinitesimal rigidity depends also only on the underlying graph of distance constraints \cite{asimow_rigidity_1978} denoted $G$. In both cases, the graph is said to be rigid if it generates infinitesimally rigid formations. We prove that a pseudo-range graph $\Gamma$ is rigid if and only if it is the \emph{union} of two \emph{independent} graphs $H$ and $G$ where $H$ is a connected graph and $G$ is a distance rigid graph. Intuitively, $H$ sets the bias differences while $G$ sets the distances between the agents. Furthermore, we prove that in multidimensional cases, the directions of the pseudo-ranges do not matter as long as both graphs do not constrain the same pseudo-range.

To prove these results, we have designed a new class of structures representing the geometry of the agents and the pseudo-range constraints. We call them \emph{conic frameworks}. They are an extension of the well-studied bar-and-join structures, that we call in the following \emph{Euclidean frameworks} to avoid any confusion. A \emph{flexing} of a framework (conic or Euclidean) is a \emph{motion} distorting the framework while preserving its constraints \cite{asimow_rigidity_1978}. A framework with no flexing is said to be \emph{rigid}: it cannot be distorted. This paper focuses on a weaker notion: \emph{infinitesimal rigidity}, it ensures that no smooth flexing exists. All these notions are properly introduced in \cref{sec: Background}. Then, \cref{sec: Rigidity of unidimensional frameworks} characterizes infinitesimal rigidity of unidimensional conic frameworks and \cref{sec: Rigidity of multidimensional frameworks} infinitesimal rigidity of multidimensional conic frameworks. Finally, \cref{sec: Discussions} discusses these two characterizations.

\section{Preliminaries}\label{sec: Background}

\subsection{Background}%\cite{cauchy_1813_polygones, euler_1862_opera}
\emph{Rigidity} was first introduced as a mechanical notion to study the stability of bar-and-join structures \cite{laman_1970_graphs, connelly_1982_rigidity}. These structures are composed by nodes linked by incompressible and inextendible bars. A bar-and-join structure is said to be flexible if it can be continuously bent. It is said to be rigid otherwise. For example, in the plane, a square is flexible as it can be turned continuously into a rhombus, whereas a triangle is rigid since the three bars impose the relative positions of the agents. Rigidity has been well-studied for decades \cite{asimow_rigidity_1978}. In modern literature, a bar-and-join structure is usually represented through the compact form of a Euclidean framework: the combination $(G, p)$ of a simple undirected graph $G = (V,E)$, whose vertex set is $V$ and edge set is $E$, and a configuration $p$. For a general background on graph definitions and properties (incidence matrix, connectivity, cycles, etc.), we refer to \cite{bollobas_1998_modern}. The vertex set $V$ of $G$ is associated with the agents while the edge set $E$ represents the distance constraints. The configuration $p$ is a map from $V$ to $\R^d$ associating to each agent $u$ its coordinates $x_u$. When the $dn$ coordinates of the agents are not root of any non-null polynomial with integer coefficients, the configuration and the framework are said to be generic. It is known that rigidity of an Euclidean framework is NP-Hard to prove in general \cite[Chap.~5]{abbott_2008_generalizations}. However, for generic Euclidean frameworks, rigidity is more tractable: it is known that it only depends on the graph \cite{asimow_rigidity_1979}. For a given graph $G$ and a given dimension $d$, either every generic $d$-dimensional framework $(G,p)$ is rigid or none is. Furthermore, complete characterizations of generic rigidity based on the graph exist but only for dimensions $1$ and $2$ \cite{laman_1970_graphs}; until now the question remains open for higher dimensions.

We define conic frameworks as an extension of Euclidean frameworks adapted for the pseudo-range context. The main difference is that the associated graph is directed since it represents pseudo-ranges which are asymmetrical. Furthermore, each agent is now characterized by its position $x \in \R^d$ and its bias $\beta \in \R$ (expressed in units of length for convenience). 
\begin{definition}\label{def: Conic framework}
	A $d$-dimensional \emph{conic framework}  $(\Gamma, p)$ is the combination of a simple directed graph $\Gamma = (V,E)$ and a configuration $p$ from $V$ to $\R^{d+1}$ assigning to each vertex $u \in V$, $p(u) = \begin{pmatrix} x_{u}^\intercal & \beta_{u}\end{pmatrix}^\intercal$ a point of $\R^{d+1}$.
\end{definition}
To avoid confusion, simple undirected graphs are denoted with the Latin letters $G$ or $H$ and their links are referred as edges while simple directed graph are denoted with the Greek letter $\Gamma$ and their links are referred as arcs. Furthermore, even if the range of $p$ belongs to $\R^{d+1}$, we still talk about $d$-dimensional frameworks, $d$ being the spatial dimension. To simplify some expressions, the point $p(u)$ may be denoted in the following as $p_u$.
 
Euclidean frameworks can be considered as particular conic frameworks whose graph $\Gamma$ is symmetric, \ie{} $uw\in E \Leftrightarrow wu\in E$. In this case, $\Gamma$ can therefore be viewed as an undirected graph. Indeed, constraining the two symmetrical pseudo-ranges between a pair of agents is equivalent to constraining their distance and bias difference, that is the idea applied in the SDS TWR method. Several other extensions have already been proposed. For example, \cite{streinu_2007_slider, cruickshank_2020_rigidity, jackson_2021_improved} introduce linear constraints imposing the movement of some agents into some subspace but still in a Euclidean context. Other authors study the consequences of changing the Euclidean distance by other ones induced by norms: \cite{ kitson_2014_infinitesimal} focuses on $p$-norms (with $p \ne 2$), \cite{kitson_2015_finite} on polyhedral norms, \cite{kitson_2020_graph} on unitary invariant matrix norms and \cite{dewar_2020_infinitesimal} more generally on non-Euclidean norms (in the plane). Another interesting extension was introduced in \cite{gortler_2014_generic}. The authors study frameworks in complex and hyperbolic spaces. In the latter, the constraint between a pair of agents (when arranged to match our notations) is:
\begin{equation}\label{eq: Hyperbolic constraint}
	\norm{x_u - x_w}^2 - \left(\beta_u - \beta_w\right)^2
\end{equation}
This expression is very similar to \eqref{eq: Pseudo-range function} but creates a completely different problem. Notably, in \eqref{eq: Hyperbolic constraint} the constraint is symmetrical in $u$ and $w$. Conic frameworks are, as far as we know, a new concept motivated by the use of pseudo-ranges. Some other works deal with asymmetrical rigidity but, to the best of our knowledge, only in the Euclidean case. One notable application is the persistence of flight formations \cite{hendrickx_2007_directed} in which each constraint is preserved by only one agent. This concept is perfectly suited for conic frameworks since pseudo-ranges naturally have a direction.

Similarly to a Euclidean framework, a conic framework is said to be flexible if it can be bent while preserving the constraints, \ie{} the pseudo-ranges on its arcs; it is said to be rigid otherwise. The key notion of \emph{infinitesimal rigidity} is introduced after the following basic examples of conic frameworks.

\subsection{Examples of conic frameworks}

\begin{figure}
	\centering
	\null\hfill
	\begin{subfigure}[t]{0.30\textwidth}
		\centering
		\resizebox{\linewidth}{!}{\input{fig/conic_graph_hyperbolic_motion.pgf}}
		\caption{$(\Gamma_1, p_1)$}
		\label{sfig: Two arcs hyperbolic constraint illustration}
	\end{subfigure}
	\hfill
	\begin{subfigure}[t]{0.30\textwidth}
		\centering
		\resizebox{\linewidth}{!}{\input{fig/conic_graph_elliptical_motion.pgf}}
		\caption{$(\Gamma_2, p_2)$}
		\label{sfig: Two arcs elliptical constraint illustration}
	\end{subfigure}
	\hfill
	\begin{subfigure}[t]{0.30\textwidth}
		\centering
		\resizebox{\linewidth}{!}{\input{fig/conic_graph_rigid.pgf}}
		\caption{$(\Gamma_3, p_3)$}
		\label{sfig: Rigid framework}
	\end{subfigure}
	\hfill\null
	\caption{Examples of $2$-dimensional frameworks. The bias axis is not represented.}
	\label{fig: Conic frameworks}
\end{figure}
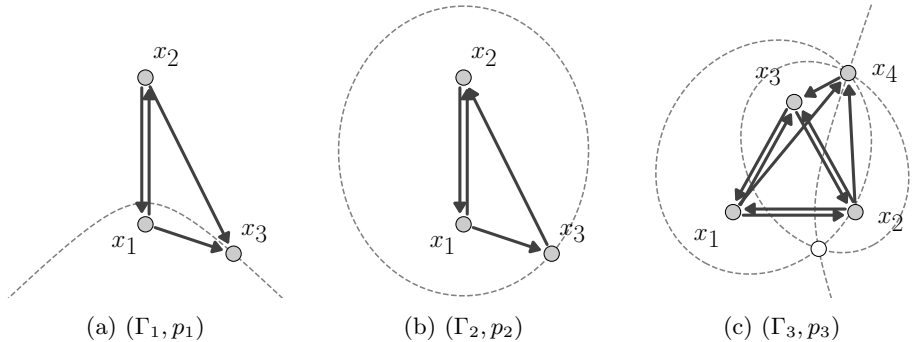

\cref{fig: Conic frameworks} presents three $2$-dimensional conic frameworks, this paragraph investigates their rigidity or flexibility. In the figures, agents are represented by circles, pseudo-range constraints by arrows and, for the sake of clarity, the bias axes are not represented.

First, consider the conic framework in \cref{sfig: Two arcs hyperbolic constraint illustration}. This framework has $4$ arcs corresponding to the pseudo-ranges from $1$ to $2$, from $2$ to $1$, from $1$ to $3$ and from $2$ to $3$; we denote $\rho_{i,j}$ the pseudo-range from $i$ to $j$. They constrain the following equations:
\begin{align}
	\norm{x_1 - x_2} + \beta_2 - \beta_1 &= \rho_{1,2} \label{eq: Conic example PSR 1 to 2}\\
	\norm{x_1 - x_2} + \beta_1 - \beta_2 &= \rho_{2,1} \label{eq: Conic example PSR 2 to 1} \\
	\norm{x_1 - x_3} + \beta_3 - \beta_1 &= \rho_{1,3} \label{eq: Conic example PSR 1 to 3} \\
	\norm{x_2 - x_3} + \beta_3 - \beta_2 &= \rho_{2,3} \label{eq: Conic example PSR 2 to 3}
\end{align}
Since both pseudo-ranges between the Agents $1$ and $2$ are constrained, the first two equations impose that both the distance and the bias difference between agents $1$ and $2$ are constrained: their relative positions and biases are set. Moreover, since the pseudo-ranges $\rho_{1,3}$ and $\rho_{2,3}$ are also constrained, subtracting \eqref{eq: Conic example PSR 1 to 3} from \eqref{eq: Conic example PSR 2 to 3} gives:
\begin{equation}
	\norm{x_1 - x_3} - \norm{x_2 - x_3} = \rho_{1,3} - \rho_{2,3} + \beta_1 - \beta_2 = \text{constant}
\end{equation}
Therefore, the position of Agent $3$ lies on a branch of hyperbola whose foci are $x_1$ and $x_2$. This branch is represented by the dashed line in the figure. Note that the bias $\beta_3$ is obtained by reinjecting the distance into either \eqref{eq: Conic example PSR 1 to 3} or \eqref{eq: Conic example PSR 2 to 3} and is not constant along this line: it decreases as the distances $\norm{x_1 - x_3}$ increases. Moving Agent $3$ along this $3$-dimensional curve of position preserves the four pseudo-range constraints but not the formation. Therefore, this conic framework is flexible.

The second conic framework in \cref{sfig: Two arcs elliptical constraint illustration} is very similar. The only difference lies in the direction of the arc between $2$ and $3$. In \cref{sfig: Two arcs hyperbolic constraint illustration}, $\rho_{2,3}$ was constrained whereas it is now $\rho_{3,2}$. This transforms \eqref{eq: Conic example PSR 2 to 3} to:
\begin{align}
	\norm{x_2 - x_3} + \beta_2 - \beta_3 &= \rho_{3,2} \label{eq: Conic example PSR 3 to 2}
\end{align}
This time, summing \eqref{eq: Conic example PSR 1 to 3} and \eqref{eq: Conic example PSR 3 to 2} gives:
\begin{equation}
	\norm{x_1 - x_3} + \norm{x_2 - x_3} = \rho_{1,3} + \rho_{3,2} + \beta_1 - \beta_2 = \text{constant}
\end{equation}
Consequently, $x_3$ lies on an ellipse, also represented by a dashed line in the figure. Similarly, moving Agent $3$ on this curve of position preserves the pseudo-range constraints therefore, this conic framework is also flexible. These two first examples underline how important the directions of the arcs are: flipping an arc changes the flexing of the framework.

The third conic framework in \cref{sfig: Rigid framework} is more complex. The three agents $1$, $2$ and $3$ are fully connected, therefore, all the distances and bias differences between them are constrained: their relative positions and biases are set. Agent $4$ is connected to each of them by one unique arc. Each pair of arcs constrains the position $x_4$ to lie on a curve which is either a branch of hyperbola or an ellipse depending on whether both arcs point in the same direction or not. These three curves are still represented by dashed lines, they intersect twice: at $x_4$ of course and at a second point represented by a white dot. Those two points are suitable positions for Agent $4$: placing it in one of these loci (with the corresponding bias) satisfies all the constraints. However, Agent $4$ cannot \emph{move} so the conic framework is rigid. Note that at this two loci, the associated bias are different since, for example, the distances to $x_2$ are different.

\subsection{Infinitesimal rigidity of conic frameworks}

This paragraph introduces the notions of \emph{rigidity} and \emph{infinitesimal rigidity} for conic frameworks. They require the following definitions that are adaptations of concepts from Euclidean rigidity theory, see \eg{} \cite{gluck_1975_almost, hendrickson_1992_conditions}.

\begin{definition}\label{def: Congruence}
	Two conic frameworks $(\Gamma, p)$ and $(\Gamma, p')$ are said to be \emph{congruent} if the pseudo-ranges between any pair of agents are equal in both configurations, \ie{} $\forall (u,w) \in V^2$, $r(p_u, p_w) = r(p'_u, p'_w)$.
\end{definition}
\cref{def: Congruence} is equivalent to say that the distances and bias differences between any pair of agents are equal in both configurations. Indeed, the distance $\norm{x_u - x_w}$ and the bias difference $\beta_u - \beta_w$ are linked with the pseudo-ranges $r(p_u, p_w)$ and $r(p_w, p_u)$ by the following invertible system:
\begin{align}
	\norm{x_u - x_w} &= \frac{r(p_u, p_w) + r(p_w, p_u)}{2} & \beta_u - \beta_w &= \frac{r(p_w, p_u) - r(p_u, p_w)}{2}
\end{align}

\begin{definition}
	A \emph{motion} from a conic framework $(\Gamma, p)$ to another conic framework $(\Gamma, p')$ is an application: $P : [0,1] \times V \to \R^{d+1}$ satisfying:
	\begin{enumerate}
		\item $\forall u \in V$, $P(0, u) = p(u)$ and $P(1, u) = p'(u)$;
		\item $\forall u$, $t \mapsto P(t, u)$ is continuous;
		\item $\forall uw \in E$, $\forall t \in [0,1]$, $r\left(P(t, u), P(t, w)\right) = r\left(p(u), p(w)\right)$.
	\end{enumerate}
\end{definition}
Note that motions must only preserve the pseudo-ranges along the arcs of the framework. They may create congruent or non-congruent frameworks. In the latter case, they are called \emph{flexing}.

\begin{definition}
	A \emph{flexing} of a conic framework $(\Gamma, p)$ is a motion starting at $(\Gamma, p)$ such that for any $t > 0$, the frameworks $\left(\Gamma, P(t)\right)$ and $(\Gamma, p)$ are not congruent.
\end{definition}
For example, the frameworks depicted in \cref{sfig: Two arcs hyperbolic constraint illustration} and \cref{sfig: Two arcs elliptical constraint illustration} have a flexing generated by the motion of Agent $3$ along the curve of position. Of course not all motions are flexing, rigid motions (combinations of rotations and translations  of the spatial position) and translations of all the biases are not. These particular motions preserve the formation and are said to be trivial. We can now define rigidity of conic frameworks.
\begin{definition}
	A conic framework with no flexing is said to be \emph{rigid}.
\end{definition}

Infinitesimal rigidity is a strengthening of rigidity at the first order. Consider that the agents are moving smoothly. The position of Agent $u$ at time $t$ is denoted $p_u(t) = \begin{pmatrix} x_u(t)^\intercal & \beta_u(t)\end{pmatrix}^\intercal$ and its instantaneous velocity is denoted $q_u(t) = \begin{pmatrix} v_u(t)^\intercal & \alpha_u(t)\end{pmatrix}^\intercal$ where:
\begin{align} 
	v_u(t) &= \frac{\diff x_u(t)}{\diff t} & \alpha_u(t) &=\frac{\diff \beta_u(t)}{\diff t}
\end{align}
If the agents always keep distinct positions, \ie{} $\forall t$, $u\ne w \Rightarrow x_u(t) \ne x_w(t)$, the distances are also differentiable. Then the preservation of a pseudo-range $\rho_{u,w}$ and its derivative give that $\forall t \in [0,1]$:
\begin{align}
	\norm{x_u(t) - x_w(t)} + \beta_w(t) - \beta_u(t) &= \rho_{u,w} \\
	\frac{\left(x_u(t) - x_w(t)\right)^\intercal \cdot \left(v_u(t) - v_w(t)\right)}{\norm{x_u(t) - x_w(t)}} + \alpha_w(t) - \alpha_u(t) &= 0 \label{eq: Derivative of the preservation of a PSR}
\end{align}
Applying \eqref{eq: Derivative of the preservation of a PSR} at $t = 0$ gives:
\begin{equation}\label{eq: Derivative of the preservation of a PSR in 0}
	\frac{\left(x_u -x_w\right)^\intercal \cdot \left(v_u -v_w\right) }{\norm{x_u - x_w}} + \alpha_w - \alpha_u = 0
\end{equation}
where $p_u = p_u(0)$ and $q_u = q_u(0)$ denote the initial position and velocity of Agent $u$ for the sake of clarity.

The vector $q_u = \begin{pmatrix}v_u^\intercal & \alpha_u\end{pmatrix}^\intercal \in \R^{d+1}$ denotes the instantaneous velocity of Agent~$u$. All these velocities are stacked into a vector $q = \begin{pmatrix} v_1^\intercal & \dots & v_n^\intercal & \alpha_1 & \dots & \alpha_n\end{pmatrix}^\intercal \in \R^{(d+1)n}$ called an instantaneous velocity vector (where $n$ denotes the number of agents). An instantaneous velocity vector is said to be admissible for the conic framework $(\Gamma, p)$ if \eqref{eq: Derivative of the preservation of a PSR in 0} is satisfied for every arc or equivalently:
\begin{align}\label{eq: Arc infinitesimal constraint}
	\forall uw &\in E, & \left(x_u - x_w\right)^\intercal \cdot \left(v_u - v_w\right) + \norm{x_u - x_w} \left(\alpha_w - \alpha_u\right) &= 0
\end{align}
Stacking the equations of \eqref{eq: Arc infinitesimal constraint} gives that $M(\Gamma, p)q = 0$ where $M(\Gamma, p)$ is a matrix called the (conic) rigidity matrix of the framework. This rigidity matrix has $(d+1)n$ columns, one per variable, and $\card{E}$ rows (where $\card{.}$ denotes the cardinal), one per arc. Therefore, the set of admissible instantaneous velocity vectors is a vector space.

The conic rigidity matrix is separated in two blocks. The first $dn$ columns correspond to the spatial variables and form the classical Euclidean rigidity matrix $M_e(\Gamma,p)$, see \eg{} \cite{hendrickson_1992_conditions}. The last $n$ columns correspond to the bias variables and form a matrix $B(\Gamma,p)$ defined as:
\begin{equation}\label{eq: Decomposition B}
	B(\Gamma,p) = D(\Gamma,p)B(\Gamma)
\end{equation}
where $D(\Gamma,p) = \diagvect\left(\dots, \norm{x_u-x_w}, \dots\right)$ is the diagonal matrix whose $i$-th entry is the distance between the points connected by the $i$-th arc and $B(\Gamma)$ is the transpose of the incidence matrix of the graph, see \eg{} \cite{bollobas_1998_modern}. Note that the direction of the arcs only appear in $B(\Gamma)$ and $B(\Gamma, p)$.

With these notations, $M(\Gamma,p) = \begin{bmatrix} M_e(\Gamma,p) & B(\Gamma,p) \end{bmatrix}$. For example, the rigid matrix of the conic framework in \cref{sfig: Two arcs hyperbolic constraint illustration} is:
\begin{equation}\label{eq: Example rigidity matrix}
	M(\Gamma,p) = \begin{bmatrix}
		x^\intercal_1 - x^\intercal_2 & x^\intercal_2 - x^\intercal_1 & 0_d^\intercal & -d_{1,2} & d_{1,2} & 0\\
		x^\intercal_1 - x^\intercal_2 & x^\intercal_2 - x^\intercal_1 & 0_d^\intercal & d_{1,2} & -d_{1,2} & 0\\
		x^\intercal_1 - x^\intercal_3 & 0_d^\intercal & x^\intercal_3 - x^\intercal_1 & -d_{1,3} & 0 & d_{1,3}\\
		0_d^\intercal & x^\intercal_2 - x^\intercal_3 & x^\intercal_3 - x^\intercal_2 & 0 & -d_{2,3} & d_{2,3}
	\end{bmatrix}
\end{equation}
%\begin{equation}
%	M(G,p) = \begin{bmatrix}
%		x_a - x_b & y_a - y_b & x_b - x_a  & y_b - y_a & 0 & 0 & -d_{a,b} & d_{a,b} & 0\\
%		x_a - x_b & y_a - y_b & x_b - x_a  & y_b - y_a & 0 & 0 & d_{a,b} & -d_{a,b} & 0\\
%		x_a - x_c & y_a - y_c & 0 & 0 & x_c - x_a  & y_c - y_a & -d_{a,c} & 0 & d_{a,c}\\
%		0 & 0 & x_b - x_c  & y_b - y_c & x_c - x_b & y_c - y_b & 0 & -d_{b,c} & d_{b,c}
%	\end{bmatrix}
%\end{equation}
where $d_{u,w} = \norm{x_u - x_w}$ and $0_d$ denotes the null vector of $\R^d$. This decomposition provides a bound on the rank of the conic rigidity matrix:
\begin{align}\label{eq: Bound rank}
	\rank M(\Gamma,p) \le S(n,d) = S_e(n,d) + n - 1
\end{align}
In \eqref{eq: Bound rank}, $n-1$ represents the maximal rank of the incidence matrix (the vector filled with ones always annihilates the incidence matrix). The term $S_e$ is the maximal rank of the Euclidean rigidity matrix \cite{hendrickson_1992_conditions}:
\begin{equation}
	S_e(n,d)=\left\{\begin{array}{ll}
		dn - \binom{d+1}{2} & \text{if } n \ge d + 1\\
		\binom{n}{2} = S_e(n,n-1) & \text{if } n \le d
	\end{array}\right.
\end{equation}

We are now in position to define infinitesimal rigidity in the context of conic frameworks. The following definition generalizes the definition of infinitesimal rigidity in the Euclidean context \cite{asimow_rigidity_1979}.
\begin{definition}\label{def: Infinitesimal rigidity}
	A $d$-dimensional conic framework $(\Gamma,p)$ with $n$ vertices is \emph{infinitesimally rigid} if and only if $\rank M(\Gamma,p) = S(n,d)$.
\end{definition}
The bound $S(n,d)$ on the rank comes from the fact that trivial motions induce instantaneous velocity vectors that are always admissible and are therefore said to be trivial. These trivial vectors form a vector space of dimension at most $\binom{d+1}{2} + 1$ depending on the number of points and their geometry. Indeed, there are $\binom{d}{2}$ spatial rotations, $d$ spatial translations and $1$ bias translation. An equivalent definition for infinitesimal rigidity is the following.
\begin{definition}\label{def: Infinitesimal rigidity alternative}
	A conic framework $(\Gamma,p)$ is infinitesimally rigid if and only if every admissible instantaneous velocity vector is trivial.
\end{definition}

If a conic framework is infinitesimally rigid, the only possible smooth motions lead to congruent frameworks, therefore it cannot be flexed. These definitions are extensions of infinitesimal rigidity of Euclidean frameworks \cite{hendrickson_1992_conditions}:
\begin{definition}\label{def: Euclidean Infinitesimal rigidity}
	A $d$-dimensional Euclidean framework $(G,p)$ with $n$ vertices is \emph{infinitesimally rigid} if and only if $\rank M_e(\Gamma,p) = S_e(n,d)$.
\end{definition}

Finally, the definition of generic configuration is also extended to conic frameworks.
\begin{definition}
	A configuration is said to be \emph{generic} if the set of the $dn$ coordinates of $x_1, \dots, x_n$ are not root of any non-trivial polynomial with integer coefficients.
	
	A conic framework $(\Gamma, p)$ is said to be \emph{generic} if $p$ is generic.
\end{definition}
Note that the definition is independent from the bias.

The remainder of this paper characterizes infinitesimal rigidity of generic conic frameworks. \cref{def: Infinitesimal rigidity} is based on the rank of the rigidity matrix. Computing a rank may be numerically imprecise, \eg{} if the agents are almost aligned, the rigidity matrix may have small eigenvalues that may cause an underestimation of the rank. The characterizations introduced in the following depend only on the graph of the framework. Since graphs are discrete mathematical objects (they can be encoded by arrays of integers), the characterizations do not suffer from the rounding issues that may appeared in a rank computation. 

\section{Infinitesimal rigidity of unidimensional frameworks}\label{sec: Rigidity of unidimensional frameworks}

This section focuses on infinitesimal rigidity of unidimensional conic frameworks, \ie{} $d = 1$. The main result is \cref{the: Unidimensional case rigidity} that provides a necessary and sufficient condition for a conic framework (generic or not) to be infinitesimally rigid in this particular case.

The specificity of the unidimensional case is the natural ordering of $\R$ which orients the arcs. An arc $uw$ is said to be increasing if $x_u < x_w$, decreasing if $x_u > x_w$ and null if $x_u = x_w$. This orientation depends on the configuration but not on the biases. The set of increasing and decreasing arcs are denoted $E_+$ and $E_-$:
\begin{align*}
	E_+ &= \{uw \in E \mid x_u < x_w\} & E_- &= \{uw \in E \mid x_u > x_w\}
\end{align*}
Note that a symmetrical pair of (non-null) arcs is distributed in both $E_+$ and $E_-$. The undirected graphs induced by these two edge sets are denoted $G_+ = (V,E_+)$ and $G_- = (V, E_-)$. They are deduced from $(\Gamma ,p)$ and depend on both $\Gamma$ and $p$ as illustrated in \cref{fig: Illustration unidimensional rigidity theorem}. The connectivities of $G_+$ and $G_-$ provide the main result of this section:

\begin{figure}
	\centering
	\centering
	\null\hfill
	\begin{subfigure}[t]{0.30\textwidth}
		\centering
		\resizebox{\linewidth}{!}{\input{fig/unidimensional_frameworks_graph.pgf}}
		\caption{Graph $\Gamma$.}
		\label{sfig: Illustration unidimensional rigidity theorem graph}
	\end{subfigure}
	\hfill
	\begin{subfigure}[t]{0.30\textwidth}
		\centering
		\resizebox{\linewidth}{!}{\input{fig/unidimensional_framework_1.pgf}}
		\caption{$(\Gamma, p)$ flexible.}
		\label{sfig: Illustration unidimensional rigidity theorem example 1}
	\end{subfigure}
	\hfill
	\begin{subfigure}[t]{0.30\textwidth}
		\centering
		\resizebox{\linewidth}{!}{\input{fig/unidimensional_framework_2.pgf}}
		\caption{$(\Gamma, p')$ rigid.}
		\label{sfig: Illustration unidimensional rigidity theorem example 2}
	\end{subfigure}
	\hfill\null

	\caption{Example of two unidimensional conic frameworks having the same graph $\Gamma$ (\cref{sfig: Illustration unidimensional rigidity theorem graph}) with their corresponding graphs $G_+$ and $G_-$. \cref{the: Unidimensional case rigidity} states that the conic framework in \cref{sfig: Illustration unidimensional rigidity theorem example 1} is not infinitesimally rigid as $G_-$ is disconnected while the conic framework in \cref{sfig: Illustration unidimensional rigidity theorem example 2} is infinitesimally rigid.}
	\label{fig: Illustration unidimensional rigidity theorem}
\end{figure}
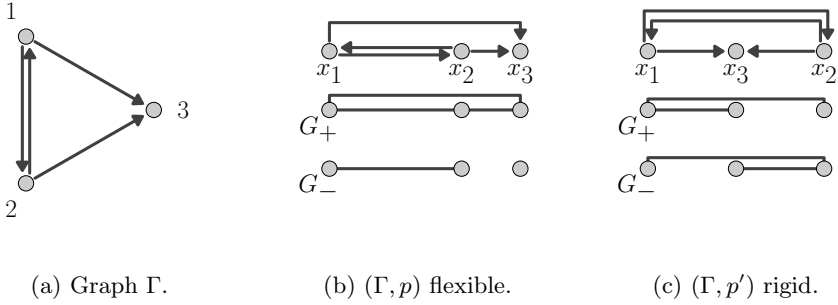

\begin{theorem}\label{the: Unidimensional case rigidity}
	Let $(\Gamma,p)$ be a unidimensional conic framework.
	$(\Gamma,p)$ is infinitesimally rigid if and only if both $G_+$ and $G_-$ are connected.
\end{theorem}
\begin{proof}
	First, assume $G_+$ and $G_-$ are both connected and let $q$ be an admissible instantaneous velocity, \ie{} $M(\Gamma,p) q = 0$. Let us prove that $q$ is trivial. In this particular unidimensional case, it means that the vectors $q_u$ are equal as the only trivial motions are translations. Furthermore, in the context, $q$ is admissible for a framework if it satisfies:
	\begin{align}\label{eq: Arc infinitesimal constraint unidimensional case}
		\forall uw &\in E, & (x_u - x_w) \left[v_u - v_w - \sign{(x_u-x_w)} (\alpha_u - \alpha_w)\right] &= 0
	\end{align}
	where $\sign$ is the sign function. This constraint is the simplification of \eqref{eq: Arc infinitesimal constraint} in the unidimensional case.
	Consequently, an increasing arc $uw \in E$, imposes $v_u + \alpha_u = v_w + \alpha_w$, and  a decreasing arc $uw\in E$ imposes $v_u - \alpha_u = v_w - \alpha_w$. Since both $G_+$ and $G_-$ are connected, both equalities hold for any pair of vertices $(u,w) \in V^2$. Consequently, for any pair $(u,w) \in V^2$, $v_u = v_w$ and $\alpha_u = \alpha_w$, \ie{} $q$ is trivial. 
	
	Conversely, suppose \Wlog{}, that $G_-$ is not connected and let us prove that there exists a non-trivial instantaneous velocity vector. Since $G_-$ is not connected, there is a non-trivial subset of vertices $U \in \rP(V)\setminus\{\emptyset, V\}$ (where $\rP(V)$ denotes the power set of $V$) such that $U$ and $U^c=V\setminus U$ are not connected in $G_-$. Let $q$ be the instantaneous velocity vector defined as: $q_u = \begin{pmatrix}1 & -1\end{pmatrix}^\intercal$ if $u \in U$ and $q_u = 0$ if $u\in U^c$. As $\card{U}\ge 1, \card{U^c} \ge 1$, $q$ is not trivial. Let us prove that $q$ is admissible for $(\Gamma, p)$, \ie{} $M(\Gamma,p) q = 0$. Clearly, $q$ satisfies all the arc constraints between any pair of vertices of $U$ and any pair of vertices of $U^c$. Let us then consider, \Wlog{}, an arc $uw \in E$ with $u \in U$ and $w \in U^c$, the case $w \in U$ and $u \in U^c$ would be symmetrical. The arc $uw$ can not be decreasing since $U$ and $U^c$ are not connected in $G_-$. If $uw$ is null, $q$ satisfies the constraint. Then, suppose that $uw$ is increasing. Equation \eqref{eq: Arc infinitesimal constraint unidimensional case} becomes:
	\begin{equation}
		\norm{x_u - x_w}(v_u - v_w + \alpha_u - \alpha_w) = \norm{x_u - x_w}(1 - 1) = 0
	\end{equation}
	Thus, $q$ is admissible and the framework is not infinitesimally rigid.
\end{proof}

This theorem gives a simple and efficient way to numerically check whether a unidimensional conic framework is infinitesimally rigid. To verify if a conic framework is infinitesimally rigid, \begin{enumerate*}[label =(\roman*)]
	\item decompose $E$ into $E_+$ and $E_-$ and \item test the connectivity of $G_+$ and $G_-$.
\end{enumerate*}
Both phases can be computed with $O(|E|)$ operations, see \eg{} \cite{even_2011_graph}. 

Furthermore, as $G_+$ and $G_-$ depend on the configuration $p$, infinitesimal rigidity of unidimensional conic frameworks is not a generic property of their graph. For example, the conic frameworks in \cref{fig: Illustration unidimensional rigidity theorem} have the same graph but only one is infinitesimally rigid. This is a major difference compared to Euclidean frameworks \cite{hendrickson_1992_conditions} or multidimensional conic frameworks as explained in \cref{sec: Rigidity of multidimensional frameworks}.
 
\section{Rigidity of multidimensional frameworks} \label{sec: Rigidity of multidimensional frameworks}

This section focuses on infinitesimal rigidity of multidimensional conic frameworks, \ie{} $d \ge 2$. It contains two main results: \cref{the: Multidimensional case rigidity} and \cref{the: Decomposition of conic rigid graph}. 

\subsection{Statement of the main results}

To clearly state the two theorems, some additional definitions are required.

Two simple directed graphs with the same vertex set are \emph{equivalent} if they have the same underlying undirected multi-graph. This multi-graph may have simple and double edges. We call \emph{conic graph} an undirected multi-graph $\tilde \Gamma = (V,E_S,E_D)$ associated with such an equivalence class. $E_S$ is the set of simple edges and $E_D$ the set of double edges. \cref{fig: Example of two equivalent graphs} presents an example of two equivalent directed graphs with the associated conic graph. For a conic framework $(\Gamma, p)$, if the equivalence class of $\Gamma$ is $\tilde \Gamma$, we say that $\tilde \Gamma$ is the conic graph of the framework. We are now in position to state the first theorem.
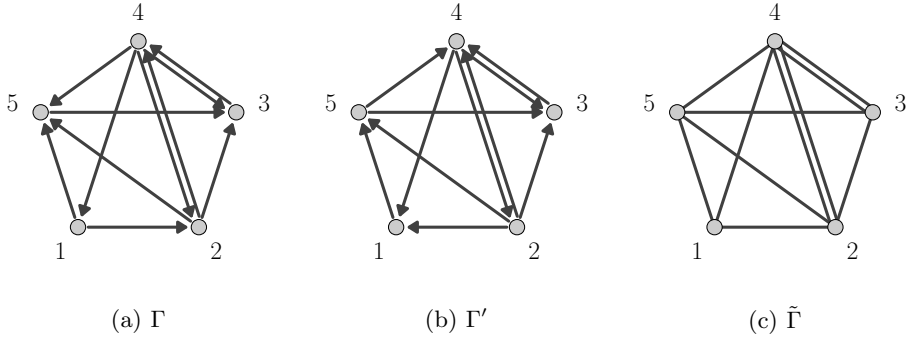
\begin{figure}
	\centering
	\null\hfill
	\begin{subfigure}[t]{0.30\textwidth}
		\centering
		\resizebox{\linewidth}{!}{\input{fig/example_equivalent_graphs_1.pgf}}
		\caption{$\Gamma$}
	\end{subfigure}
	\hfill
	\begin{subfigure}[t]{0.30\textwidth}
		\centering
		\resizebox{\linewidth}{!}{\input{fig/example_equivalent_graphs_2.pgf}}
		\caption{$\Gamma'$}
	\end{subfigure}
	\hfill
	\begin{subfigure}[t]{0.30\textwidth}
		\centering
		\resizebox{\linewidth}{!}{\input{fig/example_equivalent_graphs_class.pgf}}
		\caption{$\tilde \Gamma$}
		\label{sfig: Example conic graph}
	\end{subfigure}
	\hfill\null
	\caption{Example of two equivalent directed graphs $\Gamma$ and $\Gamma'$ with the conic graph $\tilde \Gamma$ associated with their equivalence class. Only the arcs from $1$ to $2$ and from $4$ to $5$ have been reversed from $\Gamma$ to $\Gamma'$. The set $E_D$ of double edges is represented by double lines.}
	\label{fig: Example of two equivalent graphs}
\end{figure}

\begin{theorem}\label{the: Multidimensional case rigidity}
	Let $d\ge 2$ and $\tilde \Gamma$ be a conic graph. Either every generic $d$-dimensional conic framework whose underlying conic graph is $\tilde \Gamma$ is infinitesimally rigid or none of them is. In this former case, $\tilde \Gamma$ is said to be rigid in $\R^d$.
\end{theorem}

\cref{the: Multidimensional case rigidity} implies that two generic conic frameworks having equivalent graphs have the same infinitesimal rigidity. For example, in \cref{fig: Conic frameworks} of \cref{sec: Background}, the graphs $\Gamma_1$ and $\Gamma_2$ are equivalent and neither the conic frameworks $(\Gamma_1, p_1)$ nor $(\Gamma_2, p_2)$ is infinitesimal rigid.

The second main result characterizes rigidity of conic graphs. It involves a \emph{decomposition} of the conic graph into two special \emph{Euclidean graphs}.
We call a Euclidean graph a simple undirected graph---such graphs are associated with Euclidean frameworks. Given two Euclidean graphs $G_1 = (V, E_1)$ and $G_2 = (V, E_2)$ with the same set of vertices, we define their \emph{union} as the conic graph $\tilde \Gamma = G_1 \cup G_2 = (V,E_S,E_D)$ obtained by the union of their edge sets: $E_S = E_1 \triangle E_2$ (symmetric difference) and $E_D = E_1 \cap E_2$. The pair $(G_1, G_2)$ is called a decomposition of $\tilde \Gamma$ and is generally not unique. For example, \cref{fig: Decomposition conic graph} provides three different decompositions of the conic graph of \cref{sfig: Example conic graph}.
\begin{figure}
	\centering
	\null\hfill
	\begin{subfigure}[t]{0.30\textwidth}
		\centering
		\resizebox{\linewidth}{!}{\input{fig/decomposition_conic_original_rigid.pgf}}
		\caption{$G_1$}
	\end{subfigure}
	\hfill
	\begin{subfigure}[t]{0.30\textwidth}
		\centering
		\resizebox{\linewidth}{!}{\input{fig/decomposition_conic_counterexample1_rigid.pgf}}
		\caption{$G_2$}
	\end{subfigure}
	\hfill
	\begin{subfigure}[t]{0.30\textwidth}
		\centering
		\resizebox{\linewidth}{!}{\input{fig/decomposition_conic_counterexample2_rigid.pgf}}
		\caption{$G_3$}
	\end{subfigure}
	\hfill\null\\
	\null\hfill
	\hfill
	\begin{subfigure}[t]{0.30\textwidth}
		\centering
		\resizebox{\linewidth}{!}{\input{fig/decomposition_conic_original_spanning_tree.pgf}}
		\caption{$H_1$}
	\end{subfigure}
	\hfill
	\begin{subfigure}[t]{0.30\textwidth}
		\centering
		\resizebox{\linewidth}{!}{\input{fig/decomposition_conic_counterexample1_spanning_tree.pgf}}
		\caption{$H_2$}
	\end{subfigure}
	\hfill
	\begin{subfigure}[t]{0.30\textwidth}
		\centering
		\resizebox{\linewidth}{!}{\input{fig/decomposition_conic_counterexample2_spanning_tree.pgf}}
		\caption{$H_3$}
	\end{subfigure}
	\hfill\null
	\caption{Examples of decompositions of the conic graph $\tilde \Gamma$ of \cref{sfig: Example conic graph}. Each column presents a decomposition of $\tilde \Gamma$ as $G \cup H$.}
	\label{fig: Decomposition conic graph}
\end{figure}
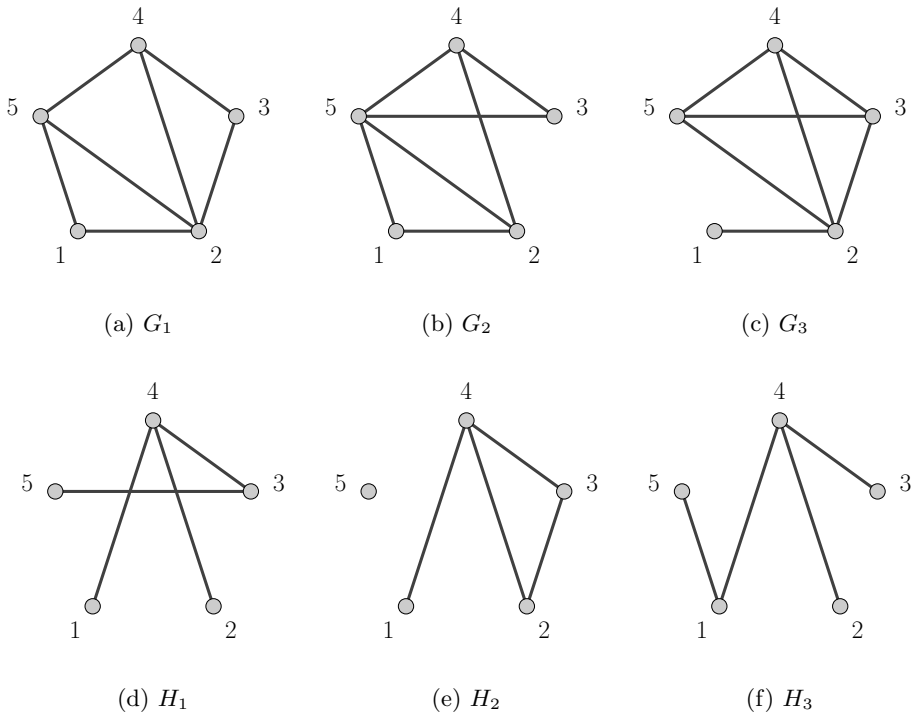
Similarly to conic graphs, a Euclidean graph $G$ is said to be rigid in $\R^d$ if the $d$-dimensional generic Euclidean frameworks whose graph is $G$ are infinitesimally rigid \cite{hendrickson_1992_conditions}. Indeed, infinitesimal rigidity of Euclidean frameworks also depends only on the graphs of the frameworks \cite{asimow_rigidity_1978}. We can now state the second theorem.

\begin{theorem}\label{the: Decomposition of conic rigid graph}
	Let $d \ge 2$ and $\tilde \Gamma$ be a conic graph.
	$\tilde \Gamma$ is rigid in $\R^d$ if and only if there exists $(G, H)$ a decomposition of $\tilde \Gamma$ such that $G$ is rigid in $\R^d$ and $H$ is connected.  
\end{theorem}

If a conic graph is rigid, \cref{the: Decomposition of conic rigid graph} solely implies that at least one decomposition satisfies the two hypotheses. It does not imply that any decomposition of a rigid conic graph does. For example, the conic graph $\tilde \Gamma$ represented in \cref{sfig: Example conic graph} is rigid in $\R^2$ since the decomposition $(G_1, H_1)$ in \cref{fig: Decomposition conic graph} satisfies the hypotheses: $G_1$ is rigid in $\R^2$ and $H_1$ is connected. However, the other decompositions $(G_2, H_2)$ and $(G_3, H_3)$ do not satisfy one of the properties: in the former $H_2$ is not connected and the latter $G_3$ is not rigid in $\R^2$. Note that finding a decomposition satisfying the two hypotheses is not trivial. Even if $\tilde \Gamma$ is rigid and $G$ is \emph{minimally rigid} (\ie{} becomes flexible after removing any edge), the resulting graph $H$ may be unconnected, \eg{} consider the decomposition $(G_2, H_2)$ in \cref{fig: Decomposition conic graph}. Similarly, if $\tilde \Gamma$ is rigid and $H$ is minimally connected (\ie{} is a spanning tree), $G$ may be flexible, \eg{} consider the decomposition $(G_3, H_3)$ in \cref{fig: Decomposition conic graph}.

The proofs of \cref{the: Multidimensional case rigidity} and \cref{the: Decomposition of conic rigid graph} are presented in the following sub-sections.

\subsection{Proof of \cref{the: Multidimensional case rigidity}}

\cref{the: Multidimensional case rigidity} states that infinitesimal rigidity of generic conic frameworks depends only on their conic graph. This section proves a slightly stronger result: the rank of the rigidity matrix of a generic conic framework depends only on its conic graph. The approach is similar to the one used for Euclidean frameworks in \eg{} \cite{hendrickson_1992_conditions}. First, remind that the rank may be defined as the order of a highest order non-vanishing minor. For conic rigidity matrices, the minors can be viewed as functions of the coordinates of the points. Contrary to the minors of Euclidean rigidity matrices, these minor functions are not polynomial in the coordinates but belong to some larger space we call $\L$. The sketch of the proof is as follows. First, \cref{lem: Field extension} characterizes the space $\L$. Then, \cref{lem: Extension generic nullity} proves that if a generic configuration annihilates a minor function, then the minor function is the null function.  As a consequence, \cref{lem: Generic configuration maximizes the rank} proves that the generic configurations maximize the rank of the rigidity matrix over all the possible configurations. Finally, by exploiting the structure of $\L$, \cref{lem: Generic configuration equivalent rank} proves similarly that this maximum generic rank is a property not only of the graph of the framework but of its conic graph. This section ends with the detailed proof of \cref{the: Multidimensional case rigidity}.

The entries of the rigidity matrix of a conic framework with $n$ agents are function of the $dn$ coordinates of the points: $\set{x_u^{(i)} \mid u \in V \text{ and } i \in \set{1, \dots, d}}$ where $x_u^{(i)}$ denotes the $i$-th coordinate of Agent $u$. These entries are either a linear function or a distance function, see \eqref{eq: Example rigidity matrix} for an example of a conic rigidity matrix. Therefore the minor functions of the conic rigidity matrix are functions of these $dn$ coordinates.

More precisely, for a conic framework $(\Gamma, p)$ whose underlying conic graph is $\tilde \Gamma = (V, E_S, E_D)$, the minor functions belong to the space $\L = \L(E_S \cup E_D)$ defined as follows. For any set of edges $E \subset \{uw \mid 1 \le u < w \le n\}$, the space $\L(E)$ is defined as:
\begin{equation}\label{eq: Definition of L}
	\L(E) = \set{ \sum_{F \in \rP(E)} P_F \prod_{uw \in F} D_{u,w} \mid \forall F, P_F \in \K }
\end{equation}
where $\K = \Q\left(X_1^{(1)}, \dots, X_n^{(d)}\right)$ is the field of rational functions with integer coefficients in $dn$ variables, $\rP(E)$ denotes the power set of $E$ and, the distance $D_{u,w}$ is the function in $dn$ variables:
\begin{equation}
	D_{u,w} : \left\{
	\begin{array}{rcl}
		\R^{dn} & \to & \R \\
		x_1^{(1)}, \dots, x_n^{(d)} & \mapsto & \sqrt{\sum_{i=1}^{d} \left(x_u^{(i)} - x_w^{(i)}\right)^2 }
	\end{array}
	\right.
\end{equation}

The definitions of $\L(E)$ and $\K$ may appear excessively complex. They have been both chosen to provide a field structure to $\L(E)$ as explained in the following lemma.
\begin{lemma}\label{lem: Field extension}
	Let $d\ge 2$, $E \subset \{uw \mid 1 \le u < w \le n\}$ be a set of edges and $m = \card{E}$. Then, $\L(E) / \K $ is a field extension of degree $2^m$. Furthermore, the family $\set{\prod_{uw \in F}D_{u,w} \mid F \in \rP(E)}$ is a basis of $\L(E)$ viewed as a $\K$-vector space. We call this basis the \emph{natural basis} of $\L(E)$. 
\end{lemma}

\cref{lem: Field extension} involves several elements from field theory. They will not be discussed as they come from a branch of mathematics, namely algebra, different than and far from the rest of the concepts used in this paper. Nonetheless, these concepts can be found \eg{} in \cite{roman_2005_field} and the proof of \cref{lem: Field extension} is provided in \cref{sapp: Proof of lemma field extension} for the sake of completeness.
Only the implications of this lemma are explained here. The first important point is that $\L$ is a field. Therefore, every nonzero element has a multiplicative inverse. Second, $\L$ is a $\K$-vector space of dimension $2^m$ and one natural basis is known. For example, if $E = \set{ab, bc, ac}$, the natural basis of $\L$ has eight elements: the constant function equals to $1$, the three distance functions $D_{a,b}, D_{b,c}, D_{a,c}$, the three products of two distance functions $D_{a,b}D_{b,c}, D_{b,c} D_{a,c}, D_{a,c}D_{a,b}$ and the product of the three distance functions $D_{a,b}D_{b,c}D_{a,c}$.
Finally, \cref{lem: Field extension} also implies that the polynomials $P_F$ involved in \eqref{eq: Definition of L} are unique as they are the coordinates on the natural basis.

Remind that a vector $x$ is said to be \emph{generic} if its coordinates are not root of any non-null polynomial with integer coefficients. By definition, the $dn$ coordinates of a generic configuration form a generic vector of $\R^{dn}$. The structure of $\L$ allows to extend this property to $\L$ as explained in the following lemma.
\begin{lemma}\label{lem: Extension generic nullity}
	Let $d \ge 2$, $E \subset \{uw \mid 1 \le u < w \le n\}$ be a set of edges and $f \in \L(E)$.
	If there exists $x \in \R^{dn}$ is a generic vector such that $f(x) = 0$, then $f = 0$.
\end{lemma}
\begin{proof}
	Let $x$ be a generic vector and $m = \card{E}$. For every $E' \subset E$, by \cref{lem: Field extension}, $\L(E')$ is a field and a $\K$-vector space of dimension $2^{\card{E'}}$ whose natural basis is composed of the products between the distance functions.
	
	Let us prove the lemma by induction on the number of distance functions appearing in the expression of $f$. Let us prove that:
	$\forall k \in \set{0, \dots, m}$, if $f \in \L(E')$ with $\card{E'} = k$ and if $f(x) = 0$, then $f = 0$.
	
	\emph{Base case:} If $k = 0$, $f$ is a rational function with integer coefficients, \ie{} $f = P/Q$ with $P$ and $Q$ two polynomials with integer coefficients. By definition, since $x$ is generic and $P(x) = 0$, $P$ is the null function and therefore $f = 0$.
	
	\emph{Inductive step:} Let $E'$ have a cardinality of $k+1$ with $k \ge 0$, $f \in \L(E')$ with $f(x) = 0$ and $uw \in E'$. Any function $h \in \L(E')$ can be uniquely decomposed, by separating the natural basis of $\L(E')$, as $h_1 + D_{u,w} h_2$ with $h_1, h_2 \in \L(E'\setminus{\set{uw}})$.
	Let $f = f_1 + D_{u,w} f_2$ be this decomposition applied to $f$. Furthermore, let $\bar f = f_1 - D_{u,w} f_2$ and $g = f\bar f = f_1^2 - D_{u,w}^2 f_2^2$. As $D_{u,w}^2$ is a polynomial, $g \in \L(E'\setminus{\set{uw}})$ whose cardinal is $k$. Since $f$ vanishes at $x$, $g$ also vanishes at $x$ and by the induction hypothesis, $g = 0$.
	Therefore, since $\L(E')$ is a field, either $f = 0$ or $\bar f = 0$. By definition, $f$ and $\bar f$ have the same coordinates up to a sign in the natural basis, thus $f = 0$.
\end{proof}

\cref{lem: Extension generic nullity} extends the definition of generic point. By definition, if a generic point annihilates a function $f \in \K$, then $f$ is the null function. This result remains true if $f$ belongs to $\L$. Since $\L$ is the field, the minor function of the rigidity matrix belong to $\L$. \cref{lem: Extension generic nullity} implies the two following main lemmas.
\begin{lemma}\label{lem: Generic configuration maximizes the rank}
	Let $d\ge2$ and $(\Gamma, p)$ be a $d$-dimensional conic framework. If $p$ is generic then:
	\begin{equation}
		\rank M(\Gamma,p) = \max_{p'} \rank M(\Gamma,p')
	\end{equation}
	where the maximum is taken over all the configurations $p'$ of $\Gamma$ in $\R^{d+1}$.
\end{lemma}
\begin{proof}
	Let $(\Gamma,p)$ be a generic conic framework, $p'$ be any other configuration and $r = \rank M(\Gamma, p)$. Let us prove that $\rank M(\Gamma, p') \le r$. To do so, let us prove that any minor of order $(r+1)$ of $M(\Gamma, p')$ is null. Let $N$ be a square sub-matrix of order $(r+1)$ of $M(\Gamma, p)$. Consider $f_N$ the function from $\R^{dn}$ to $\R$ associating to the coordinates of the points the minor corresponding to $N$. As $r = \rank M(\Gamma, p)$, by definition of the rank, $f_N(p) = 0$. As $p$ is generic and $f_N \in \L$, by \cref{lem: Extension generic nullity}, $f_N = 0$. In particular, $f_N(p') = 0$, and thus $\rank M(\Gamma, p') \le r$.
\end{proof}

The rank of the conic rigidity matrix is therefore a generic property of the graph. The following lemma extends this property to the conic graph.  

\begin{lemma}\label{lem: Generic configuration equivalent rank}
	Let $d\ge2$, $(\Gamma, p)$ be a $d$-dimensional conic framework and $\Gamma'$ be a graph equivalent to $\Gamma$. If $p$ is generic then:	
	\begin{equation}
		\rank M(\Gamma,p) = \rank M(\Gamma',p)
	\end{equation}
\end{lemma}
\begin{proof}
	Let $\Gamma$ and $\Gamma'$ be two equivalent graphs and $p$ be a generic configuration. Let us prove that $\rank M(\Gamma', p) \le \rank M(\Gamma, p)$. Then by symmetry, the equality will hold. Similarly to the proof of \cref{lem: Generic configuration maximizes the rank}, denote $r = \rank M(\Gamma, p)$, consider $N$ a square sub-matrix of order $(r+1)$ of $M(\Gamma, p)$ and denote $f_N$ the function associated to the corresponding minor. As $p$ is generic and $f_N \in \L$, by \cref{lem: Extension generic nullity}, $f_N = 0$. Furthermore, let $f'_N$ be the minor function of $M(\Gamma', p)$ associated to the same columns and to the rows associated with the equivalent arcs in $\Gamma'$. Flipping an arc $uw$ is equivalent to replace the entries of the rigidity matrix in $D_{u,w}$ by their opposites. Consequently, in the natural basis of $\L$, $f_N$ and $f'_N$ have the same coordinates up to the sign. Therefore, $f'_N = 0$ too and in particular, $f'_N(p) = 0$. Thus $\rank M(\Gamma', p) \le \rank M(\Gamma, p)$.
\end{proof}

Combining this two lemmas gives the proof of \cref{the: Multidimensional case rigidity}.
\begin{proof}[Proof of \cref{the: Multidimensional case rigidity}]
	Let $(\Gamma, p)$ and $(\Gamma', p')$ be two generic conic frameworks with $\Gamma$ and $\Gamma'$ equivalent. Since both $p$ and $p'$ are generic, by \cref{lem: Generic configuration maximizes the rank}:
	\begin{equation}
		\rank M(\Gamma, p) = \max_{p''} \rank M(\Gamma, p'') = \rank M(\Gamma, p')
	\end{equation}
	Then, since $\Gamma$ and $\Gamma'$ are equivalent, by \cref{lem: Generic configuration equivalent rank}: $\rank M(\Gamma, p') = \rank M(\Gamma', p')$.
	Therefore, $\rank M(\Gamma, p) = \rank M(\Gamma', p')$, thus $(\Gamma, p)$ and $(\Gamma', p')$ are either both infinitesimally rigid or none of them are.
\end{proof}

\subsection{Proof of \cref{the: Decomposition of conic rigid graph}}
Since \cref{the: Decomposition of conic rigid graph} is an equivalence, the two following implications need to be proved.
\begin{enumerate}[label=\textbf{I.\arabic*}, ref=Implication I.\arabic*]
	\item\label{it: Converse sense} If $G$ is a Euclidean rigid graph in $\R^d$ and $H$ is a Euclidean connected graph then, their union $\tilde \Gamma = G \cup H$ is rigid in $\R^d$.
	\item\label{it: Direct sense} If $\tilde \Gamma$ is rigid in $\R^d$ then, there exists $(G,H)$ a decomposition of $\tilde \Gamma$ such that $G$ is rigid in $\R^d$ and $H$ is connected.
\end{enumerate}

\ref{it: Converse sense} is proved first. It involves the following additional definitions on Euclidean and conic graphs. 

\begin{definition}
	Let $d \ge 2$.
	
	A Euclidean graph $G$ is \emph{independent} in $\R^d$ if for any $d$-dimensional generic framework $(G, p)$, the rows of its Euclidean rigidity matrix are independent.
	
	A conic graph $\tilde \Gamma$ is \emph{independent} in $\R^d$ if for any $d$-dimensional generic framework $(\Gamma, p)$ whose underlying conic graph is $\tilde \Gamma$, the rows of its conic rigidity matrix are independent.
\end{definition}
The definitions are correct since the rank of the Euclidean rigidity matrix does not depend on the generic configuration \cite{asimow_rigidity_1979} and similarly, the rank of the conic rigidity matrix does depend neither on the generic configuration nor on the graph but only on the equivalence class of the graph (\cref{the: Multidimensional case rigidity}).

\begin{definition}
	Let $d \ge 2$.
	
	A Euclidean graph $G$ is  \emph{minimally rigid} in $\R^d$ if it is independent in $\R^d$ and has $S_e(n,d)$ edges.	
	
	A conic graph $\tilde \Gamma$ is \emph{minimally rigid} in $\R^d$ if it is independent in $\R^d$ and has $S(n,d)$ edges.
\end{definition}
Note that the number of edges of a conic graph $\tilde \Gamma =(V, E_S, E_D)$ is $\card{E_S} + 2 \card{E_D}$. As the name suggests, minimally rigid graphs are rigid. Indeed, the generic rank of the rigidity matrix of an independent graph is equal to its numbers of edges. Minimally rigid graphs achieve therefore the maximum rank. Furthermore, as the name also suggests, minimally rigid graphs are the rigid graphs with the minimal number of edges.

\ref{it: Converse sense} is a consequence of \cref{lem: Independence union minimally rigid with forest}. To prove this lemma, the following result from algebra is required, its proof is provided in \cref{sapp: Proof of lemma Inversibility algebric}.
\begin{lemma}\label{lem: Inversibility algebric}
	Let $d \ge 2$, $E_H$ and $E_G$ be two sets of edges, and $A$ be a square matrix of order $\card{E_H}$ whose entries are functions from $\R^{dn}$ to $\R$.
	
	Suppose $E_H \cap E_G = \emptyset$ and the entries of $A$ belong to $\L(E_G)$. Then, for any generic configuration $p$, the matrix $B = D(H,p) + A(p)$ is invertible.	
\end{lemma}

\begin{lemma}\label{lem: Independence union minimally rigid with forest}
	Let $d \ge 2$, $G = (V, E_G)$ and $H = (V, E_H)$ be two Euclidean graphs.
	
	If $G$ is independent in $\R^d$ and if $H$ is a spanning forest then, their union $\tilde \Gamma = G \cup H$ is independent in $\R^d$.
\end{lemma}

\begin{proof}
	Let $G = (V, E_G)$ be an independent Euclidean graph and $H = (V, E_H)$ be a spanning forest. Furthermore, let $\tilde \Gamma = G \cup H$ and $(\Gamma,p)$ be a generic conic framework whose underlying conic graph is $\tilde \Gamma$. Let us prove that the rows of the conic rigidity matrix $M(\Gamma, p)$ are independent. Therefore, consider $\omega \in \Ker M(\Gamma, p)^\intercal$ and let us prove that $\omega = 0$.
	
	By sorting the arcs of $\Gamma$ starting with those associated with the edges of $G$ then those associated with the edges $H$, the rigidity matrix becomes:
	\begin{equation}
		M(\Gamma, p) =
		\begin{bmatrix}
			M_e(G, p) & B(G, p) \\
			M_e(H, p) & B(H, p)
		\end{bmatrix}
	\end{equation}
	By a slight abuse of notation, $B(G, p)$ and $B(H,p)$ denote the bias blocks of the conic rigidity matrix of $G$ and $H$ considered as directed graphs (whose arcs are directed according to $\Gamma$).

	By also decomposing $\omega = \begin{pmatrix}\omega_G^\intercal & \omega_H^\intercal\end{pmatrix}^\intercal$, the assumption $M(\Gamma, p)^\intercal \omega = 0$ gives:
	\begin{align}
		\label{eq: lem Independence equality Euclidean rigidity matrices} M_e(G, p)^\intercal \omega_G + M_e(H, p)^\intercal \omega_H &= 0\\
		\label{eq: lem Independence equality bias matrices}
		B(G, p)^\intercal \omega_G + B(H,p)^\intercal \omega_H &= 0
	\end{align}
	Introducing the incidence matrices of $G$ and $H$, see \eqref{eq: Decomposition B}, equation \eqref{eq: lem Independence equality bias matrices} becomes:
	\begin{equation}\label{eq: lem Independence equality bias matrices developped}
		B(G)^\intercal D(G,p) \omega_G + B(H)^\intercal  D(H,p) \omega_H = 0
	\end{equation}
	
	Then, since $G$ is independent, the rows of $M_e(G,p)$ are independent and therefore $M_e(G,p) M_e(G,p)^\intercal$ is invertible. Equation \eqref{eq: lem Independence equality Euclidean rigidity matrices} implies that:
	\begin{equation}\label{lem Independence simple equality distances}
		\omega_G + A \omega_H = 0
	\end{equation}
	with $A = \left[M_e(G,p) M_e(G,p)^\intercal\right]^{-1} M_e(G,p) M_e(H,p)^\intercal$.
	
	Similarly, since $H$ is a forest, the columns of its incidence matrix, \ie{} the rows of $B(H)$, are also independent \cite{bollobas_1998_modern}. Equation \eqref{eq: lem Independence equality bias matrices developped} implies that:
	\begin{equation}\label{lem Independence simple equality biases}
		C \omega_G + \omega_H = 0
	\end{equation}
	with $C = D(H,p)^{-1} \tilde C D(G,p)$ and $\tilde C = \left[B(H) B(H)^\intercal\right]^{-1} B(H) B(G)^\intercal$.
	
	Combining \eqref{lem Independence simple equality distances} and \eqref{lem Independence simple equality biases}, $\omega$ is in the null space of the matrix $P$ defined as:
	\begin{equation}
		P = \begin{bmatrix}
			I_{\card{E_G}} & A \\
			C & I_{\card{E_H}}
		\end{bmatrix}
	\end{equation}

	Then, proving that $P$ is invertible is sufficient to prove that $\omega = 0$.
	Introducing the Schur complement of the first block, $P$ is invertible if the matrix $S = I_{\card{E_H}} - CA$ is invertible.
	
	To prove that $S$ is invertible, \cref{lem: Inversibility algebric} will be used but to apply it, the double edges must be isolated. Let us first introduce the following edge sets.
	\begin{align}
		E_D &= E_G \cap E_H  & E_{G'} &= E_G \setminus E_D & E_{H'} &= E_H\setminus E_D
	\end{align}
	The set $E_D$ is the set of double edges of $\tilde \Gamma$, $E_{G'}$ and $E_{H'}$ form a decomposition of the set of simple edges $E_S$ of $\tilde \Gamma$: $E_{G'} \cup E_{H'} = E_S$ and $E_{G'} \cap E_{H'} = \emptyset$.
	We order both $E_G$ and $E_H$ starting with the common edges $E_D$.
	
	For a double edge, the corresponding rows in the Euclidean rigidity matrices are equal and the corresponding columns in the incidence matrices are opposed (since the arcs are opposed). Therefore, $\forall i \in \set{1, \dots, \card{E_D}}$:
	\begin{align}\label{eq: lem Independence double edge equalities}
		M(G, p)^\intercal e_i &= M(H, p)^\intercal e_i & B(G)^\intercal e_i &= -B(H)^\intercal e_i
	\end{align}
	Using once again the independence of the rows of $M_e(G,p)$ and $B(H)$, \eqref{eq: lem Independence double edge equalities} gives:
	\begin{align}
		A e_i &= e_i & C e_i &= -e_i & \tilde C e_i &= -e_i
	\end{align}
	Then, the matrices $A$, $C$ and $\tilde C$ can be decomposed as:
	\begin{align}
		A &= \begin{bmatrix}
			I_{\card{E_D}} & A_1\\
			0 & A_2
		\end{bmatrix} &
		C &= \begin{bmatrix}
			-I_{\card{E_D}} & C_1\\
			0 & C_2
		\end{bmatrix} &
		\tilde C &= \begin{bmatrix}
			-I_{\card{E_D}} & \tilde C_1\\
			0 & \tilde C_2
		\end{bmatrix}
	\end{align}
	where the blocks $A_1$, $A_2$, $C_1$, $C_2$, $\tilde C_1$ and $\tilde C_2$ do not have any particular structure.

	Therefore, the matrix $S$ is re-expressed as:
	\begin{equation}
		S = \begin{bmatrix}
			2I_{\card{E_D}} & A_1 - C_1 A_2\\
			0 & S_2
		\end{bmatrix}
	\end{equation}
	with $S_2 = I_{\card{E_{H'}}} - C_2 A_2$.
	
	Consequently, $P$ and $S$ are invertible if $S_2$ is invertible.
	
	Note that by construction, $C_2 = D(H', p)^{-1} \tilde C_2 D(G', p)$ where $H'$ and $G'$ denote the graph induced by the edge sets $E_{H'}$ and $E_{G'}$. Then, $S_2$ is invertible if the matrix $\tilde S_2 = D(H', p) - \tilde C_2 D(G', p) A_2$ is invertible.
	
	The sets $E_{H'}$ and $E_{G'}$ are disjoints. Furthermore, by considering every entry as a function of the coordinates, the entries of $\tilde C$ and $A$ belong to $\K$. Then, the entries of $\tilde C_2 D(G', p) A_2$ belong to $\L(E_{G'})$. Therefore by \cref{lem: Inversibility algebric}, $\tilde S_2$ is invertible.
	Thus $S_2$, $S$, $P$ are invertible and $\omega = 0$.
\end{proof}

With this lemma, \ref{it: Converse sense} can be proved.
\begin{proof}[Proof of \ref{it: Converse sense}]
	Let $G$ be a Euclidean rigid graph in $\R^d$ and $H$ be a Euclidean connected graph $H$. Let $\tilde \Gamma = G \cup H$ be their union. Let us prove that $\tilde \Gamma$ is rigid in $\R^d$.
	 
	Since $G$ is rigid in $\R^d$, $\exists G'$ a sub-graph of $G$ minimally rigid in $\R^d$. Since $H$ is connected, $\exists H'$ a sub-graph of $H$ which is a spanning tree. Applying \cref{lem: Independence union minimally rigid with forest}, their union $\tilde \Gamma' = G' \cup H'$ is independent in $\R^d$. Since minimally rigid Euclidean graphs have $S_e(n,d)$ edges and spanning trees have $n-1$ edges, $\tilde \Gamma'$ has $S_e(n,d) + n - 1 = S(n,d)$ edges. Since $\tilde \Gamma'$ is independent in $\R^d$ and has $S(n,d)$ edges, it is minimally rigid in $\R^d$. Thus, $\tilde \Gamma$ is rigid in $\R^d$.
\end{proof}

To prove \ref{it: Direct sense}, a decomposition of a rigid conic graph must be found. As already mentioned, a decomposition $(G, H)$ may not satisfy the conditions of \cref{the: Decomposition of conic rigid graph}. The existence of the particular decomposition is proved by induction without loss of generality for a minimally rigid conic graph. At each steps, a decomposition $(G, H)$ is proposed having $G$ minimally rigid and $H$ with fewer connected component. The induction will stop when $H$ is connected.

The proof relies on two lemmas: \cref{lem: Existence of minimially rigid Euclidean sub-graph} ensures that a decomposition $(G, H)$ having $G$ minimally rigid exists and \cref{lem: Induction decomposition} proves that if a decomposition $(G,H)$ exists with $G$ minimally rigid and $H$ not connected, then another decomposition $(G', H')$ exists with $G'$ minimally rigid and $H'$ having fewer connected components than $H$.

\begin{lemma}\label{lem: Existence of minimially rigid Euclidean sub-graph}
	Let $d \ge 2$. Let $\tilde \Gamma = (V, E_S, E_D)$ be a conic graph.
	
	If $\tilde \Gamma$ is minimally rigid in $R^d$ then, there exists $(G,H)$ a decomposition of $\tilde \Gamma$ such that $G$ is minimally rigid in $R^d$.
\end{lemma}
\begin{proof}
	Proving that there exists a set of edges $E_G$ containing the double edges $\tilde \Gamma$ such that the Euclidean graph $G = (V,E_G)$ is minimally rigid is sufficient. The graph $H$ would then be deduced by letting its edge set be $E_H = E_S \triangle E_G$ (symmetric difference). 
	
	Let $G_D = (V, E_D)$ and $G_S = (V, E_S)$ be the Euclidean graphs induced by the double edges and the simple edges of $\tilde \Gamma$. Consider a generic conic framework $(\Gamma, p)$, with the correct ordering of the arcs, the conic rigidity matrix of $(\Gamma, p)$ is:
	\begin{equation}
		M(\Gamma, p) = \begin{bmatrix}
			M_e(G_D, p) & B(G_D, p) \\
			M_e(G_D, p) & -B(G_D, p) \\
			M_e(G_S, p) & B(G_S, p)
		\end{bmatrix}
	\end{equation}

	First, let us prove that the Euclidean graph $G_D$ is independent. If $G_D$ was dependent, there would exist a non-null vector $\omega_D \in \Ker M_e(G_D)^\intercal$. Therefore the vector $\omega = \begin{pmatrix}\omega_D^\intercal & \omega_D^\intercal & 0_{\card{E_S}}^\intercal\end{pmatrix}^\intercal$ would be a non-null vector of $\Ker M(\Gamma, p)^\intercal$. This would contradict the fact that $\tilde \Gamma$ is independent, thus $G_D$ is independent.
	
	Now, let us prove that $E_D$ can be extended with edges from $E_S$ to create $E_D$ an independent set of $S_e(n,d)$ edges. Since $\tilde \Gamma$ is rigid, the Euclidean graph $G_1 = (V, E_D \cup E_S)$ is rigid. Indeed, if $G_1$ was not rigid, there would exist a non-trivial Euclidean velocity vector $v \in \Ker M_e(G_1, p)$ \cite{hendrickson_1992_conditions}. Then, $q = \begin{pmatrix} v^\intercal & 0_n^\intercal \end{pmatrix}^\intercal$ would be a non-trivial velocity vector admissible for $(\Gamma, p)$ which contradicts that $\tilde \Gamma$ is rigid.
	Since $G_1$ is rigid, the rank of the Euclidean rigidity matrix $M_e(G_1, p)$ is $S_e(n,d)$. Consequently, the independent set $E_D$ can be completed with edges from $E_S$ to create a set $E_G$ of $S_e(n,d)$ independent edges. The resulting graph $G = (V, E_G)$ is minimally rigid and generates the decomposition.
\end{proof}

To prove the second key lemma, the following result is required. It is a consequence of \cref{lem: Independence union minimally rigid with forest}.

\begin{lemma}\label{lem: Existence swap}
	Let $d\ge 2$, $G = (V, E_G)$ and $H = (V, E_H)$ be two Euclidean graphs, $\tilde \Gamma = G \cup H$ be their union and $U \subset V$ be a subset of vertices. Furthermore, assume that $G$ and $\tilde \Gamma$ are both minimally rigid in $\R^d$.
	
	If $\restr{H}{U}$, the restriction of $H$ to $U$, is connected and has a cycle, then there exist $uv \in E_H$ and $wz \in E_G$ with $u, v, w \in U$ and $z \notin U$ such that the graphs $G'$ obtained from $G$ by replacing the edge $wz$ by the edge $uv$ is also minimally rigid in $\R^d$.
\end{lemma}

For example, consider the decomposition $(G_2, H_2)$ introduced in \cref{fig: Decomposition conic graph} and the set $U = \set{1,2,3,4}$. Both $G_2$ and the union $G_2 \cup H_2$ are rigid in $\R^2$. Furthermore, the sub-graph $\restr{H_2}{U}$ is connected and has a cycle $\rC = \set{2,3,4}$. Therefore, \cref{lem: Existence swap} applies. With $u = 2$, $v = w = 3$ and $z = 5$: the graph $G'$ obtained is the graph $G_1$ of \cref{fig: Decomposition conic graph} which is indeed minimally rigid. \cref{fig: Illustration lemma existence swap} present more complex examples sharing the same decomposition $(G,H)$ but with three different sets $U$. 

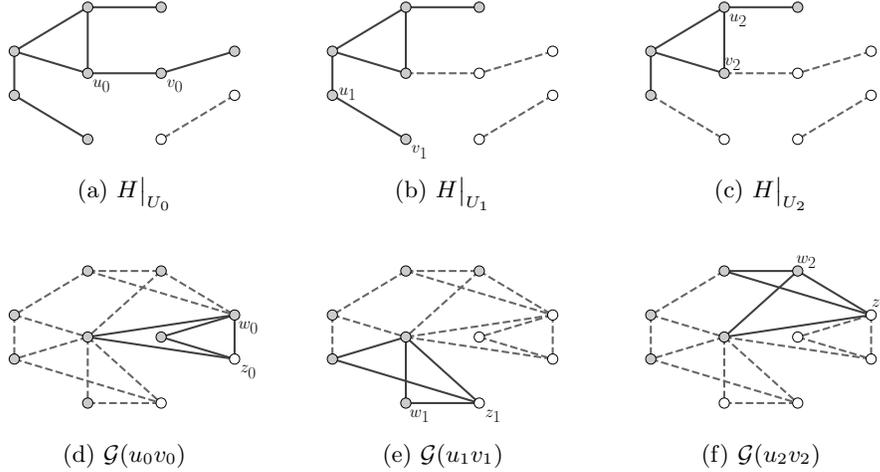
\begin{figure}
	\centering
	\null\hfill
	\begin{subfigure}[t]{0.30\textwidth}
		\centering
		\resizebox{\linewidth}{!}{\input{fig/decomposition_H_U0.pgf}}
		\caption{$\restr{H}{U_0}$}
	\end{subfigure}
	\hfill
	\begin{subfigure}[t]{0.30\textwidth}
		\centering
		\resizebox{\linewidth}{!}{\input{fig/decomposition_H_U1.pgf}}
		\caption{$\restr{H}{U_1}$}
	\end{subfigure}
	\hfill
	\begin{subfigure}[t]{0.30\textwidth}
		\centering
		\resizebox{\linewidth}{!}{\input{fig/decomposition_H_U2.pgf}}
		\caption{$\restr{H}{U_2}$}
	\end{subfigure}
	\hfill\null\\
	\null\hfill
	\begin{subfigure}[t]{0.30\textwidth}
		\centering
		\resizebox{\linewidth}{!}{\input{fig/decomposition_G_U0.pgf}}
		\caption{$\rG(u_0v_0)$}
	\end{subfigure}
	\hfill
	\begin{subfigure}[t]{0.30\textwidth}
		\centering
		\resizebox{\linewidth}{!}{\input{fig/decomposition_G_U1.pgf}}
		\caption{$\rG(u_1v_1)$}
	\end{subfigure}
	\hfill
	\begin{subfigure}[t]{0.30\textwidth}
		\centering
		\resizebox{\linewidth}{!}{\input{fig/decomposition_G_U2.pgf}}
		\caption{$\rG(u_2v_2)$}
	\end{subfigure}
	\hfill\null\\
	\caption{Illustrations of \cref{lem: Existence swap} in $\R^2$. Each column presents an example with a different set $U_i$ represented by the gray vertices (the others are white). The first row represents the graph $H$ with the edge $u_iv_i$ labeled and the second row represents the graph $G$ with the edge $w_iz_i$ labeled. The restriction of $H$ to $U_i$ is represented in solid lines while the rest of the graph is on dashed lines. For the graph $G$, the edge set $\rG(u_iv_i)$ (introduced in the proof of \cref{lem: Existence swap}) is represented in solid lines while the other edges are on dashed lines.}
	\label{fig: Illustration lemma existence swap}
\end{figure}

\begin{proof}[Proof of \cref{lem: Existence swap}]
	Let us consider a generic configuration $p$. In this proof, the edges are associated with their corresponding rows in the Euclidean rigidity matrix viewed as vectors, \ie{} the edge $uw$ is associated with the vector $M_e(\set{uw}, p)^\intercal \in \R^{dn}$.
	
	Since $G$ is minimally rigid, any edge can be uniquely written as a linear combination of $E_G$. The generation comes from rigidity ($E_G$ has the maximal rank) and the uniqueness from the independence of $E_G$. In that sense, $E_G$ is a basis of the space of edges. For any edge $uv$, an edge $wz$ from the basis $E_G$ is said to generate $uv$ if the coordinate associated with $wz$ in the decomposition of $uv$ on $E_G$ is not-null. For an edge $uv$, we denote $\rG(uv)$ the set of its generating edges. Graphically, $\rG(uv)$ is the smallest subset $F$ of $E_G$ such that the family $F \cup \set{uv}$ is dependent. For example, in \cref{fig: Illustration lemma existence swap} the three sets $\rG(u_iv_i)$ are highlighted in the graph $G$.
	For any $uv \notin E_G$, replacing any edge of $\rG(uv)$ by $uv$ creates a new basis and therefore a new minimally rigid graph. Therefore, it is sufficient to prove that there exists an edge $uv \in E_H$ with $u, v \in U$ such that $\rG(uv)$ contains an edge $wz$ with $w \in U$ and $z \notin U$.
	
	First, let us prove that there exists an edge $uv \in E_H$ with $u,v \in U$ such that $\rG(uv)$ contains an edge $wz$ with $z \notin U$ (and without any constraint on $w$). By contradiction, let us assume that for any $uv \in E_H$ with $u,v \in U$, every edge $wz \in \rG(uv)$ has $w,z \in U$; this means that the edges of $\restr{H}{U}$ are generated by edges of $\restr{G}{U}$.
	Let $\Gamma$ be a directed graph whose underlying conic graph is $\tilde \Gamma$ and let $\restr{\Gamma}{U}$ its restriction to $U$. By assumption, the graph $\restr{H}{U}$ is connected and has a cycle, let us assume \Wlog{} that $\restr{H}{U} = \restr{\hat H}{U} \cup \restr{H'}{U}$ with $\restr{\hat H}{U}$ being a spanning tree on $U$ and $\restr{H'}{U}$ forming the cycles. By assumption, the rows of the Euclidean rigidity matrix of $\restr{H}{U}$ are linear combinations of those of $\restr{G}{U}$. Therefore, there exist two matrices $A$ and $A'$ such that:
	\begin{align}
		M_e(\restr{\hat H}{U},p) &= AM_e(\restr{G}{U}, p) & M_e(\restr{H'}{U}, p) &= A' M_e(\restr{G}{U}, p)
	\end{align}
	Furthermore, since $\restr{\hat H}{U}$ is a spanning tree, the rows of $B(\restr{\hat H}{U}, p)$ generates any rows in a bias matrix associated with an edge between two vertices of $U$. This is a consequence of a well-known result in graph theory: in an incidence matrix the columns associated with a spanning tree form a basis of the columns, see \eg{} \cite{bollobas_1998_modern}. Therefore, there also exist two matrices $C$ and $C'$ such that:
	\begin{align}
		B(\restr{G}{U}, p) &= C B(\restr{\hat H}{U}, p) & B(\restr{H'}{U}, p) &= C' B(\restr{\hat H}{U}, p)
	\end{align}
	With these notations and the correct ordering of the arcs, the conic rigidity matrix of $\restr{\Gamma}{U}$ becomes:
	\begin{equation}
		M\left(\restr{\Gamma}{U},p\right) = \begin{bmatrix}
			M_e(\restr{G}{U}, p) & B(\restr{G}{U}, p) \\
			M_e(\restr{\hat H}{U}, p) & B(\restr{\hat H}{U}, p) \\
			M_e(\restr{H'}{U}, p) & B(\restr{H'}{U}, p)
		\end{bmatrix} = \begin{bmatrix}
			M_e(\restr{G}{U}, p) & C B(\restr{\hat H}{U}, p)  \\
			A M_e(\restr{G}{U}, p) & B(\restr{\hat H}{U}, p) \\
			A' M_e(\restr{G}{U}, p) & C' B(\restr{\hat H}{U}, p)
		\end{bmatrix} 
	\end{equation}
	In term of rank, this conic rigidity matrix is equivalent to the matrix:
	\begin{equation}
		\begin{bmatrix}
			(I - CA) M_e(\restr{G}{U}, p) & 0  \\
			A M_e(\restr{G}{U}, p) &  B(\restr{\hat H}{U}, p) \\
			A' M_e(\restr{G}{U}, p) & C' B(\restr{\hat H}{U}, p)
		\end{bmatrix}
	\end{equation}
	Since $\restr{G}{U}$ is independent and $\restr{\hat H}{U}$ is a spanning tree, as corollary of \cref{lem: Independence union minimally rigid with forest}, the matrix $I - CA$ is invertible (see the proof of \cref{lem: Independence union minimally rigid with forest}). Therefore the rows corresponding to $H'$ are linear combinations of the others. This is a contradiction since $\tilde \Gamma$ is independent. Thus, there exists an edge $uv \in E_H$ with $u, v \in U$ such that there exists an edge $wz \in \rG(uv)$ with $z \notin U$.
	
	We conclude by noting that the edges in $\rG(uv)$ must be connected and that $\rG(uv)$ must contain edges incident to $u$ and $v$. Consequently, there is an edge $wz \in \rG(uv)$ with $w\in U$ and $z \notin U$.
\end{proof}

\begin{lemma}\label{lem: Induction decomposition}
	Let $\tilde \Gamma$ be a conic graph and $(G, H)$ be a decomposition of $\tilde \Gamma$.
	
	If $\tilde \Gamma$ and $G$ are both minimally rigid in $\R^d$ and if $H$ has $k \ge 2$ connected components, then there exists $(G', H')$ another decomposition of $\tilde \Gamma$ such that $G'$ is minimally rigid in $\R^d$ and $H'$ has $k - 1$ connected components.
\end{lemma}
\begin{proof}
	As $\tilde \Gamma$ is minimally rigid, it has $S(n,d)$ edges. Similarly, as $G$ is minimally rigid, it has $S_e(n,d)$ edges. Therefore, the graph $H$ has $n-1$ edges. By a well-known result on graphs, see \eg{} \cite{bollobas_1998_modern}, either $H$ is a tree or it contains a cycle. Since $H$ has more than one component, it is not a tree and therefore it admits at least one cycle. Let $\rC$ be a cycle of $H$.
	
	Let us construct a new decomposition $(G', H')$ from $(G, H)$ in which $G'$ is minimally rigid and the connected component of $\rC$ in $H$ is connected with another connected component in $H'$. This new decomposition is obtained by exchanging some edges between $E_G$ and $E_H$ selected using \cref{lem: Existence swap} (possibly several times). Note that applying directly \cref{lem: Existence swap} with $U = \rC$ might not work since the exchange could create another cycle in $H$: that is the case for the cycle presented in \cref{fig: Illustration lemma existence swap}. In that example, the only exchangeable edge in $\restr{H}{\rC}$ is $u_2v_2$ since the other two edges are already in $G$ (and therefore cannot be exchanged). Unfortunately, $\rG(u_2v_2)$ contains only edges between vertices belonging to $U_0$, therefore any exchange would create a new cycle in $H$: for example exchanging $u_2v_2$ with $w_2z_2$ would give the decomposition $(G', H')$ in \cref{fig: Illustration of Lemma induction decomposition}.
 	
 	To avoid, this issue, let us first select several pairs of edges by the following induction. Start by letting $U_0$ be the connected component of $\rC$ in $H$ and $H_0$ be the restriction of $H$ to $U_0$. By \cref{lem: Existence swap}, there exist two edges $u_0v_0 \in E_H$ and $w_0z_0 \in E_G$ preserving the rigidity of $G$ with $u_0, v_0, w_0 \in U_0$ and $z_0 \notin U_0$. Then, while $u_iv_i$ is not in  a cycle of $H$ (not necessary $\rC$), let $U_{i+1}$ be the connected component of $\rC$ in $H_i\setminus\set{u_iv_i}$ (the graph $H_i$ without the edge $u_iv_i$) and $H_{i+1}$ be the restriction of $H$ to $U_{i+1}$. By \cref{lem: Existence swap}, there exist two edges $u_{i+1}v_{i+1} \in E_H$ and $w_{i+1}z_{i+1} \in E_G$ preserving the rigidity of $G$ with $u_{i+1}, v_{i+1}, w_{i+1} \in U_{i+1}$ and $z_{i+1} \notin U_{i+1}$. This procedure is illustrated on \cref{fig: Illustration lemma existence swap}: the sets $U_i$ have been chosen according to the induction.
 	
 	The induction finishes since the cardinal of $U_i$ decreases at each iteration and $U_i$ always contains at least the cycle $\rC$. Let $K$ denote the number of iterations. At the end of the induction, $K$ pairs of edges $(u_iv_i, w_iz_i)\in E_H\times E_G$ with $i \in \set{0, \dots, K-1}$ have been selected. In the example of \cref{fig: Illustration lemma existence swap}, $K = 3$.
 	
 	To construct the new decomposition, some pairs of edges among the $K$ selected are exchanged. They are also chosen by induction. The idea is to exchange at each step a pair $(u_iv_i, w_iz_i)$ whose edge $u_iv_i$ belongs to a cycle in $H$. If so, removing $u_iv_i$ from $H$ preserves its connected components and then, adding $w_{i}z_{i}$ to $H$ either connects two connected components or creates a new cycle but not in $H_i$: in some $H_{j}$ with $j < i$. Concretely, let $\sigma(i)$ denote the index of the pair chosen at step $i$. The first pair is the last selected pair: $\sigma(0) = K-1$. Then, while $z_{\sigma(i)} \in U_0$, the $(i+1)$-th index is $\sigma(i+1) = \max\set{j \mid z_{\sigma(i)} \in U_j}$. If $z_{\sigma(i)} \notin U_0$, the induction stops as $U_0$ has been connected with its complementary. If $z_{\sigma(i)} \in U_0$, then exchanging $u_{\sigma(i)}v_{\sigma(i)}$ with $w_{\sigma(i)}z_{\sigma(i)}$ creates a new cycle passing through $u_{\sigma(i+1)}v_{\sigma(i+1)}$. For example, consider again the decomposition introduced in \cref{fig: Illustration lemma existence swap}, two exchanges are performed. The first exchange between $u_2v_2$ and $w_2z_2$ creates a cycle passing through $u_0v_0$ since $z_2 \in U_0$ but $z_2 \notin U_1$. The second exchange between $u_0v_0$ and $w_0z_0$ connects the two components of $H$ and consequently stops the induction. The initial decomposition $(G,H)$ and the decompositions after each exchange are illustrated in \cref{fig: Illustration of Lemma induction decomposition}: $(G',H')$ is the decomposition after the first exchange and $(G'', H'')$ the decomposition after the second. By construction, each exchange preserves the minimal rigidity of $G$.
 \end{proof}
 	
 	\begin{figure}
 		\centering
 		\null\hfill
 		\begin{subfigure}[t]{0.30\textwidth}
 			\centering
 			\resizebox{\linewidth}{!}{\input{fig/decomposition_H0.pgf}}
 			\caption{$H$}
 		\end{subfigure}
 		\hfill
 		\begin{subfigure}[t]{0.30\textwidth}
 			\centering
 			\resizebox{\linewidth}{!}{\input{fig/decomposition_H1.pgf}}
 			\caption{$H'$}
 		\end{subfigure}
 		\hfill
 		\begin{subfigure}[t]{0.30\textwidth}
 			\centering
 			\resizebox{\linewidth}{!}{\input{fig/decomposition_H2.pgf}}
 			\caption{$H''$}
 		\end{subfigure}
 		\hfill\null\\
 		\null\hfill
 		\begin{subfigure}[t]{0.30\textwidth}
 			\centering
 			\resizebox{\linewidth}{!}{\input{fig/decomposition_G0.pgf}}
 			\caption{$G$}
 		\end{subfigure}
 		\hfill
 		\begin{subfigure}[t]{0.30\textwidth}
 			\centering
 			\resizebox{\linewidth}{!}{\input{fig/decomposition_G1.pgf}}
 			\caption{$G'$}
 		\end{subfigure}
 		\hfill
 		\begin{subfigure}[t]{0.30\textwidth}
 			\centering
 			\resizebox{\linewidth}{!}{\input{fig/decomposition_G2.pgf}}
 			\caption{$G''$}
 		\end{subfigure}
 		\hfill\null
 		\caption{Illustration of \cref{lem: Induction decomposition}.}
 		\label{fig: Illustration of Lemma induction decomposition}
 	\end{figure}

\ref{it: Direct sense} can now be proved.
\begin{proof}[Proof of \ref{it: Direct sense}]
	Let $\tilde \Gamma$ be, \Wlog, a minimally rigid conic graph (the redundant edges can be placed in any Euclidean graph).
	
	By \cref{lem: Existence of minimially rigid Euclidean sub-graph}, a decomposition $(G, H)$ of $\tilde \Gamma$ with $G$ minimally rigid exists.
	If the graph $H$ is connected then the proof is over. Otherwise, $H$ has $k \ge 2$ components, then using \cref{lem: Induction decomposition}, by a trivial induction, there exists a decomposition $(G', H')$ of $\tilde \Gamma$ with $G$ minimally rigid and $H$ connected.
\end{proof}

\section{Discussions}\label{sec: Discussions}

Two characterizations of generic infinitesimal rigidity have been introduced: the first for unidimensional conic frameworks and the second for multidimensional conic frameworks. Although different, they are based on a similar decoupling of the space and the bias variables: the positions are constrained by a Euclidean rigid graph while the biases are simply constrained by a connected graph. To reach the maximum rank and to ensure infinitesimal rigidity, those two graphs must generate complementary constraints, \ie{} their constraints must be independent. In the unidimensional case, this comes from the distinction between increasing and decreasing arcs while in the multidimensional case, it is a consequence of the linear independence between distances and coordinates.

Conic frameworks are truly a new concept. To first state the obvious, a $d$-dimensional conic framework is not simply a Euclidean framework of dimension $d+1$. If so, since rigid conic graphs are characterized by a condition on their underlying rigid Euclidean graphs, there would be an induction construction of rigid Euclidean graphs in any dimension. Unfortunately, there is no known characterization of Euclidean rigidity in dimension $d > 2$. The main difference between a $d$-dimensional conic graph and a $(d+1)$-dimensional Euclidean graph comes from the additional entries in rigidity matrices. For conic graphs, these entries are distances which are generically linearly independent. In contrast, for $(d+1)$-dimensional Euclidean graph the additional entries are dependent: \eg{} if the additional spatial variable is $\theta$, the additional entries are in the form of $\theta_u - \theta_w$ and for example: $\theta_1 - \theta_2$, $\theta_2 - \theta_3$ and $\theta_3 - \theta_1$ are clearly dependent. Furthermore, conic frameworks are also different from hyperbolic frameworks described in \cite{gortler_2014_generic}: the absence of absolute value around the bias difference in the pseudo-range equation makes the constraint asymmetrical. This also allows the graph to have pair of vertices connected by two arcs. Interestingly enough, in multidimensional cases, this asymmetry has no impact on infinitesimal rigidity: two frameworks with equivalent graphs have the same infinitesimal rigidity. However, if they are flexible, their flexing may be different as illustrated by the conic frameworks of \cref{fig: Conic frameworks}.

One of the greatest interests of the conic paradigm is the preservation of flight formation. First as mentioned in introduction, its reduces the minimal number of pseudo-range constraints required with respect to a classical SDS TWR method. Remind that SDS TWR requires at least $2S_e(n,d) \sim_n 2dn$ pseudo-range constraints whereas the conic method only requires $S(n,d) \sim_n (d+1)n$. Consequently, when used in the plane, SDS TWR method performs at least about $25\%$ of redundant measurements and about $33\%$ in 3D space. Furthermore, from an implementation point of view, the pseudo-range approach has another advantage: some agents can be only broadcasting. In the context of formation persistence \cite{hendrickx_2007_directed}, every constraint is maintained by only one agent, called the \emph{follower}. If the graph has some good properties the whole formation is preserved. This technique greatly simplifies the control. With the SDS TWR approach, an agent having several followers has to interact with every one of them to compute the distances. When the number of followers increases, the update rate necessarily decreases, which may induce a loss of precision. With the conic method in contrast, a agent having several followers may have no interaction with them: he could simply broadcast its position and bias, then, each follower could compute the pseudo-range without any feedback. This approach allows significant scale-up in the system as the number of followers would not be limited by the channel capacity.

From a computational point of view, testing infinitesimal rigidity is simple for unidimensional frameworks. It requires to divide the arcs according to their orientation and then apply twice a connectivity test algorithm to both $G_+$ and $G_-$. These three steps may be realized with a complexity of $O(\card{E})$, with a bread search first algorithm for example. In the multidimensional case, infinitesimal rigidity is a property of the conic graph. Rigidity of conic graphs relies on their underlying rigid Euclidean graphs and underlying connected graphs. When $d = 2$, several efficient algorithms have been developed to test rigidity of Euclidean graphs with a complexity of $O(n^2)$ \cite{imai1_985_combinatorial, hendrickson_1992_conditions, gabow_1992_forests, jacobs_1997_algorithm}. Unfortunately, the number of spanning trees in a graph may increase exponentially and therefore testing every possible decomposition would be inefficient. To bypass this issue, the construction of the decomposition of the conic graph presented could be implemented. The construction of the initial Euclidean rigid graph may be realized using the pebble-game algorithm of \cite{jacobs_1997_algorithm} and the construction of the sets $\rG(uv)$ using the algorithm introduced in \cite{berg_2003_algorithms} to find the redundantly rigid components of a graph. This will be the focus of further works. Randomized algorithm may also be considered to test rigidity of conic graphs. This approach was already proposed for Euclidean frameworks \cite{hendrickson_1992_conditions}. The idea is to randomly generate a configuration and computing the rank of its rigidity matrix. For example the algorithm introduced in \cite{gortler_characterizing_2010} may be adapted to test conic infinitesimal rigidity.

Finally, the limits of infinitesimal rigidity must be underlined. Infinitesimal rigidity considers only instantaneous velocities of conic frameworks. As a consequence, it ensures that no smooth deformation of the framework exists. However, it is weaker than rigidity which considers every flexing (smooth or not). For Euclidean frameworks, two stronger results exist. First, infinitesimal rigidity implies rigidity and second, those two concepts are equivalent for generic Euclidean frameworks, see \eg{} \cite{asimow_rigidity_1978}. However, the proofs of these two results use elements from differential geometry that are beyond the scope of this paper. Therefore, they are, for now, only conjectured for conic frameworks.
\begin{conjecture}
	Let $(\Gamma, p)$ be a conic framework.
	
	If $(\Gamma, p)$ is minimally rigid, then $(\Gamma, p)$ is rigid.
\end{conjecture}
\begin{conjecture}
	Let $(\Gamma, p)$ be a generic conic framework.
	
	$(\Gamma, p)$ is minimally rigid if and only if $(\Gamma, p)$ is rigid.
\end{conjecture}

\appendix
\section{Proof of algebraic results}\label{app: Proof of algebra results} 

This appendix provides the proofs of the algebraic results. It has been isolated because it uses materials very different than those introduced in the body of the paper. For general concepts of field theory (field extension, extension order, etc.) please refer to \eg{} \cite{roman_2005_field}.

\subsection{Proof of \cref{lem: Field extension}}\label{sapp: Proof of lemma field extension}
\cref{lem: Field extension} is a particular application of the following result.

\begin{lemma}\label{lem: Field extension multidimensional rigidity}
	Let $\K = \Q(X_1, \dots, X_N)$ be the field of fractions in $N$ variables with coefficients in $\Q$ and $(R_1, \dots, R_m)$ be a family of functions such that:
	\begin{enumerate}[label=\textbf{H.\arabic*}, ref=Hypothesis H.\arabic*]
		\item\label{it: H1} $\forall i \in \{1, \dots, m\}$, $R_i^2 \in \K$;
		\item\label{it: H2} $\forall i \in \{1, \dots, m\}$, $R_i^2 \notin \K^{(2)}$, with $\K^{(2)}=\{P^2 \mid P\in \K\}$;
		\item\label{it: H3} $\forall I \in \rP(\{1, \dots, m\})\setminus\{\emptyset\}$, $R_I = \prod_{i \in I} R_i \notin \K$.
	\end{enumerate}
	Then $\L = \K\left[R_1, \dots, R_m\right]$ is a field and $\L/\K$ is a field extension of order $2^m$.
\end{lemma}
\begin{proof}
	The proof is realized by induction over $m$. The property to prove is $\mathbf{P}(k)$: \guillemets{For any $R_1, \dots, R_k$ satisfying the three hypotheses, $\L = \K\left[R_1, \dots, R_k\right]$ is a field and $\L/\K$ is a field extension of order $2^k$.}
	
	\emph{Initialization.} For $k = 1$, let $R$ satisfy the three hypotheses. To prove that $\K[R]$ is a field, proving that every non-null element has an inverse is sufficient. Let $P \in \K[R]$, $P \ne 0$. Since $R^2 \in \K$, there exists $(A,B) \in \K^2$ with $(A,B) \ne (0,0)$ such that $P = A + BR$. If $B = 0$, $P = A \in \K$ therefore $P$ is invertible. If $B \neq 0$, using \ref{it: H2}, $R^2 \ne A^2/B^2$, therefore $A^2 - B^2 R^2 \neq 0$. Then, $(A - BR)/(A^2 - B^2R^2) \in \K[R]$ is the inverse of $P$. The extension is of order $2$ by \ref{it: H1} and \ref{it: H3}. Thus, $\mathbf{P}(1)$ is true.
	
	\emph{Induction step.} Let assume $\mathbf{P}(k)$ for $k \ge 1$ and prove $\mathbf{P}(k+1)$. Let $R_1, \dots, R_{k+1}$ be $k+1$ functions satisfying the three hypotheses. We denote $\L_k = \K[R_1, \dots, R_k]$. First, let us prove that $\L_{k+1}$ is a field. Proving that $R_{k+1}^2 \notin \L_{k}^{(2)}$ is sufficient since then, with the same arguments as for the initialization every non-null element of $\L_{k+1}$ would have an inverse.
	
	Let us assume by contradiction that $R_{k+1}^2 \in \L_{k}^{(2)}$. By induction hypothesis, $\L_{k} = \L_{k-1}[R_k]$. Therefore, there exist $A, B \in \L_{k-1}$ such that:
	\begin{equation}
		R_{k+1}^2 = \left(A + B R_{k}\right)^2 = A^2 + B^2 R_{k}^2 + 2AB R_k
	\end{equation}
	If $AB \neq 0$, then $R_k \in \L_{k-1}$ which contradicts the induction hypothesis. Then, necessarily $A$ or $B$ is null. If $B = 0$, then $R_{k+1}^2 = A^2 \in \L_{k-1}^{(2)}$ This also contradicts the induction hypothesis when considering the $k$ functions $R_1,\dots, R_{k-1}, R_{k+1}$. Therefore $A = 0$. If $A = 0$ then, $R_{k+1}^2 = B^2 R_k^2$ and $(R_{k+1}R_k)^2 = (B R_k^2)^2 \in \L_{k-1}^{(2)}$. Similarly, this also contradicts the induction hypothesis when considering the $k$ functions $R_1, \dots, R_{k-1}, R_k R_{k+1}$ (which satisfies the three hypotheses). Theses contradictions give that $R_{k+1}^2 \notin \L_{k}^{(2)}$ and thus, $\L_{k+1}$ is a field.
	
	To prove the order, let us use the induction hypothesis:
	\begin{equation}
		[\L_{k+1}: \K] = [\L_{k+1}: \L_{k}][\L_{k}: \K] = 2^k [\L_{k+1}: \L_k]
	\end{equation}
	Since $R_{k+1} \notin \L_k$ and $R_{k+1}^2 \in \L_k$,  $[\L_{k+1}:\L_k] = 2$ and $[\L_{k+1}: \K] = 2^{k+1}$. $\mathbf{P}(k+1)$ is true.
\end{proof}

\begin{proof}[Proof of \cref{lem: Field extension}]
	Let $E \subset \{uw \mid 1 \le u < w \le n\}$ be a set of edges and $m = \card{E}$.
	The set of $m$ distance functions $D_{u,w}$ do satisfy the three conditions of \cref{lem: Field extension multidimensional rigidity} when $d \ge 2$.
		
	Note however that when $d = 1$, the distance functions do not satisfy \ref{it: H2} of \cref{lem: Field extension multidimensional rigidity}.
\end{proof}

\subsection{Proof of \cref{lem: Inversibility algebric}}\label{sapp: Proof of lemma Inversibility algebric} 
\begin{proof}
	Let $E_H$ and $E_G$ be two disjoint sets of edges, $A$ be a square matrix of order $\card{E_H}$ whose entries are functions in $\L(E_G)$, and $p$ be a generic configuration.
	
	Let us denote $E = E_G \cup E_H$ and $B : p \mapsto D(H,p) + A(p)$.
	The entries of $B$ belong to $\L(E)$. Therefore, its determinant also belongs to $\L(E)$. The decomposition of $\det B$ on the natural basis of $\L(E)$ gives:
	\begin{equation}
		\det B = \sum_{F\in\rP(E)} \alpha_F \prod_{uw \in F} D_{u,w}
	\end{equation}
	where the $\alpha_F$ are the coordinates of $\det B$ on the natural basis of $\L(E)$.
	
	Furthermore, since $E_H$ and $E_G$ are disjoint and since the entries of $A$ are in $\L(E_G)$, by an easy induction on the cardinal of $E_H$, the coordinates associated to $E_H$ is $1$.
	
	Therefore $\det B \neq 0$. Then, by \cref{lem: Extension generic nullity}, $\det B(p)$ is not null and $B(p)$ is invertible.
\end{proof}

\section*{Acknowledgments}
We would like to thank Charlotte Hardouin from the \emph{Institut de Mathématiques de Toulouse} for her support on algebraic concepts, especially for the proof of \cref{lem: Field extension multidimensional rigidity}. We would also like to thank Christian Commault from the \emph{GIPSA-Lab} for his thorough reading of the paper.

\bibliographystyle{siamplain}
\bibliography{references}
\end{document}

%% file: shared.tex
\usepackage{lipsum}
\usepackage{amsfonts}
\usepackage{graphicx}
\usepackage{epstopdf}
\usepackage{algorithmic}
\ifpdf
  \DeclareGraphicsExtensions{.eps,.pdf,.png,.jpg}
\else
  \DeclareGraphicsExtensions{.eps}
\fi

\newsiamremark{remark}{Remark}
\newsiamremark{hypothesis}{Hypothesis}
\crefname{hypothesis}{Hypothesis}{Hypotheses}
\newsiamthm{claim}{Claim}

\headers{Conic frameworks infinitesimal rigidity}{C.~Cros, P-O.~Amblard, C.~Prieur, and J-F.~Da~Rocha}

\title{Conic frameworks infinitesimal rigidity\thanks{Submitted to the editors DATE.
}}

\author{Colin Cros\footnotemark[2]{} \footnotemark[3]
\and Pierre-Olivier Amblard\thanks{Univ. Grenoble Alpes, CNRS, Grenoble-INP, GIPSA-Lab, 38000 Grenoble, FRANCE
	(\email{colin.cros@gipsa-lab.fr}, \email{pierre-olivier.amblard@cnrs.fr}, \email{christophe.prieur@gipsa-lab.fr}).}
\and Christophe Prieur\footnotemark[2]
\and Jean-François Da Rocha\thanks{Telespazio FRANCE, 31100 Toulouse, FRANCE (\email{jeanfrancois.darocha@telespazio.com}).}}

\usepackage{amsopn}

%% file: header.tex
\usepackage{tikz}
\usepackage{pgf}
\usepackage[inline]{enumitem}

\usepackage{amsmath}
\usepackage{amssymb}
\newsiamthm{conjecture}{Conjecture}

\usepackage{subcaption}
\graphicspath{{fig/}}

\newcommand{\guillemets}[1]{``#1''}

\newcommand{\diagvect}{\mbox{diag}}

\newcommand{\set}[1]{\left\{#1\right\}}
\newcommand{\diff}{\text{d}}

\newcommand{\R}{\mathbb{R}}
\newcommand{\K}{\mathbb{K}}
\renewcommand{\L}{\mathbb{L}}
\newcommand{\Q}{\mathbb{Q}}

\newcommand{\sign}{\mbox{sign}}

\DeclareMathOperator{\rank}{rank} 
\DeclareMathOperator{\Ker}{Ker}

\newcommand{\rC}{\mathcal{C}}

\newcommand{\rG}{\mathcal{G}}

\newcommand{\rP}{\mathcal{P}}

\newcommand{\norm}[1]{\left\lVert#1\right\rVert}
\newcommand{\card}[1]{\left\lvert#1\right\rvert}

\newcommand{\ie}{\emph{i.e.},}
\newcommand{\eg}{\emph{e.g.},}
\newcommand{\Wlog}{without loss of generality}

\newcommand\restr[2]{{% we make the whole thing an ordinary symbol
		\left.\kern-\nulldelimiterspace % automatically resize the bar with \right
		#1 % the function
		\vphantom{\big|} % pretend it's a little taller at normal size
		\right|_{#2} % this is the delimiter
}}

\DeclareUnicodeCharacter{2212}{-}

%% file: fig/conic_graph_hyperbolic_motion.pgf
%% Creator: Matplotlib, PGF backend
%%
%% To include the figure in your LaTeX document, write
%%   \input{<filename>.pgf}
%%
%% Make sure the required packages are loaded in your preamble
%%   \usepackage{pgf}
%%
%% Figures using additional raster images can only be included by \input if
%% they are in the same directory as the main LaTeX file. For loading figures
%% from other directories you can use the `import` package
%%   \usepackage{import}
%%
%% and then include the figures with
%%   \import{<path to file>}{<filename>.pgf}
%%
%% Matplotlib used the following preamble
%%   \usepackage{fontspec}
%%   \setmainfont{DejaVuSerif.ttf}[Path=\detokenize{C:/Users/ccros/Anaconda3/Lib/site-packages/matplotlib/mpl-data/fonts/ttf/}]
%%   \setsansfont{DejaVuSans.ttf}[Path=\detokenize{C:/Users/ccros/Anaconda3/Lib/site-packages/matplotlib/mpl-data/fonts/ttf/}]
%%   \setmonofont{DejaVuSansMono.ttf}[Path=\detokenize{C:/Users/ccros/Anaconda3/Lib/site-packages/matplotlib/mpl-data/fonts/ttf/}]
%%
\begingroup%
\makeatletter%
\begin{pgfpicture}%
\pgfpathrectangle{\pgfpointorigin}{\pgfqpoint{4.000000in}{4.000000in}}%
\pgfusepath{use as bounding box, clip}%
\begin{pgfscope}%
\pgfsetbuttcap%
\pgfsetmiterjoin%
\pgfsetlinewidth{0.000000pt}%
\definecolor{currentstroke}{rgb}{1.000000,1.000000,1.000000}%
\pgfsetstrokecolor{currentstroke}%
\pgfsetstrokeopacity{0.000000}%
\pgfsetdash{}{0pt}%
\pgfpathmoveto{\pgfqpoint{0.000000in}{0.000000in}}%
\pgfpathlineto{\pgfqpoint{4.000000in}{0.000000in}}%
\pgfpathlineto{\pgfqpoint{4.000000in}{4.000000in}}%
\pgfpathlineto{\pgfqpoint{0.000000in}{4.000000in}}%
\pgfpathclose%
\pgfusepath{}%
\end{pgfscope}%
\begin{pgfscope}%
\pgfpathrectangle{\pgfqpoint{0.000000in}{0.000000in}}{\pgfqpoint{4.000000in}{4.000000in}}%
\pgfusepath{clip}%
\pgfsetbuttcap%
\pgfsetroundjoin%
\pgfsetlinewidth{1.505625pt}%
\definecolor{currentstroke}{rgb}{0.501961,0.501961,0.501961}%
\pgfsetstrokecolor{currentstroke}%
\pgfsetdash{{5.550000pt}{2.400000pt}}{0.000000pt}%
\pgfpathmoveto{\pgfqpoint{3.873776in}{-0.013889in}}%
\pgfpathlineto{\pgfqpoint{3.819403in}{0.037252in}}%
\pgfpathlineto{\pgfqpoint{3.742029in}{0.109616in}}%
\pgfpathlineto{\pgfqpoint{3.667500in}{0.178892in}}%
\pgfpathlineto{\pgfqpoint{3.595693in}{0.245196in}}%
\pgfpathlineto{\pgfqpoint{3.526491in}{0.308634in}}%
\pgfpathlineto{\pgfqpoint{3.459782in}{0.369311in}}%
\pgfpathlineto{\pgfqpoint{3.395456in}{0.427325in}}%
\pgfpathlineto{\pgfqpoint{3.333409in}{0.482771in}}%
\pgfpathlineto{\pgfqpoint{3.273539in}{0.535741in}}%
\pgfpathlineto{\pgfqpoint{3.215748in}{0.586319in}}%
\pgfpathlineto{\pgfqpoint{3.159941in}{0.634590in}}%
\pgfpathlineto{\pgfqpoint{3.106029in}{0.680631in}}%
\pgfpathlineto{\pgfqpoint{3.053923in}{0.724518in}}%
\pgfpathlineto{\pgfqpoint{3.003537in}{0.766323in}}%
\pgfpathlineto{\pgfqpoint{2.954790in}{0.806113in}}%
\pgfpathlineto{\pgfqpoint{2.907602in}{0.843954in}}%
\pgfpathlineto{\pgfqpoint{2.861895in}{0.879908in}}%
\pgfpathlineto{\pgfqpoint{2.817596in}{0.914033in}}%
\pgfpathlineto{\pgfqpoint{2.774632in}{0.946385in}}%
\pgfpathlineto{\pgfqpoint{2.732932in}{0.977016in}}%
\pgfpathlineto{\pgfqpoint{2.692429in}{1.005977in}}%
\pgfpathlineto{\pgfqpoint{2.653057in}{1.033316in}}%
\pgfpathlineto{\pgfqpoint{2.614751in}{1.059076in}}%
\pgfpathlineto{\pgfqpoint{2.577448in}{1.083299in}}%
\pgfpathlineto{\pgfqpoint{2.541089in}{1.106026in}}%
\pgfpathlineto{\pgfqpoint{2.505613in}{1.127294in}}%
\pgfpathlineto{\pgfqpoint{2.470962in}{1.147137in}}%
\pgfpathlineto{\pgfqpoint{2.437080in}{1.165587in}}%
\pgfpathlineto{\pgfqpoint{2.403912in}{1.182674in}}%
\pgfpathlineto{\pgfqpoint{2.371404in}{1.198428in}}%
\pgfpathlineto{\pgfqpoint{2.339501in}{1.212872in}}%
\pgfpathlineto{\pgfqpoint{2.308154in}{1.226032in}}%
\pgfpathlineto{\pgfqpoint{2.277309in}{1.237927in}}%
\pgfpathlineto{\pgfqpoint{2.246917in}{1.248579in}}%
\pgfpathlineto{\pgfqpoint{2.216928in}{1.258004in}}%
\pgfpathlineto{\pgfqpoint{2.187293in}{1.266217in}}%
\pgfpathlineto{\pgfqpoint{2.157964in}{1.273232in}}%
\pgfpathlineto{\pgfqpoint{2.128893in}{1.279060in}}%
\pgfpathlineto{\pgfqpoint{2.100033in}{1.283712in}}%
\pgfpathlineto{\pgfqpoint{2.071336in}{1.287194in}}%
\pgfpathlineto{\pgfqpoint{2.042755in}{1.289512in}}%
\pgfpathlineto{\pgfqpoint{2.014244in}{1.290670in}}%
\pgfpathlineto{\pgfqpoint{1.985756in}{1.290670in}}%
\pgfpathlineto{\pgfqpoint{1.957245in}{1.289512in}}%
\pgfpathlineto{\pgfqpoint{1.928664in}{1.287194in}}%
\pgfpathlineto{\pgfqpoint{1.899967in}{1.283712in}}%
\pgfpathlineto{\pgfqpoint{1.871107in}{1.279060in}}%
\pgfpathlineto{\pgfqpoint{1.842036in}{1.273232in}}%
\pgfpathlineto{\pgfqpoint{1.812707in}{1.266217in}}%
\pgfpathlineto{\pgfqpoint{1.783072in}{1.258004in}}%
\pgfpathlineto{\pgfqpoint{1.753083in}{1.248579in}}%
\pgfpathlineto{\pgfqpoint{1.722691in}{1.237927in}}%
\pgfpathlineto{\pgfqpoint{1.691846in}{1.226032in}}%
\pgfpathlineto{\pgfqpoint{1.660499in}{1.212872in}}%
\pgfpathlineto{\pgfqpoint{1.628596in}{1.198428in}}%
\pgfpathlineto{\pgfqpoint{1.596088in}{1.182674in}}%
\pgfpathlineto{\pgfqpoint{1.562920in}{1.165587in}}%
\pgfpathlineto{\pgfqpoint{1.529038in}{1.147137in}}%
\pgfpathlineto{\pgfqpoint{1.494387in}{1.127294in}}%
\pgfpathlineto{\pgfqpoint{1.458911in}{1.106026in}}%
\pgfpathlineto{\pgfqpoint{1.422552in}{1.083299in}}%
\pgfpathlineto{\pgfqpoint{1.385249in}{1.059076in}}%
\pgfpathlineto{\pgfqpoint{1.346943in}{1.033316in}}%
\pgfpathlineto{\pgfqpoint{1.307571in}{1.005977in}}%
\pgfpathlineto{\pgfqpoint{1.267068in}{0.977016in}}%
\pgfpathlineto{\pgfqpoint{1.225368in}{0.946385in}}%
\pgfpathlineto{\pgfqpoint{1.182404in}{0.914033in}}%
\pgfpathlineto{\pgfqpoint{1.138105in}{0.879908in}}%
\pgfpathlineto{\pgfqpoint{1.092398in}{0.843954in}}%
\pgfpathlineto{\pgfqpoint{1.045210in}{0.806113in}}%
\pgfpathlineto{\pgfqpoint{0.996463in}{0.766323in}}%
\pgfpathlineto{\pgfqpoint{0.946077in}{0.724518in}}%
\pgfpathlineto{\pgfqpoint{0.893971in}{0.680631in}}%
\pgfpathlineto{\pgfqpoint{0.840059in}{0.634590in}}%
\pgfpathlineto{\pgfqpoint{0.784252in}{0.586319in}}%
\pgfpathlineto{\pgfqpoint{0.726461in}{0.535741in}}%
\pgfpathlineto{\pgfqpoint{0.666591in}{0.482771in}}%
\pgfpathlineto{\pgfqpoint{0.604544in}{0.427325in}}%
\pgfpathlineto{\pgfqpoint{0.540218in}{0.369311in}}%
\pgfpathlineto{\pgfqpoint{0.473509in}{0.308634in}}%
\pgfpathlineto{\pgfqpoint{0.404307in}{0.245196in}}%
\pgfpathlineto{\pgfqpoint{0.332500in}{0.178892in}}%
\pgfpathlineto{\pgfqpoint{0.257971in}{0.109616in}}%
\pgfpathlineto{\pgfqpoint{0.180597in}{0.037252in}}%
\pgfpathlineto{\pgfqpoint{0.126224in}{-0.013889in}}%
\pgfusepath{stroke}%
\end{pgfscope}%
\begin{pgfscope}%
\pgfpathrectangle{\pgfqpoint{0.000000in}{0.000000in}}{\pgfqpoint{4.000000in}{4.000000in}}%
\pgfusepath{clip}%
\pgfsetbuttcap%
\pgfsetroundjoin%
\definecolor{currentfill}{rgb}{0.800000,0.800000,0.800000}%
\pgfsetfillcolor{currentfill}%
\pgfsetlinewidth{1.003750pt}%
\definecolor{currentstroke}{rgb}{0.000000,0.000000,0.000000}%
\pgfsetstrokecolor{currentstroke}%
\pgfsetdash{}{0pt}%
\pgfsys@defobject{currentmarker}{\pgfqpoint{-0.104167in}{-0.104167in}}{\pgfqpoint{0.104167in}{0.104167in}}{%
\pgfpathmoveto{\pgfqpoint{0.000000in}{-0.104167in}}%
\pgfpathcurveto{\pgfqpoint{0.027625in}{-0.104167in}}{\pgfqpoint{0.054123in}{-0.093191in}}{\pgfqpoint{0.073657in}{-0.073657in}}%
\pgfpathcurveto{\pgfqpoint{0.093191in}{-0.054123in}}{\pgfqpoint{0.104167in}{-0.027625in}}{\pgfqpoint{0.104167in}{0.000000in}}%
\pgfpathcurveto{\pgfqpoint{0.104167in}{0.027625in}}{\pgfqpoint{0.093191in}{0.054123in}}{\pgfqpoint{0.073657in}{0.073657in}}%
\pgfpathcurveto{\pgfqpoint{0.054123in}{0.093191in}}{\pgfqpoint{0.027625in}{0.104167in}}{\pgfqpoint{0.000000in}{0.104167in}}%
\pgfpathcurveto{\pgfqpoint{-0.027625in}{0.104167in}}{\pgfqpoint{-0.054123in}{0.093191in}}{\pgfqpoint{-0.073657in}{0.073657in}}%
\pgfpathcurveto{\pgfqpoint{-0.093191in}{0.054123in}}{\pgfqpoint{-0.104167in}{0.027625in}}{\pgfqpoint{-0.104167in}{0.000000in}}%
\pgfpathcurveto{\pgfqpoint{-0.104167in}{-0.027625in}}{\pgfqpoint{-0.093191in}{-0.054123in}}{\pgfqpoint{-0.073657in}{-0.073657in}}%
\pgfpathcurveto{\pgfqpoint{-0.054123in}{-0.093191in}}{\pgfqpoint{-0.027625in}{-0.104167in}}{\pgfqpoint{0.000000in}{-0.104167in}}%
\pgfpathclose%
\pgfusepath{stroke,fill}%
}%
\begin{pgfscope}%
\pgfsys@transformshift{2.000000in}{1.000000in}%
\pgfsys@useobject{currentmarker}{}%
\end{pgfscope}%
\end{pgfscope}%
\begin{pgfscope}%
\pgfpathrectangle{\pgfqpoint{0.000000in}{0.000000in}}{\pgfqpoint{4.000000in}{4.000000in}}%
\pgfusepath{clip}%
\pgfsetbuttcap%
\pgfsetroundjoin%
\definecolor{currentfill}{rgb}{0.800000,0.800000,0.800000}%
\pgfsetfillcolor{currentfill}%
\pgfsetlinewidth{1.003750pt}%
\definecolor{currentstroke}{rgb}{0.000000,0.000000,0.000000}%
\pgfsetstrokecolor{currentstroke}%
\pgfsetdash{}{0pt}%
\pgfsys@defobject{currentmarker}{\pgfqpoint{-0.104167in}{-0.104167in}}{\pgfqpoint{0.104167in}{0.104167in}}{%
\pgfpathmoveto{\pgfqpoint{0.000000in}{-0.104167in}}%
\pgfpathcurveto{\pgfqpoint{0.027625in}{-0.104167in}}{\pgfqpoint{0.054123in}{-0.093191in}}{\pgfqpoint{0.073657in}{-0.073657in}}%
\pgfpathcurveto{\pgfqpoint{0.093191in}{-0.054123in}}{\pgfqpoint{0.104167in}{-0.027625in}}{\pgfqpoint{0.104167in}{0.000000in}}%
\pgfpathcurveto{\pgfqpoint{0.104167in}{0.027625in}}{\pgfqpoint{0.093191in}{0.054123in}}{\pgfqpoint{0.073657in}{0.073657in}}%
\pgfpathcurveto{\pgfqpoint{0.054123in}{0.093191in}}{\pgfqpoint{0.027625in}{0.104167in}}{\pgfqpoint{0.000000in}{0.104167in}}%
\pgfpathcurveto{\pgfqpoint{-0.027625in}{0.104167in}}{\pgfqpoint{-0.054123in}{0.093191in}}{\pgfqpoint{-0.073657in}{0.073657in}}%
\pgfpathcurveto{\pgfqpoint{-0.093191in}{0.054123in}}{\pgfqpoint{-0.104167in}{0.027625in}}{\pgfqpoint{-0.104167in}{0.000000in}}%
\pgfpathcurveto{\pgfqpoint{-0.104167in}{-0.027625in}}{\pgfqpoint{-0.093191in}{-0.054123in}}{\pgfqpoint{-0.073657in}{-0.073657in}}%
\pgfpathcurveto{\pgfqpoint{-0.054123in}{-0.093191in}}{\pgfqpoint{-0.027625in}{-0.104167in}}{\pgfqpoint{0.000000in}{-0.104167in}}%
\pgfpathclose%
\pgfusepath{stroke,fill}%
}%
\begin{pgfscope}%
\pgfsys@transformshift{2.000000in}{3.000000in}%
\pgfsys@useobject{currentmarker}{}%
\end{pgfscope}%
\end{pgfscope}%
\begin{pgfscope}%
\pgfpathrectangle{\pgfqpoint{0.000000in}{0.000000in}}{\pgfqpoint{4.000000in}{4.000000in}}%
\pgfusepath{clip}%
\pgfsetbuttcap%
\pgfsetroundjoin%
\definecolor{currentfill}{rgb}{0.800000,0.800000,0.800000}%
\pgfsetfillcolor{currentfill}%
\pgfsetlinewidth{1.003750pt}%
\definecolor{currentstroke}{rgb}{0.000000,0.000000,0.000000}%
\pgfsetstrokecolor{currentstroke}%
\pgfsetdash{}{0pt}%
\pgfsys@defobject{currentmarker}{\pgfqpoint{-0.104167in}{-0.104167in}}{\pgfqpoint{0.104167in}{0.104167in}}{%
\pgfpathmoveto{\pgfqpoint{0.000000in}{-0.104167in}}%
\pgfpathcurveto{\pgfqpoint{0.027625in}{-0.104167in}}{\pgfqpoint{0.054123in}{-0.093191in}}{\pgfqpoint{0.073657in}{-0.073657in}}%
\pgfpathcurveto{\pgfqpoint{0.093191in}{-0.054123in}}{\pgfqpoint{0.104167in}{-0.027625in}}{\pgfqpoint{0.104167in}{0.000000in}}%
\pgfpathcurveto{\pgfqpoint{0.104167in}{0.027625in}}{\pgfqpoint{0.093191in}{0.054123in}}{\pgfqpoint{0.073657in}{0.073657in}}%
\pgfpathcurveto{\pgfqpoint{0.054123in}{0.093191in}}{\pgfqpoint{0.027625in}{0.104167in}}{\pgfqpoint{0.000000in}{0.104167in}}%
\pgfpathcurveto{\pgfqpoint{-0.027625in}{0.104167in}}{\pgfqpoint{-0.054123in}{0.093191in}}{\pgfqpoint{-0.073657in}{0.073657in}}%
\pgfpathcurveto{\pgfqpoint{-0.093191in}{0.054123in}}{\pgfqpoint{-0.104167in}{0.027625in}}{\pgfqpoint{-0.104167in}{0.000000in}}%
\pgfpathcurveto{\pgfqpoint{-0.104167in}{-0.027625in}}{\pgfqpoint{-0.093191in}{-0.054123in}}{\pgfqpoint{-0.073657in}{-0.073657in}}%
\pgfpathcurveto{\pgfqpoint{-0.054123in}{-0.093191in}}{\pgfqpoint{-0.027625in}{-0.104167in}}{\pgfqpoint{0.000000in}{-0.104167in}}%
\pgfpathclose%
\pgfusepath{stroke,fill}%
}%
\begin{pgfscope}%
\pgfsys@transformshift{3.200000in}{0.600000in}%
\pgfsys@useobject{currentmarker}{}%
\end{pgfscope}%
\end{pgfscope}%
\begin{pgfscope}%
\pgfsetroundcap%
\pgfsetroundjoin%
\pgfsetlinewidth{3.011250pt}%
\definecolor{currentstroke}{rgb}{0.250980,0.250980,0.250980}%
\pgfsetstrokecolor{currentstroke}%
\pgfsetdash{}{0pt}%
\pgfpathmoveto{\pgfqpoint{2.050000in}{1.138855in}}%
\pgfpathlineto{\pgfqpoint{2.050000in}{2.821822in}}%
\pgfusepath{stroke}%
\end{pgfscope}%
\begin{pgfscope}%
\pgfsetroundcap%
\pgfsetroundjoin%
\definecolor{currentfill}{rgb}{0.250980,0.250980,0.250980}%
\pgfsetfillcolor{currentfill}%
\pgfsetlinewidth{3.011250pt}%
\definecolor{currentstroke}{rgb}{0.250980,0.250980,0.250980}%
\pgfsetstrokecolor{currentstroke}%
\pgfsetdash{}{0pt}%
\pgfpathmoveto{\pgfqpoint{1.980556in}{2.710710in}}%
\pgfpathlineto{\pgfqpoint{2.050000in}{2.821822in}}%
\pgfpathlineto{\pgfqpoint{2.119444in}{2.710710in}}%
\pgfpathlineto{\pgfqpoint{1.980556in}{2.710710in}}%
\pgfpathclose%
\pgfusepath{stroke,fill}%
\end{pgfscope}%
\begin{pgfscope}%
\pgfsetroundcap%
\pgfsetroundjoin%
\pgfsetlinewidth{3.011250pt}%
\definecolor{currentstroke}{rgb}{0.250980,0.250980,0.250980}%
\pgfsetstrokecolor{currentstroke}%
\pgfsetdash{}{0pt}%
\pgfpathmoveto{\pgfqpoint{2.131726in}{0.956091in}}%
\pgfpathlineto{\pgfqpoint{3.030935in}{0.656355in}}%
\pgfusepath{stroke}%
\end{pgfscope}%
\begin{pgfscope}%
\pgfsetroundcap%
\pgfsetroundjoin%
\definecolor{currentfill}{rgb}{0.250980,0.250980,0.250980}%
\pgfsetfillcolor{currentfill}%
\pgfsetlinewidth{3.011250pt}%
\definecolor{currentstroke}{rgb}{0.250980,0.250980,0.250980}%
\pgfsetstrokecolor{currentstroke}%
\pgfsetdash{}{0pt}%
\pgfpathmoveto{\pgfqpoint{2.947486in}{0.757372in}}%
\pgfpathlineto{\pgfqpoint{3.030935in}{0.656355in}}%
\pgfpathlineto{\pgfqpoint{2.903566in}{0.625611in}}%
\pgfpathlineto{\pgfqpoint{2.947486in}{0.757372in}}%
\pgfpathclose%
\pgfusepath{stroke,fill}%
\end{pgfscope}%
\begin{pgfscope}%
\pgfsetroundcap%
\pgfsetroundjoin%
\pgfsetlinewidth{3.011250pt}%
\definecolor{currentstroke}{rgb}{0.250980,0.250980,0.250980}%
\pgfsetstrokecolor{currentstroke}%
\pgfsetdash{}{0pt}%
\pgfpathmoveto{\pgfqpoint{1.950000in}{2.861145in}}%
\pgfpathlineto{\pgfqpoint{1.950000in}{1.178178in}}%
\pgfusepath{stroke}%
\end{pgfscope}%
\begin{pgfscope}%
\pgfsetroundcap%
\pgfsetroundjoin%
\definecolor{currentfill}{rgb}{0.250980,0.250980,0.250980}%
\pgfsetfillcolor{currentfill}%
\pgfsetlinewidth{3.011250pt}%
\definecolor{currentstroke}{rgb}{0.250980,0.250980,0.250980}%
\pgfsetstrokecolor{currentstroke}%
\pgfsetdash{}{0pt}%
\pgfpathmoveto{\pgfqpoint{2.019444in}{1.289290in}}%
\pgfpathlineto{\pgfqpoint{1.950000in}{1.178178in}}%
\pgfpathlineto{\pgfqpoint{1.880556in}{1.289290in}}%
\pgfpathlineto{\pgfqpoint{2.019444in}{1.289290in}}%
\pgfpathclose%
\pgfusepath{stroke,fill}%
\end{pgfscope}%
\begin{pgfscope}%
\pgfsetroundcap%
\pgfsetroundjoin%
\pgfsetlinewidth{3.011250pt}%
\definecolor{currentstroke}{rgb}{0.250980,0.250980,0.250980}%
\pgfsetstrokecolor{currentstroke}%
\pgfsetdash{}{0pt}%
\pgfpathmoveto{\pgfqpoint{2.062128in}{2.875745in}}%
\pgfpathlineto{\pgfqpoint{3.120315in}{0.759370in}}%
\pgfusepath{stroke}%
\end{pgfscope}%
\begin{pgfscope}%
\pgfsetroundcap%
\pgfsetroundjoin%
\definecolor{currentfill}{rgb}{0.250980,0.250980,0.250980}%
\pgfsetfillcolor{currentfill}%
\pgfsetlinewidth{3.011250pt}%
\definecolor{currentstroke}{rgb}{0.250980,0.250980,0.250980}%
\pgfsetstrokecolor{currentstroke}%
\pgfsetdash{}{0pt}%
\pgfpathmoveto{\pgfqpoint{3.132738in}{0.889807in}}%
\pgfpathlineto{\pgfqpoint{3.120315in}{0.759370in}}%
\pgfpathlineto{\pgfqpoint{3.008512in}{0.827694in}}%
\pgfpathlineto{\pgfqpoint{3.132738in}{0.889807in}}%
\pgfpathclose%
\pgfusepath{stroke,fill}%
\end{pgfscope}%
\begin{pgfscope}%
\definecolor{textcolor}{rgb}{0.000000,0.000000,0.000000}%
\pgfsetstrokecolor{textcolor}%
\pgfsetfillcolor{textcolor}%
\pgftext[x=1.717157in,y=0.717157in,,]{\color{textcolor}\sffamily\fontsize{40.000000}{48.000000}\selectfont \(\displaystyle x_1\)}%
\end{pgfscope}%
\begin{pgfscope}%
\definecolor{textcolor}{rgb}{0.000000,0.000000,0.000000}%
\pgfsetstrokecolor{textcolor}%
\pgfsetfillcolor{textcolor}%
\pgftext[x=2.282843in,y=3.282843in,,]{\color{textcolor}\sffamily\fontsize{40.000000}{48.000000}\selectfont \(\displaystyle x_2\)}%
\end{pgfscope}%
\begin{pgfscope}%
\definecolor{textcolor}{rgb}{0.000000,0.000000,0.000000}%
\pgfsetstrokecolor{textcolor}%
\pgfsetfillcolor{textcolor}%
\pgftext[x=3.482843in,y=0.882843in,,]{\color{textcolor}\sffamily\fontsize{40.000000}{48.000000}\selectfont \(\displaystyle x_3\)}%
\end{pgfscope}%
\end{pgfpicture}%
\makeatother%
\endgroup%

%% file: fig/conic_graph_elliptical_motion.pgf
%% Creator: Matplotlib, PGF backend
%%
%% To include the figure in your LaTeX document, write
%%   \input{<filename>.pgf}
%%
%% Make sure the required packages are loaded in your preamble
%%   \usepackage{pgf}
%%
%% Figures using additional raster images can only be included by \input if
%% they are in the same directory as the main LaTeX file. For loading figures
%% from other directories you can use the `import` package
%%   \usepackage{import}
%%
%% and then include the figures with
%%   \import{<path to file>}{<filename>.pgf}
%%
%% Matplotlib used the following preamble
%%   \usepackage{fontspec}
%%   \setmainfont{DejaVuSerif.ttf}[Path=\detokenize{C:/Users/ccros/Anaconda3/Lib/site-packages/matplotlib/mpl-data/fonts/ttf/}]
%%   \setsansfont{DejaVuSans.ttf}[Path=\detokenize{C:/Users/ccros/Anaconda3/Lib/site-packages/matplotlib/mpl-data/fonts/ttf/}]
%%   \setmonofont{DejaVuSansMono.ttf}[Path=\detokenize{C:/Users/ccros/Anaconda3/Lib/site-packages/matplotlib/mpl-data/fonts/ttf/}]
%%
\begingroup%
\makeatletter%
\begin{pgfpicture}%
\pgfpathrectangle{\pgfpointorigin}{\pgfqpoint{4.000000in}{4.000000in}}%
\pgfusepath{use as bounding box, clip}%
\begin{pgfscope}%
\pgfsetbuttcap%
\pgfsetmiterjoin%
\pgfsetlinewidth{0.000000pt}%
\definecolor{currentstroke}{rgb}{1.000000,1.000000,1.000000}%
\pgfsetstrokecolor{currentstroke}%
\pgfsetstrokeopacity{0.000000}%
\pgfsetdash{}{0pt}%
\pgfpathmoveto{\pgfqpoint{0.000000in}{0.000000in}}%
\pgfpathlineto{\pgfqpoint{4.000000in}{0.000000in}}%
\pgfpathlineto{\pgfqpoint{4.000000in}{4.000000in}}%
\pgfpathlineto{\pgfqpoint{0.000000in}{4.000000in}}%
\pgfpathclose%
\pgfusepath{}%
\end{pgfscope}%
\begin{pgfscope}%
\pgfpathrectangle{\pgfqpoint{0.000000in}{0.000000in}}{\pgfqpoint{4.000000in}{4.000000in}}%
\pgfusepath{clip}%
\pgfsetbuttcap%
\pgfsetroundjoin%
\pgfsetlinewidth{1.505625pt}%
\definecolor{currentstroke}{rgb}{0.501961,0.501961,0.501961}%
\pgfsetstrokecolor{currentstroke}%
\pgfsetdash{{5.550000pt}{2.400000pt}}{0.000000pt}%
\pgfpathmoveto{\pgfqpoint{2.000000in}{3.974096in}}%
\pgfpathlineto{\pgfqpoint{1.892048in}{3.970122in}}%
\pgfpathlineto{\pgfqpoint{1.784530in}{3.958214in}}%
\pgfpathlineto{\pgfqpoint{1.677880in}{3.938422in}}%
\pgfpathlineto{\pgfqpoint{1.572528in}{3.910824in}}%
\pgfpathlineto{\pgfqpoint{1.468896in}{3.875532in}}%
\pgfpathlineto{\pgfqpoint{1.367403in}{3.832688in}}%
\pgfpathlineto{\pgfqpoint{1.268457in}{3.782464in}}%
\pgfpathlineto{\pgfqpoint{1.172457in}{3.725063in}}%
\pgfpathlineto{\pgfqpoint{1.079789in}{3.660716in}}%
\pgfpathlineto{\pgfqpoint{0.990827in}{3.589681in}}%
\pgfpathlineto{\pgfqpoint{0.905928in}{3.512246in}}%
\pgfpathlineto{\pgfqpoint{0.825434in}{3.428721in}}%
\pgfpathlineto{\pgfqpoint{0.749670in}{3.339443in}}%
\pgfpathlineto{\pgfqpoint{0.678941in}{3.244772in}}%
\pgfpathlineto{\pgfqpoint{0.613531in}{3.145088in}}%
\pgfpathlineto{\pgfqpoint{0.553704in}{3.040794in}}%
\pgfpathlineto{\pgfqpoint{0.499701in}{2.932309in}}%
\pgfpathlineto{\pgfqpoint{0.451739in}{2.820069in}}%
\pgfpathlineto{\pgfqpoint{0.410011in}{2.704528in}}%
\pgfpathlineto{\pgfqpoint{0.374686in}{2.586149in}}%
\pgfpathlineto{\pgfqpoint{0.345905in}{2.465411in}}%
\pgfpathlineto{\pgfqpoint{0.323784in}{2.342798in}}%
\pgfpathlineto{\pgfqpoint{0.308413in}{2.218805in}}%
\pgfpathlineto{\pgfqpoint{0.299854in}{2.093931in}}%
\pgfpathlineto{\pgfqpoint{0.298140in}{1.968679in}}%
\pgfpathlineto{\pgfqpoint{0.303279in}{1.843553in}}%
\pgfpathlineto{\pgfqpoint{0.315251in}{1.719057in}}%
\pgfpathlineto{\pgfqpoint{0.334006in}{1.595692in}}%
\pgfpathlineto{\pgfqpoint{0.359469in}{1.473955in}}%
\pgfpathlineto{\pgfqpoint{0.391539in}{1.354336in}}%
\pgfpathlineto{\pgfqpoint{0.430085in}{1.237318in}}%
\pgfpathlineto{\pgfqpoint{0.474952in}{1.123370in}}%
\pgfpathlineto{\pgfqpoint{0.525961in}{1.012952in}}%
\pgfpathlineto{\pgfqpoint{0.582904in}{0.906508in}}%
\pgfpathlineto{\pgfqpoint{0.645554in}{0.804468in}}%
\pgfpathlineto{\pgfqpoint{0.713658in}{0.707242in}}%
\pgfpathlineto{\pgfqpoint{0.786942in}{0.615221in}}%
\pgfpathlineto{\pgfqpoint{0.865110in}{0.528776in}}%
\pgfpathlineto{\pgfqpoint{0.947848in}{0.448255in}}%
\pgfpathlineto{\pgfqpoint{1.034822in}{0.373983in}}%
\pgfpathlineto{\pgfqpoint{1.125683in}{0.306258in}}%
\pgfpathlineto{\pgfqpoint{1.220065in}{0.245353in}}%
\pgfpathlineto{\pgfqpoint{1.317587in}{0.191514in}}%
\pgfpathlineto{\pgfqpoint{1.417856in}{0.144956in}}%
\pgfpathlineto{\pgfqpoint{1.520470in}{0.105868in}}%
\pgfpathlineto{\pgfqpoint{1.625015in}{0.074408in}}%
\pgfpathlineto{\pgfqpoint{1.731070in}{0.050701in}}%
\pgfpathlineto{\pgfqpoint{1.838208in}{0.034843in}}%
\pgfpathlineto{\pgfqpoint{1.945997in}{0.026898in}}%
\pgfpathlineto{\pgfqpoint{2.054003in}{0.026898in}}%
\pgfpathlineto{\pgfqpoint{2.161792in}{0.034843in}}%
\pgfpathlineto{\pgfqpoint{2.268930in}{0.050701in}}%
\pgfpathlineto{\pgfqpoint{2.374985in}{0.074408in}}%
\pgfpathlineto{\pgfqpoint{2.479530in}{0.105868in}}%
\pgfpathlineto{\pgfqpoint{2.582144in}{0.144956in}}%
\pgfpathlineto{\pgfqpoint{2.682413in}{0.191514in}}%
\pgfpathlineto{\pgfqpoint{2.779935in}{0.245353in}}%
\pgfpathlineto{\pgfqpoint{2.874317in}{0.306258in}}%
\pgfpathlineto{\pgfqpoint{2.965178in}{0.373983in}}%
\pgfpathlineto{\pgfqpoint{3.052152in}{0.448255in}}%
\pgfpathlineto{\pgfqpoint{3.134890in}{0.528776in}}%
\pgfpathlineto{\pgfqpoint{3.213058in}{0.615221in}}%
\pgfpathlineto{\pgfqpoint{3.286342in}{0.707242in}}%
\pgfpathlineto{\pgfqpoint{3.354446in}{0.804468in}}%
\pgfpathlineto{\pgfqpoint{3.417096in}{0.906508in}}%
\pgfpathlineto{\pgfqpoint{3.474039in}{1.012952in}}%
\pgfpathlineto{\pgfqpoint{3.525048in}{1.123370in}}%
\pgfpathlineto{\pgfqpoint{3.569915in}{1.237318in}}%
\pgfpathlineto{\pgfqpoint{3.608461in}{1.354336in}}%
\pgfpathlineto{\pgfqpoint{3.640531in}{1.473955in}}%
\pgfpathlineto{\pgfqpoint{3.665994in}{1.595692in}}%
\pgfpathlineto{\pgfqpoint{3.684749in}{1.719057in}}%
\pgfpathlineto{\pgfqpoint{3.696721in}{1.843553in}}%
\pgfpathlineto{\pgfqpoint{3.701860in}{1.968679in}}%
\pgfpathlineto{\pgfqpoint{3.700146in}{2.093931in}}%
\pgfpathlineto{\pgfqpoint{3.691587in}{2.218805in}}%
\pgfpathlineto{\pgfqpoint{3.676216in}{2.342798in}}%
\pgfpathlineto{\pgfqpoint{3.654095in}{2.465411in}}%
\pgfpathlineto{\pgfqpoint{3.625314in}{2.586149in}}%
\pgfpathlineto{\pgfqpoint{3.589989in}{2.704528in}}%
\pgfpathlineto{\pgfqpoint{3.548261in}{2.820069in}}%
\pgfpathlineto{\pgfqpoint{3.500299in}{2.932309in}}%
\pgfpathlineto{\pgfqpoint{3.446296in}{3.040794in}}%
\pgfpathlineto{\pgfqpoint{3.386469in}{3.145088in}}%
\pgfpathlineto{\pgfqpoint{3.321059in}{3.244772in}}%
\pgfpathlineto{\pgfqpoint{3.250330in}{3.339443in}}%
\pgfpathlineto{\pgfqpoint{3.174566in}{3.428721in}}%
\pgfpathlineto{\pgfqpoint{3.094072in}{3.512246in}}%
\pgfpathlineto{\pgfqpoint{3.009173in}{3.589681in}}%
\pgfpathlineto{\pgfqpoint{2.920211in}{3.660716in}}%
\pgfpathlineto{\pgfqpoint{2.827543in}{3.725063in}}%
\pgfpathlineto{\pgfqpoint{2.731543in}{3.782464in}}%
\pgfpathlineto{\pgfqpoint{2.632597in}{3.832688in}}%
\pgfpathlineto{\pgfqpoint{2.531104in}{3.875532in}}%
\pgfpathlineto{\pgfqpoint{2.427472in}{3.910824in}}%
\pgfpathlineto{\pgfqpoint{2.322120in}{3.938422in}}%
\pgfpathlineto{\pgfqpoint{2.215470in}{3.958214in}}%
\pgfpathlineto{\pgfqpoint{2.107952in}{3.970122in}}%
\pgfpathlineto{\pgfqpoint{2.000000in}{3.974096in}}%
\pgfusepath{stroke}%
\end{pgfscope}%
\begin{pgfscope}%
\pgfpathrectangle{\pgfqpoint{0.000000in}{0.000000in}}{\pgfqpoint{4.000000in}{4.000000in}}%
\pgfusepath{clip}%
\pgfsetbuttcap%
\pgfsetroundjoin%
\definecolor{currentfill}{rgb}{0.800000,0.800000,0.800000}%
\pgfsetfillcolor{currentfill}%
\pgfsetlinewidth{1.003750pt}%
\definecolor{currentstroke}{rgb}{0.000000,0.000000,0.000000}%
\pgfsetstrokecolor{currentstroke}%
\pgfsetdash{}{0pt}%
\pgfsys@defobject{currentmarker}{\pgfqpoint{-0.104167in}{-0.104167in}}{\pgfqpoint{0.104167in}{0.104167in}}{%
\pgfpathmoveto{\pgfqpoint{0.000000in}{-0.104167in}}%
\pgfpathcurveto{\pgfqpoint{0.027625in}{-0.104167in}}{\pgfqpoint{0.054123in}{-0.093191in}}{\pgfqpoint{0.073657in}{-0.073657in}}%
\pgfpathcurveto{\pgfqpoint{0.093191in}{-0.054123in}}{\pgfqpoint{0.104167in}{-0.027625in}}{\pgfqpoint{0.104167in}{0.000000in}}%
\pgfpathcurveto{\pgfqpoint{0.104167in}{0.027625in}}{\pgfqpoint{0.093191in}{0.054123in}}{\pgfqpoint{0.073657in}{0.073657in}}%
\pgfpathcurveto{\pgfqpoint{0.054123in}{0.093191in}}{\pgfqpoint{0.027625in}{0.104167in}}{\pgfqpoint{0.000000in}{0.104167in}}%
\pgfpathcurveto{\pgfqpoint{-0.027625in}{0.104167in}}{\pgfqpoint{-0.054123in}{0.093191in}}{\pgfqpoint{-0.073657in}{0.073657in}}%
\pgfpathcurveto{\pgfqpoint{-0.093191in}{0.054123in}}{\pgfqpoint{-0.104167in}{0.027625in}}{\pgfqpoint{-0.104167in}{0.000000in}}%
\pgfpathcurveto{\pgfqpoint{-0.104167in}{-0.027625in}}{\pgfqpoint{-0.093191in}{-0.054123in}}{\pgfqpoint{-0.073657in}{-0.073657in}}%
\pgfpathcurveto{\pgfqpoint{-0.054123in}{-0.093191in}}{\pgfqpoint{-0.027625in}{-0.104167in}}{\pgfqpoint{0.000000in}{-0.104167in}}%
\pgfpathclose%
\pgfusepath{stroke,fill}%
}%
\begin{pgfscope}%
\pgfsys@transformshift{2.000000in}{1.000000in}%
\pgfsys@useobject{currentmarker}{}%
\end{pgfscope}%
\end{pgfscope}%
\begin{pgfscope}%
\pgfpathrectangle{\pgfqpoint{0.000000in}{0.000000in}}{\pgfqpoint{4.000000in}{4.000000in}}%
\pgfusepath{clip}%
\pgfsetbuttcap%
\pgfsetroundjoin%
\definecolor{currentfill}{rgb}{0.800000,0.800000,0.800000}%
\pgfsetfillcolor{currentfill}%
\pgfsetlinewidth{1.003750pt}%
\definecolor{currentstroke}{rgb}{0.000000,0.000000,0.000000}%
\pgfsetstrokecolor{currentstroke}%
\pgfsetdash{}{0pt}%
\pgfsys@defobject{currentmarker}{\pgfqpoint{-0.104167in}{-0.104167in}}{\pgfqpoint{0.104167in}{0.104167in}}{%
\pgfpathmoveto{\pgfqpoint{0.000000in}{-0.104167in}}%
\pgfpathcurveto{\pgfqpoint{0.027625in}{-0.104167in}}{\pgfqpoint{0.054123in}{-0.093191in}}{\pgfqpoint{0.073657in}{-0.073657in}}%
\pgfpathcurveto{\pgfqpoint{0.093191in}{-0.054123in}}{\pgfqpoint{0.104167in}{-0.027625in}}{\pgfqpoint{0.104167in}{0.000000in}}%
\pgfpathcurveto{\pgfqpoint{0.104167in}{0.027625in}}{\pgfqpoint{0.093191in}{0.054123in}}{\pgfqpoint{0.073657in}{0.073657in}}%
\pgfpathcurveto{\pgfqpoint{0.054123in}{0.093191in}}{\pgfqpoint{0.027625in}{0.104167in}}{\pgfqpoint{0.000000in}{0.104167in}}%
\pgfpathcurveto{\pgfqpoint{-0.027625in}{0.104167in}}{\pgfqpoint{-0.054123in}{0.093191in}}{\pgfqpoint{-0.073657in}{0.073657in}}%
\pgfpathcurveto{\pgfqpoint{-0.093191in}{0.054123in}}{\pgfqpoint{-0.104167in}{0.027625in}}{\pgfqpoint{-0.104167in}{0.000000in}}%
\pgfpathcurveto{\pgfqpoint{-0.104167in}{-0.027625in}}{\pgfqpoint{-0.093191in}{-0.054123in}}{\pgfqpoint{-0.073657in}{-0.073657in}}%
\pgfpathcurveto{\pgfqpoint{-0.054123in}{-0.093191in}}{\pgfqpoint{-0.027625in}{-0.104167in}}{\pgfqpoint{0.000000in}{-0.104167in}}%
\pgfpathclose%
\pgfusepath{stroke,fill}%
}%
\begin{pgfscope}%
\pgfsys@transformshift{2.000000in}{3.000000in}%
\pgfsys@useobject{currentmarker}{}%
\end{pgfscope}%
\end{pgfscope}%
\begin{pgfscope}%
\pgfpathrectangle{\pgfqpoint{0.000000in}{0.000000in}}{\pgfqpoint{4.000000in}{4.000000in}}%
\pgfusepath{clip}%
\pgfsetbuttcap%
\pgfsetroundjoin%
\definecolor{currentfill}{rgb}{0.800000,0.800000,0.800000}%
\pgfsetfillcolor{currentfill}%
\pgfsetlinewidth{1.003750pt}%
\definecolor{currentstroke}{rgb}{0.000000,0.000000,0.000000}%
\pgfsetstrokecolor{currentstroke}%
\pgfsetdash{}{0pt}%
\pgfsys@defobject{currentmarker}{\pgfqpoint{-0.104167in}{-0.104167in}}{\pgfqpoint{0.104167in}{0.104167in}}{%
\pgfpathmoveto{\pgfqpoint{0.000000in}{-0.104167in}}%
\pgfpathcurveto{\pgfqpoint{0.027625in}{-0.104167in}}{\pgfqpoint{0.054123in}{-0.093191in}}{\pgfqpoint{0.073657in}{-0.073657in}}%
\pgfpathcurveto{\pgfqpoint{0.093191in}{-0.054123in}}{\pgfqpoint{0.104167in}{-0.027625in}}{\pgfqpoint{0.104167in}{0.000000in}}%
\pgfpathcurveto{\pgfqpoint{0.104167in}{0.027625in}}{\pgfqpoint{0.093191in}{0.054123in}}{\pgfqpoint{0.073657in}{0.073657in}}%
\pgfpathcurveto{\pgfqpoint{0.054123in}{0.093191in}}{\pgfqpoint{0.027625in}{0.104167in}}{\pgfqpoint{0.000000in}{0.104167in}}%
\pgfpathcurveto{\pgfqpoint{-0.027625in}{0.104167in}}{\pgfqpoint{-0.054123in}{0.093191in}}{\pgfqpoint{-0.073657in}{0.073657in}}%
\pgfpathcurveto{\pgfqpoint{-0.093191in}{0.054123in}}{\pgfqpoint{-0.104167in}{0.027625in}}{\pgfqpoint{-0.104167in}{0.000000in}}%
\pgfpathcurveto{\pgfqpoint{-0.104167in}{-0.027625in}}{\pgfqpoint{-0.093191in}{-0.054123in}}{\pgfqpoint{-0.073657in}{-0.073657in}}%
\pgfpathcurveto{\pgfqpoint{-0.054123in}{-0.093191in}}{\pgfqpoint{-0.027625in}{-0.104167in}}{\pgfqpoint{0.000000in}{-0.104167in}}%
\pgfpathclose%
\pgfusepath{stroke,fill}%
}%
\begin{pgfscope}%
\pgfsys@transformshift{3.200000in}{0.600000in}%
\pgfsys@useobject{currentmarker}{}%
\end{pgfscope}%
\end{pgfscope}%
\begin{pgfscope}%
\pgfsetroundcap%
\pgfsetroundjoin%
\pgfsetlinewidth{3.011250pt}%
\definecolor{currentstroke}{rgb}{0.250980,0.250980,0.250980}%
\pgfsetstrokecolor{currentstroke}%
\pgfsetdash{}{0pt}%
\pgfpathmoveto{\pgfqpoint{2.050000in}{1.138855in}}%
\pgfpathlineto{\pgfqpoint{2.050000in}{2.821822in}}%
\pgfusepath{stroke}%
\end{pgfscope}%
\begin{pgfscope}%
\pgfsetroundcap%
\pgfsetroundjoin%
\definecolor{currentfill}{rgb}{0.250980,0.250980,0.250980}%
\pgfsetfillcolor{currentfill}%
\pgfsetlinewidth{3.011250pt}%
\definecolor{currentstroke}{rgb}{0.250980,0.250980,0.250980}%
\pgfsetstrokecolor{currentstroke}%
\pgfsetdash{}{0pt}%
\pgfpathmoveto{\pgfqpoint{1.980556in}{2.710710in}}%
\pgfpathlineto{\pgfqpoint{2.050000in}{2.821822in}}%
\pgfpathlineto{\pgfqpoint{2.119444in}{2.710710in}}%
\pgfpathlineto{\pgfqpoint{1.980556in}{2.710710in}}%
\pgfpathclose%
\pgfusepath{stroke,fill}%
\end{pgfscope}%
\begin{pgfscope}%
\pgfsetroundcap%
\pgfsetroundjoin%
\pgfsetlinewidth{3.011250pt}%
\definecolor{currentstroke}{rgb}{0.250980,0.250980,0.250980}%
\pgfsetstrokecolor{currentstroke}%
\pgfsetdash{}{0pt}%
\pgfpathmoveto{\pgfqpoint{2.131726in}{0.956091in}}%
\pgfpathlineto{\pgfqpoint{3.030935in}{0.656355in}}%
\pgfusepath{stroke}%
\end{pgfscope}%
\begin{pgfscope}%
\pgfsetroundcap%
\pgfsetroundjoin%
\definecolor{currentfill}{rgb}{0.250980,0.250980,0.250980}%
\pgfsetfillcolor{currentfill}%
\pgfsetlinewidth{3.011250pt}%
\definecolor{currentstroke}{rgb}{0.250980,0.250980,0.250980}%
\pgfsetstrokecolor{currentstroke}%
\pgfsetdash{}{0pt}%
\pgfpathmoveto{\pgfqpoint{2.947486in}{0.757372in}}%
\pgfpathlineto{\pgfqpoint{3.030935in}{0.656355in}}%
\pgfpathlineto{\pgfqpoint{2.903566in}{0.625611in}}%
\pgfpathlineto{\pgfqpoint{2.947486in}{0.757372in}}%
\pgfpathclose%
\pgfusepath{stroke,fill}%
\end{pgfscope}%
\begin{pgfscope}%
\pgfsetroundcap%
\pgfsetroundjoin%
\pgfsetlinewidth{3.011250pt}%
\definecolor{currentstroke}{rgb}{0.250980,0.250980,0.250980}%
\pgfsetstrokecolor{currentstroke}%
\pgfsetdash{}{0pt}%
\pgfpathmoveto{\pgfqpoint{1.950000in}{2.861145in}}%
\pgfpathlineto{\pgfqpoint{1.950000in}{1.178178in}}%
\pgfusepath{stroke}%
\end{pgfscope}%
\begin{pgfscope}%
\pgfsetroundcap%
\pgfsetroundjoin%
\definecolor{currentfill}{rgb}{0.250980,0.250980,0.250980}%
\pgfsetfillcolor{currentfill}%
\pgfsetlinewidth{3.011250pt}%
\definecolor{currentstroke}{rgb}{0.250980,0.250980,0.250980}%
\pgfsetstrokecolor{currentstroke}%
\pgfsetdash{}{0pt}%
\pgfpathmoveto{\pgfqpoint{2.019444in}{1.289290in}}%
\pgfpathlineto{\pgfqpoint{1.950000in}{1.178178in}}%
\pgfpathlineto{\pgfqpoint{1.880556in}{1.289290in}}%
\pgfpathlineto{\pgfqpoint{2.019444in}{1.289290in}}%
\pgfpathclose%
\pgfusepath{stroke,fill}%
\end{pgfscope}%
\begin{pgfscope}%
\pgfsetroundcap%
\pgfsetroundjoin%
\pgfsetlinewidth{3.011250pt}%
\definecolor{currentstroke}{rgb}{0.250980,0.250980,0.250980}%
\pgfsetstrokecolor{currentstroke}%
\pgfsetdash{}{0pt}%
\pgfpathmoveto{\pgfqpoint{3.137872in}{0.724255in}}%
\pgfpathlineto{\pgfqpoint{2.079685in}{2.840630in}}%
\pgfusepath{stroke}%
\end{pgfscope}%
\begin{pgfscope}%
\pgfsetroundcap%
\pgfsetroundjoin%
\definecolor{currentfill}{rgb}{0.250980,0.250980,0.250980}%
\pgfsetfillcolor{currentfill}%
\pgfsetlinewidth{3.011250pt}%
\definecolor{currentstroke}{rgb}{0.250980,0.250980,0.250980}%
\pgfsetstrokecolor{currentstroke}%
\pgfsetdash{}{0pt}%
\pgfpathmoveto{\pgfqpoint{2.067262in}{2.710193in}}%
\pgfpathlineto{\pgfqpoint{2.079685in}{2.840630in}}%
\pgfpathlineto{\pgfqpoint{2.191488in}{2.772306in}}%
\pgfpathlineto{\pgfqpoint{2.067262in}{2.710193in}}%
\pgfpathclose%
\pgfusepath{stroke,fill}%
\end{pgfscope}%
\begin{pgfscope}%
\definecolor{textcolor}{rgb}{0.000000,0.000000,0.000000}%
\pgfsetstrokecolor{textcolor}%
\pgfsetfillcolor{textcolor}%
\pgftext[x=1.717157in,y=0.717157in,,]{\color{textcolor}\sffamily\fontsize{40.000000}{48.000000}\selectfont \(\displaystyle x_1\)}%
\end{pgfscope}%
\begin{pgfscope}%
\definecolor{textcolor}{rgb}{0.000000,0.000000,0.000000}%
\pgfsetstrokecolor{textcolor}%
\pgfsetfillcolor{textcolor}%
\pgftext[x=2.282843in,y=3.282843in,,]{\color{textcolor}\sffamily\fontsize{40.000000}{48.000000}\selectfont \(\displaystyle x_2\)}%
\end{pgfscope}%
\begin{pgfscope}%
\definecolor{textcolor}{rgb}{0.000000,0.000000,0.000000}%
\pgfsetstrokecolor{textcolor}%
\pgfsetfillcolor{textcolor}%
\pgftext[x=3.482843in,y=0.882843in,,]{\color{textcolor}\sffamily\fontsize{40.000000}{48.000000}\selectfont \(\displaystyle x_3\)}%
\end{pgfscope}%
\end{pgfpicture}%
\makeatother%
\endgroup%

%% file: fig/conic_graph_rigid.pgf
%% Creator: Matplotlib, PGF backend
%%
%% To include the figure in your LaTeX document, write
%%   \input{<filename>.pgf}
%%
%% Make sure the required packages are loaded in your preamble
%%   \usepackage{pgf}
%%
%% Figures using additional raster images can only be included by \input if
%% they are in the same directory as the main LaTeX file. For loading figures
%% from other directories you can use the `import` package
%%   \usepackage{import}
%%
%% and then include the figures with
%%   \import{<path to file>}{<filename>.pgf}
%%
%% Matplotlib used the following preamble
%%   \usepackage{fontspec}
%%   \setmainfont{DejaVuSerif.ttf}[Path=\detokenize{C:/Users/ccros/Anaconda3/Lib/site-packages/matplotlib/mpl-data/fonts/ttf/}]
%%   \setsansfont{DejaVuSans.ttf}[Path=\detokenize{C:/Users/ccros/Anaconda3/Lib/site-packages/matplotlib/mpl-data/fonts/ttf/}]
%%   \setmonofont{DejaVuSansMono.ttf}[Path=\detokenize{C:/Users/ccros/Anaconda3/Lib/site-packages/matplotlib/mpl-data/fonts/ttf/}]
%%
\begingroup%
\makeatletter%
\begin{pgfpicture}%
\pgfpathrectangle{\pgfpointorigin}{\pgfqpoint{4.000000in}{4.000000in}}%
\pgfusepath{use as bounding box, clip}%
\begin{pgfscope}%
\pgfsetbuttcap%
\pgfsetmiterjoin%
\pgfsetlinewidth{0.000000pt}%
\definecolor{currentstroke}{rgb}{1.000000,1.000000,1.000000}%
\pgfsetstrokecolor{currentstroke}%
\pgfsetstrokeopacity{0.000000}%
\pgfsetdash{}{0pt}%
\pgfpathmoveto{\pgfqpoint{0.000000in}{0.000000in}}%
\pgfpathlineto{\pgfqpoint{4.000000in}{0.000000in}}%
\pgfpathlineto{\pgfqpoint{4.000000in}{4.000000in}}%
\pgfpathlineto{\pgfqpoint{0.000000in}{4.000000in}}%
\pgfpathclose%
\pgfusepath{}%
\end{pgfscope}%
\begin{pgfscope}%
\pgfpathrectangle{\pgfqpoint{0.000000in}{0.000000in}}{\pgfqpoint{4.000000in}{4.000000in}}%
\pgfusepath{clip}%
\pgfsetbuttcap%
\pgfsetroundjoin%
\pgfsetlinewidth{1.505625pt}%
\definecolor{currentstroke}{rgb}{0.501961,0.501961,0.501961}%
\pgfsetstrokecolor{currentstroke}%
\pgfsetdash{{5.550000pt}{2.400000pt}}{0.000000pt}%
\pgfpathmoveto{\pgfqpoint{2.674627in}{-0.013889in}}%
\pgfpathlineto{\pgfqpoint{2.664799in}{0.018945in}}%
\pgfpathlineto{\pgfqpoint{2.639685in}{0.104923in}}%
\pgfpathlineto{\pgfqpoint{2.616477in}{0.186622in}}%
\pgfpathlineto{\pgfqpoint{2.595081in}{0.264373in}}%
\pgfpathlineto{\pgfqpoint{2.575412in}{0.338489in}}%
\pgfpathlineto{\pgfqpoint{2.557390in}{0.409267in}}%
\pgfpathlineto{\pgfqpoint{2.540942in}{0.476994in}}%
\pgfpathlineto{\pgfqpoint{2.526002in}{0.541942in}}%
\pgfpathlineto{\pgfqpoint{2.512509in}{0.604372in}}%
\pgfpathlineto{\pgfqpoint{2.500411in}{0.664538in}}%
\pgfpathlineto{\pgfqpoint{2.489657in}{0.722679in}}%
\pgfpathlineto{\pgfqpoint{2.480204in}{0.779032in}}%
\pgfpathlineto{\pgfqpoint{2.472015in}{0.833824in}}%
\pgfpathlineto{\pgfqpoint{2.465056in}{0.887274in}}%
\pgfpathlineto{\pgfqpoint{2.459299in}{0.939598in}}%
\pgfpathlineto{\pgfqpoint{2.454722in}{0.991007in}}%
\pgfpathlineto{\pgfqpoint{2.451305in}{1.041709in}}%
\pgfpathlineto{\pgfqpoint{2.449035in}{1.091907in}}%
\pgfpathlineto{\pgfqpoint{2.447902in}{1.141804in}}%
\pgfpathlineto{\pgfqpoint{2.447903in}{1.191600in}}%
\pgfpathlineto{\pgfqpoint{2.449037in}{1.241497in}}%
\pgfpathlineto{\pgfqpoint{2.451309in}{1.291696in}}%
\pgfpathlineto{\pgfqpoint{2.454727in}{1.342399in}}%
\pgfpathlineto{\pgfqpoint{2.459306in}{1.393809in}}%
\pgfpathlineto{\pgfqpoint{2.465065in}{1.446135in}}%
\pgfpathlineto{\pgfqpoint{2.472026in}{1.499587in}}%
\pgfpathlineto{\pgfqpoint{2.480217in}{1.554380in}}%
\pgfpathlineto{\pgfqpoint{2.489671in}{1.610735in}}%
\pgfpathlineto{\pgfqpoint{2.500427in}{1.668880in}}%
\pgfpathlineto{\pgfqpoint{2.512528in}{1.729048in}}%
\pgfpathlineto{\pgfqpoint{2.526022in}{1.791482in}}%
\pgfpathlineto{\pgfqpoint{2.540964in}{1.856433in}}%
\pgfpathlineto{\pgfqpoint{2.557414in}{1.924164in}}%
\pgfpathlineto{\pgfqpoint{2.575439in}{1.994947in}}%
\pgfpathlineto{\pgfqpoint{2.595111in}{2.069068in}}%
\pgfpathlineto{\pgfqpoint{2.616509in}{2.146824in}}%
\pgfpathlineto{\pgfqpoint{2.639719in}{2.228529in}}%
\pgfpathlineto{\pgfqpoint{2.664836in}{2.314513in}}%
\pgfpathlineto{\pgfqpoint{2.691959in}{2.405122in}}%
\pgfpathlineto{\pgfqpoint{2.721200in}{2.500721in}}%
\pgfpathlineto{\pgfqpoint{2.752674in}{2.601696in}}%
\pgfpathlineto{\pgfqpoint{2.786510in}{2.708452in}}%
\pgfpathlineto{\pgfqpoint{2.822844in}{2.821420in}}%
\pgfpathlineto{\pgfqpoint{2.861821in}{2.941056in}}%
\pgfpathlineto{\pgfqpoint{2.903599in}{3.067842in}}%
\pgfpathlineto{\pgfqpoint{2.948346in}{3.202287in}}%
\pgfpathlineto{\pgfqpoint{2.996243in}{3.344935in}}%
\pgfpathlineto{\pgfqpoint{3.047483in}{3.496359in}}%
\pgfpathlineto{\pgfqpoint{3.102271in}{3.657170in}}%
\pgfpathlineto{\pgfqpoint{3.160830in}{3.828016in}}%
\pgfpathlineto{\pgfqpoint{3.223394in}{4.009585in}}%
\pgfpathlineto{\pgfqpoint{3.224884in}{4.013889in}}%
\pgfusepath{stroke}%
\end{pgfscope}%
\begin{pgfscope}%
\pgfpathrectangle{\pgfqpoint{0.000000in}{0.000000in}}{\pgfqpoint{4.000000in}{4.000000in}}%
\pgfusepath{clip}%
\pgfsetbuttcap%
\pgfsetroundjoin%
\pgfsetlinewidth{1.505625pt}%
\definecolor{currentstroke}{rgb}{0.501961,0.501961,0.501961}%
\pgfsetstrokecolor{currentstroke}%
\pgfsetdash{{5.550000pt}{2.400000pt}}{0.000000pt}%
\pgfpathmoveto{\pgfqpoint{0.949560in}{0.475606in}}%
\pgfpathlineto{\pgfqpoint{1.029214in}{0.435162in}}%
\pgfpathlineto{\pgfqpoint{1.111770in}{0.400682in}}%
\pgfpathlineto{\pgfqpoint{1.196897in}{0.372307in}}%
\pgfpathlineto{\pgfqpoint{1.284250in}{0.350150in}}%
\pgfpathlineto{\pgfqpoint{1.373479in}{0.334301in}}%
\pgfpathlineto{\pgfqpoint{1.464223in}{0.324823in}}%
\pgfpathlineto{\pgfqpoint{1.556119in}{0.321755in}}%
\pgfpathlineto{\pgfqpoint{1.648795in}{0.325109in}}%
\pgfpathlineto{\pgfqpoint{1.741878in}{0.334871in}}%
\pgfpathlineto{\pgfqpoint{1.834995in}{0.351002in}}%
\pgfpathlineto{\pgfqpoint{1.927769in}{0.373438in}}%
\pgfpathlineto{\pgfqpoint{2.019827in}{0.402087in}}%
\pgfpathlineto{\pgfqpoint{2.110799in}{0.436835in}}%
\pgfpathlineto{\pgfqpoint{2.200318in}{0.477542in}}%
\pgfpathlineto{\pgfqpoint{2.288024in}{0.524043in}}%
\pgfpathlineto{\pgfqpoint{2.373563in}{0.576151in}}%
\pgfpathlineto{\pgfqpoint{2.456592in}{0.633658in}}%
\pgfpathlineto{\pgfqpoint{2.536775in}{0.696330in}}%
\pgfpathlineto{\pgfqpoint{2.613791in}{0.763916in}}%
\pgfpathlineto{\pgfqpoint{2.687328in}{0.836144in}}%
\pgfpathlineto{\pgfqpoint{2.757092in}{0.912722in}}%
\pgfpathlineto{\pgfqpoint{2.822800in}{0.993343in}}%
\pgfpathlineto{\pgfqpoint{2.884188in}{1.077681in}}%
\pgfpathlineto{\pgfqpoint{2.941010in}{1.165398in}}%
\pgfpathlineto{\pgfqpoint{2.993036in}{1.256140in}}%
\pgfpathlineto{\pgfqpoint{3.040057in}{1.349541in}}%
\pgfpathlineto{\pgfqpoint{3.081884in}{1.445226in}}%
\pgfpathlineto{\pgfqpoint{3.118347in}{1.542809in}}%
\pgfpathlineto{\pgfqpoint{3.149301in}{1.641897in}}%
\pgfpathlineto{\pgfqpoint{3.174621in}{1.742092in}}%
\pgfpathlineto{\pgfqpoint{3.194205in}{1.842990in}}%
\pgfpathlineto{\pgfqpoint{3.207973in}{1.944185in}}%
\pgfpathlineto{\pgfqpoint{3.215871in}{2.045268in}}%
\pgfpathlineto{\pgfqpoint{3.217866in}{2.145834in}}%
\pgfpathlineto{\pgfqpoint{3.213951in}{2.245477in}}%
\pgfpathlineto{\pgfqpoint{3.204141in}{2.343796in}}%
\pgfpathlineto{\pgfqpoint{3.188476in}{2.440396in}}%
\pgfpathlineto{\pgfqpoint{3.167020in}{2.534886in}}%
\pgfpathlineto{\pgfqpoint{3.139857in}{2.626887in}}%
\pgfpathlineto{\pgfqpoint{3.107098in}{2.716029in}}%
\pgfpathlineto{\pgfqpoint{3.068875in}{2.801952in}}%
\pgfpathlineto{\pgfqpoint{3.025341in}{2.884310in}}%
\pgfpathlineto{\pgfqpoint{2.976672in}{2.962772in}}%
\pgfpathlineto{\pgfqpoint{2.923064in}{3.037022in}}%
\pgfpathlineto{\pgfqpoint{2.864733in}{3.106760in}}%
\pgfpathlineto{\pgfqpoint{2.801913in}{3.171707in}}%
\pgfpathlineto{\pgfqpoint{2.734857in}{3.231600in}}%
\pgfpathlineto{\pgfqpoint{2.663836in}{3.286199in}}%
\pgfpathlineto{\pgfqpoint{2.589135in}{3.335283in}}%
\pgfpathlineto{\pgfqpoint{2.511056in}{3.378655in}}%
\pgfpathlineto{\pgfqpoint{2.429912in}{3.416141in}}%
\pgfpathlineto{\pgfqpoint{2.346030in}{3.447588in}}%
\pgfpathlineto{\pgfqpoint{2.259749in}{3.472872in}}%
\pgfpathlineto{\pgfqpoint{2.171415in}{3.491889in}}%
\pgfpathlineto{\pgfqpoint{2.081384in}{3.504564in}}%
\pgfpathlineto{\pgfqpoint{1.990019in}{3.510845in}}%
\pgfpathlineto{\pgfqpoint{1.897687in}{3.510706in}}%
\pgfpathlineto{\pgfqpoint{1.804761in}{3.504150in}}%
\pgfpathlineto{\pgfqpoint{1.711614in}{3.491201in}}%
\pgfpathlineto{\pgfqpoint{1.618622in}{3.471913in}}%
\pgfpathlineto{\pgfqpoint{1.526159in}{3.446362in}}%
\pgfpathlineto{\pgfqpoint{1.434597in}{3.414652in}}%
\pgfpathlineto{\pgfqpoint{1.344305in}{3.376911in}}%
\pgfpathlineto{\pgfqpoint{1.255646in}{3.333289in}}%
\pgfpathlineto{\pgfqpoint{1.168978in}{3.283964in}}%
\pgfpathlineto{\pgfqpoint{1.084650in}{3.229133in}}%
\pgfpathlineto{\pgfqpoint{1.003001in}{3.169017in}}%
\pgfpathlineto{\pgfqpoint{0.924359in}{3.103859in}}%
\pgfpathlineto{\pgfqpoint{0.849042in}{3.033920in}}%
\pgfpathlineto{\pgfqpoint{0.777353in}{2.959483in}}%
\pgfpathlineto{\pgfqpoint{0.709580in}{2.880847in}}%
\pgfpathlineto{\pgfqpoint{0.645996in}{2.798329in}}%
\pgfpathlineto{\pgfqpoint{0.586858in}{2.712260in}}%
\pgfpathlineto{\pgfqpoint{0.532403in}{2.622988in}}%
\pgfpathlineto{\pgfqpoint{0.482851in}{2.530872in}}%
\pgfpathlineto{\pgfqpoint{0.438401in}{2.436283in}}%
\pgfpathlineto{\pgfqpoint{0.399232in}{2.339602in}}%
\pgfpathlineto{\pgfqpoint{0.365502in}{2.241218in}}%
\pgfpathlineto{\pgfqpoint{0.337346in}{2.141527in}}%
\pgfpathlineto{\pgfqpoint{0.314879in}{2.040930in}}%
\pgfpathlineto{\pgfqpoint{0.298191in}{1.939834in}}%
\pgfpathlineto{\pgfqpoint{0.287348in}{1.838644in}}%
\pgfpathlineto{\pgfqpoint{0.282394in}{1.737768in}}%
\pgfpathlineto{\pgfqpoint{0.283350in}{1.637612in}}%
\pgfpathlineto{\pgfqpoint{0.290212in}{1.538580in}}%
\pgfpathlineto{\pgfqpoint{0.302951in}{1.441070in}}%
\pgfpathlineto{\pgfqpoint{0.321517in}{1.345476in}}%
\pgfpathlineto{\pgfqpoint{0.345834in}{1.252181in}}%
\pgfpathlineto{\pgfqpoint{0.375806in}{1.161562in}}%
\pgfpathlineto{\pgfqpoint{0.411311in}{1.073983in}}%
\pgfpathlineto{\pgfqpoint{0.452206in}{0.989798in}}%
\pgfpathlineto{\pgfqpoint{0.498327in}{0.909344in}}%
\pgfpathlineto{\pgfqpoint{0.549487in}{0.832947in}}%
\pgfpathlineto{\pgfqpoint{0.605482in}{0.760913in}}%
\pgfpathlineto{\pgfqpoint{0.666084in}{0.693533in}}%
\pgfpathlineto{\pgfqpoint{0.731052in}{0.631078in}}%
\pgfpathlineto{\pgfqpoint{0.800122in}{0.573799in}}%
\pgfpathlineto{\pgfqpoint{0.873017in}{0.521927in}}%
\pgfpathlineto{\pgfqpoint{0.949443in}{0.475671in}}%
\pgfusepath{stroke}%
\end{pgfscope}%
\begin{pgfscope}%
\pgfpathrectangle{\pgfqpoint{0.000000in}{0.000000in}}{\pgfqpoint{4.000000in}{4.000000in}}%
\pgfusepath{clip}%
\pgfsetbuttcap%
\pgfsetroundjoin%
\pgfsetlinewidth{1.505625pt}%
\definecolor{currentstroke}{rgb}{0.501961,0.501961,0.501961}%
\pgfsetstrokecolor{currentstroke}%
\pgfsetdash{{5.550000pt}{2.400000pt}}{0.000000pt}%
\pgfpathmoveto{\pgfqpoint{3.247462in}{0.721438in}}%
\pgfpathlineto{\pgfqpoint{3.305148in}{0.756644in}}%
\pgfpathlineto{\pgfqpoint{3.359928in}{0.796521in}}%
\pgfpathlineto{\pgfqpoint{3.411581in}{0.840909in}}%
\pgfpathlineto{\pgfqpoint{3.459899in}{0.889628in}}%
\pgfpathlineto{\pgfqpoint{3.504687in}{0.942482in}}%
\pgfpathlineto{\pgfqpoint{3.545766in}{0.999259in}}%
\pgfpathlineto{\pgfqpoint{3.582969in}{1.059730in}}%
\pgfpathlineto{\pgfqpoint{3.616147in}{1.123651in}}%
\pgfpathlineto{\pgfqpoint{3.645167in}{1.190766in}}%
\pgfpathlineto{\pgfqpoint{3.669911in}{1.260803in}}%
\pgfpathlineto{\pgfqpoint{3.690280in}{1.333481in}}%
\pgfpathlineto{\pgfqpoint{3.706191in}{1.408507in}}%
\pgfpathlineto{\pgfqpoint{3.717582in}{1.485580in}}%
\pgfpathlineto{\pgfqpoint{3.724405in}{1.564388in}}%
\pgfpathlineto{\pgfqpoint{3.726634in}{1.644615in}}%
\pgfpathlineto{\pgfqpoint{3.724260in}{1.725937in}}%
\pgfpathlineto{\pgfqpoint{3.717291in}{1.808027in}}%
\pgfpathlineto{\pgfqpoint{3.705757in}{1.890555in}}%
\pgfpathlineto{\pgfqpoint{3.689703in}{1.973187in}}%
\pgfpathlineto{\pgfqpoint{3.669194in}{2.055593in}}%
\pgfpathlineto{\pgfqpoint{3.644313in}{2.137438in}}%
\pgfpathlineto{\pgfqpoint{3.615160in}{2.218395in}}%
\pgfpathlineto{\pgfqpoint{3.581852in}{2.298137in}}%
\pgfpathlineto{\pgfqpoint{3.544524in}{2.376343in}}%
\pgfpathlineto{\pgfqpoint{3.503325in}{2.452698in}}%
\pgfpathlineto{\pgfqpoint{3.458422in}{2.526894in}}%
\pgfpathlineto{\pgfqpoint{3.409996in}{2.598634in}}%
\pgfpathlineto{\pgfqpoint{3.358240in}{2.667628in}}%
\pgfpathlineto{\pgfqpoint{3.303365in}{2.733597in}}%
\pgfpathlineto{\pgfqpoint{3.245591in}{2.796278in}}%
\pgfpathlineto{\pgfqpoint{3.185150in}{2.855416in}}%
\pgfpathlineto{\pgfqpoint{3.122285in}{2.910775in}}%
\pgfpathlineto{\pgfqpoint{3.057251in}{2.962131in}}%
\pgfpathlineto{\pgfqpoint{2.990308in}{3.009277in}}%
\pgfpathlineto{\pgfqpoint{2.921727in}{3.052024in}}%
\pgfpathlineto{\pgfqpoint{2.851783in}{3.090200in}}%
\pgfpathlineto{\pgfqpoint{2.780758in}{3.123650in}}%
\pgfpathlineto{\pgfqpoint{2.708938in}{3.152240in}}%
\pgfpathlineto{\pgfqpoint{2.636613in}{3.175855in}}%
\pgfpathlineto{\pgfqpoint{2.564073in}{3.194400in}}%
\pgfpathlineto{\pgfqpoint{2.491610in}{3.207800in}}%
\pgfpathlineto{\pgfqpoint{2.419517in}{3.216002in}}%
\pgfpathlineto{\pgfqpoint{2.348084in}{3.218971in}}%
\pgfpathlineto{\pgfqpoint{2.277597in}{3.216697in}}%
\pgfpathlineto{\pgfqpoint{2.208342in}{3.209188in}}%
\pgfpathlineto{\pgfqpoint{2.140597in}{3.196475in}}%
\pgfpathlineto{\pgfqpoint{2.074634in}{3.178609in}}%
\pgfpathlineto{\pgfqpoint{2.010720in}{3.155661in}}%
\pgfpathlineto{\pgfqpoint{1.949111in}{3.127725in}}%
\pgfpathlineto{\pgfqpoint{1.890056in}{3.094912in}}%
\pgfpathlineto{\pgfqpoint{1.833792in}{3.057355in}}%
\pgfpathlineto{\pgfqpoint{1.780547in}{3.015205in}}%
\pgfpathlineto{\pgfqpoint{1.730534in}{2.968631in}}%
\pgfpathlineto{\pgfqpoint{1.683955in}{2.917822in}}%
\pgfpathlineto{\pgfqpoint{1.640997in}{2.862982in}}%
\pgfpathlineto{\pgfqpoint{1.601833in}{2.804331in}}%
\pgfpathlineto{\pgfqpoint{1.566622in}{2.742106in}}%
\pgfpathlineto{\pgfqpoint{1.535504in}{2.676558in}}%
\pgfpathlineto{\pgfqpoint{1.508606in}{2.607950in}}%
\pgfpathlineto{\pgfqpoint{1.486034in}{2.536558in}}%
\pgfpathlineto{\pgfqpoint{1.467882in}{2.462670in}}%
\pgfpathlineto{\pgfqpoint{1.454220in}{2.386584in}}%
\pgfpathlineto{\pgfqpoint{1.445105in}{2.308606in}}%
\pgfpathlineto{\pgfqpoint{1.440573in}{2.229049in}}%
\pgfpathlineto{\pgfqpoint{1.440643in}{2.148235in}}%
\pgfpathlineto{\pgfqpoint{1.445313in}{2.066488in}}%
\pgfpathlineto{\pgfqpoint{1.454566in}{1.984138in}}%
\pgfpathlineto{\pgfqpoint{1.468364in}{1.901517in}}%
\pgfpathlineto{\pgfqpoint{1.486651in}{1.818956in}}%
\pgfpathlineto{\pgfqpoint{1.509354in}{1.736789in}}%
\pgfpathlineto{\pgfqpoint{1.536382in}{1.655346in}}%
\pgfpathlineto{\pgfqpoint{1.567625in}{1.574955in}}%
\pgfpathlineto{\pgfqpoint{1.602957in}{1.495940in}}%
\pgfpathlineto{\pgfqpoint{1.642238in}{1.418620in}}%
\pgfpathlineto{\pgfqpoint{1.685307in}{1.343304in}}%
\pgfpathlineto{\pgfqpoint{1.731993in}{1.270298in}}%
\pgfpathlineto{\pgfqpoint{1.782106in}{1.199894in}}%
\pgfpathlineto{\pgfqpoint{1.835446in}{1.132376in}}%
\pgfpathlineto{\pgfqpoint{1.891797in}{1.068016in}}%
\pgfpathlineto{\pgfqpoint{1.950933in}{1.007073in}}%
\pgfpathlineto{\pgfqpoint{2.012615in}{0.949792in}}%
\pgfpathlineto{\pgfqpoint{2.076595in}{0.896405in}}%
\pgfpathlineto{\pgfqpoint{2.142615in}{0.847126in}}%
\pgfpathlineto{\pgfqpoint{2.210410in}{0.802153in}}%
\pgfpathlineto{\pgfqpoint{2.279706in}{0.761668in}}%
\pgfpathlineto{\pgfqpoint{2.350226in}{0.725834in}}%
\pgfpathlineto{\pgfqpoint{2.421683in}{0.694795in}}%
\pgfpathlineto{\pgfqpoint{2.493792in}{0.668675in}}%
\pgfpathlineto{\pgfqpoint{2.566261in}{0.647581in}}%
\pgfpathlineto{\pgfqpoint{2.638799in}{0.631596in}}%
\pgfpathlineto{\pgfqpoint{2.711113in}{0.620786in}}%
\pgfpathlineto{\pgfqpoint{2.782913in}{0.615194in}}%
\pgfpathlineto{\pgfqpoint{2.853910in}{0.614842in}}%
\pgfpathlineto{\pgfqpoint{2.923817in}{0.619732in}}%
\pgfpathlineto{\pgfqpoint{2.992352in}{0.629844in}}%
\pgfpathlineto{\pgfqpoint{3.059242in}{0.645138in}}%
\pgfpathlineto{\pgfqpoint{3.124214in}{0.665551in}}%
\pgfpathlineto{\pgfqpoint{3.187009in}{0.691003in}}%
\pgfpathlineto{\pgfqpoint{3.247373in}{0.721389in}}%
\pgfusepath{stroke}%
\end{pgfscope}%
\begin{pgfscope}%
\pgfpathrectangle{\pgfqpoint{0.000000in}{0.000000in}}{\pgfqpoint{4.000000in}{4.000000in}}%
\pgfusepath{clip}%
\pgfsetbuttcap%
\pgfsetroundjoin%
\definecolor{currentfill}{rgb}{0.800000,0.800000,0.800000}%
\pgfsetfillcolor{currentfill}%
\pgfsetlinewidth{1.003750pt}%
\definecolor{currentstroke}{rgb}{0.000000,0.000000,0.000000}%
\pgfsetstrokecolor{currentstroke}%
\pgfsetdash{}{0pt}%
\pgfsys@defobject{currentmarker}{\pgfqpoint{-0.104167in}{-0.104167in}}{\pgfqpoint{0.104167in}{0.104167in}}{%
\pgfpathmoveto{\pgfqpoint{0.000000in}{-0.104167in}}%
\pgfpathcurveto{\pgfqpoint{0.027625in}{-0.104167in}}{\pgfqpoint{0.054123in}{-0.093191in}}{\pgfqpoint{0.073657in}{-0.073657in}}%
\pgfpathcurveto{\pgfqpoint{0.093191in}{-0.054123in}}{\pgfqpoint{0.104167in}{-0.027625in}}{\pgfqpoint{0.104167in}{0.000000in}}%
\pgfpathcurveto{\pgfqpoint{0.104167in}{0.027625in}}{\pgfqpoint{0.093191in}{0.054123in}}{\pgfqpoint{0.073657in}{0.073657in}}%
\pgfpathcurveto{\pgfqpoint{0.054123in}{0.093191in}}{\pgfqpoint{0.027625in}{0.104167in}}{\pgfqpoint{0.000000in}{0.104167in}}%
\pgfpathcurveto{\pgfqpoint{-0.027625in}{0.104167in}}{\pgfqpoint{-0.054123in}{0.093191in}}{\pgfqpoint{-0.073657in}{0.073657in}}%
\pgfpathcurveto{\pgfqpoint{-0.093191in}{0.054123in}}{\pgfqpoint{-0.104167in}{0.027625in}}{\pgfqpoint{-0.104167in}{0.000000in}}%
\pgfpathcurveto{\pgfqpoint{-0.104167in}{-0.027625in}}{\pgfqpoint{-0.093191in}{-0.054123in}}{\pgfqpoint{-0.073657in}{-0.073657in}}%
\pgfpathcurveto{\pgfqpoint{-0.054123in}{-0.093191in}}{\pgfqpoint{-0.027625in}{-0.104167in}}{\pgfqpoint{0.000000in}{-0.104167in}}%
\pgfpathclose%
\pgfusepath{stroke,fill}%
}%
\begin{pgfscope}%
\pgfsys@transformshift{1.333333in}{1.166667in}%
\pgfsys@useobject{currentmarker}{}%
\end{pgfscope}%
\end{pgfscope}%
\begin{pgfscope}%
\pgfpathrectangle{\pgfqpoint{0.000000in}{0.000000in}}{\pgfqpoint{4.000000in}{4.000000in}}%
\pgfusepath{clip}%
\pgfsetbuttcap%
\pgfsetroundjoin%
\definecolor{currentfill}{rgb}{0.800000,0.800000,0.800000}%
\pgfsetfillcolor{currentfill}%
\pgfsetlinewidth{1.003750pt}%
\definecolor{currentstroke}{rgb}{0.000000,0.000000,0.000000}%
\pgfsetstrokecolor{currentstroke}%
\pgfsetdash{}{0pt}%
\pgfsys@defobject{currentmarker}{\pgfqpoint{-0.104167in}{-0.104167in}}{\pgfqpoint{0.104167in}{0.104167in}}{%
\pgfpathmoveto{\pgfqpoint{0.000000in}{-0.104167in}}%
\pgfpathcurveto{\pgfqpoint{0.027625in}{-0.104167in}}{\pgfqpoint{0.054123in}{-0.093191in}}{\pgfqpoint{0.073657in}{-0.073657in}}%
\pgfpathcurveto{\pgfqpoint{0.093191in}{-0.054123in}}{\pgfqpoint{0.104167in}{-0.027625in}}{\pgfqpoint{0.104167in}{0.000000in}}%
\pgfpathcurveto{\pgfqpoint{0.104167in}{0.027625in}}{\pgfqpoint{0.093191in}{0.054123in}}{\pgfqpoint{0.073657in}{0.073657in}}%
\pgfpathcurveto{\pgfqpoint{0.054123in}{0.093191in}}{\pgfqpoint{0.027625in}{0.104167in}}{\pgfqpoint{0.000000in}{0.104167in}}%
\pgfpathcurveto{\pgfqpoint{-0.027625in}{0.104167in}}{\pgfqpoint{-0.054123in}{0.093191in}}{\pgfqpoint{-0.073657in}{0.073657in}}%
\pgfpathcurveto{\pgfqpoint{-0.093191in}{0.054123in}}{\pgfqpoint{-0.104167in}{0.027625in}}{\pgfqpoint{-0.104167in}{0.000000in}}%
\pgfpathcurveto{\pgfqpoint{-0.104167in}{-0.027625in}}{\pgfqpoint{-0.093191in}{-0.054123in}}{\pgfqpoint{-0.073657in}{-0.073657in}}%
\pgfpathcurveto{\pgfqpoint{-0.054123in}{-0.093191in}}{\pgfqpoint{-0.027625in}{-0.104167in}}{\pgfqpoint{0.000000in}{-0.104167in}}%
\pgfpathclose%
\pgfusepath{stroke,fill}%
}%
\begin{pgfscope}%
\pgfsys@transformshift{3.000000in}{1.166667in}%
\pgfsys@useobject{currentmarker}{}%
\end{pgfscope}%
\end{pgfscope}%
\begin{pgfscope}%
\pgfpathrectangle{\pgfqpoint{0.000000in}{0.000000in}}{\pgfqpoint{4.000000in}{4.000000in}}%
\pgfusepath{clip}%
\pgfsetbuttcap%
\pgfsetroundjoin%
\definecolor{currentfill}{rgb}{0.800000,0.800000,0.800000}%
\pgfsetfillcolor{currentfill}%
\pgfsetlinewidth{1.003750pt}%
\definecolor{currentstroke}{rgb}{0.000000,0.000000,0.000000}%
\pgfsetstrokecolor{currentstroke}%
\pgfsetdash{}{0pt}%
\pgfsys@defobject{currentmarker}{\pgfqpoint{-0.104167in}{-0.104167in}}{\pgfqpoint{0.104167in}{0.104167in}}{%
\pgfpathmoveto{\pgfqpoint{0.000000in}{-0.104167in}}%
\pgfpathcurveto{\pgfqpoint{0.027625in}{-0.104167in}}{\pgfqpoint{0.054123in}{-0.093191in}}{\pgfqpoint{0.073657in}{-0.073657in}}%
\pgfpathcurveto{\pgfqpoint{0.093191in}{-0.054123in}}{\pgfqpoint{0.104167in}{-0.027625in}}{\pgfqpoint{0.104167in}{0.000000in}}%
\pgfpathcurveto{\pgfqpoint{0.104167in}{0.027625in}}{\pgfqpoint{0.093191in}{0.054123in}}{\pgfqpoint{0.073657in}{0.073657in}}%
\pgfpathcurveto{\pgfqpoint{0.054123in}{0.093191in}}{\pgfqpoint{0.027625in}{0.104167in}}{\pgfqpoint{0.000000in}{0.104167in}}%
\pgfpathcurveto{\pgfqpoint{-0.027625in}{0.104167in}}{\pgfqpoint{-0.054123in}{0.093191in}}{\pgfqpoint{-0.073657in}{0.073657in}}%
\pgfpathcurveto{\pgfqpoint{-0.093191in}{0.054123in}}{\pgfqpoint{-0.104167in}{0.027625in}}{\pgfqpoint{-0.104167in}{0.000000in}}%
\pgfpathcurveto{\pgfqpoint{-0.104167in}{-0.027625in}}{\pgfqpoint{-0.093191in}{-0.054123in}}{\pgfqpoint{-0.073657in}{-0.073657in}}%
\pgfpathcurveto{\pgfqpoint{-0.054123in}{-0.093191in}}{\pgfqpoint{-0.027625in}{-0.104167in}}{\pgfqpoint{0.000000in}{-0.104167in}}%
\pgfpathclose%
\pgfusepath{stroke,fill}%
}%
\begin{pgfscope}%
\pgfsys@transformshift{2.166667in}{2.666667in}%
\pgfsys@useobject{currentmarker}{}%
\end{pgfscope}%
\end{pgfscope}%
\begin{pgfscope}%
\pgfpathrectangle{\pgfqpoint{0.000000in}{0.000000in}}{\pgfqpoint{4.000000in}{4.000000in}}%
\pgfusepath{clip}%
\pgfsetbuttcap%
\pgfsetroundjoin%
\definecolor{currentfill}{rgb}{0.800000,0.800000,0.800000}%
\pgfsetfillcolor{currentfill}%
\pgfsetlinewidth{1.003750pt}%
\definecolor{currentstroke}{rgb}{0.000000,0.000000,0.000000}%
\pgfsetstrokecolor{currentstroke}%
\pgfsetdash{}{0pt}%
\pgfsys@defobject{currentmarker}{\pgfqpoint{-0.104167in}{-0.104167in}}{\pgfqpoint{0.104167in}{0.104167in}}{%
\pgfpathmoveto{\pgfqpoint{0.000000in}{-0.104167in}}%
\pgfpathcurveto{\pgfqpoint{0.027625in}{-0.104167in}}{\pgfqpoint{0.054123in}{-0.093191in}}{\pgfqpoint{0.073657in}{-0.073657in}}%
\pgfpathcurveto{\pgfqpoint{0.093191in}{-0.054123in}}{\pgfqpoint{0.104167in}{-0.027625in}}{\pgfqpoint{0.104167in}{0.000000in}}%
\pgfpathcurveto{\pgfqpoint{0.104167in}{0.027625in}}{\pgfqpoint{0.093191in}{0.054123in}}{\pgfqpoint{0.073657in}{0.073657in}}%
\pgfpathcurveto{\pgfqpoint{0.054123in}{0.093191in}}{\pgfqpoint{0.027625in}{0.104167in}}{\pgfqpoint{0.000000in}{0.104167in}}%
\pgfpathcurveto{\pgfqpoint{-0.027625in}{0.104167in}}{\pgfqpoint{-0.054123in}{0.093191in}}{\pgfqpoint{-0.073657in}{0.073657in}}%
\pgfpathcurveto{\pgfqpoint{-0.093191in}{0.054123in}}{\pgfqpoint{-0.104167in}{0.027625in}}{\pgfqpoint{-0.104167in}{0.000000in}}%
\pgfpathcurveto{\pgfqpoint{-0.104167in}{-0.027625in}}{\pgfqpoint{-0.093191in}{-0.054123in}}{\pgfqpoint{-0.073657in}{-0.073657in}}%
\pgfpathcurveto{\pgfqpoint{-0.054123in}{-0.093191in}}{\pgfqpoint{-0.027625in}{-0.104167in}}{\pgfqpoint{0.000000in}{-0.104167in}}%
\pgfpathclose%
\pgfusepath{stroke,fill}%
}%
\begin{pgfscope}%
\pgfsys@transformshift{2.901667in}{3.061667in}%
\pgfsys@useobject{currentmarker}{}%
\end{pgfscope}%
\end{pgfscope}%
\begin{pgfscope}%
\pgfpathrectangle{\pgfqpoint{0.000000in}{0.000000in}}{\pgfqpoint{4.000000in}{4.000000in}}%
\pgfusepath{clip}%
\pgfsetbuttcap%
\pgfsetroundjoin%
\definecolor{currentfill}{rgb}{1.000000,1.000000,1.000000}%
\pgfsetfillcolor{currentfill}%
\pgfsetlinewidth{1.003750pt}%
\definecolor{currentstroke}{rgb}{0.000000,0.000000,0.000000}%
\pgfsetstrokecolor{currentstroke}%
\pgfsetdash{}{0pt}%
\pgfsys@defobject{currentmarker}{\pgfqpoint{-0.104167in}{-0.104167in}}{\pgfqpoint{0.104167in}{0.104167in}}{%
\pgfpathmoveto{\pgfqpoint{0.000000in}{-0.104167in}}%
\pgfpathcurveto{\pgfqpoint{0.027625in}{-0.104167in}}{\pgfqpoint{0.054123in}{-0.093191in}}{\pgfqpoint{0.073657in}{-0.073657in}}%
\pgfpathcurveto{\pgfqpoint{0.093191in}{-0.054123in}}{\pgfqpoint{0.104167in}{-0.027625in}}{\pgfqpoint{0.104167in}{0.000000in}}%
\pgfpathcurveto{\pgfqpoint{0.104167in}{0.027625in}}{\pgfqpoint{0.093191in}{0.054123in}}{\pgfqpoint{0.073657in}{0.073657in}}%
\pgfpathcurveto{\pgfqpoint{0.054123in}{0.093191in}}{\pgfqpoint{0.027625in}{0.104167in}}{\pgfqpoint{0.000000in}{0.104167in}}%
\pgfpathcurveto{\pgfqpoint{-0.027625in}{0.104167in}}{\pgfqpoint{-0.054123in}{0.093191in}}{\pgfqpoint{-0.073657in}{0.073657in}}%
\pgfpathcurveto{\pgfqpoint{-0.093191in}{0.054123in}}{\pgfqpoint{-0.104167in}{0.027625in}}{\pgfqpoint{-0.104167in}{0.000000in}}%
\pgfpathcurveto{\pgfqpoint{-0.104167in}{-0.027625in}}{\pgfqpoint{-0.093191in}{-0.054123in}}{\pgfqpoint{-0.073657in}{-0.073657in}}%
\pgfpathcurveto{\pgfqpoint{-0.054123in}{-0.093191in}}{\pgfqpoint{-0.027625in}{-0.104167in}}{\pgfqpoint{0.000000in}{-0.104167in}}%
\pgfpathclose%
\pgfusepath{stroke,fill}%
}%
\begin{pgfscope}%
\pgfsys@transformshift{2.500000in}{0.666667in}%
\pgfsys@useobject{currentmarker}{}%
\end{pgfscope}%
\end{pgfscope}%
\begin{pgfscope}%
\pgfsetroundcap%
\pgfsetroundjoin%
\pgfsetlinewidth{3.011250pt}%
\definecolor{currentstroke}{rgb}{0.250980,0.250980,0.250980}%
\pgfsetstrokecolor{currentstroke}%
\pgfsetdash{}{0pt}%
\pgfpathmoveto{\pgfqpoint{1.472239in}{1.125000in}}%
\pgfpathlineto{\pgfqpoint{2.821800in}{1.125000in}}%
\pgfusepath{stroke}%
\end{pgfscope}%
\begin{pgfscope}%
\pgfsetroundcap%
\pgfsetroundjoin%
\definecolor{currentfill}{rgb}{0.250980,0.250980,0.250980}%
\pgfsetfillcolor{currentfill}%
\pgfsetlinewidth{3.011250pt}%
\definecolor{currentstroke}{rgb}{0.250980,0.250980,0.250980}%
\pgfsetstrokecolor{currentstroke}%
\pgfsetdash{}{0pt}%
\pgfpathmoveto{\pgfqpoint{2.710689in}{1.194444in}}%
\pgfpathlineto{\pgfqpoint{2.821800in}{1.125000in}}%
\pgfpathlineto{\pgfqpoint{2.710689in}{1.055556in}}%
\pgfpathlineto{\pgfqpoint{2.710689in}{1.194444in}}%
\pgfpathclose%
\pgfusepath{stroke,fill}%
\end{pgfscope}%
\begin{pgfscope}%
\pgfsetroundcap%
\pgfsetroundjoin%
\pgfsetlinewidth{3.011250pt}%
\definecolor{currentstroke}{rgb}{0.250980,0.250980,0.250980}%
\pgfsetstrokecolor{currentstroke}%
\pgfsetdash{}{0pt}%
\pgfpathmoveto{\pgfqpoint{1.437226in}{1.267876in}}%
\pgfpathlineto{\pgfqpoint{2.116571in}{2.490698in}}%
\pgfusepath{stroke}%
\end{pgfscope}%
\begin{pgfscope}%
\pgfsetroundcap%
\pgfsetroundjoin%
\definecolor{currentfill}{rgb}{0.250980,0.250980,0.250980}%
\pgfsetfillcolor{currentfill}%
\pgfsetlinewidth{3.011250pt}%
\definecolor{currentstroke}{rgb}{0.250980,0.250980,0.250980}%
\pgfsetstrokecolor{currentstroke}%
\pgfsetdash{}{0pt}%
\pgfpathmoveto{\pgfqpoint{2.001905in}{2.427294in}}%
\pgfpathlineto{\pgfqpoint{2.116571in}{2.490698in}}%
\pgfpathlineto{\pgfqpoint{2.123316in}{2.359844in}}%
\pgfpathlineto{\pgfqpoint{2.001905in}{2.427294in}}%
\pgfpathclose%
\pgfusepath{stroke,fill}%
\end{pgfscope}%
\begin{pgfscope}%
\pgfsetroundcap%
\pgfsetroundjoin%
\pgfsetlinewidth{3.011250pt}%
\definecolor{currentstroke}{rgb}{0.250980,0.250980,0.250980}%
\pgfsetstrokecolor{currentstroke}%
\pgfsetdash{}{0pt}%
\pgfpathmoveto{\pgfqpoint{1.421901in}{1.273683in}}%
\pgfpathlineto{\pgfqpoint{2.788071in}{2.924410in}}%
\pgfusepath{stroke}%
\end{pgfscope}%
\begin{pgfscope}%
\pgfsetroundcap%
\pgfsetroundjoin%
\definecolor{currentfill}{rgb}{0.250980,0.250980,0.250980}%
\pgfsetfillcolor{currentfill}%
\pgfsetlinewidth{3.011250pt}%
\definecolor{currentstroke}{rgb}{0.250980,0.250980,0.250980}%
\pgfsetstrokecolor{currentstroke}%
\pgfsetdash{}{0pt}%
\pgfpathmoveto{\pgfqpoint{2.663729in}{2.883088in}}%
\pgfpathlineto{\pgfqpoint{2.788071in}{2.924410in}}%
\pgfpathlineto{\pgfqpoint{2.770727in}{2.794535in}}%
\pgfpathlineto{\pgfqpoint{2.663729in}{2.883088in}}%
\pgfpathclose%
\pgfusepath{stroke,fill}%
\end{pgfscope}%
\begin{pgfscope}%
\pgfsetroundcap%
\pgfsetroundjoin%
\pgfsetlinewidth{3.011250pt}%
\definecolor{currentstroke}{rgb}{0.250980,0.250980,0.250980}%
\pgfsetstrokecolor{currentstroke}%
\pgfsetdash{}{0pt}%
\pgfpathmoveto{\pgfqpoint{2.861094in}{1.208333in}}%
\pgfpathlineto{\pgfqpoint{1.511533in}{1.208333in}}%
\pgfusepath{stroke}%
\end{pgfscope}%
\begin{pgfscope}%
\pgfsetroundcap%
\pgfsetroundjoin%
\definecolor{currentfill}{rgb}{0.250980,0.250980,0.250980}%
\pgfsetfillcolor{currentfill}%
\pgfsetlinewidth{3.011250pt}%
\definecolor{currentstroke}{rgb}{0.250980,0.250980,0.250980}%
\pgfsetstrokecolor{currentstroke}%
\pgfsetdash{}{0pt}%
\pgfpathmoveto{\pgfqpoint{1.622644in}{1.138889in}}%
\pgfpathlineto{\pgfqpoint{1.511533in}{1.208333in}}%
\pgfpathlineto{\pgfqpoint{1.622644in}{1.277778in}}%
\pgfpathlineto{\pgfqpoint{1.622644in}{1.138889in}}%
\pgfpathclose%
\pgfusepath{stroke,fill}%
\end{pgfscope}%
\begin{pgfscope}%
\pgfsetroundcap%
\pgfsetroundjoin%
\pgfsetlinewidth{3.011250pt}%
\definecolor{currentstroke}{rgb}{0.250980,0.250980,0.250980}%
\pgfsetstrokecolor{currentstroke}%
\pgfsetdash{}{0pt}%
\pgfpathmoveto{\pgfqpoint{2.968954in}{1.308346in}}%
\pgfpathlineto{\pgfqpoint{2.289609in}{2.531168in}}%
\pgfusepath{stroke}%
\end{pgfscope}%
\begin{pgfscope}%
\pgfsetroundcap%
\pgfsetroundjoin%
\definecolor{currentfill}{rgb}{0.250980,0.250980,0.250980}%
\pgfsetfillcolor{currentfill}%
\pgfsetlinewidth{3.011250pt}%
\definecolor{currentstroke}{rgb}{0.250980,0.250980,0.250980}%
\pgfsetstrokecolor{currentstroke}%
\pgfsetdash{}{0pt}%
\pgfpathmoveto{\pgfqpoint{2.282864in}{2.400314in}}%
\pgfpathlineto{\pgfqpoint{2.289609in}{2.531168in}}%
\pgfpathlineto{\pgfqpoint{2.404274in}{2.467764in}}%
\pgfpathlineto{\pgfqpoint{2.282864in}{2.400314in}}%
\pgfpathclose%
\pgfusepath{stroke,fill}%
\end{pgfscope}%
\begin{pgfscope}%
\pgfsetroundcap%
\pgfsetroundjoin%
\pgfsetlinewidth{3.011250pt}%
\definecolor{currentstroke}{rgb}{0.250980,0.250980,0.250980}%
\pgfsetstrokecolor{currentstroke}%
\pgfsetdash{}{0pt}%
\pgfpathmoveto{\pgfqpoint{2.992801in}{1.305403in}}%
\pgfpathlineto{\pgfqpoint{2.910899in}{2.883756in}}%
\pgfusepath{stroke}%
\end{pgfscope}%
\begin{pgfscope}%
\pgfsetroundcap%
\pgfsetroundjoin%
\definecolor{currentfill}{rgb}{0.250980,0.250980,0.250980}%
\pgfsetfillcolor{currentfill}%
\pgfsetlinewidth{3.011250pt}%
\definecolor{currentstroke}{rgb}{0.250980,0.250980,0.250980}%
\pgfsetstrokecolor{currentstroke}%
\pgfsetdash{}{0pt}%
\pgfpathmoveto{\pgfqpoint{2.847305in}{2.769195in}}%
\pgfpathlineto{\pgfqpoint{2.910899in}{2.883756in}}%
\pgfpathlineto{\pgfqpoint{2.986008in}{2.776393in}}%
\pgfpathlineto{\pgfqpoint{2.847305in}{2.769195in}}%
\pgfpathclose%
\pgfusepath{stroke,fill}%
\end{pgfscope}%
\begin{pgfscope}%
\pgfsetroundcap%
\pgfsetroundjoin%
\pgfsetlinewidth{3.011250pt}%
\definecolor{currentstroke}{rgb}{0.250980,0.250980,0.250980}%
\pgfsetstrokecolor{currentstroke}%
\pgfsetdash{}{0pt}%
\pgfpathmoveto{\pgfqpoint{2.062774in}{2.565457in}}%
\pgfpathlineto{\pgfqpoint{1.383429in}{1.342636in}}%
\pgfusepath{stroke}%
\end{pgfscope}%
\begin{pgfscope}%
\pgfsetroundcap%
\pgfsetroundjoin%
\definecolor{currentfill}{rgb}{0.250980,0.250980,0.250980}%
\pgfsetfillcolor{currentfill}%
\pgfsetlinewidth{3.011250pt}%
\definecolor{currentstroke}{rgb}{0.250980,0.250980,0.250980}%
\pgfsetstrokecolor{currentstroke}%
\pgfsetdash{}{0pt}%
\pgfpathmoveto{\pgfqpoint{1.498095in}{1.406039in}}%
\pgfpathlineto{\pgfqpoint{1.383429in}{1.342636in}}%
\pgfpathlineto{\pgfqpoint{1.376684in}{1.473490in}}%
\pgfpathlineto{\pgfqpoint{1.498095in}{1.406039in}}%
\pgfpathclose%
\pgfusepath{stroke,fill}%
\end{pgfscope}%
\begin{pgfscope}%
\pgfsetroundcap%
\pgfsetroundjoin%
\pgfsetlinewidth{3.011250pt}%
\definecolor{currentstroke}{rgb}{0.250980,0.250980,0.250980}%
\pgfsetstrokecolor{currentstroke}%
\pgfsetdash{}{0pt}%
\pgfpathmoveto{\pgfqpoint{2.197713in}{2.524987in}}%
\pgfpathlineto{\pgfqpoint{2.877058in}{1.302166in}}%
\pgfusepath{stroke}%
\end{pgfscope}%
\begin{pgfscope}%
\pgfsetroundcap%
\pgfsetroundjoin%
\definecolor{currentfill}{rgb}{0.250980,0.250980,0.250980}%
\pgfsetfillcolor{currentfill}%
\pgfsetlinewidth{3.011250pt}%
\definecolor{currentstroke}{rgb}{0.250980,0.250980,0.250980}%
\pgfsetstrokecolor{currentstroke}%
\pgfsetdash{}{0pt}%
\pgfpathmoveto{\pgfqpoint{2.883803in}{1.433019in}}%
\pgfpathlineto{\pgfqpoint{2.877058in}{1.302166in}}%
\pgfpathlineto{\pgfqpoint{2.762392in}{1.365569in}}%
\pgfpathlineto{\pgfqpoint{2.883803in}{1.433019in}}%
\pgfpathclose%
\pgfusepath{stroke,fill}%
\end{pgfscope}%
\begin{pgfscope}%
\pgfsetroundcap%
\pgfsetroundjoin%
\pgfsetlinewidth{3.011250pt}%
\definecolor{currentstroke}{rgb}{0.250980,0.250980,0.250980}%
\pgfsetstrokecolor{currentstroke}%
\pgfsetdash{}{0pt}%
\pgfpathmoveto{\pgfqpoint{2.779331in}{2.995922in}}%
\pgfpathlineto{\pgfqpoint{2.323608in}{2.751009in}}%
\pgfusepath{stroke}%
\end{pgfscope}%
\begin{pgfscope}%
\pgfsetroundcap%
\pgfsetroundjoin%
\definecolor{currentfill}{rgb}{0.250980,0.250980,0.250980}%
\pgfsetfillcolor{currentfill}%
\pgfsetlinewidth{3.011250pt}%
\definecolor{currentstroke}{rgb}{0.250980,0.250980,0.250980}%
\pgfsetstrokecolor{currentstroke}%
\pgfsetdash{}{0pt}%
\pgfpathmoveto{\pgfqpoint{2.454354in}{2.742437in}}%
\pgfpathlineto{\pgfqpoint{2.323608in}{2.751009in}}%
\pgfpathlineto{\pgfqpoint{2.388606in}{2.864778in}}%
\pgfpathlineto{\pgfqpoint{2.454354in}{2.742437in}}%
\pgfpathclose%
\pgfusepath{stroke,fill}%
\end{pgfscope}%
\begin{pgfscope}%
\definecolor{textcolor}{rgb}{0.000000,0.000000,0.000000}%
\pgfsetstrokecolor{textcolor}%
\pgfsetfillcolor{textcolor}%
\pgftext[x=0.979780in,y=0.813113in,,]{\color{textcolor}\sffamily\fontsize{40.000000}{48.000000}\selectfont \(\displaystyle x_1\)}%
\end{pgfscope}%
\begin{pgfscope}%
\definecolor{textcolor}{rgb}{0.000000,0.000000,0.000000}%
\pgfsetstrokecolor{textcolor}%
\pgfsetfillcolor{textcolor}%
\pgftext[x=3.482963in,y=1.037257in,,]{\color{textcolor}\sffamily\fontsize{40.000000}{48.000000}\selectfont \(\displaystyle x_2\)}%
\end{pgfscope}%
\begin{pgfscope}%
\definecolor{textcolor}{rgb}{0.000000,0.000000,0.000000}%
\pgfsetstrokecolor{textcolor}%
\pgfsetfillcolor{textcolor}%
\pgftext[x=1.813113in,y=3.020220in,,]{\color{textcolor}\sffamily\fontsize{40.000000}{48.000000}\selectfont \(\displaystyle x_3\)}%
\end{pgfscope}%
\begin{pgfscope}%
\definecolor{textcolor}{rgb}{0.000000,0.000000,0.000000}%
\pgfsetstrokecolor{textcolor}%
\pgfsetfillcolor{textcolor}%
\pgftext[x=3.401667in,y=3.061667in,,]{\color{textcolor}\sffamily\fontsize{40.000000}{48.000000}\selectfont \(\displaystyle x_4\)}%
\end{pgfscope}%
\end{pgfpicture}%
\makeatother%
\endgroup%

%% file: fig/unidimensional_frameworks_graph.pgf
%% Creator: Matplotlib, PGF backend
%%
%% To include the figure in your LaTeX document, write
%%   \input{<filename>.pgf}
%%
%% Make sure the required packages are loaded in your preamble
%%   \usepackage{pgf}
%%
%% Figures using additional raster images can only be included by \input if
%% they are in the same directory as the main LaTeX file. For loading figures
%% from other directories you can use the `import` package
%%   \usepackage{import}
%%
%% and then include the figures with
%%   \import{<path to file>}{<filename>.pgf}
%%
%% Matplotlib used the following preamble
%%   \usepackage{fontspec}
%%   \setmainfont{DejaVuSerif.ttf}[Path=\detokenize{C:/Users/ccros/Anaconda3/Lib/site-packages/matplotlib/mpl-data/fonts/ttf/}]
%%   \setsansfont{DejaVuSans.ttf}[Path=\detokenize{C:/Users/ccros/Anaconda3/Lib/site-packages/matplotlib/mpl-data/fonts/ttf/}]
%%   \setmonofont{DejaVuSansMono.ttf}[Path=\detokenize{C:/Users/ccros/Anaconda3/Lib/site-packages/matplotlib/mpl-data/fonts/ttf/}]
%%
\begingroup%
\makeatletter%
\begin{pgfpicture}%
\pgfpathrectangle{\pgfpointorigin}{\pgfqpoint{4.000000in}{4.000000in}}%
\pgfusepath{use as bounding box, clip}%
\begin{pgfscope}%
\pgfsetbuttcap%
\pgfsetmiterjoin%
\pgfsetlinewidth{0.000000pt}%
\definecolor{currentstroke}{rgb}{1.000000,1.000000,1.000000}%
\pgfsetstrokecolor{currentstroke}%
\pgfsetstrokeopacity{0.000000}%
\pgfsetdash{}{0pt}%
\pgfpathmoveto{\pgfqpoint{0.000000in}{0.000000in}}%
\pgfpathlineto{\pgfqpoint{4.000000in}{0.000000in}}%
\pgfpathlineto{\pgfqpoint{4.000000in}{4.000000in}}%
\pgfpathlineto{\pgfqpoint{0.000000in}{4.000000in}}%
\pgfpathclose%
\pgfusepath{}%
\end{pgfscope}%
\begin{pgfscope}%
\pgfpathrectangle{\pgfqpoint{0.000000in}{0.000000in}}{\pgfqpoint{4.000000in}{4.000000in}}%
\pgfusepath{clip}%
\pgfsetbuttcap%
\pgfsetroundjoin%
\definecolor{currentfill}{rgb}{0.800000,0.800000,0.800000}%
\pgfsetfillcolor{currentfill}%
\pgfsetlinewidth{1.003750pt}%
\definecolor{currentstroke}{rgb}{0.000000,0.000000,0.000000}%
\pgfsetstrokecolor{currentstroke}%
\pgfsetdash{}{0pt}%
\pgfsys@defobject{currentmarker}{\pgfqpoint{-0.104167in}{-0.104167in}}{\pgfqpoint{0.104167in}{0.104167in}}{%
\pgfpathmoveto{\pgfqpoint{0.000000in}{-0.104167in}}%
\pgfpathcurveto{\pgfqpoint{0.027625in}{-0.104167in}}{\pgfqpoint{0.054123in}{-0.093191in}}{\pgfqpoint{0.073657in}{-0.073657in}}%
\pgfpathcurveto{\pgfqpoint{0.093191in}{-0.054123in}}{\pgfqpoint{0.104167in}{-0.027625in}}{\pgfqpoint{0.104167in}{0.000000in}}%
\pgfpathcurveto{\pgfqpoint{0.104167in}{0.027625in}}{\pgfqpoint{0.093191in}{0.054123in}}{\pgfqpoint{0.073657in}{0.073657in}}%
\pgfpathcurveto{\pgfqpoint{0.054123in}{0.093191in}}{\pgfqpoint{0.027625in}{0.104167in}}{\pgfqpoint{0.000000in}{0.104167in}}%
\pgfpathcurveto{\pgfqpoint{-0.027625in}{0.104167in}}{\pgfqpoint{-0.054123in}{0.093191in}}{\pgfqpoint{-0.073657in}{0.073657in}}%
\pgfpathcurveto{\pgfqpoint{-0.093191in}{0.054123in}}{\pgfqpoint{-0.104167in}{0.027625in}}{\pgfqpoint{-0.104167in}{0.000000in}}%
\pgfpathcurveto{\pgfqpoint{-0.104167in}{-0.027625in}}{\pgfqpoint{-0.093191in}{-0.054123in}}{\pgfqpoint{-0.073657in}{-0.073657in}}%
\pgfpathcurveto{\pgfqpoint{-0.054123in}{-0.093191in}}{\pgfqpoint{-0.027625in}{-0.104167in}}{\pgfqpoint{0.000000in}{-0.104167in}}%
\pgfpathclose%
\pgfusepath{stroke,fill}%
}%
\begin{pgfscope}%
\pgfsys@transformshift{1.000000in}{3.000000in}%
\pgfsys@useobject{currentmarker}{}%
\end{pgfscope}%
\end{pgfscope}%
\begin{pgfscope}%
\pgfpathrectangle{\pgfqpoint{0.000000in}{0.000000in}}{\pgfqpoint{4.000000in}{4.000000in}}%
\pgfusepath{clip}%
\pgfsetbuttcap%
\pgfsetroundjoin%
\definecolor{currentfill}{rgb}{0.800000,0.800000,0.800000}%
\pgfsetfillcolor{currentfill}%
\pgfsetlinewidth{1.003750pt}%
\definecolor{currentstroke}{rgb}{0.000000,0.000000,0.000000}%
\pgfsetstrokecolor{currentstroke}%
\pgfsetdash{}{0pt}%
\pgfsys@defobject{currentmarker}{\pgfqpoint{-0.104167in}{-0.104167in}}{\pgfqpoint{0.104167in}{0.104167in}}{%
\pgfpathmoveto{\pgfqpoint{0.000000in}{-0.104167in}}%
\pgfpathcurveto{\pgfqpoint{0.027625in}{-0.104167in}}{\pgfqpoint{0.054123in}{-0.093191in}}{\pgfqpoint{0.073657in}{-0.073657in}}%
\pgfpathcurveto{\pgfqpoint{0.093191in}{-0.054123in}}{\pgfqpoint{0.104167in}{-0.027625in}}{\pgfqpoint{0.104167in}{0.000000in}}%
\pgfpathcurveto{\pgfqpoint{0.104167in}{0.027625in}}{\pgfqpoint{0.093191in}{0.054123in}}{\pgfqpoint{0.073657in}{0.073657in}}%
\pgfpathcurveto{\pgfqpoint{0.054123in}{0.093191in}}{\pgfqpoint{0.027625in}{0.104167in}}{\pgfqpoint{0.000000in}{0.104167in}}%
\pgfpathcurveto{\pgfqpoint{-0.027625in}{0.104167in}}{\pgfqpoint{-0.054123in}{0.093191in}}{\pgfqpoint{-0.073657in}{0.073657in}}%
\pgfpathcurveto{\pgfqpoint{-0.093191in}{0.054123in}}{\pgfqpoint{-0.104167in}{0.027625in}}{\pgfqpoint{-0.104167in}{0.000000in}}%
\pgfpathcurveto{\pgfqpoint{-0.104167in}{-0.027625in}}{\pgfqpoint{-0.093191in}{-0.054123in}}{\pgfqpoint{-0.073657in}{-0.073657in}}%
\pgfpathcurveto{\pgfqpoint{-0.054123in}{-0.093191in}}{\pgfqpoint{-0.027625in}{-0.104167in}}{\pgfqpoint{0.000000in}{-0.104167in}}%
\pgfpathclose%
\pgfusepath{stroke,fill}%
}%
\begin{pgfscope}%
\pgfsys@transformshift{1.000000in}{1.000000in}%
\pgfsys@useobject{currentmarker}{}%
\end{pgfscope}%
\end{pgfscope}%
\begin{pgfscope}%
\pgfpathrectangle{\pgfqpoint{0.000000in}{0.000000in}}{\pgfqpoint{4.000000in}{4.000000in}}%
\pgfusepath{clip}%
\pgfsetbuttcap%
\pgfsetroundjoin%
\definecolor{currentfill}{rgb}{0.800000,0.800000,0.800000}%
\pgfsetfillcolor{currentfill}%
\pgfsetlinewidth{1.003750pt}%
\definecolor{currentstroke}{rgb}{0.000000,0.000000,0.000000}%
\pgfsetstrokecolor{currentstroke}%
\pgfsetdash{}{0pt}%
\pgfsys@defobject{currentmarker}{\pgfqpoint{-0.104167in}{-0.104167in}}{\pgfqpoint{0.104167in}{0.104167in}}{%
\pgfpathmoveto{\pgfqpoint{0.000000in}{-0.104167in}}%
\pgfpathcurveto{\pgfqpoint{0.027625in}{-0.104167in}}{\pgfqpoint{0.054123in}{-0.093191in}}{\pgfqpoint{0.073657in}{-0.073657in}}%
\pgfpathcurveto{\pgfqpoint{0.093191in}{-0.054123in}}{\pgfqpoint{0.104167in}{-0.027625in}}{\pgfqpoint{0.104167in}{0.000000in}}%
\pgfpathcurveto{\pgfqpoint{0.104167in}{0.027625in}}{\pgfqpoint{0.093191in}{0.054123in}}{\pgfqpoint{0.073657in}{0.073657in}}%
\pgfpathcurveto{\pgfqpoint{0.054123in}{0.093191in}}{\pgfqpoint{0.027625in}{0.104167in}}{\pgfqpoint{0.000000in}{0.104167in}}%
\pgfpathcurveto{\pgfqpoint{-0.027625in}{0.104167in}}{\pgfqpoint{-0.054123in}{0.093191in}}{\pgfqpoint{-0.073657in}{0.073657in}}%
\pgfpathcurveto{\pgfqpoint{-0.093191in}{0.054123in}}{\pgfqpoint{-0.104167in}{0.027625in}}{\pgfqpoint{-0.104167in}{0.000000in}}%
\pgfpathcurveto{\pgfqpoint{-0.104167in}{-0.027625in}}{\pgfqpoint{-0.093191in}{-0.054123in}}{\pgfqpoint{-0.073657in}{-0.073657in}}%
\pgfpathcurveto{\pgfqpoint{-0.054123in}{-0.093191in}}{\pgfqpoint{-0.027625in}{-0.104167in}}{\pgfqpoint{0.000000in}{-0.104167in}}%
\pgfpathclose%
\pgfusepath{stroke,fill}%
}%
\begin{pgfscope}%
\pgfsys@transformshift{2.740000in}{2.000000in}%
\pgfsys@useobject{currentmarker}{}%
\end{pgfscope}%
\end{pgfscope}%
\begin{pgfscope}%
\pgfsetroundcap%
\pgfsetroundjoin%
\pgfsetlinewidth{3.011250pt}%
\definecolor{currentstroke}{rgb}{0.250980,0.250980,0.250980}%
\pgfsetstrokecolor{currentstroke}%
\pgfsetdash{}{0pt}%
\pgfpathmoveto{\pgfqpoint{0.950000in}{2.861145in}}%
\pgfpathlineto{\pgfqpoint{0.950000in}{1.178178in}}%
\pgfusepath{stroke}%
\end{pgfscope}%
\begin{pgfscope}%
\pgfsetroundcap%
\pgfsetroundjoin%
\definecolor{currentfill}{rgb}{0.250980,0.250980,0.250980}%
\pgfsetfillcolor{currentfill}%
\pgfsetlinewidth{3.011250pt}%
\definecolor{currentstroke}{rgb}{0.250980,0.250980,0.250980}%
\pgfsetstrokecolor{currentstroke}%
\pgfsetdash{}{0pt}%
\pgfpathmoveto{\pgfqpoint{1.019444in}{1.289290in}}%
\pgfpathlineto{\pgfqpoint{0.950000in}{1.178178in}}%
\pgfpathlineto{\pgfqpoint{0.880556in}{1.289290in}}%
\pgfpathlineto{\pgfqpoint{1.019444in}{1.289290in}}%
\pgfpathclose%
\pgfusepath{stroke,fill}%
\end{pgfscope}%
\begin{pgfscope}%
\pgfsetroundcap%
\pgfsetroundjoin%
\pgfsetlinewidth{3.011250pt}%
\definecolor{currentstroke}{rgb}{0.250980,0.250980,0.250980}%
\pgfsetstrokecolor{currentstroke}%
\pgfsetdash{}{0pt}%
\pgfpathmoveto{\pgfqpoint{1.120379in}{2.930817in}}%
\pgfpathlineto{\pgfqpoint{2.585466in}{2.088813in}}%
\pgfusepath{stroke}%
\end{pgfscope}%
\begin{pgfscope}%
\pgfsetroundcap%
\pgfsetroundjoin%
\definecolor{currentfill}{rgb}{0.250980,0.250980,0.250980}%
\pgfsetfillcolor{currentfill}%
\pgfsetlinewidth{3.011250pt}%
\definecolor{currentstroke}{rgb}{0.250980,0.250980,0.250980}%
\pgfsetstrokecolor{currentstroke}%
\pgfsetdash{}{0pt}%
\pgfpathmoveto{\pgfqpoint{2.523734in}{2.204387in}}%
\pgfpathlineto{\pgfqpoint{2.585466in}{2.088813in}}%
\pgfpathlineto{\pgfqpoint{2.454528in}{2.083968in}}%
\pgfpathlineto{\pgfqpoint{2.523734in}{2.204387in}}%
\pgfpathclose%
\pgfusepath{stroke,fill}%
\end{pgfscope}%
\begin{pgfscope}%
\pgfsetroundcap%
\pgfsetroundjoin%
\pgfsetlinewidth{3.011250pt}%
\definecolor{currentstroke}{rgb}{0.250980,0.250980,0.250980}%
\pgfsetstrokecolor{currentstroke}%
\pgfsetdash{}{0pt}%
\pgfpathmoveto{\pgfqpoint{1.050000in}{1.138855in}}%
\pgfpathlineto{\pgfqpoint{1.050000in}{2.821822in}}%
\pgfusepath{stroke}%
\end{pgfscope}%
\begin{pgfscope}%
\pgfsetroundcap%
\pgfsetroundjoin%
\definecolor{currentfill}{rgb}{0.250980,0.250980,0.250980}%
\pgfsetfillcolor{currentfill}%
\pgfsetlinewidth{3.011250pt}%
\definecolor{currentstroke}{rgb}{0.250980,0.250980,0.250980}%
\pgfsetstrokecolor{currentstroke}%
\pgfsetdash{}{0pt}%
\pgfpathmoveto{\pgfqpoint{0.980556in}{2.710710in}}%
\pgfpathlineto{\pgfqpoint{1.050000in}{2.821822in}}%
\pgfpathlineto{\pgfqpoint{1.119444in}{2.710710in}}%
\pgfpathlineto{\pgfqpoint{0.980556in}{2.710710in}}%
\pgfpathclose%
\pgfusepath{stroke,fill}%
\end{pgfscope}%
\begin{pgfscope}%
\pgfsetroundcap%
\pgfsetroundjoin%
\pgfsetlinewidth{3.011250pt}%
\definecolor{currentstroke}{rgb}{0.250980,0.250980,0.250980}%
\pgfsetstrokecolor{currentstroke}%
\pgfsetdash{}{0pt}%
\pgfpathmoveto{\pgfqpoint{1.120379in}{1.069183in}}%
\pgfpathlineto{\pgfqpoint{2.585466in}{1.911187in}}%
\pgfusepath{stroke}%
\end{pgfscope}%
\begin{pgfscope}%
\pgfsetroundcap%
\pgfsetroundjoin%
\definecolor{currentfill}{rgb}{0.250980,0.250980,0.250980}%
\pgfsetfillcolor{currentfill}%
\pgfsetlinewidth{3.011250pt}%
\definecolor{currentstroke}{rgb}{0.250980,0.250980,0.250980}%
\pgfsetstrokecolor{currentstroke}%
\pgfsetdash{}{0pt}%
\pgfpathmoveto{\pgfqpoint{2.454528in}{1.916032in}}%
\pgfpathlineto{\pgfqpoint{2.585466in}{1.911187in}}%
\pgfpathlineto{\pgfqpoint{2.523734in}{1.795613in}}%
\pgfpathlineto{\pgfqpoint{2.454528in}{1.916032in}}%
\pgfpathclose%
\pgfusepath{stroke,fill}%
\end{pgfscope}%
\begin{pgfscope}%
\definecolor{textcolor}{rgb}{0.000000,0.000000,0.000000}%
\pgfsetstrokecolor{textcolor}%
\pgfsetfillcolor{textcolor}%
\pgftext[x=0.800000in,y=3.346410in,,]{\color{textcolor}\sffamily\fontsize{40.000000}{48.000000}\selectfont \(\displaystyle 1\)}%
\end{pgfscope}%
\begin{pgfscope}%
\definecolor{textcolor}{rgb}{0.000000,0.000000,0.000000}%
\pgfsetstrokecolor{textcolor}%
\pgfsetfillcolor{textcolor}%
\pgftext[x=0.800000in,y=0.653590in,,]{\color{textcolor}\sffamily\fontsize{40.000000}{48.000000}\selectfont \(\displaystyle 2\)}%
\end{pgfscope}%
\begin{pgfscope}%
\definecolor{textcolor}{rgb}{0.000000,0.000000,0.000000}%
\pgfsetstrokecolor{textcolor}%
\pgfsetfillcolor{textcolor}%
\pgftext[x=3.140000in,y=2.000000in,,]{\color{textcolor}\sffamily\fontsize{40.000000}{48.000000}\selectfont \(\displaystyle 3\)}%
\end{pgfscope}%
\end{pgfpicture}%
\makeatother%
\endgroup%

%% file: fig/unidimensional_framework_1.pgf
%% Creator: Matplotlib, PGF backend
%%
%% To include the figure in your LaTeX document, write
%%   \input{<filename>.pgf}
%%
%% Make sure the required packages are loaded in your preamble
%%   \usepackage{pgf}
%%
%% Figures using additional raster images can only be included by \input if
%% they are in the same directory as the main LaTeX file. For loading figures
%% from other directories you can use the `import` package
%%   \usepackage{import}
%%
%% and then include the figures with
%%   \import{<path to file>}{<filename>.pgf}
%%
%% Matplotlib used the following preamble
%%   \usepackage{fontspec}
%%   \setmainfont{DejaVuSerif.ttf}[Path=\detokenize{C:/Users/ccros/Anaconda3/Lib/site-packages/matplotlib/mpl-data/fonts/ttf/}]
%%   \setsansfont{DejaVuSans.ttf}[Path=\detokenize{C:/Users/ccros/Anaconda3/Lib/site-packages/matplotlib/mpl-data/fonts/ttf/}]
%%   \setmonofont{DejaVuSansMono.ttf}[Path=\detokenize{C:/Users/ccros/Anaconda3/Lib/site-packages/matplotlib/mpl-data/fonts/ttf/}]
%%
\begingroup%
\makeatletter%
\begin{pgfpicture}%
\pgfpathrectangle{\pgfpointorigin}{\pgfqpoint{4.000000in}{4.000000in}}%
\pgfusepath{use as bounding box, clip}%
\begin{pgfscope}%
\pgfsetbuttcap%
\pgfsetmiterjoin%
\pgfsetlinewidth{0.000000pt}%
\definecolor{currentstroke}{rgb}{1.000000,1.000000,1.000000}%
\pgfsetstrokecolor{currentstroke}%
\pgfsetstrokeopacity{0.000000}%
\pgfsetdash{}{0pt}%
\pgfpathmoveto{\pgfqpoint{0.000000in}{0.000000in}}%
\pgfpathlineto{\pgfqpoint{4.000000in}{0.000000in}}%
\pgfpathlineto{\pgfqpoint{4.000000in}{4.000000in}}%
\pgfpathlineto{\pgfqpoint{0.000000in}{4.000000in}}%
\pgfpathclose%
\pgfusepath{}%
\end{pgfscope}%
\begin{pgfscope}%
\pgfpathrectangle{\pgfqpoint{0.000000in}{0.000000in}}{\pgfqpoint{4.000000in}{4.000000in}}%
\pgfusepath{clip}%
\pgfsetbuttcap%
\pgfsetroundjoin%
\definecolor{currentfill}{rgb}{0.800000,0.800000,0.800000}%
\pgfsetfillcolor{currentfill}%
\pgfsetlinewidth{1.003750pt}%
\definecolor{currentstroke}{rgb}{0.000000,0.000000,0.000000}%
\pgfsetstrokecolor{currentstroke}%
\pgfsetdash{}{0pt}%
\pgfsys@defobject{currentmarker}{\pgfqpoint{-0.104167in}{-0.104167in}}{\pgfqpoint{0.104167in}{0.104167in}}{%
\pgfpathmoveto{\pgfqpoint{0.000000in}{-0.104167in}}%
\pgfpathcurveto{\pgfqpoint{0.027625in}{-0.104167in}}{\pgfqpoint{0.054123in}{-0.093191in}}{\pgfqpoint{0.073657in}{-0.073657in}}%
\pgfpathcurveto{\pgfqpoint{0.093191in}{-0.054123in}}{\pgfqpoint{0.104167in}{-0.027625in}}{\pgfqpoint{0.104167in}{0.000000in}}%
\pgfpathcurveto{\pgfqpoint{0.104167in}{0.027625in}}{\pgfqpoint{0.093191in}{0.054123in}}{\pgfqpoint{0.073657in}{0.073657in}}%
\pgfpathcurveto{\pgfqpoint{0.054123in}{0.093191in}}{\pgfqpoint{0.027625in}{0.104167in}}{\pgfqpoint{0.000000in}{0.104167in}}%
\pgfpathcurveto{\pgfqpoint{-0.027625in}{0.104167in}}{\pgfqpoint{-0.054123in}{0.093191in}}{\pgfqpoint{-0.073657in}{0.073657in}}%
\pgfpathcurveto{\pgfqpoint{-0.093191in}{0.054123in}}{\pgfqpoint{-0.104167in}{0.027625in}}{\pgfqpoint{-0.104167in}{0.000000in}}%
\pgfpathcurveto{\pgfqpoint{-0.104167in}{-0.027625in}}{\pgfqpoint{-0.093191in}{-0.054123in}}{\pgfqpoint{-0.073657in}{-0.073657in}}%
\pgfpathcurveto{\pgfqpoint{-0.054123in}{-0.093191in}}{\pgfqpoint{-0.027625in}{-0.104167in}}{\pgfqpoint{0.000000in}{-0.104167in}}%
\pgfpathclose%
\pgfusepath{stroke,fill}%
}%
\begin{pgfscope}%
\pgfsys@transformshift{0.800000in}{2.800000in}%
\pgfsys@useobject{currentmarker}{}%
\end{pgfscope}%
\end{pgfscope}%
\begin{pgfscope}%
\pgfpathrectangle{\pgfqpoint{0.000000in}{0.000000in}}{\pgfqpoint{4.000000in}{4.000000in}}%
\pgfusepath{clip}%
\pgfsetbuttcap%
\pgfsetroundjoin%
\definecolor{currentfill}{rgb}{0.800000,0.800000,0.800000}%
\pgfsetfillcolor{currentfill}%
\pgfsetlinewidth{1.003750pt}%
\definecolor{currentstroke}{rgb}{0.000000,0.000000,0.000000}%
\pgfsetstrokecolor{currentstroke}%
\pgfsetdash{}{0pt}%
\pgfsys@defobject{currentmarker}{\pgfqpoint{-0.104167in}{-0.104167in}}{\pgfqpoint{0.104167in}{0.104167in}}{%
\pgfpathmoveto{\pgfqpoint{0.000000in}{-0.104167in}}%
\pgfpathcurveto{\pgfqpoint{0.027625in}{-0.104167in}}{\pgfqpoint{0.054123in}{-0.093191in}}{\pgfqpoint{0.073657in}{-0.073657in}}%
\pgfpathcurveto{\pgfqpoint{0.093191in}{-0.054123in}}{\pgfqpoint{0.104167in}{-0.027625in}}{\pgfqpoint{0.104167in}{0.000000in}}%
\pgfpathcurveto{\pgfqpoint{0.104167in}{0.027625in}}{\pgfqpoint{0.093191in}{0.054123in}}{\pgfqpoint{0.073657in}{0.073657in}}%
\pgfpathcurveto{\pgfqpoint{0.054123in}{0.093191in}}{\pgfqpoint{0.027625in}{0.104167in}}{\pgfqpoint{0.000000in}{0.104167in}}%
\pgfpathcurveto{\pgfqpoint{-0.027625in}{0.104167in}}{\pgfqpoint{-0.054123in}{0.093191in}}{\pgfqpoint{-0.073657in}{0.073657in}}%
\pgfpathcurveto{\pgfqpoint{-0.093191in}{0.054123in}}{\pgfqpoint{-0.104167in}{0.027625in}}{\pgfqpoint{-0.104167in}{0.000000in}}%
\pgfpathcurveto{\pgfqpoint{-0.104167in}{-0.027625in}}{\pgfqpoint{-0.093191in}{-0.054123in}}{\pgfqpoint{-0.073657in}{-0.073657in}}%
\pgfpathcurveto{\pgfqpoint{-0.054123in}{-0.093191in}}{\pgfqpoint{-0.027625in}{-0.104167in}}{\pgfqpoint{0.000000in}{-0.104167in}}%
\pgfpathclose%
\pgfusepath{stroke,fill}%
}%
\begin{pgfscope}%
\pgfsys@transformshift{2.600000in}{2.800000in}%
\pgfsys@useobject{currentmarker}{}%
\end{pgfscope}%
\end{pgfscope}%
\begin{pgfscope}%
\pgfpathrectangle{\pgfqpoint{0.000000in}{0.000000in}}{\pgfqpoint{4.000000in}{4.000000in}}%
\pgfusepath{clip}%
\pgfsetbuttcap%
\pgfsetroundjoin%
\definecolor{currentfill}{rgb}{0.800000,0.800000,0.800000}%
\pgfsetfillcolor{currentfill}%
\pgfsetlinewidth{1.003750pt}%
\definecolor{currentstroke}{rgb}{0.000000,0.000000,0.000000}%
\pgfsetstrokecolor{currentstroke}%
\pgfsetdash{}{0pt}%
\pgfsys@defobject{currentmarker}{\pgfqpoint{-0.104167in}{-0.104167in}}{\pgfqpoint{0.104167in}{0.104167in}}{%
\pgfpathmoveto{\pgfqpoint{0.000000in}{-0.104167in}}%
\pgfpathcurveto{\pgfqpoint{0.027625in}{-0.104167in}}{\pgfqpoint{0.054123in}{-0.093191in}}{\pgfqpoint{0.073657in}{-0.073657in}}%
\pgfpathcurveto{\pgfqpoint{0.093191in}{-0.054123in}}{\pgfqpoint{0.104167in}{-0.027625in}}{\pgfqpoint{0.104167in}{0.000000in}}%
\pgfpathcurveto{\pgfqpoint{0.104167in}{0.027625in}}{\pgfqpoint{0.093191in}{0.054123in}}{\pgfqpoint{0.073657in}{0.073657in}}%
\pgfpathcurveto{\pgfqpoint{0.054123in}{0.093191in}}{\pgfqpoint{0.027625in}{0.104167in}}{\pgfqpoint{0.000000in}{0.104167in}}%
\pgfpathcurveto{\pgfqpoint{-0.027625in}{0.104167in}}{\pgfqpoint{-0.054123in}{0.093191in}}{\pgfqpoint{-0.073657in}{0.073657in}}%
\pgfpathcurveto{\pgfqpoint{-0.093191in}{0.054123in}}{\pgfqpoint{-0.104167in}{0.027625in}}{\pgfqpoint{-0.104167in}{0.000000in}}%
\pgfpathcurveto{\pgfqpoint{-0.104167in}{-0.027625in}}{\pgfqpoint{-0.093191in}{-0.054123in}}{\pgfqpoint{-0.073657in}{-0.073657in}}%
\pgfpathcurveto{\pgfqpoint{-0.054123in}{-0.093191in}}{\pgfqpoint{-0.027625in}{-0.104167in}}{\pgfqpoint{0.000000in}{-0.104167in}}%
\pgfpathclose%
\pgfusepath{stroke,fill}%
}%
\begin{pgfscope}%
\pgfsys@transformshift{3.400000in}{2.800000in}%
\pgfsys@useobject{currentmarker}{}%
\end{pgfscope}%
\end{pgfscope}%
\begin{pgfscope}%
\pgfpathrectangle{\pgfqpoint{0.000000in}{0.000000in}}{\pgfqpoint{4.000000in}{4.000000in}}%
\pgfusepath{clip}%
\pgfsetrectcap%
\pgfsetroundjoin%
\pgfsetlinewidth{3.011250pt}%
\definecolor{currentstroke}{rgb}{0.250980,0.250980,0.250980}%
\pgfsetstrokecolor{currentstroke}%
\pgfsetdash{}{0pt}%
\pgfpathmoveto{\pgfqpoint{2.600000in}{2.000000in}}%
\pgfpathlineto{\pgfqpoint{0.800000in}{2.000000in}}%
\pgfusepath{stroke}%
\end{pgfscope}%
\begin{pgfscope}%
\pgfpathrectangle{\pgfqpoint{0.000000in}{0.000000in}}{\pgfqpoint{4.000000in}{4.000000in}}%
\pgfusepath{clip}%
\pgfsetrectcap%
\pgfsetroundjoin%
\pgfsetlinewidth{3.011250pt}%
\definecolor{currentstroke}{rgb}{0.250980,0.250980,0.250980}%
\pgfsetstrokecolor{currentstroke}%
\pgfsetdash{}{0pt}%
\pgfpathmoveto{\pgfqpoint{3.400000in}{2.000000in}}%
\pgfpathlineto{\pgfqpoint{3.400000in}{2.200000in}}%
\pgfpathlineto{\pgfqpoint{0.800000in}{2.200000in}}%
\pgfpathlineto{\pgfqpoint{0.800000in}{2.000000in}}%
\pgfusepath{stroke}%
\end{pgfscope}%
\begin{pgfscope}%
\pgfpathrectangle{\pgfqpoint{0.000000in}{0.000000in}}{\pgfqpoint{4.000000in}{4.000000in}}%
\pgfusepath{clip}%
\pgfsetrectcap%
\pgfsetroundjoin%
\pgfsetlinewidth{3.011250pt}%
\definecolor{currentstroke}{rgb}{0.250980,0.250980,0.250980}%
\pgfsetstrokecolor{currentstroke}%
\pgfsetdash{}{0pt}%
\pgfpathmoveto{\pgfqpoint{3.400000in}{2.000000in}}%
\pgfpathlineto{\pgfqpoint{2.600000in}{2.000000in}}%
\pgfusepath{stroke}%
\end{pgfscope}%
\begin{pgfscope}%
\pgfpathrectangle{\pgfqpoint{0.000000in}{0.000000in}}{\pgfqpoint{4.000000in}{4.000000in}}%
\pgfusepath{clip}%
\pgfsetbuttcap%
\pgfsetroundjoin%
\definecolor{currentfill}{rgb}{0.800000,0.800000,0.800000}%
\pgfsetfillcolor{currentfill}%
\pgfsetlinewidth{1.003750pt}%
\definecolor{currentstroke}{rgb}{0.000000,0.000000,0.000000}%
\pgfsetstrokecolor{currentstroke}%
\pgfsetdash{}{0pt}%
\pgfsys@defobject{currentmarker}{\pgfqpoint{-0.104167in}{-0.104167in}}{\pgfqpoint{0.104167in}{0.104167in}}{%
\pgfpathmoveto{\pgfqpoint{0.000000in}{-0.104167in}}%
\pgfpathcurveto{\pgfqpoint{0.027625in}{-0.104167in}}{\pgfqpoint{0.054123in}{-0.093191in}}{\pgfqpoint{0.073657in}{-0.073657in}}%
\pgfpathcurveto{\pgfqpoint{0.093191in}{-0.054123in}}{\pgfqpoint{0.104167in}{-0.027625in}}{\pgfqpoint{0.104167in}{0.000000in}}%
\pgfpathcurveto{\pgfqpoint{0.104167in}{0.027625in}}{\pgfqpoint{0.093191in}{0.054123in}}{\pgfqpoint{0.073657in}{0.073657in}}%
\pgfpathcurveto{\pgfqpoint{0.054123in}{0.093191in}}{\pgfqpoint{0.027625in}{0.104167in}}{\pgfqpoint{0.000000in}{0.104167in}}%
\pgfpathcurveto{\pgfqpoint{-0.027625in}{0.104167in}}{\pgfqpoint{-0.054123in}{0.093191in}}{\pgfqpoint{-0.073657in}{0.073657in}}%
\pgfpathcurveto{\pgfqpoint{-0.093191in}{0.054123in}}{\pgfqpoint{-0.104167in}{0.027625in}}{\pgfqpoint{-0.104167in}{0.000000in}}%
\pgfpathcurveto{\pgfqpoint{-0.104167in}{-0.027625in}}{\pgfqpoint{-0.093191in}{-0.054123in}}{\pgfqpoint{-0.073657in}{-0.073657in}}%
\pgfpathcurveto{\pgfqpoint{-0.054123in}{-0.093191in}}{\pgfqpoint{-0.027625in}{-0.104167in}}{\pgfqpoint{0.000000in}{-0.104167in}}%
\pgfpathclose%
\pgfusepath{stroke,fill}%
}%
\begin{pgfscope}%
\pgfsys@transformshift{0.800000in}{2.000000in}%
\pgfsys@useobject{currentmarker}{}%
\end{pgfscope}%
\end{pgfscope}%
\begin{pgfscope}%
\pgfpathrectangle{\pgfqpoint{0.000000in}{0.000000in}}{\pgfqpoint{4.000000in}{4.000000in}}%
\pgfusepath{clip}%
\pgfsetbuttcap%
\pgfsetroundjoin%
\definecolor{currentfill}{rgb}{0.800000,0.800000,0.800000}%
\pgfsetfillcolor{currentfill}%
\pgfsetlinewidth{1.003750pt}%
\definecolor{currentstroke}{rgb}{0.000000,0.000000,0.000000}%
\pgfsetstrokecolor{currentstroke}%
\pgfsetdash{}{0pt}%
\pgfsys@defobject{currentmarker}{\pgfqpoint{-0.104167in}{-0.104167in}}{\pgfqpoint{0.104167in}{0.104167in}}{%
\pgfpathmoveto{\pgfqpoint{0.000000in}{-0.104167in}}%
\pgfpathcurveto{\pgfqpoint{0.027625in}{-0.104167in}}{\pgfqpoint{0.054123in}{-0.093191in}}{\pgfqpoint{0.073657in}{-0.073657in}}%
\pgfpathcurveto{\pgfqpoint{0.093191in}{-0.054123in}}{\pgfqpoint{0.104167in}{-0.027625in}}{\pgfqpoint{0.104167in}{0.000000in}}%
\pgfpathcurveto{\pgfqpoint{0.104167in}{0.027625in}}{\pgfqpoint{0.093191in}{0.054123in}}{\pgfqpoint{0.073657in}{0.073657in}}%
\pgfpathcurveto{\pgfqpoint{0.054123in}{0.093191in}}{\pgfqpoint{0.027625in}{0.104167in}}{\pgfqpoint{0.000000in}{0.104167in}}%
\pgfpathcurveto{\pgfqpoint{-0.027625in}{0.104167in}}{\pgfqpoint{-0.054123in}{0.093191in}}{\pgfqpoint{-0.073657in}{0.073657in}}%
\pgfpathcurveto{\pgfqpoint{-0.093191in}{0.054123in}}{\pgfqpoint{-0.104167in}{0.027625in}}{\pgfqpoint{-0.104167in}{0.000000in}}%
\pgfpathcurveto{\pgfqpoint{-0.104167in}{-0.027625in}}{\pgfqpoint{-0.093191in}{-0.054123in}}{\pgfqpoint{-0.073657in}{-0.073657in}}%
\pgfpathcurveto{\pgfqpoint{-0.054123in}{-0.093191in}}{\pgfqpoint{-0.027625in}{-0.104167in}}{\pgfqpoint{0.000000in}{-0.104167in}}%
\pgfpathclose%
\pgfusepath{stroke,fill}%
}%
\begin{pgfscope}%
\pgfsys@transformshift{2.600000in}{2.000000in}%
\pgfsys@useobject{currentmarker}{}%
\end{pgfscope}%
\end{pgfscope}%
\begin{pgfscope}%
\pgfpathrectangle{\pgfqpoint{0.000000in}{0.000000in}}{\pgfqpoint{4.000000in}{4.000000in}}%
\pgfusepath{clip}%
\pgfsetbuttcap%
\pgfsetroundjoin%
\definecolor{currentfill}{rgb}{0.800000,0.800000,0.800000}%
\pgfsetfillcolor{currentfill}%
\pgfsetlinewidth{1.003750pt}%
\definecolor{currentstroke}{rgb}{0.000000,0.000000,0.000000}%
\pgfsetstrokecolor{currentstroke}%
\pgfsetdash{}{0pt}%
\pgfsys@defobject{currentmarker}{\pgfqpoint{-0.104167in}{-0.104167in}}{\pgfqpoint{0.104167in}{0.104167in}}{%
\pgfpathmoveto{\pgfqpoint{0.000000in}{-0.104167in}}%
\pgfpathcurveto{\pgfqpoint{0.027625in}{-0.104167in}}{\pgfqpoint{0.054123in}{-0.093191in}}{\pgfqpoint{0.073657in}{-0.073657in}}%
\pgfpathcurveto{\pgfqpoint{0.093191in}{-0.054123in}}{\pgfqpoint{0.104167in}{-0.027625in}}{\pgfqpoint{0.104167in}{0.000000in}}%
\pgfpathcurveto{\pgfqpoint{0.104167in}{0.027625in}}{\pgfqpoint{0.093191in}{0.054123in}}{\pgfqpoint{0.073657in}{0.073657in}}%
\pgfpathcurveto{\pgfqpoint{0.054123in}{0.093191in}}{\pgfqpoint{0.027625in}{0.104167in}}{\pgfqpoint{0.000000in}{0.104167in}}%
\pgfpathcurveto{\pgfqpoint{-0.027625in}{0.104167in}}{\pgfqpoint{-0.054123in}{0.093191in}}{\pgfqpoint{-0.073657in}{0.073657in}}%
\pgfpathcurveto{\pgfqpoint{-0.093191in}{0.054123in}}{\pgfqpoint{-0.104167in}{0.027625in}}{\pgfqpoint{-0.104167in}{0.000000in}}%
\pgfpathcurveto{\pgfqpoint{-0.104167in}{-0.027625in}}{\pgfqpoint{-0.093191in}{-0.054123in}}{\pgfqpoint{-0.073657in}{-0.073657in}}%
\pgfpathcurveto{\pgfqpoint{-0.054123in}{-0.093191in}}{\pgfqpoint{-0.027625in}{-0.104167in}}{\pgfqpoint{0.000000in}{-0.104167in}}%
\pgfpathclose%
\pgfusepath{stroke,fill}%
}%
\begin{pgfscope}%
\pgfsys@transformshift{3.400000in}{2.000000in}%
\pgfsys@useobject{currentmarker}{}%
\end{pgfscope}%
\end{pgfscope}%
\begin{pgfscope}%
\pgfpathrectangle{\pgfqpoint{0.000000in}{0.000000in}}{\pgfqpoint{4.000000in}{4.000000in}}%
\pgfusepath{clip}%
\pgfsetrectcap%
\pgfsetroundjoin%
\pgfsetlinewidth{3.011250pt}%
\definecolor{currentstroke}{rgb}{0.250980,0.250980,0.250980}%
\pgfsetstrokecolor{currentstroke}%
\pgfsetdash{}{0pt}%
\pgfpathmoveto{\pgfqpoint{2.600000in}{1.200000in}}%
\pgfpathlineto{\pgfqpoint{0.800000in}{1.200000in}}%
\pgfusepath{stroke}%
\end{pgfscope}%
\begin{pgfscope}%
\pgfpathrectangle{\pgfqpoint{0.000000in}{0.000000in}}{\pgfqpoint{4.000000in}{4.000000in}}%
\pgfusepath{clip}%
\pgfsetbuttcap%
\pgfsetroundjoin%
\definecolor{currentfill}{rgb}{0.800000,0.800000,0.800000}%
\pgfsetfillcolor{currentfill}%
\pgfsetlinewidth{1.003750pt}%
\definecolor{currentstroke}{rgb}{0.000000,0.000000,0.000000}%
\pgfsetstrokecolor{currentstroke}%
\pgfsetdash{}{0pt}%
\pgfsys@defobject{currentmarker}{\pgfqpoint{-0.104167in}{-0.104167in}}{\pgfqpoint{0.104167in}{0.104167in}}{%
\pgfpathmoveto{\pgfqpoint{0.000000in}{-0.104167in}}%
\pgfpathcurveto{\pgfqpoint{0.027625in}{-0.104167in}}{\pgfqpoint{0.054123in}{-0.093191in}}{\pgfqpoint{0.073657in}{-0.073657in}}%
\pgfpathcurveto{\pgfqpoint{0.093191in}{-0.054123in}}{\pgfqpoint{0.104167in}{-0.027625in}}{\pgfqpoint{0.104167in}{0.000000in}}%
\pgfpathcurveto{\pgfqpoint{0.104167in}{0.027625in}}{\pgfqpoint{0.093191in}{0.054123in}}{\pgfqpoint{0.073657in}{0.073657in}}%
\pgfpathcurveto{\pgfqpoint{0.054123in}{0.093191in}}{\pgfqpoint{0.027625in}{0.104167in}}{\pgfqpoint{0.000000in}{0.104167in}}%
\pgfpathcurveto{\pgfqpoint{-0.027625in}{0.104167in}}{\pgfqpoint{-0.054123in}{0.093191in}}{\pgfqpoint{-0.073657in}{0.073657in}}%
\pgfpathcurveto{\pgfqpoint{-0.093191in}{0.054123in}}{\pgfqpoint{-0.104167in}{0.027625in}}{\pgfqpoint{-0.104167in}{0.000000in}}%
\pgfpathcurveto{\pgfqpoint{-0.104167in}{-0.027625in}}{\pgfqpoint{-0.093191in}{-0.054123in}}{\pgfqpoint{-0.073657in}{-0.073657in}}%
\pgfpathcurveto{\pgfqpoint{-0.054123in}{-0.093191in}}{\pgfqpoint{-0.027625in}{-0.104167in}}{\pgfqpoint{0.000000in}{-0.104167in}}%
\pgfpathclose%
\pgfusepath{stroke,fill}%
}%
\begin{pgfscope}%
\pgfsys@transformshift{0.800000in}{1.200000in}%
\pgfsys@useobject{currentmarker}{}%
\end{pgfscope}%
\end{pgfscope}%
\begin{pgfscope}%
\pgfpathrectangle{\pgfqpoint{0.000000in}{0.000000in}}{\pgfqpoint{4.000000in}{4.000000in}}%
\pgfusepath{clip}%
\pgfsetbuttcap%
\pgfsetroundjoin%
\definecolor{currentfill}{rgb}{0.800000,0.800000,0.800000}%
\pgfsetfillcolor{currentfill}%
\pgfsetlinewidth{1.003750pt}%
\definecolor{currentstroke}{rgb}{0.000000,0.000000,0.000000}%
\pgfsetstrokecolor{currentstroke}%
\pgfsetdash{}{0pt}%
\pgfsys@defobject{currentmarker}{\pgfqpoint{-0.104167in}{-0.104167in}}{\pgfqpoint{0.104167in}{0.104167in}}{%
\pgfpathmoveto{\pgfqpoint{0.000000in}{-0.104167in}}%
\pgfpathcurveto{\pgfqpoint{0.027625in}{-0.104167in}}{\pgfqpoint{0.054123in}{-0.093191in}}{\pgfqpoint{0.073657in}{-0.073657in}}%
\pgfpathcurveto{\pgfqpoint{0.093191in}{-0.054123in}}{\pgfqpoint{0.104167in}{-0.027625in}}{\pgfqpoint{0.104167in}{0.000000in}}%
\pgfpathcurveto{\pgfqpoint{0.104167in}{0.027625in}}{\pgfqpoint{0.093191in}{0.054123in}}{\pgfqpoint{0.073657in}{0.073657in}}%
\pgfpathcurveto{\pgfqpoint{0.054123in}{0.093191in}}{\pgfqpoint{0.027625in}{0.104167in}}{\pgfqpoint{0.000000in}{0.104167in}}%
\pgfpathcurveto{\pgfqpoint{-0.027625in}{0.104167in}}{\pgfqpoint{-0.054123in}{0.093191in}}{\pgfqpoint{-0.073657in}{0.073657in}}%
\pgfpathcurveto{\pgfqpoint{-0.093191in}{0.054123in}}{\pgfqpoint{-0.104167in}{0.027625in}}{\pgfqpoint{-0.104167in}{0.000000in}}%
\pgfpathcurveto{\pgfqpoint{-0.104167in}{-0.027625in}}{\pgfqpoint{-0.093191in}{-0.054123in}}{\pgfqpoint{-0.073657in}{-0.073657in}}%
\pgfpathcurveto{\pgfqpoint{-0.054123in}{-0.093191in}}{\pgfqpoint{-0.027625in}{-0.104167in}}{\pgfqpoint{0.000000in}{-0.104167in}}%
\pgfpathclose%
\pgfusepath{stroke,fill}%
}%
\begin{pgfscope}%
\pgfsys@transformshift{2.600000in}{1.200000in}%
\pgfsys@useobject{currentmarker}{}%
\end{pgfscope}%
\end{pgfscope}%
\begin{pgfscope}%
\pgfpathrectangle{\pgfqpoint{0.000000in}{0.000000in}}{\pgfqpoint{4.000000in}{4.000000in}}%
\pgfusepath{clip}%
\pgfsetbuttcap%
\pgfsetroundjoin%
\definecolor{currentfill}{rgb}{0.800000,0.800000,0.800000}%
\pgfsetfillcolor{currentfill}%
\pgfsetlinewidth{1.003750pt}%
\definecolor{currentstroke}{rgb}{0.000000,0.000000,0.000000}%
\pgfsetstrokecolor{currentstroke}%
\pgfsetdash{}{0pt}%
\pgfsys@defobject{currentmarker}{\pgfqpoint{-0.104167in}{-0.104167in}}{\pgfqpoint{0.104167in}{0.104167in}}{%
\pgfpathmoveto{\pgfqpoint{0.000000in}{-0.104167in}}%
\pgfpathcurveto{\pgfqpoint{0.027625in}{-0.104167in}}{\pgfqpoint{0.054123in}{-0.093191in}}{\pgfqpoint{0.073657in}{-0.073657in}}%
\pgfpathcurveto{\pgfqpoint{0.093191in}{-0.054123in}}{\pgfqpoint{0.104167in}{-0.027625in}}{\pgfqpoint{0.104167in}{0.000000in}}%
\pgfpathcurveto{\pgfqpoint{0.104167in}{0.027625in}}{\pgfqpoint{0.093191in}{0.054123in}}{\pgfqpoint{0.073657in}{0.073657in}}%
\pgfpathcurveto{\pgfqpoint{0.054123in}{0.093191in}}{\pgfqpoint{0.027625in}{0.104167in}}{\pgfqpoint{0.000000in}{0.104167in}}%
\pgfpathcurveto{\pgfqpoint{-0.027625in}{0.104167in}}{\pgfqpoint{-0.054123in}{0.093191in}}{\pgfqpoint{-0.073657in}{0.073657in}}%
\pgfpathcurveto{\pgfqpoint{-0.093191in}{0.054123in}}{\pgfqpoint{-0.104167in}{0.027625in}}{\pgfqpoint{-0.104167in}{0.000000in}}%
\pgfpathcurveto{\pgfqpoint{-0.104167in}{-0.027625in}}{\pgfqpoint{-0.093191in}{-0.054123in}}{\pgfqpoint{-0.073657in}{-0.073657in}}%
\pgfpathcurveto{\pgfqpoint{-0.054123in}{-0.093191in}}{\pgfqpoint{-0.027625in}{-0.104167in}}{\pgfqpoint{0.000000in}{-0.104167in}}%
\pgfpathclose%
\pgfusepath{stroke,fill}%
}%
\begin{pgfscope}%
\pgfsys@transformshift{3.400000in}{1.200000in}%
\pgfsys@useobject{currentmarker}{}%
\end{pgfscope}%
\end{pgfscope}%
\begin{pgfscope}%
\pgfsetroundcap%
\pgfsetroundjoin%
\pgfsetlinewidth{3.011250pt}%
\definecolor{currentstroke}{rgb}{0.250980,0.250980,0.250980}%
\pgfsetstrokecolor{currentstroke}%
\pgfsetdash{}{0pt}%
\pgfpathmoveto{\pgfqpoint{0.938922in}{2.750000in}}%
\pgfpathlineto{\pgfqpoint{2.421846in}{2.750000in}}%
\pgfusepath{stroke}%
\end{pgfscope}%
\begin{pgfscope}%
\pgfsetroundcap%
\pgfsetroundjoin%
\definecolor{currentfill}{rgb}{0.250980,0.250980,0.250980}%
\pgfsetfillcolor{currentfill}%
\pgfsetlinewidth{3.011250pt}%
\definecolor{currentstroke}{rgb}{0.250980,0.250980,0.250980}%
\pgfsetstrokecolor{currentstroke}%
\pgfsetdash{}{0pt}%
\pgfpathmoveto{\pgfqpoint{2.310735in}{2.819444in}}%
\pgfpathlineto{\pgfqpoint{2.421846in}{2.750000in}}%
\pgfpathlineto{\pgfqpoint{2.310735in}{2.680556in}}%
\pgfpathlineto{\pgfqpoint{2.310735in}{2.819444in}}%
\pgfpathclose%
\pgfusepath{stroke,fill}%
\end{pgfscope}%
\begin{pgfscope}%
\pgfsetroundcap%
\pgfsetroundjoin%
\pgfsetlinewidth{3.011250pt}%
\definecolor{currentstroke}{rgb}{0.250980,0.250980,0.250980}%
\pgfsetstrokecolor{currentstroke}%
\pgfsetdash{}{0pt}%
\pgfpathmoveto{\pgfqpoint{0.800000in}{2.938871in}}%
\pgfpathlineto{\pgfqpoint{0.800000in}{3.190000in}}%
\pgfpathlineto{\pgfqpoint{3.400000in}{3.190000in}}%
\pgfpathlineto{\pgfqpoint{3.400000in}{2.978179in}}%
\pgfusepath{stroke}%
\end{pgfscope}%
\begin{pgfscope}%
\pgfsetroundcap%
\pgfsetroundjoin%
\definecolor{currentfill}{rgb}{0.250980,0.250980,0.250980}%
\pgfsetfillcolor{currentfill}%
\pgfsetlinewidth{3.011250pt}%
\definecolor{currentstroke}{rgb}{0.250980,0.250980,0.250980}%
\pgfsetstrokecolor{currentstroke}%
\pgfsetdash{}{0pt}%
\pgfpathmoveto{\pgfqpoint{3.469444in}{3.089290in}}%
\pgfpathlineto{\pgfqpoint{3.400000in}{2.978179in}}%
\pgfpathlineto{\pgfqpoint{3.330556in}{3.089290in}}%
\pgfpathlineto{\pgfqpoint{3.469444in}{3.089290in}}%
\pgfpathclose%
\pgfusepath{stroke,fill}%
\end{pgfscope}%
\begin{pgfscope}%
\pgfsetroundcap%
\pgfsetroundjoin%
\pgfsetlinewidth{3.011250pt}%
\definecolor{currentstroke}{rgb}{0.250980,0.250980,0.250980}%
\pgfsetstrokecolor{currentstroke}%
\pgfsetdash{}{0pt}%
\pgfpathmoveto{\pgfqpoint{2.461078in}{2.850000in}}%
\pgfpathlineto{\pgfqpoint{0.978154in}{2.850000in}}%
\pgfusepath{stroke}%
\end{pgfscope}%
\begin{pgfscope}%
\pgfsetroundcap%
\pgfsetroundjoin%
\definecolor{currentfill}{rgb}{0.250980,0.250980,0.250980}%
\pgfsetfillcolor{currentfill}%
\pgfsetlinewidth{3.011250pt}%
\definecolor{currentstroke}{rgb}{0.250980,0.250980,0.250980}%
\pgfsetstrokecolor{currentstroke}%
\pgfsetdash{}{0pt}%
\pgfpathmoveto{\pgfqpoint{1.089265in}{2.780556in}}%
\pgfpathlineto{\pgfqpoint{0.978154in}{2.850000in}}%
\pgfpathlineto{\pgfqpoint{1.089265in}{2.919444in}}%
\pgfpathlineto{\pgfqpoint{1.089265in}{2.780556in}}%
\pgfpathclose%
\pgfusepath{stroke,fill}%
\end{pgfscope}%
\begin{pgfscope}%
\pgfsetroundcap%
\pgfsetroundjoin%
\pgfsetlinewidth{3.011250pt}%
\definecolor{currentstroke}{rgb}{0.250980,0.250980,0.250980}%
\pgfsetstrokecolor{currentstroke}%
\pgfsetdash{}{0pt}%
\pgfpathmoveto{\pgfqpoint{2.738916in}{2.800000in}}%
\pgfpathlineto{\pgfqpoint{3.221769in}{2.800000in}}%
\pgfusepath{stroke}%
\end{pgfscope}%
\begin{pgfscope}%
\pgfsetroundcap%
\pgfsetroundjoin%
\definecolor{currentfill}{rgb}{0.250980,0.250980,0.250980}%
\pgfsetfillcolor{currentfill}%
\pgfsetlinewidth{3.011250pt}%
\definecolor{currentstroke}{rgb}{0.250980,0.250980,0.250980}%
\pgfsetstrokecolor{currentstroke}%
\pgfsetdash{}{0pt}%
\pgfpathmoveto{\pgfqpoint{3.110658in}{2.869444in}}%
\pgfpathlineto{\pgfqpoint{3.221769in}{2.800000in}}%
\pgfpathlineto{\pgfqpoint{3.110658in}{2.730556in}}%
\pgfpathlineto{\pgfqpoint{3.110658in}{2.869444in}}%
\pgfpathclose%
\pgfusepath{stroke,fill}%
\end{pgfscope}%
\begin{pgfscope}%
\definecolor{textcolor}{rgb}{0.000000,0.000000,0.000000}%
\pgfsetstrokecolor{textcolor}%
\pgfsetfillcolor{textcolor}%
\pgftext[x=0.800000in,y=2.500000in,,]{\color{textcolor}\sffamily\fontsize{40.000000}{48.000000}\selectfont \(\displaystyle x_1\)}%
\end{pgfscope}%
\begin{pgfscope}%
\definecolor{textcolor}{rgb}{0.000000,0.000000,0.000000}%
\pgfsetstrokecolor{textcolor}%
\pgfsetfillcolor{textcolor}%
\pgftext[x=2.600000in,y=2.500000in,,]{\color{textcolor}\sffamily\fontsize{40.000000}{48.000000}\selectfont \(\displaystyle x_2\)}%
\end{pgfscope}%
\begin{pgfscope}%
\definecolor{textcolor}{rgb}{0.000000,0.000000,0.000000}%
\pgfsetstrokecolor{textcolor}%
\pgfsetfillcolor{textcolor}%
\pgftext[x=3.400000in,y=2.500000in,,]{\color{textcolor}\sffamily\fontsize{40.000000}{48.000000}\selectfont \(\displaystyle x_3\)}%
\end{pgfscope}%
\begin{pgfscope}%
\definecolor{textcolor}{rgb}{0.000000,0.000000,0.000000}%
\pgfsetstrokecolor{textcolor}%
\pgfsetfillcolor{textcolor}%
\pgftext[x=0.650000in,y=1.740192in,,]{\color{textcolor}\sffamily\fontsize{40.000000}{48.000000}\selectfont \(\displaystyle G_+\)}%
\end{pgfscope}%
\begin{pgfscope}%
\definecolor{textcolor}{rgb}{0.000000,0.000000,0.000000}%
\pgfsetstrokecolor{textcolor}%
\pgfsetfillcolor{textcolor}%
\pgftext[x=0.650000in,y=0.940192in,,]{\color{textcolor}\sffamily\fontsize{40.000000}{48.000000}\selectfont \(\displaystyle G_-\)}%
\end{pgfscope}%
\end{pgfpicture}%
\makeatother%
\endgroup%

%% file: fig/unidimensional_framework_2.pgf
%% Creator: Matplotlib, PGF backend
%%
%% To include the figure in your LaTeX document, write
%%   \input{<filename>.pgf}
%%
%% Make sure the required packages are loaded in your preamble
%%   \usepackage{pgf}
%%
%% Figures using additional raster images can only be included by \input if
%% they are in the same directory as the main LaTeX file. For loading figures
%% from other directories you can use the `import` package
%%   \usepackage{import}
%%
%% and then include the figures with
%%   \import{<path to file>}{<filename>.pgf}
%%
%% Matplotlib used the following preamble
%%   \usepackage{fontspec}
%%   \setmainfont{DejaVuSerif.ttf}[Path=\detokenize{C:/Users/ccros/Anaconda3/Lib/site-packages/matplotlib/mpl-data/fonts/ttf/}]
%%   \setsansfont{DejaVuSans.ttf}[Path=\detokenize{C:/Users/ccros/Anaconda3/Lib/site-packages/matplotlib/mpl-data/fonts/ttf/}]
%%   \setmonofont{DejaVuSansMono.ttf}[Path=\detokenize{C:/Users/ccros/Anaconda3/Lib/site-packages/matplotlib/mpl-data/fonts/ttf/}]
%%
\begingroup%
\makeatletter%
\begin{pgfpicture}%
\pgfpathrectangle{\pgfpointorigin}{\pgfqpoint{4.000000in}{4.000000in}}%
\pgfusepath{use as bounding box, clip}%
\begin{pgfscope}%
\pgfsetbuttcap%
\pgfsetmiterjoin%
\pgfsetlinewidth{0.000000pt}%
\definecolor{currentstroke}{rgb}{1.000000,1.000000,1.000000}%
\pgfsetstrokecolor{currentstroke}%
\pgfsetstrokeopacity{0.000000}%
\pgfsetdash{}{0pt}%
\pgfpathmoveto{\pgfqpoint{0.000000in}{0.000000in}}%
\pgfpathlineto{\pgfqpoint{4.000000in}{0.000000in}}%
\pgfpathlineto{\pgfqpoint{4.000000in}{4.000000in}}%
\pgfpathlineto{\pgfqpoint{0.000000in}{4.000000in}}%
\pgfpathclose%
\pgfusepath{}%
\end{pgfscope}%
\begin{pgfscope}%
\pgfpathrectangle{\pgfqpoint{0.000000in}{0.000000in}}{\pgfqpoint{4.000000in}{4.000000in}}%
\pgfusepath{clip}%
\pgfsetbuttcap%
\pgfsetroundjoin%
\definecolor{currentfill}{rgb}{0.800000,0.800000,0.800000}%
\pgfsetfillcolor{currentfill}%
\pgfsetlinewidth{1.003750pt}%
\definecolor{currentstroke}{rgb}{0.000000,0.000000,0.000000}%
\pgfsetstrokecolor{currentstroke}%
\pgfsetdash{}{0pt}%
\pgfsys@defobject{currentmarker}{\pgfqpoint{-0.104167in}{-0.104167in}}{\pgfqpoint{0.104167in}{0.104167in}}{%
\pgfpathmoveto{\pgfqpoint{0.000000in}{-0.104167in}}%
\pgfpathcurveto{\pgfqpoint{0.027625in}{-0.104167in}}{\pgfqpoint{0.054123in}{-0.093191in}}{\pgfqpoint{0.073657in}{-0.073657in}}%
\pgfpathcurveto{\pgfqpoint{0.093191in}{-0.054123in}}{\pgfqpoint{0.104167in}{-0.027625in}}{\pgfqpoint{0.104167in}{0.000000in}}%
\pgfpathcurveto{\pgfqpoint{0.104167in}{0.027625in}}{\pgfqpoint{0.093191in}{0.054123in}}{\pgfqpoint{0.073657in}{0.073657in}}%
\pgfpathcurveto{\pgfqpoint{0.054123in}{0.093191in}}{\pgfqpoint{0.027625in}{0.104167in}}{\pgfqpoint{0.000000in}{0.104167in}}%
\pgfpathcurveto{\pgfqpoint{-0.027625in}{0.104167in}}{\pgfqpoint{-0.054123in}{0.093191in}}{\pgfqpoint{-0.073657in}{0.073657in}}%
\pgfpathcurveto{\pgfqpoint{-0.093191in}{0.054123in}}{\pgfqpoint{-0.104167in}{0.027625in}}{\pgfqpoint{-0.104167in}{0.000000in}}%
\pgfpathcurveto{\pgfqpoint{-0.104167in}{-0.027625in}}{\pgfqpoint{-0.093191in}{-0.054123in}}{\pgfqpoint{-0.073657in}{-0.073657in}}%
\pgfpathcurveto{\pgfqpoint{-0.054123in}{-0.093191in}}{\pgfqpoint{-0.027625in}{-0.104167in}}{\pgfqpoint{0.000000in}{-0.104167in}}%
\pgfpathclose%
\pgfusepath{stroke,fill}%
}%
\begin{pgfscope}%
\pgfsys@transformshift{0.800000in}{2.800000in}%
\pgfsys@useobject{currentmarker}{}%
\end{pgfscope}%
\end{pgfscope}%
\begin{pgfscope}%
\pgfpathrectangle{\pgfqpoint{0.000000in}{0.000000in}}{\pgfqpoint{4.000000in}{4.000000in}}%
\pgfusepath{clip}%
\pgfsetbuttcap%
\pgfsetroundjoin%
\definecolor{currentfill}{rgb}{0.800000,0.800000,0.800000}%
\pgfsetfillcolor{currentfill}%
\pgfsetlinewidth{1.003750pt}%
\definecolor{currentstroke}{rgb}{0.000000,0.000000,0.000000}%
\pgfsetstrokecolor{currentstroke}%
\pgfsetdash{}{0pt}%
\pgfsys@defobject{currentmarker}{\pgfqpoint{-0.104167in}{-0.104167in}}{\pgfqpoint{0.104167in}{0.104167in}}{%
\pgfpathmoveto{\pgfqpoint{0.000000in}{-0.104167in}}%
\pgfpathcurveto{\pgfqpoint{0.027625in}{-0.104167in}}{\pgfqpoint{0.054123in}{-0.093191in}}{\pgfqpoint{0.073657in}{-0.073657in}}%
\pgfpathcurveto{\pgfqpoint{0.093191in}{-0.054123in}}{\pgfqpoint{0.104167in}{-0.027625in}}{\pgfqpoint{0.104167in}{0.000000in}}%
\pgfpathcurveto{\pgfqpoint{0.104167in}{0.027625in}}{\pgfqpoint{0.093191in}{0.054123in}}{\pgfqpoint{0.073657in}{0.073657in}}%
\pgfpathcurveto{\pgfqpoint{0.054123in}{0.093191in}}{\pgfqpoint{0.027625in}{0.104167in}}{\pgfqpoint{0.000000in}{0.104167in}}%
\pgfpathcurveto{\pgfqpoint{-0.027625in}{0.104167in}}{\pgfqpoint{-0.054123in}{0.093191in}}{\pgfqpoint{-0.073657in}{0.073657in}}%
\pgfpathcurveto{\pgfqpoint{-0.093191in}{0.054123in}}{\pgfqpoint{-0.104167in}{0.027625in}}{\pgfqpoint{-0.104167in}{0.000000in}}%
\pgfpathcurveto{\pgfqpoint{-0.104167in}{-0.027625in}}{\pgfqpoint{-0.093191in}{-0.054123in}}{\pgfqpoint{-0.073657in}{-0.073657in}}%
\pgfpathcurveto{\pgfqpoint{-0.054123in}{-0.093191in}}{\pgfqpoint{-0.027625in}{-0.104167in}}{\pgfqpoint{0.000000in}{-0.104167in}}%
\pgfpathclose%
\pgfusepath{stroke,fill}%
}%
\begin{pgfscope}%
\pgfsys@transformshift{3.200000in}{2.800000in}%
\pgfsys@useobject{currentmarker}{}%
\end{pgfscope}%
\end{pgfscope}%
\begin{pgfscope}%
\pgfpathrectangle{\pgfqpoint{0.000000in}{0.000000in}}{\pgfqpoint{4.000000in}{4.000000in}}%
\pgfusepath{clip}%
\pgfsetbuttcap%
\pgfsetroundjoin%
\definecolor{currentfill}{rgb}{0.800000,0.800000,0.800000}%
\pgfsetfillcolor{currentfill}%
\pgfsetlinewidth{1.003750pt}%
\definecolor{currentstroke}{rgb}{0.000000,0.000000,0.000000}%
\pgfsetstrokecolor{currentstroke}%
\pgfsetdash{}{0pt}%
\pgfsys@defobject{currentmarker}{\pgfqpoint{-0.104167in}{-0.104167in}}{\pgfqpoint{0.104167in}{0.104167in}}{%
\pgfpathmoveto{\pgfqpoint{0.000000in}{-0.104167in}}%
\pgfpathcurveto{\pgfqpoint{0.027625in}{-0.104167in}}{\pgfqpoint{0.054123in}{-0.093191in}}{\pgfqpoint{0.073657in}{-0.073657in}}%
\pgfpathcurveto{\pgfqpoint{0.093191in}{-0.054123in}}{\pgfqpoint{0.104167in}{-0.027625in}}{\pgfqpoint{0.104167in}{0.000000in}}%
\pgfpathcurveto{\pgfqpoint{0.104167in}{0.027625in}}{\pgfqpoint{0.093191in}{0.054123in}}{\pgfqpoint{0.073657in}{0.073657in}}%
\pgfpathcurveto{\pgfqpoint{0.054123in}{0.093191in}}{\pgfqpoint{0.027625in}{0.104167in}}{\pgfqpoint{0.000000in}{0.104167in}}%
\pgfpathcurveto{\pgfqpoint{-0.027625in}{0.104167in}}{\pgfqpoint{-0.054123in}{0.093191in}}{\pgfqpoint{-0.073657in}{0.073657in}}%
\pgfpathcurveto{\pgfqpoint{-0.093191in}{0.054123in}}{\pgfqpoint{-0.104167in}{0.027625in}}{\pgfqpoint{-0.104167in}{0.000000in}}%
\pgfpathcurveto{\pgfqpoint{-0.104167in}{-0.027625in}}{\pgfqpoint{-0.093191in}{-0.054123in}}{\pgfqpoint{-0.073657in}{-0.073657in}}%
\pgfpathcurveto{\pgfqpoint{-0.054123in}{-0.093191in}}{\pgfqpoint{-0.027625in}{-0.104167in}}{\pgfqpoint{0.000000in}{-0.104167in}}%
\pgfpathclose%
\pgfusepath{stroke,fill}%
}%
\begin{pgfscope}%
\pgfsys@transformshift{2.000000in}{2.800000in}%
\pgfsys@useobject{currentmarker}{}%
\end{pgfscope}%
\end{pgfscope}%
\begin{pgfscope}%
\pgfpathrectangle{\pgfqpoint{0.000000in}{0.000000in}}{\pgfqpoint{4.000000in}{4.000000in}}%
\pgfusepath{clip}%
\pgfsetrectcap%
\pgfsetroundjoin%
\pgfsetlinewidth{3.011250pt}%
\definecolor{currentstroke}{rgb}{0.250980,0.250980,0.250980}%
\pgfsetstrokecolor{currentstroke}%
\pgfsetdash{}{0pt}%
\pgfpathmoveto{\pgfqpoint{3.200000in}{2.000000in}}%
\pgfpathlineto{\pgfqpoint{3.200000in}{2.150000in}}%
\pgfpathlineto{\pgfqpoint{0.800000in}{2.150000in}}%
\pgfpathlineto{\pgfqpoint{0.800000in}{2.000000in}}%
\pgfusepath{stroke}%
\end{pgfscope}%
\begin{pgfscope}%
\pgfpathrectangle{\pgfqpoint{0.000000in}{0.000000in}}{\pgfqpoint{4.000000in}{4.000000in}}%
\pgfusepath{clip}%
\pgfsetrectcap%
\pgfsetroundjoin%
\pgfsetlinewidth{3.011250pt}%
\definecolor{currentstroke}{rgb}{0.250980,0.250980,0.250980}%
\pgfsetstrokecolor{currentstroke}%
\pgfsetdash{}{0pt}%
\pgfpathmoveto{\pgfqpoint{2.000000in}{2.000000in}}%
\pgfpathlineto{\pgfqpoint{0.800000in}{2.000000in}}%
\pgfusepath{stroke}%
\end{pgfscope}%
\begin{pgfscope}%
\pgfpathrectangle{\pgfqpoint{0.000000in}{0.000000in}}{\pgfqpoint{4.000000in}{4.000000in}}%
\pgfusepath{clip}%
\pgfsetbuttcap%
\pgfsetroundjoin%
\definecolor{currentfill}{rgb}{0.800000,0.800000,0.800000}%
\pgfsetfillcolor{currentfill}%
\pgfsetlinewidth{1.003750pt}%
\definecolor{currentstroke}{rgb}{0.000000,0.000000,0.000000}%
\pgfsetstrokecolor{currentstroke}%
\pgfsetdash{}{0pt}%
\pgfsys@defobject{currentmarker}{\pgfqpoint{-0.104167in}{-0.104167in}}{\pgfqpoint{0.104167in}{0.104167in}}{%
\pgfpathmoveto{\pgfqpoint{0.000000in}{-0.104167in}}%
\pgfpathcurveto{\pgfqpoint{0.027625in}{-0.104167in}}{\pgfqpoint{0.054123in}{-0.093191in}}{\pgfqpoint{0.073657in}{-0.073657in}}%
\pgfpathcurveto{\pgfqpoint{0.093191in}{-0.054123in}}{\pgfqpoint{0.104167in}{-0.027625in}}{\pgfqpoint{0.104167in}{0.000000in}}%
\pgfpathcurveto{\pgfqpoint{0.104167in}{0.027625in}}{\pgfqpoint{0.093191in}{0.054123in}}{\pgfqpoint{0.073657in}{0.073657in}}%
\pgfpathcurveto{\pgfqpoint{0.054123in}{0.093191in}}{\pgfqpoint{0.027625in}{0.104167in}}{\pgfqpoint{0.000000in}{0.104167in}}%
\pgfpathcurveto{\pgfqpoint{-0.027625in}{0.104167in}}{\pgfqpoint{-0.054123in}{0.093191in}}{\pgfqpoint{-0.073657in}{0.073657in}}%
\pgfpathcurveto{\pgfqpoint{-0.093191in}{0.054123in}}{\pgfqpoint{-0.104167in}{0.027625in}}{\pgfqpoint{-0.104167in}{0.000000in}}%
\pgfpathcurveto{\pgfqpoint{-0.104167in}{-0.027625in}}{\pgfqpoint{-0.093191in}{-0.054123in}}{\pgfqpoint{-0.073657in}{-0.073657in}}%
\pgfpathcurveto{\pgfqpoint{-0.054123in}{-0.093191in}}{\pgfqpoint{-0.027625in}{-0.104167in}}{\pgfqpoint{0.000000in}{-0.104167in}}%
\pgfpathclose%
\pgfusepath{stroke,fill}%
}%
\begin{pgfscope}%
\pgfsys@transformshift{0.800000in}{2.000000in}%
\pgfsys@useobject{currentmarker}{}%
\end{pgfscope}%
\end{pgfscope}%
\begin{pgfscope}%
\pgfpathrectangle{\pgfqpoint{0.000000in}{0.000000in}}{\pgfqpoint{4.000000in}{4.000000in}}%
\pgfusepath{clip}%
\pgfsetbuttcap%
\pgfsetroundjoin%
\definecolor{currentfill}{rgb}{0.800000,0.800000,0.800000}%
\pgfsetfillcolor{currentfill}%
\pgfsetlinewidth{1.003750pt}%
\definecolor{currentstroke}{rgb}{0.000000,0.000000,0.000000}%
\pgfsetstrokecolor{currentstroke}%
\pgfsetdash{}{0pt}%
\pgfsys@defobject{currentmarker}{\pgfqpoint{-0.104167in}{-0.104167in}}{\pgfqpoint{0.104167in}{0.104167in}}{%
\pgfpathmoveto{\pgfqpoint{0.000000in}{-0.104167in}}%
\pgfpathcurveto{\pgfqpoint{0.027625in}{-0.104167in}}{\pgfqpoint{0.054123in}{-0.093191in}}{\pgfqpoint{0.073657in}{-0.073657in}}%
\pgfpathcurveto{\pgfqpoint{0.093191in}{-0.054123in}}{\pgfqpoint{0.104167in}{-0.027625in}}{\pgfqpoint{0.104167in}{0.000000in}}%
\pgfpathcurveto{\pgfqpoint{0.104167in}{0.027625in}}{\pgfqpoint{0.093191in}{0.054123in}}{\pgfqpoint{0.073657in}{0.073657in}}%
\pgfpathcurveto{\pgfqpoint{0.054123in}{0.093191in}}{\pgfqpoint{0.027625in}{0.104167in}}{\pgfqpoint{0.000000in}{0.104167in}}%
\pgfpathcurveto{\pgfqpoint{-0.027625in}{0.104167in}}{\pgfqpoint{-0.054123in}{0.093191in}}{\pgfqpoint{-0.073657in}{0.073657in}}%
\pgfpathcurveto{\pgfqpoint{-0.093191in}{0.054123in}}{\pgfqpoint{-0.104167in}{0.027625in}}{\pgfqpoint{-0.104167in}{0.000000in}}%
\pgfpathcurveto{\pgfqpoint{-0.104167in}{-0.027625in}}{\pgfqpoint{-0.093191in}{-0.054123in}}{\pgfqpoint{-0.073657in}{-0.073657in}}%
\pgfpathcurveto{\pgfqpoint{-0.054123in}{-0.093191in}}{\pgfqpoint{-0.027625in}{-0.104167in}}{\pgfqpoint{0.000000in}{-0.104167in}}%
\pgfpathclose%
\pgfusepath{stroke,fill}%
}%
\begin{pgfscope}%
\pgfsys@transformshift{3.200000in}{2.000000in}%
\pgfsys@useobject{currentmarker}{}%
\end{pgfscope}%
\end{pgfscope}%
\begin{pgfscope}%
\pgfpathrectangle{\pgfqpoint{0.000000in}{0.000000in}}{\pgfqpoint{4.000000in}{4.000000in}}%
\pgfusepath{clip}%
\pgfsetbuttcap%
\pgfsetroundjoin%
\definecolor{currentfill}{rgb}{0.800000,0.800000,0.800000}%
\pgfsetfillcolor{currentfill}%
\pgfsetlinewidth{1.003750pt}%
\definecolor{currentstroke}{rgb}{0.000000,0.000000,0.000000}%
\pgfsetstrokecolor{currentstroke}%
\pgfsetdash{}{0pt}%
\pgfsys@defobject{currentmarker}{\pgfqpoint{-0.104167in}{-0.104167in}}{\pgfqpoint{0.104167in}{0.104167in}}{%
\pgfpathmoveto{\pgfqpoint{0.000000in}{-0.104167in}}%
\pgfpathcurveto{\pgfqpoint{0.027625in}{-0.104167in}}{\pgfqpoint{0.054123in}{-0.093191in}}{\pgfqpoint{0.073657in}{-0.073657in}}%
\pgfpathcurveto{\pgfqpoint{0.093191in}{-0.054123in}}{\pgfqpoint{0.104167in}{-0.027625in}}{\pgfqpoint{0.104167in}{0.000000in}}%
\pgfpathcurveto{\pgfqpoint{0.104167in}{0.027625in}}{\pgfqpoint{0.093191in}{0.054123in}}{\pgfqpoint{0.073657in}{0.073657in}}%
\pgfpathcurveto{\pgfqpoint{0.054123in}{0.093191in}}{\pgfqpoint{0.027625in}{0.104167in}}{\pgfqpoint{0.000000in}{0.104167in}}%
\pgfpathcurveto{\pgfqpoint{-0.027625in}{0.104167in}}{\pgfqpoint{-0.054123in}{0.093191in}}{\pgfqpoint{-0.073657in}{0.073657in}}%
\pgfpathcurveto{\pgfqpoint{-0.093191in}{0.054123in}}{\pgfqpoint{-0.104167in}{0.027625in}}{\pgfqpoint{-0.104167in}{0.000000in}}%
\pgfpathcurveto{\pgfqpoint{-0.104167in}{-0.027625in}}{\pgfqpoint{-0.093191in}{-0.054123in}}{\pgfqpoint{-0.073657in}{-0.073657in}}%
\pgfpathcurveto{\pgfqpoint{-0.054123in}{-0.093191in}}{\pgfqpoint{-0.027625in}{-0.104167in}}{\pgfqpoint{0.000000in}{-0.104167in}}%
\pgfpathclose%
\pgfusepath{stroke,fill}%
}%
\begin{pgfscope}%
\pgfsys@transformshift{2.000000in}{2.000000in}%
\pgfsys@useobject{currentmarker}{}%
\end{pgfscope}%
\end{pgfscope}%
\begin{pgfscope}%
\pgfpathrectangle{\pgfqpoint{0.000000in}{0.000000in}}{\pgfqpoint{4.000000in}{4.000000in}}%
\pgfusepath{clip}%
\pgfsetrectcap%
\pgfsetroundjoin%
\pgfsetlinewidth{3.011250pt}%
\definecolor{currentstroke}{rgb}{0.250980,0.250980,0.250980}%
\pgfsetstrokecolor{currentstroke}%
\pgfsetdash{}{0pt}%
\pgfpathmoveto{\pgfqpoint{3.200000in}{1.200000in}}%
\pgfpathlineto{\pgfqpoint{3.200000in}{1.350000in}}%
\pgfpathlineto{\pgfqpoint{0.800000in}{1.350000in}}%
\pgfpathlineto{\pgfqpoint{0.800000in}{1.200000in}}%
\pgfusepath{stroke}%
\end{pgfscope}%
\begin{pgfscope}%
\pgfpathrectangle{\pgfqpoint{0.000000in}{0.000000in}}{\pgfqpoint{4.000000in}{4.000000in}}%
\pgfusepath{clip}%
\pgfsetrectcap%
\pgfsetroundjoin%
\pgfsetlinewidth{3.011250pt}%
\definecolor{currentstroke}{rgb}{0.250980,0.250980,0.250980}%
\pgfsetstrokecolor{currentstroke}%
\pgfsetdash{}{0pt}%
\pgfpathmoveto{\pgfqpoint{2.000000in}{1.200000in}}%
\pgfpathlineto{\pgfqpoint{3.200000in}{1.200000in}}%
\pgfusepath{stroke}%
\end{pgfscope}%
\begin{pgfscope}%
\pgfpathrectangle{\pgfqpoint{0.000000in}{0.000000in}}{\pgfqpoint{4.000000in}{4.000000in}}%
\pgfusepath{clip}%
\pgfsetbuttcap%
\pgfsetroundjoin%
\definecolor{currentfill}{rgb}{0.800000,0.800000,0.800000}%
\pgfsetfillcolor{currentfill}%
\pgfsetlinewidth{1.003750pt}%
\definecolor{currentstroke}{rgb}{0.000000,0.000000,0.000000}%
\pgfsetstrokecolor{currentstroke}%
\pgfsetdash{}{0pt}%
\pgfsys@defobject{currentmarker}{\pgfqpoint{-0.104167in}{-0.104167in}}{\pgfqpoint{0.104167in}{0.104167in}}{%
\pgfpathmoveto{\pgfqpoint{0.000000in}{-0.104167in}}%
\pgfpathcurveto{\pgfqpoint{0.027625in}{-0.104167in}}{\pgfqpoint{0.054123in}{-0.093191in}}{\pgfqpoint{0.073657in}{-0.073657in}}%
\pgfpathcurveto{\pgfqpoint{0.093191in}{-0.054123in}}{\pgfqpoint{0.104167in}{-0.027625in}}{\pgfqpoint{0.104167in}{0.000000in}}%
\pgfpathcurveto{\pgfqpoint{0.104167in}{0.027625in}}{\pgfqpoint{0.093191in}{0.054123in}}{\pgfqpoint{0.073657in}{0.073657in}}%
\pgfpathcurveto{\pgfqpoint{0.054123in}{0.093191in}}{\pgfqpoint{0.027625in}{0.104167in}}{\pgfqpoint{0.000000in}{0.104167in}}%
\pgfpathcurveto{\pgfqpoint{-0.027625in}{0.104167in}}{\pgfqpoint{-0.054123in}{0.093191in}}{\pgfqpoint{-0.073657in}{0.073657in}}%
\pgfpathcurveto{\pgfqpoint{-0.093191in}{0.054123in}}{\pgfqpoint{-0.104167in}{0.027625in}}{\pgfqpoint{-0.104167in}{0.000000in}}%
\pgfpathcurveto{\pgfqpoint{-0.104167in}{-0.027625in}}{\pgfqpoint{-0.093191in}{-0.054123in}}{\pgfqpoint{-0.073657in}{-0.073657in}}%
\pgfpathcurveto{\pgfqpoint{-0.054123in}{-0.093191in}}{\pgfqpoint{-0.027625in}{-0.104167in}}{\pgfqpoint{0.000000in}{-0.104167in}}%
\pgfpathclose%
\pgfusepath{stroke,fill}%
}%
\begin{pgfscope}%
\pgfsys@transformshift{0.800000in}{1.200000in}%
\pgfsys@useobject{currentmarker}{}%
\end{pgfscope}%
\end{pgfscope}%
\begin{pgfscope}%
\pgfpathrectangle{\pgfqpoint{0.000000in}{0.000000in}}{\pgfqpoint{4.000000in}{4.000000in}}%
\pgfusepath{clip}%
\pgfsetbuttcap%
\pgfsetroundjoin%
\definecolor{currentfill}{rgb}{0.800000,0.800000,0.800000}%
\pgfsetfillcolor{currentfill}%
\pgfsetlinewidth{1.003750pt}%
\definecolor{currentstroke}{rgb}{0.000000,0.000000,0.000000}%
\pgfsetstrokecolor{currentstroke}%
\pgfsetdash{}{0pt}%
\pgfsys@defobject{currentmarker}{\pgfqpoint{-0.104167in}{-0.104167in}}{\pgfqpoint{0.104167in}{0.104167in}}{%
\pgfpathmoveto{\pgfqpoint{0.000000in}{-0.104167in}}%
\pgfpathcurveto{\pgfqpoint{0.027625in}{-0.104167in}}{\pgfqpoint{0.054123in}{-0.093191in}}{\pgfqpoint{0.073657in}{-0.073657in}}%
\pgfpathcurveto{\pgfqpoint{0.093191in}{-0.054123in}}{\pgfqpoint{0.104167in}{-0.027625in}}{\pgfqpoint{0.104167in}{0.000000in}}%
\pgfpathcurveto{\pgfqpoint{0.104167in}{0.027625in}}{\pgfqpoint{0.093191in}{0.054123in}}{\pgfqpoint{0.073657in}{0.073657in}}%
\pgfpathcurveto{\pgfqpoint{0.054123in}{0.093191in}}{\pgfqpoint{0.027625in}{0.104167in}}{\pgfqpoint{0.000000in}{0.104167in}}%
\pgfpathcurveto{\pgfqpoint{-0.027625in}{0.104167in}}{\pgfqpoint{-0.054123in}{0.093191in}}{\pgfqpoint{-0.073657in}{0.073657in}}%
\pgfpathcurveto{\pgfqpoint{-0.093191in}{0.054123in}}{\pgfqpoint{-0.104167in}{0.027625in}}{\pgfqpoint{-0.104167in}{0.000000in}}%
\pgfpathcurveto{\pgfqpoint{-0.104167in}{-0.027625in}}{\pgfqpoint{-0.093191in}{-0.054123in}}{\pgfqpoint{-0.073657in}{-0.073657in}}%
\pgfpathcurveto{\pgfqpoint{-0.054123in}{-0.093191in}}{\pgfqpoint{-0.027625in}{-0.104167in}}{\pgfqpoint{0.000000in}{-0.104167in}}%
\pgfpathclose%
\pgfusepath{stroke,fill}%
}%
\begin{pgfscope}%
\pgfsys@transformshift{3.200000in}{1.200000in}%
\pgfsys@useobject{currentmarker}{}%
\end{pgfscope}%
\end{pgfscope}%
\begin{pgfscope}%
\pgfpathrectangle{\pgfqpoint{0.000000in}{0.000000in}}{\pgfqpoint{4.000000in}{4.000000in}}%
\pgfusepath{clip}%
\pgfsetbuttcap%
\pgfsetroundjoin%
\definecolor{currentfill}{rgb}{0.800000,0.800000,0.800000}%
\pgfsetfillcolor{currentfill}%
\pgfsetlinewidth{1.003750pt}%
\definecolor{currentstroke}{rgb}{0.000000,0.000000,0.000000}%
\pgfsetstrokecolor{currentstroke}%
\pgfsetdash{}{0pt}%
\pgfsys@defobject{currentmarker}{\pgfqpoint{-0.104167in}{-0.104167in}}{\pgfqpoint{0.104167in}{0.104167in}}{%
\pgfpathmoveto{\pgfqpoint{0.000000in}{-0.104167in}}%
\pgfpathcurveto{\pgfqpoint{0.027625in}{-0.104167in}}{\pgfqpoint{0.054123in}{-0.093191in}}{\pgfqpoint{0.073657in}{-0.073657in}}%
\pgfpathcurveto{\pgfqpoint{0.093191in}{-0.054123in}}{\pgfqpoint{0.104167in}{-0.027625in}}{\pgfqpoint{0.104167in}{0.000000in}}%
\pgfpathcurveto{\pgfqpoint{0.104167in}{0.027625in}}{\pgfqpoint{0.093191in}{0.054123in}}{\pgfqpoint{0.073657in}{0.073657in}}%
\pgfpathcurveto{\pgfqpoint{0.054123in}{0.093191in}}{\pgfqpoint{0.027625in}{0.104167in}}{\pgfqpoint{0.000000in}{0.104167in}}%
\pgfpathcurveto{\pgfqpoint{-0.027625in}{0.104167in}}{\pgfqpoint{-0.054123in}{0.093191in}}{\pgfqpoint{-0.073657in}{0.073657in}}%
\pgfpathcurveto{\pgfqpoint{-0.093191in}{0.054123in}}{\pgfqpoint{-0.104167in}{0.027625in}}{\pgfqpoint{-0.104167in}{0.000000in}}%
\pgfpathcurveto{\pgfqpoint{-0.104167in}{-0.027625in}}{\pgfqpoint{-0.093191in}{-0.054123in}}{\pgfqpoint{-0.073657in}{-0.073657in}}%
\pgfpathcurveto{\pgfqpoint{-0.054123in}{-0.093191in}}{\pgfqpoint{-0.027625in}{-0.104167in}}{\pgfqpoint{0.000000in}{-0.104167in}}%
\pgfpathclose%
\pgfusepath{stroke,fill}%
}%
\begin{pgfscope}%
\pgfsys@transformshift{2.000000in}{1.200000in}%
\pgfsys@useobject{currentmarker}{}%
\end{pgfscope}%
\end{pgfscope}%
\begin{pgfscope}%
\pgfsetroundcap%
\pgfsetroundjoin%
\pgfsetlinewidth{3.011250pt}%
\definecolor{currentstroke}{rgb}{0.250980,0.250980,0.250980}%
\pgfsetstrokecolor{currentstroke}%
\pgfsetdash{}{0pt}%
\pgfpathmoveto{\pgfqpoint{0.750000in}{2.938910in}}%
\pgfpathlineto{\pgfqpoint{0.750000in}{3.350000in}}%
\pgfpathlineto{\pgfqpoint{3.250000in}{3.350000in}}%
\pgfpathlineto{\pgfqpoint{3.250000in}{2.978218in}}%
\pgfusepath{stroke}%
\end{pgfscope}%
\begin{pgfscope}%
\pgfsetroundcap%
\pgfsetroundjoin%
\definecolor{currentfill}{rgb}{0.250980,0.250980,0.250980}%
\pgfsetfillcolor{currentfill}%
\pgfsetlinewidth{3.011250pt}%
\definecolor{currentstroke}{rgb}{0.250980,0.250980,0.250980}%
\pgfsetstrokecolor{currentstroke}%
\pgfsetdash{}{0pt}%
\pgfpathmoveto{\pgfqpoint{3.319444in}{3.089329in}}%
\pgfpathlineto{\pgfqpoint{3.250000in}{2.978218in}}%
\pgfpathlineto{\pgfqpoint{3.180556in}{3.089329in}}%
\pgfpathlineto{\pgfqpoint{3.319444in}{3.089329in}}%
\pgfpathclose%
\pgfusepath{stroke,fill}%
\end{pgfscope}%
\begin{pgfscope}%
\pgfsetroundcap%
\pgfsetroundjoin%
\pgfsetlinewidth{3.011250pt}%
\definecolor{currentstroke}{rgb}{0.250980,0.250980,0.250980}%
\pgfsetstrokecolor{currentstroke}%
\pgfsetdash{}{0pt}%
\pgfpathmoveto{\pgfqpoint{0.938904in}{2.800000in}}%
\pgfpathlineto{\pgfqpoint{1.821773in}{2.800000in}}%
\pgfusepath{stroke}%
\end{pgfscope}%
\begin{pgfscope}%
\pgfsetroundcap%
\pgfsetroundjoin%
\definecolor{currentfill}{rgb}{0.250980,0.250980,0.250980}%
\pgfsetfillcolor{currentfill}%
\pgfsetlinewidth{3.011250pt}%
\definecolor{currentstroke}{rgb}{0.250980,0.250980,0.250980}%
\pgfsetstrokecolor{currentstroke}%
\pgfsetdash{}{0pt}%
\pgfpathmoveto{\pgfqpoint{1.710661in}{2.869444in}}%
\pgfpathlineto{\pgfqpoint{1.821773in}{2.800000in}}%
\pgfpathlineto{\pgfqpoint{1.710661in}{2.730556in}}%
\pgfpathlineto{\pgfqpoint{1.710661in}{2.869444in}}%
\pgfpathclose%
\pgfusepath{stroke,fill}%
\end{pgfscope}%
\begin{pgfscope}%
\pgfsetroundcap%
\pgfsetroundjoin%
\pgfsetlinewidth{3.011250pt}%
\definecolor{currentstroke}{rgb}{0.250980,0.250980,0.250980}%
\pgfsetstrokecolor{currentstroke}%
\pgfsetdash{}{0pt}%
\pgfpathmoveto{\pgfqpoint{3.150000in}{2.938927in}}%
\pgfpathlineto{\pgfqpoint{3.150000in}{3.214000in}}%
\pgfpathlineto{\pgfqpoint{0.850000in}{3.214000in}}%
\pgfpathlineto{\pgfqpoint{0.850000in}{2.978235in}}%
\pgfusepath{stroke}%
\end{pgfscope}%
\begin{pgfscope}%
\pgfsetroundcap%
\pgfsetroundjoin%
\definecolor{currentfill}{rgb}{0.250980,0.250980,0.250980}%
\pgfsetfillcolor{currentfill}%
\pgfsetlinewidth{3.011250pt}%
\definecolor{currentstroke}{rgb}{0.250980,0.250980,0.250980}%
\pgfsetstrokecolor{currentstroke}%
\pgfsetdash{}{0pt}%
\pgfpathmoveto{\pgfqpoint{0.919444in}{3.089346in}}%
\pgfpathlineto{\pgfqpoint{0.850000in}{2.978235in}}%
\pgfpathlineto{\pgfqpoint{0.780556in}{3.089346in}}%
\pgfpathlineto{\pgfqpoint{0.919444in}{3.089346in}}%
\pgfpathclose%
\pgfusepath{stroke,fill}%
\end{pgfscope}%
\begin{pgfscope}%
\pgfsetroundcap%
\pgfsetroundjoin%
\pgfsetlinewidth{3.011250pt}%
\definecolor{currentstroke}{rgb}{0.250980,0.250980,0.250980}%
\pgfsetstrokecolor{currentstroke}%
\pgfsetdash{}{0pt}%
\pgfpathmoveto{\pgfqpoint{3.061096in}{2.800000in}}%
\pgfpathlineto{\pgfqpoint{2.178227in}{2.800000in}}%
\pgfusepath{stroke}%
\end{pgfscope}%
\begin{pgfscope}%
\pgfsetroundcap%
\pgfsetroundjoin%
\definecolor{currentfill}{rgb}{0.250980,0.250980,0.250980}%
\pgfsetfillcolor{currentfill}%
\pgfsetlinewidth{3.011250pt}%
\definecolor{currentstroke}{rgb}{0.250980,0.250980,0.250980}%
\pgfsetstrokecolor{currentstroke}%
\pgfsetdash{}{0pt}%
\pgfpathmoveto{\pgfqpoint{2.289339in}{2.730556in}}%
\pgfpathlineto{\pgfqpoint{2.178227in}{2.800000in}}%
\pgfpathlineto{\pgfqpoint{2.289339in}{2.869444in}}%
\pgfpathlineto{\pgfqpoint{2.289339in}{2.730556in}}%
\pgfpathclose%
\pgfusepath{stroke,fill}%
\end{pgfscope}%
\begin{pgfscope}%
\definecolor{textcolor}{rgb}{0.000000,0.000000,0.000000}%
\pgfsetstrokecolor{textcolor}%
\pgfsetfillcolor{textcolor}%
\pgftext[x=0.800000in,y=2.500000in,,]{\color{textcolor}\sffamily\fontsize{40.000000}{48.000000}\selectfont \(\displaystyle x_1\)}%
\end{pgfscope}%
\begin{pgfscope}%
\definecolor{textcolor}{rgb}{0.000000,0.000000,0.000000}%
\pgfsetstrokecolor{textcolor}%
\pgfsetfillcolor{textcolor}%
\pgftext[x=3.200000in,y=2.500000in,,]{\color{textcolor}\sffamily\fontsize{40.000000}{48.000000}\selectfont \(\displaystyle x_2\)}%
\end{pgfscope}%
\begin{pgfscope}%
\definecolor{textcolor}{rgb}{0.000000,0.000000,0.000000}%
\pgfsetstrokecolor{textcolor}%
\pgfsetfillcolor{textcolor}%
\pgftext[x=2.000000in,y=2.500000in,,]{\color{textcolor}\sffamily\fontsize{40.000000}{48.000000}\selectfont \(\displaystyle x_3\)}%
\end{pgfscope}%
\begin{pgfscope}%
\definecolor{textcolor}{rgb}{0.000000,0.000000,0.000000}%
\pgfsetstrokecolor{textcolor}%
\pgfsetfillcolor{textcolor}%
\pgftext[x=0.650000in,y=1.740192in,,]{\color{textcolor}\sffamily\fontsize{40.000000}{48.000000}\selectfont \(\displaystyle G_+\)}%
\end{pgfscope}%
\begin{pgfscope}%
\definecolor{textcolor}{rgb}{0.000000,0.000000,0.000000}%
\pgfsetstrokecolor{textcolor}%
\pgfsetfillcolor{textcolor}%
\pgftext[x=0.650000in,y=0.940192in,,]{\color{textcolor}\sffamily\fontsize{40.000000}{48.000000}\selectfont \(\displaystyle G_-\)}%
\end{pgfscope}%
\end{pgfpicture}%
\makeatother%
\endgroup%

%% file: fig/example_equivalent_graphs_1.pgf
%% Creator: Matplotlib, PGF backend
%%
%% To include the figure in your LaTeX document, write
%%   \input{<filename>.pgf}
%%
%% Make sure the required packages are loaded in your preamble
%%   \usepackage{pgf}
%%
%% Figures using additional raster images can only be included by \input if
%% they are in the same directory as the main LaTeX file. For loading figures
%% from other directories you can use the `import` package
%%   \usepackage{import}
%%
%% and then include the figures with
%%   \import{<path to file>}{<filename>.pgf}
%%
%% Matplotlib used the following preamble
%%   \usepackage{fontspec}
%%   \setmainfont{DejaVuSerif.ttf}[Path=\detokenize{C:/Users/ccros/Anaconda3/Lib/site-packages/matplotlib/mpl-data/fonts/ttf/}]
%%   \setsansfont{DejaVuSans.ttf}[Path=\detokenize{C:/Users/ccros/Anaconda3/Lib/site-packages/matplotlib/mpl-data/fonts/ttf/}]
%%   \setmonofont{DejaVuSansMono.ttf}[Path=\detokenize{C:/Users/ccros/Anaconda3/Lib/site-packages/matplotlib/mpl-data/fonts/ttf/}]
%%
\begingroup%
\makeatletter%
\begin{pgfpicture}%
\pgfpathrectangle{\pgfpointorigin}{\pgfqpoint{4.000000in}{4.068826in}}%
\pgfusepath{use as bounding box, clip}%
\begin{pgfscope}%
\pgfsetbuttcap%
\pgfsetmiterjoin%
\pgfsetlinewidth{0.000000pt}%
\definecolor{currentstroke}{rgb}{1.000000,1.000000,1.000000}%
\pgfsetstrokecolor{currentstroke}%
\pgfsetstrokeopacity{0.000000}%
\pgfsetdash{}{0pt}%
\pgfpathmoveto{\pgfqpoint{0.000000in}{0.000000in}}%
\pgfpathlineto{\pgfqpoint{4.000000in}{0.000000in}}%
\pgfpathlineto{\pgfqpoint{4.000000in}{4.068826in}}%
\pgfpathlineto{\pgfqpoint{0.000000in}{4.068826in}}%
\pgfpathclose%
\pgfusepath{}%
\end{pgfscope}%
\begin{pgfscope}%
\pgfpathrectangle{\pgfqpoint{0.000000in}{0.000000in}}{\pgfqpoint{4.000000in}{4.000000in}}%
\pgfusepath{clip}%
\pgfsetbuttcap%
\pgfsetroundjoin%
\definecolor{currentfill}{rgb}{0.800000,0.800000,0.800000}%
\pgfsetfillcolor{currentfill}%
\pgfsetlinewidth{1.003750pt}%
\definecolor{currentstroke}{rgb}{0.000000,0.000000,0.000000}%
\pgfsetstrokecolor{currentstroke}%
\pgfsetdash{}{0pt}%
\pgfsys@defobject{currentmarker}{\pgfqpoint{-0.104167in}{-0.104167in}}{\pgfqpoint{0.104167in}{0.104167in}}{%
\pgfpathmoveto{\pgfqpoint{0.000000in}{-0.104167in}}%
\pgfpathcurveto{\pgfqpoint{0.027625in}{-0.104167in}}{\pgfqpoint{0.054123in}{-0.093191in}}{\pgfqpoint{0.073657in}{-0.073657in}}%
\pgfpathcurveto{\pgfqpoint{0.093191in}{-0.054123in}}{\pgfqpoint{0.104167in}{-0.027625in}}{\pgfqpoint{0.104167in}{0.000000in}}%
\pgfpathcurveto{\pgfqpoint{0.104167in}{0.027625in}}{\pgfqpoint{0.093191in}{0.054123in}}{\pgfqpoint{0.073657in}{0.073657in}}%
\pgfpathcurveto{\pgfqpoint{0.054123in}{0.093191in}}{\pgfqpoint{0.027625in}{0.104167in}}{\pgfqpoint{0.000000in}{0.104167in}}%
\pgfpathcurveto{\pgfqpoint{-0.027625in}{0.104167in}}{\pgfqpoint{-0.054123in}{0.093191in}}{\pgfqpoint{-0.073657in}{0.073657in}}%
\pgfpathcurveto{\pgfqpoint{-0.093191in}{0.054123in}}{\pgfqpoint{-0.104167in}{0.027625in}}{\pgfqpoint{-0.104167in}{0.000000in}}%
\pgfpathcurveto{\pgfqpoint{-0.104167in}{-0.027625in}}{\pgfqpoint{-0.093191in}{-0.054123in}}{\pgfqpoint{-0.073657in}{-0.073657in}}%
\pgfpathcurveto{\pgfqpoint{-0.054123in}{-0.093191in}}{\pgfqpoint{-0.027625in}{-0.104167in}}{\pgfqpoint{0.000000in}{-0.104167in}}%
\pgfpathclose%
\pgfusepath{stroke,fill}%
}%
\begin{pgfscope}%
\pgfsys@transformshift{1.177101in}{0.867376in}%
\pgfsys@useobject{currentmarker}{}%
\end{pgfscope}%
\end{pgfscope}%
\begin{pgfscope}%
\pgfpathrectangle{\pgfqpoint{0.000000in}{0.000000in}}{\pgfqpoint{4.000000in}{4.000000in}}%
\pgfusepath{clip}%
\pgfsetbuttcap%
\pgfsetroundjoin%
\definecolor{currentfill}{rgb}{0.800000,0.800000,0.800000}%
\pgfsetfillcolor{currentfill}%
\pgfsetlinewidth{1.003750pt}%
\definecolor{currentstroke}{rgb}{0.000000,0.000000,0.000000}%
\pgfsetstrokecolor{currentstroke}%
\pgfsetdash{}{0pt}%
\pgfsys@defobject{currentmarker}{\pgfqpoint{-0.104167in}{-0.104167in}}{\pgfqpoint{0.104167in}{0.104167in}}{%
\pgfpathmoveto{\pgfqpoint{0.000000in}{-0.104167in}}%
\pgfpathcurveto{\pgfqpoint{0.027625in}{-0.104167in}}{\pgfqpoint{0.054123in}{-0.093191in}}{\pgfqpoint{0.073657in}{-0.073657in}}%
\pgfpathcurveto{\pgfqpoint{0.093191in}{-0.054123in}}{\pgfqpoint{0.104167in}{-0.027625in}}{\pgfqpoint{0.104167in}{0.000000in}}%
\pgfpathcurveto{\pgfqpoint{0.104167in}{0.027625in}}{\pgfqpoint{0.093191in}{0.054123in}}{\pgfqpoint{0.073657in}{0.073657in}}%
\pgfpathcurveto{\pgfqpoint{0.054123in}{0.093191in}}{\pgfqpoint{0.027625in}{0.104167in}}{\pgfqpoint{0.000000in}{0.104167in}}%
\pgfpathcurveto{\pgfqpoint{-0.027625in}{0.104167in}}{\pgfqpoint{-0.054123in}{0.093191in}}{\pgfqpoint{-0.073657in}{0.073657in}}%
\pgfpathcurveto{\pgfqpoint{-0.093191in}{0.054123in}}{\pgfqpoint{-0.104167in}{0.027625in}}{\pgfqpoint{-0.104167in}{0.000000in}}%
\pgfpathcurveto{\pgfqpoint{-0.104167in}{-0.027625in}}{\pgfqpoint{-0.093191in}{-0.054123in}}{\pgfqpoint{-0.073657in}{-0.073657in}}%
\pgfpathcurveto{\pgfqpoint{-0.054123in}{-0.093191in}}{\pgfqpoint{-0.027625in}{-0.104167in}}{\pgfqpoint{0.000000in}{-0.104167in}}%
\pgfpathclose%
\pgfusepath{stroke,fill}%
}%
\begin{pgfscope}%
\pgfsys@transformshift{2.822899in}{0.867376in}%
\pgfsys@useobject{currentmarker}{}%
\end{pgfscope}%
\end{pgfscope}%
\begin{pgfscope}%
\pgfpathrectangle{\pgfqpoint{0.000000in}{0.000000in}}{\pgfqpoint{4.000000in}{4.000000in}}%
\pgfusepath{clip}%
\pgfsetbuttcap%
\pgfsetroundjoin%
\definecolor{currentfill}{rgb}{0.800000,0.800000,0.800000}%
\pgfsetfillcolor{currentfill}%
\pgfsetlinewidth{1.003750pt}%
\definecolor{currentstroke}{rgb}{0.000000,0.000000,0.000000}%
\pgfsetstrokecolor{currentstroke}%
\pgfsetdash{}{0pt}%
\pgfsys@defobject{currentmarker}{\pgfqpoint{-0.104167in}{-0.104167in}}{\pgfqpoint{0.104167in}{0.104167in}}{%
\pgfpathmoveto{\pgfqpoint{0.000000in}{-0.104167in}}%
\pgfpathcurveto{\pgfqpoint{0.027625in}{-0.104167in}}{\pgfqpoint{0.054123in}{-0.093191in}}{\pgfqpoint{0.073657in}{-0.073657in}}%
\pgfpathcurveto{\pgfqpoint{0.093191in}{-0.054123in}}{\pgfqpoint{0.104167in}{-0.027625in}}{\pgfqpoint{0.104167in}{0.000000in}}%
\pgfpathcurveto{\pgfqpoint{0.104167in}{0.027625in}}{\pgfqpoint{0.093191in}{0.054123in}}{\pgfqpoint{0.073657in}{0.073657in}}%
\pgfpathcurveto{\pgfqpoint{0.054123in}{0.093191in}}{\pgfqpoint{0.027625in}{0.104167in}}{\pgfqpoint{0.000000in}{0.104167in}}%
\pgfpathcurveto{\pgfqpoint{-0.027625in}{0.104167in}}{\pgfqpoint{-0.054123in}{0.093191in}}{\pgfqpoint{-0.073657in}{0.073657in}}%
\pgfpathcurveto{\pgfqpoint{-0.093191in}{0.054123in}}{\pgfqpoint{-0.104167in}{0.027625in}}{\pgfqpoint{-0.104167in}{0.000000in}}%
\pgfpathcurveto{\pgfqpoint{-0.104167in}{-0.027625in}}{\pgfqpoint{-0.093191in}{-0.054123in}}{\pgfqpoint{-0.073657in}{-0.073657in}}%
\pgfpathcurveto{\pgfqpoint{-0.054123in}{-0.093191in}}{\pgfqpoint{-0.027625in}{-0.104167in}}{\pgfqpoint{0.000000in}{-0.104167in}}%
\pgfpathclose%
\pgfusepath{stroke,fill}%
}%
\begin{pgfscope}%
\pgfsys@transformshift{3.331479in}{2.432624in}%
\pgfsys@useobject{currentmarker}{}%
\end{pgfscope}%
\end{pgfscope}%
\begin{pgfscope}%
\pgfpathrectangle{\pgfqpoint{0.000000in}{0.000000in}}{\pgfqpoint{4.000000in}{4.000000in}}%
\pgfusepath{clip}%
\pgfsetbuttcap%
\pgfsetroundjoin%
\definecolor{currentfill}{rgb}{0.800000,0.800000,0.800000}%
\pgfsetfillcolor{currentfill}%
\pgfsetlinewidth{1.003750pt}%
\definecolor{currentstroke}{rgb}{0.000000,0.000000,0.000000}%
\pgfsetstrokecolor{currentstroke}%
\pgfsetdash{}{0pt}%
\pgfsys@defobject{currentmarker}{\pgfqpoint{-0.104167in}{-0.104167in}}{\pgfqpoint{0.104167in}{0.104167in}}{%
\pgfpathmoveto{\pgfqpoint{0.000000in}{-0.104167in}}%
\pgfpathcurveto{\pgfqpoint{0.027625in}{-0.104167in}}{\pgfqpoint{0.054123in}{-0.093191in}}{\pgfqpoint{0.073657in}{-0.073657in}}%
\pgfpathcurveto{\pgfqpoint{0.093191in}{-0.054123in}}{\pgfqpoint{0.104167in}{-0.027625in}}{\pgfqpoint{0.104167in}{0.000000in}}%
\pgfpathcurveto{\pgfqpoint{0.104167in}{0.027625in}}{\pgfqpoint{0.093191in}{0.054123in}}{\pgfqpoint{0.073657in}{0.073657in}}%
\pgfpathcurveto{\pgfqpoint{0.054123in}{0.093191in}}{\pgfqpoint{0.027625in}{0.104167in}}{\pgfqpoint{0.000000in}{0.104167in}}%
\pgfpathcurveto{\pgfqpoint{-0.027625in}{0.104167in}}{\pgfqpoint{-0.054123in}{0.093191in}}{\pgfqpoint{-0.073657in}{0.073657in}}%
\pgfpathcurveto{\pgfqpoint{-0.093191in}{0.054123in}}{\pgfqpoint{-0.104167in}{0.027625in}}{\pgfqpoint{-0.104167in}{0.000000in}}%
\pgfpathcurveto{\pgfqpoint{-0.104167in}{-0.027625in}}{\pgfqpoint{-0.093191in}{-0.054123in}}{\pgfqpoint{-0.073657in}{-0.073657in}}%
\pgfpathcurveto{\pgfqpoint{-0.054123in}{-0.093191in}}{\pgfqpoint{-0.027625in}{-0.104167in}}{\pgfqpoint{0.000000in}{-0.104167in}}%
\pgfpathclose%
\pgfusepath{stroke,fill}%
}%
\begin{pgfscope}%
\pgfsys@transformshift{2.000000in}{3.400000in}%
\pgfsys@useobject{currentmarker}{}%
\end{pgfscope}%
\end{pgfscope}%
\begin{pgfscope}%
\pgfpathrectangle{\pgfqpoint{0.000000in}{0.000000in}}{\pgfqpoint{4.000000in}{4.000000in}}%
\pgfusepath{clip}%
\pgfsetbuttcap%
\pgfsetroundjoin%
\definecolor{currentfill}{rgb}{0.800000,0.800000,0.800000}%
\pgfsetfillcolor{currentfill}%
\pgfsetlinewidth{1.003750pt}%
\definecolor{currentstroke}{rgb}{0.000000,0.000000,0.000000}%
\pgfsetstrokecolor{currentstroke}%
\pgfsetdash{}{0pt}%
\pgfsys@defobject{currentmarker}{\pgfqpoint{-0.104167in}{-0.104167in}}{\pgfqpoint{0.104167in}{0.104167in}}{%
\pgfpathmoveto{\pgfqpoint{0.000000in}{-0.104167in}}%
\pgfpathcurveto{\pgfqpoint{0.027625in}{-0.104167in}}{\pgfqpoint{0.054123in}{-0.093191in}}{\pgfqpoint{0.073657in}{-0.073657in}}%
\pgfpathcurveto{\pgfqpoint{0.093191in}{-0.054123in}}{\pgfqpoint{0.104167in}{-0.027625in}}{\pgfqpoint{0.104167in}{0.000000in}}%
\pgfpathcurveto{\pgfqpoint{0.104167in}{0.027625in}}{\pgfqpoint{0.093191in}{0.054123in}}{\pgfqpoint{0.073657in}{0.073657in}}%
\pgfpathcurveto{\pgfqpoint{0.054123in}{0.093191in}}{\pgfqpoint{0.027625in}{0.104167in}}{\pgfqpoint{0.000000in}{0.104167in}}%
\pgfpathcurveto{\pgfqpoint{-0.027625in}{0.104167in}}{\pgfqpoint{-0.054123in}{0.093191in}}{\pgfqpoint{-0.073657in}{0.073657in}}%
\pgfpathcurveto{\pgfqpoint{-0.093191in}{0.054123in}}{\pgfqpoint{-0.104167in}{0.027625in}}{\pgfqpoint{-0.104167in}{0.000000in}}%
\pgfpathcurveto{\pgfqpoint{-0.104167in}{-0.027625in}}{\pgfqpoint{-0.093191in}{-0.054123in}}{\pgfqpoint{-0.073657in}{-0.073657in}}%
\pgfpathcurveto{\pgfqpoint{-0.054123in}{-0.093191in}}{\pgfqpoint{-0.027625in}{-0.104167in}}{\pgfqpoint{0.000000in}{-0.104167in}}%
\pgfpathclose%
\pgfusepath{stroke,fill}%
}%
\begin{pgfscope}%
\pgfsys@transformshift{0.668521in}{2.432624in}%
\pgfsys@useobject{currentmarker}{}%
\end{pgfscope}%
\end{pgfscope}%
\begin{pgfscope}%
\pgfsetroundcap%
\pgfsetroundjoin%
\pgfsetlinewidth{3.011250pt}%
\definecolor{currentstroke}{rgb}{0.250980,0.250980,0.250980}%
\pgfsetstrokecolor{currentstroke}%
\pgfsetdash{}{0pt}%
\pgfpathmoveto{\pgfqpoint{1.315975in}{0.867376in}}%
\pgfpathlineto{\pgfqpoint{2.644662in}{0.867376in}}%
\pgfusepath{stroke}%
\end{pgfscope}%
\begin{pgfscope}%
\pgfsetroundcap%
\pgfsetroundjoin%
\definecolor{currentfill}{rgb}{0.250980,0.250980,0.250980}%
\pgfsetfillcolor{currentfill}%
\pgfsetlinewidth{3.011250pt}%
\definecolor{currentstroke}{rgb}{0.250980,0.250980,0.250980}%
\pgfsetstrokecolor{currentstroke}%
\pgfsetdash{}{0pt}%
\pgfpathmoveto{\pgfqpoint{2.533551in}{0.936821in}}%
\pgfpathlineto{\pgfqpoint{2.644662in}{0.867376in}}%
\pgfpathlineto{\pgfqpoint{2.533551in}{0.797932in}}%
\pgfpathlineto{\pgfqpoint{2.533551in}{0.936821in}}%
\pgfpathclose%
\pgfusepath{stroke,fill}%
\end{pgfscope}%
\begin{pgfscope}%
\pgfsetroundcap%
\pgfsetroundjoin%
\pgfsetlinewidth{3.011250pt}%
\definecolor{currentstroke}{rgb}{0.250980,0.250980,0.250980}%
\pgfsetstrokecolor{currentstroke}%
\pgfsetdash{}{0pt}%
\pgfpathmoveto{\pgfqpoint{1.134186in}{0.999454in}}%
\pgfpathlineto{\pgfqpoint{0.723599in}{2.263110in}}%
\pgfusepath{stroke}%
\end{pgfscope}%
\begin{pgfscope}%
\pgfsetroundcap%
\pgfsetroundjoin%
\definecolor{currentfill}{rgb}{0.250980,0.250980,0.250980}%
\pgfsetfillcolor{currentfill}%
\pgfsetlinewidth{3.011250pt}%
\definecolor{currentstroke}{rgb}{0.250980,0.250980,0.250980}%
\pgfsetstrokecolor{currentstroke}%
\pgfsetdash{}{0pt}%
\pgfpathmoveto{\pgfqpoint{0.691889in}{2.135978in}}%
\pgfpathlineto{\pgfqpoint{0.723599in}{2.263110in}}%
\pgfpathlineto{\pgfqpoint{0.823980in}{2.178897in}}%
\pgfpathlineto{\pgfqpoint{0.691889in}{2.135978in}}%
\pgfpathclose%
\pgfusepath{stroke,fill}%
\end{pgfscope}%
\begin{pgfscope}%
\pgfsetroundcap%
\pgfsetroundjoin%
\pgfsetlinewidth{3.011250pt}%
\definecolor{currentstroke}{rgb}{0.250980,0.250980,0.250980}%
\pgfsetstrokecolor{currentstroke}%
\pgfsetdash{}{0pt}%
\pgfpathmoveto{\pgfqpoint{2.865814in}{0.999454in}}%
\pgfpathlineto{\pgfqpoint{3.276401in}{2.263110in}}%
\pgfusepath{stroke}%
\end{pgfscope}%
\begin{pgfscope}%
\pgfsetroundcap%
\pgfsetroundjoin%
\definecolor{currentfill}{rgb}{0.250980,0.250980,0.250980}%
\pgfsetfillcolor{currentfill}%
\pgfsetlinewidth{3.011250pt}%
\definecolor{currentstroke}{rgb}{0.250980,0.250980,0.250980}%
\pgfsetstrokecolor{currentstroke}%
\pgfsetdash{}{0pt}%
\pgfpathmoveto{\pgfqpoint{3.176020in}{2.178897in}}%
\pgfpathlineto{\pgfqpoint{3.276401in}{2.263110in}}%
\pgfpathlineto{\pgfqpoint{3.308111in}{2.135978in}}%
\pgfpathlineto{\pgfqpoint{3.176020in}{2.178897in}}%
\pgfpathclose%
\pgfusepath{stroke,fill}%
\end{pgfscope}%
\begin{pgfscope}%
\pgfsetroundcap%
\pgfsetroundjoin%
\pgfsetlinewidth{3.011250pt}%
\definecolor{currentstroke}{rgb}{0.250980,0.250980,0.250980}%
\pgfsetstrokecolor{currentstroke}%
\pgfsetdash{}{0pt}%
\pgfpathmoveto{\pgfqpoint{2.827522in}{1.014954in}}%
\pgfpathlineto{\pgfqpoint{2.102628in}{3.245947in}}%
\pgfusepath{stroke}%
\end{pgfscope}%
\begin{pgfscope}%
\pgfsetroundcap%
\pgfsetroundjoin%
\definecolor{currentfill}{rgb}{0.250980,0.250980,0.250980}%
\pgfsetfillcolor{currentfill}%
\pgfsetlinewidth{3.011250pt}%
\definecolor{currentstroke}{rgb}{0.250980,0.250980,0.250980}%
\pgfsetstrokecolor{currentstroke}%
\pgfsetdash{}{0pt}%
\pgfpathmoveto{\pgfqpoint{2.070918in}{3.118814in}}%
\pgfpathlineto{\pgfqpoint{2.102628in}{3.245947in}}%
\pgfpathlineto{\pgfqpoint{2.203009in}{3.161733in}}%
\pgfpathlineto{\pgfqpoint{2.070918in}{3.118814in}}%
\pgfpathclose%
\pgfusepath{stroke,fill}%
\end{pgfscope}%
\begin{pgfscope}%
\pgfsetroundcap%
\pgfsetroundjoin%
\pgfsetlinewidth{3.011250pt}%
\definecolor{currentstroke}{rgb}{0.250980,0.250980,0.250980}%
\pgfsetstrokecolor{currentstroke}%
\pgfsetdash{}{0pt}%
\pgfpathmoveto{\pgfqpoint{2.710506in}{0.949035in}}%
\pgfpathlineto{\pgfqpoint{0.812710in}{2.327865in}}%
\pgfusepath{stroke}%
\end{pgfscope}%
\begin{pgfscope}%
\pgfsetroundcap%
\pgfsetroundjoin%
\definecolor{currentfill}{rgb}{0.250980,0.250980,0.250980}%
\pgfsetfillcolor{currentfill}%
\pgfsetlinewidth{3.011250pt}%
\definecolor{currentstroke}{rgb}{0.250980,0.250980,0.250980}%
\pgfsetstrokecolor{currentstroke}%
\pgfsetdash{}{0pt}%
\pgfpathmoveto{\pgfqpoint{0.861782in}{2.206373in}}%
\pgfpathlineto{\pgfqpoint{0.812710in}{2.327865in}}%
\pgfpathlineto{\pgfqpoint{0.943419in}{2.318737in}}%
\pgfpathlineto{\pgfqpoint{0.861782in}{2.206373in}}%
\pgfpathclose%
\pgfusepath{stroke,fill}%
\end{pgfscope}%
\begin{pgfscope}%
\pgfsetroundcap%
\pgfsetroundjoin%
\pgfsetlinewidth{3.011250pt}%
\definecolor{currentstroke}{rgb}{0.250980,0.250980,0.250980}%
\pgfsetstrokecolor{currentstroke}%
\pgfsetdash{}{0pt}%
\pgfpathmoveto{\pgfqpoint{3.248517in}{2.554703in}}%
\pgfpathlineto{\pgfqpoint{2.173586in}{3.335686in}}%
\pgfusepath{stroke}%
\end{pgfscope}%
\begin{pgfscope}%
\pgfsetroundcap%
\pgfsetroundjoin%
\definecolor{currentfill}{rgb}{0.250980,0.250980,0.250980}%
\pgfsetfillcolor{currentfill}%
\pgfsetlinewidth{3.011250pt}%
\definecolor{currentstroke}{rgb}{0.250980,0.250980,0.250980}%
\pgfsetstrokecolor{currentstroke}%
\pgfsetdash{}{0pt}%
\pgfpathmoveto{\pgfqpoint{2.222658in}{3.214195in}}%
\pgfpathlineto{\pgfqpoint{2.173586in}{3.335686in}}%
\pgfpathlineto{\pgfqpoint{2.304295in}{3.326558in}}%
\pgfpathlineto{\pgfqpoint{2.222658in}{3.214195in}}%
\pgfpathclose%
\pgfusepath{stroke,fill}%
\end{pgfscope}%
\begin{pgfscope}%
\pgfsetroundcap%
\pgfsetroundjoin%
\pgfsetlinewidth{3.011250pt}%
\definecolor{currentstroke}{rgb}{0.250980,0.250980,0.250980}%
\pgfsetstrokecolor{currentstroke}%
\pgfsetdash{}{0pt}%
\pgfpathmoveto{\pgfqpoint{1.957070in}{3.267874in}}%
\pgfpathlineto{\pgfqpoint{1.232176in}{1.036880in}}%
\pgfusepath{stroke}%
\end{pgfscope}%
\begin{pgfscope}%
\pgfsetroundcap%
\pgfsetroundjoin%
\definecolor{currentfill}{rgb}{0.250980,0.250980,0.250980}%
\pgfsetfillcolor{currentfill}%
\pgfsetlinewidth{3.011250pt}%
\definecolor{currentstroke}{rgb}{0.250980,0.250980,0.250980}%
\pgfsetstrokecolor{currentstroke}%
\pgfsetdash{}{0pt}%
\pgfpathmoveto{\pgfqpoint{1.332557in}{1.121094in}}%
\pgfpathlineto{\pgfqpoint{1.232176in}{1.036880in}}%
\pgfpathlineto{\pgfqpoint{1.200465in}{1.164013in}}%
\pgfpathlineto{\pgfqpoint{1.332557in}{1.121094in}}%
\pgfpathclose%
\pgfusepath{stroke,fill}%
\end{pgfscope}%
\begin{pgfscope}%
\pgfsetroundcap%
\pgfsetroundjoin%
\pgfsetlinewidth{3.011250pt}%
\definecolor{currentstroke}{rgb}{0.250980,0.250980,0.250980}%
\pgfsetstrokecolor{currentstroke}%
\pgfsetdash{}{0pt}%
\pgfpathmoveto{\pgfqpoint{1.995378in}{3.252423in}}%
\pgfpathlineto{\pgfqpoint{2.720271in}{1.021429in}}%
\pgfusepath{stroke}%
\end{pgfscope}%
\begin{pgfscope}%
\pgfsetroundcap%
\pgfsetroundjoin%
\definecolor{currentfill}{rgb}{0.250980,0.250980,0.250980}%
\pgfsetfillcolor{currentfill}%
\pgfsetlinewidth{3.011250pt}%
\definecolor{currentstroke}{rgb}{0.250980,0.250980,0.250980}%
\pgfsetstrokecolor{currentstroke}%
\pgfsetdash{}{0pt}%
\pgfpathmoveto{\pgfqpoint{2.751982in}{1.148562in}}%
\pgfpathlineto{\pgfqpoint{2.720271in}{1.021429in}}%
\pgfpathlineto{\pgfqpoint{2.619891in}{1.105643in}}%
\pgfpathlineto{\pgfqpoint{2.751982in}{1.148562in}}%
\pgfpathclose%
\pgfusepath{stroke,fill}%
\end{pgfscope}%
\begin{pgfscope}%
\pgfsetroundcap%
\pgfsetroundjoin%
\pgfsetlinewidth{3.011250pt}%
\definecolor{currentstroke}{rgb}{0.250980,0.250980,0.250980}%
\pgfsetstrokecolor{currentstroke}%
\pgfsetdash{}{0pt}%
\pgfpathmoveto{\pgfqpoint{2.082962in}{3.277921in}}%
\pgfpathlineto{\pgfqpoint{3.157893in}{2.496938in}}%
\pgfusepath{stroke}%
\end{pgfscope}%
\begin{pgfscope}%
\pgfsetroundcap%
\pgfsetroundjoin%
\definecolor{currentfill}{rgb}{0.250980,0.250980,0.250980}%
\pgfsetfillcolor{currentfill}%
\pgfsetlinewidth{3.011250pt}%
\definecolor{currentstroke}{rgb}{0.250980,0.250980,0.250980}%
\pgfsetstrokecolor{currentstroke}%
\pgfsetdash{}{0pt}%
\pgfpathmoveto{\pgfqpoint{3.108821in}{2.618429in}}%
\pgfpathlineto{\pgfqpoint{3.157893in}{2.496938in}}%
\pgfpathlineto{\pgfqpoint{3.027184in}{2.506066in}}%
\pgfpathlineto{\pgfqpoint{3.108821in}{2.618429in}}%
\pgfpathclose%
\pgfusepath{stroke,fill}%
\end{pgfscope}%
\begin{pgfscope}%
\pgfsetroundcap%
\pgfsetroundjoin%
\pgfsetlinewidth{3.011250pt}%
\definecolor{currentstroke}{rgb}{0.250980,0.250980,0.250980}%
\pgfsetstrokecolor{currentstroke}%
\pgfsetdash{}{0pt}%
\pgfpathmoveto{\pgfqpoint{1.887648in}{3.318372in}}%
\pgfpathlineto{\pgfqpoint{0.812718in}{2.537389in}}%
\pgfusepath{stroke}%
\end{pgfscope}%
\begin{pgfscope}%
\pgfsetroundcap%
\pgfsetroundjoin%
\definecolor{currentfill}{rgb}{0.250980,0.250980,0.250980}%
\pgfsetfillcolor{currentfill}%
\pgfsetlinewidth{3.011250pt}%
\definecolor{currentstroke}{rgb}{0.250980,0.250980,0.250980}%
\pgfsetstrokecolor{currentstroke}%
\pgfsetdash{}{0pt}%
\pgfpathmoveto{\pgfqpoint{0.943427in}{2.546517in}}%
\pgfpathlineto{\pgfqpoint{0.812718in}{2.537389in}}%
\pgfpathlineto{\pgfqpoint{0.861790in}{2.658880in}}%
\pgfpathlineto{\pgfqpoint{0.943427in}{2.546517in}}%
\pgfpathclose%
\pgfusepath{stroke,fill}%
\end{pgfscope}%
\begin{pgfscope}%
\pgfsetroundcap%
\pgfsetroundjoin%
\pgfsetlinewidth{3.011250pt}%
\definecolor{currentstroke}{rgb}{0.250980,0.250980,0.250980}%
\pgfsetstrokecolor{currentstroke}%
\pgfsetdash{}{0pt}%
\pgfpathmoveto{\pgfqpoint{0.807447in}{2.432624in}}%
\pgfpathlineto{\pgfqpoint{3.153252in}{2.432624in}}%
\pgfusepath{stroke}%
\end{pgfscope}%
\begin{pgfscope}%
\pgfsetroundcap%
\pgfsetroundjoin%
\definecolor{currentfill}{rgb}{0.250980,0.250980,0.250980}%
\pgfsetfillcolor{currentfill}%
\pgfsetlinewidth{3.011250pt}%
\definecolor{currentstroke}{rgb}{0.250980,0.250980,0.250980}%
\pgfsetstrokecolor{currentstroke}%
\pgfsetdash{}{0pt}%
\pgfpathmoveto{\pgfqpoint{3.042141in}{2.502068in}}%
\pgfpathlineto{\pgfqpoint{3.153252in}{2.432624in}}%
\pgfpathlineto{\pgfqpoint{3.042141in}{2.363179in}}%
\pgfpathlineto{\pgfqpoint{3.042141in}{2.502068in}}%
\pgfpathclose%
\pgfusepath{stroke,fill}%
\end{pgfscope}%
\begin{pgfscope}%
\definecolor{textcolor}{rgb}{0.000000,0.000000,0.000000}%
\pgfsetstrokecolor{textcolor}%
\pgfsetfillcolor{textcolor}%
\pgftext[x=0.941987in,y=0.543769in,,]{\color{textcolor}\sffamily\fontsize{40.000000}{48.000000}\selectfont \(\displaystyle 1\)}%
\end{pgfscope}%
\begin{pgfscope}%
\definecolor{textcolor}{rgb}{0.000000,0.000000,0.000000}%
\pgfsetstrokecolor{textcolor}%
\pgfsetfillcolor{textcolor}%
\pgftext[x=3.058013in,y=0.543769in,,]{\color{textcolor}\sffamily\fontsize{40.000000}{48.000000}\selectfont \(\displaystyle 2\)}%
\end{pgfscope}%
\begin{pgfscope}%
\definecolor{textcolor}{rgb}{0.000000,0.000000,0.000000}%
\pgfsetstrokecolor{textcolor}%
\pgfsetfillcolor{textcolor}%
\pgftext[x=3.711902in,y=2.556231in,,]{\color{textcolor}\sffamily\fontsize{40.000000}{48.000000}\selectfont \(\displaystyle 3\)}%
\end{pgfscope}%
\begin{pgfscope}%
\definecolor{textcolor}{rgb}{0.000000,0.000000,0.000000}%
\pgfsetstrokecolor{textcolor}%
\pgfsetfillcolor{textcolor}%
\pgftext[x=2.000000in,y=3.800000in,,]{\color{textcolor}\sffamily\fontsize{40.000000}{48.000000}\selectfont \(\displaystyle 4\)}%
\end{pgfscope}%
\begin{pgfscope}%
\definecolor{textcolor}{rgb}{0.000000,0.000000,0.000000}%
\pgfsetstrokecolor{textcolor}%
\pgfsetfillcolor{textcolor}%
\pgftext[x=0.288098in,y=2.556231in,,]{\color{textcolor}\sffamily\fontsize{40.000000}{48.000000}\selectfont \(\displaystyle 5\)}%
\end{pgfscope}%
\end{pgfpicture}%
\makeatother%
\endgroup%

%% file: fig/example_equivalent_graphs_2.pgf
%% Creator: Matplotlib, PGF backend
%%
%% To include the figure in your LaTeX document, write
%%   \input{<filename>.pgf}
%%
%% Make sure the required packages are loaded in your preamble
%%   \usepackage{pgf}
%%
%% Figures using additional raster images can only be included by \input if
%% they are in the same directory as the main LaTeX file. For loading figures
%% from other directories you can use the `import` package
%%   \usepackage{import}
%%
%% and then include the figures with
%%   \import{<path to file>}{<filename>.pgf}
%%
%% Matplotlib used the following preamble
%%   \usepackage{fontspec}
%%   \setmainfont{DejaVuSerif.ttf}[Path=\detokenize{C:/Users/ccros/Anaconda3/Lib/site-packages/matplotlib/mpl-data/fonts/ttf/}]
%%   \setsansfont{DejaVuSans.ttf}[Path=\detokenize{C:/Users/ccros/Anaconda3/Lib/site-packages/matplotlib/mpl-data/fonts/ttf/}]
%%   \setmonofont{DejaVuSansMono.ttf}[Path=\detokenize{C:/Users/ccros/Anaconda3/Lib/site-packages/matplotlib/mpl-data/fonts/ttf/}]
%%
\begingroup%
\makeatletter%
\begin{pgfpicture}%
\pgfpathrectangle{\pgfpointorigin}{\pgfqpoint{4.000000in}{4.068826in}}%
\pgfusepath{use as bounding box, clip}%
\begin{pgfscope}%
\pgfsetbuttcap%
\pgfsetmiterjoin%
\pgfsetlinewidth{0.000000pt}%
\definecolor{currentstroke}{rgb}{1.000000,1.000000,1.000000}%
\pgfsetstrokecolor{currentstroke}%
\pgfsetstrokeopacity{0.000000}%
\pgfsetdash{}{0pt}%
\pgfpathmoveto{\pgfqpoint{0.000000in}{0.000000in}}%
\pgfpathlineto{\pgfqpoint{4.000000in}{0.000000in}}%
\pgfpathlineto{\pgfqpoint{4.000000in}{4.068826in}}%
\pgfpathlineto{\pgfqpoint{0.000000in}{4.068826in}}%
\pgfpathclose%
\pgfusepath{}%
\end{pgfscope}%
\begin{pgfscope}%
\pgfpathrectangle{\pgfqpoint{0.000000in}{0.000000in}}{\pgfqpoint{4.000000in}{4.000000in}}%
\pgfusepath{clip}%
\pgfsetbuttcap%
\pgfsetroundjoin%
\definecolor{currentfill}{rgb}{0.800000,0.800000,0.800000}%
\pgfsetfillcolor{currentfill}%
\pgfsetlinewidth{1.003750pt}%
\definecolor{currentstroke}{rgb}{0.000000,0.000000,0.000000}%
\pgfsetstrokecolor{currentstroke}%
\pgfsetdash{}{0pt}%
\pgfsys@defobject{currentmarker}{\pgfqpoint{-0.104167in}{-0.104167in}}{\pgfqpoint{0.104167in}{0.104167in}}{%
\pgfpathmoveto{\pgfqpoint{0.000000in}{-0.104167in}}%
\pgfpathcurveto{\pgfqpoint{0.027625in}{-0.104167in}}{\pgfqpoint{0.054123in}{-0.093191in}}{\pgfqpoint{0.073657in}{-0.073657in}}%
\pgfpathcurveto{\pgfqpoint{0.093191in}{-0.054123in}}{\pgfqpoint{0.104167in}{-0.027625in}}{\pgfqpoint{0.104167in}{0.000000in}}%
\pgfpathcurveto{\pgfqpoint{0.104167in}{0.027625in}}{\pgfqpoint{0.093191in}{0.054123in}}{\pgfqpoint{0.073657in}{0.073657in}}%
\pgfpathcurveto{\pgfqpoint{0.054123in}{0.093191in}}{\pgfqpoint{0.027625in}{0.104167in}}{\pgfqpoint{0.000000in}{0.104167in}}%
\pgfpathcurveto{\pgfqpoint{-0.027625in}{0.104167in}}{\pgfqpoint{-0.054123in}{0.093191in}}{\pgfqpoint{-0.073657in}{0.073657in}}%
\pgfpathcurveto{\pgfqpoint{-0.093191in}{0.054123in}}{\pgfqpoint{-0.104167in}{0.027625in}}{\pgfqpoint{-0.104167in}{0.000000in}}%
\pgfpathcurveto{\pgfqpoint{-0.104167in}{-0.027625in}}{\pgfqpoint{-0.093191in}{-0.054123in}}{\pgfqpoint{-0.073657in}{-0.073657in}}%
\pgfpathcurveto{\pgfqpoint{-0.054123in}{-0.093191in}}{\pgfqpoint{-0.027625in}{-0.104167in}}{\pgfqpoint{0.000000in}{-0.104167in}}%
\pgfpathclose%
\pgfusepath{stroke,fill}%
}%
\begin{pgfscope}%
\pgfsys@transformshift{1.177101in}{0.867376in}%
\pgfsys@useobject{currentmarker}{}%
\end{pgfscope}%
\end{pgfscope}%
\begin{pgfscope}%
\pgfpathrectangle{\pgfqpoint{0.000000in}{0.000000in}}{\pgfqpoint{4.000000in}{4.000000in}}%
\pgfusepath{clip}%
\pgfsetbuttcap%
\pgfsetroundjoin%
\definecolor{currentfill}{rgb}{0.800000,0.800000,0.800000}%
\pgfsetfillcolor{currentfill}%
\pgfsetlinewidth{1.003750pt}%
\definecolor{currentstroke}{rgb}{0.000000,0.000000,0.000000}%
\pgfsetstrokecolor{currentstroke}%
\pgfsetdash{}{0pt}%
\pgfsys@defobject{currentmarker}{\pgfqpoint{-0.104167in}{-0.104167in}}{\pgfqpoint{0.104167in}{0.104167in}}{%
\pgfpathmoveto{\pgfqpoint{0.000000in}{-0.104167in}}%
\pgfpathcurveto{\pgfqpoint{0.027625in}{-0.104167in}}{\pgfqpoint{0.054123in}{-0.093191in}}{\pgfqpoint{0.073657in}{-0.073657in}}%
\pgfpathcurveto{\pgfqpoint{0.093191in}{-0.054123in}}{\pgfqpoint{0.104167in}{-0.027625in}}{\pgfqpoint{0.104167in}{0.000000in}}%
\pgfpathcurveto{\pgfqpoint{0.104167in}{0.027625in}}{\pgfqpoint{0.093191in}{0.054123in}}{\pgfqpoint{0.073657in}{0.073657in}}%
\pgfpathcurveto{\pgfqpoint{0.054123in}{0.093191in}}{\pgfqpoint{0.027625in}{0.104167in}}{\pgfqpoint{0.000000in}{0.104167in}}%
\pgfpathcurveto{\pgfqpoint{-0.027625in}{0.104167in}}{\pgfqpoint{-0.054123in}{0.093191in}}{\pgfqpoint{-0.073657in}{0.073657in}}%
\pgfpathcurveto{\pgfqpoint{-0.093191in}{0.054123in}}{\pgfqpoint{-0.104167in}{0.027625in}}{\pgfqpoint{-0.104167in}{0.000000in}}%
\pgfpathcurveto{\pgfqpoint{-0.104167in}{-0.027625in}}{\pgfqpoint{-0.093191in}{-0.054123in}}{\pgfqpoint{-0.073657in}{-0.073657in}}%
\pgfpathcurveto{\pgfqpoint{-0.054123in}{-0.093191in}}{\pgfqpoint{-0.027625in}{-0.104167in}}{\pgfqpoint{0.000000in}{-0.104167in}}%
\pgfpathclose%
\pgfusepath{stroke,fill}%
}%
\begin{pgfscope}%
\pgfsys@transformshift{2.822899in}{0.867376in}%
\pgfsys@useobject{currentmarker}{}%
\end{pgfscope}%
\end{pgfscope}%
\begin{pgfscope}%
\pgfpathrectangle{\pgfqpoint{0.000000in}{0.000000in}}{\pgfqpoint{4.000000in}{4.000000in}}%
\pgfusepath{clip}%
\pgfsetbuttcap%
\pgfsetroundjoin%
\definecolor{currentfill}{rgb}{0.800000,0.800000,0.800000}%
\pgfsetfillcolor{currentfill}%
\pgfsetlinewidth{1.003750pt}%
\definecolor{currentstroke}{rgb}{0.000000,0.000000,0.000000}%
\pgfsetstrokecolor{currentstroke}%
\pgfsetdash{}{0pt}%
\pgfsys@defobject{currentmarker}{\pgfqpoint{-0.104167in}{-0.104167in}}{\pgfqpoint{0.104167in}{0.104167in}}{%
\pgfpathmoveto{\pgfqpoint{0.000000in}{-0.104167in}}%
\pgfpathcurveto{\pgfqpoint{0.027625in}{-0.104167in}}{\pgfqpoint{0.054123in}{-0.093191in}}{\pgfqpoint{0.073657in}{-0.073657in}}%
\pgfpathcurveto{\pgfqpoint{0.093191in}{-0.054123in}}{\pgfqpoint{0.104167in}{-0.027625in}}{\pgfqpoint{0.104167in}{0.000000in}}%
\pgfpathcurveto{\pgfqpoint{0.104167in}{0.027625in}}{\pgfqpoint{0.093191in}{0.054123in}}{\pgfqpoint{0.073657in}{0.073657in}}%
\pgfpathcurveto{\pgfqpoint{0.054123in}{0.093191in}}{\pgfqpoint{0.027625in}{0.104167in}}{\pgfqpoint{0.000000in}{0.104167in}}%
\pgfpathcurveto{\pgfqpoint{-0.027625in}{0.104167in}}{\pgfqpoint{-0.054123in}{0.093191in}}{\pgfqpoint{-0.073657in}{0.073657in}}%
\pgfpathcurveto{\pgfqpoint{-0.093191in}{0.054123in}}{\pgfqpoint{-0.104167in}{0.027625in}}{\pgfqpoint{-0.104167in}{0.000000in}}%
\pgfpathcurveto{\pgfqpoint{-0.104167in}{-0.027625in}}{\pgfqpoint{-0.093191in}{-0.054123in}}{\pgfqpoint{-0.073657in}{-0.073657in}}%
\pgfpathcurveto{\pgfqpoint{-0.054123in}{-0.093191in}}{\pgfqpoint{-0.027625in}{-0.104167in}}{\pgfqpoint{0.000000in}{-0.104167in}}%
\pgfpathclose%
\pgfusepath{stroke,fill}%
}%
\begin{pgfscope}%
\pgfsys@transformshift{3.331479in}{2.432624in}%
\pgfsys@useobject{currentmarker}{}%
\end{pgfscope}%
\end{pgfscope}%
\begin{pgfscope}%
\pgfpathrectangle{\pgfqpoint{0.000000in}{0.000000in}}{\pgfqpoint{4.000000in}{4.000000in}}%
\pgfusepath{clip}%
\pgfsetbuttcap%
\pgfsetroundjoin%
\definecolor{currentfill}{rgb}{0.800000,0.800000,0.800000}%
\pgfsetfillcolor{currentfill}%
\pgfsetlinewidth{1.003750pt}%
\definecolor{currentstroke}{rgb}{0.000000,0.000000,0.000000}%
\pgfsetstrokecolor{currentstroke}%
\pgfsetdash{}{0pt}%
\pgfsys@defobject{currentmarker}{\pgfqpoint{-0.104167in}{-0.104167in}}{\pgfqpoint{0.104167in}{0.104167in}}{%
\pgfpathmoveto{\pgfqpoint{0.000000in}{-0.104167in}}%
\pgfpathcurveto{\pgfqpoint{0.027625in}{-0.104167in}}{\pgfqpoint{0.054123in}{-0.093191in}}{\pgfqpoint{0.073657in}{-0.073657in}}%
\pgfpathcurveto{\pgfqpoint{0.093191in}{-0.054123in}}{\pgfqpoint{0.104167in}{-0.027625in}}{\pgfqpoint{0.104167in}{0.000000in}}%
\pgfpathcurveto{\pgfqpoint{0.104167in}{0.027625in}}{\pgfqpoint{0.093191in}{0.054123in}}{\pgfqpoint{0.073657in}{0.073657in}}%
\pgfpathcurveto{\pgfqpoint{0.054123in}{0.093191in}}{\pgfqpoint{0.027625in}{0.104167in}}{\pgfqpoint{0.000000in}{0.104167in}}%
\pgfpathcurveto{\pgfqpoint{-0.027625in}{0.104167in}}{\pgfqpoint{-0.054123in}{0.093191in}}{\pgfqpoint{-0.073657in}{0.073657in}}%
\pgfpathcurveto{\pgfqpoint{-0.093191in}{0.054123in}}{\pgfqpoint{-0.104167in}{0.027625in}}{\pgfqpoint{-0.104167in}{0.000000in}}%
\pgfpathcurveto{\pgfqpoint{-0.104167in}{-0.027625in}}{\pgfqpoint{-0.093191in}{-0.054123in}}{\pgfqpoint{-0.073657in}{-0.073657in}}%
\pgfpathcurveto{\pgfqpoint{-0.054123in}{-0.093191in}}{\pgfqpoint{-0.027625in}{-0.104167in}}{\pgfqpoint{0.000000in}{-0.104167in}}%
\pgfpathclose%
\pgfusepath{stroke,fill}%
}%
\begin{pgfscope}%
\pgfsys@transformshift{2.000000in}{3.400000in}%
\pgfsys@useobject{currentmarker}{}%
\end{pgfscope}%
\end{pgfscope}%
\begin{pgfscope}%
\pgfpathrectangle{\pgfqpoint{0.000000in}{0.000000in}}{\pgfqpoint{4.000000in}{4.000000in}}%
\pgfusepath{clip}%
\pgfsetbuttcap%
\pgfsetroundjoin%
\definecolor{currentfill}{rgb}{0.800000,0.800000,0.800000}%
\pgfsetfillcolor{currentfill}%
\pgfsetlinewidth{1.003750pt}%
\definecolor{currentstroke}{rgb}{0.000000,0.000000,0.000000}%
\pgfsetstrokecolor{currentstroke}%
\pgfsetdash{}{0pt}%
\pgfsys@defobject{currentmarker}{\pgfqpoint{-0.104167in}{-0.104167in}}{\pgfqpoint{0.104167in}{0.104167in}}{%
\pgfpathmoveto{\pgfqpoint{0.000000in}{-0.104167in}}%
\pgfpathcurveto{\pgfqpoint{0.027625in}{-0.104167in}}{\pgfqpoint{0.054123in}{-0.093191in}}{\pgfqpoint{0.073657in}{-0.073657in}}%
\pgfpathcurveto{\pgfqpoint{0.093191in}{-0.054123in}}{\pgfqpoint{0.104167in}{-0.027625in}}{\pgfqpoint{0.104167in}{0.000000in}}%
\pgfpathcurveto{\pgfqpoint{0.104167in}{0.027625in}}{\pgfqpoint{0.093191in}{0.054123in}}{\pgfqpoint{0.073657in}{0.073657in}}%
\pgfpathcurveto{\pgfqpoint{0.054123in}{0.093191in}}{\pgfqpoint{0.027625in}{0.104167in}}{\pgfqpoint{0.000000in}{0.104167in}}%
\pgfpathcurveto{\pgfqpoint{-0.027625in}{0.104167in}}{\pgfqpoint{-0.054123in}{0.093191in}}{\pgfqpoint{-0.073657in}{0.073657in}}%
\pgfpathcurveto{\pgfqpoint{-0.093191in}{0.054123in}}{\pgfqpoint{-0.104167in}{0.027625in}}{\pgfqpoint{-0.104167in}{0.000000in}}%
\pgfpathcurveto{\pgfqpoint{-0.104167in}{-0.027625in}}{\pgfqpoint{-0.093191in}{-0.054123in}}{\pgfqpoint{-0.073657in}{-0.073657in}}%
\pgfpathcurveto{\pgfqpoint{-0.054123in}{-0.093191in}}{\pgfqpoint{-0.027625in}{-0.104167in}}{\pgfqpoint{0.000000in}{-0.104167in}}%
\pgfpathclose%
\pgfusepath{stroke,fill}%
}%
\begin{pgfscope}%
\pgfsys@transformshift{0.668521in}{2.432624in}%
\pgfsys@useobject{currentmarker}{}%
\end{pgfscope}%
\end{pgfscope}%
\begin{pgfscope}%
\pgfsetroundcap%
\pgfsetroundjoin%
\pgfsetlinewidth{3.011250pt}%
\definecolor{currentstroke}{rgb}{0.250980,0.250980,0.250980}%
\pgfsetstrokecolor{currentstroke}%
\pgfsetdash{}{0pt}%
\pgfpathmoveto{\pgfqpoint{1.134186in}{0.999454in}}%
\pgfpathlineto{\pgfqpoint{0.723599in}{2.263110in}}%
\pgfusepath{stroke}%
\end{pgfscope}%
\begin{pgfscope}%
\pgfsetroundcap%
\pgfsetroundjoin%
\definecolor{currentfill}{rgb}{0.250980,0.250980,0.250980}%
\pgfsetfillcolor{currentfill}%
\pgfsetlinewidth{3.011250pt}%
\definecolor{currentstroke}{rgb}{0.250980,0.250980,0.250980}%
\pgfsetstrokecolor{currentstroke}%
\pgfsetdash{}{0pt}%
\pgfpathmoveto{\pgfqpoint{0.691889in}{2.135978in}}%
\pgfpathlineto{\pgfqpoint{0.723599in}{2.263110in}}%
\pgfpathlineto{\pgfqpoint{0.823980in}{2.178897in}}%
\pgfpathlineto{\pgfqpoint{0.691889in}{2.135978in}}%
\pgfpathclose%
\pgfusepath{stroke,fill}%
\end{pgfscope}%
\begin{pgfscope}%
\pgfsetroundcap%
\pgfsetroundjoin%
\pgfsetlinewidth{3.011250pt}%
\definecolor{currentstroke}{rgb}{0.250980,0.250980,0.250980}%
\pgfsetstrokecolor{currentstroke}%
\pgfsetdash{}{0pt}%
\pgfpathmoveto{\pgfqpoint{2.684025in}{0.867376in}}%
\pgfpathlineto{\pgfqpoint{1.355338in}{0.867376in}}%
\pgfusepath{stroke}%
\end{pgfscope}%
\begin{pgfscope}%
\pgfsetroundcap%
\pgfsetroundjoin%
\definecolor{currentfill}{rgb}{0.250980,0.250980,0.250980}%
\pgfsetfillcolor{currentfill}%
\pgfsetlinewidth{3.011250pt}%
\definecolor{currentstroke}{rgb}{0.250980,0.250980,0.250980}%
\pgfsetstrokecolor{currentstroke}%
\pgfsetdash{}{0pt}%
\pgfpathmoveto{\pgfqpoint{1.466449in}{0.797932in}}%
\pgfpathlineto{\pgfqpoint{1.355338in}{0.867376in}}%
\pgfpathlineto{\pgfqpoint{1.466449in}{0.936821in}}%
\pgfpathlineto{\pgfqpoint{1.466449in}{0.797932in}}%
\pgfpathclose%
\pgfusepath{stroke,fill}%
\end{pgfscope}%
\begin{pgfscope}%
\pgfsetroundcap%
\pgfsetroundjoin%
\pgfsetlinewidth{3.011250pt}%
\definecolor{currentstroke}{rgb}{0.250980,0.250980,0.250980}%
\pgfsetstrokecolor{currentstroke}%
\pgfsetdash{}{0pt}%
\pgfpathmoveto{\pgfqpoint{2.865814in}{0.999454in}}%
\pgfpathlineto{\pgfqpoint{3.276401in}{2.263110in}}%
\pgfusepath{stroke}%
\end{pgfscope}%
\begin{pgfscope}%
\pgfsetroundcap%
\pgfsetroundjoin%
\definecolor{currentfill}{rgb}{0.250980,0.250980,0.250980}%
\pgfsetfillcolor{currentfill}%
\pgfsetlinewidth{3.011250pt}%
\definecolor{currentstroke}{rgb}{0.250980,0.250980,0.250980}%
\pgfsetstrokecolor{currentstroke}%
\pgfsetdash{}{0pt}%
\pgfpathmoveto{\pgfqpoint{3.176020in}{2.178897in}}%
\pgfpathlineto{\pgfqpoint{3.276401in}{2.263110in}}%
\pgfpathlineto{\pgfqpoint{3.308111in}{2.135978in}}%
\pgfpathlineto{\pgfqpoint{3.176020in}{2.178897in}}%
\pgfpathclose%
\pgfusepath{stroke,fill}%
\end{pgfscope}%
\begin{pgfscope}%
\pgfsetroundcap%
\pgfsetroundjoin%
\pgfsetlinewidth{3.011250pt}%
\definecolor{currentstroke}{rgb}{0.250980,0.250980,0.250980}%
\pgfsetstrokecolor{currentstroke}%
\pgfsetdash{}{0pt}%
\pgfpathmoveto{\pgfqpoint{2.827522in}{1.014954in}}%
\pgfpathlineto{\pgfqpoint{2.102628in}{3.245947in}}%
\pgfusepath{stroke}%
\end{pgfscope}%
\begin{pgfscope}%
\pgfsetroundcap%
\pgfsetroundjoin%
\definecolor{currentfill}{rgb}{0.250980,0.250980,0.250980}%
\pgfsetfillcolor{currentfill}%
\pgfsetlinewidth{3.011250pt}%
\definecolor{currentstroke}{rgb}{0.250980,0.250980,0.250980}%
\pgfsetstrokecolor{currentstroke}%
\pgfsetdash{}{0pt}%
\pgfpathmoveto{\pgfqpoint{2.070918in}{3.118814in}}%
\pgfpathlineto{\pgfqpoint{2.102628in}{3.245947in}}%
\pgfpathlineto{\pgfqpoint{2.203009in}{3.161733in}}%
\pgfpathlineto{\pgfqpoint{2.070918in}{3.118814in}}%
\pgfpathclose%
\pgfusepath{stroke,fill}%
\end{pgfscope}%
\begin{pgfscope}%
\pgfsetroundcap%
\pgfsetroundjoin%
\pgfsetlinewidth{3.011250pt}%
\definecolor{currentstroke}{rgb}{0.250980,0.250980,0.250980}%
\pgfsetstrokecolor{currentstroke}%
\pgfsetdash{}{0pt}%
\pgfpathmoveto{\pgfqpoint{2.710506in}{0.949035in}}%
\pgfpathlineto{\pgfqpoint{0.812710in}{2.327865in}}%
\pgfusepath{stroke}%
\end{pgfscope}%
\begin{pgfscope}%
\pgfsetroundcap%
\pgfsetroundjoin%
\definecolor{currentfill}{rgb}{0.250980,0.250980,0.250980}%
\pgfsetfillcolor{currentfill}%
\pgfsetlinewidth{3.011250pt}%
\definecolor{currentstroke}{rgb}{0.250980,0.250980,0.250980}%
\pgfsetstrokecolor{currentstroke}%
\pgfsetdash{}{0pt}%
\pgfpathmoveto{\pgfqpoint{0.861782in}{2.206373in}}%
\pgfpathlineto{\pgfqpoint{0.812710in}{2.327865in}}%
\pgfpathlineto{\pgfqpoint{0.943419in}{2.318737in}}%
\pgfpathlineto{\pgfqpoint{0.861782in}{2.206373in}}%
\pgfpathclose%
\pgfusepath{stroke,fill}%
\end{pgfscope}%
\begin{pgfscope}%
\pgfsetroundcap%
\pgfsetroundjoin%
\pgfsetlinewidth{3.011250pt}%
\definecolor{currentstroke}{rgb}{0.250980,0.250980,0.250980}%
\pgfsetstrokecolor{currentstroke}%
\pgfsetdash{}{0pt}%
\pgfpathmoveto{\pgfqpoint{3.248517in}{2.554703in}}%
\pgfpathlineto{\pgfqpoint{2.173586in}{3.335686in}}%
\pgfusepath{stroke}%
\end{pgfscope}%
\begin{pgfscope}%
\pgfsetroundcap%
\pgfsetroundjoin%
\definecolor{currentfill}{rgb}{0.250980,0.250980,0.250980}%
\pgfsetfillcolor{currentfill}%
\pgfsetlinewidth{3.011250pt}%
\definecolor{currentstroke}{rgb}{0.250980,0.250980,0.250980}%
\pgfsetstrokecolor{currentstroke}%
\pgfsetdash{}{0pt}%
\pgfpathmoveto{\pgfqpoint{2.222658in}{3.214195in}}%
\pgfpathlineto{\pgfqpoint{2.173586in}{3.335686in}}%
\pgfpathlineto{\pgfqpoint{2.304295in}{3.326558in}}%
\pgfpathlineto{\pgfqpoint{2.222658in}{3.214195in}}%
\pgfpathclose%
\pgfusepath{stroke,fill}%
\end{pgfscope}%
\begin{pgfscope}%
\pgfsetroundcap%
\pgfsetroundjoin%
\pgfsetlinewidth{3.011250pt}%
\definecolor{currentstroke}{rgb}{0.250980,0.250980,0.250980}%
\pgfsetstrokecolor{currentstroke}%
\pgfsetdash{}{0pt}%
\pgfpathmoveto{\pgfqpoint{1.957070in}{3.267874in}}%
\pgfpathlineto{\pgfqpoint{1.232176in}{1.036880in}}%
\pgfusepath{stroke}%
\end{pgfscope}%
\begin{pgfscope}%
\pgfsetroundcap%
\pgfsetroundjoin%
\definecolor{currentfill}{rgb}{0.250980,0.250980,0.250980}%
\pgfsetfillcolor{currentfill}%
\pgfsetlinewidth{3.011250pt}%
\definecolor{currentstroke}{rgb}{0.250980,0.250980,0.250980}%
\pgfsetstrokecolor{currentstroke}%
\pgfsetdash{}{0pt}%
\pgfpathmoveto{\pgfqpoint{1.332557in}{1.121094in}}%
\pgfpathlineto{\pgfqpoint{1.232176in}{1.036880in}}%
\pgfpathlineto{\pgfqpoint{1.200465in}{1.164013in}}%
\pgfpathlineto{\pgfqpoint{1.332557in}{1.121094in}}%
\pgfpathclose%
\pgfusepath{stroke,fill}%
\end{pgfscope}%
\begin{pgfscope}%
\pgfsetroundcap%
\pgfsetroundjoin%
\pgfsetlinewidth{3.011250pt}%
\definecolor{currentstroke}{rgb}{0.250980,0.250980,0.250980}%
\pgfsetstrokecolor{currentstroke}%
\pgfsetdash{}{0pt}%
\pgfpathmoveto{\pgfqpoint{1.995378in}{3.252423in}}%
\pgfpathlineto{\pgfqpoint{2.720271in}{1.021429in}}%
\pgfusepath{stroke}%
\end{pgfscope}%
\begin{pgfscope}%
\pgfsetroundcap%
\pgfsetroundjoin%
\definecolor{currentfill}{rgb}{0.250980,0.250980,0.250980}%
\pgfsetfillcolor{currentfill}%
\pgfsetlinewidth{3.011250pt}%
\definecolor{currentstroke}{rgb}{0.250980,0.250980,0.250980}%
\pgfsetstrokecolor{currentstroke}%
\pgfsetdash{}{0pt}%
\pgfpathmoveto{\pgfqpoint{2.751982in}{1.148562in}}%
\pgfpathlineto{\pgfqpoint{2.720271in}{1.021429in}}%
\pgfpathlineto{\pgfqpoint{2.619891in}{1.105643in}}%
\pgfpathlineto{\pgfqpoint{2.751982in}{1.148562in}}%
\pgfpathclose%
\pgfusepath{stroke,fill}%
\end{pgfscope}%
\begin{pgfscope}%
\pgfsetroundcap%
\pgfsetroundjoin%
\pgfsetlinewidth{3.011250pt}%
\definecolor{currentstroke}{rgb}{0.250980,0.250980,0.250980}%
\pgfsetstrokecolor{currentstroke}%
\pgfsetdash{}{0pt}%
\pgfpathmoveto{\pgfqpoint{2.082962in}{3.277921in}}%
\pgfpathlineto{\pgfqpoint{3.157893in}{2.496938in}}%
\pgfusepath{stroke}%
\end{pgfscope}%
\begin{pgfscope}%
\pgfsetroundcap%
\pgfsetroundjoin%
\definecolor{currentfill}{rgb}{0.250980,0.250980,0.250980}%
\pgfsetfillcolor{currentfill}%
\pgfsetlinewidth{3.011250pt}%
\definecolor{currentstroke}{rgb}{0.250980,0.250980,0.250980}%
\pgfsetstrokecolor{currentstroke}%
\pgfsetdash{}{0pt}%
\pgfpathmoveto{\pgfqpoint{3.108821in}{2.618429in}}%
\pgfpathlineto{\pgfqpoint{3.157893in}{2.496938in}}%
\pgfpathlineto{\pgfqpoint{3.027184in}{2.506066in}}%
\pgfpathlineto{\pgfqpoint{3.108821in}{2.618429in}}%
\pgfpathclose%
\pgfusepath{stroke,fill}%
\end{pgfscope}%
\begin{pgfscope}%
\pgfsetroundcap%
\pgfsetroundjoin%
\pgfsetlinewidth{3.011250pt}%
\definecolor{currentstroke}{rgb}{0.250980,0.250980,0.250980}%
\pgfsetstrokecolor{currentstroke}%
\pgfsetdash{}{0pt}%
\pgfpathmoveto{\pgfqpoint{0.807447in}{2.432624in}}%
\pgfpathlineto{\pgfqpoint{3.153252in}{2.432624in}}%
\pgfusepath{stroke}%
\end{pgfscope}%
\begin{pgfscope}%
\pgfsetroundcap%
\pgfsetroundjoin%
\definecolor{currentfill}{rgb}{0.250980,0.250980,0.250980}%
\pgfsetfillcolor{currentfill}%
\pgfsetlinewidth{3.011250pt}%
\definecolor{currentstroke}{rgb}{0.250980,0.250980,0.250980}%
\pgfsetstrokecolor{currentstroke}%
\pgfsetdash{}{0pt}%
\pgfpathmoveto{\pgfqpoint{3.042141in}{2.502068in}}%
\pgfpathlineto{\pgfqpoint{3.153252in}{2.432624in}}%
\pgfpathlineto{\pgfqpoint{3.042141in}{2.363179in}}%
\pgfpathlineto{\pgfqpoint{3.042141in}{2.502068in}}%
\pgfpathclose%
\pgfusepath{stroke,fill}%
\end{pgfscope}%
\begin{pgfscope}%
\pgfsetroundcap%
\pgfsetroundjoin%
\pgfsetlinewidth{3.011250pt}%
\definecolor{currentstroke}{rgb}{0.250980,0.250980,0.250980}%
\pgfsetstrokecolor{currentstroke}%
\pgfsetdash{}{0pt}%
\pgfpathmoveto{\pgfqpoint{0.780873in}{2.514252in}}%
\pgfpathlineto{\pgfqpoint{1.855803in}{3.295235in}}%
\pgfusepath{stroke}%
\end{pgfscope}%
\begin{pgfscope}%
\pgfsetroundcap%
\pgfsetroundjoin%
\definecolor{currentfill}{rgb}{0.250980,0.250980,0.250980}%
\pgfsetfillcolor{currentfill}%
\pgfsetlinewidth{3.011250pt}%
\definecolor{currentstroke}{rgb}{0.250980,0.250980,0.250980}%
\pgfsetstrokecolor{currentstroke}%
\pgfsetdash{}{0pt}%
\pgfpathmoveto{\pgfqpoint{1.725094in}{3.286107in}}%
\pgfpathlineto{\pgfqpoint{1.855803in}{3.295235in}}%
\pgfpathlineto{\pgfqpoint{1.806731in}{3.173744in}}%
\pgfpathlineto{\pgfqpoint{1.725094in}{3.286107in}}%
\pgfpathclose%
\pgfusepath{stroke,fill}%
\end{pgfscope}%
\begin{pgfscope}%
\definecolor{textcolor}{rgb}{0.000000,0.000000,0.000000}%
\pgfsetstrokecolor{textcolor}%
\pgfsetfillcolor{textcolor}%
\pgftext[x=0.941987in,y=0.543769in,,]{\color{textcolor}\sffamily\fontsize{40.000000}{48.000000}\selectfont \(\displaystyle 1\)}%
\end{pgfscope}%
\begin{pgfscope}%
\definecolor{textcolor}{rgb}{0.000000,0.000000,0.000000}%
\pgfsetstrokecolor{textcolor}%
\pgfsetfillcolor{textcolor}%
\pgftext[x=3.058013in,y=0.543769in,,]{\color{textcolor}\sffamily\fontsize{40.000000}{48.000000}\selectfont \(\displaystyle 2\)}%
\end{pgfscope}%
\begin{pgfscope}%
\definecolor{textcolor}{rgb}{0.000000,0.000000,0.000000}%
\pgfsetstrokecolor{textcolor}%
\pgfsetfillcolor{textcolor}%
\pgftext[x=3.711902in,y=2.556231in,,]{\color{textcolor}\sffamily\fontsize{40.000000}{48.000000}\selectfont \(\displaystyle 3\)}%
\end{pgfscope}%
\begin{pgfscope}%
\definecolor{textcolor}{rgb}{0.000000,0.000000,0.000000}%
\pgfsetstrokecolor{textcolor}%
\pgfsetfillcolor{textcolor}%
\pgftext[x=2.000000in,y=3.800000in,,]{\color{textcolor}\sffamily\fontsize{40.000000}{48.000000}\selectfont \(\displaystyle 4\)}%
\end{pgfscope}%
\begin{pgfscope}%
\definecolor{textcolor}{rgb}{0.000000,0.000000,0.000000}%
\pgfsetstrokecolor{textcolor}%
\pgfsetfillcolor{textcolor}%
\pgftext[x=0.288098in,y=2.556231in,,]{\color{textcolor}\sffamily\fontsize{40.000000}{48.000000}\selectfont \(\displaystyle 5\)}%
\end{pgfscope}%
\end{pgfpicture}%
\makeatother%
\endgroup%

%% file: fig/example_equivalent_graphs_class.pgf
%% Creator: Matplotlib, PGF backend
%%
%% To include the figure in your LaTeX document, write
%%   \input{<filename>.pgf}
%%
%% Make sure the required packages are loaded in your preamble
%%   \usepackage{pgf}
%%
%% Figures using additional raster images can only be included by \input if
%% they are in the same directory as the main LaTeX file. For loading figures
%% from other directories you can use the `import` package
%%   \usepackage{import}
%%
%% and then include the figures with
%%   \import{<path to file>}{<filename>.pgf}
%%
%% Matplotlib used the following preamble
%%   \usepackage{fontspec}
%%   \setmainfont{DejaVuSerif.ttf}[Path=\detokenize{C:/Users/ccros/Anaconda3/Lib/site-packages/matplotlib/mpl-data/fonts/ttf/}]
%%   \setsansfont{DejaVuSans.ttf}[Path=\detokenize{C:/Users/ccros/Anaconda3/Lib/site-packages/matplotlib/mpl-data/fonts/ttf/}]
%%   \setmonofont{DejaVuSansMono.ttf}[Path=\detokenize{C:/Users/ccros/Anaconda3/Lib/site-packages/matplotlib/mpl-data/fonts/ttf/}]
%%
\begingroup%
\makeatletter%
\begin{pgfpicture}%
\pgfpathrectangle{\pgfpointorigin}{\pgfqpoint{4.000000in}{4.068826in}}%
\pgfusepath{use as bounding box, clip}%
\begin{pgfscope}%
\pgfsetbuttcap%
\pgfsetmiterjoin%
\pgfsetlinewidth{0.000000pt}%
\definecolor{currentstroke}{rgb}{1.000000,1.000000,1.000000}%
\pgfsetstrokecolor{currentstroke}%
\pgfsetstrokeopacity{0.000000}%
\pgfsetdash{}{0pt}%
\pgfpathmoveto{\pgfqpoint{0.000000in}{0.000000in}}%
\pgfpathlineto{\pgfqpoint{4.000000in}{0.000000in}}%
\pgfpathlineto{\pgfqpoint{4.000000in}{4.068826in}}%
\pgfpathlineto{\pgfqpoint{0.000000in}{4.068826in}}%
\pgfpathclose%
\pgfusepath{}%
\end{pgfscope}%
\begin{pgfscope}%
\pgfpathrectangle{\pgfqpoint{0.000000in}{0.000000in}}{\pgfqpoint{4.000000in}{4.000000in}}%
\pgfusepath{clip}%
\pgfsetrectcap%
\pgfsetroundjoin%
\pgfsetlinewidth{3.011250pt}%
\definecolor{currentstroke}{rgb}{0.250980,0.250980,0.250980}%
\pgfsetstrokecolor{currentstroke}%
\pgfsetdash{}{0pt}%
\pgfpathmoveto{\pgfqpoint{2.822899in}{0.867376in}}%
\pgfpathlineto{\pgfqpoint{1.177101in}{0.867376in}}%
\pgfusepath{stroke}%
\end{pgfscope}%
\begin{pgfscope}%
\pgfpathrectangle{\pgfqpoint{0.000000in}{0.000000in}}{\pgfqpoint{4.000000in}{4.000000in}}%
\pgfusepath{clip}%
\pgfsetrectcap%
\pgfsetroundjoin%
\pgfsetlinewidth{3.011250pt}%
\definecolor{currentstroke}{rgb}{0.250980,0.250980,0.250980}%
\pgfsetstrokecolor{currentstroke}%
\pgfsetdash{}{0pt}%
\pgfpathmoveto{\pgfqpoint{3.331479in}{2.432624in}}%
\pgfpathlineto{\pgfqpoint{2.822899in}{0.867376in}}%
\pgfusepath{stroke}%
\end{pgfscope}%
\begin{pgfscope}%
\pgfpathrectangle{\pgfqpoint{0.000000in}{0.000000in}}{\pgfqpoint{4.000000in}{4.000000in}}%
\pgfusepath{clip}%
\pgfsetrectcap%
\pgfsetroundjoin%
\pgfsetlinewidth{3.011250pt}%
\definecolor{currentstroke}{rgb}{0.250980,0.250980,0.250980}%
\pgfsetstrokecolor{currentstroke}%
\pgfsetdash{}{0pt}%
\pgfpathmoveto{\pgfqpoint{2.000000in}{3.400000in}}%
\pgfpathlineto{\pgfqpoint{1.177101in}{0.867376in}}%
\pgfusepath{stroke}%
\end{pgfscope}%
\begin{pgfscope}%
\pgfpathrectangle{\pgfqpoint{0.000000in}{0.000000in}}{\pgfqpoint{4.000000in}{4.000000in}}%
\pgfusepath{clip}%
\pgfsetrectcap%
\pgfsetroundjoin%
\pgfsetlinewidth{3.011250pt}%
\definecolor{currentstroke}{rgb}{0.250980,0.250980,0.250980}%
\pgfsetstrokecolor{currentstroke}%
\pgfsetdash{}{0pt}%
\pgfpathmoveto{\pgfqpoint{1.952447in}{3.384549in}}%
\pgfpathlineto{\pgfqpoint{2.775347in}{0.851925in}}%
\pgfusepath{stroke}%
\end{pgfscope}%
\begin{pgfscope}%
\pgfpathrectangle{\pgfqpoint{0.000000in}{0.000000in}}{\pgfqpoint{4.000000in}{4.000000in}}%
\pgfusepath{clip}%
\pgfsetrectcap%
\pgfsetroundjoin%
\pgfsetlinewidth{3.011250pt}%
\definecolor{currentstroke}{rgb}{0.250980,0.250980,0.250980}%
\pgfsetstrokecolor{currentstroke}%
\pgfsetdash{}{0pt}%
\pgfpathmoveto{\pgfqpoint{2.047553in}{3.415451in}}%
\pgfpathlineto{\pgfqpoint{2.870452in}{0.882827in}}%
\pgfusepath{stroke}%
\end{pgfscope}%
\begin{pgfscope}%
\pgfpathrectangle{\pgfqpoint{0.000000in}{0.000000in}}{\pgfqpoint{4.000000in}{4.000000in}}%
\pgfusepath{clip}%
\pgfsetrectcap%
\pgfsetroundjoin%
\pgfsetlinewidth{3.011250pt}%
\definecolor{currentstroke}{rgb}{0.250980,0.250980,0.250980}%
\pgfsetstrokecolor{currentstroke}%
\pgfsetdash{}{0pt}%
\pgfpathmoveto{\pgfqpoint{1.970611in}{3.359549in}}%
\pgfpathlineto{\pgfqpoint{3.302090in}{2.392173in}}%
\pgfusepath{stroke}%
\end{pgfscope}%
\begin{pgfscope}%
\pgfpathrectangle{\pgfqpoint{0.000000in}{0.000000in}}{\pgfqpoint{4.000000in}{4.000000in}}%
\pgfusepath{clip}%
\pgfsetrectcap%
\pgfsetroundjoin%
\pgfsetlinewidth{3.011250pt}%
\definecolor{currentstroke}{rgb}{0.250980,0.250980,0.250980}%
\pgfsetstrokecolor{currentstroke}%
\pgfsetdash{}{0pt}%
\pgfpathmoveto{\pgfqpoint{2.029389in}{3.440451in}}%
\pgfpathlineto{\pgfqpoint{3.360868in}{2.473075in}}%
\pgfusepath{stroke}%
\end{pgfscope}%
\begin{pgfscope}%
\pgfpathrectangle{\pgfqpoint{0.000000in}{0.000000in}}{\pgfqpoint{4.000000in}{4.000000in}}%
\pgfusepath{clip}%
\pgfsetrectcap%
\pgfsetroundjoin%
\pgfsetlinewidth{3.011250pt}%
\definecolor{currentstroke}{rgb}{0.250980,0.250980,0.250980}%
\pgfsetstrokecolor{currentstroke}%
\pgfsetdash{}{0pt}%
\pgfpathmoveto{\pgfqpoint{0.668521in}{2.432624in}}%
\pgfpathlineto{\pgfqpoint{1.177101in}{0.867376in}}%
\pgfusepath{stroke}%
\end{pgfscope}%
\begin{pgfscope}%
\pgfpathrectangle{\pgfqpoint{0.000000in}{0.000000in}}{\pgfqpoint{4.000000in}{4.000000in}}%
\pgfusepath{clip}%
\pgfsetrectcap%
\pgfsetroundjoin%
\pgfsetlinewidth{3.011250pt}%
\definecolor{currentstroke}{rgb}{0.250980,0.250980,0.250980}%
\pgfsetstrokecolor{currentstroke}%
\pgfsetdash{}{0pt}%
\pgfpathmoveto{\pgfqpoint{0.668521in}{2.432624in}}%
\pgfpathlineto{\pgfqpoint{2.822899in}{0.867376in}}%
\pgfusepath{stroke}%
\end{pgfscope}%
\begin{pgfscope}%
\pgfpathrectangle{\pgfqpoint{0.000000in}{0.000000in}}{\pgfqpoint{4.000000in}{4.000000in}}%
\pgfusepath{clip}%
\pgfsetrectcap%
\pgfsetroundjoin%
\pgfsetlinewidth{3.011250pt}%
\definecolor{currentstroke}{rgb}{0.250980,0.250980,0.250980}%
\pgfsetstrokecolor{currentstroke}%
\pgfsetdash{}{0pt}%
\pgfpathmoveto{\pgfqpoint{0.668521in}{2.432624in}}%
\pgfpathlineto{\pgfqpoint{3.331479in}{2.432624in}}%
\pgfusepath{stroke}%
\end{pgfscope}%
\begin{pgfscope}%
\pgfpathrectangle{\pgfqpoint{0.000000in}{0.000000in}}{\pgfqpoint{4.000000in}{4.000000in}}%
\pgfusepath{clip}%
\pgfsetrectcap%
\pgfsetroundjoin%
\pgfsetlinewidth{3.011250pt}%
\definecolor{currentstroke}{rgb}{0.250980,0.250980,0.250980}%
\pgfsetstrokecolor{currentstroke}%
\pgfsetdash{}{0pt}%
\pgfpathmoveto{\pgfqpoint{0.668521in}{2.432624in}}%
\pgfpathlineto{\pgfqpoint{2.000000in}{3.400000in}}%
\pgfusepath{stroke}%
\end{pgfscope}%
\begin{pgfscope}%
\pgfpathrectangle{\pgfqpoint{0.000000in}{0.000000in}}{\pgfqpoint{4.000000in}{4.000000in}}%
\pgfusepath{clip}%
\pgfsetbuttcap%
\pgfsetroundjoin%
\definecolor{currentfill}{rgb}{0.800000,0.800000,0.800000}%
\pgfsetfillcolor{currentfill}%
\pgfsetlinewidth{1.003750pt}%
\definecolor{currentstroke}{rgb}{0.000000,0.000000,0.000000}%
\pgfsetstrokecolor{currentstroke}%
\pgfsetdash{}{0pt}%
\pgfsys@defobject{currentmarker}{\pgfqpoint{-0.104167in}{-0.104167in}}{\pgfqpoint{0.104167in}{0.104167in}}{%
\pgfpathmoveto{\pgfqpoint{0.000000in}{-0.104167in}}%
\pgfpathcurveto{\pgfqpoint{0.027625in}{-0.104167in}}{\pgfqpoint{0.054123in}{-0.093191in}}{\pgfqpoint{0.073657in}{-0.073657in}}%
\pgfpathcurveto{\pgfqpoint{0.093191in}{-0.054123in}}{\pgfqpoint{0.104167in}{-0.027625in}}{\pgfqpoint{0.104167in}{0.000000in}}%
\pgfpathcurveto{\pgfqpoint{0.104167in}{0.027625in}}{\pgfqpoint{0.093191in}{0.054123in}}{\pgfqpoint{0.073657in}{0.073657in}}%
\pgfpathcurveto{\pgfqpoint{0.054123in}{0.093191in}}{\pgfqpoint{0.027625in}{0.104167in}}{\pgfqpoint{0.000000in}{0.104167in}}%
\pgfpathcurveto{\pgfqpoint{-0.027625in}{0.104167in}}{\pgfqpoint{-0.054123in}{0.093191in}}{\pgfqpoint{-0.073657in}{0.073657in}}%
\pgfpathcurveto{\pgfqpoint{-0.093191in}{0.054123in}}{\pgfqpoint{-0.104167in}{0.027625in}}{\pgfqpoint{-0.104167in}{0.000000in}}%
\pgfpathcurveto{\pgfqpoint{-0.104167in}{-0.027625in}}{\pgfqpoint{-0.093191in}{-0.054123in}}{\pgfqpoint{-0.073657in}{-0.073657in}}%
\pgfpathcurveto{\pgfqpoint{-0.054123in}{-0.093191in}}{\pgfqpoint{-0.027625in}{-0.104167in}}{\pgfqpoint{0.000000in}{-0.104167in}}%
\pgfpathclose%
\pgfusepath{stroke,fill}%
}%
\begin{pgfscope}%
\pgfsys@transformshift{1.177101in}{0.867376in}%
\pgfsys@useobject{currentmarker}{}%
\end{pgfscope}%
\end{pgfscope}%
\begin{pgfscope}%
\pgfpathrectangle{\pgfqpoint{0.000000in}{0.000000in}}{\pgfqpoint{4.000000in}{4.000000in}}%
\pgfusepath{clip}%
\pgfsetbuttcap%
\pgfsetroundjoin%
\definecolor{currentfill}{rgb}{0.800000,0.800000,0.800000}%
\pgfsetfillcolor{currentfill}%
\pgfsetlinewidth{1.003750pt}%
\definecolor{currentstroke}{rgb}{0.000000,0.000000,0.000000}%
\pgfsetstrokecolor{currentstroke}%
\pgfsetdash{}{0pt}%
\pgfsys@defobject{currentmarker}{\pgfqpoint{-0.104167in}{-0.104167in}}{\pgfqpoint{0.104167in}{0.104167in}}{%
\pgfpathmoveto{\pgfqpoint{0.000000in}{-0.104167in}}%
\pgfpathcurveto{\pgfqpoint{0.027625in}{-0.104167in}}{\pgfqpoint{0.054123in}{-0.093191in}}{\pgfqpoint{0.073657in}{-0.073657in}}%
\pgfpathcurveto{\pgfqpoint{0.093191in}{-0.054123in}}{\pgfqpoint{0.104167in}{-0.027625in}}{\pgfqpoint{0.104167in}{0.000000in}}%
\pgfpathcurveto{\pgfqpoint{0.104167in}{0.027625in}}{\pgfqpoint{0.093191in}{0.054123in}}{\pgfqpoint{0.073657in}{0.073657in}}%
\pgfpathcurveto{\pgfqpoint{0.054123in}{0.093191in}}{\pgfqpoint{0.027625in}{0.104167in}}{\pgfqpoint{0.000000in}{0.104167in}}%
\pgfpathcurveto{\pgfqpoint{-0.027625in}{0.104167in}}{\pgfqpoint{-0.054123in}{0.093191in}}{\pgfqpoint{-0.073657in}{0.073657in}}%
\pgfpathcurveto{\pgfqpoint{-0.093191in}{0.054123in}}{\pgfqpoint{-0.104167in}{0.027625in}}{\pgfqpoint{-0.104167in}{0.000000in}}%
\pgfpathcurveto{\pgfqpoint{-0.104167in}{-0.027625in}}{\pgfqpoint{-0.093191in}{-0.054123in}}{\pgfqpoint{-0.073657in}{-0.073657in}}%
\pgfpathcurveto{\pgfqpoint{-0.054123in}{-0.093191in}}{\pgfqpoint{-0.027625in}{-0.104167in}}{\pgfqpoint{0.000000in}{-0.104167in}}%
\pgfpathclose%
\pgfusepath{stroke,fill}%
}%
\begin{pgfscope}%
\pgfsys@transformshift{2.822899in}{0.867376in}%
\pgfsys@useobject{currentmarker}{}%
\end{pgfscope}%
\end{pgfscope}%
\begin{pgfscope}%
\pgfpathrectangle{\pgfqpoint{0.000000in}{0.000000in}}{\pgfqpoint{4.000000in}{4.000000in}}%
\pgfusepath{clip}%
\pgfsetbuttcap%
\pgfsetroundjoin%
\definecolor{currentfill}{rgb}{0.800000,0.800000,0.800000}%
\pgfsetfillcolor{currentfill}%
\pgfsetlinewidth{1.003750pt}%
\definecolor{currentstroke}{rgb}{0.000000,0.000000,0.000000}%
\pgfsetstrokecolor{currentstroke}%
\pgfsetdash{}{0pt}%
\pgfsys@defobject{currentmarker}{\pgfqpoint{-0.104167in}{-0.104167in}}{\pgfqpoint{0.104167in}{0.104167in}}{%
\pgfpathmoveto{\pgfqpoint{0.000000in}{-0.104167in}}%
\pgfpathcurveto{\pgfqpoint{0.027625in}{-0.104167in}}{\pgfqpoint{0.054123in}{-0.093191in}}{\pgfqpoint{0.073657in}{-0.073657in}}%
\pgfpathcurveto{\pgfqpoint{0.093191in}{-0.054123in}}{\pgfqpoint{0.104167in}{-0.027625in}}{\pgfqpoint{0.104167in}{0.000000in}}%
\pgfpathcurveto{\pgfqpoint{0.104167in}{0.027625in}}{\pgfqpoint{0.093191in}{0.054123in}}{\pgfqpoint{0.073657in}{0.073657in}}%
\pgfpathcurveto{\pgfqpoint{0.054123in}{0.093191in}}{\pgfqpoint{0.027625in}{0.104167in}}{\pgfqpoint{0.000000in}{0.104167in}}%
\pgfpathcurveto{\pgfqpoint{-0.027625in}{0.104167in}}{\pgfqpoint{-0.054123in}{0.093191in}}{\pgfqpoint{-0.073657in}{0.073657in}}%
\pgfpathcurveto{\pgfqpoint{-0.093191in}{0.054123in}}{\pgfqpoint{-0.104167in}{0.027625in}}{\pgfqpoint{-0.104167in}{0.000000in}}%
\pgfpathcurveto{\pgfqpoint{-0.104167in}{-0.027625in}}{\pgfqpoint{-0.093191in}{-0.054123in}}{\pgfqpoint{-0.073657in}{-0.073657in}}%
\pgfpathcurveto{\pgfqpoint{-0.054123in}{-0.093191in}}{\pgfqpoint{-0.027625in}{-0.104167in}}{\pgfqpoint{0.000000in}{-0.104167in}}%
\pgfpathclose%
\pgfusepath{stroke,fill}%
}%
\begin{pgfscope}%
\pgfsys@transformshift{3.331479in}{2.432624in}%
\pgfsys@useobject{currentmarker}{}%
\end{pgfscope}%
\end{pgfscope}%
\begin{pgfscope}%
\pgfpathrectangle{\pgfqpoint{0.000000in}{0.000000in}}{\pgfqpoint{4.000000in}{4.000000in}}%
\pgfusepath{clip}%
\pgfsetbuttcap%
\pgfsetroundjoin%
\definecolor{currentfill}{rgb}{0.800000,0.800000,0.800000}%
\pgfsetfillcolor{currentfill}%
\pgfsetlinewidth{1.003750pt}%
\definecolor{currentstroke}{rgb}{0.000000,0.000000,0.000000}%
\pgfsetstrokecolor{currentstroke}%
\pgfsetdash{}{0pt}%
\pgfsys@defobject{currentmarker}{\pgfqpoint{-0.104167in}{-0.104167in}}{\pgfqpoint{0.104167in}{0.104167in}}{%
\pgfpathmoveto{\pgfqpoint{0.000000in}{-0.104167in}}%
\pgfpathcurveto{\pgfqpoint{0.027625in}{-0.104167in}}{\pgfqpoint{0.054123in}{-0.093191in}}{\pgfqpoint{0.073657in}{-0.073657in}}%
\pgfpathcurveto{\pgfqpoint{0.093191in}{-0.054123in}}{\pgfqpoint{0.104167in}{-0.027625in}}{\pgfqpoint{0.104167in}{0.000000in}}%
\pgfpathcurveto{\pgfqpoint{0.104167in}{0.027625in}}{\pgfqpoint{0.093191in}{0.054123in}}{\pgfqpoint{0.073657in}{0.073657in}}%
\pgfpathcurveto{\pgfqpoint{0.054123in}{0.093191in}}{\pgfqpoint{0.027625in}{0.104167in}}{\pgfqpoint{0.000000in}{0.104167in}}%
\pgfpathcurveto{\pgfqpoint{-0.027625in}{0.104167in}}{\pgfqpoint{-0.054123in}{0.093191in}}{\pgfqpoint{-0.073657in}{0.073657in}}%
\pgfpathcurveto{\pgfqpoint{-0.093191in}{0.054123in}}{\pgfqpoint{-0.104167in}{0.027625in}}{\pgfqpoint{-0.104167in}{0.000000in}}%
\pgfpathcurveto{\pgfqpoint{-0.104167in}{-0.027625in}}{\pgfqpoint{-0.093191in}{-0.054123in}}{\pgfqpoint{-0.073657in}{-0.073657in}}%
\pgfpathcurveto{\pgfqpoint{-0.054123in}{-0.093191in}}{\pgfqpoint{-0.027625in}{-0.104167in}}{\pgfqpoint{0.000000in}{-0.104167in}}%
\pgfpathclose%
\pgfusepath{stroke,fill}%
}%
\begin{pgfscope}%
\pgfsys@transformshift{2.000000in}{3.400000in}%
\pgfsys@useobject{currentmarker}{}%
\end{pgfscope}%
\end{pgfscope}%
\begin{pgfscope}%
\pgfpathrectangle{\pgfqpoint{0.000000in}{0.000000in}}{\pgfqpoint{4.000000in}{4.000000in}}%
\pgfusepath{clip}%
\pgfsetbuttcap%
\pgfsetroundjoin%
\definecolor{currentfill}{rgb}{0.800000,0.800000,0.800000}%
\pgfsetfillcolor{currentfill}%
\pgfsetlinewidth{1.003750pt}%
\definecolor{currentstroke}{rgb}{0.000000,0.000000,0.000000}%
\pgfsetstrokecolor{currentstroke}%
\pgfsetdash{}{0pt}%
\pgfsys@defobject{currentmarker}{\pgfqpoint{-0.104167in}{-0.104167in}}{\pgfqpoint{0.104167in}{0.104167in}}{%
\pgfpathmoveto{\pgfqpoint{0.000000in}{-0.104167in}}%
\pgfpathcurveto{\pgfqpoint{0.027625in}{-0.104167in}}{\pgfqpoint{0.054123in}{-0.093191in}}{\pgfqpoint{0.073657in}{-0.073657in}}%
\pgfpathcurveto{\pgfqpoint{0.093191in}{-0.054123in}}{\pgfqpoint{0.104167in}{-0.027625in}}{\pgfqpoint{0.104167in}{0.000000in}}%
\pgfpathcurveto{\pgfqpoint{0.104167in}{0.027625in}}{\pgfqpoint{0.093191in}{0.054123in}}{\pgfqpoint{0.073657in}{0.073657in}}%
\pgfpathcurveto{\pgfqpoint{0.054123in}{0.093191in}}{\pgfqpoint{0.027625in}{0.104167in}}{\pgfqpoint{0.000000in}{0.104167in}}%
\pgfpathcurveto{\pgfqpoint{-0.027625in}{0.104167in}}{\pgfqpoint{-0.054123in}{0.093191in}}{\pgfqpoint{-0.073657in}{0.073657in}}%
\pgfpathcurveto{\pgfqpoint{-0.093191in}{0.054123in}}{\pgfqpoint{-0.104167in}{0.027625in}}{\pgfqpoint{-0.104167in}{0.000000in}}%
\pgfpathcurveto{\pgfqpoint{-0.104167in}{-0.027625in}}{\pgfqpoint{-0.093191in}{-0.054123in}}{\pgfqpoint{-0.073657in}{-0.073657in}}%
\pgfpathcurveto{\pgfqpoint{-0.054123in}{-0.093191in}}{\pgfqpoint{-0.027625in}{-0.104167in}}{\pgfqpoint{0.000000in}{-0.104167in}}%
\pgfpathclose%
\pgfusepath{stroke,fill}%
}%
\begin{pgfscope}%
\pgfsys@transformshift{0.668521in}{2.432624in}%
\pgfsys@useobject{currentmarker}{}%
\end{pgfscope}%
\end{pgfscope}%
\begin{pgfscope}%
\definecolor{textcolor}{rgb}{0.000000,0.000000,0.000000}%
\pgfsetstrokecolor{textcolor}%
\pgfsetfillcolor{textcolor}%
\pgftext[x=0.941987in,y=0.543769in,,]{\color{textcolor}\sffamily\fontsize{40.000000}{48.000000}\selectfont \(\displaystyle 1\)}%
\end{pgfscope}%
\begin{pgfscope}%
\definecolor{textcolor}{rgb}{0.000000,0.000000,0.000000}%
\pgfsetstrokecolor{textcolor}%
\pgfsetfillcolor{textcolor}%
\pgftext[x=3.058013in,y=0.543769in,,]{\color{textcolor}\sffamily\fontsize{40.000000}{48.000000}\selectfont \(\displaystyle 2\)}%
\end{pgfscope}%
\begin{pgfscope}%
\definecolor{textcolor}{rgb}{0.000000,0.000000,0.000000}%
\pgfsetstrokecolor{textcolor}%
\pgfsetfillcolor{textcolor}%
\pgftext[x=3.711902in,y=2.556231in,,]{\color{textcolor}\sffamily\fontsize{40.000000}{48.000000}\selectfont \(\displaystyle 3\)}%
\end{pgfscope}%
\begin{pgfscope}%
\definecolor{textcolor}{rgb}{0.000000,0.000000,0.000000}%
\pgfsetstrokecolor{textcolor}%
\pgfsetfillcolor{textcolor}%
\pgftext[x=2.000000in,y=3.800000in,,]{\color{textcolor}\sffamily\fontsize{40.000000}{48.000000}\selectfont \(\displaystyle 4\)}%
\end{pgfscope}%
\begin{pgfscope}%
\definecolor{textcolor}{rgb}{0.000000,0.000000,0.000000}%
\pgfsetstrokecolor{textcolor}%
\pgfsetfillcolor{textcolor}%
\pgftext[x=0.288098in,y=2.556231in,,]{\color{textcolor}\sffamily\fontsize{40.000000}{48.000000}\selectfont \(\displaystyle 5\)}%
\end{pgfscope}%
\end{pgfpicture}%
\makeatother%
\endgroup%

%% file: fig/decomposition_conic_original_rigid.pgf
%% Creator: Matplotlib, PGF backend
%%
%% To include the figure in your LaTeX document, write
%%   \input{<filename>.pgf}
%%
%% Make sure the required packages are loaded in your preamble
%%   \usepackage{pgf}
%%
%% Figures using additional raster images can only be included by \input if
%% they are in the same directory as the main LaTeX file. For loading figures
%% from other directories you can use the `import` package
%%   \usepackage{import}
%%
%% and then include the figures with
%%   \import{<path to file>}{<filename>.pgf}
%%
%% Matplotlib used the following preamble
%%   \usepackage{fontspec}
%%   \setmainfont{DejaVuSerif.ttf}[Path=\detokenize{C:/Users/ccros/Anaconda3/Lib/site-packages/matplotlib/mpl-data/fonts/ttf/}]
%%   \setsansfont{DejaVuSans.ttf}[Path=\detokenize{C:/Users/ccros/Anaconda3/Lib/site-packages/matplotlib/mpl-data/fonts/ttf/}]
%%   \setmonofont{DejaVuSansMono.ttf}[Path=\detokenize{C:/Users/ccros/Anaconda3/Lib/site-packages/matplotlib/mpl-data/fonts/ttf/}]
%%
\begingroup%
\makeatletter%
\begin{pgfpicture}%
\pgfpathrectangle{\pgfpointorigin}{\pgfqpoint{4.000000in}{4.068826in}}%
\pgfusepath{use as bounding box, clip}%
\begin{pgfscope}%
\pgfsetbuttcap%
\pgfsetmiterjoin%
\pgfsetlinewidth{0.000000pt}%
\definecolor{currentstroke}{rgb}{1.000000,1.000000,1.000000}%
\pgfsetstrokecolor{currentstroke}%
\pgfsetstrokeopacity{0.000000}%
\pgfsetdash{}{0pt}%
\pgfpathmoveto{\pgfqpoint{0.000000in}{0.000000in}}%
\pgfpathlineto{\pgfqpoint{4.000000in}{0.000000in}}%
\pgfpathlineto{\pgfqpoint{4.000000in}{4.068826in}}%
\pgfpathlineto{\pgfqpoint{0.000000in}{4.068826in}}%
\pgfpathclose%
\pgfusepath{}%
\end{pgfscope}%
\begin{pgfscope}%
\pgfpathrectangle{\pgfqpoint{0.000000in}{0.000000in}}{\pgfqpoint{4.000000in}{4.000000in}}%
\pgfusepath{clip}%
\pgfsetrectcap%
\pgfsetroundjoin%
\pgfsetlinewidth{3.011250pt}%
\definecolor{currentstroke}{rgb}{0.250980,0.250980,0.250980}%
\pgfsetstrokecolor{currentstroke}%
\pgfsetdash{}{0pt}%
\pgfpathmoveto{\pgfqpoint{2.822899in}{0.867376in}}%
\pgfpathlineto{\pgfqpoint{1.177101in}{0.867376in}}%
\pgfusepath{stroke}%
\end{pgfscope}%
\begin{pgfscope}%
\pgfpathrectangle{\pgfqpoint{0.000000in}{0.000000in}}{\pgfqpoint{4.000000in}{4.000000in}}%
\pgfusepath{clip}%
\pgfsetrectcap%
\pgfsetroundjoin%
\pgfsetlinewidth{3.011250pt}%
\definecolor{currentstroke}{rgb}{0.250980,0.250980,0.250980}%
\pgfsetstrokecolor{currentstroke}%
\pgfsetdash{}{0pt}%
\pgfpathmoveto{\pgfqpoint{3.331479in}{2.432624in}}%
\pgfpathlineto{\pgfqpoint{2.822899in}{0.867376in}}%
\pgfusepath{stroke}%
\end{pgfscope}%
\begin{pgfscope}%
\pgfpathrectangle{\pgfqpoint{0.000000in}{0.000000in}}{\pgfqpoint{4.000000in}{4.000000in}}%
\pgfusepath{clip}%
\pgfsetrectcap%
\pgfsetroundjoin%
\pgfsetlinewidth{3.011250pt}%
\definecolor{currentstroke}{rgb}{0.250980,0.250980,0.250980}%
\pgfsetstrokecolor{currentstroke}%
\pgfsetdash{}{0pt}%
\pgfpathmoveto{\pgfqpoint{2.000000in}{3.400000in}}%
\pgfpathlineto{\pgfqpoint{2.822899in}{0.867376in}}%
\pgfusepath{stroke}%
\end{pgfscope}%
\begin{pgfscope}%
\pgfpathrectangle{\pgfqpoint{0.000000in}{0.000000in}}{\pgfqpoint{4.000000in}{4.000000in}}%
\pgfusepath{clip}%
\pgfsetrectcap%
\pgfsetroundjoin%
\pgfsetlinewidth{3.011250pt}%
\definecolor{currentstroke}{rgb}{0.250980,0.250980,0.250980}%
\pgfsetstrokecolor{currentstroke}%
\pgfsetdash{}{0pt}%
\pgfpathmoveto{\pgfqpoint{2.000000in}{3.400000in}}%
\pgfpathlineto{\pgfqpoint{3.331479in}{2.432624in}}%
\pgfusepath{stroke}%
\end{pgfscope}%
\begin{pgfscope}%
\pgfpathrectangle{\pgfqpoint{0.000000in}{0.000000in}}{\pgfqpoint{4.000000in}{4.000000in}}%
\pgfusepath{clip}%
\pgfsetrectcap%
\pgfsetroundjoin%
\pgfsetlinewidth{3.011250pt}%
\definecolor{currentstroke}{rgb}{0.250980,0.250980,0.250980}%
\pgfsetstrokecolor{currentstroke}%
\pgfsetdash{}{0pt}%
\pgfpathmoveto{\pgfqpoint{0.668521in}{2.432624in}}%
\pgfpathlineto{\pgfqpoint{1.177101in}{0.867376in}}%
\pgfusepath{stroke}%
\end{pgfscope}%
\begin{pgfscope}%
\pgfpathrectangle{\pgfqpoint{0.000000in}{0.000000in}}{\pgfqpoint{4.000000in}{4.000000in}}%
\pgfusepath{clip}%
\pgfsetrectcap%
\pgfsetroundjoin%
\pgfsetlinewidth{3.011250pt}%
\definecolor{currentstroke}{rgb}{0.250980,0.250980,0.250980}%
\pgfsetstrokecolor{currentstroke}%
\pgfsetdash{}{0pt}%
\pgfpathmoveto{\pgfqpoint{0.668521in}{2.432624in}}%
\pgfpathlineto{\pgfqpoint{2.822899in}{0.867376in}}%
\pgfusepath{stroke}%
\end{pgfscope}%
\begin{pgfscope}%
\pgfpathrectangle{\pgfqpoint{0.000000in}{0.000000in}}{\pgfqpoint{4.000000in}{4.000000in}}%
\pgfusepath{clip}%
\pgfsetrectcap%
\pgfsetroundjoin%
\pgfsetlinewidth{3.011250pt}%
\definecolor{currentstroke}{rgb}{0.250980,0.250980,0.250980}%
\pgfsetstrokecolor{currentstroke}%
\pgfsetdash{}{0pt}%
\pgfpathmoveto{\pgfqpoint{0.668521in}{2.432624in}}%
\pgfpathlineto{\pgfqpoint{2.000000in}{3.400000in}}%
\pgfusepath{stroke}%
\end{pgfscope}%
\begin{pgfscope}%
\pgfpathrectangle{\pgfqpoint{0.000000in}{0.000000in}}{\pgfqpoint{4.000000in}{4.000000in}}%
\pgfusepath{clip}%
\pgfsetbuttcap%
\pgfsetroundjoin%
\definecolor{currentfill}{rgb}{0.800000,0.800000,0.800000}%
\pgfsetfillcolor{currentfill}%
\pgfsetlinewidth{1.003750pt}%
\definecolor{currentstroke}{rgb}{0.000000,0.000000,0.000000}%
\pgfsetstrokecolor{currentstroke}%
\pgfsetdash{}{0pt}%
\pgfsys@defobject{currentmarker}{\pgfqpoint{-0.104167in}{-0.104167in}}{\pgfqpoint{0.104167in}{0.104167in}}{%
\pgfpathmoveto{\pgfqpoint{0.000000in}{-0.104167in}}%
\pgfpathcurveto{\pgfqpoint{0.027625in}{-0.104167in}}{\pgfqpoint{0.054123in}{-0.093191in}}{\pgfqpoint{0.073657in}{-0.073657in}}%
\pgfpathcurveto{\pgfqpoint{0.093191in}{-0.054123in}}{\pgfqpoint{0.104167in}{-0.027625in}}{\pgfqpoint{0.104167in}{0.000000in}}%
\pgfpathcurveto{\pgfqpoint{0.104167in}{0.027625in}}{\pgfqpoint{0.093191in}{0.054123in}}{\pgfqpoint{0.073657in}{0.073657in}}%
\pgfpathcurveto{\pgfqpoint{0.054123in}{0.093191in}}{\pgfqpoint{0.027625in}{0.104167in}}{\pgfqpoint{0.000000in}{0.104167in}}%
\pgfpathcurveto{\pgfqpoint{-0.027625in}{0.104167in}}{\pgfqpoint{-0.054123in}{0.093191in}}{\pgfqpoint{-0.073657in}{0.073657in}}%
\pgfpathcurveto{\pgfqpoint{-0.093191in}{0.054123in}}{\pgfqpoint{-0.104167in}{0.027625in}}{\pgfqpoint{-0.104167in}{0.000000in}}%
\pgfpathcurveto{\pgfqpoint{-0.104167in}{-0.027625in}}{\pgfqpoint{-0.093191in}{-0.054123in}}{\pgfqpoint{-0.073657in}{-0.073657in}}%
\pgfpathcurveto{\pgfqpoint{-0.054123in}{-0.093191in}}{\pgfqpoint{-0.027625in}{-0.104167in}}{\pgfqpoint{0.000000in}{-0.104167in}}%
\pgfpathclose%
\pgfusepath{stroke,fill}%
}%
\begin{pgfscope}%
\pgfsys@transformshift{1.177101in}{0.867376in}%
\pgfsys@useobject{currentmarker}{}%
\end{pgfscope}%
\end{pgfscope}%
\begin{pgfscope}%
\pgfpathrectangle{\pgfqpoint{0.000000in}{0.000000in}}{\pgfqpoint{4.000000in}{4.000000in}}%
\pgfusepath{clip}%
\pgfsetbuttcap%
\pgfsetroundjoin%
\definecolor{currentfill}{rgb}{0.800000,0.800000,0.800000}%
\pgfsetfillcolor{currentfill}%
\pgfsetlinewidth{1.003750pt}%
\definecolor{currentstroke}{rgb}{0.000000,0.000000,0.000000}%
\pgfsetstrokecolor{currentstroke}%
\pgfsetdash{}{0pt}%
\pgfsys@defobject{currentmarker}{\pgfqpoint{-0.104167in}{-0.104167in}}{\pgfqpoint{0.104167in}{0.104167in}}{%
\pgfpathmoveto{\pgfqpoint{0.000000in}{-0.104167in}}%
\pgfpathcurveto{\pgfqpoint{0.027625in}{-0.104167in}}{\pgfqpoint{0.054123in}{-0.093191in}}{\pgfqpoint{0.073657in}{-0.073657in}}%
\pgfpathcurveto{\pgfqpoint{0.093191in}{-0.054123in}}{\pgfqpoint{0.104167in}{-0.027625in}}{\pgfqpoint{0.104167in}{0.000000in}}%
\pgfpathcurveto{\pgfqpoint{0.104167in}{0.027625in}}{\pgfqpoint{0.093191in}{0.054123in}}{\pgfqpoint{0.073657in}{0.073657in}}%
\pgfpathcurveto{\pgfqpoint{0.054123in}{0.093191in}}{\pgfqpoint{0.027625in}{0.104167in}}{\pgfqpoint{0.000000in}{0.104167in}}%
\pgfpathcurveto{\pgfqpoint{-0.027625in}{0.104167in}}{\pgfqpoint{-0.054123in}{0.093191in}}{\pgfqpoint{-0.073657in}{0.073657in}}%
\pgfpathcurveto{\pgfqpoint{-0.093191in}{0.054123in}}{\pgfqpoint{-0.104167in}{0.027625in}}{\pgfqpoint{-0.104167in}{0.000000in}}%
\pgfpathcurveto{\pgfqpoint{-0.104167in}{-0.027625in}}{\pgfqpoint{-0.093191in}{-0.054123in}}{\pgfqpoint{-0.073657in}{-0.073657in}}%
\pgfpathcurveto{\pgfqpoint{-0.054123in}{-0.093191in}}{\pgfqpoint{-0.027625in}{-0.104167in}}{\pgfqpoint{0.000000in}{-0.104167in}}%
\pgfpathclose%
\pgfusepath{stroke,fill}%
}%
\begin{pgfscope}%
\pgfsys@transformshift{2.822899in}{0.867376in}%
\pgfsys@useobject{currentmarker}{}%
\end{pgfscope}%
\end{pgfscope}%
\begin{pgfscope}%
\pgfpathrectangle{\pgfqpoint{0.000000in}{0.000000in}}{\pgfqpoint{4.000000in}{4.000000in}}%
\pgfusepath{clip}%
\pgfsetbuttcap%
\pgfsetroundjoin%
\definecolor{currentfill}{rgb}{0.800000,0.800000,0.800000}%
\pgfsetfillcolor{currentfill}%
\pgfsetlinewidth{1.003750pt}%
\definecolor{currentstroke}{rgb}{0.000000,0.000000,0.000000}%
\pgfsetstrokecolor{currentstroke}%
\pgfsetdash{}{0pt}%
\pgfsys@defobject{currentmarker}{\pgfqpoint{-0.104167in}{-0.104167in}}{\pgfqpoint{0.104167in}{0.104167in}}{%
\pgfpathmoveto{\pgfqpoint{0.000000in}{-0.104167in}}%
\pgfpathcurveto{\pgfqpoint{0.027625in}{-0.104167in}}{\pgfqpoint{0.054123in}{-0.093191in}}{\pgfqpoint{0.073657in}{-0.073657in}}%
\pgfpathcurveto{\pgfqpoint{0.093191in}{-0.054123in}}{\pgfqpoint{0.104167in}{-0.027625in}}{\pgfqpoint{0.104167in}{0.000000in}}%
\pgfpathcurveto{\pgfqpoint{0.104167in}{0.027625in}}{\pgfqpoint{0.093191in}{0.054123in}}{\pgfqpoint{0.073657in}{0.073657in}}%
\pgfpathcurveto{\pgfqpoint{0.054123in}{0.093191in}}{\pgfqpoint{0.027625in}{0.104167in}}{\pgfqpoint{0.000000in}{0.104167in}}%
\pgfpathcurveto{\pgfqpoint{-0.027625in}{0.104167in}}{\pgfqpoint{-0.054123in}{0.093191in}}{\pgfqpoint{-0.073657in}{0.073657in}}%
\pgfpathcurveto{\pgfqpoint{-0.093191in}{0.054123in}}{\pgfqpoint{-0.104167in}{0.027625in}}{\pgfqpoint{-0.104167in}{0.000000in}}%
\pgfpathcurveto{\pgfqpoint{-0.104167in}{-0.027625in}}{\pgfqpoint{-0.093191in}{-0.054123in}}{\pgfqpoint{-0.073657in}{-0.073657in}}%
\pgfpathcurveto{\pgfqpoint{-0.054123in}{-0.093191in}}{\pgfqpoint{-0.027625in}{-0.104167in}}{\pgfqpoint{0.000000in}{-0.104167in}}%
\pgfpathclose%
\pgfusepath{stroke,fill}%
}%
\begin{pgfscope}%
\pgfsys@transformshift{3.331479in}{2.432624in}%
\pgfsys@useobject{currentmarker}{}%
\end{pgfscope}%
\end{pgfscope}%
\begin{pgfscope}%
\pgfpathrectangle{\pgfqpoint{0.000000in}{0.000000in}}{\pgfqpoint{4.000000in}{4.000000in}}%
\pgfusepath{clip}%
\pgfsetbuttcap%
\pgfsetroundjoin%
\definecolor{currentfill}{rgb}{0.800000,0.800000,0.800000}%
\pgfsetfillcolor{currentfill}%
\pgfsetlinewidth{1.003750pt}%
\definecolor{currentstroke}{rgb}{0.000000,0.000000,0.000000}%
\pgfsetstrokecolor{currentstroke}%
\pgfsetdash{}{0pt}%
\pgfsys@defobject{currentmarker}{\pgfqpoint{-0.104167in}{-0.104167in}}{\pgfqpoint{0.104167in}{0.104167in}}{%
\pgfpathmoveto{\pgfqpoint{0.000000in}{-0.104167in}}%
\pgfpathcurveto{\pgfqpoint{0.027625in}{-0.104167in}}{\pgfqpoint{0.054123in}{-0.093191in}}{\pgfqpoint{0.073657in}{-0.073657in}}%
\pgfpathcurveto{\pgfqpoint{0.093191in}{-0.054123in}}{\pgfqpoint{0.104167in}{-0.027625in}}{\pgfqpoint{0.104167in}{0.000000in}}%
\pgfpathcurveto{\pgfqpoint{0.104167in}{0.027625in}}{\pgfqpoint{0.093191in}{0.054123in}}{\pgfqpoint{0.073657in}{0.073657in}}%
\pgfpathcurveto{\pgfqpoint{0.054123in}{0.093191in}}{\pgfqpoint{0.027625in}{0.104167in}}{\pgfqpoint{0.000000in}{0.104167in}}%
\pgfpathcurveto{\pgfqpoint{-0.027625in}{0.104167in}}{\pgfqpoint{-0.054123in}{0.093191in}}{\pgfqpoint{-0.073657in}{0.073657in}}%
\pgfpathcurveto{\pgfqpoint{-0.093191in}{0.054123in}}{\pgfqpoint{-0.104167in}{0.027625in}}{\pgfqpoint{-0.104167in}{0.000000in}}%
\pgfpathcurveto{\pgfqpoint{-0.104167in}{-0.027625in}}{\pgfqpoint{-0.093191in}{-0.054123in}}{\pgfqpoint{-0.073657in}{-0.073657in}}%
\pgfpathcurveto{\pgfqpoint{-0.054123in}{-0.093191in}}{\pgfqpoint{-0.027625in}{-0.104167in}}{\pgfqpoint{0.000000in}{-0.104167in}}%
\pgfpathclose%
\pgfusepath{stroke,fill}%
}%
\begin{pgfscope}%
\pgfsys@transformshift{2.000000in}{3.400000in}%
\pgfsys@useobject{currentmarker}{}%
\end{pgfscope}%
\end{pgfscope}%
\begin{pgfscope}%
\pgfpathrectangle{\pgfqpoint{0.000000in}{0.000000in}}{\pgfqpoint{4.000000in}{4.000000in}}%
\pgfusepath{clip}%
\pgfsetbuttcap%
\pgfsetroundjoin%
\definecolor{currentfill}{rgb}{0.800000,0.800000,0.800000}%
\pgfsetfillcolor{currentfill}%
\pgfsetlinewidth{1.003750pt}%
\definecolor{currentstroke}{rgb}{0.000000,0.000000,0.000000}%
\pgfsetstrokecolor{currentstroke}%
\pgfsetdash{}{0pt}%
\pgfsys@defobject{currentmarker}{\pgfqpoint{-0.104167in}{-0.104167in}}{\pgfqpoint{0.104167in}{0.104167in}}{%
\pgfpathmoveto{\pgfqpoint{0.000000in}{-0.104167in}}%
\pgfpathcurveto{\pgfqpoint{0.027625in}{-0.104167in}}{\pgfqpoint{0.054123in}{-0.093191in}}{\pgfqpoint{0.073657in}{-0.073657in}}%
\pgfpathcurveto{\pgfqpoint{0.093191in}{-0.054123in}}{\pgfqpoint{0.104167in}{-0.027625in}}{\pgfqpoint{0.104167in}{0.000000in}}%
\pgfpathcurveto{\pgfqpoint{0.104167in}{0.027625in}}{\pgfqpoint{0.093191in}{0.054123in}}{\pgfqpoint{0.073657in}{0.073657in}}%
\pgfpathcurveto{\pgfqpoint{0.054123in}{0.093191in}}{\pgfqpoint{0.027625in}{0.104167in}}{\pgfqpoint{0.000000in}{0.104167in}}%
\pgfpathcurveto{\pgfqpoint{-0.027625in}{0.104167in}}{\pgfqpoint{-0.054123in}{0.093191in}}{\pgfqpoint{-0.073657in}{0.073657in}}%
\pgfpathcurveto{\pgfqpoint{-0.093191in}{0.054123in}}{\pgfqpoint{-0.104167in}{0.027625in}}{\pgfqpoint{-0.104167in}{0.000000in}}%
\pgfpathcurveto{\pgfqpoint{-0.104167in}{-0.027625in}}{\pgfqpoint{-0.093191in}{-0.054123in}}{\pgfqpoint{-0.073657in}{-0.073657in}}%
\pgfpathcurveto{\pgfqpoint{-0.054123in}{-0.093191in}}{\pgfqpoint{-0.027625in}{-0.104167in}}{\pgfqpoint{0.000000in}{-0.104167in}}%
\pgfpathclose%
\pgfusepath{stroke,fill}%
}%
\begin{pgfscope}%
\pgfsys@transformshift{0.668521in}{2.432624in}%
\pgfsys@useobject{currentmarker}{}%
\end{pgfscope}%
\end{pgfscope}%
\begin{pgfscope}%
\definecolor{textcolor}{rgb}{0.000000,0.000000,0.000000}%
\pgfsetstrokecolor{textcolor}%
\pgfsetfillcolor{textcolor}%
\pgftext[x=0.941987in,y=0.543769in,,]{\color{textcolor}\sffamily\fontsize{40.000000}{48.000000}\selectfont \(\displaystyle 1\)}%
\end{pgfscope}%
\begin{pgfscope}%
\definecolor{textcolor}{rgb}{0.000000,0.000000,0.000000}%
\pgfsetstrokecolor{textcolor}%
\pgfsetfillcolor{textcolor}%
\pgftext[x=3.058013in,y=0.543769in,,]{\color{textcolor}\sffamily\fontsize{40.000000}{48.000000}\selectfont \(\displaystyle 2\)}%
\end{pgfscope}%
\begin{pgfscope}%
\definecolor{textcolor}{rgb}{0.000000,0.000000,0.000000}%
\pgfsetstrokecolor{textcolor}%
\pgfsetfillcolor{textcolor}%
\pgftext[x=3.711902in,y=2.556231in,,]{\color{textcolor}\sffamily\fontsize{40.000000}{48.000000}\selectfont \(\displaystyle 3\)}%
\end{pgfscope}%
\begin{pgfscope}%
\definecolor{textcolor}{rgb}{0.000000,0.000000,0.000000}%
\pgfsetstrokecolor{textcolor}%
\pgfsetfillcolor{textcolor}%
\pgftext[x=2.000000in,y=3.800000in,,]{\color{textcolor}\sffamily\fontsize{40.000000}{48.000000}\selectfont \(\displaystyle 4\)}%
\end{pgfscope}%
\begin{pgfscope}%
\definecolor{textcolor}{rgb}{0.000000,0.000000,0.000000}%
\pgfsetstrokecolor{textcolor}%
\pgfsetfillcolor{textcolor}%
\pgftext[x=0.288098in,y=2.556231in,,]{\color{textcolor}\sffamily\fontsize{40.000000}{48.000000}\selectfont \(\displaystyle 5\)}%
\end{pgfscope}%
\end{pgfpicture}%
\makeatother%
\endgroup%

%% file: fig/decomposition_conic_counterexample1_rigid.pgf
%% Creator: Matplotlib, PGF backend
%%
%% To include the figure in your LaTeX document, write
%%   \input{<filename>.pgf}
%%
%% Make sure the required packages are loaded in your preamble
%%   \usepackage{pgf}
%%
%% Figures using additional raster images can only be included by \input if
%% they are in the same directory as the main LaTeX file. For loading figures
%% from other directories you can use the `import` package
%%   \usepackage{import}
%%
%% and then include the figures with
%%   \import{<path to file>}{<filename>.pgf}
%%
%% Matplotlib used the following preamble
%%   \usepackage{fontspec}
%%   \setmainfont{DejaVuSerif.ttf}[Path=\detokenize{C:/Users/ccros/Anaconda3/Lib/site-packages/matplotlib/mpl-data/fonts/ttf/}]
%%   \setsansfont{DejaVuSans.ttf}[Path=\detokenize{C:/Users/ccros/Anaconda3/Lib/site-packages/matplotlib/mpl-data/fonts/ttf/}]
%%   \setmonofont{DejaVuSansMono.ttf}[Path=\detokenize{C:/Users/ccros/Anaconda3/Lib/site-packages/matplotlib/mpl-data/fonts/ttf/}]
%%
\begingroup%
\makeatletter%
\begin{pgfpicture}%
\pgfpathrectangle{\pgfpointorigin}{\pgfqpoint{4.000000in}{4.068826in}}%
\pgfusepath{use as bounding box, clip}%
\begin{pgfscope}%
\pgfsetbuttcap%
\pgfsetmiterjoin%
\pgfsetlinewidth{0.000000pt}%
\definecolor{currentstroke}{rgb}{1.000000,1.000000,1.000000}%
\pgfsetstrokecolor{currentstroke}%
\pgfsetstrokeopacity{0.000000}%
\pgfsetdash{}{0pt}%
\pgfpathmoveto{\pgfqpoint{0.000000in}{0.000000in}}%
\pgfpathlineto{\pgfqpoint{4.000000in}{0.000000in}}%
\pgfpathlineto{\pgfqpoint{4.000000in}{4.068826in}}%
\pgfpathlineto{\pgfqpoint{0.000000in}{4.068826in}}%
\pgfpathclose%
\pgfusepath{}%
\end{pgfscope}%
\begin{pgfscope}%
\pgfpathrectangle{\pgfqpoint{0.000000in}{0.000000in}}{\pgfqpoint{4.000000in}{4.000000in}}%
\pgfusepath{clip}%
\pgfsetrectcap%
\pgfsetroundjoin%
\pgfsetlinewidth{3.011250pt}%
\definecolor{currentstroke}{rgb}{0.250980,0.250980,0.250980}%
\pgfsetstrokecolor{currentstroke}%
\pgfsetdash{}{0pt}%
\pgfpathmoveto{\pgfqpoint{2.822899in}{0.867376in}}%
\pgfpathlineto{\pgfqpoint{1.177101in}{0.867376in}}%
\pgfusepath{stroke}%
\end{pgfscope}%
\begin{pgfscope}%
\pgfpathrectangle{\pgfqpoint{0.000000in}{0.000000in}}{\pgfqpoint{4.000000in}{4.000000in}}%
\pgfusepath{clip}%
\pgfsetrectcap%
\pgfsetroundjoin%
\pgfsetlinewidth{3.011250pt}%
\definecolor{currentstroke}{rgb}{0.250980,0.250980,0.250980}%
\pgfsetstrokecolor{currentstroke}%
\pgfsetdash{}{0pt}%
\pgfpathmoveto{\pgfqpoint{2.000000in}{3.400000in}}%
\pgfpathlineto{\pgfqpoint{2.822899in}{0.867376in}}%
\pgfusepath{stroke}%
\end{pgfscope}%
\begin{pgfscope}%
\pgfpathrectangle{\pgfqpoint{0.000000in}{0.000000in}}{\pgfqpoint{4.000000in}{4.000000in}}%
\pgfusepath{clip}%
\pgfsetrectcap%
\pgfsetroundjoin%
\pgfsetlinewidth{3.011250pt}%
\definecolor{currentstroke}{rgb}{0.250980,0.250980,0.250980}%
\pgfsetstrokecolor{currentstroke}%
\pgfsetdash{}{0pt}%
\pgfpathmoveto{\pgfqpoint{2.000000in}{3.400000in}}%
\pgfpathlineto{\pgfqpoint{3.331479in}{2.432624in}}%
\pgfusepath{stroke}%
\end{pgfscope}%
\begin{pgfscope}%
\pgfpathrectangle{\pgfqpoint{0.000000in}{0.000000in}}{\pgfqpoint{4.000000in}{4.000000in}}%
\pgfusepath{clip}%
\pgfsetrectcap%
\pgfsetroundjoin%
\pgfsetlinewidth{3.011250pt}%
\definecolor{currentstroke}{rgb}{0.250980,0.250980,0.250980}%
\pgfsetstrokecolor{currentstroke}%
\pgfsetdash{}{0pt}%
\pgfpathmoveto{\pgfqpoint{0.668521in}{2.432624in}}%
\pgfpathlineto{\pgfqpoint{1.177101in}{0.867376in}}%
\pgfusepath{stroke}%
\end{pgfscope}%
\begin{pgfscope}%
\pgfpathrectangle{\pgfqpoint{0.000000in}{0.000000in}}{\pgfqpoint{4.000000in}{4.000000in}}%
\pgfusepath{clip}%
\pgfsetrectcap%
\pgfsetroundjoin%
\pgfsetlinewidth{3.011250pt}%
\definecolor{currentstroke}{rgb}{0.250980,0.250980,0.250980}%
\pgfsetstrokecolor{currentstroke}%
\pgfsetdash{}{0pt}%
\pgfpathmoveto{\pgfqpoint{0.668521in}{2.432624in}}%
\pgfpathlineto{\pgfqpoint{2.822899in}{0.867376in}}%
\pgfusepath{stroke}%
\end{pgfscope}%
\begin{pgfscope}%
\pgfpathrectangle{\pgfqpoint{0.000000in}{0.000000in}}{\pgfqpoint{4.000000in}{4.000000in}}%
\pgfusepath{clip}%
\pgfsetrectcap%
\pgfsetroundjoin%
\pgfsetlinewidth{3.011250pt}%
\definecolor{currentstroke}{rgb}{0.250980,0.250980,0.250980}%
\pgfsetstrokecolor{currentstroke}%
\pgfsetdash{}{0pt}%
\pgfpathmoveto{\pgfqpoint{0.668521in}{2.432624in}}%
\pgfpathlineto{\pgfqpoint{3.331479in}{2.432624in}}%
\pgfusepath{stroke}%
\end{pgfscope}%
\begin{pgfscope}%
\pgfpathrectangle{\pgfqpoint{0.000000in}{0.000000in}}{\pgfqpoint{4.000000in}{4.000000in}}%
\pgfusepath{clip}%
\pgfsetrectcap%
\pgfsetroundjoin%
\pgfsetlinewidth{3.011250pt}%
\definecolor{currentstroke}{rgb}{0.250980,0.250980,0.250980}%
\pgfsetstrokecolor{currentstroke}%
\pgfsetdash{}{0pt}%
\pgfpathmoveto{\pgfqpoint{0.668521in}{2.432624in}}%
\pgfpathlineto{\pgfqpoint{2.000000in}{3.400000in}}%
\pgfusepath{stroke}%
\end{pgfscope}%
\begin{pgfscope}%
\pgfpathrectangle{\pgfqpoint{0.000000in}{0.000000in}}{\pgfqpoint{4.000000in}{4.000000in}}%
\pgfusepath{clip}%
\pgfsetbuttcap%
\pgfsetroundjoin%
\definecolor{currentfill}{rgb}{0.800000,0.800000,0.800000}%
\pgfsetfillcolor{currentfill}%
\pgfsetlinewidth{1.003750pt}%
\definecolor{currentstroke}{rgb}{0.000000,0.000000,0.000000}%
\pgfsetstrokecolor{currentstroke}%
\pgfsetdash{}{0pt}%
\pgfsys@defobject{currentmarker}{\pgfqpoint{-0.104167in}{-0.104167in}}{\pgfqpoint{0.104167in}{0.104167in}}{%
\pgfpathmoveto{\pgfqpoint{0.000000in}{-0.104167in}}%
\pgfpathcurveto{\pgfqpoint{0.027625in}{-0.104167in}}{\pgfqpoint{0.054123in}{-0.093191in}}{\pgfqpoint{0.073657in}{-0.073657in}}%
\pgfpathcurveto{\pgfqpoint{0.093191in}{-0.054123in}}{\pgfqpoint{0.104167in}{-0.027625in}}{\pgfqpoint{0.104167in}{0.000000in}}%
\pgfpathcurveto{\pgfqpoint{0.104167in}{0.027625in}}{\pgfqpoint{0.093191in}{0.054123in}}{\pgfqpoint{0.073657in}{0.073657in}}%
\pgfpathcurveto{\pgfqpoint{0.054123in}{0.093191in}}{\pgfqpoint{0.027625in}{0.104167in}}{\pgfqpoint{0.000000in}{0.104167in}}%
\pgfpathcurveto{\pgfqpoint{-0.027625in}{0.104167in}}{\pgfqpoint{-0.054123in}{0.093191in}}{\pgfqpoint{-0.073657in}{0.073657in}}%
\pgfpathcurveto{\pgfqpoint{-0.093191in}{0.054123in}}{\pgfqpoint{-0.104167in}{0.027625in}}{\pgfqpoint{-0.104167in}{0.000000in}}%
\pgfpathcurveto{\pgfqpoint{-0.104167in}{-0.027625in}}{\pgfqpoint{-0.093191in}{-0.054123in}}{\pgfqpoint{-0.073657in}{-0.073657in}}%
\pgfpathcurveto{\pgfqpoint{-0.054123in}{-0.093191in}}{\pgfqpoint{-0.027625in}{-0.104167in}}{\pgfqpoint{0.000000in}{-0.104167in}}%
\pgfpathclose%
\pgfusepath{stroke,fill}%
}%
\begin{pgfscope}%
\pgfsys@transformshift{1.177101in}{0.867376in}%
\pgfsys@useobject{currentmarker}{}%
\end{pgfscope}%
\end{pgfscope}%
\begin{pgfscope}%
\pgfpathrectangle{\pgfqpoint{0.000000in}{0.000000in}}{\pgfqpoint{4.000000in}{4.000000in}}%
\pgfusepath{clip}%
\pgfsetbuttcap%
\pgfsetroundjoin%
\definecolor{currentfill}{rgb}{0.800000,0.800000,0.800000}%
\pgfsetfillcolor{currentfill}%
\pgfsetlinewidth{1.003750pt}%
\definecolor{currentstroke}{rgb}{0.000000,0.000000,0.000000}%
\pgfsetstrokecolor{currentstroke}%
\pgfsetdash{}{0pt}%
\pgfsys@defobject{currentmarker}{\pgfqpoint{-0.104167in}{-0.104167in}}{\pgfqpoint{0.104167in}{0.104167in}}{%
\pgfpathmoveto{\pgfqpoint{0.000000in}{-0.104167in}}%
\pgfpathcurveto{\pgfqpoint{0.027625in}{-0.104167in}}{\pgfqpoint{0.054123in}{-0.093191in}}{\pgfqpoint{0.073657in}{-0.073657in}}%
\pgfpathcurveto{\pgfqpoint{0.093191in}{-0.054123in}}{\pgfqpoint{0.104167in}{-0.027625in}}{\pgfqpoint{0.104167in}{0.000000in}}%
\pgfpathcurveto{\pgfqpoint{0.104167in}{0.027625in}}{\pgfqpoint{0.093191in}{0.054123in}}{\pgfqpoint{0.073657in}{0.073657in}}%
\pgfpathcurveto{\pgfqpoint{0.054123in}{0.093191in}}{\pgfqpoint{0.027625in}{0.104167in}}{\pgfqpoint{0.000000in}{0.104167in}}%
\pgfpathcurveto{\pgfqpoint{-0.027625in}{0.104167in}}{\pgfqpoint{-0.054123in}{0.093191in}}{\pgfqpoint{-0.073657in}{0.073657in}}%
\pgfpathcurveto{\pgfqpoint{-0.093191in}{0.054123in}}{\pgfqpoint{-0.104167in}{0.027625in}}{\pgfqpoint{-0.104167in}{0.000000in}}%
\pgfpathcurveto{\pgfqpoint{-0.104167in}{-0.027625in}}{\pgfqpoint{-0.093191in}{-0.054123in}}{\pgfqpoint{-0.073657in}{-0.073657in}}%
\pgfpathcurveto{\pgfqpoint{-0.054123in}{-0.093191in}}{\pgfqpoint{-0.027625in}{-0.104167in}}{\pgfqpoint{0.000000in}{-0.104167in}}%
\pgfpathclose%
\pgfusepath{stroke,fill}%
}%
\begin{pgfscope}%
\pgfsys@transformshift{2.822899in}{0.867376in}%
\pgfsys@useobject{currentmarker}{}%
\end{pgfscope}%
\end{pgfscope}%
\begin{pgfscope}%
\pgfpathrectangle{\pgfqpoint{0.000000in}{0.000000in}}{\pgfqpoint{4.000000in}{4.000000in}}%
\pgfusepath{clip}%
\pgfsetbuttcap%
\pgfsetroundjoin%
\definecolor{currentfill}{rgb}{0.800000,0.800000,0.800000}%
\pgfsetfillcolor{currentfill}%
\pgfsetlinewidth{1.003750pt}%
\definecolor{currentstroke}{rgb}{0.000000,0.000000,0.000000}%
\pgfsetstrokecolor{currentstroke}%
\pgfsetdash{}{0pt}%
\pgfsys@defobject{currentmarker}{\pgfqpoint{-0.104167in}{-0.104167in}}{\pgfqpoint{0.104167in}{0.104167in}}{%
\pgfpathmoveto{\pgfqpoint{0.000000in}{-0.104167in}}%
\pgfpathcurveto{\pgfqpoint{0.027625in}{-0.104167in}}{\pgfqpoint{0.054123in}{-0.093191in}}{\pgfqpoint{0.073657in}{-0.073657in}}%
\pgfpathcurveto{\pgfqpoint{0.093191in}{-0.054123in}}{\pgfqpoint{0.104167in}{-0.027625in}}{\pgfqpoint{0.104167in}{0.000000in}}%
\pgfpathcurveto{\pgfqpoint{0.104167in}{0.027625in}}{\pgfqpoint{0.093191in}{0.054123in}}{\pgfqpoint{0.073657in}{0.073657in}}%
\pgfpathcurveto{\pgfqpoint{0.054123in}{0.093191in}}{\pgfqpoint{0.027625in}{0.104167in}}{\pgfqpoint{0.000000in}{0.104167in}}%
\pgfpathcurveto{\pgfqpoint{-0.027625in}{0.104167in}}{\pgfqpoint{-0.054123in}{0.093191in}}{\pgfqpoint{-0.073657in}{0.073657in}}%
\pgfpathcurveto{\pgfqpoint{-0.093191in}{0.054123in}}{\pgfqpoint{-0.104167in}{0.027625in}}{\pgfqpoint{-0.104167in}{0.000000in}}%
\pgfpathcurveto{\pgfqpoint{-0.104167in}{-0.027625in}}{\pgfqpoint{-0.093191in}{-0.054123in}}{\pgfqpoint{-0.073657in}{-0.073657in}}%
\pgfpathcurveto{\pgfqpoint{-0.054123in}{-0.093191in}}{\pgfqpoint{-0.027625in}{-0.104167in}}{\pgfqpoint{0.000000in}{-0.104167in}}%
\pgfpathclose%
\pgfusepath{stroke,fill}%
}%
\begin{pgfscope}%
\pgfsys@transformshift{3.331479in}{2.432624in}%
\pgfsys@useobject{currentmarker}{}%
\end{pgfscope}%
\end{pgfscope}%
\begin{pgfscope}%
\pgfpathrectangle{\pgfqpoint{0.000000in}{0.000000in}}{\pgfqpoint{4.000000in}{4.000000in}}%
\pgfusepath{clip}%
\pgfsetbuttcap%
\pgfsetroundjoin%
\definecolor{currentfill}{rgb}{0.800000,0.800000,0.800000}%
\pgfsetfillcolor{currentfill}%
\pgfsetlinewidth{1.003750pt}%
\definecolor{currentstroke}{rgb}{0.000000,0.000000,0.000000}%
\pgfsetstrokecolor{currentstroke}%
\pgfsetdash{}{0pt}%
\pgfsys@defobject{currentmarker}{\pgfqpoint{-0.104167in}{-0.104167in}}{\pgfqpoint{0.104167in}{0.104167in}}{%
\pgfpathmoveto{\pgfqpoint{0.000000in}{-0.104167in}}%
\pgfpathcurveto{\pgfqpoint{0.027625in}{-0.104167in}}{\pgfqpoint{0.054123in}{-0.093191in}}{\pgfqpoint{0.073657in}{-0.073657in}}%
\pgfpathcurveto{\pgfqpoint{0.093191in}{-0.054123in}}{\pgfqpoint{0.104167in}{-0.027625in}}{\pgfqpoint{0.104167in}{0.000000in}}%
\pgfpathcurveto{\pgfqpoint{0.104167in}{0.027625in}}{\pgfqpoint{0.093191in}{0.054123in}}{\pgfqpoint{0.073657in}{0.073657in}}%
\pgfpathcurveto{\pgfqpoint{0.054123in}{0.093191in}}{\pgfqpoint{0.027625in}{0.104167in}}{\pgfqpoint{0.000000in}{0.104167in}}%
\pgfpathcurveto{\pgfqpoint{-0.027625in}{0.104167in}}{\pgfqpoint{-0.054123in}{0.093191in}}{\pgfqpoint{-0.073657in}{0.073657in}}%
\pgfpathcurveto{\pgfqpoint{-0.093191in}{0.054123in}}{\pgfqpoint{-0.104167in}{0.027625in}}{\pgfqpoint{-0.104167in}{0.000000in}}%
\pgfpathcurveto{\pgfqpoint{-0.104167in}{-0.027625in}}{\pgfqpoint{-0.093191in}{-0.054123in}}{\pgfqpoint{-0.073657in}{-0.073657in}}%
\pgfpathcurveto{\pgfqpoint{-0.054123in}{-0.093191in}}{\pgfqpoint{-0.027625in}{-0.104167in}}{\pgfqpoint{0.000000in}{-0.104167in}}%
\pgfpathclose%
\pgfusepath{stroke,fill}%
}%
\begin{pgfscope}%
\pgfsys@transformshift{2.000000in}{3.400000in}%
\pgfsys@useobject{currentmarker}{}%
\end{pgfscope}%
\end{pgfscope}%
\begin{pgfscope}%
\pgfpathrectangle{\pgfqpoint{0.000000in}{0.000000in}}{\pgfqpoint{4.000000in}{4.000000in}}%
\pgfusepath{clip}%
\pgfsetbuttcap%
\pgfsetroundjoin%
\definecolor{currentfill}{rgb}{0.800000,0.800000,0.800000}%
\pgfsetfillcolor{currentfill}%
\pgfsetlinewidth{1.003750pt}%
\definecolor{currentstroke}{rgb}{0.000000,0.000000,0.000000}%
\pgfsetstrokecolor{currentstroke}%
\pgfsetdash{}{0pt}%
\pgfsys@defobject{currentmarker}{\pgfqpoint{-0.104167in}{-0.104167in}}{\pgfqpoint{0.104167in}{0.104167in}}{%
\pgfpathmoveto{\pgfqpoint{0.000000in}{-0.104167in}}%
\pgfpathcurveto{\pgfqpoint{0.027625in}{-0.104167in}}{\pgfqpoint{0.054123in}{-0.093191in}}{\pgfqpoint{0.073657in}{-0.073657in}}%
\pgfpathcurveto{\pgfqpoint{0.093191in}{-0.054123in}}{\pgfqpoint{0.104167in}{-0.027625in}}{\pgfqpoint{0.104167in}{0.000000in}}%
\pgfpathcurveto{\pgfqpoint{0.104167in}{0.027625in}}{\pgfqpoint{0.093191in}{0.054123in}}{\pgfqpoint{0.073657in}{0.073657in}}%
\pgfpathcurveto{\pgfqpoint{0.054123in}{0.093191in}}{\pgfqpoint{0.027625in}{0.104167in}}{\pgfqpoint{0.000000in}{0.104167in}}%
\pgfpathcurveto{\pgfqpoint{-0.027625in}{0.104167in}}{\pgfqpoint{-0.054123in}{0.093191in}}{\pgfqpoint{-0.073657in}{0.073657in}}%
\pgfpathcurveto{\pgfqpoint{-0.093191in}{0.054123in}}{\pgfqpoint{-0.104167in}{0.027625in}}{\pgfqpoint{-0.104167in}{0.000000in}}%
\pgfpathcurveto{\pgfqpoint{-0.104167in}{-0.027625in}}{\pgfqpoint{-0.093191in}{-0.054123in}}{\pgfqpoint{-0.073657in}{-0.073657in}}%
\pgfpathcurveto{\pgfqpoint{-0.054123in}{-0.093191in}}{\pgfqpoint{-0.027625in}{-0.104167in}}{\pgfqpoint{0.000000in}{-0.104167in}}%
\pgfpathclose%
\pgfusepath{stroke,fill}%
}%
\begin{pgfscope}%
\pgfsys@transformshift{0.668521in}{2.432624in}%
\pgfsys@useobject{currentmarker}{}%
\end{pgfscope}%
\end{pgfscope}%
\begin{pgfscope}%
\definecolor{textcolor}{rgb}{0.000000,0.000000,0.000000}%
\pgfsetstrokecolor{textcolor}%
\pgfsetfillcolor{textcolor}%
\pgftext[x=0.941987in,y=0.543769in,,]{\color{textcolor}\sffamily\fontsize{40.000000}{48.000000}\selectfont \(\displaystyle 1\)}%
\end{pgfscope}%
\begin{pgfscope}%
\definecolor{textcolor}{rgb}{0.000000,0.000000,0.000000}%
\pgfsetstrokecolor{textcolor}%
\pgfsetfillcolor{textcolor}%
\pgftext[x=3.058013in,y=0.543769in,,]{\color{textcolor}\sffamily\fontsize{40.000000}{48.000000}\selectfont \(\displaystyle 2\)}%
\end{pgfscope}%
\begin{pgfscope}%
\definecolor{textcolor}{rgb}{0.000000,0.000000,0.000000}%
\pgfsetstrokecolor{textcolor}%
\pgfsetfillcolor{textcolor}%
\pgftext[x=3.711902in,y=2.556231in,,]{\color{textcolor}\sffamily\fontsize{40.000000}{48.000000}\selectfont \(\displaystyle 3\)}%
\end{pgfscope}%
\begin{pgfscope}%
\definecolor{textcolor}{rgb}{0.000000,0.000000,0.000000}%
\pgfsetstrokecolor{textcolor}%
\pgfsetfillcolor{textcolor}%
\pgftext[x=2.000000in,y=3.800000in,,]{\color{textcolor}\sffamily\fontsize{40.000000}{48.000000}\selectfont \(\displaystyle 4\)}%
\end{pgfscope}%
\begin{pgfscope}%
\definecolor{textcolor}{rgb}{0.000000,0.000000,0.000000}%
\pgfsetstrokecolor{textcolor}%
\pgfsetfillcolor{textcolor}%
\pgftext[x=0.288098in,y=2.556231in,,]{\color{textcolor}\sffamily\fontsize{40.000000}{48.000000}\selectfont \(\displaystyle 5\)}%
\end{pgfscope}%
\end{pgfpicture}%
\makeatother%
\endgroup%

%% file: fig/decomposition_conic_counterexample2_rigid.pgf
%% Creator: Matplotlib, PGF backend
%%
%% To include the figure in your LaTeX document, write
%%   \input{<filename>.pgf}
%%
%% Make sure the required packages are loaded in your preamble
%%   \usepackage{pgf}
%%
%% Figures using additional raster images can only be included by \input if
%% they are in the same directory as the main LaTeX file. For loading figures
%% from other directories you can use the `import` package
%%   \usepackage{import}
%%
%% and then include the figures with
%%   \import{<path to file>}{<filename>.pgf}
%%
%% Matplotlib used the following preamble
%%   \usepackage{fontspec}
%%   \setmainfont{DejaVuSerif.ttf}[Path=\detokenize{C:/Users/ccros/Anaconda3/Lib/site-packages/matplotlib/mpl-data/fonts/ttf/}]
%%   \setsansfont{DejaVuSans.ttf}[Path=\detokenize{C:/Users/ccros/Anaconda3/Lib/site-packages/matplotlib/mpl-data/fonts/ttf/}]
%%   \setmonofont{DejaVuSansMono.ttf}[Path=\detokenize{C:/Users/ccros/Anaconda3/Lib/site-packages/matplotlib/mpl-data/fonts/ttf/}]
%%
\begingroup%
\makeatletter%
\begin{pgfpicture}%
\pgfpathrectangle{\pgfpointorigin}{\pgfqpoint{4.000000in}{4.068826in}}%
\pgfusepath{use as bounding box, clip}%
\begin{pgfscope}%
\pgfsetbuttcap%
\pgfsetmiterjoin%
\pgfsetlinewidth{0.000000pt}%
\definecolor{currentstroke}{rgb}{1.000000,1.000000,1.000000}%
\pgfsetstrokecolor{currentstroke}%
\pgfsetstrokeopacity{0.000000}%
\pgfsetdash{}{0pt}%
\pgfpathmoveto{\pgfqpoint{0.000000in}{0.000000in}}%
\pgfpathlineto{\pgfqpoint{4.000000in}{0.000000in}}%
\pgfpathlineto{\pgfqpoint{4.000000in}{4.068826in}}%
\pgfpathlineto{\pgfqpoint{0.000000in}{4.068826in}}%
\pgfpathclose%
\pgfusepath{}%
\end{pgfscope}%
\begin{pgfscope}%
\pgfpathrectangle{\pgfqpoint{0.000000in}{0.000000in}}{\pgfqpoint{4.000000in}{4.000000in}}%
\pgfusepath{clip}%
\pgfsetrectcap%
\pgfsetroundjoin%
\pgfsetlinewidth{3.011250pt}%
\definecolor{currentstroke}{rgb}{0.250980,0.250980,0.250980}%
\pgfsetstrokecolor{currentstroke}%
\pgfsetdash{}{0pt}%
\pgfpathmoveto{\pgfqpoint{2.822899in}{0.867376in}}%
\pgfpathlineto{\pgfqpoint{1.177101in}{0.867376in}}%
\pgfusepath{stroke}%
\end{pgfscope}%
\begin{pgfscope}%
\pgfpathrectangle{\pgfqpoint{0.000000in}{0.000000in}}{\pgfqpoint{4.000000in}{4.000000in}}%
\pgfusepath{clip}%
\pgfsetrectcap%
\pgfsetroundjoin%
\pgfsetlinewidth{3.011250pt}%
\definecolor{currentstroke}{rgb}{0.250980,0.250980,0.250980}%
\pgfsetstrokecolor{currentstroke}%
\pgfsetdash{}{0pt}%
\pgfpathmoveto{\pgfqpoint{3.331479in}{2.432624in}}%
\pgfpathlineto{\pgfqpoint{2.822899in}{0.867376in}}%
\pgfusepath{stroke}%
\end{pgfscope}%
\begin{pgfscope}%
\pgfpathrectangle{\pgfqpoint{0.000000in}{0.000000in}}{\pgfqpoint{4.000000in}{4.000000in}}%
\pgfusepath{clip}%
\pgfsetrectcap%
\pgfsetroundjoin%
\pgfsetlinewidth{3.011250pt}%
\definecolor{currentstroke}{rgb}{0.250980,0.250980,0.250980}%
\pgfsetstrokecolor{currentstroke}%
\pgfsetdash{}{0pt}%
\pgfpathmoveto{\pgfqpoint{2.000000in}{3.400000in}}%
\pgfpathlineto{\pgfqpoint{2.822899in}{0.867376in}}%
\pgfusepath{stroke}%
\end{pgfscope}%
\begin{pgfscope}%
\pgfpathrectangle{\pgfqpoint{0.000000in}{0.000000in}}{\pgfqpoint{4.000000in}{4.000000in}}%
\pgfusepath{clip}%
\pgfsetrectcap%
\pgfsetroundjoin%
\pgfsetlinewidth{3.011250pt}%
\definecolor{currentstroke}{rgb}{0.250980,0.250980,0.250980}%
\pgfsetstrokecolor{currentstroke}%
\pgfsetdash{}{0pt}%
\pgfpathmoveto{\pgfqpoint{2.000000in}{3.400000in}}%
\pgfpathlineto{\pgfqpoint{3.331479in}{2.432624in}}%
\pgfusepath{stroke}%
\end{pgfscope}%
\begin{pgfscope}%
\pgfpathrectangle{\pgfqpoint{0.000000in}{0.000000in}}{\pgfqpoint{4.000000in}{4.000000in}}%
\pgfusepath{clip}%
\pgfsetrectcap%
\pgfsetroundjoin%
\pgfsetlinewidth{3.011250pt}%
\definecolor{currentstroke}{rgb}{0.250980,0.250980,0.250980}%
\pgfsetstrokecolor{currentstroke}%
\pgfsetdash{}{0pt}%
\pgfpathmoveto{\pgfqpoint{0.668521in}{2.432624in}}%
\pgfpathlineto{\pgfqpoint{2.822899in}{0.867376in}}%
\pgfusepath{stroke}%
\end{pgfscope}%
\begin{pgfscope}%
\pgfpathrectangle{\pgfqpoint{0.000000in}{0.000000in}}{\pgfqpoint{4.000000in}{4.000000in}}%
\pgfusepath{clip}%
\pgfsetrectcap%
\pgfsetroundjoin%
\pgfsetlinewidth{3.011250pt}%
\definecolor{currentstroke}{rgb}{0.250980,0.250980,0.250980}%
\pgfsetstrokecolor{currentstroke}%
\pgfsetdash{}{0pt}%
\pgfpathmoveto{\pgfqpoint{0.668521in}{2.432624in}}%
\pgfpathlineto{\pgfqpoint{3.331479in}{2.432624in}}%
\pgfusepath{stroke}%
\end{pgfscope}%
\begin{pgfscope}%
\pgfpathrectangle{\pgfqpoint{0.000000in}{0.000000in}}{\pgfqpoint{4.000000in}{4.000000in}}%
\pgfusepath{clip}%
\pgfsetrectcap%
\pgfsetroundjoin%
\pgfsetlinewidth{3.011250pt}%
\definecolor{currentstroke}{rgb}{0.250980,0.250980,0.250980}%
\pgfsetstrokecolor{currentstroke}%
\pgfsetdash{}{0pt}%
\pgfpathmoveto{\pgfqpoint{0.668521in}{2.432624in}}%
\pgfpathlineto{\pgfqpoint{2.000000in}{3.400000in}}%
\pgfusepath{stroke}%
\end{pgfscope}%
\begin{pgfscope}%
\pgfpathrectangle{\pgfqpoint{0.000000in}{0.000000in}}{\pgfqpoint{4.000000in}{4.000000in}}%
\pgfusepath{clip}%
\pgfsetbuttcap%
\pgfsetroundjoin%
\definecolor{currentfill}{rgb}{0.800000,0.800000,0.800000}%
\pgfsetfillcolor{currentfill}%
\pgfsetlinewidth{1.003750pt}%
\definecolor{currentstroke}{rgb}{0.000000,0.000000,0.000000}%
\pgfsetstrokecolor{currentstroke}%
\pgfsetdash{}{0pt}%
\pgfsys@defobject{currentmarker}{\pgfqpoint{-0.104167in}{-0.104167in}}{\pgfqpoint{0.104167in}{0.104167in}}{%
\pgfpathmoveto{\pgfqpoint{0.000000in}{-0.104167in}}%
\pgfpathcurveto{\pgfqpoint{0.027625in}{-0.104167in}}{\pgfqpoint{0.054123in}{-0.093191in}}{\pgfqpoint{0.073657in}{-0.073657in}}%
\pgfpathcurveto{\pgfqpoint{0.093191in}{-0.054123in}}{\pgfqpoint{0.104167in}{-0.027625in}}{\pgfqpoint{0.104167in}{0.000000in}}%
\pgfpathcurveto{\pgfqpoint{0.104167in}{0.027625in}}{\pgfqpoint{0.093191in}{0.054123in}}{\pgfqpoint{0.073657in}{0.073657in}}%
\pgfpathcurveto{\pgfqpoint{0.054123in}{0.093191in}}{\pgfqpoint{0.027625in}{0.104167in}}{\pgfqpoint{0.000000in}{0.104167in}}%
\pgfpathcurveto{\pgfqpoint{-0.027625in}{0.104167in}}{\pgfqpoint{-0.054123in}{0.093191in}}{\pgfqpoint{-0.073657in}{0.073657in}}%
\pgfpathcurveto{\pgfqpoint{-0.093191in}{0.054123in}}{\pgfqpoint{-0.104167in}{0.027625in}}{\pgfqpoint{-0.104167in}{0.000000in}}%
\pgfpathcurveto{\pgfqpoint{-0.104167in}{-0.027625in}}{\pgfqpoint{-0.093191in}{-0.054123in}}{\pgfqpoint{-0.073657in}{-0.073657in}}%
\pgfpathcurveto{\pgfqpoint{-0.054123in}{-0.093191in}}{\pgfqpoint{-0.027625in}{-0.104167in}}{\pgfqpoint{0.000000in}{-0.104167in}}%
\pgfpathclose%
\pgfusepath{stroke,fill}%
}%
\begin{pgfscope}%
\pgfsys@transformshift{1.177101in}{0.867376in}%
\pgfsys@useobject{currentmarker}{}%
\end{pgfscope}%
\end{pgfscope}%
\begin{pgfscope}%
\pgfpathrectangle{\pgfqpoint{0.000000in}{0.000000in}}{\pgfqpoint{4.000000in}{4.000000in}}%
\pgfusepath{clip}%
\pgfsetbuttcap%
\pgfsetroundjoin%
\definecolor{currentfill}{rgb}{0.800000,0.800000,0.800000}%
\pgfsetfillcolor{currentfill}%
\pgfsetlinewidth{1.003750pt}%
\definecolor{currentstroke}{rgb}{0.000000,0.000000,0.000000}%
\pgfsetstrokecolor{currentstroke}%
\pgfsetdash{}{0pt}%
\pgfsys@defobject{currentmarker}{\pgfqpoint{-0.104167in}{-0.104167in}}{\pgfqpoint{0.104167in}{0.104167in}}{%
\pgfpathmoveto{\pgfqpoint{0.000000in}{-0.104167in}}%
\pgfpathcurveto{\pgfqpoint{0.027625in}{-0.104167in}}{\pgfqpoint{0.054123in}{-0.093191in}}{\pgfqpoint{0.073657in}{-0.073657in}}%
\pgfpathcurveto{\pgfqpoint{0.093191in}{-0.054123in}}{\pgfqpoint{0.104167in}{-0.027625in}}{\pgfqpoint{0.104167in}{0.000000in}}%
\pgfpathcurveto{\pgfqpoint{0.104167in}{0.027625in}}{\pgfqpoint{0.093191in}{0.054123in}}{\pgfqpoint{0.073657in}{0.073657in}}%
\pgfpathcurveto{\pgfqpoint{0.054123in}{0.093191in}}{\pgfqpoint{0.027625in}{0.104167in}}{\pgfqpoint{0.000000in}{0.104167in}}%
\pgfpathcurveto{\pgfqpoint{-0.027625in}{0.104167in}}{\pgfqpoint{-0.054123in}{0.093191in}}{\pgfqpoint{-0.073657in}{0.073657in}}%
\pgfpathcurveto{\pgfqpoint{-0.093191in}{0.054123in}}{\pgfqpoint{-0.104167in}{0.027625in}}{\pgfqpoint{-0.104167in}{0.000000in}}%
\pgfpathcurveto{\pgfqpoint{-0.104167in}{-0.027625in}}{\pgfqpoint{-0.093191in}{-0.054123in}}{\pgfqpoint{-0.073657in}{-0.073657in}}%
\pgfpathcurveto{\pgfqpoint{-0.054123in}{-0.093191in}}{\pgfqpoint{-0.027625in}{-0.104167in}}{\pgfqpoint{0.000000in}{-0.104167in}}%
\pgfpathclose%
\pgfusepath{stroke,fill}%
}%
\begin{pgfscope}%
\pgfsys@transformshift{2.822899in}{0.867376in}%
\pgfsys@useobject{currentmarker}{}%
\end{pgfscope}%
\end{pgfscope}%
\begin{pgfscope}%
\pgfpathrectangle{\pgfqpoint{0.000000in}{0.000000in}}{\pgfqpoint{4.000000in}{4.000000in}}%
\pgfusepath{clip}%
\pgfsetbuttcap%
\pgfsetroundjoin%
\definecolor{currentfill}{rgb}{0.800000,0.800000,0.800000}%
\pgfsetfillcolor{currentfill}%
\pgfsetlinewidth{1.003750pt}%
\definecolor{currentstroke}{rgb}{0.000000,0.000000,0.000000}%
\pgfsetstrokecolor{currentstroke}%
\pgfsetdash{}{0pt}%
\pgfsys@defobject{currentmarker}{\pgfqpoint{-0.104167in}{-0.104167in}}{\pgfqpoint{0.104167in}{0.104167in}}{%
\pgfpathmoveto{\pgfqpoint{0.000000in}{-0.104167in}}%
\pgfpathcurveto{\pgfqpoint{0.027625in}{-0.104167in}}{\pgfqpoint{0.054123in}{-0.093191in}}{\pgfqpoint{0.073657in}{-0.073657in}}%
\pgfpathcurveto{\pgfqpoint{0.093191in}{-0.054123in}}{\pgfqpoint{0.104167in}{-0.027625in}}{\pgfqpoint{0.104167in}{0.000000in}}%
\pgfpathcurveto{\pgfqpoint{0.104167in}{0.027625in}}{\pgfqpoint{0.093191in}{0.054123in}}{\pgfqpoint{0.073657in}{0.073657in}}%
\pgfpathcurveto{\pgfqpoint{0.054123in}{0.093191in}}{\pgfqpoint{0.027625in}{0.104167in}}{\pgfqpoint{0.000000in}{0.104167in}}%
\pgfpathcurveto{\pgfqpoint{-0.027625in}{0.104167in}}{\pgfqpoint{-0.054123in}{0.093191in}}{\pgfqpoint{-0.073657in}{0.073657in}}%
\pgfpathcurveto{\pgfqpoint{-0.093191in}{0.054123in}}{\pgfqpoint{-0.104167in}{0.027625in}}{\pgfqpoint{-0.104167in}{0.000000in}}%
\pgfpathcurveto{\pgfqpoint{-0.104167in}{-0.027625in}}{\pgfqpoint{-0.093191in}{-0.054123in}}{\pgfqpoint{-0.073657in}{-0.073657in}}%
\pgfpathcurveto{\pgfqpoint{-0.054123in}{-0.093191in}}{\pgfqpoint{-0.027625in}{-0.104167in}}{\pgfqpoint{0.000000in}{-0.104167in}}%
\pgfpathclose%
\pgfusepath{stroke,fill}%
}%
\begin{pgfscope}%
\pgfsys@transformshift{3.331479in}{2.432624in}%
\pgfsys@useobject{currentmarker}{}%
\end{pgfscope}%
\end{pgfscope}%
\begin{pgfscope}%
\pgfpathrectangle{\pgfqpoint{0.000000in}{0.000000in}}{\pgfqpoint{4.000000in}{4.000000in}}%
\pgfusepath{clip}%
\pgfsetbuttcap%
\pgfsetroundjoin%
\definecolor{currentfill}{rgb}{0.800000,0.800000,0.800000}%
\pgfsetfillcolor{currentfill}%
\pgfsetlinewidth{1.003750pt}%
\definecolor{currentstroke}{rgb}{0.000000,0.000000,0.000000}%
\pgfsetstrokecolor{currentstroke}%
\pgfsetdash{}{0pt}%
\pgfsys@defobject{currentmarker}{\pgfqpoint{-0.104167in}{-0.104167in}}{\pgfqpoint{0.104167in}{0.104167in}}{%
\pgfpathmoveto{\pgfqpoint{0.000000in}{-0.104167in}}%
\pgfpathcurveto{\pgfqpoint{0.027625in}{-0.104167in}}{\pgfqpoint{0.054123in}{-0.093191in}}{\pgfqpoint{0.073657in}{-0.073657in}}%
\pgfpathcurveto{\pgfqpoint{0.093191in}{-0.054123in}}{\pgfqpoint{0.104167in}{-0.027625in}}{\pgfqpoint{0.104167in}{0.000000in}}%
\pgfpathcurveto{\pgfqpoint{0.104167in}{0.027625in}}{\pgfqpoint{0.093191in}{0.054123in}}{\pgfqpoint{0.073657in}{0.073657in}}%
\pgfpathcurveto{\pgfqpoint{0.054123in}{0.093191in}}{\pgfqpoint{0.027625in}{0.104167in}}{\pgfqpoint{0.000000in}{0.104167in}}%
\pgfpathcurveto{\pgfqpoint{-0.027625in}{0.104167in}}{\pgfqpoint{-0.054123in}{0.093191in}}{\pgfqpoint{-0.073657in}{0.073657in}}%
\pgfpathcurveto{\pgfqpoint{-0.093191in}{0.054123in}}{\pgfqpoint{-0.104167in}{0.027625in}}{\pgfqpoint{-0.104167in}{0.000000in}}%
\pgfpathcurveto{\pgfqpoint{-0.104167in}{-0.027625in}}{\pgfqpoint{-0.093191in}{-0.054123in}}{\pgfqpoint{-0.073657in}{-0.073657in}}%
\pgfpathcurveto{\pgfqpoint{-0.054123in}{-0.093191in}}{\pgfqpoint{-0.027625in}{-0.104167in}}{\pgfqpoint{0.000000in}{-0.104167in}}%
\pgfpathclose%
\pgfusepath{stroke,fill}%
}%
\begin{pgfscope}%
\pgfsys@transformshift{2.000000in}{3.400000in}%
\pgfsys@useobject{currentmarker}{}%
\end{pgfscope}%
\end{pgfscope}%
\begin{pgfscope}%
\pgfpathrectangle{\pgfqpoint{0.000000in}{0.000000in}}{\pgfqpoint{4.000000in}{4.000000in}}%
\pgfusepath{clip}%
\pgfsetbuttcap%
\pgfsetroundjoin%
\definecolor{currentfill}{rgb}{0.800000,0.800000,0.800000}%
\pgfsetfillcolor{currentfill}%
\pgfsetlinewidth{1.003750pt}%
\definecolor{currentstroke}{rgb}{0.000000,0.000000,0.000000}%
\pgfsetstrokecolor{currentstroke}%
\pgfsetdash{}{0pt}%
\pgfsys@defobject{currentmarker}{\pgfqpoint{-0.104167in}{-0.104167in}}{\pgfqpoint{0.104167in}{0.104167in}}{%
\pgfpathmoveto{\pgfqpoint{0.000000in}{-0.104167in}}%
\pgfpathcurveto{\pgfqpoint{0.027625in}{-0.104167in}}{\pgfqpoint{0.054123in}{-0.093191in}}{\pgfqpoint{0.073657in}{-0.073657in}}%
\pgfpathcurveto{\pgfqpoint{0.093191in}{-0.054123in}}{\pgfqpoint{0.104167in}{-0.027625in}}{\pgfqpoint{0.104167in}{0.000000in}}%
\pgfpathcurveto{\pgfqpoint{0.104167in}{0.027625in}}{\pgfqpoint{0.093191in}{0.054123in}}{\pgfqpoint{0.073657in}{0.073657in}}%
\pgfpathcurveto{\pgfqpoint{0.054123in}{0.093191in}}{\pgfqpoint{0.027625in}{0.104167in}}{\pgfqpoint{0.000000in}{0.104167in}}%
\pgfpathcurveto{\pgfqpoint{-0.027625in}{0.104167in}}{\pgfqpoint{-0.054123in}{0.093191in}}{\pgfqpoint{-0.073657in}{0.073657in}}%
\pgfpathcurveto{\pgfqpoint{-0.093191in}{0.054123in}}{\pgfqpoint{-0.104167in}{0.027625in}}{\pgfqpoint{-0.104167in}{0.000000in}}%
\pgfpathcurveto{\pgfqpoint{-0.104167in}{-0.027625in}}{\pgfqpoint{-0.093191in}{-0.054123in}}{\pgfqpoint{-0.073657in}{-0.073657in}}%
\pgfpathcurveto{\pgfqpoint{-0.054123in}{-0.093191in}}{\pgfqpoint{-0.027625in}{-0.104167in}}{\pgfqpoint{0.000000in}{-0.104167in}}%
\pgfpathclose%
\pgfusepath{stroke,fill}%
}%
\begin{pgfscope}%
\pgfsys@transformshift{0.668521in}{2.432624in}%
\pgfsys@useobject{currentmarker}{}%
\end{pgfscope}%
\end{pgfscope}%
\begin{pgfscope}%
\definecolor{textcolor}{rgb}{0.000000,0.000000,0.000000}%
\pgfsetstrokecolor{textcolor}%
\pgfsetfillcolor{textcolor}%
\pgftext[x=0.941987in,y=0.543769in,,]{\color{textcolor}\sffamily\fontsize{40.000000}{48.000000}\selectfont \(\displaystyle 1\)}%
\end{pgfscope}%
\begin{pgfscope}%
\definecolor{textcolor}{rgb}{0.000000,0.000000,0.000000}%
\pgfsetstrokecolor{textcolor}%
\pgfsetfillcolor{textcolor}%
\pgftext[x=3.058013in,y=0.543769in,,]{\color{textcolor}\sffamily\fontsize{40.000000}{48.000000}\selectfont \(\displaystyle 2\)}%
\end{pgfscope}%
\begin{pgfscope}%
\definecolor{textcolor}{rgb}{0.000000,0.000000,0.000000}%
\pgfsetstrokecolor{textcolor}%
\pgfsetfillcolor{textcolor}%
\pgftext[x=3.711902in,y=2.556231in,,]{\color{textcolor}\sffamily\fontsize{40.000000}{48.000000}\selectfont \(\displaystyle 3\)}%
\end{pgfscope}%
\begin{pgfscope}%
\definecolor{textcolor}{rgb}{0.000000,0.000000,0.000000}%
\pgfsetstrokecolor{textcolor}%
\pgfsetfillcolor{textcolor}%
\pgftext[x=2.000000in,y=3.800000in,,]{\color{textcolor}\sffamily\fontsize{40.000000}{48.000000}\selectfont \(\displaystyle 4\)}%
\end{pgfscope}%
\begin{pgfscope}%
\definecolor{textcolor}{rgb}{0.000000,0.000000,0.000000}%
\pgfsetstrokecolor{textcolor}%
\pgfsetfillcolor{textcolor}%
\pgftext[x=0.288098in,y=2.556231in,,]{\color{textcolor}\sffamily\fontsize{40.000000}{48.000000}\selectfont \(\displaystyle 5\)}%
\end{pgfscope}%
\end{pgfpicture}%
\makeatother%
\endgroup%

%% file: fig/decomposition_conic_original_spanning_tree.pgf
%% Creator: Matplotlib, PGF backend
%%
%% To include the figure in your LaTeX document, write
%%   \input{<filename>.pgf}
%%
%% Make sure the required packages are loaded in your preamble
%%   \usepackage{pgf}
%%
%% Figures using additional raster images can only be included by \input if
%% they are in the same directory as the main LaTeX file. For loading figures
%% from other directories you can use the `import` package
%%   \usepackage{import}
%%
%% and then include the figures with
%%   \import{<path to file>}{<filename>.pgf}
%%
%% Matplotlib used the following preamble
%%   \usepackage{fontspec}
%%   \setmainfont{DejaVuSerif.ttf}[Path=\detokenize{C:/Users/ccros/Anaconda3/Lib/site-packages/matplotlib/mpl-data/fonts/ttf/}]
%%   \setsansfont{DejaVuSans.ttf}[Path=\detokenize{C:/Users/ccros/Anaconda3/Lib/site-packages/matplotlib/mpl-data/fonts/ttf/}]
%%   \setmonofont{DejaVuSansMono.ttf}[Path=\detokenize{C:/Users/ccros/Anaconda3/Lib/site-packages/matplotlib/mpl-data/fonts/ttf/}]
%%
\begingroup%
\makeatletter%
\begin{pgfpicture}%
\pgfpathrectangle{\pgfpointorigin}{\pgfqpoint{4.000000in}{4.068826in}}%
\pgfusepath{use as bounding box, clip}%
\begin{pgfscope}%
\pgfsetbuttcap%
\pgfsetmiterjoin%
\pgfsetlinewidth{0.000000pt}%
\definecolor{currentstroke}{rgb}{1.000000,1.000000,1.000000}%
\pgfsetstrokecolor{currentstroke}%
\pgfsetstrokeopacity{0.000000}%
\pgfsetdash{}{0pt}%
\pgfpathmoveto{\pgfqpoint{0.000000in}{0.000000in}}%
\pgfpathlineto{\pgfqpoint{4.000000in}{0.000000in}}%
\pgfpathlineto{\pgfqpoint{4.000000in}{4.068826in}}%
\pgfpathlineto{\pgfqpoint{0.000000in}{4.068826in}}%
\pgfpathclose%
\pgfusepath{}%
\end{pgfscope}%
\begin{pgfscope}%
\pgfpathrectangle{\pgfqpoint{0.000000in}{0.000000in}}{\pgfqpoint{4.000000in}{4.000000in}}%
\pgfusepath{clip}%
\pgfsetrectcap%
\pgfsetroundjoin%
\pgfsetlinewidth{3.011250pt}%
\definecolor{currentstroke}{rgb}{0.250980,0.250980,0.250980}%
\pgfsetstrokecolor{currentstroke}%
\pgfsetdash{}{0pt}%
\pgfpathmoveto{\pgfqpoint{2.000000in}{3.400000in}}%
\pgfpathlineto{\pgfqpoint{1.177101in}{0.867376in}}%
\pgfusepath{stroke}%
\end{pgfscope}%
\begin{pgfscope}%
\pgfpathrectangle{\pgfqpoint{0.000000in}{0.000000in}}{\pgfqpoint{4.000000in}{4.000000in}}%
\pgfusepath{clip}%
\pgfsetrectcap%
\pgfsetroundjoin%
\pgfsetlinewidth{3.011250pt}%
\definecolor{currentstroke}{rgb}{0.250980,0.250980,0.250980}%
\pgfsetstrokecolor{currentstroke}%
\pgfsetdash{}{0pt}%
\pgfpathmoveto{\pgfqpoint{2.000000in}{3.400000in}}%
\pgfpathlineto{\pgfqpoint{2.822899in}{0.867376in}}%
\pgfusepath{stroke}%
\end{pgfscope}%
\begin{pgfscope}%
\pgfpathrectangle{\pgfqpoint{0.000000in}{0.000000in}}{\pgfqpoint{4.000000in}{4.000000in}}%
\pgfusepath{clip}%
\pgfsetrectcap%
\pgfsetroundjoin%
\pgfsetlinewidth{3.011250pt}%
\definecolor{currentstroke}{rgb}{0.250980,0.250980,0.250980}%
\pgfsetstrokecolor{currentstroke}%
\pgfsetdash{}{0pt}%
\pgfpathmoveto{\pgfqpoint{2.000000in}{3.400000in}}%
\pgfpathlineto{\pgfqpoint{3.331479in}{2.432624in}}%
\pgfusepath{stroke}%
\end{pgfscope}%
\begin{pgfscope}%
\pgfpathrectangle{\pgfqpoint{0.000000in}{0.000000in}}{\pgfqpoint{4.000000in}{4.000000in}}%
\pgfusepath{clip}%
\pgfsetrectcap%
\pgfsetroundjoin%
\pgfsetlinewidth{3.011250pt}%
\definecolor{currentstroke}{rgb}{0.250980,0.250980,0.250980}%
\pgfsetstrokecolor{currentstroke}%
\pgfsetdash{}{0pt}%
\pgfpathmoveto{\pgfqpoint{0.668521in}{2.432624in}}%
\pgfpathlineto{\pgfqpoint{3.331479in}{2.432624in}}%
\pgfusepath{stroke}%
\end{pgfscope}%
\begin{pgfscope}%
\pgfpathrectangle{\pgfqpoint{0.000000in}{0.000000in}}{\pgfqpoint{4.000000in}{4.000000in}}%
\pgfusepath{clip}%
\pgfsetbuttcap%
\pgfsetroundjoin%
\definecolor{currentfill}{rgb}{0.800000,0.800000,0.800000}%
\pgfsetfillcolor{currentfill}%
\pgfsetlinewidth{1.003750pt}%
\definecolor{currentstroke}{rgb}{0.000000,0.000000,0.000000}%
\pgfsetstrokecolor{currentstroke}%
\pgfsetdash{}{0pt}%
\pgfsys@defobject{currentmarker}{\pgfqpoint{-0.104167in}{-0.104167in}}{\pgfqpoint{0.104167in}{0.104167in}}{%
\pgfpathmoveto{\pgfqpoint{0.000000in}{-0.104167in}}%
\pgfpathcurveto{\pgfqpoint{0.027625in}{-0.104167in}}{\pgfqpoint{0.054123in}{-0.093191in}}{\pgfqpoint{0.073657in}{-0.073657in}}%
\pgfpathcurveto{\pgfqpoint{0.093191in}{-0.054123in}}{\pgfqpoint{0.104167in}{-0.027625in}}{\pgfqpoint{0.104167in}{0.000000in}}%
\pgfpathcurveto{\pgfqpoint{0.104167in}{0.027625in}}{\pgfqpoint{0.093191in}{0.054123in}}{\pgfqpoint{0.073657in}{0.073657in}}%
\pgfpathcurveto{\pgfqpoint{0.054123in}{0.093191in}}{\pgfqpoint{0.027625in}{0.104167in}}{\pgfqpoint{0.000000in}{0.104167in}}%
\pgfpathcurveto{\pgfqpoint{-0.027625in}{0.104167in}}{\pgfqpoint{-0.054123in}{0.093191in}}{\pgfqpoint{-0.073657in}{0.073657in}}%
\pgfpathcurveto{\pgfqpoint{-0.093191in}{0.054123in}}{\pgfqpoint{-0.104167in}{0.027625in}}{\pgfqpoint{-0.104167in}{0.000000in}}%
\pgfpathcurveto{\pgfqpoint{-0.104167in}{-0.027625in}}{\pgfqpoint{-0.093191in}{-0.054123in}}{\pgfqpoint{-0.073657in}{-0.073657in}}%
\pgfpathcurveto{\pgfqpoint{-0.054123in}{-0.093191in}}{\pgfqpoint{-0.027625in}{-0.104167in}}{\pgfqpoint{0.000000in}{-0.104167in}}%
\pgfpathclose%
\pgfusepath{stroke,fill}%
}%
\begin{pgfscope}%
\pgfsys@transformshift{1.177101in}{0.867376in}%
\pgfsys@useobject{currentmarker}{}%
\end{pgfscope}%
\end{pgfscope}%
\begin{pgfscope}%
\pgfpathrectangle{\pgfqpoint{0.000000in}{0.000000in}}{\pgfqpoint{4.000000in}{4.000000in}}%
\pgfusepath{clip}%
\pgfsetbuttcap%
\pgfsetroundjoin%
\definecolor{currentfill}{rgb}{0.800000,0.800000,0.800000}%
\pgfsetfillcolor{currentfill}%
\pgfsetlinewidth{1.003750pt}%
\definecolor{currentstroke}{rgb}{0.000000,0.000000,0.000000}%
\pgfsetstrokecolor{currentstroke}%
\pgfsetdash{}{0pt}%
\pgfsys@defobject{currentmarker}{\pgfqpoint{-0.104167in}{-0.104167in}}{\pgfqpoint{0.104167in}{0.104167in}}{%
\pgfpathmoveto{\pgfqpoint{0.000000in}{-0.104167in}}%
\pgfpathcurveto{\pgfqpoint{0.027625in}{-0.104167in}}{\pgfqpoint{0.054123in}{-0.093191in}}{\pgfqpoint{0.073657in}{-0.073657in}}%
\pgfpathcurveto{\pgfqpoint{0.093191in}{-0.054123in}}{\pgfqpoint{0.104167in}{-0.027625in}}{\pgfqpoint{0.104167in}{0.000000in}}%
\pgfpathcurveto{\pgfqpoint{0.104167in}{0.027625in}}{\pgfqpoint{0.093191in}{0.054123in}}{\pgfqpoint{0.073657in}{0.073657in}}%
\pgfpathcurveto{\pgfqpoint{0.054123in}{0.093191in}}{\pgfqpoint{0.027625in}{0.104167in}}{\pgfqpoint{0.000000in}{0.104167in}}%
\pgfpathcurveto{\pgfqpoint{-0.027625in}{0.104167in}}{\pgfqpoint{-0.054123in}{0.093191in}}{\pgfqpoint{-0.073657in}{0.073657in}}%
\pgfpathcurveto{\pgfqpoint{-0.093191in}{0.054123in}}{\pgfqpoint{-0.104167in}{0.027625in}}{\pgfqpoint{-0.104167in}{0.000000in}}%
\pgfpathcurveto{\pgfqpoint{-0.104167in}{-0.027625in}}{\pgfqpoint{-0.093191in}{-0.054123in}}{\pgfqpoint{-0.073657in}{-0.073657in}}%
\pgfpathcurveto{\pgfqpoint{-0.054123in}{-0.093191in}}{\pgfqpoint{-0.027625in}{-0.104167in}}{\pgfqpoint{0.000000in}{-0.104167in}}%
\pgfpathclose%
\pgfusepath{stroke,fill}%
}%
\begin{pgfscope}%
\pgfsys@transformshift{2.822899in}{0.867376in}%
\pgfsys@useobject{currentmarker}{}%
\end{pgfscope}%
\end{pgfscope}%
\begin{pgfscope}%
\pgfpathrectangle{\pgfqpoint{0.000000in}{0.000000in}}{\pgfqpoint{4.000000in}{4.000000in}}%
\pgfusepath{clip}%
\pgfsetbuttcap%
\pgfsetroundjoin%
\definecolor{currentfill}{rgb}{0.800000,0.800000,0.800000}%
\pgfsetfillcolor{currentfill}%
\pgfsetlinewidth{1.003750pt}%
\definecolor{currentstroke}{rgb}{0.000000,0.000000,0.000000}%
\pgfsetstrokecolor{currentstroke}%
\pgfsetdash{}{0pt}%
\pgfsys@defobject{currentmarker}{\pgfqpoint{-0.104167in}{-0.104167in}}{\pgfqpoint{0.104167in}{0.104167in}}{%
\pgfpathmoveto{\pgfqpoint{0.000000in}{-0.104167in}}%
\pgfpathcurveto{\pgfqpoint{0.027625in}{-0.104167in}}{\pgfqpoint{0.054123in}{-0.093191in}}{\pgfqpoint{0.073657in}{-0.073657in}}%
\pgfpathcurveto{\pgfqpoint{0.093191in}{-0.054123in}}{\pgfqpoint{0.104167in}{-0.027625in}}{\pgfqpoint{0.104167in}{0.000000in}}%
\pgfpathcurveto{\pgfqpoint{0.104167in}{0.027625in}}{\pgfqpoint{0.093191in}{0.054123in}}{\pgfqpoint{0.073657in}{0.073657in}}%
\pgfpathcurveto{\pgfqpoint{0.054123in}{0.093191in}}{\pgfqpoint{0.027625in}{0.104167in}}{\pgfqpoint{0.000000in}{0.104167in}}%
\pgfpathcurveto{\pgfqpoint{-0.027625in}{0.104167in}}{\pgfqpoint{-0.054123in}{0.093191in}}{\pgfqpoint{-0.073657in}{0.073657in}}%
\pgfpathcurveto{\pgfqpoint{-0.093191in}{0.054123in}}{\pgfqpoint{-0.104167in}{0.027625in}}{\pgfqpoint{-0.104167in}{0.000000in}}%
\pgfpathcurveto{\pgfqpoint{-0.104167in}{-0.027625in}}{\pgfqpoint{-0.093191in}{-0.054123in}}{\pgfqpoint{-0.073657in}{-0.073657in}}%
\pgfpathcurveto{\pgfqpoint{-0.054123in}{-0.093191in}}{\pgfqpoint{-0.027625in}{-0.104167in}}{\pgfqpoint{0.000000in}{-0.104167in}}%
\pgfpathclose%
\pgfusepath{stroke,fill}%
}%
\begin{pgfscope}%
\pgfsys@transformshift{3.331479in}{2.432624in}%
\pgfsys@useobject{currentmarker}{}%
\end{pgfscope}%
\end{pgfscope}%
\begin{pgfscope}%
\pgfpathrectangle{\pgfqpoint{0.000000in}{0.000000in}}{\pgfqpoint{4.000000in}{4.000000in}}%
\pgfusepath{clip}%
\pgfsetbuttcap%
\pgfsetroundjoin%
\definecolor{currentfill}{rgb}{0.800000,0.800000,0.800000}%
\pgfsetfillcolor{currentfill}%
\pgfsetlinewidth{1.003750pt}%
\definecolor{currentstroke}{rgb}{0.000000,0.000000,0.000000}%
\pgfsetstrokecolor{currentstroke}%
\pgfsetdash{}{0pt}%
\pgfsys@defobject{currentmarker}{\pgfqpoint{-0.104167in}{-0.104167in}}{\pgfqpoint{0.104167in}{0.104167in}}{%
\pgfpathmoveto{\pgfqpoint{0.000000in}{-0.104167in}}%
\pgfpathcurveto{\pgfqpoint{0.027625in}{-0.104167in}}{\pgfqpoint{0.054123in}{-0.093191in}}{\pgfqpoint{0.073657in}{-0.073657in}}%
\pgfpathcurveto{\pgfqpoint{0.093191in}{-0.054123in}}{\pgfqpoint{0.104167in}{-0.027625in}}{\pgfqpoint{0.104167in}{0.000000in}}%
\pgfpathcurveto{\pgfqpoint{0.104167in}{0.027625in}}{\pgfqpoint{0.093191in}{0.054123in}}{\pgfqpoint{0.073657in}{0.073657in}}%
\pgfpathcurveto{\pgfqpoint{0.054123in}{0.093191in}}{\pgfqpoint{0.027625in}{0.104167in}}{\pgfqpoint{0.000000in}{0.104167in}}%
\pgfpathcurveto{\pgfqpoint{-0.027625in}{0.104167in}}{\pgfqpoint{-0.054123in}{0.093191in}}{\pgfqpoint{-0.073657in}{0.073657in}}%
\pgfpathcurveto{\pgfqpoint{-0.093191in}{0.054123in}}{\pgfqpoint{-0.104167in}{0.027625in}}{\pgfqpoint{-0.104167in}{0.000000in}}%
\pgfpathcurveto{\pgfqpoint{-0.104167in}{-0.027625in}}{\pgfqpoint{-0.093191in}{-0.054123in}}{\pgfqpoint{-0.073657in}{-0.073657in}}%
\pgfpathcurveto{\pgfqpoint{-0.054123in}{-0.093191in}}{\pgfqpoint{-0.027625in}{-0.104167in}}{\pgfqpoint{0.000000in}{-0.104167in}}%
\pgfpathclose%
\pgfusepath{stroke,fill}%
}%
\begin{pgfscope}%
\pgfsys@transformshift{2.000000in}{3.400000in}%
\pgfsys@useobject{currentmarker}{}%
\end{pgfscope}%
\end{pgfscope}%
\begin{pgfscope}%
\pgfpathrectangle{\pgfqpoint{0.000000in}{0.000000in}}{\pgfqpoint{4.000000in}{4.000000in}}%
\pgfusepath{clip}%
\pgfsetbuttcap%
\pgfsetroundjoin%
\definecolor{currentfill}{rgb}{0.800000,0.800000,0.800000}%
\pgfsetfillcolor{currentfill}%
\pgfsetlinewidth{1.003750pt}%
\definecolor{currentstroke}{rgb}{0.000000,0.000000,0.000000}%
\pgfsetstrokecolor{currentstroke}%
\pgfsetdash{}{0pt}%
\pgfsys@defobject{currentmarker}{\pgfqpoint{-0.104167in}{-0.104167in}}{\pgfqpoint{0.104167in}{0.104167in}}{%
\pgfpathmoveto{\pgfqpoint{0.000000in}{-0.104167in}}%
\pgfpathcurveto{\pgfqpoint{0.027625in}{-0.104167in}}{\pgfqpoint{0.054123in}{-0.093191in}}{\pgfqpoint{0.073657in}{-0.073657in}}%
\pgfpathcurveto{\pgfqpoint{0.093191in}{-0.054123in}}{\pgfqpoint{0.104167in}{-0.027625in}}{\pgfqpoint{0.104167in}{0.000000in}}%
\pgfpathcurveto{\pgfqpoint{0.104167in}{0.027625in}}{\pgfqpoint{0.093191in}{0.054123in}}{\pgfqpoint{0.073657in}{0.073657in}}%
\pgfpathcurveto{\pgfqpoint{0.054123in}{0.093191in}}{\pgfqpoint{0.027625in}{0.104167in}}{\pgfqpoint{0.000000in}{0.104167in}}%
\pgfpathcurveto{\pgfqpoint{-0.027625in}{0.104167in}}{\pgfqpoint{-0.054123in}{0.093191in}}{\pgfqpoint{-0.073657in}{0.073657in}}%
\pgfpathcurveto{\pgfqpoint{-0.093191in}{0.054123in}}{\pgfqpoint{-0.104167in}{0.027625in}}{\pgfqpoint{-0.104167in}{0.000000in}}%
\pgfpathcurveto{\pgfqpoint{-0.104167in}{-0.027625in}}{\pgfqpoint{-0.093191in}{-0.054123in}}{\pgfqpoint{-0.073657in}{-0.073657in}}%
\pgfpathcurveto{\pgfqpoint{-0.054123in}{-0.093191in}}{\pgfqpoint{-0.027625in}{-0.104167in}}{\pgfqpoint{0.000000in}{-0.104167in}}%
\pgfpathclose%
\pgfusepath{stroke,fill}%
}%
\begin{pgfscope}%
\pgfsys@transformshift{0.668521in}{2.432624in}%
\pgfsys@useobject{currentmarker}{}%
\end{pgfscope}%
\end{pgfscope}%
\begin{pgfscope}%
\definecolor{textcolor}{rgb}{0.000000,0.000000,0.000000}%
\pgfsetstrokecolor{textcolor}%
\pgfsetfillcolor{textcolor}%
\pgftext[x=0.941987in,y=0.543769in,,]{\color{textcolor}\sffamily\fontsize{40.000000}{48.000000}\selectfont \(\displaystyle 1\)}%
\end{pgfscope}%
\begin{pgfscope}%
\definecolor{textcolor}{rgb}{0.000000,0.000000,0.000000}%
\pgfsetstrokecolor{textcolor}%
\pgfsetfillcolor{textcolor}%
\pgftext[x=3.058013in,y=0.543769in,,]{\color{textcolor}\sffamily\fontsize{40.000000}{48.000000}\selectfont \(\displaystyle 2\)}%
\end{pgfscope}%
\begin{pgfscope}%
\definecolor{textcolor}{rgb}{0.000000,0.000000,0.000000}%
\pgfsetstrokecolor{textcolor}%
\pgfsetfillcolor{textcolor}%
\pgftext[x=3.711902in,y=2.556231in,,]{\color{textcolor}\sffamily\fontsize{40.000000}{48.000000}\selectfont \(\displaystyle 3\)}%
\end{pgfscope}%
\begin{pgfscope}%
\definecolor{textcolor}{rgb}{0.000000,0.000000,0.000000}%
\pgfsetstrokecolor{textcolor}%
\pgfsetfillcolor{textcolor}%
\pgftext[x=2.000000in,y=3.800000in,,]{\color{textcolor}\sffamily\fontsize{40.000000}{48.000000}\selectfont \(\displaystyle 4\)}%
\end{pgfscope}%
\begin{pgfscope}%
\definecolor{textcolor}{rgb}{0.000000,0.000000,0.000000}%
\pgfsetstrokecolor{textcolor}%
\pgfsetfillcolor{textcolor}%
\pgftext[x=0.288098in,y=2.556231in,,]{\color{textcolor}\sffamily\fontsize{40.000000}{48.000000}\selectfont \(\displaystyle 5\)}%
\end{pgfscope}%
\end{pgfpicture}%
\makeatother%
\endgroup%

%% file: fig/decomposition_conic_counterexample1_spanning_tree.pgf
%% Creator: Matplotlib, PGF backend
%%
%% To include the figure in your LaTeX document, write
%%   \input{<filename>.pgf}
%%
%% Make sure the required packages are loaded in your preamble
%%   \usepackage{pgf}
%%
%% Figures using additional raster images can only be included by \input if
%% they are in the same directory as the main LaTeX file. For loading figures
%% from other directories you can use the `import` package
%%   \usepackage{import}
%%
%% and then include the figures with
%%   \import{<path to file>}{<filename>.pgf}
%%
%% Matplotlib used the following preamble
%%   \usepackage{fontspec}
%%   \setmainfont{DejaVuSerif.ttf}[Path=\detokenize{C:/Users/ccros/Anaconda3/Lib/site-packages/matplotlib/mpl-data/fonts/ttf/}]
%%   \setsansfont{DejaVuSans.ttf}[Path=\detokenize{C:/Users/ccros/Anaconda3/Lib/site-packages/matplotlib/mpl-data/fonts/ttf/}]
%%   \setmonofont{DejaVuSansMono.ttf}[Path=\detokenize{C:/Users/ccros/Anaconda3/Lib/site-packages/matplotlib/mpl-data/fonts/ttf/}]
%%
\begingroup%
\makeatletter%
\begin{pgfpicture}%
\pgfpathrectangle{\pgfpointorigin}{\pgfqpoint{4.000000in}{4.068826in}}%
\pgfusepath{use as bounding box, clip}%
\begin{pgfscope}%
\pgfsetbuttcap%
\pgfsetmiterjoin%
\pgfsetlinewidth{0.000000pt}%
\definecolor{currentstroke}{rgb}{1.000000,1.000000,1.000000}%
\pgfsetstrokecolor{currentstroke}%
\pgfsetstrokeopacity{0.000000}%
\pgfsetdash{}{0pt}%
\pgfpathmoveto{\pgfqpoint{0.000000in}{0.000000in}}%
\pgfpathlineto{\pgfqpoint{4.000000in}{0.000000in}}%
\pgfpathlineto{\pgfqpoint{4.000000in}{4.068826in}}%
\pgfpathlineto{\pgfqpoint{0.000000in}{4.068826in}}%
\pgfpathclose%
\pgfusepath{}%
\end{pgfscope}%
\begin{pgfscope}%
\pgfpathrectangle{\pgfqpoint{0.000000in}{0.000000in}}{\pgfqpoint{4.000000in}{4.000000in}}%
\pgfusepath{clip}%
\pgfsetrectcap%
\pgfsetroundjoin%
\pgfsetlinewidth{3.011250pt}%
\definecolor{currentstroke}{rgb}{0.250980,0.250980,0.250980}%
\pgfsetstrokecolor{currentstroke}%
\pgfsetdash{}{0pt}%
\pgfpathmoveto{\pgfqpoint{3.331479in}{2.432624in}}%
\pgfpathlineto{\pgfqpoint{2.822899in}{0.867376in}}%
\pgfusepath{stroke}%
\end{pgfscope}%
\begin{pgfscope}%
\pgfpathrectangle{\pgfqpoint{0.000000in}{0.000000in}}{\pgfqpoint{4.000000in}{4.000000in}}%
\pgfusepath{clip}%
\pgfsetrectcap%
\pgfsetroundjoin%
\pgfsetlinewidth{3.011250pt}%
\definecolor{currentstroke}{rgb}{0.250980,0.250980,0.250980}%
\pgfsetstrokecolor{currentstroke}%
\pgfsetdash{}{0pt}%
\pgfpathmoveto{\pgfqpoint{2.000000in}{3.400000in}}%
\pgfpathlineto{\pgfqpoint{1.177101in}{0.867376in}}%
\pgfusepath{stroke}%
\end{pgfscope}%
\begin{pgfscope}%
\pgfpathrectangle{\pgfqpoint{0.000000in}{0.000000in}}{\pgfqpoint{4.000000in}{4.000000in}}%
\pgfusepath{clip}%
\pgfsetrectcap%
\pgfsetroundjoin%
\pgfsetlinewidth{3.011250pt}%
\definecolor{currentstroke}{rgb}{0.250980,0.250980,0.250980}%
\pgfsetstrokecolor{currentstroke}%
\pgfsetdash{}{0pt}%
\pgfpathmoveto{\pgfqpoint{2.000000in}{3.400000in}}%
\pgfpathlineto{\pgfqpoint{2.822899in}{0.867376in}}%
\pgfusepath{stroke}%
\end{pgfscope}%
\begin{pgfscope}%
\pgfpathrectangle{\pgfqpoint{0.000000in}{0.000000in}}{\pgfqpoint{4.000000in}{4.000000in}}%
\pgfusepath{clip}%
\pgfsetrectcap%
\pgfsetroundjoin%
\pgfsetlinewidth{3.011250pt}%
\definecolor{currentstroke}{rgb}{0.250980,0.250980,0.250980}%
\pgfsetstrokecolor{currentstroke}%
\pgfsetdash{}{0pt}%
\pgfpathmoveto{\pgfqpoint{2.000000in}{3.400000in}}%
\pgfpathlineto{\pgfqpoint{3.331479in}{2.432624in}}%
\pgfusepath{stroke}%
\end{pgfscope}%
\begin{pgfscope}%
\pgfpathrectangle{\pgfqpoint{0.000000in}{0.000000in}}{\pgfqpoint{4.000000in}{4.000000in}}%
\pgfusepath{clip}%
\pgfsetbuttcap%
\pgfsetroundjoin%
\definecolor{currentfill}{rgb}{0.800000,0.800000,0.800000}%
\pgfsetfillcolor{currentfill}%
\pgfsetlinewidth{1.003750pt}%
\definecolor{currentstroke}{rgb}{0.000000,0.000000,0.000000}%
\pgfsetstrokecolor{currentstroke}%
\pgfsetdash{}{0pt}%
\pgfsys@defobject{currentmarker}{\pgfqpoint{-0.104167in}{-0.104167in}}{\pgfqpoint{0.104167in}{0.104167in}}{%
\pgfpathmoveto{\pgfqpoint{0.000000in}{-0.104167in}}%
\pgfpathcurveto{\pgfqpoint{0.027625in}{-0.104167in}}{\pgfqpoint{0.054123in}{-0.093191in}}{\pgfqpoint{0.073657in}{-0.073657in}}%
\pgfpathcurveto{\pgfqpoint{0.093191in}{-0.054123in}}{\pgfqpoint{0.104167in}{-0.027625in}}{\pgfqpoint{0.104167in}{0.000000in}}%
\pgfpathcurveto{\pgfqpoint{0.104167in}{0.027625in}}{\pgfqpoint{0.093191in}{0.054123in}}{\pgfqpoint{0.073657in}{0.073657in}}%
\pgfpathcurveto{\pgfqpoint{0.054123in}{0.093191in}}{\pgfqpoint{0.027625in}{0.104167in}}{\pgfqpoint{0.000000in}{0.104167in}}%
\pgfpathcurveto{\pgfqpoint{-0.027625in}{0.104167in}}{\pgfqpoint{-0.054123in}{0.093191in}}{\pgfqpoint{-0.073657in}{0.073657in}}%
\pgfpathcurveto{\pgfqpoint{-0.093191in}{0.054123in}}{\pgfqpoint{-0.104167in}{0.027625in}}{\pgfqpoint{-0.104167in}{0.000000in}}%
\pgfpathcurveto{\pgfqpoint{-0.104167in}{-0.027625in}}{\pgfqpoint{-0.093191in}{-0.054123in}}{\pgfqpoint{-0.073657in}{-0.073657in}}%
\pgfpathcurveto{\pgfqpoint{-0.054123in}{-0.093191in}}{\pgfqpoint{-0.027625in}{-0.104167in}}{\pgfqpoint{0.000000in}{-0.104167in}}%
\pgfpathclose%
\pgfusepath{stroke,fill}%
}%
\begin{pgfscope}%
\pgfsys@transformshift{1.177101in}{0.867376in}%
\pgfsys@useobject{currentmarker}{}%
\end{pgfscope}%
\end{pgfscope}%
\begin{pgfscope}%
\pgfpathrectangle{\pgfqpoint{0.000000in}{0.000000in}}{\pgfqpoint{4.000000in}{4.000000in}}%
\pgfusepath{clip}%
\pgfsetbuttcap%
\pgfsetroundjoin%
\definecolor{currentfill}{rgb}{0.800000,0.800000,0.800000}%
\pgfsetfillcolor{currentfill}%
\pgfsetlinewidth{1.003750pt}%
\definecolor{currentstroke}{rgb}{0.000000,0.000000,0.000000}%
\pgfsetstrokecolor{currentstroke}%
\pgfsetdash{}{0pt}%
\pgfsys@defobject{currentmarker}{\pgfqpoint{-0.104167in}{-0.104167in}}{\pgfqpoint{0.104167in}{0.104167in}}{%
\pgfpathmoveto{\pgfqpoint{0.000000in}{-0.104167in}}%
\pgfpathcurveto{\pgfqpoint{0.027625in}{-0.104167in}}{\pgfqpoint{0.054123in}{-0.093191in}}{\pgfqpoint{0.073657in}{-0.073657in}}%
\pgfpathcurveto{\pgfqpoint{0.093191in}{-0.054123in}}{\pgfqpoint{0.104167in}{-0.027625in}}{\pgfqpoint{0.104167in}{0.000000in}}%
\pgfpathcurveto{\pgfqpoint{0.104167in}{0.027625in}}{\pgfqpoint{0.093191in}{0.054123in}}{\pgfqpoint{0.073657in}{0.073657in}}%
\pgfpathcurveto{\pgfqpoint{0.054123in}{0.093191in}}{\pgfqpoint{0.027625in}{0.104167in}}{\pgfqpoint{0.000000in}{0.104167in}}%
\pgfpathcurveto{\pgfqpoint{-0.027625in}{0.104167in}}{\pgfqpoint{-0.054123in}{0.093191in}}{\pgfqpoint{-0.073657in}{0.073657in}}%
\pgfpathcurveto{\pgfqpoint{-0.093191in}{0.054123in}}{\pgfqpoint{-0.104167in}{0.027625in}}{\pgfqpoint{-0.104167in}{0.000000in}}%
\pgfpathcurveto{\pgfqpoint{-0.104167in}{-0.027625in}}{\pgfqpoint{-0.093191in}{-0.054123in}}{\pgfqpoint{-0.073657in}{-0.073657in}}%
\pgfpathcurveto{\pgfqpoint{-0.054123in}{-0.093191in}}{\pgfqpoint{-0.027625in}{-0.104167in}}{\pgfqpoint{0.000000in}{-0.104167in}}%
\pgfpathclose%
\pgfusepath{stroke,fill}%
}%
\begin{pgfscope}%
\pgfsys@transformshift{2.822899in}{0.867376in}%
\pgfsys@useobject{currentmarker}{}%
\end{pgfscope}%
\end{pgfscope}%
\begin{pgfscope}%
\pgfpathrectangle{\pgfqpoint{0.000000in}{0.000000in}}{\pgfqpoint{4.000000in}{4.000000in}}%
\pgfusepath{clip}%
\pgfsetbuttcap%
\pgfsetroundjoin%
\definecolor{currentfill}{rgb}{0.800000,0.800000,0.800000}%
\pgfsetfillcolor{currentfill}%
\pgfsetlinewidth{1.003750pt}%
\definecolor{currentstroke}{rgb}{0.000000,0.000000,0.000000}%
\pgfsetstrokecolor{currentstroke}%
\pgfsetdash{}{0pt}%
\pgfsys@defobject{currentmarker}{\pgfqpoint{-0.104167in}{-0.104167in}}{\pgfqpoint{0.104167in}{0.104167in}}{%
\pgfpathmoveto{\pgfqpoint{0.000000in}{-0.104167in}}%
\pgfpathcurveto{\pgfqpoint{0.027625in}{-0.104167in}}{\pgfqpoint{0.054123in}{-0.093191in}}{\pgfqpoint{0.073657in}{-0.073657in}}%
\pgfpathcurveto{\pgfqpoint{0.093191in}{-0.054123in}}{\pgfqpoint{0.104167in}{-0.027625in}}{\pgfqpoint{0.104167in}{0.000000in}}%
\pgfpathcurveto{\pgfqpoint{0.104167in}{0.027625in}}{\pgfqpoint{0.093191in}{0.054123in}}{\pgfqpoint{0.073657in}{0.073657in}}%
\pgfpathcurveto{\pgfqpoint{0.054123in}{0.093191in}}{\pgfqpoint{0.027625in}{0.104167in}}{\pgfqpoint{0.000000in}{0.104167in}}%
\pgfpathcurveto{\pgfqpoint{-0.027625in}{0.104167in}}{\pgfqpoint{-0.054123in}{0.093191in}}{\pgfqpoint{-0.073657in}{0.073657in}}%
\pgfpathcurveto{\pgfqpoint{-0.093191in}{0.054123in}}{\pgfqpoint{-0.104167in}{0.027625in}}{\pgfqpoint{-0.104167in}{0.000000in}}%
\pgfpathcurveto{\pgfqpoint{-0.104167in}{-0.027625in}}{\pgfqpoint{-0.093191in}{-0.054123in}}{\pgfqpoint{-0.073657in}{-0.073657in}}%
\pgfpathcurveto{\pgfqpoint{-0.054123in}{-0.093191in}}{\pgfqpoint{-0.027625in}{-0.104167in}}{\pgfqpoint{0.000000in}{-0.104167in}}%
\pgfpathclose%
\pgfusepath{stroke,fill}%
}%
\begin{pgfscope}%
\pgfsys@transformshift{3.331479in}{2.432624in}%
\pgfsys@useobject{currentmarker}{}%
\end{pgfscope}%
\end{pgfscope}%
\begin{pgfscope}%
\pgfpathrectangle{\pgfqpoint{0.000000in}{0.000000in}}{\pgfqpoint{4.000000in}{4.000000in}}%
\pgfusepath{clip}%
\pgfsetbuttcap%
\pgfsetroundjoin%
\definecolor{currentfill}{rgb}{0.800000,0.800000,0.800000}%
\pgfsetfillcolor{currentfill}%
\pgfsetlinewidth{1.003750pt}%
\definecolor{currentstroke}{rgb}{0.000000,0.000000,0.000000}%
\pgfsetstrokecolor{currentstroke}%
\pgfsetdash{}{0pt}%
\pgfsys@defobject{currentmarker}{\pgfqpoint{-0.104167in}{-0.104167in}}{\pgfqpoint{0.104167in}{0.104167in}}{%
\pgfpathmoveto{\pgfqpoint{0.000000in}{-0.104167in}}%
\pgfpathcurveto{\pgfqpoint{0.027625in}{-0.104167in}}{\pgfqpoint{0.054123in}{-0.093191in}}{\pgfqpoint{0.073657in}{-0.073657in}}%
\pgfpathcurveto{\pgfqpoint{0.093191in}{-0.054123in}}{\pgfqpoint{0.104167in}{-0.027625in}}{\pgfqpoint{0.104167in}{0.000000in}}%
\pgfpathcurveto{\pgfqpoint{0.104167in}{0.027625in}}{\pgfqpoint{0.093191in}{0.054123in}}{\pgfqpoint{0.073657in}{0.073657in}}%
\pgfpathcurveto{\pgfqpoint{0.054123in}{0.093191in}}{\pgfqpoint{0.027625in}{0.104167in}}{\pgfqpoint{0.000000in}{0.104167in}}%
\pgfpathcurveto{\pgfqpoint{-0.027625in}{0.104167in}}{\pgfqpoint{-0.054123in}{0.093191in}}{\pgfqpoint{-0.073657in}{0.073657in}}%
\pgfpathcurveto{\pgfqpoint{-0.093191in}{0.054123in}}{\pgfqpoint{-0.104167in}{0.027625in}}{\pgfqpoint{-0.104167in}{0.000000in}}%
\pgfpathcurveto{\pgfqpoint{-0.104167in}{-0.027625in}}{\pgfqpoint{-0.093191in}{-0.054123in}}{\pgfqpoint{-0.073657in}{-0.073657in}}%
\pgfpathcurveto{\pgfqpoint{-0.054123in}{-0.093191in}}{\pgfqpoint{-0.027625in}{-0.104167in}}{\pgfqpoint{0.000000in}{-0.104167in}}%
\pgfpathclose%
\pgfusepath{stroke,fill}%
}%
\begin{pgfscope}%
\pgfsys@transformshift{2.000000in}{3.400000in}%
\pgfsys@useobject{currentmarker}{}%
\end{pgfscope}%
\end{pgfscope}%
\begin{pgfscope}%
\pgfpathrectangle{\pgfqpoint{0.000000in}{0.000000in}}{\pgfqpoint{4.000000in}{4.000000in}}%
\pgfusepath{clip}%
\pgfsetbuttcap%
\pgfsetroundjoin%
\definecolor{currentfill}{rgb}{0.800000,0.800000,0.800000}%
\pgfsetfillcolor{currentfill}%
\pgfsetlinewidth{1.003750pt}%
\definecolor{currentstroke}{rgb}{0.000000,0.000000,0.000000}%
\pgfsetstrokecolor{currentstroke}%
\pgfsetdash{}{0pt}%
\pgfsys@defobject{currentmarker}{\pgfqpoint{-0.104167in}{-0.104167in}}{\pgfqpoint{0.104167in}{0.104167in}}{%
\pgfpathmoveto{\pgfqpoint{0.000000in}{-0.104167in}}%
\pgfpathcurveto{\pgfqpoint{0.027625in}{-0.104167in}}{\pgfqpoint{0.054123in}{-0.093191in}}{\pgfqpoint{0.073657in}{-0.073657in}}%
\pgfpathcurveto{\pgfqpoint{0.093191in}{-0.054123in}}{\pgfqpoint{0.104167in}{-0.027625in}}{\pgfqpoint{0.104167in}{0.000000in}}%
\pgfpathcurveto{\pgfqpoint{0.104167in}{0.027625in}}{\pgfqpoint{0.093191in}{0.054123in}}{\pgfqpoint{0.073657in}{0.073657in}}%
\pgfpathcurveto{\pgfqpoint{0.054123in}{0.093191in}}{\pgfqpoint{0.027625in}{0.104167in}}{\pgfqpoint{0.000000in}{0.104167in}}%
\pgfpathcurveto{\pgfqpoint{-0.027625in}{0.104167in}}{\pgfqpoint{-0.054123in}{0.093191in}}{\pgfqpoint{-0.073657in}{0.073657in}}%
\pgfpathcurveto{\pgfqpoint{-0.093191in}{0.054123in}}{\pgfqpoint{-0.104167in}{0.027625in}}{\pgfqpoint{-0.104167in}{0.000000in}}%
\pgfpathcurveto{\pgfqpoint{-0.104167in}{-0.027625in}}{\pgfqpoint{-0.093191in}{-0.054123in}}{\pgfqpoint{-0.073657in}{-0.073657in}}%
\pgfpathcurveto{\pgfqpoint{-0.054123in}{-0.093191in}}{\pgfqpoint{-0.027625in}{-0.104167in}}{\pgfqpoint{0.000000in}{-0.104167in}}%
\pgfpathclose%
\pgfusepath{stroke,fill}%
}%
\begin{pgfscope}%
\pgfsys@transformshift{0.668521in}{2.432624in}%
\pgfsys@useobject{currentmarker}{}%
\end{pgfscope}%
\end{pgfscope}%
\begin{pgfscope}%
\definecolor{textcolor}{rgb}{0.000000,0.000000,0.000000}%
\pgfsetstrokecolor{textcolor}%
\pgfsetfillcolor{textcolor}%
\pgftext[x=0.941987in,y=0.543769in,,]{\color{textcolor}\sffamily\fontsize{40.000000}{48.000000}\selectfont \(\displaystyle 1\)}%
\end{pgfscope}%
\begin{pgfscope}%
\definecolor{textcolor}{rgb}{0.000000,0.000000,0.000000}%
\pgfsetstrokecolor{textcolor}%
\pgfsetfillcolor{textcolor}%
\pgftext[x=3.058013in,y=0.543769in,,]{\color{textcolor}\sffamily\fontsize{40.000000}{48.000000}\selectfont \(\displaystyle 2\)}%
\end{pgfscope}%
\begin{pgfscope}%
\definecolor{textcolor}{rgb}{0.000000,0.000000,0.000000}%
\pgfsetstrokecolor{textcolor}%
\pgfsetfillcolor{textcolor}%
\pgftext[x=3.711902in,y=2.556231in,,]{\color{textcolor}\sffamily\fontsize{40.000000}{48.000000}\selectfont \(\displaystyle 3\)}%
\end{pgfscope}%
\begin{pgfscope}%
\definecolor{textcolor}{rgb}{0.000000,0.000000,0.000000}%
\pgfsetstrokecolor{textcolor}%
\pgfsetfillcolor{textcolor}%
\pgftext[x=2.000000in,y=3.800000in,,]{\color{textcolor}\sffamily\fontsize{40.000000}{48.000000}\selectfont \(\displaystyle 4\)}%
\end{pgfscope}%
\begin{pgfscope}%
\definecolor{textcolor}{rgb}{0.000000,0.000000,0.000000}%
\pgfsetstrokecolor{textcolor}%
\pgfsetfillcolor{textcolor}%
\pgftext[x=0.288098in,y=2.556231in,,]{\color{textcolor}\sffamily\fontsize{40.000000}{48.000000}\selectfont \(\displaystyle 5\)}%
\end{pgfscope}%
\end{pgfpicture}%
\makeatother%
\endgroup%

%% file: fig/decomposition_conic_counterexample2_spanning_tree.pgf
%% Creator: Matplotlib, PGF backend
%%
%% To include the figure in your LaTeX document, write
%%   \input{<filename>.pgf}
%%
%% Make sure the required packages are loaded in your preamble
%%   \usepackage{pgf}
%%
%% Figures using additional raster images can only be included by \input if
%% they are in the same directory as the main LaTeX file. For loading figures
%% from other directories you can use the `import` package
%%   \usepackage{import}
%%
%% and then include the figures with
%%   \import{<path to file>}{<filename>.pgf}
%%
%% Matplotlib used the following preamble
%%   \usepackage{fontspec}
%%   \setmainfont{DejaVuSerif.ttf}[Path=\detokenize{C:/Users/ccros/Anaconda3/Lib/site-packages/matplotlib/mpl-data/fonts/ttf/}]
%%   \setsansfont{DejaVuSans.ttf}[Path=\detokenize{C:/Users/ccros/Anaconda3/Lib/site-packages/matplotlib/mpl-data/fonts/ttf/}]
%%   \setmonofont{DejaVuSansMono.ttf}[Path=\detokenize{C:/Users/ccros/Anaconda3/Lib/site-packages/matplotlib/mpl-data/fonts/ttf/}]
%%
\begingroup%
\makeatletter%
\begin{pgfpicture}%
\pgfpathrectangle{\pgfpointorigin}{\pgfqpoint{4.000000in}{4.068826in}}%
\pgfusepath{use as bounding box, clip}%
\begin{pgfscope}%
\pgfsetbuttcap%
\pgfsetmiterjoin%
\pgfsetlinewidth{0.000000pt}%
\definecolor{currentstroke}{rgb}{1.000000,1.000000,1.000000}%
\pgfsetstrokecolor{currentstroke}%
\pgfsetstrokeopacity{0.000000}%
\pgfsetdash{}{0pt}%
\pgfpathmoveto{\pgfqpoint{0.000000in}{0.000000in}}%
\pgfpathlineto{\pgfqpoint{4.000000in}{0.000000in}}%
\pgfpathlineto{\pgfqpoint{4.000000in}{4.068826in}}%
\pgfpathlineto{\pgfqpoint{0.000000in}{4.068826in}}%
\pgfpathclose%
\pgfusepath{}%
\end{pgfscope}%
\begin{pgfscope}%
\pgfpathrectangle{\pgfqpoint{0.000000in}{0.000000in}}{\pgfqpoint{4.000000in}{4.000000in}}%
\pgfusepath{clip}%
\pgfsetrectcap%
\pgfsetroundjoin%
\pgfsetlinewidth{3.011250pt}%
\definecolor{currentstroke}{rgb}{0.250980,0.250980,0.250980}%
\pgfsetstrokecolor{currentstroke}%
\pgfsetdash{}{0pt}%
\pgfpathmoveto{\pgfqpoint{2.000000in}{3.400000in}}%
\pgfpathlineto{\pgfqpoint{1.177101in}{0.867376in}}%
\pgfusepath{stroke}%
\end{pgfscope}%
\begin{pgfscope}%
\pgfpathrectangle{\pgfqpoint{0.000000in}{0.000000in}}{\pgfqpoint{4.000000in}{4.000000in}}%
\pgfusepath{clip}%
\pgfsetrectcap%
\pgfsetroundjoin%
\pgfsetlinewidth{3.011250pt}%
\definecolor{currentstroke}{rgb}{0.250980,0.250980,0.250980}%
\pgfsetstrokecolor{currentstroke}%
\pgfsetdash{}{0pt}%
\pgfpathmoveto{\pgfqpoint{2.000000in}{3.400000in}}%
\pgfpathlineto{\pgfqpoint{2.822899in}{0.867376in}}%
\pgfusepath{stroke}%
\end{pgfscope}%
\begin{pgfscope}%
\pgfpathrectangle{\pgfqpoint{0.000000in}{0.000000in}}{\pgfqpoint{4.000000in}{4.000000in}}%
\pgfusepath{clip}%
\pgfsetrectcap%
\pgfsetroundjoin%
\pgfsetlinewidth{3.011250pt}%
\definecolor{currentstroke}{rgb}{0.250980,0.250980,0.250980}%
\pgfsetstrokecolor{currentstroke}%
\pgfsetdash{}{0pt}%
\pgfpathmoveto{\pgfqpoint{2.000000in}{3.400000in}}%
\pgfpathlineto{\pgfqpoint{3.331479in}{2.432624in}}%
\pgfusepath{stroke}%
\end{pgfscope}%
\begin{pgfscope}%
\pgfpathrectangle{\pgfqpoint{0.000000in}{0.000000in}}{\pgfqpoint{4.000000in}{4.000000in}}%
\pgfusepath{clip}%
\pgfsetrectcap%
\pgfsetroundjoin%
\pgfsetlinewidth{3.011250pt}%
\definecolor{currentstroke}{rgb}{0.250980,0.250980,0.250980}%
\pgfsetstrokecolor{currentstroke}%
\pgfsetdash{}{0pt}%
\pgfpathmoveto{\pgfqpoint{0.668521in}{2.432624in}}%
\pgfpathlineto{\pgfqpoint{1.177101in}{0.867376in}}%
\pgfusepath{stroke}%
\end{pgfscope}%
\begin{pgfscope}%
\pgfpathrectangle{\pgfqpoint{0.000000in}{0.000000in}}{\pgfqpoint{4.000000in}{4.000000in}}%
\pgfusepath{clip}%
\pgfsetbuttcap%
\pgfsetroundjoin%
\definecolor{currentfill}{rgb}{0.800000,0.800000,0.800000}%
\pgfsetfillcolor{currentfill}%
\pgfsetlinewidth{1.003750pt}%
\definecolor{currentstroke}{rgb}{0.000000,0.000000,0.000000}%
\pgfsetstrokecolor{currentstroke}%
\pgfsetdash{}{0pt}%
\pgfsys@defobject{currentmarker}{\pgfqpoint{-0.104167in}{-0.104167in}}{\pgfqpoint{0.104167in}{0.104167in}}{%
\pgfpathmoveto{\pgfqpoint{0.000000in}{-0.104167in}}%
\pgfpathcurveto{\pgfqpoint{0.027625in}{-0.104167in}}{\pgfqpoint{0.054123in}{-0.093191in}}{\pgfqpoint{0.073657in}{-0.073657in}}%
\pgfpathcurveto{\pgfqpoint{0.093191in}{-0.054123in}}{\pgfqpoint{0.104167in}{-0.027625in}}{\pgfqpoint{0.104167in}{0.000000in}}%
\pgfpathcurveto{\pgfqpoint{0.104167in}{0.027625in}}{\pgfqpoint{0.093191in}{0.054123in}}{\pgfqpoint{0.073657in}{0.073657in}}%
\pgfpathcurveto{\pgfqpoint{0.054123in}{0.093191in}}{\pgfqpoint{0.027625in}{0.104167in}}{\pgfqpoint{0.000000in}{0.104167in}}%
\pgfpathcurveto{\pgfqpoint{-0.027625in}{0.104167in}}{\pgfqpoint{-0.054123in}{0.093191in}}{\pgfqpoint{-0.073657in}{0.073657in}}%
\pgfpathcurveto{\pgfqpoint{-0.093191in}{0.054123in}}{\pgfqpoint{-0.104167in}{0.027625in}}{\pgfqpoint{-0.104167in}{0.000000in}}%
\pgfpathcurveto{\pgfqpoint{-0.104167in}{-0.027625in}}{\pgfqpoint{-0.093191in}{-0.054123in}}{\pgfqpoint{-0.073657in}{-0.073657in}}%
\pgfpathcurveto{\pgfqpoint{-0.054123in}{-0.093191in}}{\pgfqpoint{-0.027625in}{-0.104167in}}{\pgfqpoint{0.000000in}{-0.104167in}}%
\pgfpathclose%
\pgfusepath{stroke,fill}%
}%
\begin{pgfscope}%
\pgfsys@transformshift{1.177101in}{0.867376in}%
\pgfsys@useobject{currentmarker}{}%
\end{pgfscope}%
\end{pgfscope}%
\begin{pgfscope}%
\pgfpathrectangle{\pgfqpoint{0.000000in}{0.000000in}}{\pgfqpoint{4.000000in}{4.000000in}}%
\pgfusepath{clip}%
\pgfsetbuttcap%
\pgfsetroundjoin%
\definecolor{currentfill}{rgb}{0.800000,0.800000,0.800000}%
\pgfsetfillcolor{currentfill}%
\pgfsetlinewidth{1.003750pt}%
\definecolor{currentstroke}{rgb}{0.000000,0.000000,0.000000}%
\pgfsetstrokecolor{currentstroke}%
\pgfsetdash{}{0pt}%
\pgfsys@defobject{currentmarker}{\pgfqpoint{-0.104167in}{-0.104167in}}{\pgfqpoint{0.104167in}{0.104167in}}{%
\pgfpathmoveto{\pgfqpoint{0.000000in}{-0.104167in}}%
\pgfpathcurveto{\pgfqpoint{0.027625in}{-0.104167in}}{\pgfqpoint{0.054123in}{-0.093191in}}{\pgfqpoint{0.073657in}{-0.073657in}}%
\pgfpathcurveto{\pgfqpoint{0.093191in}{-0.054123in}}{\pgfqpoint{0.104167in}{-0.027625in}}{\pgfqpoint{0.104167in}{0.000000in}}%
\pgfpathcurveto{\pgfqpoint{0.104167in}{0.027625in}}{\pgfqpoint{0.093191in}{0.054123in}}{\pgfqpoint{0.073657in}{0.073657in}}%
\pgfpathcurveto{\pgfqpoint{0.054123in}{0.093191in}}{\pgfqpoint{0.027625in}{0.104167in}}{\pgfqpoint{0.000000in}{0.104167in}}%
\pgfpathcurveto{\pgfqpoint{-0.027625in}{0.104167in}}{\pgfqpoint{-0.054123in}{0.093191in}}{\pgfqpoint{-0.073657in}{0.073657in}}%
\pgfpathcurveto{\pgfqpoint{-0.093191in}{0.054123in}}{\pgfqpoint{-0.104167in}{0.027625in}}{\pgfqpoint{-0.104167in}{0.000000in}}%
\pgfpathcurveto{\pgfqpoint{-0.104167in}{-0.027625in}}{\pgfqpoint{-0.093191in}{-0.054123in}}{\pgfqpoint{-0.073657in}{-0.073657in}}%
\pgfpathcurveto{\pgfqpoint{-0.054123in}{-0.093191in}}{\pgfqpoint{-0.027625in}{-0.104167in}}{\pgfqpoint{0.000000in}{-0.104167in}}%
\pgfpathclose%
\pgfusepath{stroke,fill}%
}%
\begin{pgfscope}%
\pgfsys@transformshift{2.822899in}{0.867376in}%
\pgfsys@useobject{currentmarker}{}%
\end{pgfscope}%
\end{pgfscope}%
\begin{pgfscope}%
\pgfpathrectangle{\pgfqpoint{0.000000in}{0.000000in}}{\pgfqpoint{4.000000in}{4.000000in}}%
\pgfusepath{clip}%
\pgfsetbuttcap%
\pgfsetroundjoin%
\definecolor{currentfill}{rgb}{0.800000,0.800000,0.800000}%
\pgfsetfillcolor{currentfill}%
\pgfsetlinewidth{1.003750pt}%
\definecolor{currentstroke}{rgb}{0.000000,0.000000,0.000000}%
\pgfsetstrokecolor{currentstroke}%
\pgfsetdash{}{0pt}%
\pgfsys@defobject{currentmarker}{\pgfqpoint{-0.104167in}{-0.104167in}}{\pgfqpoint{0.104167in}{0.104167in}}{%
\pgfpathmoveto{\pgfqpoint{0.000000in}{-0.104167in}}%
\pgfpathcurveto{\pgfqpoint{0.027625in}{-0.104167in}}{\pgfqpoint{0.054123in}{-0.093191in}}{\pgfqpoint{0.073657in}{-0.073657in}}%
\pgfpathcurveto{\pgfqpoint{0.093191in}{-0.054123in}}{\pgfqpoint{0.104167in}{-0.027625in}}{\pgfqpoint{0.104167in}{0.000000in}}%
\pgfpathcurveto{\pgfqpoint{0.104167in}{0.027625in}}{\pgfqpoint{0.093191in}{0.054123in}}{\pgfqpoint{0.073657in}{0.073657in}}%
\pgfpathcurveto{\pgfqpoint{0.054123in}{0.093191in}}{\pgfqpoint{0.027625in}{0.104167in}}{\pgfqpoint{0.000000in}{0.104167in}}%
\pgfpathcurveto{\pgfqpoint{-0.027625in}{0.104167in}}{\pgfqpoint{-0.054123in}{0.093191in}}{\pgfqpoint{-0.073657in}{0.073657in}}%
\pgfpathcurveto{\pgfqpoint{-0.093191in}{0.054123in}}{\pgfqpoint{-0.104167in}{0.027625in}}{\pgfqpoint{-0.104167in}{0.000000in}}%
\pgfpathcurveto{\pgfqpoint{-0.104167in}{-0.027625in}}{\pgfqpoint{-0.093191in}{-0.054123in}}{\pgfqpoint{-0.073657in}{-0.073657in}}%
\pgfpathcurveto{\pgfqpoint{-0.054123in}{-0.093191in}}{\pgfqpoint{-0.027625in}{-0.104167in}}{\pgfqpoint{0.000000in}{-0.104167in}}%
\pgfpathclose%
\pgfusepath{stroke,fill}%
}%
\begin{pgfscope}%
\pgfsys@transformshift{3.331479in}{2.432624in}%
\pgfsys@useobject{currentmarker}{}%
\end{pgfscope}%
\end{pgfscope}%
\begin{pgfscope}%
\pgfpathrectangle{\pgfqpoint{0.000000in}{0.000000in}}{\pgfqpoint{4.000000in}{4.000000in}}%
\pgfusepath{clip}%
\pgfsetbuttcap%
\pgfsetroundjoin%
\definecolor{currentfill}{rgb}{0.800000,0.800000,0.800000}%
\pgfsetfillcolor{currentfill}%
\pgfsetlinewidth{1.003750pt}%
\definecolor{currentstroke}{rgb}{0.000000,0.000000,0.000000}%
\pgfsetstrokecolor{currentstroke}%
\pgfsetdash{}{0pt}%
\pgfsys@defobject{currentmarker}{\pgfqpoint{-0.104167in}{-0.104167in}}{\pgfqpoint{0.104167in}{0.104167in}}{%
\pgfpathmoveto{\pgfqpoint{0.000000in}{-0.104167in}}%
\pgfpathcurveto{\pgfqpoint{0.027625in}{-0.104167in}}{\pgfqpoint{0.054123in}{-0.093191in}}{\pgfqpoint{0.073657in}{-0.073657in}}%
\pgfpathcurveto{\pgfqpoint{0.093191in}{-0.054123in}}{\pgfqpoint{0.104167in}{-0.027625in}}{\pgfqpoint{0.104167in}{0.000000in}}%
\pgfpathcurveto{\pgfqpoint{0.104167in}{0.027625in}}{\pgfqpoint{0.093191in}{0.054123in}}{\pgfqpoint{0.073657in}{0.073657in}}%
\pgfpathcurveto{\pgfqpoint{0.054123in}{0.093191in}}{\pgfqpoint{0.027625in}{0.104167in}}{\pgfqpoint{0.000000in}{0.104167in}}%
\pgfpathcurveto{\pgfqpoint{-0.027625in}{0.104167in}}{\pgfqpoint{-0.054123in}{0.093191in}}{\pgfqpoint{-0.073657in}{0.073657in}}%
\pgfpathcurveto{\pgfqpoint{-0.093191in}{0.054123in}}{\pgfqpoint{-0.104167in}{0.027625in}}{\pgfqpoint{-0.104167in}{0.000000in}}%
\pgfpathcurveto{\pgfqpoint{-0.104167in}{-0.027625in}}{\pgfqpoint{-0.093191in}{-0.054123in}}{\pgfqpoint{-0.073657in}{-0.073657in}}%
\pgfpathcurveto{\pgfqpoint{-0.054123in}{-0.093191in}}{\pgfqpoint{-0.027625in}{-0.104167in}}{\pgfqpoint{0.000000in}{-0.104167in}}%
\pgfpathclose%
\pgfusepath{stroke,fill}%
}%
\begin{pgfscope}%
\pgfsys@transformshift{2.000000in}{3.400000in}%
\pgfsys@useobject{currentmarker}{}%
\end{pgfscope}%
\end{pgfscope}%
\begin{pgfscope}%
\pgfpathrectangle{\pgfqpoint{0.000000in}{0.000000in}}{\pgfqpoint{4.000000in}{4.000000in}}%
\pgfusepath{clip}%
\pgfsetbuttcap%
\pgfsetroundjoin%
\definecolor{currentfill}{rgb}{0.800000,0.800000,0.800000}%
\pgfsetfillcolor{currentfill}%
\pgfsetlinewidth{1.003750pt}%
\definecolor{currentstroke}{rgb}{0.000000,0.000000,0.000000}%
\pgfsetstrokecolor{currentstroke}%
\pgfsetdash{}{0pt}%
\pgfsys@defobject{currentmarker}{\pgfqpoint{-0.104167in}{-0.104167in}}{\pgfqpoint{0.104167in}{0.104167in}}{%
\pgfpathmoveto{\pgfqpoint{0.000000in}{-0.104167in}}%
\pgfpathcurveto{\pgfqpoint{0.027625in}{-0.104167in}}{\pgfqpoint{0.054123in}{-0.093191in}}{\pgfqpoint{0.073657in}{-0.073657in}}%
\pgfpathcurveto{\pgfqpoint{0.093191in}{-0.054123in}}{\pgfqpoint{0.104167in}{-0.027625in}}{\pgfqpoint{0.104167in}{0.000000in}}%
\pgfpathcurveto{\pgfqpoint{0.104167in}{0.027625in}}{\pgfqpoint{0.093191in}{0.054123in}}{\pgfqpoint{0.073657in}{0.073657in}}%
\pgfpathcurveto{\pgfqpoint{0.054123in}{0.093191in}}{\pgfqpoint{0.027625in}{0.104167in}}{\pgfqpoint{0.000000in}{0.104167in}}%
\pgfpathcurveto{\pgfqpoint{-0.027625in}{0.104167in}}{\pgfqpoint{-0.054123in}{0.093191in}}{\pgfqpoint{-0.073657in}{0.073657in}}%
\pgfpathcurveto{\pgfqpoint{-0.093191in}{0.054123in}}{\pgfqpoint{-0.104167in}{0.027625in}}{\pgfqpoint{-0.104167in}{0.000000in}}%
\pgfpathcurveto{\pgfqpoint{-0.104167in}{-0.027625in}}{\pgfqpoint{-0.093191in}{-0.054123in}}{\pgfqpoint{-0.073657in}{-0.073657in}}%
\pgfpathcurveto{\pgfqpoint{-0.054123in}{-0.093191in}}{\pgfqpoint{-0.027625in}{-0.104167in}}{\pgfqpoint{0.000000in}{-0.104167in}}%
\pgfpathclose%
\pgfusepath{stroke,fill}%
}%
\begin{pgfscope}%
\pgfsys@transformshift{0.668521in}{2.432624in}%
\pgfsys@useobject{currentmarker}{}%
\end{pgfscope}%
\end{pgfscope}%
\begin{pgfscope}%
\definecolor{textcolor}{rgb}{0.000000,0.000000,0.000000}%
\pgfsetstrokecolor{textcolor}%
\pgfsetfillcolor{textcolor}%
\pgftext[x=0.941987in,y=0.543769in,,]{\color{textcolor}\sffamily\fontsize{40.000000}{48.000000}\selectfont \(\displaystyle 1\)}%
\end{pgfscope}%
\begin{pgfscope}%
\definecolor{textcolor}{rgb}{0.000000,0.000000,0.000000}%
\pgfsetstrokecolor{textcolor}%
\pgfsetfillcolor{textcolor}%
\pgftext[x=3.058013in,y=0.543769in,,]{\color{textcolor}\sffamily\fontsize{40.000000}{48.000000}\selectfont \(\displaystyle 2\)}%
\end{pgfscope}%
\begin{pgfscope}%
\definecolor{textcolor}{rgb}{0.000000,0.000000,0.000000}%
\pgfsetstrokecolor{textcolor}%
\pgfsetfillcolor{textcolor}%
\pgftext[x=3.711902in,y=2.556231in,,]{\color{textcolor}\sffamily\fontsize{40.000000}{48.000000}\selectfont \(\displaystyle 3\)}%
\end{pgfscope}%
\begin{pgfscope}%
\definecolor{textcolor}{rgb}{0.000000,0.000000,0.000000}%
\pgfsetstrokecolor{textcolor}%
\pgfsetfillcolor{textcolor}%
\pgftext[x=2.000000in,y=3.800000in,,]{\color{textcolor}\sffamily\fontsize{40.000000}{48.000000}\selectfont \(\displaystyle 4\)}%
\end{pgfscope}%
\begin{pgfscope}%
\definecolor{textcolor}{rgb}{0.000000,0.000000,0.000000}%
\pgfsetstrokecolor{textcolor}%
\pgfsetfillcolor{textcolor}%
\pgftext[x=0.288098in,y=2.556231in,,]{\color{textcolor}\sffamily\fontsize{40.000000}{48.000000}\selectfont \(\displaystyle 5\)}%
\end{pgfscope}%
\end{pgfpicture}%
\makeatother%
\endgroup%

%% file: fig/decomposition_H_U0.pgf
%% Creator: Matplotlib, PGF backend
%%
%% To include the figure in your LaTeX document, write
%%   \input{<filename>.pgf}
%%
%% Make sure the required packages are loaded in your preamble
%%   \usepackage{pgf}
%%
%% Figures using additional raster images can only be included by \input if
%% they are in the same directory as the main LaTeX file. For loading figures
%% from other directories you can use the `import` package
%%   \usepackage{import}
%%
%% and then include the figures with
%%   \import{<path to file>}{<filename>.pgf}
%%
%% Matplotlib used the following preamble
%%   \usepackage{fontspec}
%%   \setmainfont{DejaVuSerif.ttf}[Path=\detokenize{C:/Users/ccros/Anaconda3/Lib/site-packages/matplotlib/mpl-data/fonts/ttf/}]
%%   \setsansfont{DejaVuSans.ttf}[Path=\detokenize{C:/Users/ccros/Anaconda3/Lib/site-packages/matplotlib/mpl-data/fonts/ttf/}]
%%   \setmonofont{DejaVuSansMono.ttf}[Path=\detokenize{C:/Users/ccros/Anaconda3/Lib/site-packages/matplotlib/mpl-data/fonts/ttf/}]
%%
\begingroup%
\makeatletter%
\begin{pgfpicture}%
\pgfpathrectangle{\pgfpointorigin}{\pgfqpoint{6.000000in}{3.600000in}}%
\pgfusepath{use as bounding box, clip}%
\begin{pgfscope}%
\pgfsetbuttcap%
\pgfsetmiterjoin%
\definecolor{currentfill}{rgb}{1.000000,1.000000,1.000000}%
\pgfsetfillcolor{currentfill}%
\pgfsetlinewidth{0.000000pt}%
\definecolor{currentstroke}{rgb}{1.000000,1.000000,1.000000}%
\pgfsetstrokecolor{currentstroke}%
\pgfsetdash{}{0pt}%
\pgfpathmoveto{\pgfqpoint{0.000000in}{0.000000in}}%
\pgfpathlineto{\pgfqpoint{6.000000in}{0.000000in}}%
\pgfpathlineto{\pgfqpoint{6.000000in}{3.600000in}}%
\pgfpathlineto{\pgfqpoint{0.000000in}{3.600000in}}%
\pgfpathclose%
\pgfusepath{fill}%
\end{pgfscope}%
\begin{pgfscope}%
\pgfpathrectangle{\pgfqpoint{0.000000in}{0.000000in}}{\pgfqpoint{6.000000in}{3.600000in}}%
\pgfusepath{clip}%
\pgfsetbuttcap%
\pgfsetroundjoin%
\pgfsetlinewidth{3.011250pt}%
\definecolor{currentstroke}{rgb}{0.400000,0.400000,0.400000}%
\pgfsetstrokecolor{currentstroke}%
\pgfsetdash{{11.100000pt}{4.800000pt}}{0.000000pt}%
\pgfpathmoveto{\pgfqpoint{5.250000in}{1.350000in}}%
\pgfpathlineto{\pgfqpoint{3.750000in}{0.450000in}}%
\pgfusepath{stroke}%
\end{pgfscope}%
\begin{pgfscope}%
\pgfpathrectangle{\pgfqpoint{0.000000in}{0.000000in}}{\pgfqpoint{6.000000in}{3.600000in}}%
\pgfusepath{clip}%
\pgfsetrectcap%
\pgfsetroundjoin%
\pgfsetlinewidth{3.011250pt}%
\definecolor{currentstroke}{rgb}{0.250980,0.250980,0.250980}%
\pgfsetstrokecolor{currentstroke}%
\pgfsetdash{}{0pt}%
\pgfpathmoveto{\pgfqpoint{3.750000in}{3.150000in}}%
\pgfpathlineto{\pgfqpoint{2.250000in}{3.150000in}}%
\pgfusepath{stroke}%
\end{pgfscope}%
\begin{pgfscope}%
\pgfpathrectangle{\pgfqpoint{0.000000in}{0.000000in}}{\pgfqpoint{6.000000in}{3.600000in}}%
\pgfusepath{clip}%
\pgfsetrectcap%
\pgfsetroundjoin%
\pgfsetlinewidth{3.011250pt}%
\definecolor{currentstroke}{rgb}{0.250980,0.250980,0.250980}%
\pgfsetstrokecolor{currentstroke}%
\pgfsetdash{}{0pt}%
\pgfpathmoveto{\pgfqpoint{0.750000in}{2.250000in}}%
\pgfpathlineto{\pgfqpoint{2.250000in}{3.150000in}}%
\pgfusepath{stroke}%
\end{pgfscope}%
\begin{pgfscope}%
\pgfpathrectangle{\pgfqpoint{0.000000in}{0.000000in}}{\pgfqpoint{6.000000in}{3.600000in}}%
\pgfusepath{clip}%
\pgfsetrectcap%
\pgfsetroundjoin%
\pgfsetlinewidth{3.011250pt}%
\definecolor{currentstroke}{rgb}{0.250980,0.250980,0.250980}%
\pgfsetstrokecolor{currentstroke}%
\pgfsetdash{}{0pt}%
\pgfpathmoveto{\pgfqpoint{2.250000in}{1.800000in}}%
\pgfpathlineto{\pgfqpoint{2.250000in}{3.150000in}}%
\pgfusepath{stroke}%
\end{pgfscope}%
\begin{pgfscope}%
\pgfpathrectangle{\pgfqpoint{0.000000in}{0.000000in}}{\pgfqpoint{6.000000in}{3.600000in}}%
\pgfusepath{clip}%
\pgfsetrectcap%
\pgfsetroundjoin%
\pgfsetlinewidth{3.011250pt}%
\definecolor{currentstroke}{rgb}{0.250980,0.250980,0.250980}%
\pgfsetstrokecolor{currentstroke}%
\pgfsetdash{}{0pt}%
\pgfpathmoveto{\pgfqpoint{2.250000in}{1.800000in}}%
\pgfpathlineto{\pgfqpoint{0.750000in}{2.250000in}}%
\pgfusepath{stroke}%
\end{pgfscope}%
\begin{pgfscope}%
\pgfpathrectangle{\pgfqpoint{0.000000in}{0.000000in}}{\pgfqpoint{6.000000in}{3.600000in}}%
\pgfusepath{clip}%
\pgfsetrectcap%
\pgfsetroundjoin%
\pgfsetlinewidth{3.011250pt}%
\definecolor{currentstroke}{rgb}{0.250980,0.250980,0.250980}%
\pgfsetstrokecolor{currentstroke}%
\pgfsetdash{}{0pt}%
\pgfpathmoveto{\pgfqpoint{3.750000in}{1.800000in}}%
\pgfpathlineto{\pgfqpoint{5.250000in}{2.250000in}}%
\pgfusepath{stroke}%
\end{pgfscope}%
\begin{pgfscope}%
\pgfpathrectangle{\pgfqpoint{0.000000in}{0.000000in}}{\pgfqpoint{6.000000in}{3.600000in}}%
\pgfusepath{clip}%
\pgfsetrectcap%
\pgfsetroundjoin%
\pgfsetlinewidth{3.011250pt}%
\definecolor{currentstroke}{rgb}{0.250980,0.250980,0.250980}%
\pgfsetstrokecolor{currentstroke}%
\pgfsetdash{}{0pt}%
\pgfpathmoveto{\pgfqpoint{3.750000in}{1.800000in}}%
\pgfpathlineto{\pgfqpoint{2.250000in}{1.800000in}}%
\pgfusepath{stroke}%
\end{pgfscope}%
\begin{pgfscope}%
\pgfpathrectangle{\pgfqpoint{0.000000in}{0.000000in}}{\pgfqpoint{6.000000in}{3.600000in}}%
\pgfusepath{clip}%
\pgfsetrectcap%
\pgfsetroundjoin%
\pgfsetlinewidth{3.011250pt}%
\definecolor{currentstroke}{rgb}{0.250980,0.250980,0.250980}%
\pgfsetstrokecolor{currentstroke}%
\pgfsetdash{}{0pt}%
\pgfpathmoveto{\pgfqpoint{0.750000in}{1.350000in}}%
\pgfpathlineto{\pgfqpoint{0.750000in}{2.250000in}}%
\pgfusepath{stroke}%
\end{pgfscope}%
\begin{pgfscope}%
\pgfpathrectangle{\pgfqpoint{0.000000in}{0.000000in}}{\pgfqpoint{6.000000in}{3.600000in}}%
\pgfusepath{clip}%
\pgfsetrectcap%
\pgfsetroundjoin%
\pgfsetlinewidth{3.011250pt}%
\definecolor{currentstroke}{rgb}{0.250980,0.250980,0.250980}%
\pgfsetstrokecolor{currentstroke}%
\pgfsetdash{}{0pt}%
\pgfpathmoveto{\pgfqpoint{2.250000in}{0.450000in}}%
\pgfpathlineto{\pgfqpoint{0.750000in}{1.350000in}}%
\pgfusepath{stroke}%
\end{pgfscope}%
\begin{pgfscope}%
\pgfpathrectangle{\pgfqpoint{0.000000in}{0.000000in}}{\pgfqpoint{6.000000in}{3.600000in}}%
\pgfusepath{clip}%
\pgfsetbuttcap%
\pgfsetroundjoin%
\definecolor{currentfill}{rgb}{0.800000,0.800000,0.800000}%
\pgfsetfillcolor{currentfill}%
\pgfsetlinewidth{1.003750pt}%
\definecolor{currentstroke}{rgb}{0.000000,0.000000,0.000000}%
\pgfsetstrokecolor{currentstroke}%
\pgfsetdash{}{0pt}%
\pgfsys@defobject{currentmarker}{\pgfqpoint{-0.104167in}{-0.104167in}}{\pgfqpoint{0.104167in}{0.104167in}}{%
\pgfpathmoveto{\pgfqpoint{0.000000in}{-0.104167in}}%
\pgfpathcurveto{\pgfqpoint{0.027625in}{-0.104167in}}{\pgfqpoint{0.054123in}{-0.093191in}}{\pgfqpoint{0.073657in}{-0.073657in}}%
\pgfpathcurveto{\pgfqpoint{0.093191in}{-0.054123in}}{\pgfqpoint{0.104167in}{-0.027625in}}{\pgfqpoint{0.104167in}{0.000000in}}%
\pgfpathcurveto{\pgfqpoint{0.104167in}{0.027625in}}{\pgfqpoint{0.093191in}{0.054123in}}{\pgfqpoint{0.073657in}{0.073657in}}%
\pgfpathcurveto{\pgfqpoint{0.054123in}{0.093191in}}{\pgfqpoint{0.027625in}{0.104167in}}{\pgfqpoint{0.000000in}{0.104167in}}%
\pgfpathcurveto{\pgfqpoint{-0.027625in}{0.104167in}}{\pgfqpoint{-0.054123in}{0.093191in}}{\pgfqpoint{-0.073657in}{0.073657in}}%
\pgfpathcurveto{\pgfqpoint{-0.093191in}{0.054123in}}{\pgfqpoint{-0.104167in}{0.027625in}}{\pgfqpoint{-0.104167in}{0.000000in}}%
\pgfpathcurveto{\pgfqpoint{-0.104167in}{-0.027625in}}{\pgfqpoint{-0.093191in}{-0.054123in}}{\pgfqpoint{-0.073657in}{-0.073657in}}%
\pgfpathcurveto{\pgfqpoint{-0.054123in}{-0.093191in}}{\pgfqpoint{-0.027625in}{-0.104167in}}{\pgfqpoint{0.000000in}{-0.104167in}}%
\pgfpathclose%
\pgfusepath{stroke,fill}%
}%
\begin{pgfscope}%
\pgfsys@transformshift{2.250000in}{3.150000in}%
\pgfsys@useobject{currentmarker}{}%
\end{pgfscope}%
\end{pgfscope}%
\begin{pgfscope}%
\pgfpathrectangle{\pgfqpoint{0.000000in}{0.000000in}}{\pgfqpoint{6.000000in}{3.600000in}}%
\pgfusepath{clip}%
\pgfsetbuttcap%
\pgfsetroundjoin%
\definecolor{currentfill}{rgb}{0.800000,0.800000,0.800000}%
\pgfsetfillcolor{currentfill}%
\pgfsetlinewidth{1.003750pt}%
\definecolor{currentstroke}{rgb}{0.000000,0.000000,0.000000}%
\pgfsetstrokecolor{currentstroke}%
\pgfsetdash{}{0pt}%
\pgfsys@defobject{currentmarker}{\pgfqpoint{-0.104167in}{-0.104167in}}{\pgfqpoint{0.104167in}{0.104167in}}{%
\pgfpathmoveto{\pgfqpoint{0.000000in}{-0.104167in}}%
\pgfpathcurveto{\pgfqpoint{0.027625in}{-0.104167in}}{\pgfqpoint{0.054123in}{-0.093191in}}{\pgfqpoint{0.073657in}{-0.073657in}}%
\pgfpathcurveto{\pgfqpoint{0.093191in}{-0.054123in}}{\pgfqpoint{0.104167in}{-0.027625in}}{\pgfqpoint{0.104167in}{0.000000in}}%
\pgfpathcurveto{\pgfqpoint{0.104167in}{0.027625in}}{\pgfqpoint{0.093191in}{0.054123in}}{\pgfqpoint{0.073657in}{0.073657in}}%
\pgfpathcurveto{\pgfqpoint{0.054123in}{0.093191in}}{\pgfqpoint{0.027625in}{0.104167in}}{\pgfqpoint{0.000000in}{0.104167in}}%
\pgfpathcurveto{\pgfqpoint{-0.027625in}{0.104167in}}{\pgfqpoint{-0.054123in}{0.093191in}}{\pgfqpoint{-0.073657in}{0.073657in}}%
\pgfpathcurveto{\pgfqpoint{-0.093191in}{0.054123in}}{\pgfqpoint{-0.104167in}{0.027625in}}{\pgfqpoint{-0.104167in}{0.000000in}}%
\pgfpathcurveto{\pgfqpoint{-0.104167in}{-0.027625in}}{\pgfqpoint{-0.093191in}{-0.054123in}}{\pgfqpoint{-0.073657in}{-0.073657in}}%
\pgfpathcurveto{\pgfqpoint{-0.054123in}{-0.093191in}}{\pgfqpoint{-0.027625in}{-0.104167in}}{\pgfqpoint{0.000000in}{-0.104167in}}%
\pgfpathclose%
\pgfusepath{stroke,fill}%
}%
\begin{pgfscope}%
\pgfsys@transformshift{3.750000in}{3.150000in}%
\pgfsys@useobject{currentmarker}{}%
\end{pgfscope}%
\end{pgfscope}%
\begin{pgfscope}%
\pgfpathrectangle{\pgfqpoint{0.000000in}{0.000000in}}{\pgfqpoint{6.000000in}{3.600000in}}%
\pgfusepath{clip}%
\pgfsetbuttcap%
\pgfsetroundjoin%
\definecolor{currentfill}{rgb}{0.800000,0.800000,0.800000}%
\pgfsetfillcolor{currentfill}%
\pgfsetlinewidth{1.003750pt}%
\definecolor{currentstroke}{rgb}{0.000000,0.000000,0.000000}%
\pgfsetstrokecolor{currentstroke}%
\pgfsetdash{}{0pt}%
\pgfsys@defobject{currentmarker}{\pgfqpoint{-0.104167in}{-0.104167in}}{\pgfqpoint{0.104167in}{0.104167in}}{%
\pgfpathmoveto{\pgfqpoint{0.000000in}{-0.104167in}}%
\pgfpathcurveto{\pgfqpoint{0.027625in}{-0.104167in}}{\pgfqpoint{0.054123in}{-0.093191in}}{\pgfqpoint{0.073657in}{-0.073657in}}%
\pgfpathcurveto{\pgfqpoint{0.093191in}{-0.054123in}}{\pgfqpoint{0.104167in}{-0.027625in}}{\pgfqpoint{0.104167in}{0.000000in}}%
\pgfpathcurveto{\pgfqpoint{0.104167in}{0.027625in}}{\pgfqpoint{0.093191in}{0.054123in}}{\pgfqpoint{0.073657in}{0.073657in}}%
\pgfpathcurveto{\pgfqpoint{0.054123in}{0.093191in}}{\pgfqpoint{0.027625in}{0.104167in}}{\pgfqpoint{0.000000in}{0.104167in}}%
\pgfpathcurveto{\pgfqpoint{-0.027625in}{0.104167in}}{\pgfqpoint{-0.054123in}{0.093191in}}{\pgfqpoint{-0.073657in}{0.073657in}}%
\pgfpathcurveto{\pgfqpoint{-0.093191in}{0.054123in}}{\pgfqpoint{-0.104167in}{0.027625in}}{\pgfqpoint{-0.104167in}{0.000000in}}%
\pgfpathcurveto{\pgfqpoint{-0.104167in}{-0.027625in}}{\pgfqpoint{-0.093191in}{-0.054123in}}{\pgfqpoint{-0.073657in}{-0.073657in}}%
\pgfpathcurveto{\pgfqpoint{-0.054123in}{-0.093191in}}{\pgfqpoint{-0.027625in}{-0.104167in}}{\pgfqpoint{0.000000in}{-0.104167in}}%
\pgfpathclose%
\pgfusepath{stroke,fill}%
}%
\begin{pgfscope}%
\pgfsys@transformshift{0.750000in}{2.250000in}%
\pgfsys@useobject{currentmarker}{}%
\end{pgfscope}%
\end{pgfscope}%
\begin{pgfscope}%
\pgfpathrectangle{\pgfqpoint{0.000000in}{0.000000in}}{\pgfqpoint{6.000000in}{3.600000in}}%
\pgfusepath{clip}%
\pgfsetbuttcap%
\pgfsetroundjoin%
\definecolor{currentfill}{rgb}{0.800000,0.800000,0.800000}%
\pgfsetfillcolor{currentfill}%
\pgfsetlinewidth{1.003750pt}%
\definecolor{currentstroke}{rgb}{0.000000,0.000000,0.000000}%
\pgfsetstrokecolor{currentstroke}%
\pgfsetdash{}{0pt}%
\pgfsys@defobject{currentmarker}{\pgfqpoint{-0.104167in}{-0.104167in}}{\pgfqpoint{0.104167in}{0.104167in}}{%
\pgfpathmoveto{\pgfqpoint{0.000000in}{-0.104167in}}%
\pgfpathcurveto{\pgfqpoint{0.027625in}{-0.104167in}}{\pgfqpoint{0.054123in}{-0.093191in}}{\pgfqpoint{0.073657in}{-0.073657in}}%
\pgfpathcurveto{\pgfqpoint{0.093191in}{-0.054123in}}{\pgfqpoint{0.104167in}{-0.027625in}}{\pgfqpoint{0.104167in}{0.000000in}}%
\pgfpathcurveto{\pgfqpoint{0.104167in}{0.027625in}}{\pgfqpoint{0.093191in}{0.054123in}}{\pgfqpoint{0.073657in}{0.073657in}}%
\pgfpathcurveto{\pgfqpoint{0.054123in}{0.093191in}}{\pgfqpoint{0.027625in}{0.104167in}}{\pgfqpoint{0.000000in}{0.104167in}}%
\pgfpathcurveto{\pgfqpoint{-0.027625in}{0.104167in}}{\pgfqpoint{-0.054123in}{0.093191in}}{\pgfqpoint{-0.073657in}{0.073657in}}%
\pgfpathcurveto{\pgfqpoint{-0.093191in}{0.054123in}}{\pgfqpoint{-0.104167in}{0.027625in}}{\pgfqpoint{-0.104167in}{0.000000in}}%
\pgfpathcurveto{\pgfqpoint{-0.104167in}{-0.027625in}}{\pgfqpoint{-0.093191in}{-0.054123in}}{\pgfqpoint{-0.073657in}{-0.073657in}}%
\pgfpathcurveto{\pgfqpoint{-0.054123in}{-0.093191in}}{\pgfqpoint{-0.027625in}{-0.104167in}}{\pgfqpoint{0.000000in}{-0.104167in}}%
\pgfpathclose%
\pgfusepath{stroke,fill}%
}%
\begin{pgfscope}%
\pgfsys@transformshift{5.250000in}{2.250000in}%
\pgfsys@useobject{currentmarker}{}%
\end{pgfscope}%
\end{pgfscope}%
\begin{pgfscope}%
\pgfpathrectangle{\pgfqpoint{0.000000in}{0.000000in}}{\pgfqpoint{6.000000in}{3.600000in}}%
\pgfusepath{clip}%
\pgfsetbuttcap%
\pgfsetroundjoin%
\definecolor{currentfill}{rgb}{0.800000,0.800000,0.800000}%
\pgfsetfillcolor{currentfill}%
\pgfsetlinewidth{1.003750pt}%
\definecolor{currentstroke}{rgb}{0.000000,0.000000,0.000000}%
\pgfsetstrokecolor{currentstroke}%
\pgfsetdash{}{0pt}%
\pgfsys@defobject{currentmarker}{\pgfqpoint{-0.104167in}{-0.104167in}}{\pgfqpoint{0.104167in}{0.104167in}}{%
\pgfpathmoveto{\pgfqpoint{0.000000in}{-0.104167in}}%
\pgfpathcurveto{\pgfqpoint{0.027625in}{-0.104167in}}{\pgfqpoint{0.054123in}{-0.093191in}}{\pgfqpoint{0.073657in}{-0.073657in}}%
\pgfpathcurveto{\pgfqpoint{0.093191in}{-0.054123in}}{\pgfqpoint{0.104167in}{-0.027625in}}{\pgfqpoint{0.104167in}{0.000000in}}%
\pgfpathcurveto{\pgfqpoint{0.104167in}{0.027625in}}{\pgfqpoint{0.093191in}{0.054123in}}{\pgfqpoint{0.073657in}{0.073657in}}%
\pgfpathcurveto{\pgfqpoint{0.054123in}{0.093191in}}{\pgfqpoint{0.027625in}{0.104167in}}{\pgfqpoint{0.000000in}{0.104167in}}%
\pgfpathcurveto{\pgfqpoint{-0.027625in}{0.104167in}}{\pgfqpoint{-0.054123in}{0.093191in}}{\pgfqpoint{-0.073657in}{0.073657in}}%
\pgfpathcurveto{\pgfqpoint{-0.093191in}{0.054123in}}{\pgfqpoint{-0.104167in}{0.027625in}}{\pgfqpoint{-0.104167in}{0.000000in}}%
\pgfpathcurveto{\pgfqpoint{-0.104167in}{-0.027625in}}{\pgfqpoint{-0.093191in}{-0.054123in}}{\pgfqpoint{-0.073657in}{-0.073657in}}%
\pgfpathcurveto{\pgfqpoint{-0.054123in}{-0.093191in}}{\pgfqpoint{-0.027625in}{-0.104167in}}{\pgfqpoint{0.000000in}{-0.104167in}}%
\pgfpathclose%
\pgfusepath{stroke,fill}%
}%
\begin{pgfscope}%
\pgfsys@transformshift{2.250000in}{1.800000in}%
\pgfsys@useobject{currentmarker}{}%
\end{pgfscope}%
\end{pgfscope}%
\begin{pgfscope}%
\pgfpathrectangle{\pgfqpoint{0.000000in}{0.000000in}}{\pgfqpoint{6.000000in}{3.600000in}}%
\pgfusepath{clip}%
\pgfsetbuttcap%
\pgfsetroundjoin%
\definecolor{currentfill}{rgb}{0.800000,0.800000,0.800000}%
\pgfsetfillcolor{currentfill}%
\pgfsetlinewidth{1.003750pt}%
\definecolor{currentstroke}{rgb}{0.000000,0.000000,0.000000}%
\pgfsetstrokecolor{currentstroke}%
\pgfsetdash{}{0pt}%
\pgfsys@defobject{currentmarker}{\pgfqpoint{-0.104167in}{-0.104167in}}{\pgfqpoint{0.104167in}{0.104167in}}{%
\pgfpathmoveto{\pgfqpoint{0.000000in}{-0.104167in}}%
\pgfpathcurveto{\pgfqpoint{0.027625in}{-0.104167in}}{\pgfqpoint{0.054123in}{-0.093191in}}{\pgfqpoint{0.073657in}{-0.073657in}}%
\pgfpathcurveto{\pgfqpoint{0.093191in}{-0.054123in}}{\pgfqpoint{0.104167in}{-0.027625in}}{\pgfqpoint{0.104167in}{0.000000in}}%
\pgfpathcurveto{\pgfqpoint{0.104167in}{0.027625in}}{\pgfqpoint{0.093191in}{0.054123in}}{\pgfqpoint{0.073657in}{0.073657in}}%
\pgfpathcurveto{\pgfqpoint{0.054123in}{0.093191in}}{\pgfqpoint{0.027625in}{0.104167in}}{\pgfqpoint{0.000000in}{0.104167in}}%
\pgfpathcurveto{\pgfqpoint{-0.027625in}{0.104167in}}{\pgfqpoint{-0.054123in}{0.093191in}}{\pgfqpoint{-0.073657in}{0.073657in}}%
\pgfpathcurveto{\pgfqpoint{-0.093191in}{0.054123in}}{\pgfqpoint{-0.104167in}{0.027625in}}{\pgfqpoint{-0.104167in}{0.000000in}}%
\pgfpathcurveto{\pgfqpoint{-0.104167in}{-0.027625in}}{\pgfqpoint{-0.093191in}{-0.054123in}}{\pgfqpoint{-0.073657in}{-0.073657in}}%
\pgfpathcurveto{\pgfqpoint{-0.054123in}{-0.093191in}}{\pgfqpoint{-0.027625in}{-0.104167in}}{\pgfqpoint{0.000000in}{-0.104167in}}%
\pgfpathclose%
\pgfusepath{stroke,fill}%
}%
\begin{pgfscope}%
\pgfsys@transformshift{3.750000in}{1.800000in}%
\pgfsys@useobject{currentmarker}{}%
\end{pgfscope}%
\end{pgfscope}%
\begin{pgfscope}%
\pgfpathrectangle{\pgfqpoint{0.000000in}{0.000000in}}{\pgfqpoint{6.000000in}{3.600000in}}%
\pgfusepath{clip}%
\pgfsetbuttcap%
\pgfsetroundjoin%
\definecolor{currentfill}{rgb}{0.800000,0.800000,0.800000}%
\pgfsetfillcolor{currentfill}%
\pgfsetlinewidth{1.003750pt}%
\definecolor{currentstroke}{rgb}{0.000000,0.000000,0.000000}%
\pgfsetstrokecolor{currentstroke}%
\pgfsetdash{}{0pt}%
\pgfsys@defobject{currentmarker}{\pgfqpoint{-0.104167in}{-0.104167in}}{\pgfqpoint{0.104167in}{0.104167in}}{%
\pgfpathmoveto{\pgfqpoint{0.000000in}{-0.104167in}}%
\pgfpathcurveto{\pgfqpoint{0.027625in}{-0.104167in}}{\pgfqpoint{0.054123in}{-0.093191in}}{\pgfqpoint{0.073657in}{-0.073657in}}%
\pgfpathcurveto{\pgfqpoint{0.093191in}{-0.054123in}}{\pgfqpoint{0.104167in}{-0.027625in}}{\pgfqpoint{0.104167in}{0.000000in}}%
\pgfpathcurveto{\pgfqpoint{0.104167in}{0.027625in}}{\pgfqpoint{0.093191in}{0.054123in}}{\pgfqpoint{0.073657in}{0.073657in}}%
\pgfpathcurveto{\pgfqpoint{0.054123in}{0.093191in}}{\pgfqpoint{0.027625in}{0.104167in}}{\pgfqpoint{0.000000in}{0.104167in}}%
\pgfpathcurveto{\pgfqpoint{-0.027625in}{0.104167in}}{\pgfqpoint{-0.054123in}{0.093191in}}{\pgfqpoint{-0.073657in}{0.073657in}}%
\pgfpathcurveto{\pgfqpoint{-0.093191in}{0.054123in}}{\pgfqpoint{-0.104167in}{0.027625in}}{\pgfqpoint{-0.104167in}{0.000000in}}%
\pgfpathcurveto{\pgfqpoint{-0.104167in}{-0.027625in}}{\pgfqpoint{-0.093191in}{-0.054123in}}{\pgfqpoint{-0.073657in}{-0.073657in}}%
\pgfpathcurveto{\pgfqpoint{-0.054123in}{-0.093191in}}{\pgfqpoint{-0.027625in}{-0.104167in}}{\pgfqpoint{0.000000in}{-0.104167in}}%
\pgfpathclose%
\pgfusepath{stroke,fill}%
}%
\begin{pgfscope}%
\pgfsys@transformshift{0.750000in}{1.350000in}%
\pgfsys@useobject{currentmarker}{}%
\end{pgfscope}%
\end{pgfscope}%
\begin{pgfscope}%
\pgfpathrectangle{\pgfqpoint{0.000000in}{0.000000in}}{\pgfqpoint{6.000000in}{3.600000in}}%
\pgfusepath{clip}%
\pgfsetbuttcap%
\pgfsetroundjoin%
\definecolor{currentfill}{rgb}{1.000000,1.000000,1.000000}%
\pgfsetfillcolor{currentfill}%
\pgfsetlinewidth{1.003750pt}%
\definecolor{currentstroke}{rgb}{0.000000,0.000000,0.000000}%
\pgfsetstrokecolor{currentstroke}%
\pgfsetdash{}{0pt}%
\pgfsys@defobject{currentmarker}{\pgfqpoint{-0.104167in}{-0.104167in}}{\pgfqpoint{0.104167in}{0.104167in}}{%
\pgfpathmoveto{\pgfqpoint{0.000000in}{-0.104167in}}%
\pgfpathcurveto{\pgfqpoint{0.027625in}{-0.104167in}}{\pgfqpoint{0.054123in}{-0.093191in}}{\pgfqpoint{0.073657in}{-0.073657in}}%
\pgfpathcurveto{\pgfqpoint{0.093191in}{-0.054123in}}{\pgfqpoint{0.104167in}{-0.027625in}}{\pgfqpoint{0.104167in}{0.000000in}}%
\pgfpathcurveto{\pgfqpoint{0.104167in}{0.027625in}}{\pgfqpoint{0.093191in}{0.054123in}}{\pgfqpoint{0.073657in}{0.073657in}}%
\pgfpathcurveto{\pgfqpoint{0.054123in}{0.093191in}}{\pgfqpoint{0.027625in}{0.104167in}}{\pgfqpoint{0.000000in}{0.104167in}}%
\pgfpathcurveto{\pgfqpoint{-0.027625in}{0.104167in}}{\pgfqpoint{-0.054123in}{0.093191in}}{\pgfqpoint{-0.073657in}{0.073657in}}%
\pgfpathcurveto{\pgfqpoint{-0.093191in}{0.054123in}}{\pgfqpoint{-0.104167in}{0.027625in}}{\pgfqpoint{-0.104167in}{0.000000in}}%
\pgfpathcurveto{\pgfqpoint{-0.104167in}{-0.027625in}}{\pgfqpoint{-0.093191in}{-0.054123in}}{\pgfqpoint{-0.073657in}{-0.073657in}}%
\pgfpathcurveto{\pgfqpoint{-0.054123in}{-0.093191in}}{\pgfqpoint{-0.027625in}{-0.104167in}}{\pgfqpoint{0.000000in}{-0.104167in}}%
\pgfpathclose%
\pgfusepath{stroke,fill}%
}%
\begin{pgfscope}%
\pgfsys@transformshift{5.250000in}{1.350000in}%
\pgfsys@useobject{currentmarker}{}%
\end{pgfscope}%
\end{pgfscope}%
\begin{pgfscope}%
\pgfpathrectangle{\pgfqpoint{0.000000in}{0.000000in}}{\pgfqpoint{6.000000in}{3.600000in}}%
\pgfusepath{clip}%
\pgfsetbuttcap%
\pgfsetroundjoin%
\definecolor{currentfill}{rgb}{0.800000,0.800000,0.800000}%
\pgfsetfillcolor{currentfill}%
\pgfsetlinewidth{1.003750pt}%
\definecolor{currentstroke}{rgb}{0.000000,0.000000,0.000000}%
\pgfsetstrokecolor{currentstroke}%
\pgfsetdash{}{0pt}%
\pgfsys@defobject{currentmarker}{\pgfqpoint{-0.104167in}{-0.104167in}}{\pgfqpoint{0.104167in}{0.104167in}}{%
\pgfpathmoveto{\pgfqpoint{0.000000in}{-0.104167in}}%
\pgfpathcurveto{\pgfqpoint{0.027625in}{-0.104167in}}{\pgfqpoint{0.054123in}{-0.093191in}}{\pgfqpoint{0.073657in}{-0.073657in}}%
\pgfpathcurveto{\pgfqpoint{0.093191in}{-0.054123in}}{\pgfqpoint{0.104167in}{-0.027625in}}{\pgfqpoint{0.104167in}{0.000000in}}%
\pgfpathcurveto{\pgfqpoint{0.104167in}{0.027625in}}{\pgfqpoint{0.093191in}{0.054123in}}{\pgfqpoint{0.073657in}{0.073657in}}%
\pgfpathcurveto{\pgfqpoint{0.054123in}{0.093191in}}{\pgfqpoint{0.027625in}{0.104167in}}{\pgfqpoint{0.000000in}{0.104167in}}%
\pgfpathcurveto{\pgfqpoint{-0.027625in}{0.104167in}}{\pgfqpoint{-0.054123in}{0.093191in}}{\pgfqpoint{-0.073657in}{0.073657in}}%
\pgfpathcurveto{\pgfqpoint{-0.093191in}{0.054123in}}{\pgfqpoint{-0.104167in}{0.027625in}}{\pgfqpoint{-0.104167in}{0.000000in}}%
\pgfpathcurveto{\pgfqpoint{-0.104167in}{-0.027625in}}{\pgfqpoint{-0.093191in}{-0.054123in}}{\pgfqpoint{-0.073657in}{-0.073657in}}%
\pgfpathcurveto{\pgfqpoint{-0.054123in}{-0.093191in}}{\pgfqpoint{-0.027625in}{-0.104167in}}{\pgfqpoint{0.000000in}{-0.104167in}}%
\pgfpathclose%
\pgfusepath{stroke,fill}%
}%
\begin{pgfscope}%
\pgfsys@transformshift{2.250000in}{0.450000in}%
\pgfsys@useobject{currentmarker}{}%
\end{pgfscope}%
\end{pgfscope}%
\begin{pgfscope}%
\pgfpathrectangle{\pgfqpoint{0.000000in}{0.000000in}}{\pgfqpoint{6.000000in}{3.600000in}}%
\pgfusepath{clip}%
\pgfsetbuttcap%
\pgfsetroundjoin%
\definecolor{currentfill}{rgb}{1.000000,1.000000,1.000000}%
\pgfsetfillcolor{currentfill}%
\pgfsetlinewidth{1.003750pt}%
\definecolor{currentstroke}{rgb}{0.000000,0.000000,0.000000}%
\pgfsetstrokecolor{currentstroke}%
\pgfsetdash{}{0pt}%
\pgfsys@defobject{currentmarker}{\pgfqpoint{-0.104167in}{-0.104167in}}{\pgfqpoint{0.104167in}{0.104167in}}{%
\pgfpathmoveto{\pgfqpoint{0.000000in}{-0.104167in}}%
\pgfpathcurveto{\pgfqpoint{0.027625in}{-0.104167in}}{\pgfqpoint{0.054123in}{-0.093191in}}{\pgfqpoint{0.073657in}{-0.073657in}}%
\pgfpathcurveto{\pgfqpoint{0.093191in}{-0.054123in}}{\pgfqpoint{0.104167in}{-0.027625in}}{\pgfqpoint{0.104167in}{0.000000in}}%
\pgfpathcurveto{\pgfqpoint{0.104167in}{0.027625in}}{\pgfqpoint{0.093191in}{0.054123in}}{\pgfqpoint{0.073657in}{0.073657in}}%
\pgfpathcurveto{\pgfqpoint{0.054123in}{0.093191in}}{\pgfqpoint{0.027625in}{0.104167in}}{\pgfqpoint{0.000000in}{0.104167in}}%
\pgfpathcurveto{\pgfqpoint{-0.027625in}{0.104167in}}{\pgfqpoint{-0.054123in}{0.093191in}}{\pgfqpoint{-0.073657in}{0.073657in}}%
\pgfpathcurveto{\pgfqpoint{-0.093191in}{0.054123in}}{\pgfqpoint{-0.104167in}{0.027625in}}{\pgfqpoint{-0.104167in}{0.000000in}}%
\pgfpathcurveto{\pgfqpoint{-0.104167in}{-0.027625in}}{\pgfqpoint{-0.093191in}{-0.054123in}}{\pgfqpoint{-0.073657in}{-0.073657in}}%
\pgfpathcurveto{\pgfqpoint{-0.054123in}{-0.093191in}}{\pgfqpoint{-0.027625in}{-0.104167in}}{\pgfqpoint{0.000000in}{-0.104167in}}%
\pgfpathclose%
\pgfusepath{stroke,fill}%
}%
\begin{pgfscope}%
\pgfsys@transformshift{3.750000in}{0.450000in}%
\pgfsys@useobject{currentmarker}{}%
\end{pgfscope}%
\end{pgfscope}%
\begin{pgfscope}%
\definecolor{textcolor}{rgb}{0.000000,0.000000,0.000000}%
\pgfsetstrokecolor{textcolor}%
\pgfsetfillcolor{textcolor}%
\pgftext[x=2.529863in,y=1.520137in,,base]{\color{textcolor}\sffamily\fontsize{40.000000}{48.000000}\selectfont \(\displaystyle u_0\)}%
\end{pgfscope}%
\begin{pgfscope}%
\definecolor{textcolor}{rgb}{0.000000,0.000000,0.000000}%
\pgfsetstrokecolor{textcolor}%
\pgfsetfillcolor{textcolor}%
\pgftext[x=4.029863in,y=1.520137in,,base]{\color{textcolor}\sffamily\fontsize{40.000000}{48.000000}\selectfont \(\displaystyle v_0\)}%
\end{pgfscope}%
\end{pgfpicture}%
\makeatother%
\endgroup%

%% file: fig/decomposition_H_U1.pgf
%% Creator: Matplotlib, PGF backend
%%
%% To include the figure in your LaTeX document, write
%%   \input{<filename>.pgf}
%%
%% Make sure the required packages are loaded in your preamble
%%   \usepackage{pgf}
%%
%% Figures using additional raster images can only be included by \input if
%% they are in the same directory as the main LaTeX file. For loading figures
%% from other directories you can use the `import` package
%%   \usepackage{import}
%%
%% and then include the figures with
%%   \import{<path to file>}{<filename>.pgf}
%%
%% Matplotlib used the following preamble
%%   \usepackage{fontspec}
%%   \setmainfont{DejaVuSerif.ttf}[Path=\detokenize{C:/Users/ccros/Anaconda3/Lib/site-packages/matplotlib/mpl-data/fonts/ttf/}]
%%   \setsansfont{DejaVuSans.ttf}[Path=\detokenize{C:/Users/ccros/Anaconda3/Lib/site-packages/matplotlib/mpl-data/fonts/ttf/}]
%%   \setmonofont{DejaVuSansMono.ttf}[Path=\detokenize{C:/Users/ccros/Anaconda3/Lib/site-packages/matplotlib/mpl-data/fonts/ttf/}]
%%
\begingroup%
\makeatletter%
\begin{pgfpicture}%
\pgfpathrectangle{\pgfpointorigin}{\pgfqpoint{6.000000in}{3.600000in}}%
\pgfusepath{use as bounding box, clip}%
\begin{pgfscope}%
\pgfsetbuttcap%
\pgfsetmiterjoin%
\definecolor{currentfill}{rgb}{1.000000,1.000000,1.000000}%
\pgfsetfillcolor{currentfill}%
\pgfsetlinewidth{0.000000pt}%
\definecolor{currentstroke}{rgb}{1.000000,1.000000,1.000000}%
\pgfsetstrokecolor{currentstroke}%
\pgfsetdash{}{0pt}%
\pgfpathmoveto{\pgfqpoint{0.000000in}{0.000000in}}%
\pgfpathlineto{\pgfqpoint{6.000000in}{0.000000in}}%
\pgfpathlineto{\pgfqpoint{6.000000in}{3.600000in}}%
\pgfpathlineto{\pgfqpoint{0.000000in}{3.600000in}}%
\pgfpathclose%
\pgfusepath{fill}%
\end{pgfscope}%
\begin{pgfscope}%
\pgfpathrectangle{\pgfqpoint{0.000000in}{0.000000in}}{\pgfqpoint{6.000000in}{3.600000in}}%
\pgfusepath{clip}%
\pgfsetbuttcap%
\pgfsetroundjoin%
\pgfsetlinewidth{3.011250pt}%
\definecolor{currentstroke}{rgb}{0.400000,0.400000,0.400000}%
\pgfsetstrokecolor{currentstroke}%
\pgfsetdash{{11.100000pt}{4.800000pt}}{0.000000pt}%
\pgfpathmoveto{\pgfqpoint{5.250000in}{2.250000in}}%
\pgfpathlineto{\pgfqpoint{3.750000in}{1.800000in}}%
\pgfusepath{stroke}%
\end{pgfscope}%
\begin{pgfscope}%
\pgfpathrectangle{\pgfqpoint{0.000000in}{0.000000in}}{\pgfqpoint{6.000000in}{3.600000in}}%
\pgfusepath{clip}%
\pgfsetbuttcap%
\pgfsetroundjoin%
\pgfsetlinewidth{3.011250pt}%
\definecolor{currentstroke}{rgb}{0.400000,0.400000,0.400000}%
\pgfsetstrokecolor{currentstroke}%
\pgfsetdash{{11.100000pt}{4.800000pt}}{0.000000pt}%
\pgfpathmoveto{\pgfqpoint{2.250000in}{1.800000in}}%
\pgfpathlineto{\pgfqpoint{3.750000in}{1.800000in}}%
\pgfusepath{stroke}%
\end{pgfscope}%
\begin{pgfscope}%
\pgfpathrectangle{\pgfqpoint{0.000000in}{0.000000in}}{\pgfqpoint{6.000000in}{3.600000in}}%
\pgfusepath{clip}%
\pgfsetbuttcap%
\pgfsetroundjoin%
\pgfsetlinewidth{3.011250pt}%
\definecolor{currentstroke}{rgb}{0.400000,0.400000,0.400000}%
\pgfsetstrokecolor{currentstroke}%
\pgfsetdash{{11.100000pt}{4.800000pt}}{0.000000pt}%
\pgfpathmoveto{\pgfqpoint{5.250000in}{1.350000in}}%
\pgfpathlineto{\pgfqpoint{3.750000in}{0.450000in}}%
\pgfusepath{stroke}%
\end{pgfscope}%
\begin{pgfscope}%
\pgfpathrectangle{\pgfqpoint{0.000000in}{0.000000in}}{\pgfqpoint{6.000000in}{3.600000in}}%
\pgfusepath{clip}%
\pgfsetrectcap%
\pgfsetroundjoin%
\pgfsetlinewidth{3.011250pt}%
\definecolor{currentstroke}{rgb}{0.250980,0.250980,0.250980}%
\pgfsetstrokecolor{currentstroke}%
\pgfsetdash{}{0pt}%
\pgfpathmoveto{\pgfqpoint{3.750000in}{3.150000in}}%
\pgfpathlineto{\pgfqpoint{2.250000in}{3.150000in}}%
\pgfusepath{stroke}%
\end{pgfscope}%
\begin{pgfscope}%
\pgfpathrectangle{\pgfqpoint{0.000000in}{0.000000in}}{\pgfqpoint{6.000000in}{3.600000in}}%
\pgfusepath{clip}%
\pgfsetrectcap%
\pgfsetroundjoin%
\pgfsetlinewidth{3.011250pt}%
\definecolor{currentstroke}{rgb}{0.250980,0.250980,0.250980}%
\pgfsetstrokecolor{currentstroke}%
\pgfsetdash{}{0pt}%
\pgfpathmoveto{\pgfqpoint{0.750000in}{2.250000in}}%
\pgfpathlineto{\pgfqpoint{2.250000in}{3.150000in}}%
\pgfusepath{stroke}%
\end{pgfscope}%
\begin{pgfscope}%
\pgfpathrectangle{\pgfqpoint{0.000000in}{0.000000in}}{\pgfqpoint{6.000000in}{3.600000in}}%
\pgfusepath{clip}%
\pgfsetrectcap%
\pgfsetroundjoin%
\pgfsetlinewidth{3.011250pt}%
\definecolor{currentstroke}{rgb}{0.250980,0.250980,0.250980}%
\pgfsetstrokecolor{currentstroke}%
\pgfsetdash{}{0pt}%
\pgfpathmoveto{\pgfqpoint{2.250000in}{1.800000in}}%
\pgfpathlineto{\pgfqpoint{2.250000in}{3.150000in}}%
\pgfusepath{stroke}%
\end{pgfscope}%
\begin{pgfscope}%
\pgfpathrectangle{\pgfqpoint{0.000000in}{0.000000in}}{\pgfqpoint{6.000000in}{3.600000in}}%
\pgfusepath{clip}%
\pgfsetrectcap%
\pgfsetroundjoin%
\pgfsetlinewidth{3.011250pt}%
\definecolor{currentstroke}{rgb}{0.250980,0.250980,0.250980}%
\pgfsetstrokecolor{currentstroke}%
\pgfsetdash{}{0pt}%
\pgfpathmoveto{\pgfqpoint{2.250000in}{1.800000in}}%
\pgfpathlineto{\pgfqpoint{0.750000in}{2.250000in}}%
\pgfusepath{stroke}%
\end{pgfscope}%
\begin{pgfscope}%
\pgfpathrectangle{\pgfqpoint{0.000000in}{0.000000in}}{\pgfqpoint{6.000000in}{3.600000in}}%
\pgfusepath{clip}%
\pgfsetrectcap%
\pgfsetroundjoin%
\pgfsetlinewidth{3.011250pt}%
\definecolor{currentstroke}{rgb}{0.250980,0.250980,0.250980}%
\pgfsetstrokecolor{currentstroke}%
\pgfsetdash{}{0pt}%
\pgfpathmoveto{\pgfqpoint{0.750000in}{1.350000in}}%
\pgfpathlineto{\pgfqpoint{0.750000in}{2.250000in}}%
\pgfusepath{stroke}%
\end{pgfscope}%
\begin{pgfscope}%
\pgfpathrectangle{\pgfqpoint{0.000000in}{0.000000in}}{\pgfqpoint{6.000000in}{3.600000in}}%
\pgfusepath{clip}%
\pgfsetrectcap%
\pgfsetroundjoin%
\pgfsetlinewidth{3.011250pt}%
\definecolor{currentstroke}{rgb}{0.250980,0.250980,0.250980}%
\pgfsetstrokecolor{currentstroke}%
\pgfsetdash{}{0pt}%
\pgfpathmoveto{\pgfqpoint{2.250000in}{0.450000in}}%
\pgfpathlineto{\pgfqpoint{0.750000in}{1.350000in}}%
\pgfusepath{stroke}%
\end{pgfscope}%
\begin{pgfscope}%
\pgfpathrectangle{\pgfqpoint{0.000000in}{0.000000in}}{\pgfqpoint{6.000000in}{3.600000in}}%
\pgfusepath{clip}%
\pgfsetbuttcap%
\pgfsetroundjoin%
\definecolor{currentfill}{rgb}{0.800000,0.800000,0.800000}%
\pgfsetfillcolor{currentfill}%
\pgfsetlinewidth{1.003750pt}%
\definecolor{currentstroke}{rgb}{0.000000,0.000000,0.000000}%
\pgfsetstrokecolor{currentstroke}%
\pgfsetdash{}{0pt}%
\pgfsys@defobject{currentmarker}{\pgfqpoint{-0.104167in}{-0.104167in}}{\pgfqpoint{0.104167in}{0.104167in}}{%
\pgfpathmoveto{\pgfqpoint{0.000000in}{-0.104167in}}%
\pgfpathcurveto{\pgfqpoint{0.027625in}{-0.104167in}}{\pgfqpoint{0.054123in}{-0.093191in}}{\pgfqpoint{0.073657in}{-0.073657in}}%
\pgfpathcurveto{\pgfqpoint{0.093191in}{-0.054123in}}{\pgfqpoint{0.104167in}{-0.027625in}}{\pgfqpoint{0.104167in}{0.000000in}}%
\pgfpathcurveto{\pgfqpoint{0.104167in}{0.027625in}}{\pgfqpoint{0.093191in}{0.054123in}}{\pgfqpoint{0.073657in}{0.073657in}}%
\pgfpathcurveto{\pgfqpoint{0.054123in}{0.093191in}}{\pgfqpoint{0.027625in}{0.104167in}}{\pgfqpoint{0.000000in}{0.104167in}}%
\pgfpathcurveto{\pgfqpoint{-0.027625in}{0.104167in}}{\pgfqpoint{-0.054123in}{0.093191in}}{\pgfqpoint{-0.073657in}{0.073657in}}%
\pgfpathcurveto{\pgfqpoint{-0.093191in}{0.054123in}}{\pgfqpoint{-0.104167in}{0.027625in}}{\pgfqpoint{-0.104167in}{0.000000in}}%
\pgfpathcurveto{\pgfqpoint{-0.104167in}{-0.027625in}}{\pgfqpoint{-0.093191in}{-0.054123in}}{\pgfqpoint{-0.073657in}{-0.073657in}}%
\pgfpathcurveto{\pgfqpoint{-0.054123in}{-0.093191in}}{\pgfqpoint{-0.027625in}{-0.104167in}}{\pgfqpoint{0.000000in}{-0.104167in}}%
\pgfpathclose%
\pgfusepath{stroke,fill}%
}%
\begin{pgfscope}%
\pgfsys@transformshift{2.250000in}{3.150000in}%
\pgfsys@useobject{currentmarker}{}%
\end{pgfscope}%
\end{pgfscope}%
\begin{pgfscope}%
\pgfpathrectangle{\pgfqpoint{0.000000in}{0.000000in}}{\pgfqpoint{6.000000in}{3.600000in}}%
\pgfusepath{clip}%
\pgfsetbuttcap%
\pgfsetroundjoin%
\definecolor{currentfill}{rgb}{0.800000,0.800000,0.800000}%
\pgfsetfillcolor{currentfill}%
\pgfsetlinewidth{1.003750pt}%
\definecolor{currentstroke}{rgb}{0.000000,0.000000,0.000000}%
\pgfsetstrokecolor{currentstroke}%
\pgfsetdash{}{0pt}%
\pgfsys@defobject{currentmarker}{\pgfqpoint{-0.104167in}{-0.104167in}}{\pgfqpoint{0.104167in}{0.104167in}}{%
\pgfpathmoveto{\pgfqpoint{0.000000in}{-0.104167in}}%
\pgfpathcurveto{\pgfqpoint{0.027625in}{-0.104167in}}{\pgfqpoint{0.054123in}{-0.093191in}}{\pgfqpoint{0.073657in}{-0.073657in}}%
\pgfpathcurveto{\pgfqpoint{0.093191in}{-0.054123in}}{\pgfqpoint{0.104167in}{-0.027625in}}{\pgfqpoint{0.104167in}{0.000000in}}%
\pgfpathcurveto{\pgfqpoint{0.104167in}{0.027625in}}{\pgfqpoint{0.093191in}{0.054123in}}{\pgfqpoint{0.073657in}{0.073657in}}%
\pgfpathcurveto{\pgfqpoint{0.054123in}{0.093191in}}{\pgfqpoint{0.027625in}{0.104167in}}{\pgfqpoint{0.000000in}{0.104167in}}%
\pgfpathcurveto{\pgfqpoint{-0.027625in}{0.104167in}}{\pgfqpoint{-0.054123in}{0.093191in}}{\pgfqpoint{-0.073657in}{0.073657in}}%
\pgfpathcurveto{\pgfqpoint{-0.093191in}{0.054123in}}{\pgfqpoint{-0.104167in}{0.027625in}}{\pgfqpoint{-0.104167in}{0.000000in}}%
\pgfpathcurveto{\pgfqpoint{-0.104167in}{-0.027625in}}{\pgfqpoint{-0.093191in}{-0.054123in}}{\pgfqpoint{-0.073657in}{-0.073657in}}%
\pgfpathcurveto{\pgfqpoint{-0.054123in}{-0.093191in}}{\pgfqpoint{-0.027625in}{-0.104167in}}{\pgfqpoint{0.000000in}{-0.104167in}}%
\pgfpathclose%
\pgfusepath{stroke,fill}%
}%
\begin{pgfscope}%
\pgfsys@transformshift{3.750000in}{3.150000in}%
\pgfsys@useobject{currentmarker}{}%
\end{pgfscope}%
\end{pgfscope}%
\begin{pgfscope}%
\pgfpathrectangle{\pgfqpoint{0.000000in}{0.000000in}}{\pgfqpoint{6.000000in}{3.600000in}}%
\pgfusepath{clip}%
\pgfsetbuttcap%
\pgfsetroundjoin%
\definecolor{currentfill}{rgb}{0.800000,0.800000,0.800000}%
\pgfsetfillcolor{currentfill}%
\pgfsetlinewidth{1.003750pt}%
\definecolor{currentstroke}{rgb}{0.000000,0.000000,0.000000}%
\pgfsetstrokecolor{currentstroke}%
\pgfsetdash{}{0pt}%
\pgfsys@defobject{currentmarker}{\pgfqpoint{-0.104167in}{-0.104167in}}{\pgfqpoint{0.104167in}{0.104167in}}{%
\pgfpathmoveto{\pgfqpoint{0.000000in}{-0.104167in}}%
\pgfpathcurveto{\pgfqpoint{0.027625in}{-0.104167in}}{\pgfqpoint{0.054123in}{-0.093191in}}{\pgfqpoint{0.073657in}{-0.073657in}}%
\pgfpathcurveto{\pgfqpoint{0.093191in}{-0.054123in}}{\pgfqpoint{0.104167in}{-0.027625in}}{\pgfqpoint{0.104167in}{0.000000in}}%
\pgfpathcurveto{\pgfqpoint{0.104167in}{0.027625in}}{\pgfqpoint{0.093191in}{0.054123in}}{\pgfqpoint{0.073657in}{0.073657in}}%
\pgfpathcurveto{\pgfqpoint{0.054123in}{0.093191in}}{\pgfqpoint{0.027625in}{0.104167in}}{\pgfqpoint{0.000000in}{0.104167in}}%
\pgfpathcurveto{\pgfqpoint{-0.027625in}{0.104167in}}{\pgfqpoint{-0.054123in}{0.093191in}}{\pgfqpoint{-0.073657in}{0.073657in}}%
\pgfpathcurveto{\pgfqpoint{-0.093191in}{0.054123in}}{\pgfqpoint{-0.104167in}{0.027625in}}{\pgfqpoint{-0.104167in}{0.000000in}}%
\pgfpathcurveto{\pgfqpoint{-0.104167in}{-0.027625in}}{\pgfqpoint{-0.093191in}{-0.054123in}}{\pgfqpoint{-0.073657in}{-0.073657in}}%
\pgfpathcurveto{\pgfqpoint{-0.054123in}{-0.093191in}}{\pgfqpoint{-0.027625in}{-0.104167in}}{\pgfqpoint{0.000000in}{-0.104167in}}%
\pgfpathclose%
\pgfusepath{stroke,fill}%
}%
\begin{pgfscope}%
\pgfsys@transformshift{0.750000in}{2.250000in}%
\pgfsys@useobject{currentmarker}{}%
\end{pgfscope}%
\end{pgfscope}%
\begin{pgfscope}%
\pgfpathrectangle{\pgfqpoint{0.000000in}{0.000000in}}{\pgfqpoint{6.000000in}{3.600000in}}%
\pgfusepath{clip}%
\pgfsetbuttcap%
\pgfsetroundjoin%
\definecolor{currentfill}{rgb}{1.000000,1.000000,1.000000}%
\pgfsetfillcolor{currentfill}%
\pgfsetlinewidth{1.003750pt}%
\definecolor{currentstroke}{rgb}{0.000000,0.000000,0.000000}%
\pgfsetstrokecolor{currentstroke}%
\pgfsetdash{}{0pt}%
\pgfsys@defobject{currentmarker}{\pgfqpoint{-0.104167in}{-0.104167in}}{\pgfqpoint{0.104167in}{0.104167in}}{%
\pgfpathmoveto{\pgfqpoint{0.000000in}{-0.104167in}}%
\pgfpathcurveto{\pgfqpoint{0.027625in}{-0.104167in}}{\pgfqpoint{0.054123in}{-0.093191in}}{\pgfqpoint{0.073657in}{-0.073657in}}%
\pgfpathcurveto{\pgfqpoint{0.093191in}{-0.054123in}}{\pgfqpoint{0.104167in}{-0.027625in}}{\pgfqpoint{0.104167in}{0.000000in}}%
\pgfpathcurveto{\pgfqpoint{0.104167in}{0.027625in}}{\pgfqpoint{0.093191in}{0.054123in}}{\pgfqpoint{0.073657in}{0.073657in}}%
\pgfpathcurveto{\pgfqpoint{0.054123in}{0.093191in}}{\pgfqpoint{0.027625in}{0.104167in}}{\pgfqpoint{0.000000in}{0.104167in}}%
\pgfpathcurveto{\pgfqpoint{-0.027625in}{0.104167in}}{\pgfqpoint{-0.054123in}{0.093191in}}{\pgfqpoint{-0.073657in}{0.073657in}}%
\pgfpathcurveto{\pgfqpoint{-0.093191in}{0.054123in}}{\pgfqpoint{-0.104167in}{0.027625in}}{\pgfqpoint{-0.104167in}{0.000000in}}%
\pgfpathcurveto{\pgfqpoint{-0.104167in}{-0.027625in}}{\pgfqpoint{-0.093191in}{-0.054123in}}{\pgfqpoint{-0.073657in}{-0.073657in}}%
\pgfpathcurveto{\pgfqpoint{-0.054123in}{-0.093191in}}{\pgfqpoint{-0.027625in}{-0.104167in}}{\pgfqpoint{0.000000in}{-0.104167in}}%
\pgfpathclose%
\pgfusepath{stroke,fill}%
}%
\begin{pgfscope}%
\pgfsys@transformshift{5.250000in}{2.250000in}%
\pgfsys@useobject{currentmarker}{}%
\end{pgfscope}%
\end{pgfscope}%
\begin{pgfscope}%
\pgfpathrectangle{\pgfqpoint{0.000000in}{0.000000in}}{\pgfqpoint{6.000000in}{3.600000in}}%
\pgfusepath{clip}%
\pgfsetbuttcap%
\pgfsetroundjoin%
\definecolor{currentfill}{rgb}{0.800000,0.800000,0.800000}%
\pgfsetfillcolor{currentfill}%
\pgfsetlinewidth{1.003750pt}%
\definecolor{currentstroke}{rgb}{0.000000,0.000000,0.000000}%
\pgfsetstrokecolor{currentstroke}%
\pgfsetdash{}{0pt}%
\pgfsys@defobject{currentmarker}{\pgfqpoint{-0.104167in}{-0.104167in}}{\pgfqpoint{0.104167in}{0.104167in}}{%
\pgfpathmoveto{\pgfqpoint{0.000000in}{-0.104167in}}%
\pgfpathcurveto{\pgfqpoint{0.027625in}{-0.104167in}}{\pgfqpoint{0.054123in}{-0.093191in}}{\pgfqpoint{0.073657in}{-0.073657in}}%
\pgfpathcurveto{\pgfqpoint{0.093191in}{-0.054123in}}{\pgfqpoint{0.104167in}{-0.027625in}}{\pgfqpoint{0.104167in}{0.000000in}}%
\pgfpathcurveto{\pgfqpoint{0.104167in}{0.027625in}}{\pgfqpoint{0.093191in}{0.054123in}}{\pgfqpoint{0.073657in}{0.073657in}}%
\pgfpathcurveto{\pgfqpoint{0.054123in}{0.093191in}}{\pgfqpoint{0.027625in}{0.104167in}}{\pgfqpoint{0.000000in}{0.104167in}}%
\pgfpathcurveto{\pgfqpoint{-0.027625in}{0.104167in}}{\pgfqpoint{-0.054123in}{0.093191in}}{\pgfqpoint{-0.073657in}{0.073657in}}%
\pgfpathcurveto{\pgfqpoint{-0.093191in}{0.054123in}}{\pgfqpoint{-0.104167in}{0.027625in}}{\pgfqpoint{-0.104167in}{0.000000in}}%
\pgfpathcurveto{\pgfqpoint{-0.104167in}{-0.027625in}}{\pgfqpoint{-0.093191in}{-0.054123in}}{\pgfqpoint{-0.073657in}{-0.073657in}}%
\pgfpathcurveto{\pgfqpoint{-0.054123in}{-0.093191in}}{\pgfqpoint{-0.027625in}{-0.104167in}}{\pgfqpoint{0.000000in}{-0.104167in}}%
\pgfpathclose%
\pgfusepath{stroke,fill}%
}%
\begin{pgfscope}%
\pgfsys@transformshift{2.250000in}{1.800000in}%
\pgfsys@useobject{currentmarker}{}%
\end{pgfscope}%
\end{pgfscope}%
\begin{pgfscope}%
\pgfpathrectangle{\pgfqpoint{0.000000in}{0.000000in}}{\pgfqpoint{6.000000in}{3.600000in}}%
\pgfusepath{clip}%
\pgfsetbuttcap%
\pgfsetroundjoin%
\definecolor{currentfill}{rgb}{1.000000,1.000000,1.000000}%
\pgfsetfillcolor{currentfill}%
\pgfsetlinewidth{1.003750pt}%
\definecolor{currentstroke}{rgb}{0.000000,0.000000,0.000000}%
\pgfsetstrokecolor{currentstroke}%
\pgfsetdash{}{0pt}%
\pgfsys@defobject{currentmarker}{\pgfqpoint{-0.104167in}{-0.104167in}}{\pgfqpoint{0.104167in}{0.104167in}}{%
\pgfpathmoveto{\pgfqpoint{0.000000in}{-0.104167in}}%
\pgfpathcurveto{\pgfqpoint{0.027625in}{-0.104167in}}{\pgfqpoint{0.054123in}{-0.093191in}}{\pgfqpoint{0.073657in}{-0.073657in}}%
\pgfpathcurveto{\pgfqpoint{0.093191in}{-0.054123in}}{\pgfqpoint{0.104167in}{-0.027625in}}{\pgfqpoint{0.104167in}{0.000000in}}%
\pgfpathcurveto{\pgfqpoint{0.104167in}{0.027625in}}{\pgfqpoint{0.093191in}{0.054123in}}{\pgfqpoint{0.073657in}{0.073657in}}%
\pgfpathcurveto{\pgfqpoint{0.054123in}{0.093191in}}{\pgfqpoint{0.027625in}{0.104167in}}{\pgfqpoint{0.000000in}{0.104167in}}%
\pgfpathcurveto{\pgfqpoint{-0.027625in}{0.104167in}}{\pgfqpoint{-0.054123in}{0.093191in}}{\pgfqpoint{-0.073657in}{0.073657in}}%
\pgfpathcurveto{\pgfqpoint{-0.093191in}{0.054123in}}{\pgfqpoint{-0.104167in}{0.027625in}}{\pgfqpoint{-0.104167in}{0.000000in}}%
\pgfpathcurveto{\pgfqpoint{-0.104167in}{-0.027625in}}{\pgfqpoint{-0.093191in}{-0.054123in}}{\pgfqpoint{-0.073657in}{-0.073657in}}%
\pgfpathcurveto{\pgfqpoint{-0.054123in}{-0.093191in}}{\pgfqpoint{-0.027625in}{-0.104167in}}{\pgfqpoint{0.000000in}{-0.104167in}}%
\pgfpathclose%
\pgfusepath{stroke,fill}%
}%
\begin{pgfscope}%
\pgfsys@transformshift{3.750000in}{1.800000in}%
\pgfsys@useobject{currentmarker}{}%
\end{pgfscope}%
\end{pgfscope}%
\begin{pgfscope}%
\pgfpathrectangle{\pgfqpoint{0.000000in}{0.000000in}}{\pgfqpoint{6.000000in}{3.600000in}}%
\pgfusepath{clip}%
\pgfsetbuttcap%
\pgfsetroundjoin%
\definecolor{currentfill}{rgb}{0.800000,0.800000,0.800000}%
\pgfsetfillcolor{currentfill}%
\pgfsetlinewidth{1.003750pt}%
\definecolor{currentstroke}{rgb}{0.000000,0.000000,0.000000}%
\pgfsetstrokecolor{currentstroke}%
\pgfsetdash{}{0pt}%
\pgfsys@defobject{currentmarker}{\pgfqpoint{-0.104167in}{-0.104167in}}{\pgfqpoint{0.104167in}{0.104167in}}{%
\pgfpathmoveto{\pgfqpoint{0.000000in}{-0.104167in}}%
\pgfpathcurveto{\pgfqpoint{0.027625in}{-0.104167in}}{\pgfqpoint{0.054123in}{-0.093191in}}{\pgfqpoint{0.073657in}{-0.073657in}}%
\pgfpathcurveto{\pgfqpoint{0.093191in}{-0.054123in}}{\pgfqpoint{0.104167in}{-0.027625in}}{\pgfqpoint{0.104167in}{0.000000in}}%
\pgfpathcurveto{\pgfqpoint{0.104167in}{0.027625in}}{\pgfqpoint{0.093191in}{0.054123in}}{\pgfqpoint{0.073657in}{0.073657in}}%
\pgfpathcurveto{\pgfqpoint{0.054123in}{0.093191in}}{\pgfqpoint{0.027625in}{0.104167in}}{\pgfqpoint{0.000000in}{0.104167in}}%
\pgfpathcurveto{\pgfqpoint{-0.027625in}{0.104167in}}{\pgfqpoint{-0.054123in}{0.093191in}}{\pgfqpoint{-0.073657in}{0.073657in}}%
\pgfpathcurveto{\pgfqpoint{-0.093191in}{0.054123in}}{\pgfqpoint{-0.104167in}{0.027625in}}{\pgfqpoint{-0.104167in}{0.000000in}}%
\pgfpathcurveto{\pgfqpoint{-0.104167in}{-0.027625in}}{\pgfqpoint{-0.093191in}{-0.054123in}}{\pgfqpoint{-0.073657in}{-0.073657in}}%
\pgfpathcurveto{\pgfqpoint{-0.054123in}{-0.093191in}}{\pgfqpoint{-0.027625in}{-0.104167in}}{\pgfqpoint{0.000000in}{-0.104167in}}%
\pgfpathclose%
\pgfusepath{stroke,fill}%
}%
\begin{pgfscope}%
\pgfsys@transformshift{0.750000in}{1.350000in}%
\pgfsys@useobject{currentmarker}{}%
\end{pgfscope}%
\end{pgfscope}%
\begin{pgfscope}%
\pgfpathrectangle{\pgfqpoint{0.000000in}{0.000000in}}{\pgfqpoint{6.000000in}{3.600000in}}%
\pgfusepath{clip}%
\pgfsetbuttcap%
\pgfsetroundjoin%
\definecolor{currentfill}{rgb}{1.000000,1.000000,1.000000}%
\pgfsetfillcolor{currentfill}%
\pgfsetlinewidth{1.003750pt}%
\definecolor{currentstroke}{rgb}{0.000000,0.000000,0.000000}%
\pgfsetstrokecolor{currentstroke}%
\pgfsetdash{}{0pt}%
\pgfsys@defobject{currentmarker}{\pgfqpoint{-0.104167in}{-0.104167in}}{\pgfqpoint{0.104167in}{0.104167in}}{%
\pgfpathmoveto{\pgfqpoint{0.000000in}{-0.104167in}}%
\pgfpathcurveto{\pgfqpoint{0.027625in}{-0.104167in}}{\pgfqpoint{0.054123in}{-0.093191in}}{\pgfqpoint{0.073657in}{-0.073657in}}%
\pgfpathcurveto{\pgfqpoint{0.093191in}{-0.054123in}}{\pgfqpoint{0.104167in}{-0.027625in}}{\pgfqpoint{0.104167in}{0.000000in}}%
\pgfpathcurveto{\pgfqpoint{0.104167in}{0.027625in}}{\pgfqpoint{0.093191in}{0.054123in}}{\pgfqpoint{0.073657in}{0.073657in}}%
\pgfpathcurveto{\pgfqpoint{0.054123in}{0.093191in}}{\pgfqpoint{0.027625in}{0.104167in}}{\pgfqpoint{0.000000in}{0.104167in}}%
\pgfpathcurveto{\pgfqpoint{-0.027625in}{0.104167in}}{\pgfqpoint{-0.054123in}{0.093191in}}{\pgfqpoint{-0.073657in}{0.073657in}}%
\pgfpathcurveto{\pgfqpoint{-0.093191in}{0.054123in}}{\pgfqpoint{-0.104167in}{0.027625in}}{\pgfqpoint{-0.104167in}{0.000000in}}%
\pgfpathcurveto{\pgfqpoint{-0.104167in}{-0.027625in}}{\pgfqpoint{-0.093191in}{-0.054123in}}{\pgfqpoint{-0.073657in}{-0.073657in}}%
\pgfpathcurveto{\pgfqpoint{-0.054123in}{-0.093191in}}{\pgfqpoint{-0.027625in}{-0.104167in}}{\pgfqpoint{0.000000in}{-0.104167in}}%
\pgfpathclose%
\pgfusepath{stroke,fill}%
}%
\begin{pgfscope}%
\pgfsys@transformshift{5.250000in}{1.350000in}%
\pgfsys@useobject{currentmarker}{}%
\end{pgfscope}%
\end{pgfscope}%
\begin{pgfscope}%
\pgfpathrectangle{\pgfqpoint{0.000000in}{0.000000in}}{\pgfqpoint{6.000000in}{3.600000in}}%
\pgfusepath{clip}%
\pgfsetbuttcap%
\pgfsetroundjoin%
\definecolor{currentfill}{rgb}{0.800000,0.800000,0.800000}%
\pgfsetfillcolor{currentfill}%
\pgfsetlinewidth{1.003750pt}%
\definecolor{currentstroke}{rgb}{0.000000,0.000000,0.000000}%
\pgfsetstrokecolor{currentstroke}%
\pgfsetdash{}{0pt}%
\pgfsys@defobject{currentmarker}{\pgfqpoint{-0.104167in}{-0.104167in}}{\pgfqpoint{0.104167in}{0.104167in}}{%
\pgfpathmoveto{\pgfqpoint{0.000000in}{-0.104167in}}%
\pgfpathcurveto{\pgfqpoint{0.027625in}{-0.104167in}}{\pgfqpoint{0.054123in}{-0.093191in}}{\pgfqpoint{0.073657in}{-0.073657in}}%
\pgfpathcurveto{\pgfqpoint{0.093191in}{-0.054123in}}{\pgfqpoint{0.104167in}{-0.027625in}}{\pgfqpoint{0.104167in}{0.000000in}}%
\pgfpathcurveto{\pgfqpoint{0.104167in}{0.027625in}}{\pgfqpoint{0.093191in}{0.054123in}}{\pgfqpoint{0.073657in}{0.073657in}}%
\pgfpathcurveto{\pgfqpoint{0.054123in}{0.093191in}}{\pgfqpoint{0.027625in}{0.104167in}}{\pgfqpoint{0.000000in}{0.104167in}}%
\pgfpathcurveto{\pgfqpoint{-0.027625in}{0.104167in}}{\pgfqpoint{-0.054123in}{0.093191in}}{\pgfqpoint{-0.073657in}{0.073657in}}%
\pgfpathcurveto{\pgfqpoint{-0.093191in}{0.054123in}}{\pgfqpoint{-0.104167in}{0.027625in}}{\pgfqpoint{-0.104167in}{0.000000in}}%
\pgfpathcurveto{\pgfqpoint{-0.104167in}{-0.027625in}}{\pgfqpoint{-0.093191in}{-0.054123in}}{\pgfqpoint{-0.073657in}{-0.073657in}}%
\pgfpathcurveto{\pgfqpoint{-0.054123in}{-0.093191in}}{\pgfqpoint{-0.027625in}{-0.104167in}}{\pgfqpoint{0.000000in}{-0.104167in}}%
\pgfpathclose%
\pgfusepath{stroke,fill}%
}%
\begin{pgfscope}%
\pgfsys@transformshift{2.250000in}{0.450000in}%
\pgfsys@useobject{currentmarker}{}%
\end{pgfscope}%
\end{pgfscope}%
\begin{pgfscope}%
\pgfpathrectangle{\pgfqpoint{0.000000in}{0.000000in}}{\pgfqpoint{6.000000in}{3.600000in}}%
\pgfusepath{clip}%
\pgfsetbuttcap%
\pgfsetroundjoin%
\definecolor{currentfill}{rgb}{1.000000,1.000000,1.000000}%
\pgfsetfillcolor{currentfill}%
\pgfsetlinewidth{1.003750pt}%
\definecolor{currentstroke}{rgb}{0.000000,0.000000,0.000000}%
\pgfsetstrokecolor{currentstroke}%
\pgfsetdash{}{0pt}%
\pgfsys@defobject{currentmarker}{\pgfqpoint{-0.104167in}{-0.104167in}}{\pgfqpoint{0.104167in}{0.104167in}}{%
\pgfpathmoveto{\pgfqpoint{0.000000in}{-0.104167in}}%
\pgfpathcurveto{\pgfqpoint{0.027625in}{-0.104167in}}{\pgfqpoint{0.054123in}{-0.093191in}}{\pgfqpoint{0.073657in}{-0.073657in}}%
\pgfpathcurveto{\pgfqpoint{0.093191in}{-0.054123in}}{\pgfqpoint{0.104167in}{-0.027625in}}{\pgfqpoint{0.104167in}{0.000000in}}%
\pgfpathcurveto{\pgfqpoint{0.104167in}{0.027625in}}{\pgfqpoint{0.093191in}{0.054123in}}{\pgfqpoint{0.073657in}{0.073657in}}%
\pgfpathcurveto{\pgfqpoint{0.054123in}{0.093191in}}{\pgfqpoint{0.027625in}{0.104167in}}{\pgfqpoint{0.000000in}{0.104167in}}%
\pgfpathcurveto{\pgfqpoint{-0.027625in}{0.104167in}}{\pgfqpoint{-0.054123in}{0.093191in}}{\pgfqpoint{-0.073657in}{0.073657in}}%
\pgfpathcurveto{\pgfqpoint{-0.093191in}{0.054123in}}{\pgfqpoint{-0.104167in}{0.027625in}}{\pgfqpoint{-0.104167in}{0.000000in}}%
\pgfpathcurveto{\pgfqpoint{-0.104167in}{-0.027625in}}{\pgfqpoint{-0.093191in}{-0.054123in}}{\pgfqpoint{-0.073657in}{-0.073657in}}%
\pgfpathcurveto{\pgfqpoint{-0.054123in}{-0.093191in}}{\pgfqpoint{-0.027625in}{-0.104167in}}{\pgfqpoint{0.000000in}{-0.104167in}}%
\pgfpathclose%
\pgfusepath{stroke,fill}%
}%
\begin{pgfscope}%
\pgfsys@transformshift{3.750000in}{0.450000in}%
\pgfsys@useobject{currentmarker}{}%
\end{pgfscope}%
\end{pgfscope}%
\begin{pgfscope}%
\definecolor{textcolor}{rgb}{0.000000,0.000000,0.000000}%
\pgfsetstrokecolor{textcolor}%
\pgfsetfillcolor{textcolor}%
\pgftext[x=1.062850in,y=1.350000in,,base]{\color{textcolor}\sffamily\fontsize{40.000000}{48.000000}\selectfont \(\displaystyle u_1\)}%
\end{pgfscope}%
\begin{pgfscope}%
\definecolor{textcolor}{rgb}{0.000000,0.000000,0.000000}%
\pgfsetstrokecolor{textcolor}%
\pgfsetfillcolor{textcolor}%
\pgftext[x=2.529863in,y=0.170137in,,base]{\color{textcolor}\sffamily\fontsize{40.000000}{48.000000}\selectfont \(\displaystyle v_1\)}%
\end{pgfscope}%
\end{pgfpicture}%
\makeatother%
\endgroup%

%% file: fig/decomposition_H_U2.pgf
%% Creator: Matplotlib, PGF backend
%%
%% To include the figure in your LaTeX document, write
%%   \input{<filename>.pgf}
%%
%% Make sure the required packages are loaded in your preamble
%%   \usepackage{pgf}
%%
%% Figures using additional raster images can only be included by \input if
%% they are in the same directory as the main LaTeX file. For loading figures
%% from other directories you can use the `import` package
%%   \usepackage{import}
%%
%% and then include the figures with
%%   \import{<path to file>}{<filename>.pgf}
%%
%% Matplotlib used the following preamble
%%   \usepackage{fontspec}
%%   \setmainfont{DejaVuSerif.ttf}[Path=\detokenize{C:/Users/ccros/Anaconda3/Lib/site-packages/matplotlib/mpl-data/fonts/ttf/}]
%%   \setsansfont{DejaVuSans.ttf}[Path=\detokenize{C:/Users/ccros/Anaconda3/Lib/site-packages/matplotlib/mpl-data/fonts/ttf/}]
%%   \setmonofont{DejaVuSansMono.ttf}[Path=\detokenize{C:/Users/ccros/Anaconda3/Lib/site-packages/matplotlib/mpl-data/fonts/ttf/}]
%%
\begingroup%
\makeatletter%
\begin{pgfpicture}%
\pgfpathrectangle{\pgfpointorigin}{\pgfqpoint{6.000000in}{3.600000in}}%
\pgfusepath{use as bounding box, clip}%
\begin{pgfscope}%
\pgfsetbuttcap%
\pgfsetmiterjoin%
\definecolor{currentfill}{rgb}{1.000000,1.000000,1.000000}%
\pgfsetfillcolor{currentfill}%
\pgfsetlinewidth{0.000000pt}%
\definecolor{currentstroke}{rgb}{1.000000,1.000000,1.000000}%
\pgfsetstrokecolor{currentstroke}%
\pgfsetdash{}{0pt}%
\pgfpathmoveto{\pgfqpoint{0.000000in}{0.000000in}}%
\pgfpathlineto{\pgfqpoint{6.000000in}{0.000000in}}%
\pgfpathlineto{\pgfqpoint{6.000000in}{3.600000in}}%
\pgfpathlineto{\pgfqpoint{0.000000in}{3.600000in}}%
\pgfpathclose%
\pgfusepath{fill}%
\end{pgfscope}%
\begin{pgfscope}%
\pgfpathrectangle{\pgfqpoint{0.000000in}{0.000000in}}{\pgfqpoint{6.000000in}{3.600000in}}%
\pgfusepath{clip}%
\pgfsetbuttcap%
\pgfsetroundjoin%
\pgfsetlinewidth{3.011250pt}%
\definecolor{currentstroke}{rgb}{0.400000,0.400000,0.400000}%
\pgfsetstrokecolor{currentstroke}%
\pgfsetdash{{11.100000pt}{4.800000pt}}{0.000000pt}%
\pgfpathmoveto{\pgfqpoint{5.250000in}{2.250000in}}%
\pgfpathlineto{\pgfqpoint{3.750000in}{1.800000in}}%
\pgfusepath{stroke}%
\end{pgfscope}%
\begin{pgfscope}%
\pgfpathrectangle{\pgfqpoint{0.000000in}{0.000000in}}{\pgfqpoint{6.000000in}{3.600000in}}%
\pgfusepath{clip}%
\pgfsetbuttcap%
\pgfsetroundjoin%
\pgfsetlinewidth{3.011250pt}%
\definecolor{currentstroke}{rgb}{0.400000,0.400000,0.400000}%
\pgfsetstrokecolor{currentstroke}%
\pgfsetdash{{11.100000pt}{4.800000pt}}{0.000000pt}%
\pgfpathmoveto{\pgfqpoint{2.250000in}{1.800000in}}%
\pgfpathlineto{\pgfqpoint{3.750000in}{1.800000in}}%
\pgfusepath{stroke}%
\end{pgfscope}%
\begin{pgfscope}%
\pgfpathrectangle{\pgfqpoint{0.000000in}{0.000000in}}{\pgfqpoint{6.000000in}{3.600000in}}%
\pgfusepath{clip}%
\pgfsetbuttcap%
\pgfsetroundjoin%
\pgfsetlinewidth{3.011250pt}%
\definecolor{currentstroke}{rgb}{0.400000,0.400000,0.400000}%
\pgfsetstrokecolor{currentstroke}%
\pgfsetdash{{11.100000pt}{4.800000pt}}{0.000000pt}%
\pgfpathmoveto{\pgfqpoint{0.750000in}{1.350000in}}%
\pgfpathlineto{\pgfqpoint{2.250000in}{0.450000in}}%
\pgfusepath{stroke}%
\end{pgfscope}%
\begin{pgfscope}%
\pgfpathrectangle{\pgfqpoint{0.000000in}{0.000000in}}{\pgfqpoint{6.000000in}{3.600000in}}%
\pgfusepath{clip}%
\pgfsetbuttcap%
\pgfsetroundjoin%
\pgfsetlinewidth{3.011250pt}%
\definecolor{currentstroke}{rgb}{0.400000,0.400000,0.400000}%
\pgfsetstrokecolor{currentstroke}%
\pgfsetdash{{11.100000pt}{4.800000pt}}{0.000000pt}%
\pgfpathmoveto{\pgfqpoint{5.250000in}{1.350000in}}%
\pgfpathlineto{\pgfqpoint{3.750000in}{0.450000in}}%
\pgfusepath{stroke}%
\end{pgfscope}%
\begin{pgfscope}%
\pgfpathrectangle{\pgfqpoint{0.000000in}{0.000000in}}{\pgfqpoint{6.000000in}{3.600000in}}%
\pgfusepath{clip}%
\pgfsetrectcap%
\pgfsetroundjoin%
\pgfsetlinewidth{3.011250pt}%
\definecolor{currentstroke}{rgb}{0.250980,0.250980,0.250980}%
\pgfsetstrokecolor{currentstroke}%
\pgfsetdash{}{0pt}%
\pgfpathmoveto{\pgfqpoint{3.750000in}{3.150000in}}%
\pgfpathlineto{\pgfqpoint{2.250000in}{3.150000in}}%
\pgfusepath{stroke}%
\end{pgfscope}%
\begin{pgfscope}%
\pgfpathrectangle{\pgfqpoint{0.000000in}{0.000000in}}{\pgfqpoint{6.000000in}{3.600000in}}%
\pgfusepath{clip}%
\pgfsetrectcap%
\pgfsetroundjoin%
\pgfsetlinewidth{3.011250pt}%
\definecolor{currentstroke}{rgb}{0.250980,0.250980,0.250980}%
\pgfsetstrokecolor{currentstroke}%
\pgfsetdash{}{0pt}%
\pgfpathmoveto{\pgfqpoint{0.750000in}{2.250000in}}%
\pgfpathlineto{\pgfqpoint{2.250000in}{3.150000in}}%
\pgfusepath{stroke}%
\end{pgfscope}%
\begin{pgfscope}%
\pgfpathrectangle{\pgfqpoint{0.000000in}{0.000000in}}{\pgfqpoint{6.000000in}{3.600000in}}%
\pgfusepath{clip}%
\pgfsetrectcap%
\pgfsetroundjoin%
\pgfsetlinewidth{3.011250pt}%
\definecolor{currentstroke}{rgb}{0.250980,0.250980,0.250980}%
\pgfsetstrokecolor{currentstroke}%
\pgfsetdash{}{0pt}%
\pgfpathmoveto{\pgfqpoint{2.250000in}{1.800000in}}%
\pgfpathlineto{\pgfqpoint{2.250000in}{3.150000in}}%
\pgfusepath{stroke}%
\end{pgfscope}%
\begin{pgfscope}%
\pgfpathrectangle{\pgfqpoint{0.000000in}{0.000000in}}{\pgfqpoint{6.000000in}{3.600000in}}%
\pgfusepath{clip}%
\pgfsetrectcap%
\pgfsetroundjoin%
\pgfsetlinewidth{3.011250pt}%
\definecolor{currentstroke}{rgb}{0.250980,0.250980,0.250980}%
\pgfsetstrokecolor{currentstroke}%
\pgfsetdash{}{0pt}%
\pgfpathmoveto{\pgfqpoint{2.250000in}{1.800000in}}%
\pgfpathlineto{\pgfqpoint{0.750000in}{2.250000in}}%
\pgfusepath{stroke}%
\end{pgfscope}%
\begin{pgfscope}%
\pgfpathrectangle{\pgfqpoint{0.000000in}{0.000000in}}{\pgfqpoint{6.000000in}{3.600000in}}%
\pgfusepath{clip}%
\pgfsetrectcap%
\pgfsetroundjoin%
\pgfsetlinewidth{3.011250pt}%
\definecolor{currentstroke}{rgb}{0.250980,0.250980,0.250980}%
\pgfsetstrokecolor{currentstroke}%
\pgfsetdash{}{0pt}%
\pgfpathmoveto{\pgfqpoint{0.750000in}{1.350000in}}%
\pgfpathlineto{\pgfqpoint{0.750000in}{2.250000in}}%
\pgfusepath{stroke}%
\end{pgfscope}%
\begin{pgfscope}%
\pgfpathrectangle{\pgfqpoint{0.000000in}{0.000000in}}{\pgfqpoint{6.000000in}{3.600000in}}%
\pgfusepath{clip}%
\pgfsetbuttcap%
\pgfsetroundjoin%
\definecolor{currentfill}{rgb}{0.800000,0.800000,0.800000}%
\pgfsetfillcolor{currentfill}%
\pgfsetlinewidth{1.003750pt}%
\definecolor{currentstroke}{rgb}{0.000000,0.000000,0.000000}%
\pgfsetstrokecolor{currentstroke}%
\pgfsetdash{}{0pt}%
\pgfsys@defobject{currentmarker}{\pgfqpoint{-0.104167in}{-0.104167in}}{\pgfqpoint{0.104167in}{0.104167in}}{%
\pgfpathmoveto{\pgfqpoint{0.000000in}{-0.104167in}}%
\pgfpathcurveto{\pgfqpoint{0.027625in}{-0.104167in}}{\pgfqpoint{0.054123in}{-0.093191in}}{\pgfqpoint{0.073657in}{-0.073657in}}%
\pgfpathcurveto{\pgfqpoint{0.093191in}{-0.054123in}}{\pgfqpoint{0.104167in}{-0.027625in}}{\pgfqpoint{0.104167in}{0.000000in}}%
\pgfpathcurveto{\pgfqpoint{0.104167in}{0.027625in}}{\pgfqpoint{0.093191in}{0.054123in}}{\pgfqpoint{0.073657in}{0.073657in}}%
\pgfpathcurveto{\pgfqpoint{0.054123in}{0.093191in}}{\pgfqpoint{0.027625in}{0.104167in}}{\pgfqpoint{0.000000in}{0.104167in}}%
\pgfpathcurveto{\pgfqpoint{-0.027625in}{0.104167in}}{\pgfqpoint{-0.054123in}{0.093191in}}{\pgfqpoint{-0.073657in}{0.073657in}}%
\pgfpathcurveto{\pgfqpoint{-0.093191in}{0.054123in}}{\pgfqpoint{-0.104167in}{0.027625in}}{\pgfqpoint{-0.104167in}{0.000000in}}%
\pgfpathcurveto{\pgfqpoint{-0.104167in}{-0.027625in}}{\pgfqpoint{-0.093191in}{-0.054123in}}{\pgfqpoint{-0.073657in}{-0.073657in}}%
\pgfpathcurveto{\pgfqpoint{-0.054123in}{-0.093191in}}{\pgfqpoint{-0.027625in}{-0.104167in}}{\pgfqpoint{0.000000in}{-0.104167in}}%
\pgfpathclose%
\pgfusepath{stroke,fill}%
}%
\begin{pgfscope}%
\pgfsys@transformshift{2.250000in}{3.150000in}%
\pgfsys@useobject{currentmarker}{}%
\end{pgfscope}%
\end{pgfscope}%
\begin{pgfscope}%
\pgfpathrectangle{\pgfqpoint{0.000000in}{0.000000in}}{\pgfqpoint{6.000000in}{3.600000in}}%
\pgfusepath{clip}%
\pgfsetbuttcap%
\pgfsetroundjoin%
\definecolor{currentfill}{rgb}{0.800000,0.800000,0.800000}%
\pgfsetfillcolor{currentfill}%
\pgfsetlinewidth{1.003750pt}%
\definecolor{currentstroke}{rgb}{0.000000,0.000000,0.000000}%
\pgfsetstrokecolor{currentstroke}%
\pgfsetdash{}{0pt}%
\pgfsys@defobject{currentmarker}{\pgfqpoint{-0.104167in}{-0.104167in}}{\pgfqpoint{0.104167in}{0.104167in}}{%
\pgfpathmoveto{\pgfqpoint{0.000000in}{-0.104167in}}%
\pgfpathcurveto{\pgfqpoint{0.027625in}{-0.104167in}}{\pgfqpoint{0.054123in}{-0.093191in}}{\pgfqpoint{0.073657in}{-0.073657in}}%
\pgfpathcurveto{\pgfqpoint{0.093191in}{-0.054123in}}{\pgfqpoint{0.104167in}{-0.027625in}}{\pgfqpoint{0.104167in}{0.000000in}}%
\pgfpathcurveto{\pgfqpoint{0.104167in}{0.027625in}}{\pgfqpoint{0.093191in}{0.054123in}}{\pgfqpoint{0.073657in}{0.073657in}}%
\pgfpathcurveto{\pgfqpoint{0.054123in}{0.093191in}}{\pgfqpoint{0.027625in}{0.104167in}}{\pgfqpoint{0.000000in}{0.104167in}}%
\pgfpathcurveto{\pgfqpoint{-0.027625in}{0.104167in}}{\pgfqpoint{-0.054123in}{0.093191in}}{\pgfqpoint{-0.073657in}{0.073657in}}%
\pgfpathcurveto{\pgfqpoint{-0.093191in}{0.054123in}}{\pgfqpoint{-0.104167in}{0.027625in}}{\pgfqpoint{-0.104167in}{0.000000in}}%
\pgfpathcurveto{\pgfqpoint{-0.104167in}{-0.027625in}}{\pgfqpoint{-0.093191in}{-0.054123in}}{\pgfqpoint{-0.073657in}{-0.073657in}}%
\pgfpathcurveto{\pgfqpoint{-0.054123in}{-0.093191in}}{\pgfqpoint{-0.027625in}{-0.104167in}}{\pgfqpoint{0.000000in}{-0.104167in}}%
\pgfpathclose%
\pgfusepath{stroke,fill}%
}%
\begin{pgfscope}%
\pgfsys@transformshift{3.750000in}{3.150000in}%
\pgfsys@useobject{currentmarker}{}%
\end{pgfscope}%
\end{pgfscope}%
\begin{pgfscope}%
\pgfpathrectangle{\pgfqpoint{0.000000in}{0.000000in}}{\pgfqpoint{6.000000in}{3.600000in}}%
\pgfusepath{clip}%
\pgfsetbuttcap%
\pgfsetroundjoin%
\definecolor{currentfill}{rgb}{0.800000,0.800000,0.800000}%
\pgfsetfillcolor{currentfill}%
\pgfsetlinewidth{1.003750pt}%
\definecolor{currentstroke}{rgb}{0.000000,0.000000,0.000000}%
\pgfsetstrokecolor{currentstroke}%
\pgfsetdash{}{0pt}%
\pgfsys@defobject{currentmarker}{\pgfqpoint{-0.104167in}{-0.104167in}}{\pgfqpoint{0.104167in}{0.104167in}}{%
\pgfpathmoveto{\pgfqpoint{0.000000in}{-0.104167in}}%
\pgfpathcurveto{\pgfqpoint{0.027625in}{-0.104167in}}{\pgfqpoint{0.054123in}{-0.093191in}}{\pgfqpoint{0.073657in}{-0.073657in}}%
\pgfpathcurveto{\pgfqpoint{0.093191in}{-0.054123in}}{\pgfqpoint{0.104167in}{-0.027625in}}{\pgfqpoint{0.104167in}{0.000000in}}%
\pgfpathcurveto{\pgfqpoint{0.104167in}{0.027625in}}{\pgfqpoint{0.093191in}{0.054123in}}{\pgfqpoint{0.073657in}{0.073657in}}%
\pgfpathcurveto{\pgfqpoint{0.054123in}{0.093191in}}{\pgfqpoint{0.027625in}{0.104167in}}{\pgfqpoint{0.000000in}{0.104167in}}%
\pgfpathcurveto{\pgfqpoint{-0.027625in}{0.104167in}}{\pgfqpoint{-0.054123in}{0.093191in}}{\pgfqpoint{-0.073657in}{0.073657in}}%
\pgfpathcurveto{\pgfqpoint{-0.093191in}{0.054123in}}{\pgfqpoint{-0.104167in}{0.027625in}}{\pgfqpoint{-0.104167in}{0.000000in}}%
\pgfpathcurveto{\pgfqpoint{-0.104167in}{-0.027625in}}{\pgfqpoint{-0.093191in}{-0.054123in}}{\pgfqpoint{-0.073657in}{-0.073657in}}%
\pgfpathcurveto{\pgfqpoint{-0.054123in}{-0.093191in}}{\pgfqpoint{-0.027625in}{-0.104167in}}{\pgfqpoint{0.000000in}{-0.104167in}}%
\pgfpathclose%
\pgfusepath{stroke,fill}%
}%
\begin{pgfscope}%
\pgfsys@transformshift{0.750000in}{2.250000in}%
\pgfsys@useobject{currentmarker}{}%
\end{pgfscope}%
\end{pgfscope}%
\begin{pgfscope}%
\pgfpathrectangle{\pgfqpoint{0.000000in}{0.000000in}}{\pgfqpoint{6.000000in}{3.600000in}}%
\pgfusepath{clip}%
\pgfsetbuttcap%
\pgfsetroundjoin%
\definecolor{currentfill}{rgb}{1.000000,1.000000,1.000000}%
\pgfsetfillcolor{currentfill}%
\pgfsetlinewidth{1.003750pt}%
\definecolor{currentstroke}{rgb}{0.000000,0.000000,0.000000}%
\pgfsetstrokecolor{currentstroke}%
\pgfsetdash{}{0pt}%
\pgfsys@defobject{currentmarker}{\pgfqpoint{-0.104167in}{-0.104167in}}{\pgfqpoint{0.104167in}{0.104167in}}{%
\pgfpathmoveto{\pgfqpoint{0.000000in}{-0.104167in}}%
\pgfpathcurveto{\pgfqpoint{0.027625in}{-0.104167in}}{\pgfqpoint{0.054123in}{-0.093191in}}{\pgfqpoint{0.073657in}{-0.073657in}}%
\pgfpathcurveto{\pgfqpoint{0.093191in}{-0.054123in}}{\pgfqpoint{0.104167in}{-0.027625in}}{\pgfqpoint{0.104167in}{0.000000in}}%
\pgfpathcurveto{\pgfqpoint{0.104167in}{0.027625in}}{\pgfqpoint{0.093191in}{0.054123in}}{\pgfqpoint{0.073657in}{0.073657in}}%
\pgfpathcurveto{\pgfqpoint{0.054123in}{0.093191in}}{\pgfqpoint{0.027625in}{0.104167in}}{\pgfqpoint{0.000000in}{0.104167in}}%
\pgfpathcurveto{\pgfqpoint{-0.027625in}{0.104167in}}{\pgfqpoint{-0.054123in}{0.093191in}}{\pgfqpoint{-0.073657in}{0.073657in}}%
\pgfpathcurveto{\pgfqpoint{-0.093191in}{0.054123in}}{\pgfqpoint{-0.104167in}{0.027625in}}{\pgfqpoint{-0.104167in}{0.000000in}}%
\pgfpathcurveto{\pgfqpoint{-0.104167in}{-0.027625in}}{\pgfqpoint{-0.093191in}{-0.054123in}}{\pgfqpoint{-0.073657in}{-0.073657in}}%
\pgfpathcurveto{\pgfqpoint{-0.054123in}{-0.093191in}}{\pgfqpoint{-0.027625in}{-0.104167in}}{\pgfqpoint{0.000000in}{-0.104167in}}%
\pgfpathclose%
\pgfusepath{stroke,fill}%
}%
\begin{pgfscope}%
\pgfsys@transformshift{5.250000in}{2.250000in}%
\pgfsys@useobject{currentmarker}{}%
\end{pgfscope}%
\end{pgfscope}%
\begin{pgfscope}%
\pgfpathrectangle{\pgfqpoint{0.000000in}{0.000000in}}{\pgfqpoint{6.000000in}{3.600000in}}%
\pgfusepath{clip}%
\pgfsetbuttcap%
\pgfsetroundjoin%
\definecolor{currentfill}{rgb}{0.800000,0.800000,0.800000}%
\pgfsetfillcolor{currentfill}%
\pgfsetlinewidth{1.003750pt}%
\definecolor{currentstroke}{rgb}{0.000000,0.000000,0.000000}%
\pgfsetstrokecolor{currentstroke}%
\pgfsetdash{}{0pt}%
\pgfsys@defobject{currentmarker}{\pgfqpoint{-0.104167in}{-0.104167in}}{\pgfqpoint{0.104167in}{0.104167in}}{%
\pgfpathmoveto{\pgfqpoint{0.000000in}{-0.104167in}}%
\pgfpathcurveto{\pgfqpoint{0.027625in}{-0.104167in}}{\pgfqpoint{0.054123in}{-0.093191in}}{\pgfqpoint{0.073657in}{-0.073657in}}%
\pgfpathcurveto{\pgfqpoint{0.093191in}{-0.054123in}}{\pgfqpoint{0.104167in}{-0.027625in}}{\pgfqpoint{0.104167in}{0.000000in}}%
\pgfpathcurveto{\pgfqpoint{0.104167in}{0.027625in}}{\pgfqpoint{0.093191in}{0.054123in}}{\pgfqpoint{0.073657in}{0.073657in}}%
\pgfpathcurveto{\pgfqpoint{0.054123in}{0.093191in}}{\pgfqpoint{0.027625in}{0.104167in}}{\pgfqpoint{0.000000in}{0.104167in}}%
\pgfpathcurveto{\pgfqpoint{-0.027625in}{0.104167in}}{\pgfqpoint{-0.054123in}{0.093191in}}{\pgfqpoint{-0.073657in}{0.073657in}}%
\pgfpathcurveto{\pgfqpoint{-0.093191in}{0.054123in}}{\pgfqpoint{-0.104167in}{0.027625in}}{\pgfqpoint{-0.104167in}{0.000000in}}%
\pgfpathcurveto{\pgfqpoint{-0.104167in}{-0.027625in}}{\pgfqpoint{-0.093191in}{-0.054123in}}{\pgfqpoint{-0.073657in}{-0.073657in}}%
\pgfpathcurveto{\pgfqpoint{-0.054123in}{-0.093191in}}{\pgfqpoint{-0.027625in}{-0.104167in}}{\pgfqpoint{0.000000in}{-0.104167in}}%
\pgfpathclose%
\pgfusepath{stroke,fill}%
}%
\begin{pgfscope}%
\pgfsys@transformshift{2.250000in}{1.800000in}%
\pgfsys@useobject{currentmarker}{}%
\end{pgfscope}%
\end{pgfscope}%
\begin{pgfscope}%
\pgfpathrectangle{\pgfqpoint{0.000000in}{0.000000in}}{\pgfqpoint{6.000000in}{3.600000in}}%
\pgfusepath{clip}%
\pgfsetbuttcap%
\pgfsetroundjoin%
\definecolor{currentfill}{rgb}{1.000000,1.000000,1.000000}%
\pgfsetfillcolor{currentfill}%
\pgfsetlinewidth{1.003750pt}%
\definecolor{currentstroke}{rgb}{0.000000,0.000000,0.000000}%
\pgfsetstrokecolor{currentstroke}%
\pgfsetdash{}{0pt}%
\pgfsys@defobject{currentmarker}{\pgfqpoint{-0.104167in}{-0.104167in}}{\pgfqpoint{0.104167in}{0.104167in}}{%
\pgfpathmoveto{\pgfqpoint{0.000000in}{-0.104167in}}%
\pgfpathcurveto{\pgfqpoint{0.027625in}{-0.104167in}}{\pgfqpoint{0.054123in}{-0.093191in}}{\pgfqpoint{0.073657in}{-0.073657in}}%
\pgfpathcurveto{\pgfqpoint{0.093191in}{-0.054123in}}{\pgfqpoint{0.104167in}{-0.027625in}}{\pgfqpoint{0.104167in}{0.000000in}}%
\pgfpathcurveto{\pgfqpoint{0.104167in}{0.027625in}}{\pgfqpoint{0.093191in}{0.054123in}}{\pgfqpoint{0.073657in}{0.073657in}}%
\pgfpathcurveto{\pgfqpoint{0.054123in}{0.093191in}}{\pgfqpoint{0.027625in}{0.104167in}}{\pgfqpoint{0.000000in}{0.104167in}}%
\pgfpathcurveto{\pgfqpoint{-0.027625in}{0.104167in}}{\pgfqpoint{-0.054123in}{0.093191in}}{\pgfqpoint{-0.073657in}{0.073657in}}%
\pgfpathcurveto{\pgfqpoint{-0.093191in}{0.054123in}}{\pgfqpoint{-0.104167in}{0.027625in}}{\pgfqpoint{-0.104167in}{0.000000in}}%
\pgfpathcurveto{\pgfqpoint{-0.104167in}{-0.027625in}}{\pgfqpoint{-0.093191in}{-0.054123in}}{\pgfqpoint{-0.073657in}{-0.073657in}}%
\pgfpathcurveto{\pgfqpoint{-0.054123in}{-0.093191in}}{\pgfqpoint{-0.027625in}{-0.104167in}}{\pgfqpoint{0.000000in}{-0.104167in}}%
\pgfpathclose%
\pgfusepath{stroke,fill}%
}%
\begin{pgfscope}%
\pgfsys@transformshift{3.750000in}{1.800000in}%
\pgfsys@useobject{currentmarker}{}%
\end{pgfscope}%
\end{pgfscope}%
\begin{pgfscope}%
\pgfpathrectangle{\pgfqpoint{0.000000in}{0.000000in}}{\pgfqpoint{6.000000in}{3.600000in}}%
\pgfusepath{clip}%
\pgfsetbuttcap%
\pgfsetroundjoin%
\definecolor{currentfill}{rgb}{0.800000,0.800000,0.800000}%
\pgfsetfillcolor{currentfill}%
\pgfsetlinewidth{1.003750pt}%
\definecolor{currentstroke}{rgb}{0.000000,0.000000,0.000000}%
\pgfsetstrokecolor{currentstroke}%
\pgfsetdash{}{0pt}%
\pgfsys@defobject{currentmarker}{\pgfqpoint{-0.104167in}{-0.104167in}}{\pgfqpoint{0.104167in}{0.104167in}}{%
\pgfpathmoveto{\pgfqpoint{0.000000in}{-0.104167in}}%
\pgfpathcurveto{\pgfqpoint{0.027625in}{-0.104167in}}{\pgfqpoint{0.054123in}{-0.093191in}}{\pgfqpoint{0.073657in}{-0.073657in}}%
\pgfpathcurveto{\pgfqpoint{0.093191in}{-0.054123in}}{\pgfqpoint{0.104167in}{-0.027625in}}{\pgfqpoint{0.104167in}{0.000000in}}%
\pgfpathcurveto{\pgfqpoint{0.104167in}{0.027625in}}{\pgfqpoint{0.093191in}{0.054123in}}{\pgfqpoint{0.073657in}{0.073657in}}%
\pgfpathcurveto{\pgfqpoint{0.054123in}{0.093191in}}{\pgfqpoint{0.027625in}{0.104167in}}{\pgfqpoint{0.000000in}{0.104167in}}%
\pgfpathcurveto{\pgfqpoint{-0.027625in}{0.104167in}}{\pgfqpoint{-0.054123in}{0.093191in}}{\pgfqpoint{-0.073657in}{0.073657in}}%
\pgfpathcurveto{\pgfqpoint{-0.093191in}{0.054123in}}{\pgfqpoint{-0.104167in}{0.027625in}}{\pgfqpoint{-0.104167in}{0.000000in}}%
\pgfpathcurveto{\pgfqpoint{-0.104167in}{-0.027625in}}{\pgfqpoint{-0.093191in}{-0.054123in}}{\pgfqpoint{-0.073657in}{-0.073657in}}%
\pgfpathcurveto{\pgfqpoint{-0.054123in}{-0.093191in}}{\pgfqpoint{-0.027625in}{-0.104167in}}{\pgfqpoint{0.000000in}{-0.104167in}}%
\pgfpathclose%
\pgfusepath{stroke,fill}%
}%
\begin{pgfscope}%
\pgfsys@transformshift{0.750000in}{1.350000in}%
\pgfsys@useobject{currentmarker}{}%
\end{pgfscope}%
\end{pgfscope}%
\begin{pgfscope}%
\pgfpathrectangle{\pgfqpoint{0.000000in}{0.000000in}}{\pgfqpoint{6.000000in}{3.600000in}}%
\pgfusepath{clip}%
\pgfsetbuttcap%
\pgfsetroundjoin%
\definecolor{currentfill}{rgb}{1.000000,1.000000,1.000000}%
\pgfsetfillcolor{currentfill}%
\pgfsetlinewidth{1.003750pt}%
\definecolor{currentstroke}{rgb}{0.000000,0.000000,0.000000}%
\pgfsetstrokecolor{currentstroke}%
\pgfsetdash{}{0pt}%
\pgfsys@defobject{currentmarker}{\pgfqpoint{-0.104167in}{-0.104167in}}{\pgfqpoint{0.104167in}{0.104167in}}{%
\pgfpathmoveto{\pgfqpoint{0.000000in}{-0.104167in}}%
\pgfpathcurveto{\pgfqpoint{0.027625in}{-0.104167in}}{\pgfqpoint{0.054123in}{-0.093191in}}{\pgfqpoint{0.073657in}{-0.073657in}}%
\pgfpathcurveto{\pgfqpoint{0.093191in}{-0.054123in}}{\pgfqpoint{0.104167in}{-0.027625in}}{\pgfqpoint{0.104167in}{0.000000in}}%
\pgfpathcurveto{\pgfqpoint{0.104167in}{0.027625in}}{\pgfqpoint{0.093191in}{0.054123in}}{\pgfqpoint{0.073657in}{0.073657in}}%
\pgfpathcurveto{\pgfqpoint{0.054123in}{0.093191in}}{\pgfqpoint{0.027625in}{0.104167in}}{\pgfqpoint{0.000000in}{0.104167in}}%
\pgfpathcurveto{\pgfqpoint{-0.027625in}{0.104167in}}{\pgfqpoint{-0.054123in}{0.093191in}}{\pgfqpoint{-0.073657in}{0.073657in}}%
\pgfpathcurveto{\pgfqpoint{-0.093191in}{0.054123in}}{\pgfqpoint{-0.104167in}{0.027625in}}{\pgfqpoint{-0.104167in}{0.000000in}}%
\pgfpathcurveto{\pgfqpoint{-0.104167in}{-0.027625in}}{\pgfqpoint{-0.093191in}{-0.054123in}}{\pgfqpoint{-0.073657in}{-0.073657in}}%
\pgfpathcurveto{\pgfqpoint{-0.054123in}{-0.093191in}}{\pgfqpoint{-0.027625in}{-0.104167in}}{\pgfqpoint{0.000000in}{-0.104167in}}%
\pgfpathclose%
\pgfusepath{stroke,fill}%
}%
\begin{pgfscope}%
\pgfsys@transformshift{5.250000in}{1.350000in}%
\pgfsys@useobject{currentmarker}{}%
\end{pgfscope}%
\end{pgfscope}%
\begin{pgfscope}%
\pgfpathrectangle{\pgfqpoint{0.000000in}{0.000000in}}{\pgfqpoint{6.000000in}{3.600000in}}%
\pgfusepath{clip}%
\pgfsetbuttcap%
\pgfsetroundjoin%
\definecolor{currentfill}{rgb}{1.000000,1.000000,1.000000}%
\pgfsetfillcolor{currentfill}%
\pgfsetlinewidth{1.003750pt}%
\definecolor{currentstroke}{rgb}{0.000000,0.000000,0.000000}%
\pgfsetstrokecolor{currentstroke}%
\pgfsetdash{}{0pt}%
\pgfsys@defobject{currentmarker}{\pgfqpoint{-0.104167in}{-0.104167in}}{\pgfqpoint{0.104167in}{0.104167in}}{%
\pgfpathmoveto{\pgfqpoint{0.000000in}{-0.104167in}}%
\pgfpathcurveto{\pgfqpoint{0.027625in}{-0.104167in}}{\pgfqpoint{0.054123in}{-0.093191in}}{\pgfqpoint{0.073657in}{-0.073657in}}%
\pgfpathcurveto{\pgfqpoint{0.093191in}{-0.054123in}}{\pgfqpoint{0.104167in}{-0.027625in}}{\pgfqpoint{0.104167in}{0.000000in}}%
\pgfpathcurveto{\pgfqpoint{0.104167in}{0.027625in}}{\pgfqpoint{0.093191in}{0.054123in}}{\pgfqpoint{0.073657in}{0.073657in}}%
\pgfpathcurveto{\pgfqpoint{0.054123in}{0.093191in}}{\pgfqpoint{0.027625in}{0.104167in}}{\pgfqpoint{0.000000in}{0.104167in}}%
\pgfpathcurveto{\pgfqpoint{-0.027625in}{0.104167in}}{\pgfqpoint{-0.054123in}{0.093191in}}{\pgfqpoint{-0.073657in}{0.073657in}}%
\pgfpathcurveto{\pgfqpoint{-0.093191in}{0.054123in}}{\pgfqpoint{-0.104167in}{0.027625in}}{\pgfqpoint{-0.104167in}{0.000000in}}%
\pgfpathcurveto{\pgfqpoint{-0.104167in}{-0.027625in}}{\pgfqpoint{-0.093191in}{-0.054123in}}{\pgfqpoint{-0.073657in}{-0.073657in}}%
\pgfpathcurveto{\pgfqpoint{-0.054123in}{-0.093191in}}{\pgfqpoint{-0.027625in}{-0.104167in}}{\pgfqpoint{0.000000in}{-0.104167in}}%
\pgfpathclose%
\pgfusepath{stroke,fill}%
}%
\begin{pgfscope}%
\pgfsys@transformshift{2.250000in}{0.450000in}%
\pgfsys@useobject{currentmarker}{}%
\end{pgfscope}%
\end{pgfscope}%
\begin{pgfscope}%
\pgfpathrectangle{\pgfqpoint{0.000000in}{0.000000in}}{\pgfqpoint{6.000000in}{3.600000in}}%
\pgfusepath{clip}%
\pgfsetbuttcap%
\pgfsetroundjoin%
\definecolor{currentfill}{rgb}{1.000000,1.000000,1.000000}%
\pgfsetfillcolor{currentfill}%
\pgfsetlinewidth{1.003750pt}%
\definecolor{currentstroke}{rgb}{0.000000,0.000000,0.000000}%
\pgfsetstrokecolor{currentstroke}%
\pgfsetdash{}{0pt}%
\pgfsys@defobject{currentmarker}{\pgfqpoint{-0.104167in}{-0.104167in}}{\pgfqpoint{0.104167in}{0.104167in}}{%
\pgfpathmoveto{\pgfqpoint{0.000000in}{-0.104167in}}%
\pgfpathcurveto{\pgfqpoint{0.027625in}{-0.104167in}}{\pgfqpoint{0.054123in}{-0.093191in}}{\pgfqpoint{0.073657in}{-0.073657in}}%
\pgfpathcurveto{\pgfqpoint{0.093191in}{-0.054123in}}{\pgfqpoint{0.104167in}{-0.027625in}}{\pgfqpoint{0.104167in}{0.000000in}}%
\pgfpathcurveto{\pgfqpoint{0.104167in}{0.027625in}}{\pgfqpoint{0.093191in}{0.054123in}}{\pgfqpoint{0.073657in}{0.073657in}}%
\pgfpathcurveto{\pgfqpoint{0.054123in}{0.093191in}}{\pgfqpoint{0.027625in}{0.104167in}}{\pgfqpoint{0.000000in}{0.104167in}}%
\pgfpathcurveto{\pgfqpoint{-0.027625in}{0.104167in}}{\pgfqpoint{-0.054123in}{0.093191in}}{\pgfqpoint{-0.073657in}{0.073657in}}%
\pgfpathcurveto{\pgfqpoint{-0.093191in}{0.054123in}}{\pgfqpoint{-0.104167in}{0.027625in}}{\pgfqpoint{-0.104167in}{0.000000in}}%
\pgfpathcurveto{\pgfqpoint{-0.104167in}{-0.027625in}}{\pgfqpoint{-0.093191in}{-0.054123in}}{\pgfqpoint{-0.073657in}{-0.073657in}}%
\pgfpathcurveto{\pgfqpoint{-0.054123in}{-0.093191in}}{\pgfqpoint{-0.027625in}{-0.104167in}}{\pgfqpoint{0.000000in}{-0.104167in}}%
\pgfpathclose%
\pgfusepath{stroke,fill}%
}%
\begin{pgfscope}%
\pgfsys@transformshift{3.750000in}{0.450000in}%
\pgfsys@useobject{currentmarker}{}%
\end{pgfscope}%
\end{pgfscope}%
\begin{pgfscope}%
\definecolor{textcolor}{rgb}{0.000000,0.000000,0.000000}%
\pgfsetstrokecolor{textcolor}%
\pgfsetfillcolor{textcolor}%
\pgftext[x=2.529863in,y=2.870137in,,base]{\color{textcolor}\sffamily\fontsize{40.000000}{48.000000}\selectfont \(\displaystyle u_2\)}%
\end{pgfscope}%
\begin{pgfscope}%
\definecolor{textcolor}{rgb}{0.000000,0.000000,0.000000}%
\pgfsetstrokecolor{textcolor}%
\pgfsetfillcolor{textcolor}%
\pgftext[x=2.424863in,y=1.974863in,,base]{\color{textcolor}\sffamily\fontsize{40.000000}{48.000000}\selectfont \(\displaystyle v_2\)}%
\end{pgfscope}%
\end{pgfpicture}%
\makeatother%
\endgroup%

%% file: fig/decomposition_G_U0.pgf
%% Creator: Matplotlib, PGF backend
%%
%% To include the figure in your LaTeX document, write
%%   \input{<filename>.pgf}
%%
%% Make sure the required packages are loaded in your preamble
%%   \usepackage{pgf}
%%
%% Figures using additional raster images can only be included by \input if
%% they are in the same directory as the main LaTeX file. For loading figures
%% from other directories you can use the `import` package
%%   \usepackage{import}
%%
%% and then include the figures with
%%   \import{<path to file>}{<filename>.pgf}
%%
%% Matplotlib used the following preamble
%%   \usepackage{fontspec}
%%   \setmainfont{DejaVuSerif.ttf}[Path=\detokenize{C:/Users/ccros/Anaconda3/Lib/site-packages/matplotlib/mpl-data/fonts/ttf/}]
%%   \setsansfont{DejaVuSans.ttf}[Path=\detokenize{C:/Users/ccros/Anaconda3/Lib/site-packages/matplotlib/mpl-data/fonts/ttf/}]
%%   \setmonofont{DejaVuSansMono.ttf}[Path=\detokenize{C:/Users/ccros/Anaconda3/Lib/site-packages/matplotlib/mpl-data/fonts/ttf/}]
%%
\begingroup%
\makeatletter%
\begin{pgfpicture}%
\pgfpathrectangle{\pgfpointorigin}{\pgfqpoint{6.000000in}{3.600000in}}%
\pgfusepath{use as bounding box, clip}%
\begin{pgfscope}%
\pgfsetbuttcap%
\pgfsetmiterjoin%
\definecolor{currentfill}{rgb}{1.000000,1.000000,1.000000}%
\pgfsetfillcolor{currentfill}%
\pgfsetlinewidth{0.000000pt}%
\definecolor{currentstroke}{rgb}{1.000000,1.000000,1.000000}%
\pgfsetstrokecolor{currentstroke}%
\pgfsetdash{}{0pt}%
\pgfpathmoveto{\pgfqpoint{0.000000in}{0.000000in}}%
\pgfpathlineto{\pgfqpoint{6.000000in}{0.000000in}}%
\pgfpathlineto{\pgfqpoint{6.000000in}{3.600000in}}%
\pgfpathlineto{\pgfqpoint{0.000000in}{3.600000in}}%
\pgfpathclose%
\pgfusepath{fill}%
\end{pgfscope}%
\begin{pgfscope}%
\pgfpathrectangle{\pgfqpoint{0.000000in}{0.000000in}}{\pgfqpoint{6.000000in}{3.600000in}}%
\pgfusepath{clip}%
\pgfsetbuttcap%
\pgfsetroundjoin%
\pgfsetlinewidth{3.011250pt}%
\definecolor{currentstroke}{rgb}{0.400000,0.400000,0.400000}%
\pgfsetstrokecolor{currentstroke}%
\pgfsetdash{{11.100000pt}{4.800000pt}}{0.000000pt}%
\pgfpathmoveto{\pgfqpoint{2.250000in}{3.150000in}}%
\pgfpathlineto{\pgfqpoint{3.750000in}{3.150000in}}%
\pgfusepath{stroke}%
\end{pgfscope}%
\begin{pgfscope}%
\pgfpathrectangle{\pgfqpoint{0.000000in}{0.000000in}}{\pgfqpoint{6.000000in}{3.600000in}}%
\pgfusepath{clip}%
\pgfsetbuttcap%
\pgfsetroundjoin%
\pgfsetlinewidth{3.011250pt}%
\definecolor{currentstroke}{rgb}{0.400000,0.400000,0.400000}%
\pgfsetstrokecolor{currentstroke}%
\pgfsetdash{{11.100000pt}{4.800000pt}}{0.000000pt}%
\pgfpathmoveto{\pgfqpoint{2.250000in}{3.150000in}}%
\pgfpathlineto{\pgfqpoint{0.750000in}{2.250000in}}%
\pgfusepath{stroke}%
\end{pgfscope}%
\begin{pgfscope}%
\pgfpathrectangle{\pgfqpoint{0.000000in}{0.000000in}}{\pgfqpoint{6.000000in}{3.600000in}}%
\pgfusepath{clip}%
\pgfsetbuttcap%
\pgfsetroundjoin%
\pgfsetlinewidth{3.011250pt}%
\definecolor{currentstroke}{rgb}{0.400000,0.400000,0.400000}%
\pgfsetstrokecolor{currentstroke}%
\pgfsetdash{{11.100000pt}{4.800000pt}}{0.000000pt}%
\pgfpathmoveto{\pgfqpoint{2.250000in}{3.150000in}}%
\pgfpathlineto{\pgfqpoint{5.250000in}{2.250000in}}%
\pgfusepath{stroke}%
\end{pgfscope}%
\begin{pgfscope}%
\pgfpathrectangle{\pgfqpoint{0.000000in}{0.000000in}}{\pgfqpoint{6.000000in}{3.600000in}}%
\pgfusepath{clip}%
\pgfsetbuttcap%
\pgfsetroundjoin%
\pgfsetlinewidth{3.011250pt}%
\definecolor{currentstroke}{rgb}{0.400000,0.400000,0.400000}%
\pgfsetstrokecolor{currentstroke}%
\pgfsetdash{{11.100000pt}{4.800000pt}}{0.000000pt}%
\pgfpathmoveto{\pgfqpoint{3.750000in}{3.150000in}}%
\pgfpathlineto{\pgfqpoint{5.250000in}{2.250000in}}%
\pgfusepath{stroke}%
\end{pgfscope}%
\begin{pgfscope}%
\pgfpathrectangle{\pgfqpoint{0.000000in}{0.000000in}}{\pgfqpoint{6.000000in}{3.600000in}}%
\pgfusepath{clip}%
\pgfsetbuttcap%
\pgfsetroundjoin%
\pgfsetlinewidth{3.011250pt}%
\definecolor{currentstroke}{rgb}{0.400000,0.400000,0.400000}%
\pgfsetstrokecolor{currentstroke}%
\pgfsetdash{{11.100000pt}{4.800000pt}}{0.000000pt}%
\pgfpathmoveto{\pgfqpoint{3.750000in}{3.150000in}}%
\pgfpathlineto{\pgfqpoint{2.250000in}{1.800000in}}%
\pgfusepath{stroke}%
\end{pgfscope}%
\begin{pgfscope}%
\pgfpathrectangle{\pgfqpoint{0.000000in}{0.000000in}}{\pgfqpoint{6.000000in}{3.600000in}}%
\pgfusepath{clip}%
\pgfsetbuttcap%
\pgfsetroundjoin%
\pgfsetlinewidth{3.011250pt}%
\definecolor{currentstroke}{rgb}{0.400000,0.400000,0.400000}%
\pgfsetstrokecolor{currentstroke}%
\pgfsetdash{{11.100000pt}{4.800000pt}}{0.000000pt}%
\pgfpathmoveto{\pgfqpoint{0.750000in}{2.250000in}}%
\pgfpathlineto{\pgfqpoint{2.250000in}{1.800000in}}%
\pgfusepath{stroke}%
\end{pgfscope}%
\begin{pgfscope}%
\pgfpathrectangle{\pgfqpoint{0.000000in}{0.000000in}}{\pgfqpoint{6.000000in}{3.600000in}}%
\pgfusepath{clip}%
\pgfsetbuttcap%
\pgfsetroundjoin%
\pgfsetlinewidth{3.011250pt}%
\definecolor{currentstroke}{rgb}{0.400000,0.400000,0.400000}%
\pgfsetstrokecolor{currentstroke}%
\pgfsetdash{{11.100000pt}{4.800000pt}}{0.000000pt}%
\pgfpathmoveto{\pgfqpoint{0.750000in}{2.250000in}}%
\pgfpathlineto{\pgfqpoint{0.750000in}{1.350000in}}%
\pgfusepath{stroke}%
\end{pgfscope}%
\begin{pgfscope}%
\pgfpathrectangle{\pgfqpoint{0.000000in}{0.000000in}}{\pgfqpoint{6.000000in}{3.600000in}}%
\pgfusepath{clip}%
\pgfsetbuttcap%
\pgfsetroundjoin%
\pgfsetlinewidth{3.011250pt}%
\definecolor{currentstroke}{rgb}{0.400000,0.400000,0.400000}%
\pgfsetstrokecolor{currentstroke}%
\pgfsetdash{{11.100000pt}{4.800000pt}}{0.000000pt}%
\pgfpathmoveto{\pgfqpoint{0.750000in}{1.350000in}}%
\pgfpathlineto{\pgfqpoint{3.750000in}{0.450000in}}%
\pgfusepath{stroke}%
\end{pgfscope}%
\begin{pgfscope}%
\pgfpathrectangle{\pgfqpoint{0.000000in}{0.000000in}}{\pgfqpoint{6.000000in}{3.600000in}}%
\pgfusepath{clip}%
\pgfsetbuttcap%
\pgfsetroundjoin%
\pgfsetlinewidth{3.011250pt}%
\definecolor{currentstroke}{rgb}{0.400000,0.400000,0.400000}%
\pgfsetstrokecolor{currentstroke}%
\pgfsetdash{{11.100000pt}{4.800000pt}}{0.000000pt}%
\pgfpathmoveto{\pgfqpoint{2.250000in}{1.800000in}}%
\pgfpathlineto{\pgfqpoint{0.750000in}{1.350000in}}%
\pgfusepath{stroke}%
\end{pgfscope}%
\begin{pgfscope}%
\pgfpathrectangle{\pgfqpoint{0.000000in}{0.000000in}}{\pgfqpoint{6.000000in}{3.600000in}}%
\pgfusepath{clip}%
\pgfsetbuttcap%
\pgfsetroundjoin%
\pgfsetlinewidth{3.011250pt}%
\definecolor{currentstroke}{rgb}{0.400000,0.400000,0.400000}%
\pgfsetstrokecolor{currentstroke}%
\pgfsetdash{{11.100000pt}{4.800000pt}}{0.000000pt}%
\pgfpathmoveto{\pgfqpoint{2.250000in}{1.800000in}}%
\pgfpathlineto{\pgfqpoint{2.250000in}{0.450000in}}%
\pgfusepath{stroke}%
\end{pgfscope}%
\begin{pgfscope}%
\pgfpathrectangle{\pgfqpoint{0.000000in}{0.000000in}}{\pgfqpoint{6.000000in}{3.600000in}}%
\pgfusepath{clip}%
\pgfsetbuttcap%
\pgfsetroundjoin%
\pgfsetlinewidth{3.011250pt}%
\definecolor{currentstroke}{rgb}{0.400000,0.400000,0.400000}%
\pgfsetstrokecolor{currentstroke}%
\pgfsetdash{{11.100000pt}{4.800000pt}}{0.000000pt}%
\pgfpathmoveto{\pgfqpoint{2.250000in}{1.800000in}}%
\pgfpathlineto{\pgfqpoint{3.750000in}{0.450000in}}%
\pgfusepath{stroke}%
\end{pgfscope}%
\begin{pgfscope}%
\pgfpathrectangle{\pgfqpoint{0.000000in}{0.000000in}}{\pgfqpoint{6.000000in}{3.600000in}}%
\pgfusepath{clip}%
\pgfsetbuttcap%
\pgfsetroundjoin%
\pgfsetlinewidth{3.011250pt}%
\definecolor{currentstroke}{rgb}{0.400000,0.400000,0.400000}%
\pgfsetstrokecolor{currentstroke}%
\pgfsetdash{{11.100000pt}{4.800000pt}}{0.000000pt}%
\pgfpathmoveto{\pgfqpoint{2.250000in}{0.450000in}}%
\pgfpathlineto{\pgfqpoint{3.750000in}{0.450000in}}%
\pgfusepath{stroke}%
\end{pgfscope}%
\begin{pgfscope}%
\pgfpathrectangle{\pgfqpoint{0.000000in}{0.000000in}}{\pgfqpoint{6.000000in}{3.600000in}}%
\pgfusepath{clip}%
\pgfsetrectcap%
\pgfsetroundjoin%
\pgfsetlinewidth{3.011250pt}%
\definecolor{currentstroke}{rgb}{0.250980,0.250980,0.250980}%
\pgfsetstrokecolor{currentstroke}%
\pgfsetdash{}{0pt}%
\pgfpathmoveto{\pgfqpoint{2.250000in}{1.800000in}}%
\pgfpathlineto{\pgfqpoint{5.250000in}{2.250000in}}%
\pgfusepath{stroke}%
\end{pgfscope}%
\begin{pgfscope}%
\pgfpathrectangle{\pgfqpoint{0.000000in}{0.000000in}}{\pgfqpoint{6.000000in}{3.600000in}}%
\pgfusepath{clip}%
\pgfsetrectcap%
\pgfsetroundjoin%
\pgfsetlinewidth{3.011250pt}%
\definecolor{currentstroke}{rgb}{0.250980,0.250980,0.250980}%
\pgfsetstrokecolor{currentstroke}%
\pgfsetdash{}{0pt}%
\pgfpathmoveto{\pgfqpoint{3.750000in}{1.800000in}}%
\pgfpathlineto{\pgfqpoint{5.250000in}{2.250000in}}%
\pgfusepath{stroke}%
\end{pgfscope}%
\begin{pgfscope}%
\pgfpathrectangle{\pgfqpoint{0.000000in}{0.000000in}}{\pgfqpoint{6.000000in}{3.600000in}}%
\pgfusepath{clip}%
\pgfsetrectcap%
\pgfsetroundjoin%
\pgfsetlinewidth{3.011250pt}%
\definecolor{currentstroke}{rgb}{0.250980,0.250980,0.250980}%
\pgfsetstrokecolor{currentstroke}%
\pgfsetdash{}{0pt}%
\pgfpathmoveto{\pgfqpoint{5.250000in}{1.350000in}}%
\pgfpathlineto{\pgfqpoint{5.250000in}{2.250000in}}%
\pgfusepath{stroke}%
\end{pgfscope}%
\begin{pgfscope}%
\pgfpathrectangle{\pgfqpoint{0.000000in}{0.000000in}}{\pgfqpoint{6.000000in}{3.600000in}}%
\pgfusepath{clip}%
\pgfsetrectcap%
\pgfsetroundjoin%
\pgfsetlinewidth{3.011250pt}%
\definecolor{currentstroke}{rgb}{0.250980,0.250980,0.250980}%
\pgfsetstrokecolor{currentstroke}%
\pgfsetdash{}{0pt}%
\pgfpathmoveto{\pgfqpoint{5.250000in}{1.350000in}}%
\pgfpathlineto{\pgfqpoint{2.250000in}{1.800000in}}%
\pgfusepath{stroke}%
\end{pgfscope}%
\begin{pgfscope}%
\pgfpathrectangle{\pgfqpoint{0.000000in}{0.000000in}}{\pgfqpoint{6.000000in}{3.600000in}}%
\pgfusepath{clip}%
\pgfsetrectcap%
\pgfsetroundjoin%
\pgfsetlinewidth{3.011250pt}%
\definecolor{currentstroke}{rgb}{0.250980,0.250980,0.250980}%
\pgfsetstrokecolor{currentstroke}%
\pgfsetdash{}{0pt}%
\pgfpathmoveto{\pgfqpoint{5.250000in}{1.350000in}}%
\pgfpathlineto{\pgfqpoint{3.750000in}{1.800000in}}%
\pgfusepath{stroke}%
\end{pgfscope}%
\begin{pgfscope}%
\pgfpathrectangle{\pgfqpoint{0.000000in}{0.000000in}}{\pgfqpoint{6.000000in}{3.600000in}}%
\pgfusepath{clip}%
\pgfsetbuttcap%
\pgfsetroundjoin%
\definecolor{currentfill}{rgb}{0.800000,0.800000,0.800000}%
\pgfsetfillcolor{currentfill}%
\pgfsetlinewidth{1.003750pt}%
\definecolor{currentstroke}{rgb}{0.000000,0.000000,0.000000}%
\pgfsetstrokecolor{currentstroke}%
\pgfsetdash{}{0pt}%
\pgfsys@defobject{currentmarker}{\pgfqpoint{-0.104167in}{-0.104167in}}{\pgfqpoint{0.104167in}{0.104167in}}{%
\pgfpathmoveto{\pgfqpoint{0.000000in}{-0.104167in}}%
\pgfpathcurveto{\pgfqpoint{0.027625in}{-0.104167in}}{\pgfqpoint{0.054123in}{-0.093191in}}{\pgfqpoint{0.073657in}{-0.073657in}}%
\pgfpathcurveto{\pgfqpoint{0.093191in}{-0.054123in}}{\pgfqpoint{0.104167in}{-0.027625in}}{\pgfqpoint{0.104167in}{0.000000in}}%
\pgfpathcurveto{\pgfqpoint{0.104167in}{0.027625in}}{\pgfqpoint{0.093191in}{0.054123in}}{\pgfqpoint{0.073657in}{0.073657in}}%
\pgfpathcurveto{\pgfqpoint{0.054123in}{0.093191in}}{\pgfqpoint{0.027625in}{0.104167in}}{\pgfqpoint{0.000000in}{0.104167in}}%
\pgfpathcurveto{\pgfqpoint{-0.027625in}{0.104167in}}{\pgfqpoint{-0.054123in}{0.093191in}}{\pgfqpoint{-0.073657in}{0.073657in}}%
\pgfpathcurveto{\pgfqpoint{-0.093191in}{0.054123in}}{\pgfqpoint{-0.104167in}{0.027625in}}{\pgfqpoint{-0.104167in}{0.000000in}}%
\pgfpathcurveto{\pgfqpoint{-0.104167in}{-0.027625in}}{\pgfqpoint{-0.093191in}{-0.054123in}}{\pgfqpoint{-0.073657in}{-0.073657in}}%
\pgfpathcurveto{\pgfqpoint{-0.054123in}{-0.093191in}}{\pgfqpoint{-0.027625in}{-0.104167in}}{\pgfqpoint{0.000000in}{-0.104167in}}%
\pgfpathclose%
\pgfusepath{stroke,fill}%
}%
\begin{pgfscope}%
\pgfsys@transformshift{2.250000in}{3.150000in}%
\pgfsys@useobject{currentmarker}{}%
\end{pgfscope}%
\end{pgfscope}%
\begin{pgfscope}%
\pgfpathrectangle{\pgfqpoint{0.000000in}{0.000000in}}{\pgfqpoint{6.000000in}{3.600000in}}%
\pgfusepath{clip}%
\pgfsetbuttcap%
\pgfsetroundjoin%
\definecolor{currentfill}{rgb}{0.800000,0.800000,0.800000}%
\pgfsetfillcolor{currentfill}%
\pgfsetlinewidth{1.003750pt}%
\definecolor{currentstroke}{rgb}{0.000000,0.000000,0.000000}%
\pgfsetstrokecolor{currentstroke}%
\pgfsetdash{}{0pt}%
\pgfsys@defobject{currentmarker}{\pgfqpoint{-0.104167in}{-0.104167in}}{\pgfqpoint{0.104167in}{0.104167in}}{%
\pgfpathmoveto{\pgfqpoint{0.000000in}{-0.104167in}}%
\pgfpathcurveto{\pgfqpoint{0.027625in}{-0.104167in}}{\pgfqpoint{0.054123in}{-0.093191in}}{\pgfqpoint{0.073657in}{-0.073657in}}%
\pgfpathcurveto{\pgfqpoint{0.093191in}{-0.054123in}}{\pgfqpoint{0.104167in}{-0.027625in}}{\pgfqpoint{0.104167in}{0.000000in}}%
\pgfpathcurveto{\pgfqpoint{0.104167in}{0.027625in}}{\pgfqpoint{0.093191in}{0.054123in}}{\pgfqpoint{0.073657in}{0.073657in}}%
\pgfpathcurveto{\pgfqpoint{0.054123in}{0.093191in}}{\pgfqpoint{0.027625in}{0.104167in}}{\pgfqpoint{0.000000in}{0.104167in}}%
\pgfpathcurveto{\pgfqpoint{-0.027625in}{0.104167in}}{\pgfqpoint{-0.054123in}{0.093191in}}{\pgfqpoint{-0.073657in}{0.073657in}}%
\pgfpathcurveto{\pgfqpoint{-0.093191in}{0.054123in}}{\pgfqpoint{-0.104167in}{0.027625in}}{\pgfqpoint{-0.104167in}{0.000000in}}%
\pgfpathcurveto{\pgfqpoint{-0.104167in}{-0.027625in}}{\pgfqpoint{-0.093191in}{-0.054123in}}{\pgfqpoint{-0.073657in}{-0.073657in}}%
\pgfpathcurveto{\pgfqpoint{-0.054123in}{-0.093191in}}{\pgfqpoint{-0.027625in}{-0.104167in}}{\pgfqpoint{0.000000in}{-0.104167in}}%
\pgfpathclose%
\pgfusepath{stroke,fill}%
}%
\begin{pgfscope}%
\pgfsys@transformshift{3.750000in}{3.150000in}%
\pgfsys@useobject{currentmarker}{}%
\end{pgfscope}%
\end{pgfscope}%
\begin{pgfscope}%
\pgfpathrectangle{\pgfqpoint{0.000000in}{0.000000in}}{\pgfqpoint{6.000000in}{3.600000in}}%
\pgfusepath{clip}%
\pgfsetbuttcap%
\pgfsetroundjoin%
\definecolor{currentfill}{rgb}{0.800000,0.800000,0.800000}%
\pgfsetfillcolor{currentfill}%
\pgfsetlinewidth{1.003750pt}%
\definecolor{currentstroke}{rgb}{0.000000,0.000000,0.000000}%
\pgfsetstrokecolor{currentstroke}%
\pgfsetdash{}{0pt}%
\pgfsys@defobject{currentmarker}{\pgfqpoint{-0.104167in}{-0.104167in}}{\pgfqpoint{0.104167in}{0.104167in}}{%
\pgfpathmoveto{\pgfqpoint{0.000000in}{-0.104167in}}%
\pgfpathcurveto{\pgfqpoint{0.027625in}{-0.104167in}}{\pgfqpoint{0.054123in}{-0.093191in}}{\pgfqpoint{0.073657in}{-0.073657in}}%
\pgfpathcurveto{\pgfqpoint{0.093191in}{-0.054123in}}{\pgfqpoint{0.104167in}{-0.027625in}}{\pgfqpoint{0.104167in}{0.000000in}}%
\pgfpathcurveto{\pgfqpoint{0.104167in}{0.027625in}}{\pgfqpoint{0.093191in}{0.054123in}}{\pgfqpoint{0.073657in}{0.073657in}}%
\pgfpathcurveto{\pgfqpoint{0.054123in}{0.093191in}}{\pgfqpoint{0.027625in}{0.104167in}}{\pgfqpoint{0.000000in}{0.104167in}}%
\pgfpathcurveto{\pgfqpoint{-0.027625in}{0.104167in}}{\pgfqpoint{-0.054123in}{0.093191in}}{\pgfqpoint{-0.073657in}{0.073657in}}%
\pgfpathcurveto{\pgfqpoint{-0.093191in}{0.054123in}}{\pgfqpoint{-0.104167in}{0.027625in}}{\pgfqpoint{-0.104167in}{0.000000in}}%
\pgfpathcurveto{\pgfqpoint{-0.104167in}{-0.027625in}}{\pgfqpoint{-0.093191in}{-0.054123in}}{\pgfqpoint{-0.073657in}{-0.073657in}}%
\pgfpathcurveto{\pgfqpoint{-0.054123in}{-0.093191in}}{\pgfqpoint{-0.027625in}{-0.104167in}}{\pgfqpoint{0.000000in}{-0.104167in}}%
\pgfpathclose%
\pgfusepath{stroke,fill}%
}%
\begin{pgfscope}%
\pgfsys@transformshift{0.750000in}{2.250000in}%
\pgfsys@useobject{currentmarker}{}%
\end{pgfscope}%
\end{pgfscope}%
\begin{pgfscope}%
\pgfpathrectangle{\pgfqpoint{0.000000in}{0.000000in}}{\pgfqpoint{6.000000in}{3.600000in}}%
\pgfusepath{clip}%
\pgfsetbuttcap%
\pgfsetroundjoin%
\definecolor{currentfill}{rgb}{0.800000,0.800000,0.800000}%
\pgfsetfillcolor{currentfill}%
\pgfsetlinewidth{1.003750pt}%
\definecolor{currentstroke}{rgb}{0.000000,0.000000,0.000000}%
\pgfsetstrokecolor{currentstroke}%
\pgfsetdash{}{0pt}%
\pgfsys@defobject{currentmarker}{\pgfqpoint{-0.104167in}{-0.104167in}}{\pgfqpoint{0.104167in}{0.104167in}}{%
\pgfpathmoveto{\pgfqpoint{0.000000in}{-0.104167in}}%
\pgfpathcurveto{\pgfqpoint{0.027625in}{-0.104167in}}{\pgfqpoint{0.054123in}{-0.093191in}}{\pgfqpoint{0.073657in}{-0.073657in}}%
\pgfpathcurveto{\pgfqpoint{0.093191in}{-0.054123in}}{\pgfqpoint{0.104167in}{-0.027625in}}{\pgfqpoint{0.104167in}{0.000000in}}%
\pgfpathcurveto{\pgfqpoint{0.104167in}{0.027625in}}{\pgfqpoint{0.093191in}{0.054123in}}{\pgfqpoint{0.073657in}{0.073657in}}%
\pgfpathcurveto{\pgfqpoint{0.054123in}{0.093191in}}{\pgfqpoint{0.027625in}{0.104167in}}{\pgfqpoint{0.000000in}{0.104167in}}%
\pgfpathcurveto{\pgfqpoint{-0.027625in}{0.104167in}}{\pgfqpoint{-0.054123in}{0.093191in}}{\pgfqpoint{-0.073657in}{0.073657in}}%
\pgfpathcurveto{\pgfqpoint{-0.093191in}{0.054123in}}{\pgfqpoint{-0.104167in}{0.027625in}}{\pgfqpoint{-0.104167in}{0.000000in}}%
\pgfpathcurveto{\pgfqpoint{-0.104167in}{-0.027625in}}{\pgfqpoint{-0.093191in}{-0.054123in}}{\pgfqpoint{-0.073657in}{-0.073657in}}%
\pgfpathcurveto{\pgfqpoint{-0.054123in}{-0.093191in}}{\pgfqpoint{-0.027625in}{-0.104167in}}{\pgfqpoint{0.000000in}{-0.104167in}}%
\pgfpathclose%
\pgfusepath{stroke,fill}%
}%
\begin{pgfscope}%
\pgfsys@transformshift{5.250000in}{2.250000in}%
\pgfsys@useobject{currentmarker}{}%
\end{pgfscope}%
\end{pgfscope}%
\begin{pgfscope}%
\pgfpathrectangle{\pgfqpoint{0.000000in}{0.000000in}}{\pgfqpoint{6.000000in}{3.600000in}}%
\pgfusepath{clip}%
\pgfsetbuttcap%
\pgfsetroundjoin%
\definecolor{currentfill}{rgb}{0.800000,0.800000,0.800000}%
\pgfsetfillcolor{currentfill}%
\pgfsetlinewidth{1.003750pt}%
\definecolor{currentstroke}{rgb}{0.000000,0.000000,0.000000}%
\pgfsetstrokecolor{currentstroke}%
\pgfsetdash{}{0pt}%
\pgfsys@defobject{currentmarker}{\pgfqpoint{-0.104167in}{-0.104167in}}{\pgfqpoint{0.104167in}{0.104167in}}{%
\pgfpathmoveto{\pgfqpoint{0.000000in}{-0.104167in}}%
\pgfpathcurveto{\pgfqpoint{0.027625in}{-0.104167in}}{\pgfqpoint{0.054123in}{-0.093191in}}{\pgfqpoint{0.073657in}{-0.073657in}}%
\pgfpathcurveto{\pgfqpoint{0.093191in}{-0.054123in}}{\pgfqpoint{0.104167in}{-0.027625in}}{\pgfqpoint{0.104167in}{0.000000in}}%
\pgfpathcurveto{\pgfqpoint{0.104167in}{0.027625in}}{\pgfqpoint{0.093191in}{0.054123in}}{\pgfqpoint{0.073657in}{0.073657in}}%
\pgfpathcurveto{\pgfqpoint{0.054123in}{0.093191in}}{\pgfqpoint{0.027625in}{0.104167in}}{\pgfqpoint{0.000000in}{0.104167in}}%
\pgfpathcurveto{\pgfqpoint{-0.027625in}{0.104167in}}{\pgfqpoint{-0.054123in}{0.093191in}}{\pgfqpoint{-0.073657in}{0.073657in}}%
\pgfpathcurveto{\pgfqpoint{-0.093191in}{0.054123in}}{\pgfqpoint{-0.104167in}{0.027625in}}{\pgfqpoint{-0.104167in}{0.000000in}}%
\pgfpathcurveto{\pgfqpoint{-0.104167in}{-0.027625in}}{\pgfqpoint{-0.093191in}{-0.054123in}}{\pgfqpoint{-0.073657in}{-0.073657in}}%
\pgfpathcurveto{\pgfqpoint{-0.054123in}{-0.093191in}}{\pgfqpoint{-0.027625in}{-0.104167in}}{\pgfqpoint{0.000000in}{-0.104167in}}%
\pgfpathclose%
\pgfusepath{stroke,fill}%
}%
\begin{pgfscope}%
\pgfsys@transformshift{2.250000in}{1.800000in}%
\pgfsys@useobject{currentmarker}{}%
\end{pgfscope}%
\end{pgfscope}%
\begin{pgfscope}%
\pgfpathrectangle{\pgfqpoint{0.000000in}{0.000000in}}{\pgfqpoint{6.000000in}{3.600000in}}%
\pgfusepath{clip}%
\pgfsetbuttcap%
\pgfsetroundjoin%
\definecolor{currentfill}{rgb}{0.800000,0.800000,0.800000}%
\pgfsetfillcolor{currentfill}%
\pgfsetlinewidth{1.003750pt}%
\definecolor{currentstroke}{rgb}{0.000000,0.000000,0.000000}%
\pgfsetstrokecolor{currentstroke}%
\pgfsetdash{}{0pt}%
\pgfsys@defobject{currentmarker}{\pgfqpoint{-0.104167in}{-0.104167in}}{\pgfqpoint{0.104167in}{0.104167in}}{%
\pgfpathmoveto{\pgfqpoint{0.000000in}{-0.104167in}}%
\pgfpathcurveto{\pgfqpoint{0.027625in}{-0.104167in}}{\pgfqpoint{0.054123in}{-0.093191in}}{\pgfqpoint{0.073657in}{-0.073657in}}%
\pgfpathcurveto{\pgfqpoint{0.093191in}{-0.054123in}}{\pgfqpoint{0.104167in}{-0.027625in}}{\pgfqpoint{0.104167in}{0.000000in}}%
\pgfpathcurveto{\pgfqpoint{0.104167in}{0.027625in}}{\pgfqpoint{0.093191in}{0.054123in}}{\pgfqpoint{0.073657in}{0.073657in}}%
\pgfpathcurveto{\pgfqpoint{0.054123in}{0.093191in}}{\pgfqpoint{0.027625in}{0.104167in}}{\pgfqpoint{0.000000in}{0.104167in}}%
\pgfpathcurveto{\pgfqpoint{-0.027625in}{0.104167in}}{\pgfqpoint{-0.054123in}{0.093191in}}{\pgfqpoint{-0.073657in}{0.073657in}}%
\pgfpathcurveto{\pgfqpoint{-0.093191in}{0.054123in}}{\pgfqpoint{-0.104167in}{0.027625in}}{\pgfqpoint{-0.104167in}{0.000000in}}%
\pgfpathcurveto{\pgfqpoint{-0.104167in}{-0.027625in}}{\pgfqpoint{-0.093191in}{-0.054123in}}{\pgfqpoint{-0.073657in}{-0.073657in}}%
\pgfpathcurveto{\pgfqpoint{-0.054123in}{-0.093191in}}{\pgfqpoint{-0.027625in}{-0.104167in}}{\pgfqpoint{0.000000in}{-0.104167in}}%
\pgfpathclose%
\pgfusepath{stroke,fill}%
}%
\begin{pgfscope}%
\pgfsys@transformshift{3.750000in}{1.800000in}%
\pgfsys@useobject{currentmarker}{}%
\end{pgfscope}%
\end{pgfscope}%
\begin{pgfscope}%
\pgfpathrectangle{\pgfqpoint{0.000000in}{0.000000in}}{\pgfqpoint{6.000000in}{3.600000in}}%
\pgfusepath{clip}%
\pgfsetbuttcap%
\pgfsetroundjoin%
\definecolor{currentfill}{rgb}{0.800000,0.800000,0.800000}%
\pgfsetfillcolor{currentfill}%
\pgfsetlinewidth{1.003750pt}%
\definecolor{currentstroke}{rgb}{0.000000,0.000000,0.000000}%
\pgfsetstrokecolor{currentstroke}%
\pgfsetdash{}{0pt}%
\pgfsys@defobject{currentmarker}{\pgfqpoint{-0.104167in}{-0.104167in}}{\pgfqpoint{0.104167in}{0.104167in}}{%
\pgfpathmoveto{\pgfqpoint{0.000000in}{-0.104167in}}%
\pgfpathcurveto{\pgfqpoint{0.027625in}{-0.104167in}}{\pgfqpoint{0.054123in}{-0.093191in}}{\pgfqpoint{0.073657in}{-0.073657in}}%
\pgfpathcurveto{\pgfqpoint{0.093191in}{-0.054123in}}{\pgfqpoint{0.104167in}{-0.027625in}}{\pgfqpoint{0.104167in}{0.000000in}}%
\pgfpathcurveto{\pgfqpoint{0.104167in}{0.027625in}}{\pgfqpoint{0.093191in}{0.054123in}}{\pgfqpoint{0.073657in}{0.073657in}}%
\pgfpathcurveto{\pgfqpoint{0.054123in}{0.093191in}}{\pgfqpoint{0.027625in}{0.104167in}}{\pgfqpoint{0.000000in}{0.104167in}}%
\pgfpathcurveto{\pgfqpoint{-0.027625in}{0.104167in}}{\pgfqpoint{-0.054123in}{0.093191in}}{\pgfqpoint{-0.073657in}{0.073657in}}%
\pgfpathcurveto{\pgfqpoint{-0.093191in}{0.054123in}}{\pgfqpoint{-0.104167in}{0.027625in}}{\pgfqpoint{-0.104167in}{0.000000in}}%
\pgfpathcurveto{\pgfqpoint{-0.104167in}{-0.027625in}}{\pgfqpoint{-0.093191in}{-0.054123in}}{\pgfqpoint{-0.073657in}{-0.073657in}}%
\pgfpathcurveto{\pgfqpoint{-0.054123in}{-0.093191in}}{\pgfqpoint{-0.027625in}{-0.104167in}}{\pgfqpoint{0.000000in}{-0.104167in}}%
\pgfpathclose%
\pgfusepath{stroke,fill}%
}%
\begin{pgfscope}%
\pgfsys@transformshift{0.750000in}{1.350000in}%
\pgfsys@useobject{currentmarker}{}%
\end{pgfscope}%
\end{pgfscope}%
\begin{pgfscope}%
\pgfpathrectangle{\pgfqpoint{0.000000in}{0.000000in}}{\pgfqpoint{6.000000in}{3.600000in}}%
\pgfusepath{clip}%
\pgfsetbuttcap%
\pgfsetroundjoin%
\definecolor{currentfill}{rgb}{1.000000,1.000000,1.000000}%
\pgfsetfillcolor{currentfill}%
\pgfsetlinewidth{1.003750pt}%
\definecolor{currentstroke}{rgb}{0.000000,0.000000,0.000000}%
\pgfsetstrokecolor{currentstroke}%
\pgfsetdash{}{0pt}%
\pgfsys@defobject{currentmarker}{\pgfqpoint{-0.104167in}{-0.104167in}}{\pgfqpoint{0.104167in}{0.104167in}}{%
\pgfpathmoveto{\pgfqpoint{0.000000in}{-0.104167in}}%
\pgfpathcurveto{\pgfqpoint{0.027625in}{-0.104167in}}{\pgfqpoint{0.054123in}{-0.093191in}}{\pgfqpoint{0.073657in}{-0.073657in}}%
\pgfpathcurveto{\pgfqpoint{0.093191in}{-0.054123in}}{\pgfqpoint{0.104167in}{-0.027625in}}{\pgfqpoint{0.104167in}{0.000000in}}%
\pgfpathcurveto{\pgfqpoint{0.104167in}{0.027625in}}{\pgfqpoint{0.093191in}{0.054123in}}{\pgfqpoint{0.073657in}{0.073657in}}%
\pgfpathcurveto{\pgfqpoint{0.054123in}{0.093191in}}{\pgfqpoint{0.027625in}{0.104167in}}{\pgfqpoint{0.000000in}{0.104167in}}%
\pgfpathcurveto{\pgfqpoint{-0.027625in}{0.104167in}}{\pgfqpoint{-0.054123in}{0.093191in}}{\pgfqpoint{-0.073657in}{0.073657in}}%
\pgfpathcurveto{\pgfqpoint{-0.093191in}{0.054123in}}{\pgfqpoint{-0.104167in}{0.027625in}}{\pgfqpoint{-0.104167in}{0.000000in}}%
\pgfpathcurveto{\pgfqpoint{-0.104167in}{-0.027625in}}{\pgfqpoint{-0.093191in}{-0.054123in}}{\pgfqpoint{-0.073657in}{-0.073657in}}%
\pgfpathcurveto{\pgfqpoint{-0.054123in}{-0.093191in}}{\pgfqpoint{-0.027625in}{-0.104167in}}{\pgfqpoint{0.000000in}{-0.104167in}}%
\pgfpathclose%
\pgfusepath{stroke,fill}%
}%
\begin{pgfscope}%
\pgfsys@transformshift{5.250000in}{1.350000in}%
\pgfsys@useobject{currentmarker}{}%
\end{pgfscope}%
\end{pgfscope}%
\begin{pgfscope}%
\pgfpathrectangle{\pgfqpoint{0.000000in}{0.000000in}}{\pgfqpoint{6.000000in}{3.600000in}}%
\pgfusepath{clip}%
\pgfsetbuttcap%
\pgfsetroundjoin%
\definecolor{currentfill}{rgb}{0.800000,0.800000,0.800000}%
\pgfsetfillcolor{currentfill}%
\pgfsetlinewidth{1.003750pt}%
\definecolor{currentstroke}{rgb}{0.000000,0.000000,0.000000}%
\pgfsetstrokecolor{currentstroke}%
\pgfsetdash{}{0pt}%
\pgfsys@defobject{currentmarker}{\pgfqpoint{-0.104167in}{-0.104167in}}{\pgfqpoint{0.104167in}{0.104167in}}{%
\pgfpathmoveto{\pgfqpoint{0.000000in}{-0.104167in}}%
\pgfpathcurveto{\pgfqpoint{0.027625in}{-0.104167in}}{\pgfqpoint{0.054123in}{-0.093191in}}{\pgfqpoint{0.073657in}{-0.073657in}}%
\pgfpathcurveto{\pgfqpoint{0.093191in}{-0.054123in}}{\pgfqpoint{0.104167in}{-0.027625in}}{\pgfqpoint{0.104167in}{0.000000in}}%
\pgfpathcurveto{\pgfqpoint{0.104167in}{0.027625in}}{\pgfqpoint{0.093191in}{0.054123in}}{\pgfqpoint{0.073657in}{0.073657in}}%
\pgfpathcurveto{\pgfqpoint{0.054123in}{0.093191in}}{\pgfqpoint{0.027625in}{0.104167in}}{\pgfqpoint{0.000000in}{0.104167in}}%
\pgfpathcurveto{\pgfqpoint{-0.027625in}{0.104167in}}{\pgfqpoint{-0.054123in}{0.093191in}}{\pgfqpoint{-0.073657in}{0.073657in}}%
\pgfpathcurveto{\pgfqpoint{-0.093191in}{0.054123in}}{\pgfqpoint{-0.104167in}{0.027625in}}{\pgfqpoint{-0.104167in}{0.000000in}}%
\pgfpathcurveto{\pgfqpoint{-0.104167in}{-0.027625in}}{\pgfqpoint{-0.093191in}{-0.054123in}}{\pgfqpoint{-0.073657in}{-0.073657in}}%
\pgfpathcurveto{\pgfqpoint{-0.054123in}{-0.093191in}}{\pgfqpoint{-0.027625in}{-0.104167in}}{\pgfqpoint{0.000000in}{-0.104167in}}%
\pgfpathclose%
\pgfusepath{stroke,fill}%
}%
\begin{pgfscope}%
\pgfsys@transformshift{2.250000in}{0.450000in}%
\pgfsys@useobject{currentmarker}{}%
\end{pgfscope}%
\end{pgfscope}%
\begin{pgfscope}%
\pgfpathrectangle{\pgfqpoint{0.000000in}{0.000000in}}{\pgfqpoint{6.000000in}{3.600000in}}%
\pgfusepath{clip}%
\pgfsetbuttcap%
\pgfsetroundjoin%
\definecolor{currentfill}{rgb}{1.000000,1.000000,1.000000}%
\pgfsetfillcolor{currentfill}%
\pgfsetlinewidth{1.003750pt}%
\definecolor{currentstroke}{rgb}{0.000000,0.000000,0.000000}%
\pgfsetstrokecolor{currentstroke}%
\pgfsetdash{}{0pt}%
\pgfsys@defobject{currentmarker}{\pgfqpoint{-0.104167in}{-0.104167in}}{\pgfqpoint{0.104167in}{0.104167in}}{%
\pgfpathmoveto{\pgfqpoint{0.000000in}{-0.104167in}}%
\pgfpathcurveto{\pgfqpoint{0.027625in}{-0.104167in}}{\pgfqpoint{0.054123in}{-0.093191in}}{\pgfqpoint{0.073657in}{-0.073657in}}%
\pgfpathcurveto{\pgfqpoint{0.093191in}{-0.054123in}}{\pgfqpoint{0.104167in}{-0.027625in}}{\pgfqpoint{0.104167in}{0.000000in}}%
\pgfpathcurveto{\pgfqpoint{0.104167in}{0.027625in}}{\pgfqpoint{0.093191in}{0.054123in}}{\pgfqpoint{0.073657in}{0.073657in}}%
\pgfpathcurveto{\pgfqpoint{0.054123in}{0.093191in}}{\pgfqpoint{0.027625in}{0.104167in}}{\pgfqpoint{0.000000in}{0.104167in}}%
\pgfpathcurveto{\pgfqpoint{-0.027625in}{0.104167in}}{\pgfqpoint{-0.054123in}{0.093191in}}{\pgfqpoint{-0.073657in}{0.073657in}}%
\pgfpathcurveto{\pgfqpoint{-0.093191in}{0.054123in}}{\pgfqpoint{-0.104167in}{0.027625in}}{\pgfqpoint{-0.104167in}{0.000000in}}%
\pgfpathcurveto{\pgfqpoint{-0.104167in}{-0.027625in}}{\pgfqpoint{-0.093191in}{-0.054123in}}{\pgfqpoint{-0.073657in}{-0.073657in}}%
\pgfpathcurveto{\pgfqpoint{-0.054123in}{-0.093191in}}{\pgfqpoint{-0.027625in}{-0.104167in}}{\pgfqpoint{0.000000in}{-0.104167in}}%
\pgfpathclose%
\pgfusepath{stroke,fill}%
}%
\begin{pgfscope}%
\pgfsys@transformshift{3.750000in}{0.450000in}%
\pgfsys@useobject{currentmarker}{}%
\end{pgfscope}%
\end{pgfscope}%
\begin{pgfscope}%
\definecolor{textcolor}{rgb}{0.000000,0.000000,0.000000}%
\pgfsetstrokecolor{textcolor}%
\pgfsetfillcolor{textcolor}%
\pgftext[x=5.529863in,y=1.970137in,,base]{\color{textcolor}\sffamily\fontsize{40.000000}{48.000000}\selectfont \(\displaystyle w_0\)}%
\end{pgfscope}%
\begin{pgfscope}%
\definecolor{textcolor}{rgb}{0.000000,0.000000,0.000000}%
\pgfsetstrokecolor{textcolor}%
\pgfsetfillcolor{textcolor}%
\pgftext[x=5.529863in,y=1.070137in,,base]{\color{textcolor}\sffamily\fontsize{40.000000}{48.000000}\selectfont \(\displaystyle z_0\)}%
\end{pgfscope}%
\end{pgfpicture}%
\makeatother%
\endgroup%

%% file: fig/decomposition_G_U1.pgf
%% Creator: Matplotlib, PGF backend
%%
%% To include the figure in your LaTeX document, write
%%   \input{<filename>.pgf}
%%
%% Make sure the required packages are loaded in your preamble
%%   \usepackage{pgf}
%%
%% Figures using additional raster images can only be included by \input if
%% they are in the same directory as the main LaTeX file. For loading figures
%% from other directories you can use the `import` package
%%   \usepackage{import}
%%
%% and then include the figures with
%%   \import{<path to file>}{<filename>.pgf}
%%
%% Matplotlib used the following preamble
%%   \usepackage{fontspec}
%%   \setmainfont{DejaVuSerif.ttf}[Path=\detokenize{C:/Users/ccros/Anaconda3/Lib/site-packages/matplotlib/mpl-data/fonts/ttf/}]
%%   \setsansfont{DejaVuSans.ttf}[Path=\detokenize{C:/Users/ccros/Anaconda3/Lib/site-packages/matplotlib/mpl-data/fonts/ttf/}]
%%   \setmonofont{DejaVuSansMono.ttf}[Path=\detokenize{C:/Users/ccros/Anaconda3/Lib/site-packages/matplotlib/mpl-data/fonts/ttf/}]
%%
\begingroup%
\makeatletter%
\begin{pgfpicture}%
\pgfpathrectangle{\pgfpointorigin}{\pgfqpoint{6.000000in}{3.600000in}}%
\pgfusepath{use as bounding box, clip}%
\begin{pgfscope}%
\pgfsetbuttcap%
\pgfsetmiterjoin%
\definecolor{currentfill}{rgb}{1.000000,1.000000,1.000000}%
\pgfsetfillcolor{currentfill}%
\pgfsetlinewidth{0.000000pt}%
\definecolor{currentstroke}{rgb}{1.000000,1.000000,1.000000}%
\pgfsetstrokecolor{currentstroke}%
\pgfsetdash{}{0pt}%
\pgfpathmoveto{\pgfqpoint{0.000000in}{0.000000in}}%
\pgfpathlineto{\pgfqpoint{6.000000in}{0.000000in}}%
\pgfpathlineto{\pgfqpoint{6.000000in}{3.600000in}}%
\pgfpathlineto{\pgfqpoint{0.000000in}{3.600000in}}%
\pgfpathclose%
\pgfusepath{fill}%
\end{pgfscope}%
\begin{pgfscope}%
\pgfpathrectangle{\pgfqpoint{0.000000in}{0.000000in}}{\pgfqpoint{6.000000in}{3.600000in}}%
\pgfusepath{clip}%
\pgfsetbuttcap%
\pgfsetroundjoin%
\pgfsetlinewidth{3.011250pt}%
\definecolor{currentstroke}{rgb}{0.400000,0.400000,0.400000}%
\pgfsetstrokecolor{currentstroke}%
\pgfsetdash{{11.100000pt}{4.800000pt}}{0.000000pt}%
\pgfpathmoveto{\pgfqpoint{2.250000in}{3.150000in}}%
\pgfpathlineto{\pgfqpoint{3.750000in}{3.150000in}}%
\pgfusepath{stroke}%
\end{pgfscope}%
\begin{pgfscope}%
\pgfpathrectangle{\pgfqpoint{0.000000in}{0.000000in}}{\pgfqpoint{6.000000in}{3.600000in}}%
\pgfusepath{clip}%
\pgfsetbuttcap%
\pgfsetroundjoin%
\pgfsetlinewidth{3.011250pt}%
\definecolor{currentstroke}{rgb}{0.400000,0.400000,0.400000}%
\pgfsetstrokecolor{currentstroke}%
\pgfsetdash{{11.100000pt}{4.800000pt}}{0.000000pt}%
\pgfpathmoveto{\pgfqpoint{2.250000in}{3.150000in}}%
\pgfpathlineto{\pgfqpoint{0.750000in}{2.250000in}}%
\pgfusepath{stroke}%
\end{pgfscope}%
\begin{pgfscope}%
\pgfpathrectangle{\pgfqpoint{0.000000in}{0.000000in}}{\pgfqpoint{6.000000in}{3.600000in}}%
\pgfusepath{clip}%
\pgfsetbuttcap%
\pgfsetroundjoin%
\pgfsetlinewidth{3.011250pt}%
\definecolor{currentstroke}{rgb}{0.400000,0.400000,0.400000}%
\pgfsetstrokecolor{currentstroke}%
\pgfsetdash{{11.100000pt}{4.800000pt}}{0.000000pt}%
\pgfpathmoveto{\pgfqpoint{2.250000in}{3.150000in}}%
\pgfpathlineto{\pgfqpoint{5.250000in}{2.250000in}}%
\pgfusepath{stroke}%
\end{pgfscope}%
\begin{pgfscope}%
\pgfpathrectangle{\pgfqpoint{0.000000in}{0.000000in}}{\pgfqpoint{6.000000in}{3.600000in}}%
\pgfusepath{clip}%
\pgfsetbuttcap%
\pgfsetroundjoin%
\pgfsetlinewidth{3.011250pt}%
\definecolor{currentstroke}{rgb}{0.400000,0.400000,0.400000}%
\pgfsetstrokecolor{currentstroke}%
\pgfsetdash{{11.100000pt}{4.800000pt}}{0.000000pt}%
\pgfpathmoveto{\pgfqpoint{3.750000in}{3.150000in}}%
\pgfpathlineto{\pgfqpoint{5.250000in}{2.250000in}}%
\pgfusepath{stroke}%
\end{pgfscope}%
\begin{pgfscope}%
\pgfpathrectangle{\pgfqpoint{0.000000in}{0.000000in}}{\pgfqpoint{6.000000in}{3.600000in}}%
\pgfusepath{clip}%
\pgfsetbuttcap%
\pgfsetroundjoin%
\pgfsetlinewidth{3.011250pt}%
\definecolor{currentstroke}{rgb}{0.400000,0.400000,0.400000}%
\pgfsetstrokecolor{currentstroke}%
\pgfsetdash{{11.100000pt}{4.800000pt}}{0.000000pt}%
\pgfpathmoveto{\pgfqpoint{3.750000in}{3.150000in}}%
\pgfpathlineto{\pgfqpoint{2.250000in}{1.800000in}}%
\pgfusepath{stroke}%
\end{pgfscope}%
\begin{pgfscope}%
\pgfpathrectangle{\pgfqpoint{0.000000in}{0.000000in}}{\pgfqpoint{6.000000in}{3.600000in}}%
\pgfusepath{clip}%
\pgfsetbuttcap%
\pgfsetroundjoin%
\pgfsetlinewidth{3.011250pt}%
\definecolor{currentstroke}{rgb}{0.400000,0.400000,0.400000}%
\pgfsetstrokecolor{currentstroke}%
\pgfsetdash{{11.100000pt}{4.800000pt}}{0.000000pt}%
\pgfpathmoveto{\pgfqpoint{0.750000in}{2.250000in}}%
\pgfpathlineto{\pgfqpoint{2.250000in}{1.800000in}}%
\pgfusepath{stroke}%
\end{pgfscope}%
\begin{pgfscope}%
\pgfpathrectangle{\pgfqpoint{0.000000in}{0.000000in}}{\pgfqpoint{6.000000in}{3.600000in}}%
\pgfusepath{clip}%
\pgfsetbuttcap%
\pgfsetroundjoin%
\pgfsetlinewidth{3.011250pt}%
\definecolor{currentstroke}{rgb}{0.400000,0.400000,0.400000}%
\pgfsetstrokecolor{currentstroke}%
\pgfsetdash{{11.100000pt}{4.800000pt}}{0.000000pt}%
\pgfpathmoveto{\pgfqpoint{0.750000in}{2.250000in}}%
\pgfpathlineto{\pgfqpoint{0.750000in}{1.350000in}}%
\pgfusepath{stroke}%
\end{pgfscope}%
\begin{pgfscope}%
\pgfpathrectangle{\pgfqpoint{0.000000in}{0.000000in}}{\pgfqpoint{6.000000in}{3.600000in}}%
\pgfusepath{clip}%
\pgfsetbuttcap%
\pgfsetroundjoin%
\pgfsetlinewidth{3.011250pt}%
\definecolor{currentstroke}{rgb}{0.400000,0.400000,0.400000}%
\pgfsetstrokecolor{currentstroke}%
\pgfsetdash{{11.100000pt}{4.800000pt}}{0.000000pt}%
\pgfpathmoveto{\pgfqpoint{5.250000in}{2.250000in}}%
\pgfpathlineto{\pgfqpoint{2.250000in}{1.800000in}}%
\pgfusepath{stroke}%
\end{pgfscope}%
\begin{pgfscope}%
\pgfpathrectangle{\pgfqpoint{0.000000in}{0.000000in}}{\pgfqpoint{6.000000in}{3.600000in}}%
\pgfusepath{clip}%
\pgfsetbuttcap%
\pgfsetroundjoin%
\pgfsetlinewidth{3.011250pt}%
\definecolor{currentstroke}{rgb}{0.400000,0.400000,0.400000}%
\pgfsetstrokecolor{currentstroke}%
\pgfsetdash{{11.100000pt}{4.800000pt}}{0.000000pt}%
\pgfpathmoveto{\pgfqpoint{5.250000in}{2.250000in}}%
\pgfpathlineto{\pgfqpoint{3.750000in}{1.800000in}}%
\pgfusepath{stroke}%
\end{pgfscope}%
\begin{pgfscope}%
\pgfpathrectangle{\pgfqpoint{0.000000in}{0.000000in}}{\pgfqpoint{6.000000in}{3.600000in}}%
\pgfusepath{clip}%
\pgfsetbuttcap%
\pgfsetroundjoin%
\pgfsetlinewidth{3.011250pt}%
\definecolor{currentstroke}{rgb}{0.400000,0.400000,0.400000}%
\pgfsetstrokecolor{currentstroke}%
\pgfsetdash{{11.100000pt}{4.800000pt}}{0.000000pt}%
\pgfpathmoveto{\pgfqpoint{5.250000in}{2.250000in}}%
\pgfpathlineto{\pgfqpoint{5.250000in}{1.350000in}}%
\pgfusepath{stroke}%
\end{pgfscope}%
\begin{pgfscope}%
\pgfpathrectangle{\pgfqpoint{0.000000in}{0.000000in}}{\pgfqpoint{6.000000in}{3.600000in}}%
\pgfusepath{clip}%
\pgfsetbuttcap%
\pgfsetroundjoin%
\pgfsetlinewidth{3.011250pt}%
\definecolor{currentstroke}{rgb}{0.400000,0.400000,0.400000}%
\pgfsetstrokecolor{currentstroke}%
\pgfsetdash{{11.100000pt}{4.800000pt}}{0.000000pt}%
\pgfpathmoveto{\pgfqpoint{2.250000in}{1.800000in}}%
\pgfpathlineto{\pgfqpoint{5.250000in}{1.350000in}}%
\pgfusepath{stroke}%
\end{pgfscope}%
\begin{pgfscope}%
\pgfpathrectangle{\pgfqpoint{0.000000in}{0.000000in}}{\pgfqpoint{6.000000in}{3.600000in}}%
\pgfusepath{clip}%
\pgfsetbuttcap%
\pgfsetroundjoin%
\pgfsetlinewidth{3.011250pt}%
\definecolor{currentstroke}{rgb}{0.400000,0.400000,0.400000}%
\pgfsetstrokecolor{currentstroke}%
\pgfsetdash{{11.100000pt}{4.800000pt}}{0.000000pt}%
\pgfpathmoveto{\pgfqpoint{3.750000in}{1.800000in}}%
\pgfpathlineto{\pgfqpoint{5.250000in}{1.350000in}}%
\pgfusepath{stroke}%
\end{pgfscope}%
\begin{pgfscope}%
\pgfpathrectangle{\pgfqpoint{0.000000in}{0.000000in}}{\pgfqpoint{6.000000in}{3.600000in}}%
\pgfusepath{clip}%
\pgfsetrectcap%
\pgfsetroundjoin%
\pgfsetlinewidth{3.011250pt}%
\definecolor{currentstroke}{rgb}{0.250980,0.250980,0.250980}%
\pgfsetstrokecolor{currentstroke}%
\pgfsetdash{}{0pt}%
\pgfpathmoveto{\pgfqpoint{0.750000in}{1.350000in}}%
\pgfpathlineto{\pgfqpoint{2.250000in}{1.800000in}}%
\pgfusepath{stroke}%
\end{pgfscope}%
\begin{pgfscope}%
\pgfpathrectangle{\pgfqpoint{0.000000in}{0.000000in}}{\pgfqpoint{6.000000in}{3.600000in}}%
\pgfusepath{clip}%
\pgfsetrectcap%
\pgfsetroundjoin%
\pgfsetlinewidth{3.011250pt}%
\definecolor{currentstroke}{rgb}{0.250980,0.250980,0.250980}%
\pgfsetstrokecolor{currentstroke}%
\pgfsetdash{}{0pt}%
\pgfpathmoveto{\pgfqpoint{2.250000in}{0.450000in}}%
\pgfpathlineto{\pgfqpoint{2.250000in}{1.800000in}}%
\pgfusepath{stroke}%
\end{pgfscope}%
\begin{pgfscope}%
\pgfpathrectangle{\pgfqpoint{0.000000in}{0.000000in}}{\pgfqpoint{6.000000in}{3.600000in}}%
\pgfusepath{clip}%
\pgfsetrectcap%
\pgfsetroundjoin%
\pgfsetlinewidth{3.011250pt}%
\definecolor{currentstroke}{rgb}{0.250980,0.250980,0.250980}%
\pgfsetstrokecolor{currentstroke}%
\pgfsetdash{}{0pt}%
\pgfpathmoveto{\pgfqpoint{3.750000in}{0.450000in}}%
\pgfpathlineto{\pgfqpoint{2.250000in}{1.800000in}}%
\pgfusepath{stroke}%
\end{pgfscope}%
\begin{pgfscope}%
\pgfpathrectangle{\pgfqpoint{0.000000in}{0.000000in}}{\pgfqpoint{6.000000in}{3.600000in}}%
\pgfusepath{clip}%
\pgfsetrectcap%
\pgfsetroundjoin%
\pgfsetlinewidth{3.011250pt}%
\definecolor{currentstroke}{rgb}{0.250980,0.250980,0.250980}%
\pgfsetstrokecolor{currentstroke}%
\pgfsetdash{}{0pt}%
\pgfpathmoveto{\pgfqpoint{3.750000in}{0.450000in}}%
\pgfpathlineto{\pgfqpoint{0.750000in}{1.350000in}}%
\pgfusepath{stroke}%
\end{pgfscope}%
\begin{pgfscope}%
\pgfpathrectangle{\pgfqpoint{0.000000in}{0.000000in}}{\pgfqpoint{6.000000in}{3.600000in}}%
\pgfusepath{clip}%
\pgfsetrectcap%
\pgfsetroundjoin%
\pgfsetlinewidth{3.011250pt}%
\definecolor{currentstroke}{rgb}{0.250980,0.250980,0.250980}%
\pgfsetstrokecolor{currentstroke}%
\pgfsetdash{}{0pt}%
\pgfpathmoveto{\pgfqpoint{3.750000in}{0.450000in}}%
\pgfpathlineto{\pgfqpoint{2.250000in}{0.450000in}}%
\pgfusepath{stroke}%
\end{pgfscope}%
\begin{pgfscope}%
\pgfpathrectangle{\pgfqpoint{0.000000in}{0.000000in}}{\pgfqpoint{6.000000in}{3.600000in}}%
\pgfusepath{clip}%
\pgfsetbuttcap%
\pgfsetroundjoin%
\definecolor{currentfill}{rgb}{0.800000,0.800000,0.800000}%
\pgfsetfillcolor{currentfill}%
\pgfsetlinewidth{1.003750pt}%
\definecolor{currentstroke}{rgb}{0.000000,0.000000,0.000000}%
\pgfsetstrokecolor{currentstroke}%
\pgfsetdash{}{0pt}%
\pgfsys@defobject{currentmarker}{\pgfqpoint{-0.104167in}{-0.104167in}}{\pgfqpoint{0.104167in}{0.104167in}}{%
\pgfpathmoveto{\pgfqpoint{0.000000in}{-0.104167in}}%
\pgfpathcurveto{\pgfqpoint{0.027625in}{-0.104167in}}{\pgfqpoint{0.054123in}{-0.093191in}}{\pgfqpoint{0.073657in}{-0.073657in}}%
\pgfpathcurveto{\pgfqpoint{0.093191in}{-0.054123in}}{\pgfqpoint{0.104167in}{-0.027625in}}{\pgfqpoint{0.104167in}{0.000000in}}%
\pgfpathcurveto{\pgfqpoint{0.104167in}{0.027625in}}{\pgfqpoint{0.093191in}{0.054123in}}{\pgfqpoint{0.073657in}{0.073657in}}%
\pgfpathcurveto{\pgfqpoint{0.054123in}{0.093191in}}{\pgfqpoint{0.027625in}{0.104167in}}{\pgfqpoint{0.000000in}{0.104167in}}%
\pgfpathcurveto{\pgfqpoint{-0.027625in}{0.104167in}}{\pgfqpoint{-0.054123in}{0.093191in}}{\pgfqpoint{-0.073657in}{0.073657in}}%
\pgfpathcurveto{\pgfqpoint{-0.093191in}{0.054123in}}{\pgfqpoint{-0.104167in}{0.027625in}}{\pgfqpoint{-0.104167in}{0.000000in}}%
\pgfpathcurveto{\pgfqpoint{-0.104167in}{-0.027625in}}{\pgfqpoint{-0.093191in}{-0.054123in}}{\pgfqpoint{-0.073657in}{-0.073657in}}%
\pgfpathcurveto{\pgfqpoint{-0.054123in}{-0.093191in}}{\pgfqpoint{-0.027625in}{-0.104167in}}{\pgfqpoint{0.000000in}{-0.104167in}}%
\pgfpathclose%
\pgfusepath{stroke,fill}%
}%
\begin{pgfscope}%
\pgfsys@transformshift{2.250000in}{3.150000in}%
\pgfsys@useobject{currentmarker}{}%
\end{pgfscope}%
\end{pgfscope}%
\begin{pgfscope}%
\pgfpathrectangle{\pgfqpoint{0.000000in}{0.000000in}}{\pgfqpoint{6.000000in}{3.600000in}}%
\pgfusepath{clip}%
\pgfsetbuttcap%
\pgfsetroundjoin%
\definecolor{currentfill}{rgb}{0.800000,0.800000,0.800000}%
\pgfsetfillcolor{currentfill}%
\pgfsetlinewidth{1.003750pt}%
\definecolor{currentstroke}{rgb}{0.000000,0.000000,0.000000}%
\pgfsetstrokecolor{currentstroke}%
\pgfsetdash{}{0pt}%
\pgfsys@defobject{currentmarker}{\pgfqpoint{-0.104167in}{-0.104167in}}{\pgfqpoint{0.104167in}{0.104167in}}{%
\pgfpathmoveto{\pgfqpoint{0.000000in}{-0.104167in}}%
\pgfpathcurveto{\pgfqpoint{0.027625in}{-0.104167in}}{\pgfqpoint{0.054123in}{-0.093191in}}{\pgfqpoint{0.073657in}{-0.073657in}}%
\pgfpathcurveto{\pgfqpoint{0.093191in}{-0.054123in}}{\pgfqpoint{0.104167in}{-0.027625in}}{\pgfqpoint{0.104167in}{0.000000in}}%
\pgfpathcurveto{\pgfqpoint{0.104167in}{0.027625in}}{\pgfqpoint{0.093191in}{0.054123in}}{\pgfqpoint{0.073657in}{0.073657in}}%
\pgfpathcurveto{\pgfqpoint{0.054123in}{0.093191in}}{\pgfqpoint{0.027625in}{0.104167in}}{\pgfqpoint{0.000000in}{0.104167in}}%
\pgfpathcurveto{\pgfqpoint{-0.027625in}{0.104167in}}{\pgfqpoint{-0.054123in}{0.093191in}}{\pgfqpoint{-0.073657in}{0.073657in}}%
\pgfpathcurveto{\pgfqpoint{-0.093191in}{0.054123in}}{\pgfqpoint{-0.104167in}{0.027625in}}{\pgfqpoint{-0.104167in}{0.000000in}}%
\pgfpathcurveto{\pgfqpoint{-0.104167in}{-0.027625in}}{\pgfqpoint{-0.093191in}{-0.054123in}}{\pgfqpoint{-0.073657in}{-0.073657in}}%
\pgfpathcurveto{\pgfqpoint{-0.054123in}{-0.093191in}}{\pgfqpoint{-0.027625in}{-0.104167in}}{\pgfqpoint{0.000000in}{-0.104167in}}%
\pgfpathclose%
\pgfusepath{stroke,fill}%
}%
\begin{pgfscope}%
\pgfsys@transformshift{3.750000in}{3.150000in}%
\pgfsys@useobject{currentmarker}{}%
\end{pgfscope}%
\end{pgfscope}%
\begin{pgfscope}%
\pgfpathrectangle{\pgfqpoint{0.000000in}{0.000000in}}{\pgfqpoint{6.000000in}{3.600000in}}%
\pgfusepath{clip}%
\pgfsetbuttcap%
\pgfsetroundjoin%
\definecolor{currentfill}{rgb}{0.800000,0.800000,0.800000}%
\pgfsetfillcolor{currentfill}%
\pgfsetlinewidth{1.003750pt}%
\definecolor{currentstroke}{rgb}{0.000000,0.000000,0.000000}%
\pgfsetstrokecolor{currentstroke}%
\pgfsetdash{}{0pt}%
\pgfsys@defobject{currentmarker}{\pgfqpoint{-0.104167in}{-0.104167in}}{\pgfqpoint{0.104167in}{0.104167in}}{%
\pgfpathmoveto{\pgfqpoint{0.000000in}{-0.104167in}}%
\pgfpathcurveto{\pgfqpoint{0.027625in}{-0.104167in}}{\pgfqpoint{0.054123in}{-0.093191in}}{\pgfqpoint{0.073657in}{-0.073657in}}%
\pgfpathcurveto{\pgfqpoint{0.093191in}{-0.054123in}}{\pgfqpoint{0.104167in}{-0.027625in}}{\pgfqpoint{0.104167in}{0.000000in}}%
\pgfpathcurveto{\pgfqpoint{0.104167in}{0.027625in}}{\pgfqpoint{0.093191in}{0.054123in}}{\pgfqpoint{0.073657in}{0.073657in}}%
\pgfpathcurveto{\pgfqpoint{0.054123in}{0.093191in}}{\pgfqpoint{0.027625in}{0.104167in}}{\pgfqpoint{0.000000in}{0.104167in}}%
\pgfpathcurveto{\pgfqpoint{-0.027625in}{0.104167in}}{\pgfqpoint{-0.054123in}{0.093191in}}{\pgfqpoint{-0.073657in}{0.073657in}}%
\pgfpathcurveto{\pgfqpoint{-0.093191in}{0.054123in}}{\pgfqpoint{-0.104167in}{0.027625in}}{\pgfqpoint{-0.104167in}{0.000000in}}%
\pgfpathcurveto{\pgfqpoint{-0.104167in}{-0.027625in}}{\pgfqpoint{-0.093191in}{-0.054123in}}{\pgfqpoint{-0.073657in}{-0.073657in}}%
\pgfpathcurveto{\pgfqpoint{-0.054123in}{-0.093191in}}{\pgfqpoint{-0.027625in}{-0.104167in}}{\pgfqpoint{0.000000in}{-0.104167in}}%
\pgfpathclose%
\pgfusepath{stroke,fill}%
}%
\begin{pgfscope}%
\pgfsys@transformshift{0.750000in}{2.250000in}%
\pgfsys@useobject{currentmarker}{}%
\end{pgfscope}%
\end{pgfscope}%
\begin{pgfscope}%
\pgfpathrectangle{\pgfqpoint{0.000000in}{0.000000in}}{\pgfqpoint{6.000000in}{3.600000in}}%
\pgfusepath{clip}%
\pgfsetbuttcap%
\pgfsetroundjoin%
\definecolor{currentfill}{rgb}{1.000000,1.000000,1.000000}%
\pgfsetfillcolor{currentfill}%
\pgfsetlinewidth{1.003750pt}%
\definecolor{currentstroke}{rgb}{0.000000,0.000000,0.000000}%
\pgfsetstrokecolor{currentstroke}%
\pgfsetdash{}{0pt}%
\pgfsys@defobject{currentmarker}{\pgfqpoint{-0.104167in}{-0.104167in}}{\pgfqpoint{0.104167in}{0.104167in}}{%
\pgfpathmoveto{\pgfqpoint{0.000000in}{-0.104167in}}%
\pgfpathcurveto{\pgfqpoint{0.027625in}{-0.104167in}}{\pgfqpoint{0.054123in}{-0.093191in}}{\pgfqpoint{0.073657in}{-0.073657in}}%
\pgfpathcurveto{\pgfqpoint{0.093191in}{-0.054123in}}{\pgfqpoint{0.104167in}{-0.027625in}}{\pgfqpoint{0.104167in}{0.000000in}}%
\pgfpathcurveto{\pgfqpoint{0.104167in}{0.027625in}}{\pgfqpoint{0.093191in}{0.054123in}}{\pgfqpoint{0.073657in}{0.073657in}}%
\pgfpathcurveto{\pgfqpoint{0.054123in}{0.093191in}}{\pgfqpoint{0.027625in}{0.104167in}}{\pgfqpoint{0.000000in}{0.104167in}}%
\pgfpathcurveto{\pgfqpoint{-0.027625in}{0.104167in}}{\pgfqpoint{-0.054123in}{0.093191in}}{\pgfqpoint{-0.073657in}{0.073657in}}%
\pgfpathcurveto{\pgfqpoint{-0.093191in}{0.054123in}}{\pgfqpoint{-0.104167in}{0.027625in}}{\pgfqpoint{-0.104167in}{0.000000in}}%
\pgfpathcurveto{\pgfqpoint{-0.104167in}{-0.027625in}}{\pgfqpoint{-0.093191in}{-0.054123in}}{\pgfqpoint{-0.073657in}{-0.073657in}}%
\pgfpathcurveto{\pgfqpoint{-0.054123in}{-0.093191in}}{\pgfqpoint{-0.027625in}{-0.104167in}}{\pgfqpoint{0.000000in}{-0.104167in}}%
\pgfpathclose%
\pgfusepath{stroke,fill}%
}%
\begin{pgfscope}%
\pgfsys@transformshift{5.250000in}{2.250000in}%
\pgfsys@useobject{currentmarker}{}%
\end{pgfscope}%
\end{pgfscope}%
\begin{pgfscope}%
\pgfpathrectangle{\pgfqpoint{0.000000in}{0.000000in}}{\pgfqpoint{6.000000in}{3.600000in}}%
\pgfusepath{clip}%
\pgfsetbuttcap%
\pgfsetroundjoin%
\definecolor{currentfill}{rgb}{0.800000,0.800000,0.800000}%
\pgfsetfillcolor{currentfill}%
\pgfsetlinewidth{1.003750pt}%
\definecolor{currentstroke}{rgb}{0.000000,0.000000,0.000000}%
\pgfsetstrokecolor{currentstroke}%
\pgfsetdash{}{0pt}%
\pgfsys@defobject{currentmarker}{\pgfqpoint{-0.104167in}{-0.104167in}}{\pgfqpoint{0.104167in}{0.104167in}}{%
\pgfpathmoveto{\pgfqpoint{0.000000in}{-0.104167in}}%
\pgfpathcurveto{\pgfqpoint{0.027625in}{-0.104167in}}{\pgfqpoint{0.054123in}{-0.093191in}}{\pgfqpoint{0.073657in}{-0.073657in}}%
\pgfpathcurveto{\pgfqpoint{0.093191in}{-0.054123in}}{\pgfqpoint{0.104167in}{-0.027625in}}{\pgfqpoint{0.104167in}{0.000000in}}%
\pgfpathcurveto{\pgfqpoint{0.104167in}{0.027625in}}{\pgfqpoint{0.093191in}{0.054123in}}{\pgfqpoint{0.073657in}{0.073657in}}%
\pgfpathcurveto{\pgfqpoint{0.054123in}{0.093191in}}{\pgfqpoint{0.027625in}{0.104167in}}{\pgfqpoint{0.000000in}{0.104167in}}%
\pgfpathcurveto{\pgfqpoint{-0.027625in}{0.104167in}}{\pgfqpoint{-0.054123in}{0.093191in}}{\pgfqpoint{-0.073657in}{0.073657in}}%
\pgfpathcurveto{\pgfqpoint{-0.093191in}{0.054123in}}{\pgfqpoint{-0.104167in}{0.027625in}}{\pgfqpoint{-0.104167in}{0.000000in}}%
\pgfpathcurveto{\pgfqpoint{-0.104167in}{-0.027625in}}{\pgfqpoint{-0.093191in}{-0.054123in}}{\pgfqpoint{-0.073657in}{-0.073657in}}%
\pgfpathcurveto{\pgfqpoint{-0.054123in}{-0.093191in}}{\pgfqpoint{-0.027625in}{-0.104167in}}{\pgfqpoint{0.000000in}{-0.104167in}}%
\pgfpathclose%
\pgfusepath{stroke,fill}%
}%
\begin{pgfscope}%
\pgfsys@transformshift{2.250000in}{1.800000in}%
\pgfsys@useobject{currentmarker}{}%
\end{pgfscope}%
\end{pgfscope}%
\begin{pgfscope}%
\pgfpathrectangle{\pgfqpoint{0.000000in}{0.000000in}}{\pgfqpoint{6.000000in}{3.600000in}}%
\pgfusepath{clip}%
\pgfsetbuttcap%
\pgfsetroundjoin%
\definecolor{currentfill}{rgb}{1.000000,1.000000,1.000000}%
\pgfsetfillcolor{currentfill}%
\pgfsetlinewidth{1.003750pt}%
\definecolor{currentstroke}{rgb}{0.000000,0.000000,0.000000}%
\pgfsetstrokecolor{currentstroke}%
\pgfsetdash{}{0pt}%
\pgfsys@defobject{currentmarker}{\pgfqpoint{-0.104167in}{-0.104167in}}{\pgfqpoint{0.104167in}{0.104167in}}{%
\pgfpathmoveto{\pgfqpoint{0.000000in}{-0.104167in}}%
\pgfpathcurveto{\pgfqpoint{0.027625in}{-0.104167in}}{\pgfqpoint{0.054123in}{-0.093191in}}{\pgfqpoint{0.073657in}{-0.073657in}}%
\pgfpathcurveto{\pgfqpoint{0.093191in}{-0.054123in}}{\pgfqpoint{0.104167in}{-0.027625in}}{\pgfqpoint{0.104167in}{0.000000in}}%
\pgfpathcurveto{\pgfqpoint{0.104167in}{0.027625in}}{\pgfqpoint{0.093191in}{0.054123in}}{\pgfqpoint{0.073657in}{0.073657in}}%
\pgfpathcurveto{\pgfqpoint{0.054123in}{0.093191in}}{\pgfqpoint{0.027625in}{0.104167in}}{\pgfqpoint{0.000000in}{0.104167in}}%
\pgfpathcurveto{\pgfqpoint{-0.027625in}{0.104167in}}{\pgfqpoint{-0.054123in}{0.093191in}}{\pgfqpoint{-0.073657in}{0.073657in}}%
\pgfpathcurveto{\pgfqpoint{-0.093191in}{0.054123in}}{\pgfqpoint{-0.104167in}{0.027625in}}{\pgfqpoint{-0.104167in}{0.000000in}}%
\pgfpathcurveto{\pgfqpoint{-0.104167in}{-0.027625in}}{\pgfqpoint{-0.093191in}{-0.054123in}}{\pgfqpoint{-0.073657in}{-0.073657in}}%
\pgfpathcurveto{\pgfqpoint{-0.054123in}{-0.093191in}}{\pgfqpoint{-0.027625in}{-0.104167in}}{\pgfqpoint{0.000000in}{-0.104167in}}%
\pgfpathclose%
\pgfusepath{stroke,fill}%
}%
\begin{pgfscope}%
\pgfsys@transformshift{3.750000in}{1.800000in}%
\pgfsys@useobject{currentmarker}{}%
\end{pgfscope}%
\end{pgfscope}%
\begin{pgfscope}%
\pgfpathrectangle{\pgfqpoint{0.000000in}{0.000000in}}{\pgfqpoint{6.000000in}{3.600000in}}%
\pgfusepath{clip}%
\pgfsetbuttcap%
\pgfsetroundjoin%
\definecolor{currentfill}{rgb}{0.800000,0.800000,0.800000}%
\pgfsetfillcolor{currentfill}%
\pgfsetlinewidth{1.003750pt}%
\definecolor{currentstroke}{rgb}{0.000000,0.000000,0.000000}%
\pgfsetstrokecolor{currentstroke}%
\pgfsetdash{}{0pt}%
\pgfsys@defobject{currentmarker}{\pgfqpoint{-0.104167in}{-0.104167in}}{\pgfqpoint{0.104167in}{0.104167in}}{%
\pgfpathmoveto{\pgfqpoint{0.000000in}{-0.104167in}}%
\pgfpathcurveto{\pgfqpoint{0.027625in}{-0.104167in}}{\pgfqpoint{0.054123in}{-0.093191in}}{\pgfqpoint{0.073657in}{-0.073657in}}%
\pgfpathcurveto{\pgfqpoint{0.093191in}{-0.054123in}}{\pgfqpoint{0.104167in}{-0.027625in}}{\pgfqpoint{0.104167in}{0.000000in}}%
\pgfpathcurveto{\pgfqpoint{0.104167in}{0.027625in}}{\pgfqpoint{0.093191in}{0.054123in}}{\pgfqpoint{0.073657in}{0.073657in}}%
\pgfpathcurveto{\pgfqpoint{0.054123in}{0.093191in}}{\pgfqpoint{0.027625in}{0.104167in}}{\pgfqpoint{0.000000in}{0.104167in}}%
\pgfpathcurveto{\pgfqpoint{-0.027625in}{0.104167in}}{\pgfqpoint{-0.054123in}{0.093191in}}{\pgfqpoint{-0.073657in}{0.073657in}}%
\pgfpathcurveto{\pgfqpoint{-0.093191in}{0.054123in}}{\pgfqpoint{-0.104167in}{0.027625in}}{\pgfqpoint{-0.104167in}{0.000000in}}%
\pgfpathcurveto{\pgfqpoint{-0.104167in}{-0.027625in}}{\pgfqpoint{-0.093191in}{-0.054123in}}{\pgfqpoint{-0.073657in}{-0.073657in}}%
\pgfpathcurveto{\pgfqpoint{-0.054123in}{-0.093191in}}{\pgfqpoint{-0.027625in}{-0.104167in}}{\pgfqpoint{0.000000in}{-0.104167in}}%
\pgfpathclose%
\pgfusepath{stroke,fill}%
}%
\begin{pgfscope}%
\pgfsys@transformshift{0.750000in}{1.350000in}%
\pgfsys@useobject{currentmarker}{}%
\end{pgfscope}%
\end{pgfscope}%
\begin{pgfscope}%
\pgfpathrectangle{\pgfqpoint{0.000000in}{0.000000in}}{\pgfqpoint{6.000000in}{3.600000in}}%
\pgfusepath{clip}%
\pgfsetbuttcap%
\pgfsetroundjoin%
\definecolor{currentfill}{rgb}{1.000000,1.000000,1.000000}%
\pgfsetfillcolor{currentfill}%
\pgfsetlinewidth{1.003750pt}%
\definecolor{currentstroke}{rgb}{0.000000,0.000000,0.000000}%
\pgfsetstrokecolor{currentstroke}%
\pgfsetdash{}{0pt}%
\pgfsys@defobject{currentmarker}{\pgfqpoint{-0.104167in}{-0.104167in}}{\pgfqpoint{0.104167in}{0.104167in}}{%
\pgfpathmoveto{\pgfqpoint{0.000000in}{-0.104167in}}%
\pgfpathcurveto{\pgfqpoint{0.027625in}{-0.104167in}}{\pgfqpoint{0.054123in}{-0.093191in}}{\pgfqpoint{0.073657in}{-0.073657in}}%
\pgfpathcurveto{\pgfqpoint{0.093191in}{-0.054123in}}{\pgfqpoint{0.104167in}{-0.027625in}}{\pgfqpoint{0.104167in}{0.000000in}}%
\pgfpathcurveto{\pgfqpoint{0.104167in}{0.027625in}}{\pgfqpoint{0.093191in}{0.054123in}}{\pgfqpoint{0.073657in}{0.073657in}}%
\pgfpathcurveto{\pgfqpoint{0.054123in}{0.093191in}}{\pgfqpoint{0.027625in}{0.104167in}}{\pgfqpoint{0.000000in}{0.104167in}}%
\pgfpathcurveto{\pgfqpoint{-0.027625in}{0.104167in}}{\pgfqpoint{-0.054123in}{0.093191in}}{\pgfqpoint{-0.073657in}{0.073657in}}%
\pgfpathcurveto{\pgfqpoint{-0.093191in}{0.054123in}}{\pgfqpoint{-0.104167in}{0.027625in}}{\pgfqpoint{-0.104167in}{0.000000in}}%
\pgfpathcurveto{\pgfqpoint{-0.104167in}{-0.027625in}}{\pgfqpoint{-0.093191in}{-0.054123in}}{\pgfqpoint{-0.073657in}{-0.073657in}}%
\pgfpathcurveto{\pgfqpoint{-0.054123in}{-0.093191in}}{\pgfqpoint{-0.027625in}{-0.104167in}}{\pgfqpoint{0.000000in}{-0.104167in}}%
\pgfpathclose%
\pgfusepath{stroke,fill}%
}%
\begin{pgfscope}%
\pgfsys@transformshift{5.250000in}{1.350000in}%
\pgfsys@useobject{currentmarker}{}%
\end{pgfscope}%
\end{pgfscope}%
\begin{pgfscope}%
\pgfpathrectangle{\pgfqpoint{0.000000in}{0.000000in}}{\pgfqpoint{6.000000in}{3.600000in}}%
\pgfusepath{clip}%
\pgfsetbuttcap%
\pgfsetroundjoin%
\definecolor{currentfill}{rgb}{0.800000,0.800000,0.800000}%
\pgfsetfillcolor{currentfill}%
\pgfsetlinewidth{1.003750pt}%
\definecolor{currentstroke}{rgb}{0.000000,0.000000,0.000000}%
\pgfsetstrokecolor{currentstroke}%
\pgfsetdash{}{0pt}%
\pgfsys@defobject{currentmarker}{\pgfqpoint{-0.104167in}{-0.104167in}}{\pgfqpoint{0.104167in}{0.104167in}}{%
\pgfpathmoveto{\pgfqpoint{0.000000in}{-0.104167in}}%
\pgfpathcurveto{\pgfqpoint{0.027625in}{-0.104167in}}{\pgfqpoint{0.054123in}{-0.093191in}}{\pgfqpoint{0.073657in}{-0.073657in}}%
\pgfpathcurveto{\pgfqpoint{0.093191in}{-0.054123in}}{\pgfqpoint{0.104167in}{-0.027625in}}{\pgfqpoint{0.104167in}{0.000000in}}%
\pgfpathcurveto{\pgfqpoint{0.104167in}{0.027625in}}{\pgfqpoint{0.093191in}{0.054123in}}{\pgfqpoint{0.073657in}{0.073657in}}%
\pgfpathcurveto{\pgfqpoint{0.054123in}{0.093191in}}{\pgfqpoint{0.027625in}{0.104167in}}{\pgfqpoint{0.000000in}{0.104167in}}%
\pgfpathcurveto{\pgfqpoint{-0.027625in}{0.104167in}}{\pgfqpoint{-0.054123in}{0.093191in}}{\pgfqpoint{-0.073657in}{0.073657in}}%
\pgfpathcurveto{\pgfqpoint{-0.093191in}{0.054123in}}{\pgfqpoint{-0.104167in}{0.027625in}}{\pgfqpoint{-0.104167in}{0.000000in}}%
\pgfpathcurveto{\pgfqpoint{-0.104167in}{-0.027625in}}{\pgfqpoint{-0.093191in}{-0.054123in}}{\pgfqpoint{-0.073657in}{-0.073657in}}%
\pgfpathcurveto{\pgfqpoint{-0.054123in}{-0.093191in}}{\pgfqpoint{-0.027625in}{-0.104167in}}{\pgfqpoint{0.000000in}{-0.104167in}}%
\pgfpathclose%
\pgfusepath{stroke,fill}%
}%
\begin{pgfscope}%
\pgfsys@transformshift{2.250000in}{0.450000in}%
\pgfsys@useobject{currentmarker}{}%
\end{pgfscope}%
\end{pgfscope}%
\begin{pgfscope}%
\pgfpathrectangle{\pgfqpoint{0.000000in}{0.000000in}}{\pgfqpoint{6.000000in}{3.600000in}}%
\pgfusepath{clip}%
\pgfsetbuttcap%
\pgfsetroundjoin%
\definecolor{currentfill}{rgb}{1.000000,1.000000,1.000000}%
\pgfsetfillcolor{currentfill}%
\pgfsetlinewidth{1.003750pt}%
\definecolor{currentstroke}{rgb}{0.000000,0.000000,0.000000}%
\pgfsetstrokecolor{currentstroke}%
\pgfsetdash{}{0pt}%
\pgfsys@defobject{currentmarker}{\pgfqpoint{-0.104167in}{-0.104167in}}{\pgfqpoint{0.104167in}{0.104167in}}{%
\pgfpathmoveto{\pgfqpoint{0.000000in}{-0.104167in}}%
\pgfpathcurveto{\pgfqpoint{0.027625in}{-0.104167in}}{\pgfqpoint{0.054123in}{-0.093191in}}{\pgfqpoint{0.073657in}{-0.073657in}}%
\pgfpathcurveto{\pgfqpoint{0.093191in}{-0.054123in}}{\pgfqpoint{0.104167in}{-0.027625in}}{\pgfqpoint{0.104167in}{0.000000in}}%
\pgfpathcurveto{\pgfqpoint{0.104167in}{0.027625in}}{\pgfqpoint{0.093191in}{0.054123in}}{\pgfqpoint{0.073657in}{0.073657in}}%
\pgfpathcurveto{\pgfqpoint{0.054123in}{0.093191in}}{\pgfqpoint{0.027625in}{0.104167in}}{\pgfqpoint{0.000000in}{0.104167in}}%
\pgfpathcurveto{\pgfqpoint{-0.027625in}{0.104167in}}{\pgfqpoint{-0.054123in}{0.093191in}}{\pgfqpoint{-0.073657in}{0.073657in}}%
\pgfpathcurveto{\pgfqpoint{-0.093191in}{0.054123in}}{\pgfqpoint{-0.104167in}{0.027625in}}{\pgfqpoint{-0.104167in}{0.000000in}}%
\pgfpathcurveto{\pgfqpoint{-0.104167in}{-0.027625in}}{\pgfqpoint{-0.093191in}{-0.054123in}}{\pgfqpoint{-0.073657in}{-0.073657in}}%
\pgfpathcurveto{\pgfqpoint{-0.054123in}{-0.093191in}}{\pgfqpoint{-0.027625in}{-0.104167in}}{\pgfqpoint{0.000000in}{-0.104167in}}%
\pgfpathclose%
\pgfusepath{stroke,fill}%
}%
\begin{pgfscope}%
\pgfsys@transformshift{3.750000in}{0.450000in}%
\pgfsys@useobject{currentmarker}{}%
\end{pgfscope}%
\end{pgfscope}%
\begin{pgfscope}%
\definecolor{textcolor}{rgb}{0.000000,0.000000,0.000000}%
\pgfsetstrokecolor{textcolor}%
\pgfsetfillcolor{textcolor}%
\pgftext[x=2.529863in,y=0.170137in,,base]{\color{textcolor}\sffamily\fontsize{40.000000}{48.000000}\selectfont \(\displaystyle w_1\)}%
\end{pgfscope}%
\begin{pgfscope}%
\definecolor{textcolor}{rgb}{0.000000,0.000000,0.000000}%
\pgfsetstrokecolor{textcolor}%
\pgfsetfillcolor{textcolor}%
\pgftext[x=4.029863in,y=0.170137in,,base]{\color{textcolor}\sffamily\fontsize{40.000000}{48.000000}\selectfont \(\displaystyle z_1\)}%
\end{pgfscope}%
\end{pgfpicture}%
\makeatother%
\endgroup%

%% file: fig/decomposition_G_U2.pgf
%% Creator: Matplotlib, PGF backend
%%
%% To include the figure in your LaTeX document, write
%%   \input{<filename>.pgf}
%%
%% Make sure the required packages are loaded in your preamble
%%   \usepackage{pgf}
%%
%% Figures using additional raster images can only be included by \input if
%% they are in the same directory as the main LaTeX file. For loading figures
%% from other directories you can use the `import` package
%%   \usepackage{import}
%%
%% and then include the figures with
%%   \import{<path to file>}{<filename>.pgf}
%%
%% Matplotlib used the following preamble
%%   \usepackage{fontspec}
%%   \setmainfont{DejaVuSerif.ttf}[Path=\detokenize{C:/Users/ccros/Anaconda3/Lib/site-packages/matplotlib/mpl-data/fonts/ttf/}]
%%   \setsansfont{DejaVuSans.ttf}[Path=\detokenize{C:/Users/ccros/Anaconda3/Lib/site-packages/matplotlib/mpl-data/fonts/ttf/}]
%%   \setmonofont{DejaVuSansMono.ttf}[Path=\detokenize{C:/Users/ccros/Anaconda3/Lib/site-packages/matplotlib/mpl-data/fonts/ttf/}]
%%
\begingroup%
\makeatletter%
\begin{pgfpicture}%
\pgfpathrectangle{\pgfpointorigin}{\pgfqpoint{6.000000in}{3.746955in}}%
\pgfusepath{use as bounding box, clip}%
\begin{pgfscope}%
\pgfsetbuttcap%
\pgfsetmiterjoin%
\definecolor{currentfill}{rgb}{1.000000,1.000000,1.000000}%
\pgfsetfillcolor{currentfill}%
\pgfsetlinewidth{0.000000pt}%
\definecolor{currentstroke}{rgb}{1.000000,1.000000,1.000000}%
\pgfsetstrokecolor{currentstroke}%
\pgfsetdash{}{0pt}%
\pgfpathmoveto{\pgfqpoint{0.000000in}{0.000000in}}%
\pgfpathlineto{\pgfqpoint{6.000000in}{0.000000in}}%
\pgfpathlineto{\pgfqpoint{6.000000in}{3.746955in}}%
\pgfpathlineto{\pgfqpoint{0.000000in}{3.746955in}}%
\pgfpathclose%
\pgfusepath{fill}%
\end{pgfscope}%
\begin{pgfscope}%
\pgfpathrectangle{\pgfqpoint{0.000000in}{0.000000in}}{\pgfqpoint{6.000000in}{3.600000in}}%
\pgfusepath{clip}%
\pgfsetbuttcap%
\pgfsetroundjoin%
\pgfsetlinewidth{3.011250pt}%
\definecolor{currentstroke}{rgb}{0.400000,0.400000,0.400000}%
\pgfsetstrokecolor{currentstroke}%
\pgfsetdash{{11.100000pt}{4.800000pt}}{0.000000pt}%
\pgfpathmoveto{\pgfqpoint{2.250000in}{3.150000in}}%
\pgfpathlineto{\pgfqpoint{0.750000in}{2.250000in}}%
\pgfusepath{stroke}%
\end{pgfscope}%
\begin{pgfscope}%
\pgfpathrectangle{\pgfqpoint{0.000000in}{0.000000in}}{\pgfqpoint{6.000000in}{3.600000in}}%
\pgfusepath{clip}%
\pgfsetbuttcap%
\pgfsetroundjoin%
\pgfsetlinewidth{3.011250pt}%
\definecolor{currentstroke}{rgb}{0.400000,0.400000,0.400000}%
\pgfsetstrokecolor{currentstroke}%
\pgfsetdash{{11.100000pt}{4.800000pt}}{0.000000pt}%
\pgfpathmoveto{\pgfqpoint{0.750000in}{2.250000in}}%
\pgfpathlineto{\pgfqpoint{2.250000in}{1.800000in}}%
\pgfusepath{stroke}%
\end{pgfscope}%
\begin{pgfscope}%
\pgfpathrectangle{\pgfqpoint{0.000000in}{0.000000in}}{\pgfqpoint{6.000000in}{3.600000in}}%
\pgfusepath{clip}%
\pgfsetbuttcap%
\pgfsetroundjoin%
\pgfsetlinewidth{3.011250pt}%
\definecolor{currentstroke}{rgb}{0.400000,0.400000,0.400000}%
\pgfsetstrokecolor{currentstroke}%
\pgfsetdash{{11.100000pt}{4.800000pt}}{0.000000pt}%
\pgfpathmoveto{\pgfqpoint{0.750000in}{2.250000in}}%
\pgfpathlineto{\pgfqpoint{0.750000in}{1.350000in}}%
\pgfusepath{stroke}%
\end{pgfscope}%
\begin{pgfscope}%
\pgfpathrectangle{\pgfqpoint{0.000000in}{0.000000in}}{\pgfqpoint{6.000000in}{3.600000in}}%
\pgfusepath{clip}%
\pgfsetbuttcap%
\pgfsetroundjoin%
\pgfsetlinewidth{3.011250pt}%
\definecolor{currentstroke}{rgb}{0.400000,0.400000,0.400000}%
\pgfsetstrokecolor{currentstroke}%
\pgfsetdash{{11.100000pt}{4.800000pt}}{0.000000pt}%
\pgfpathmoveto{\pgfqpoint{0.750000in}{1.350000in}}%
\pgfpathlineto{\pgfqpoint{3.750000in}{0.450000in}}%
\pgfusepath{stroke}%
\end{pgfscope}%
\begin{pgfscope}%
\pgfpathrectangle{\pgfqpoint{0.000000in}{0.000000in}}{\pgfqpoint{6.000000in}{3.600000in}}%
\pgfusepath{clip}%
\pgfsetbuttcap%
\pgfsetroundjoin%
\pgfsetlinewidth{3.011250pt}%
\definecolor{currentstroke}{rgb}{0.400000,0.400000,0.400000}%
\pgfsetstrokecolor{currentstroke}%
\pgfsetdash{{11.100000pt}{4.800000pt}}{0.000000pt}%
\pgfpathmoveto{\pgfqpoint{5.250000in}{2.250000in}}%
\pgfpathlineto{\pgfqpoint{3.750000in}{1.800000in}}%
\pgfusepath{stroke}%
\end{pgfscope}%
\begin{pgfscope}%
\pgfpathrectangle{\pgfqpoint{0.000000in}{0.000000in}}{\pgfqpoint{6.000000in}{3.600000in}}%
\pgfusepath{clip}%
\pgfsetbuttcap%
\pgfsetroundjoin%
\pgfsetlinewidth{3.011250pt}%
\definecolor{currentstroke}{rgb}{0.400000,0.400000,0.400000}%
\pgfsetstrokecolor{currentstroke}%
\pgfsetdash{{11.100000pt}{4.800000pt}}{0.000000pt}%
\pgfpathmoveto{\pgfqpoint{5.250000in}{2.250000in}}%
\pgfpathlineto{\pgfqpoint{5.250000in}{1.350000in}}%
\pgfusepath{stroke}%
\end{pgfscope}%
\begin{pgfscope}%
\pgfpathrectangle{\pgfqpoint{0.000000in}{0.000000in}}{\pgfqpoint{6.000000in}{3.600000in}}%
\pgfusepath{clip}%
\pgfsetbuttcap%
\pgfsetroundjoin%
\pgfsetlinewidth{3.011250pt}%
\definecolor{currentstroke}{rgb}{0.400000,0.400000,0.400000}%
\pgfsetstrokecolor{currentstroke}%
\pgfsetdash{{11.100000pt}{4.800000pt}}{0.000000pt}%
\pgfpathmoveto{\pgfqpoint{2.250000in}{1.800000in}}%
\pgfpathlineto{\pgfqpoint{0.750000in}{1.350000in}}%
\pgfusepath{stroke}%
\end{pgfscope}%
\begin{pgfscope}%
\pgfpathrectangle{\pgfqpoint{0.000000in}{0.000000in}}{\pgfqpoint{6.000000in}{3.600000in}}%
\pgfusepath{clip}%
\pgfsetbuttcap%
\pgfsetroundjoin%
\pgfsetlinewidth{3.011250pt}%
\definecolor{currentstroke}{rgb}{0.400000,0.400000,0.400000}%
\pgfsetstrokecolor{currentstroke}%
\pgfsetdash{{11.100000pt}{4.800000pt}}{0.000000pt}%
\pgfpathmoveto{\pgfqpoint{2.250000in}{1.800000in}}%
\pgfpathlineto{\pgfqpoint{5.250000in}{1.350000in}}%
\pgfusepath{stroke}%
\end{pgfscope}%
\begin{pgfscope}%
\pgfpathrectangle{\pgfqpoint{0.000000in}{0.000000in}}{\pgfqpoint{6.000000in}{3.600000in}}%
\pgfusepath{clip}%
\pgfsetbuttcap%
\pgfsetroundjoin%
\pgfsetlinewidth{3.011250pt}%
\definecolor{currentstroke}{rgb}{0.400000,0.400000,0.400000}%
\pgfsetstrokecolor{currentstroke}%
\pgfsetdash{{11.100000pt}{4.800000pt}}{0.000000pt}%
\pgfpathmoveto{\pgfqpoint{2.250000in}{1.800000in}}%
\pgfpathlineto{\pgfqpoint{2.250000in}{0.450000in}}%
\pgfusepath{stroke}%
\end{pgfscope}%
\begin{pgfscope}%
\pgfpathrectangle{\pgfqpoint{0.000000in}{0.000000in}}{\pgfqpoint{6.000000in}{3.600000in}}%
\pgfusepath{clip}%
\pgfsetbuttcap%
\pgfsetroundjoin%
\pgfsetlinewidth{3.011250pt}%
\definecolor{currentstroke}{rgb}{0.400000,0.400000,0.400000}%
\pgfsetstrokecolor{currentstroke}%
\pgfsetdash{{11.100000pt}{4.800000pt}}{0.000000pt}%
\pgfpathmoveto{\pgfqpoint{2.250000in}{1.800000in}}%
\pgfpathlineto{\pgfqpoint{3.750000in}{0.450000in}}%
\pgfusepath{stroke}%
\end{pgfscope}%
\begin{pgfscope}%
\pgfpathrectangle{\pgfqpoint{0.000000in}{0.000000in}}{\pgfqpoint{6.000000in}{3.600000in}}%
\pgfusepath{clip}%
\pgfsetbuttcap%
\pgfsetroundjoin%
\pgfsetlinewidth{3.011250pt}%
\definecolor{currentstroke}{rgb}{0.400000,0.400000,0.400000}%
\pgfsetstrokecolor{currentstroke}%
\pgfsetdash{{11.100000pt}{4.800000pt}}{0.000000pt}%
\pgfpathmoveto{\pgfqpoint{3.750000in}{1.800000in}}%
\pgfpathlineto{\pgfqpoint{5.250000in}{1.350000in}}%
\pgfusepath{stroke}%
\end{pgfscope}%
\begin{pgfscope}%
\pgfpathrectangle{\pgfqpoint{0.000000in}{0.000000in}}{\pgfqpoint{6.000000in}{3.600000in}}%
\pgfusepath{clip}%
\pgfsetbuttcap%
\pgfsetroundjoin%
\pgfsetlinewidth{3.011250pt}%
\definecolor{currentstroke}{rgb}{0.400000,0.400000,0.400000}%
\pgfsetstrokecolor{currentstroke}%
\pgfsetdash{{11.100000pt}{4.800000pt}}{0.000000pt}%
\pgfpathmoveto{\pgfqpoint{2.250000in}{0.450000in}}%
\pgfpathlineto{\pgfqpoint{3.750000in}{0.450000in}}%
\pgfusepath{stroke}%
\end{pgfscope}%
\begin{pgfscope}%
\pgfpathrectangle{\pgfqpoint{0.000000in}{0.000000in}}{\pgfqpoint{6.000000in}{3.600000in}}%
\pgfusepath{clip}%
\pgfsetrectcap%
\pgfsetroundjoin%
\pgfsetlinewidth{3.011250pt}%
\definecolor{currentstroke}{rgb}{0.250980,0.250980,0.250980}%
\pgfsetstrokecolor{currentstroke}%
\pgfsetdash{}{0pt}%
\pgfpathmoveto{\pgfqpoint{3.750000in}{3.150000in}}%
\pgfpathlineto{\pgfqpoint{2.250000in}{3.150000in}}%
\pgfusepath{stroke}%
\end{pgfscope}%
\begin{pgfscope}%
\pgfpathrectangle{\pgfqpoint{0.000000in}{0.000000in}}{\pgfqpoint{6.000000in}{3.600000in}}%
\pgfusepath{clip}%
\pgfsetrectcap%
\pgfsetroundjoin%
\pgfsetlinewidth{3.011250pt}%
\definecolor{currentstroke}{rgb}{0.250980,0.250980,0.250980}%
\pgfsetstrokecolor{currentstroke}%
\pgfsetdash{}{0pt}%
\pgfpathmoveto{\pgfqpoint{5.250000in}{2.250000in}}%
\pgfpathlineto{\pgfqpoint{2.250000in}{3.150000in}}%
\pgfusepath{stroke}%
\end{pgfscope}%
\begin{pgfscope}%
\pgfpathrectangle{\pgfqpoint{0.000000in}{0.000000in}}{\pgfqpoint{6.000000in}{3.600000in}}%
\pgfusepath{clip}%
\pgfsetrectcap%
\pgfsetroundjoin%
\pgfsetlinewidth{3.011250pt}%
\definecolor{currentstroke}{rgb}{0.250980,0.250980,0.250980}%
\pgfsetstrokecolor{currentstroke}%
\pgfsetdash{}{0pt}%
\pgfpathmoveto{\pgfqpoint{5.250000in}{2.250000in}}%
\pgfpathlineto{\pgfqpoint{3.750000in}{3.150000in}}%
\pgfusepath{stroke}%
\end{pgfscope}%
\begin{pgfscope}%
\pgfpathrectangle{\pgfqpoint{0.000000in}{0.000000in}}{\pgfqpoint{6.000000in}{3.600000in}}%
\pgfusepath{clip}%
\pgfsetrectcap%
\pgfsetroundjoin%
\pgfsetlinewidth{3.011250pt}%
\definecolor{currentstroke}{rgb}{0.250980,0.250980,0.250980}%
\pgfsetstrokecolor{currentstroke}%
\pgfsetdash{}{0pt}%
\pgfpathmoveto{\pgfqpoint{2.250000in}{1.800000in}}%
\pgfpathlineto{\pgfqpoint{3.750000in}{3.150000in}}%
\pgfusepath{stroke}%
\end{pgfscope}%
\begin{pgfscope}%
\pgfpathrectangle{\pgfqpoint{0.000000in}{0.000000in}}{\pgfqpoint{6.000000in}{3.600000in}}%
\pgfusepath{clip}%
\pgfsetrectcap%
\pgfsetroundjoin%
\pgfsetlinewidth{3.011250pt}%
\definecolor{currentstroke}{rgb}{0.250980,0.250980,0.250980}%
\pgfsetstrokecolor{currentstroke}%
\pgfsetdash{}{0pt}%
\pgfpathmoveto{\pgfqpoint{2.250000in}{1.800000in}}%
\pgfpathlineto{\pgfqpoint{5.250000in}{2.250000in}}%
\pgfusepath{stroke}%
\end{pgfscope}%
\begin{pgfscope}%
\pgfpathrectangle{\pgfqpoint{0.000000in}{0.000000in}}{\pgfqpoint{6.000000in}{3.600000in}}%
\pgfusepath{clip}%
\pgfsetbuttcap%
\pgfsetroundjoin%
\definecolor{currentfill}{rgb}{0.800000,0.800000,0.800000}%
\pgfsetfillcolor{currentfill}%
\pgfsetlinewidth{1.003750pt}%
\definecolor{currentstroke}{rgb}{0.000000,0.000000,0.000000}%
\pgfsetstrokecolor{currentstroke}%
\pgfsetdash{}{0pt}%
\pgfsys@defobject{currentmarker}{\pgfqpoint{-0.104167in}{-0.104167in}}{\pgfqpoint{0.104167in}{0.104167in}}{%
\pgfpathmoveto{\pgfqpoint{0.000000in}{-0.104167in}}%
\pgfpathcurveto{\pgfqpoint{0.027625in}{-0.104167in}}{\pgfqpoint{0.054123in}{-0.093191in}}{\pgfqpoint{0.073657in}{-0.073657in}}%
\pgfpathcurveto{\pgfqpoint{0.093191in}{-0.054123in}}{\pgfqpoint{0.104167in}{-0.027625in}}{\pgfqpoint{0.104167in}{0.000000in}}%
\pgfpathcurveto{\pgfqpoint{0.104167in}{0.027625in}}{\pgfqpoint{0.093191in}{0.054123in}}{\pgfqpoint{0.073657in}{0.073657in}}%
\pgfpathcurveto{\pgfqpoint{0.054123in}{0.093191in}}{\pgfqpoint{0.027625in}{0.104167in}}{\pgfqpoint{0.000000in}{0.104167in}}%
\pgfpathcurveto{\pgfqpoint{-0.027625in}{0.104167in}}{\pgfqpoint{-0.054123in}{0.093191in}}{\pgfqpoint{-0.073657in}{0.073657in}}%
\pgfpathcurveto{\pgfqpoint{-0.093191in}{0.054123in}}{\pgfqpoint{-0.104167in}{0.027625in}}{\pgfqpoint{-0.104167in}{0.000000in}}%
\pgfpathcurveto{\pgfqpoint{-0.104167in}{-0.027625in}}{\pgfqpoint{-0.093191in}{-0.054123in}}{\pgfqpoint{-0.073657in}{-0.073657in}}%
\pgfpathcurveto{\pgfqpoint{-0.054123in}{-0.093191in}}{\pgfqpoint{-0.027625in}{-0.104167in}}{\pgfqpoint{0.000000in}{-0.104167in}}%
\pgfpathclose%
\pgfusepath{stroke,fill}%
}%
\begin{pgfscope}%
\pgfsys@transformshift{2.250000in}{3.150000in}%
\pgfsys@useobject{currentmarker}{}%
\end{pgfscope}%
\end{pgfscope}%
\begin{pgfscope}%
\pgfpathrectangle{\pgfqpoint{0.000000in}{0.000000in}}{\pgfqpoint{6.000000in}{3.600000in}}%
\pgfusepath{clip}%
\pgfsetbuttcap%
\pgfsetroundjoin%
\definecolor{currentfill}{rgb}{0.800000,0.800000,0.800000}%
\pgfsetfillcolor{currentfill}%
\pgfsetlinewidth{1.003750pt}%
\definecolor{currentstroke}{rgb}{0.000000,0.000000,0.000000}%
\pgfsetstrokecolor{currentstroke}%
\pgfsetdash{}{0pt}%
\pgfsys@defobject{currentmarker}{\pgfqpoint{-0.104167in}{-0.104167in}}{\pgfqpoint{0.104167in}{0.104167in}}{%
\pgfpathmoveto{\pgfqpoint{0.000000in}{-0.104167in}}%
\pgfpathcurveto{\pgfqpoint{0.027625in}{-0.104167in}}{\pgfqpoint{0.054123in}{-0.093191in}}{\pgfqpoint{0.073657in}{-0.073657in}}%
\pgfpathcurveto{\pgfqpoint{0.093191in}{-0.054123in}}{\pgfqpoint{0.104167in}{-0.027625in}}{\pgfqpoint{0.104167in}{0.000000in}}%
\pgfpathcurveto{\pgfqpoint{0.104167in}{0.027625in}}{\pgfqpoint{0.093191in}{0.054123in}}{\pgfqpoint{0.073657in}{0.073657in}}%
\pgfpathcurveto{\pgfqpoint{0.054123in}{0.093191in}}{\pgfqpoint{0.027625in}{0.104167in}}{\pgfqpoint{0.000000in}{0.104167in}}%
\pgfpathcurveto{\pgfqpoint{-0.027625in}{0.104167in}}{\pgfqpoint{-0.054123in}{0.093191in}}{\pgfqpoint{-0.073657in}{0.073657in}}%
\pgfpathcurveto{\pgfqpoint{-0.093191in}{0.054123in}}{\pgfqpoint{-0.104167in}{0.027625in}}{\pgfqpoint{-0.104167in}{0.000000in}}%
\pgfpathcurveto{\pgfqpoint{-0.104167in}{-0.027625in}}{\pgfqpoint{-0.093191in}{-0.054123in}}{\pgfqpoint{-0.073657in}{-0.073657in}}%
\pgfpathcurveto{\pgfqpoint{-0.054123in}{-0.093191in}}{\pgfqpoint{-0.027625in}{-0.104167in}}{\pgfqpoint{0.000000in}{-0.104167in}}%
\pgfpathclose%
\pgfusepath{stroke,fill}%
}%
\begin{pgfscope}%
\pgfsys@transformshift{3.750000in}{3.150000in}%
\pgfsys@useobject{currentmarker}{}%
\end{pgfscope}%
\end{pgfscope}%
\begin{pgfscope}%
\pgfpathrectangle{\pgfqpoint{0.000000in}{0.000000in}}{\pgfqpoint{6.000000in}{3.600000in}}%
\pgfusepath{clip}%
\pgfsetbuttcap%
\pgfsetroundjoin%
\definecolor{currentfill}{rgb}{0.800000,0.800000,0.800000}%
\pgfsetfillcolor{currentfill}%
\pgfsetlinewidth{1.003750pt}%
\definecolor{currentstroke}{rgb}{0.000000,0.000000,0.000000}%
\pgfsetstrokecolor{currentstroke}%
\pgfsetdash{}{0pt}%
\pgfsys@defobject{currentmarker}{\pgfqpoint{-0.104167in}{-0.104167in}}{\pgfqpoint{0.104167in}{0.104167in}}{%
\pgfpathmoveto{\pgfqpoint{0.000000in}{-0.104167in}}%
\pgfpathcurveto{\pgfqpoint{0.027625in}{-0.104167in}}{\pgfqpoint{0.054123in}{-0.093191in}}{\pgfqpoint{0.073657in}{-0.073657in}}%
\pgfpathcurveto{\pgfqpoint{0.093191in}{-0.054123in}}{\pgfqpoint{0.104167in}{-0.027625in}}{\pgfqpoint{0.104167in}{0.000000in}}%
\pgfpathcurveto{\pgfqpoint{0.104167in}{0.027625in}}{\pgfqpoint{0.093191in}{0.054123in}}{\pgfqpoint{0.073657in}{0.073657in}}%
\pgfpathcurveto{\pgfqpoint{0.054123in}{0.093191in}}{\pgfqpoint{0.027625in}{0.104167in}}{\pgfqpoint{0.000000in}{0.104167in}}%
\pgfpathcurveto{\pgfqpoint{-0.027625in}{0.104167in}}{\pgfqpoint{-0.054123in}{0.093191in}}{\pgfqpoint{-0.073657in}{0.073657in}}%
\pgfpathcurveto{\pgfqpoint{-0.093191in}{0.054123in}}{\pgfqpoint{-0.104167in}{0.027625in}}{\pgfqpoint{-0.104167in}{0.000000in}}%
\pgfpathcurveto{\pgfqpoint{-0.104167in}{-0.027625in}}{\pgfqpoint{-0.093191in}{-0.054123in}}{\pgfqpoint{-0.073657in}{-0.073657in}}%
\pgfpathcurveto{\pgfqpoint{-0.054123in}{-0.093191in}}{\pgfqpoint{-0.027625in}{-0.104167in}}{\pgfqpoint{0.000000in}{-0.104167in}}%
\pgfpathclose%
\pgfusepath{stroke,fill}%
}%
\begin{pgfscope}%
\pgfsys@transformshift{0.750000in}{2.250000in}%
\pgfsys@useobject{currentmarker}{}%
\end{pgfscope}%
\end{pgfscope}%
\begin{pgfscope}%
\pgfpathrectangle{\pgfqpoint{0.000000in}{0.000000in}}{\pgfqpoint{6.000000in}{3.600000in}}%
\pgfusepath{clip}%
\pgfsetbuttcap%
\pgfsetroundjoin%
\definecolor{currentfill}{rgb}{1.000000,1.000000,1.000000}%
\pgfsetfillcolor{currentfill}%
\pgfsetlinewidth{1.003750pt}%
\definecolor{currentstroke}{rgb}{0.000000,0.000000,0.000000}%
\pgfsetstrokecolor{currentstroke}%
\pgfsetdash{}{0pt}%
\pgfsys@defobject{currentmarker}{\pgfqpoint{-0.104167in}{-0.104167in}}{\pgfqpoint{0.104167in}{0.104167in}}{%
\pgfpathmoveto{\pgfqpoint{0.000000in}{-0.104167in}}%
\pgfpathcurveto{\pgfqpoint{0.027625in}{-0.104167in}}{\pgfqpoint{0.054123in}{-0.093191in}}{\pgfqpoint{0.073657in}{-0.073657in}}%
\pgfpathcurveto{\pgfqpoint{0.093191in}{-0.054123in}}{\pgfqpoint{0.104167in}{-0.027625in}}{\pgfqpoint{0.104167in}{0.000000in}}%
\pgfpathcurveto{\pgfqpoint{0.104167in}{0.027625in}}{\pgfqpoint{0.093191in}{0.054123in}}{\pgfqpoint{0.073657in}{0.073657in}}%
\pgfpathcurveto{\pgfqpoint{0.054123in}{0.093191in}}{\pgfqpoint{0.027625in}{0.104167in}}{\pgfqpoint{0.000000in}{0.104167in}}%
\pgfpathcurveto{\pgfqpoint{-0.027625in}{0.104167in}}{\pgfqpoint{-0.054123in}{0.093191in}}{\pgfqpoint{-0.073657in}{0.073657in}}%
\pgfpathcurveto{\pgfqpoint{-0.093191in}{0.054123in}}{\pgfqpoint{-0.104167in}{0.027625in}}{\pgfqpoint{-0.104167in}{0.000000in}}%
\pgfpathcurveto{\pgfqpoint{-0.104167in}{-0.027625in}}{\pgfqpoint{-0.093191in}{-0.054123in}}{\pgfqpoint{-0.073657in}{-0.073657in}}%
\pgfpathcurveto{\pgfqpoint{-0.054123in}{-0.093191in}}{\pgfqpoint{-0.027625in}{-0.104167in}}{\pgfqpoint{0.000000in}{-0.104167in}}%
\pgfpathclose%
\pgfusepath{stroke,fill}%
}%
\begin{pgfscope}%
\pgfsys@transformshift{5.250000in}{2.250000in}%
\pgfsys@useobject{currentmarker}{}%
\end{pgfscope}%
\end{pgfscope}%
\begin{pgfscope}%
\pgfpathrectangle{\pgfqpoint{0.000000in}{0.000000in}}{\pgfqpoint{6.000000in}{3.600000in}}%
\pgfusepath{clip}%
\pgfsetbuttcap%
\pgfsetroundjoin%
\definecolor{currentfill}{rgb}{0.800000,0.800000,0.800000}%
\pgfsetfillcolor{currentfill}%
\pgfsetlinewidth{1.003750pt}%
\definecolor{currentstroke}{rgb}{0.000000,0.000000,0.000000}%
\pgfsetstrokecolor{currentstroke}%
\pgfsetdash{}{0pt}%
\pgfsys@defobject{currentmarker}{\pgfqpoint{-0.104167in}{-0.104167in}}{\pgfqpoint{0.104167in}{0.104167in}}{%
\pgfpathmoveto{\pgfqpoint{0.000000in}{-0.104167in}}%
\pgfpathcurveto{\pgfqpoint{0.027625in}{-0.104167in}}{\pgfqpoint{0.054123in}{-0.093191in}}{\pgfqpoint{0.073657in}{-0.073657in}}%
\pgfpathcurveto{\pgfqpoint{0.093191in}{-0.054123in}}{\pgfqpoint{0.104167in}{-0.027625in}}{\pgfqpoint{0.104167in}{0.000000in}}%
\pgfpathcurveto{\pgfqpoint{0.104167in}{0.027625in}}{\pgfqpoint{0.093191in}{0.054123in}}{\pgfqpoint{0.073657in}{0.073657in}}%
\pgfpathcurveto{\pgfqpoint{0.054123in}{0.093191in}}{\pgfqpoint{0.027625in}{0.104167in}}{\pgfqpoint{0.000000in}{0.104167in}}%
\pgfpathcurveto{\pgfqpoint{-0.027625in}{0.104167in}}{\pgfqpoint{-0.054123in}{0.093191in}}{\pgfqpoint{-0.073657in}{0.073657in}}%
\pgfpathcurveto{\pgfqpoint{-0.093191in}{0.054123in}}{\pgfqpoint{-0.104167in}{0.027625in}}{\pgfqpoint{-0.104167in}{0.000000in}}%
\pgfpathcurveto{\pgfqpoint{-0.104167in}{-0.027625in}}{\pgfqpoint{-0.093191in}{-0.054123in}}{\pgfqpoint{-0.073657in}{-0.073657in}}%
\pgfpathcurveto{\pgfqpoint{-0.054123in}{-0.093191in}}{\pgfqpoint{-0.027625in}{-0.104167in}}{\pgfqpoint{0.000000in}{-0.104167in}}%
\pgfpathclose%
\pgfusepath{stroke,fill}%
}%
\begin{pgfscope}%
\pgfsys@transformshift{2.250000in}{1.800000in}%
\pgfsys@useobject{currentmarker}{}%
\end{pgfscope}%
\end{pgfscope}%
\begin{pgfscope}%
\pgfpathrectangle{\pgfqpoint{0.000000in}{0.000000in}}{\pgfqpoint{6.000000in}{3.600000in}}%
\pgfusepath{clip}%
\pgfsetbuttcap%
\pgfsetroundjoin%
\definecolor{currentfill}{rgb}{1.000000,1.000000,1.000000}%
\pgfsetfillcolor{currentfill}%
\pgfsetlinewidth{1.003750pt}%
\definecolor{currentstroke}{rgb}{0.000000,0.000000,0.000000}%
\pgfsetstrokecolor{currentstroke}%
\pgfsetdash{}{0pt}%
\pgfsys@defobject{currentmarker}{\pgfqpoint{-0.104167in}{-0.104167in}}{\pgfqpoint{0.104167in}{0.104167in}}{%
\pgfpathmoveto{\pgfqpoint{0.000000in}{-0.104167in}}%
\pgfpathcurveto{\pgfqpoint{0.027625in}{-0.104167in}}{\pgfqpoint{0.054123in}{-0.093191in}}{\pgfqpoint{0.073657in}{-0.073657in}}%
\pgfpathcurveto{\pgfqpoint{0.093191in}{-0.054123in}}{\pgfqpoint{0.104167in}{-0.027625in}}{\pgfqpoint{0.104167in}{0.000000in}}%
\pgfpathcurveto{\pgfqpoint{0.104167in}{0.027625in}}{\pgfqpoint{0.093191in}{0.054123in}}{\pgfqpoint{0.073657in}{0.073657in}}%
\pgfpathcurveto{\pgfqpoint{0.054123in}{0.093191in}}{\pgfqpoint{0.027625in}{0.104167in}}{\pgfqpoint{0.000000in}{0.104167in}}%
\pgfpathcurveto{\pgfqpoint{-0.027625in}{0.104167in}}{\pgfqpoint{-0.054123in}{0.093191in}}{\pgfqpoint{-0.073657in}{0.073657in}}%
\pgfpathcurveto{\pgfqpoint{-0.093191in}{0.054123in}}{\pgfqpoint{-0.104167in}{0.027625in}}{\pgfqpoint{-0.104167in}{0.000000in}}%
\pgfpathcurveto{\pgfqpoint{-0.104167in}{-0.027625in}}{\pgfqpoint{-0.093191in}{-0.054123in}}{\pgfqpoint{-0.073657in}{-0.073657in}}%
\pgfpathcurveto{\pgfqpoint{-0.054123in}{-0.093191in}}{\pgfqpoint{-0.027625in}{-0.104167in}}{\pgfqpoint{0.000000in}{-0.104167in}}%
\pgfpathclose%
\pgfusepath{stroke,fill}%
}%
\begin{pgfscope}%
\pgfsys@transformshift{3.750000in}{1.800000in}%
\pgfsys@useobject{currentmarker}{}%
\end{pgfscope}%
\end{pgfscope}%
\begin{pgfscope}%
\pgfpathrectangle{\pgfqpoint{0.000000in}{0.000000in}}{\pgfqpoint{6.000000in}{3.600000in}}%
\pgfusepath{clip}%
\pgfsetbuttcap%
\pgfsetroundjoin%
\definecolor{currentfill}{rgb}{0.800000,0.800000,0.800000}%
\pgfsetfillcolor{currentfill}%
\pgfsetlinewidth{1.003750pt}%
\definecolor{currentstroke}{rgb}{0.000000,0.000000,0.000000}%
\pgfsetstrokecolor{currentstroke}%
\pgfsetdash{}{0pt}%
\pgfsys@defobject{currentmarker}{\pgfqpoint{-0.104167in}{-0.104167in}}{\pgfqpoint{0.104167in}{0.104167in}}{%
\pgfpathmoveto{\pgfqpoint{0.000000in}{-0.104167in}}%
\pgfpathcurveto{\pgfqpoint{0.027625in}{-0.104167in}}{\pgfqpoint{0.054123in}{-0.093191in}}{\pgfqpoint{0.073657in}{-0.073657in}}%
\pgfpathcurveto{\pgfqpoint{0.093191in}{-0.054123in}}{\pgfqpoint{0.104167in}{-0.027625in}}{\pgfqpoint{0.104167in}{0.000000in}}%
\pgfpathcurveto{\pgfqpoint{0.104167in}{0.027625in}}{\pgfqpoint{0.093191in}{0.054123in}}{\pgfqpoint{0.073657in}{0.073657in}}%
\pgfpathcurveto{\pgfqpoint{0.054123in}{0.093191in}}{\pgfqpoint{0.027625in}{0.104167in}}{\pgfqpoint{0.000000in}{0.104167in}}%
\pgfpathcurveto{\pgfqpoint{-0.027625in}{0.104167in}}{\pgfqpoint{-0.054123in}{0.093191in}}{\pgfqpoint{-0.073657in}{0.073657in}}%
\pgfpathcurveto{\pgfqpoint{-0.093191in}{0.054123in}}{\pgfqpoint{-0.104167in}{0.027625in}}{\pgfqpoint{-0.104167in}{0.000000in}}%
\pgfpathcurveto{\pgfqpoint{-0.104167in}{-0.027625in}}{\pgfqpoint{-0.093191in}{-0.054123in}}{\pgfqpoint{-0.073657in}{-0.073657in}}%
\pgfpathcurveto{\pgfqpoint{-0.054123in}{-0.093191in}}{\pgfqpoint{-0.027625in}{-0.104167in}}{\pgfqpoint{0.000000in}{-0.104167in}}%
\pgfpathclose%
\pgfusepath{stroke,fill}%
}%
\begin{pgfscope}%
\pgfsys@transformshift{0.750000in}{1.350000in}%
\pgfsys@useobject{currentmarker}{}%
\end{pgfscope}%
\end{pgfscope}%
\begin{pgfscope}%
\pgfpathrectangle{\pgfqpoint{0.000000in}{0.000000in}}{\pgfqpoint{6.000000in}{3.600000in}}%
\pgfusepath{clip}%
\pgfsetbuttcap%
\pgfsetroundjoin%
\definecolor{currentfill}{rgb}{1.000000,1.000000,1.000000}%
\pgfsetfillcolor{currentfill}%
\pgfsetlinewidth{1.003750pt}%
\definecolor{currentstroke}{rgb}{0.000000,0.000000,0.000000}%
\pgfsetstrokecolor{currentstroke}%
\pgfsetdash{}{0pt}%
\pgfsys@defobject{currentmarker}{\pgfqpoint{-0.104167in}{-0.104167in}}{\pgfqpoint{0.104167in}{0.104167in}}{%
\pgfpathmoveto{\pgfqpoint{0.000000in}{-0.104167in}}%
\pgfpathcurveto{\pgfqpoint{0.027625in}{-0.104167in}}{\pgfqpoint{0.054123in}{-0.093191in}}{\pgfqpoint{0.073657in}{-0.073657in}}%
\pgfpathcurveto{\pgfqpoint{0.093191in}{-0.054123in}}{\pgfqpoint{0.104167in}{-0.027625in}}{\pgfqpoint{0.104167in}{0.000000in}}%
\pgfpathcurveto{\pgfqpoint{0.104167in}{0.027625in}}{\pgfqpoint{0.093191in}{0.054123in}}{\pgfqpoint{0.073657in}{0.073657in}}%
\pgfpathcurveto{\pgfqpoint{0.054123in}{0.093191in}}{\pgfqpoint{0.027625in}{0.104167in}}{\pgfqpoint{0.000000in}{0.104167in}}%
\pgfpathcurveto{\pgfqpoint{-0.027625in}{0.104167in}}{\pgfqpoint{-0.054123in}{0.093191in}}{\pgfqpoint{-0.073657in}{0.073657in}}%
\pgfpathcurveto{\pgfqpoint{-0.093191in}{0.054123in}}{\pgfqpoint{-0.104167in}{0.027625in}}{\pgfqpoint{-0.104167in}{0.000000in}}%
\pgfpathcurveto{\pgfqpoint{-0.104167in}{-0.027625in}}{\pgfqpoint{-0.093191in}{-0.054123in}}{\pgfqpoint{-0.073657in}{-0.073657in}}%
\pgfpathcurveto{\pgfqpoint{-0.054123in}{-0.093191in}}{\pgfqpoint{-0.027625in}{-0.104167in}}{\pgfqpoint{0.000000in}{-0.104167in}}%
\pgfpathclose%
\pgfusepath{stroke,fill}%
}%
\begin{pgfscope}%
\pgfsys@transformshift{5.250000in}{1.350000in}%
\pgfsys@useobject{currentmarker}{}%
\end{pgfscope}%
\end{pgfscope}%
\begin{pgfscope}%
\pgfpathrectangle{\pgfqpoint{0.000000in}{0.000000in}}{\pgfqpoint{6.000000in}{3.600000in}}%
\pgfusepath{clip}%
\pgfsetbuttcap%
\pgfsetroundjoin%
\definecolor{currentfill}{rgb}{1.000000,1.000000,1.000000}%
\pgfsetfillcolor{currentfill}%
\pgfsetlinewidth{1.003750pt}%
\definecolor{currentstroke}{rgb}{0.000000,0.000000,0.000000}%
\pgfsetstrokecolor{currentstroke}%
\pgfsetdash{}{0pt}%
\pgfsys@defobject{currentmarker}{\pgfqpoint{-0.104167in}{-0.104167in}}{\pgfqpoint{0.104167in}{0.104167in}}{%
\pgfpathmoveto{\pgfqpoint{0.000000in}{-0.104167in}}%
\pgfpathcurveto{\pgfqpoint{0.027625in}{-0.104167in}}{\pgfqpoint{0.054123in}{-0.093191in}}{\pgfqpoint{0.073657in}{-0.073657in}}%
\pgfpathcurveto{\pgfqpoint{0.093191in}{-0.054123in}}{\pgfqpoint{0.104167in}{-0.027625in}}{\pgfqpoint{0.104167in}{0.000000in}}%
\pgfpathcurveto{\pgfqpoint{0.104167in}{0.027625in}}{\pgfqpoint{0.093191in}{0.054123in}}{\pgfqpoint{0.073657in}{0.073657in}}%
\pgfpathcurveto{\pgfqpoint{0.054123in}{0.093191in}}{\pgfqpoint{0.027625in}{0.104167in}}{\pgfqpoint{0.000000in}{0.104167in}}%
\pgfpathcurveto{\pgfqpoint{-0.027625in}{0.104167in}}{\pgfqpoint{-0.054123in}{0.093191in}}{\pgfqpoint{-0.073657in}{0.073657in}}%
\pgfpathcurveto{\pgfqpoint{-0.093191in}{0.054123in}}{\pgfqpoint{-0.104167in}{0.027625in}}{\pgfqpoint{-0.104167in}{0.000000in}}%
\pgfpathcurveto{\pgfqpoint{-0.104167in}{-0.027625in}}{\pgfqpoint{-0.093191in}{-0.054123in}}{\pgfqpoint{-0.073657in}{-0.073657in}}%
\pgfpathcurveto{\pgfqpoint{-0.054123in}{-0.093191in}}{\pgfqpoint{-0.027625in}{-0.104167in}}{\pgfqpoint{0.000000in}{-0.104167in}}%
\pgfpathclose%
\pgfusepath{stroke,fill}%
}%
\begin{pgfscope}%
\pgfsys@transformshift{2.250000in}{0.450000in}%
\pgfsys@useobject{currentmarker}{}%
\end{pgfscope}%
\end{pgfscope}%
\begin{pgfscope}%
\pgfpathrectangle{\pgfqpoint{0.000000in}{0.000000in}}{\pgfqpoint{6.000000in}{3.600000in}}%
\pgfusepath{clip}%
\pgfsetbuttcap%
\pgfsetroundjoin%
\definecolor{currentfill}{rgb}{1.000000,1.000000,1.000000}%
\pgfsetfillcolor{currentfill}%
\pgfsetlinewidth{1.003750pt}%
\definecolor{currentstroke}{rgb}{0.000000,0.000000,0.000000}%
\pgfsetstrokecolor{currentstroke}%
\pgfsetdash{}{0pt}%
\pgfsys@defobject{currentmarker}{\pgfqpoint{-0.104167in}{-0.104167in}}{\pgfqpoint{0.104167in}{0.104167in}}{%
\pgfpathmoveto{\pgfqpoint{0.000000in}{-0.104167in}}%
\pgfpathcurveto{\pgfqpoint{0.027625in}{-0.104167in}}{\pgfqpoint{0.054123in}{-0.093191in}}{\pgfqpoint{0.073657in}{-0.073657in}}%
\pgfpathcurveto{\pgfqpoint{0.093191in}{-0.054123in}}{\pgfqpoint{0.104167in}{-0.027625in}}{\pgfqpoint{0.104167in}{0.000000in}}%
\pgfpathcurveto{\pgfqpoint{0.104167in}{0.027625in}}{\pgfqpoint{0.093191in}{0.054123in}}{\pgfqpoint{0.073657in}{0.073657in}}%
\pgfpathcurveto{\pgfqpoint{0.054123in}{0.093191in}}{\pgfqpoint{0.027625in}{0.104167in}}{\pgfqpoint{0.000000in}{0.104167in}}%
\pgfpathcurveto{\pgfqpoint{-0.027625in}{0.104167in}}{\pgfqpoint{-0.054123in}{0.093191in}}{\pgfqpoint{-0.073657in}{0.073657in}}%
\pgfpathcurveto{\pgfqpoint{-0.093191in}{0.054123in}}{\pgfqpoint{-0.104167in}{0.027625in}}{\pgfqpoint{-0.104167in}{0.000000in}}%
\pgfpathcurveto{\pgfqpoint{-0.104167in}{-0.027625in}}{\pgfqpoint{-0.093191in}{-0.054123in}}{\pgfqpoint{-0.073657in}{-0.073657in}}%
\pgfpathcurveto{\pgfqpoint{-0.054123in}{-0.093191in}}{\pgfqpoint{-0.027625in}{-0.104167in}}{\pgfqpoint{0.000000in}{-0.104167in}}%
\pgfpathclose%
\pgfusepath{stroke,fill}%
}%
\begin{pgfscope}%
\pgfsys@transformshift{3.750000in}{0.450000in}%
\pgfsys@useobject{currentmarker}{}%
\end{pgfscope}%
\end{pgfscope}%
\begin{pgfscope}%
\definecolor{textcolor}{rgb}{0.000000,0.000000,0.000000}%
\pgfsetstrokecolor{textcolor}%
\pgfsetfillcolor{textcolor}%
\pgftext[x=3.924863in,y=3.324863in,,base]{\color{textcolor}\sffamily\fontsize{40.000000}{48.000000}\selectfont \(\displaystyle w_2\)}%
\end{pgfscope}%
\begin{pgfscope}%
\definecolor{textcolor}{rgb}{0.000000,0.000000,0.000000}%
\pgfsetstrokecolor{textcolor}%
\pgfsetfillcolor{textcolor}%
\pgftext[x=5.424863in,y=2.424863in,,base]{\color{textcolor}\sffamily\fontsize{40.000000}{48.000000}\selectfont \(\displaystyle z_2\)}%
\end{pgfscope}%
\end{pgfpicture}%
\makeatother%
\endgroup%

%% file: fig/decomposition_H0.pgf
%% Creator: Matplotlib, PGF backend
%%
%% To include the figure in your LaTeX document, write
%%   \input{<filename>.pgf}
%%
%% Make sure the required packages are loaded in your preamble
%%   \usepackage{pgf}
%%
%% Figures using additional raster images can only be included by \input if
%% they are in the same directory as the main LaTeX file. For loading figures
%% from other directories you can use the `import` package
%%   \usepackage{import}
%%
%% and then include the figures with
%%   \import{<path to file>}{<filename>.pgf}
%%
%% Matplotlib used the following preamble
%%   \usepackage{fontspec}
%%   \setmainfont{DejaVuSerif.ttf}[Path=\detokenize{C:/Users/ccros/Anaconda3/Lib/site-packages/matplotlib/mpl-data/fonts/ttf/}]
%%   \setsansfont{DejaVuSans.ttf}[Path=\detokenize{C:/Users/ccros/Anaconda3/Lib/site-packages/matplotlib/mpl-data/fonts/ttf/}]
%%   \setmonofont{DejaVuSansMono.ttf}[Path=\detokenize{C:/Users/ccros/Anaconda3/Lib/site-packages/matplotlib/mpl-data/fonts/ttf/}]
%%
\begingroup%
\makeatletter%
\begin{pgfpicture}%
\pgfpathrectangle{\pgfpointorigin}{\pgfqpoint{6.000000in}{3.600000in}}%
\pgfusepath{use as bounding box, clip}%
\begin{pgfscope}%
\pgfsetbuttcap%
\pgfsetmiterjoin%
\definecolor{currentfill}{rgb}{1.000000,1.000000,1.000000}%
\pgfsetfillcolor{currentfill}%
\pgfsetlinewidth{0.000000pt}%
\definecolor{currentstroke}{rgb}{1.000000,1.000000,1.000000}%
\pgfsetstrokecolor{currentstroke}%
\pgfsetdash{}{0pt}%
\pgfpathmoveto{\pgfqpoint{0.000000in}{0.000000in}}%
\pgfpathlineto{\pgfqpoint{6.000000in}{0.000000in}}%
\pgfpathlineto{\pgfqpoint{6.000000in}{3.600000in}}%
\pgfpathlineto{\pgfqpoint{0.000000in}{3.600000in}}%
\pgfpathclose%
\pgfusepath{fill}%
\end{pgfscope}%
\begin{pgfscope}%
\pgfpathrectangle{\pgfqpoint{0.000000in}{0.000000in}}{\pgfqpoint{6.000000in}{3.600000in}}%
\pgfusepath{clip}%
\pgfsetrectcap%
\pgfsetroundjoin%
\pgfsetlinewidth{3.011250pt}%
\definecolor{currentstroke}{rgb}{0.250980,0.250980,0.250980}%
\pgfsetstrokecolor{currentstroke}%
\pgfsetdash{}{0pt}%
\pgfpathmoveto{\pgfqpoint{3.750000in}{3.150000in}}%
\pgfpathlineto{\pgfqpoint{2.250000in}{3.150000in}}%
\pgfusepath{stroke}%
\end{pgfscope}%
\begin{pgfscope}%
\pgfpathrectangle{\pgfqpoint{0.000000in}{0.000000in}}{\pgfqpoint{6.000000in}{3.600000in}}%
\pgfusepath{clip}%
\pgfsetrectcap%
\pgfsetroundjoin%
\pgfsetlinewidth{3.011250pt}%
\definecolor{currentstroke}{rgb}{0.250980,0.250980,0.250980}%
\pgfsetstrokecolor{currentstroke}%
\pgfsetdash{}{0pt}%
\pgfpathmoveto{\pgfqpoint{0.750000in}{2.250000in}}%
\pgfpathlineto{\pgfqpoint{2.250000in}{3.150000in}}%
\pgfusepath{stroke}%
\end{pgfscope}%
\begin{pgfscope}%
\pgfpathrectangle{\pgfqpoint{0.000000in}{0.000000in}}{\pgfqpoint{6.000000in}{3.600000in}}%
\pgfusepath{clip}%
\pgfsetrectcap%
\pgfsetroundjoin%
\pgfsetlinewidth{3.011250pt}%
\definecolor{currentstroke}{rgb}{0.250980,0.250980,0.250980}%
\pgfsetstrokecolor{currentstroke}%
\pgfsetdash{}{0pt}%
\pgfpathmoveto{\pgfqpoint{2.250000in}{1.800000in}}%
\pgfpathlineto{\pgfqpoint{2.250000in}{3.150000in}}%
\pgfusepath{stroke}%
\end{pgfscope}%
\begin{pgfscope}%
\pgfpathrectangle{\pgfqpoint{0.000000in}{0.000000in}}{\pgfqpoint{6.000000in}{3.600000in}}%
\pgfusepath{clip}%
\pgfsetrectcap%
\pgfsetroundjoin%
\pgfsetlinewidth{3.011250pt}%
\definecolor{currentstroke}{rgb}{0.250980,0.250980,0.250980}%
\pgfsetstrokecolor{currentstroke}%
\pgfsetdash{}{0pt}%
\pgfpathmoveto{\pgfqpoint{2.250000in}{1.800000in}}%
\pgfpathlineto{\pgfqpoint{0.750000in}{2.250000in}}%
\pgfusepath{stroke}%
\end{pgfscope}%
\begin{pgfscope}%
\pgfpathrectangle{\pgfqpoint{0.000000in}{0.000000in}}{\pgfqpoint{6.000000in}{3.600000in}}%
\pgfusepath{clip}%
\pgfsetrectcap%
\pgfsetroundjoin%
\pgfsetlinewidth{3.011250pt}%
\definecolor{currentstroke}{rgb}{0.250980,0.250980,0.250980}%
\pgfsetstrokecolor{currentstroke}%
\pgfsetdash{}{0pt}%
\pgfpathmoveto{\pgfqpoint{3.750000in}{1.800000in}}%
\pgfpathlineto{\pgfqpoint{5.250000in}{2.250000in}}%
\pgfusepath{stroke}%
\end{pgfscope}%
\begin{pgfscope}%
\pgfpathrectangle{\pgfqpoint{0.000000in}{0.000000in}}{\pgfqpoint{6.000000in}{3.600000in}}%
\pgfusepath{clip}%
\pgfsetrectcap%
\pgfsetroundjoin%
\pgfsetlinewidth{3.011250pt}%
\definecolor{currentstroke}{rgb}{0.250980,0.250980,0.250980}%
\pgfsetstrokecolor{currentstroke}%
\pgfsetdash{}{0pt}%
\pgfpathmoveto{\pgfqpoint{3.750000in}{1.800000in}}%
\pgfpathlineto{\pgfqpoint{2.250000in}{1.800000in}}%
\pgfusepath{stroke}%
\end{pgfscope}%
\begin{pgfscope}%
\pgfpathrectangle{\pgfqpoint{0.000000in}{0.000000in}}{\pgfqpoint{6.000000in}{3.600000in}}%
\pgfusepath{clip}%
\pgfsetrectcap%
\pgfsetroundjoin%
\pgfsetlinewidth{3.011250pt}%
\definecolor{currentstroke}{rgb}{0.250980,0.250980,0.250980}%
\pgfsetstrokecolor{currentstroke}%
\pgfsetdash{}{0pt}%
\pgfpathmoveto{\pgfqpoint{0.750000in}{1.350000in}}%
\pgfpathlineto{\pgfqpoint{0.750000in}{2.250000in}}%
\pgfusepath{stroke}%
\end{pgfscope}%
\begin{pgfscope}%
\pgfpathrectangle{\pgfqpoint{0.000000in}{0.000000in}}{\pgfqpoint{6.000000in}{3.600000in}}%
\pgfusepath{clip}%
\pgfsetrectcap%
\pgfsetroundjoin%
\pgfsetlinewidth{3.011250pt}%
\definecolor{currentstroke}{rgb}{0.250980,0.250980,0.250980}%
\pgfsetstrokecolor{currentstroke}%
\pgfsetdash{}{0pt}%
\pgfpathmoveto{\pgfqpoint{2.250000in}{0.450000in}}%
\pgfpathlineto{\pgfqpoint{0.750000in}{1.350000in}}%
\pgfusepath{stroke}%
\end{pgfscope}%
\begin{pgfscope}%
\pgfpathrectangle{\pgfqpoint{0.000000in}{0.000000in}}{\pgfqpoint{6.000000in}{3.600000in}}%
\pgfusepath{clip}%
\pgfsetrectcap%
\pgfsetroundjoin%
\pgfsetlinewidth{3.011250pt}%
\definecolor{currentstroke}{rgb}{0.250980,0.250980,0.250980}%
\pgfsetstrokecolor{currentstroke}%
\pgfsetdash{}{0pt}%
\pgfpathmoveto{\pgfqpoint{3.750000in}{0.450000in}}%
\pgfpathlineto{\pgfqpoint{5.250000in}{1.350000in}}%
\pgfusepath{stroke}%
\end{pgfscope}%
\begin{pgfscope}%
\pgfpathrectangle{\pgfqpoint{0.000000in}{0.000000in}}{\pgfqpoint{6.000000in}{3.600000in}}%
\pgfusepath{clip}%
\pgfsetbuttcap%
\pgfsetroundjoin%
\definecolor{currentfill}{rgb}{0.800000,0.800000,0.800000}%
\pgfsetfillcolor{currentfill}%
\pgfsetlinewidth{1.003750pt}%
\definecolor{currentstroke}{rgb}{0.000000,0.000000,0.000000}%
\pgfsetstrokecolor{currentstroke}%
\pgfsetdash{}{0pt}%
\pgfsys@defobject{currentmarker}{\pgfqpoint{-0.104167in}{-0.104167in}}{\pgfqpoint{0.104167in}{0.104167in}}{%
\pgfpathmoveto{\pgfqpoint{0.000000in}{-0.104167in}}%
\pgfpathcurveto{\pgfqpoint{0.027625in}{-0.104167in}}{\pgfqpoint{0.054123in}{-0.093191in}}{\pgfqpoint{0.073657in}{-0.073657in}}%
\pgfpathcurveto{\pgfqpoint{0.093191in}{-0.054123in}}{\pgfqpoint{0.104167in}{-0.027625in}}{\pgfqpoint{0.104167in}{0.000000in}}%
\pgfpathcurveto{\pgfqpoint{0.104167in}{0.027625in}}{\pgfqpoint{0.093191in}{0.054123in}}{\pgfqpoint{0.073657in}{0.073657in}}%
\pgfpathcurveto{\pgfqpoint{0.054123in}{0.093191in}}{\pgfqpoint{0.027625in}{0.104167in}}{\pgfqpoint{0.000000in}{0.104167in}}%
\pgfpathcurveto{\pgfqpoint{-0.027625in}{0.104167in}}{\pgfqpoint{-0.054123in}{0.093191in}}{\pgfqpoint{-0.073657in}{0.073657in}}%
\pgfpathcurveto{\pgfqpoint{-0.093191in}{0.054123in}}{\pgfqpoint{-0.104167in}{0.027625in}}{\pgfqpoint{-0.104167in}{0.000000in}}%
\pgfpathcurveto{\pgfqpoint{-0.104167in}{-0.027625in}}{\pgfqpoint{-0.093191in}{-0.054123in}}{\pgfqpoint{-0.073657in}{-0.073657in}}%
\pgfpathcurveto{\pgfqpoint{-0.054123in}{-0.093191in}}{\pgfqpoint{-0.027625in}{-0.104167in}}{\pgfqpoint{0.000000in}{-0.104167in}}%
\pgfpathclose%
\pgfusepath{stroke,fill}%
}%
\begin{pgfscope}%
\pgfsys@transformshift{2.250000in}{3.150000in}%
\pgfsys@useobject{currentmarker}{}%
\end{pgfscope}%
\end{pgfscope}%
\begin{pgfscope}%
\pgfpathrectangle{\pgfqpoint{0.000000in}{0.000000in}}{\pgfqpoint{6.000000in}{3.600000in}}%
\pgfusepath{clip}%
\pgfsetbuttcap%
\pgfsetroundjoin%
\definecolor{currentfill}{rgb}{0.800000,0.800000,0.800000}%
\pgfsetfillcolor{currentfill}%
\pgfsetlinewidth{1.003750pt}%
\definecolor{currentstroke}{rgb}{0.000000,0.000000,0.000000}%
\pgfsetstrokecolor{currentstroke}%
\pgfsetdash{}{0pt}%
\pgfsys@defobject{currentmarker}{\pgfqpoint{-0.104167in}{-0.104167in}}{\pgfqpoint{0.104167in}{0.104167in}}{%
\pgfpathmoveto{\pgfqpoint{0.000000in}{-0.104167in}}%
\pgfpathcurveto{\pgfqpoint{0.027625in}{-0.104167in}}{\pgfqpoint{0.054123in}{-0.093191in}}{\pgfqpoint{0.073657in}{-0.073657in}}%
\pgfpathcurveto{\pgfqpoint{0.093191in}{-0.054123in}}{\pgfqpoint{0.104167in}{-0.027625in}}{\pgfqpoint{0.104167in}{0.000000in}}%
\pgfpathcurveto{\pgfqpoint{0.104167in}{0.027625in}}{\pgfqpoint{0.093191in}{0.054123in}}{\pgfqpoint{0.073657in}{0.073657in}}%
\pgfpathcurveto{\pgfqpoint{0.054123in}{0.093191in}}{\pgfqpoint{0.027625in}{0.104167in}}{\pgfqpoint{0.000000in}{0.104167in}}%
\pgfpathcurveto{\pgfqpoint{-0.027625in}{0.104167in}}{\pgfqpoint{-0.054123in}{0.093191in}}{\pgfqpoint{-0.073657in}{0.073657in}}%
\pgfpathcurveto{\pgfqpoint{-0.093191in}{0.054123in}}{\pgfqpoint{-0.104167in}{0.027625in}}{\pgfqpoint{-0.104167in}{0.000000in}}%
\pgfpathcurveto{\pgfqpoint{-0.104167in}{-0.027625in}}{\pgfqpoint{-0.093191in}{-0.054123in}}{\pgfqpoint{-0.073657in}{-0.073657in}}%
\pgfpathcurveto{\pgfqpoint{-0.054123in}{-0.093191in}}{\pgfqpoint{-0.027625in}{-0.104167in}}{\pgfqpoint{0.000000in}{-0.104167in}}%
\pgfpathclose%
\pgfusepath{stroke,fill}%
}%
\begin{pgfscope}%
\pgfsys@transformshift{3.750000in}{3.150000in}%
\pgfsys@useobject{currentmarker}{}%
\end{pgfscope}%
\end{pgfscope}%
\begin{pgfscope}%
\pgfpathrectangle{\pgfqpoint{0.000000in}{0.000000in}}{\pgfqpoint{6.000000in}{3.600000in}}%
\pgfusepath{clip}%
\pgfsetbuttcap%
\pgfsetroundjoin%
\definecolor{currentfill}{rgb}{0.800000,0.800000,0.800000}%
\pgfsetfillcolor{currentfill}%
\pgfsetlinewidth{1.003750pt}%
\definecolor{currentstroke}{rgb}{0.000000,0.000000,0.000000}%
\pgfsetstrokecolor{currentstroke}%
\pgfsetdash{}{0pt}%
\pgfsys@defobject{currentmarker}{\pgfqpoint{-0.104167in}{-0.104167in}}{\pgfqpoint{0.104167in}{0.104167in}}{%
\pgfpathmoveto{\pgfqpoint{0.000000in}{-0.104167in}}%
\pgfpathcurveto{\pgfqpoint{0.027625in}{-0.104167in}}{\pgfqpoint{0.054123in}{-0.093191in}}{\pgfqpoint{0.073657in}{-0.073657in}}%
\pgfpathcurveto{\pgfqpoint{0.093191in}{-0.054123in}}{\pgfqpoint{0.104167in}{-0.027625in}}{\pgfqpoint{0.104167in}{0.000000in}}%
\pgfpathcurveto{\pgfqpoint{0.104167in}{0.027625in}}{\pgfqpoint{0.093191in}{0.054123in}}{\pgfqpoint{0.073657in}{0.073657in}}%
\pgfpathcurveto{\pgfqpoint{0.054123in}{0.093191in}}{\pgfqpoint{0.027625in}{0.104167in}}{\pgfqpoint{0.000000in}{0.104167in}}%
\pgfpathcurveto{\pgfqpoint{-0.027625in}{0.104167in}}{\pgfqpoint{-0.054123in}{0.093191in}}{\pgfqpoint{-0.073657in}{0.073657in}}%
\pgfpathcurveto{\pgfqpoint{-0.093191in}{0.054123in}}{\pgfqpoint{-0.104167in}{0.027625in}}{\pgfqpoint{-0.104167in}{0.000000in}}%
\pgfpathcurveto{\pgfqpoint{-0.104167in}{-0.027625in}}{\pgfqpoint{-0.093191in}{-0.054123in}}{\pgfqpoint{-0.073657in}{-0.073657in}}%
\pgfpathcurveto{\pgfqpoint{-0.054123in}{-0.093191in}}{\pgfqpoint{-0.027625in}{-0.104167in}}{\pgfqpoint{0.000000in}{-0.104167in}}%
\pgfpathclose%
\pgfusepath{stroke,fill}%
}%
\begin{pgfscope}%
\pgfsys@transformshift{0.750000in}{2.250000in}%
\pgfsys@useobject{currentmarker}{}%
\end{pgfscope}%
\end{pgfscope}%
\begin{pgfscope}%
\pgfpathrectangle{\pgfqpoint{0.000000in}{0.000000in}}{\pgfqpoint{6.000000in}{3.600000in}}%
\pgfusepath{clip}%
\pgfsetbuttcap%
\pgfsetroundjoin%
\definecolor{currentfill}{rgb}{0.800000,0.800000,0.800000}%
\pgfsetfillcolor{currentfill}%
\pgfsetlinewidth{1.003750pt}%
\definecolor{currentstroke}{rgb}{0.000000,0.000000,0.000000}%
\pgfsetstrokecolor{currentstroke}%
\pgfsetdash{}{0pt}%
\pgfsys@defobject{currentmarker}{\pgfqpoint{-0.104167in}{-0.104167in}}{\pgfqpoint{0.104167in}{0.104167in}}{%
\pgfpathmoveto{\pgfqpoint{0.000000in}{-0.104167in}}%
\pgfpathcurveto{\pgfqpoint{0.027625in}{-0.104167in}}{\pgfqpoint{0.054123in}{-0.093191in}}{\pgfqpoint{0.073657in}{-0.073657in}}%
\pgfpathcurveto{\pgfqpoint{0.093191in}{-0.054123in}}{\pgfqpoint{0.104167in}{-0.027625in}}{\pgfqpoint{0.104167in}{0.000000in}}%
\pgfpathcurveto{\pgfqpoint{0.104167in}{0.027625in}}{\pgfqpoint{0.093191in}{0.054123in}}{\pgfqpoint{0.073657in}{0.073657in}}%
\pgfpathcurveto{\pgfqpoint{0.054123in}{0.093191in}}{\pgfqpoint{0.027625in}{0.104167in}}{\pgfqpoint{0.000000in}{0.104167in}}%
\pgfpathcurveto{\pgfqpoint{-0.027625in}{0.104167in}}{\pgfqpoint{-0.054123in}{0.093191in}}{\pgfqpoint{-0.073657in}{0.073657in}}%
\pgfpathcurveto{\pgfqpoint{-0.093191in}{0.054123in}}{\pgfqpoint{-0.104167in}{0.027625in}}{\pgfqpoint{-0.104167in}{0.000000in}}%
\pgfpathcurveto{\pgfqpoint{-0.104167in}{-0.027625in}}{\pgfqpoint{-0.093191in}{-0.054123in}}{\pgfqpoint{-0.073657in}{-0.073657in}}%
\pgfpathcurveto{\pgfqpoint{-0.054123in}{-0.093191in}}{\pgfqpoint{-0.027625in}{-0.104167in}}{\pgfqpoint{0.000000in}{-0.104167in}}%
\pgfpathclose%
\pgfusepath{stroke,fill}%
}%
\begin{pgfscope}%
\pgfsys@transformshift{5.250000in}{2.250000in}%
\pgfsys@useobject{currentmarker}{}%
\end{pgfscope}%
\end{pgfscope}%
\begin{pgfscope}%
\pgfpathrectangle{\pgfqpoint{0.000000in}{0.000000in}}{\pgfqpoint{6.000000in}{3.600000in}}%
\pgfusepath{clip}%
\pgfsetbuttcap%
\pgfsetroundjoin%
\definecolor{currentfill}{rgb}{0.800000,0.800000,0.800000}%
\pgfsetfillcolor{currentfill}%
\pgfsetlinewidth{1.003750pt}%
\definecolor{currentstroke}{rgb}{0.000000,0.000000,0.000000}%
\pgfsetstrokecolor{currentstroke}%
\pgfsetdash{}{0pt}%
\pgfsys@defobject{currentmarker}{\pgfqpoint{-0.104167in}{-0.104167in}}{\pgfqpoint{0.104167in}{0.104167in}}{%
\pgfpathmoveto{\pgfqpoint{0.000000in}{-0.104167in}}%
\pgfpathcurveto{\pgfqpoint{0.027625in}{-0.104167in}}{\pgfqpoint{0.054123in}{-0.093191in}}{\pgfqpoint{0.073657in}{-0.073657in}}%
\pgfpathcurveto{\pgfqpoint{0.093191in}{-0.054123in}}{\pgfqpoint{0.104167in}{-0.027625in}}{\pgfqpoint{0.104167in}{0.000000in}}%
\pgfpathcurveto{\pgfqpoint{0.104167in}{0.027625in}}{\pgfqpoint{0.093191in}{0.054123in}}{\pgfqpoint{0.073657in}{0.073657in}}%
\pgfpathcurveto{\pgfqpoint{0.054123in}{0.093191in}}{\pgfqpoint{0.027625in}{0.104167in}}{\pgfqpoint{0.000000in}{0.104167in}}%
\pgfpathcurveto{\pgfqpoint{-0.027625in}{0.104167in}}{\pgfqpoint{-0.054123in}{0.093191in}}{\pgfqpoint{-0.073657in}{0.073657in}}%
\pgfpathcurveto{\pgfqpoint{-0.093191in}{0.054123in}}{\pgfqpoint{-0.104167in}{0.027625in}}{\pgfqpoint{-0.104167in}{0.000000in}}%
\pgfpathcurveto{\pgfqpoint{-0.104167in}{-0.027625in}}{\pgfqpoint{-0.093191in}{-0.054123in}}{\pgfqpoint{-0.073657in}{-0.073657in}}%
\pgfpathcurveto{\pgfqpoint{-0.054123in}{-0.093191in}}{\pgfqpoint{-0.027625in}{-0.104167in}}{\pgfqpoint{0.000000in}{-0.104167in}}%
\pgfpathclose%
\pgfusepath{stroke,fill}%
}%
\begin{pgfscope}%
\pgfsys@transformshift{2.250000in}{1.800000in}%
\pgfsys@useobject{currentmarker}{}%
\end{pgfscope}%
\end{pgfscope}%
\begin{pgfscope}%
\pgfpathrectangle{\pgfqpoint{0.000000in}{0.000000in}}{\pgfqpoint{6.000000in}{3.600000in}}%
\pgfusepath{clip}%
\pgfsetbuttcap%
\pgfsetroundjoin%
\definecolor{currentfill}{rgb}{0.800000,0.800000,0.800000}%
\pgfsetfillcolor{currentfill}%
\pgfsetlinewidth{1.003750pt}%
\definecolor{currentstroke}{rgb}{0.000000,0.000000,0.000000}%
\pgfsetstrokecolor{currentstroke}%
\pgfsetdash{}{0pt}%
\pgfsys@defobject{currentmarker}{\pgfqpoint{-0.104167in}{-0.104167in}}{\pgfqpoint{0.104167in}{0.104167in}}{%
\pgfpathmoveto{\pgfqpoint{0.000000in}{-0.104167in}}%
\pgfpathcurveto{\pgfqpoint{0.027625in}{-0.104167in}}{\pgfqpoint{0.054123in}{-0.093191in}}{\pgfqpoint{0.073657in}{-0.073657in}}%
\pgfpathcurveto{\pgfqpoint{0.093191in}{-0.054123in}}{\pgfqpoint{0.104167in}{-0.027625in}}{\pgfqpoint{0.104167in}{0.000000in}}%
\pgfpathcurveto{\pgfqpoint{0.104167in}{0.027625in}}{\pgfqpoint{0.093191in}{0.054123in}}{\pgfqpoint{0.073657in}{0.073657in}}%
\pgfpathcurveto{\pgfqpoint{0.054123in}{0.093191in}}{\pgfqpoint{0.027625in}{0.104167in}}{\pgfqpoint{0.000000in}{0.104167in}}%
\pgfpathcurveto{\pgfqpoint{-0.027625in}{0.104167in}}{\pgfqpoint{-0.054123in}{0.093191in}}{\pgfqpoint{-0.073657in}{0.073657in}}%
\pgfpathcurveto{\pgfqpoint{-0.093191in}{0.054123in}}{\pgfqpoint{-0.104167in}{0.027625in}}{\pgfqpoint{-0.104167in}{0.000000in}}%
\pgfpathcurveto{\pgfqpoint{-0.104167in}{-0.027625in}}{\pgfqpoint{-0.093191in}{-0.054123in}}{\pgfqpoint{-0.073657in}{-0.073657in}}%
\pgfpathcurveto{\pgfqpoint{-0.054123in}{-0.093191in}}{\pgfqpoint{-0.027625in}{-0.104167in}}{\pgfqpoint{0.000000in}{-0.104167in}}%
\pgfpathclose%
\pgfusepath{stroke,fill}%
}%
\begin{pgfscope}%
\pgfsys@transformshift{3.750000in}{1.800000in}%
\pgfsys@useobject{currentmarker}{}%
\end{pgfscope}%
\end{pgfscope}%
\begin{pgfscope}%
\pgfpathrectangle{\pgfqpoint{0.000000in}{0.000000in}}{\pgfqpoint{6.000000in}{3.600000in}}%
\pgfusepath{clip}%
\pgfsetbuttcap%
\pgfsetroundjoin%
\definecolor{currentfill}{rgb}{0.800000,0.800000,0.800000}%
\pgfsetfillcolor{currentfill}%
\pgfsetlinewidth{1.003750pt}%
\definecolor{currentstroke}{rgb}{0.000000,0.000000,0.000000}%
\pgfsetstrokecolor{currentstroke}%
\pgfsetdash{}{0pt}%
\pgfsys@defobject{currentmarker}{\pgfqpoint{-0.104167in}{-0.104167in}}{\pgfqpoint{0.104167in}{0.104167in}}{%
\pgfpathmoveto{\pgfqpoint{0.000000in}{-0.104167in}}%
\pgfpathcurveto{\pgfqpoint{0.027625in}{-0.104167in}}{\pgfqpoint{0.054123in}{-0.093191in}}{\pgfqpoint{0.073657in}{-0.073657in}}%
\pgfpathcurveto{\pgfqpoint{0.093191in}{-0.054123in}}{\pgfqpoint{0.104167in}{-0.027625in}}{\pgfqpoint{0.104167in}{0.000000in}}%
\pgfpathcurveto{\pgfqpoint{0.104167in}{0.027625in}}{\pgfqpoint{0.093191in}{0.054123in}}{\pgfqpoint{0.073657in}{0.073657in}}%
\pgfpathcurveto{\pgfqpoint{0.054123in}{0.093191in}}{\pgfqpoint{0.027625in}{0.104167in}}{\pgfqpoint{0.000000in}{0.104167in}}%
\pgfpathcurveto{\pgfqpoint{-0.027625in}{0.104167in}}{\pgfqpoint{-0.054123in}{0.093191in}}{\pgfqpoint{-0.073657in}{0.073657in}}%
\pgfpathcurveto{\pgfqpoint{-0.093191in}{0.054123in}}{\pgfqpoint{-0.104167in}{0.027625in}}{\pgfqpoint{-0.104167in}{0.000000in}}%
\pgfpathcurveto{\pgfqpoint{-0.104167in}{-0.027625in}}{\pgfqpoint{-0.093191in}{-0.054123in}}{\pgfqpoint{-0.073657in}{-0.073657in}}%
\pgfpathcurveto{\pgfqpoint{-0.054123in}{-0.093191in}}{\pgfqpoint{-0.027625in}{-0.104167in}}{\pgfqpoint{0.000000in}{-0.104167in}}%
\pgfpathclose%
\pgfusepath{stroke,fill}%
}%
\begin{pgfscope}%
\pgfsys@transformshift{0.750000in}{1.350000in}%
\pgfsys@useobject{currentmarker}{}%
\end{pgfscope}%
\end{pgfscope}%
\begin{pgfscope}%
\pgfpathrectangle{\pgfqpoint{0.000000in}{0.000000in}}{\pgfqpoint{6.000000in}{3.600000in}}%
\pgfusepath{clip}%
\pgfsetbuttcap%
\pgfsetroundjoin%
\definecolor{currentfill}{rgb}{0.800000,0.800000,0.800000}%
\pgfsetfillcolor{currentfill}%
\pgfsetlinewidth{1.003750pt}%
\definecolor{currentstroke}{rgb}{0.000000,0.000000,0.000000}%
\pgfsetstrokecolor{currentstroke}%
\pgfsetdash{}{0pt}%
\pgfsys@defobject{currentmarker}{\pgfqpoint{-0.104167in}{-0.104167in}}{\pgfqpoint{0.104167in}{0.104167in}}{%
\pgfpathmoveto{\pgfqpoint{0.000000in}{-0.104167in}}%
\pgfpathcurveto{\pgfqpoint{0.027625in}{-0.104167in}}{\pgfqpoint{0.054123in}{-0.093191in}}{\pgfqpoint{0.073657in}{-0.073657in}}%
\pgfpathcurveto{\pgfqpoint{0.093191in}{-0.054123in}}{\pgfqpoint{0.104167in}{-0.027625in}}{\pgfqpoint{0.104167in}{0.000000in}}%
\pgfpathcurveto{\pgfqpoint{0.104167in}{0.027625in}}{\pgfqpoint{0.093191in}{0.054123in}}{\pgfqpoint{0.073657in}{0.073657in}}%
\pgfpathcurveto{\pgfqpoint{0.054123in}{0.093191in}}{\pgfqpoint{0.027625in}{0.104167in}}{\pgfqpoint{0.000000in}{0.104167in}}%
\pgfpathcurveto{\pgfqpoint{-0.027625in}{0.104167in}}{\pgfqpoint{-0.054123in}{0.093191in}}{\pgfqpoint{-0.073657in}{0.073657in}}%
\pgfpathcurveto{\pgfqpoint{-0.093191in}{0.054123in}}{\pgfqpoint{-0.104167in}{0.027625in}}{\pgfqpoint{-0.104167in}{0.000000in}}%
\pgfpathcurveto{\pgfqpoint{-0.104167in}{-0.027625in}}{\pgfqpoint{-0.093191in}{-0.054123in}}{\pgfqpoint{-0.073657in}{-0.073657in}}%
\pgfpathcurveto{\pgfqpoint{-0.054123in}{-0.093191in}}{\pgfqpoint{-0.027625in}{-0.104167in}}{\pgfqpoint{0.000000in}{-0.104167in}}%
\pgfpathclose%
\pgfusepath{stroke,fill}%
}%
\begin{pgfscope}%
\pgfsys@transformshift{5.250000in}{1.350000in}%
\pgfsys@useobject{currentmarker}{}%
\end{pgfscope}%
\end{pgfscope}%
\begin{pgfscope}%
\pgfpathrectangle{\pgfqpoint{0.000000in}{0.000000in}}{\pgfqpoint{6.000000in}{3.600000in}}%
\pgfusepath{clip}%
\pgfsetbuttcap%
\pgfsetroundjoin%
\definecolor{currentfill}{rgb}{0.800000,0.800000,0.800000}%
\pgfsetfillcolor{currentfill}%
\pgfsetlinewidth{1.003750pt}%
\definecolor{currentstroke}{rgb}{0.000000,0.000000,0.000000}%
\pgfsetstrokecolor{currentstroke}%
\pgfsetdash{}{0pt}%
\pgfsys@defobject{currentmarker}{\pgfqpoint{-0.104167in}{-0.104167in}}{\pgfqpoint{0.104167in}{0.104167in}}{%
\pgfpathmoveto{\pgfqpoint{0.000000in}{-0.104167in}}%
\pgfpathcurveto{\pgfqpoint{0.027625in}{-0.104167in}}{\pgfqpoint{0.054123in}{-0.093191in}}{\pgfqpoint{0.073657in}{-0.073657in}}%
\pgfpathcurveto{\pgfqpoint{0.093191in}{-0.054123in}}{\pgfqpoint{0.104167in}{-0.027625in}}{\pgfqpoint{0.104167in}{0.000000in}}%
\pgfpathcurveto{\pgfqpoint{0.104167in}{0.027625in}}{\pgfqpoint{0.093191in}{0.054123in}}{\pgfqpoint{0.073657in}{0.073657in}}%
\pgfpathcurveto{\pgfqpoint{0.054123in}{0.093191in}}{\pgfqpoint{0.027625in}{0.104167in}}{\pgfqpoint{0.000000in}{0.104167in}}%
\pgfpathcurveto{\pgfqpoint{-0.027625in}{0.104167in}}{\pgfqpoint{-0.054123in}{0.093191in}}{\pgfqpoint{-0.073657in}{0.073657in}}%
\pgfpathcurveto{\pgfqpoint{-0.093191in}{0.054123in}}{\pgfqpoint{-0.104167in}{0.027625in}}{\pgfqpoint{-0.104167in}{0.000000in}}%
\pgfpathcurveto{\pgfqpoint{-0.104167in}{-0.027625in}}{\pgfqpoint{-0.093191in}{-0.054123in}}{\pgfqpoint{-0.073657in}{-0.073657in}}%
\pgfpathcurveto{\pgfqpoint{-0.054123in}{-0.093191in}}{\pgfqpoint{-0.027625in}{-0.104167in}}{\pgfqpoint{0.000000in}{-0.104167in}}%
\pgfpathclose%
\pgfusepath{stroke,fill}%
}%
\begin{pgfscope}%
\pgfsys@transformshift{2.250000in}{0.450000in}%
\pgfsys@useobject{currentmarker}{}%
\end{pgfscope}%
\end{pgfscope}%
\begin{pgfscope}%
\pgfpathrectangle{\pgfqpoint{0.000000in}{0.000000in}}{\pgfqpoint{6.000000in}{3.600000in}}%
\pgfusepath{clip}%
\pgfsetbuttcap%
\pgfsetroundjoin%
\definecolor{currentfill}{rgb}{0.800000,0.800000,0.800000}%
\pgfsetfillcolor{currentfill}%
\pgfsetlinewidth{1.003750pt}%
\definecolor{currentstroke}{rgb}{0.000000,0.000000,0.000000}%
\pgfsetstrokecolor{currentstroke}%
\pgfsetdash{}{0pt}%
\pgfsys@defobject{currentmarker}{\pgfqpoint{-0.104167in}{-0.104167in}}{\pgfqpoint{0.104167in}{0.104167in}}{%
\pgfpathmoveto{\pgfqpoint{0.000000in}{-0.104167in}}%
\pgfpathcurveto{\pgfqpoint{0.027625in}{-0.104167in}}{\pgfqpoint{0.054123in}{-0.093191in}}{\pgfqpoint{0.073657in}{-0.073657in}}%
\pgfpathcurveto{\pgfqpoint{0.093191in}{-0.054123in}}{\pgfqpoint{0.104167in}{-0.027625in}}{\pgfqpoint{0.104167in}{0.000000in}}%
\pgfpathcurveto{\pgfqpoint{0.104167in}{0.027625in}}{\pgfqpoint{0.093191in}{0.054123in}}{\pgfqpoint{0.073657in}{0.073657in}}%
\pgfpathcurveto{\pgfqpoint{0.054123in}{0.093191in}}{\pgfqpoint{0.027625in}{0.104167in}}{\pgfqpoint{0.000000in}{0.104167in}}%
\pgfpathcurveto{\pgfqpoint{-0.027625in}{0.104167in}}{\pgfqpoint{-0.054123in}{0.093191in}}{\pgfqpoint{-0.073657in}{0.073657in}}%
\pgfpathcurveto{\pgfqpoint{-0.093191in}{0.054123in}}{\pgfqpoint{-0.104167in}{0.027625in}}{\pgfqpoint{-0.104167in}{0.000000in}}%
\pgfpathcurveto{\pgfqpoint{-0.104167in}{-0.027625in}}{\pgfqpoint{-0.093191in}{-0.054123in}}{\pgfqpoint{-0.073657in}{-0.073657in}}%
\pgfpathcurveto{\pgfqpoint{-0.054123in}{-0.093191in}}{\pgfqpoint{-0.027625in}{-0.104167in}}{\pgfqpoint{0.000000in}{-0.104167in}}%
\pgfpathclose%
\pgfusepath{stroke,fill}%
}%
\begin{pgfscope}%
\pgfsys@transformshift{3.750000in}{0.450000in}%
\pgfsys@useobject{currentmarker}{}%
\end{pgfscope}%
\end{pgfscope}%
\end{pgfpicture}%
\makeatother%
\endgroup%

%% file: fig/decomposition_H1.pgf
%% Creator: Matplotlib, PGF backend
%%
%% To include the figure in your LaTeX document, write
%%   \input{<filename>.pgf}
%%
%% Make sure the required packages are loaded in your preamble
%%   \usepackage{pgf}
%%
%% Figures using additional raster images can only be included by \input if
%% they are in the same directory as the main LaTeX file. For loading figures
%% from other directories you can use the `import` package
%%   \usepackage{import}
%%
%% and then include the figures with
%%   \import{<path to file>}{<filename>.pgf}
%%
%% Matplotlib used the following preamble
%%   \usepackage{fontspec}
%%   \setmainfont{DejaVuSerif.ttf}[Path=\detokenize{C:/Users/ccros/Anaconda3/Lib/site-packages/matplotlib/mpl-data/fonts/ttf/}]
%%   \setsansfont{DejaVuSans.ttf}[Path=\detokenize{C:/Users/ccros/Anaconda3/Lib/site-packages/matplotlib/mpl-data/fonts/ttf/}]
%%   \setmonofont{DejaVuSansMono.ttf}[Path=\detokenize{C:/Users/ccros/Anaconda3/Lib/site-packages/matplotlib/mpl-data/fonts/ttf/}]
%%
\begingroup%
\makeatletter%
\begin{pgfpicture}%
\pgfpathrectangle{\pgfpointorigin}{\pgfqpoint{6.000000in}{3.600000in}}%
\pgfusepath{use as bounding box, clip}%
\begin{pgfscope}%
\pgfsetbuttcap%
\pgfsetmiterjoin%
\definecolor{currentfill}{rgb}{1.000000,1.000000,1.000000}%
\pgfsetfillcolor{currentfill}%
\pgfsetlinewidth{0.000000pt}%
\definecolor{currentstroke}{rgb}{1.000000,1.000000,1.000000}%
\pgfsetstrokecolor{currentstroke}%
\pgfsetdash{}{0pt}%
\pgfpathmoveto{\pgfqpoint{0.000000in}{0.000000in}}%
\pgfpathlineto{\pgfqpoint{6.000000in}{0.000000in}}%
\pgfpathlineto{\pgfqpoint{6.000000in}{3.600000in}}%
\pgfpathlineto{\pgfqpoint{0.000000in}{3.600000in}}%
\pgfpathclose%
\pgfusepath{fill}%
\end{pgfscope}%
\begin{pgfscope}%
\pgfpathrectangle{\pgfqpoint{0.000000in}{0.000000in}}{\pgfqpoint{6.000000in}{3.600000in}}%
\pgfusepath{clip}%
\pgfsetrectcap%
\pgfsetroundjoin%
\pgfsetlinewidth{3.011250pt}%
\definecolor{currentstroke}{rgb}{0.250980,0.250980,0.250980}%
\pgfsetstrokecolor{currentstroke}%
\pgfsetdash{}{0pt}%
\pgfpathmoveto{\pgfqpoint{3.750000in}{3.150000in}}%
\pgfpathlineto{\pgfqpoint{2.250000in}{3.150000in}}%
\pgfusepath{stroke}%
\end{pgfscope}%
\begin{pgfscope}%
\pgfpathrectangle{\pgfqpoint{0.000000in}{0.000000in}}{\pgfqpoint{6.000000in}{3.600000in}}%
\pgfusepath{clip}%
\pgfsetrectcap%
\pgfsetroundjoin%
\pgfsetlinewidth{3.011250pt}%
\definecolor{currentstroke}{rgb}{0.250980,0.250980,0.250980}%
\pgfsetstrokecolor{currentstroke}%
\pgfsetdash{}{0pt}%
\pgfpathmoveto{\pgfqpoint{0.750000in}{2.250000in}}%
\pgfpathlineto{\pgfqpoint{2.250000in}{3.150000in}}%
\pgfusepath{stroke}%
\end{pgfscope}%
\begin{pgfscope}%
\pgfpathrectangle{\pgfqpoint{0.000000in}{0.000000in}}{\pgfqpoint{6.000000in}{3.600000in}}%
\pgfusepath{clip}%
\pgfsetrectcap%
\pgfsetroundjoin%
\pgfsetlinewidth{3.011250pt}%
\definecolor{currentstroke}{rgb}{0.250980,0.250980,0.250980}%
\pgfsetstrokecolor{currentstroke}%
\pgfsetdash{}{0pt}%
\pgfpathmoveto{\pgfqpoint{5.250000in}{2.250000in}}%
\pgfpathlineto{\pgfqpoint{3.750000in}{3.150000in}}%
\pgfusepath{stroke}%
\end{pgfscope}%
\begin{pgfscope}%
\pgfpathrectangle{\pgfqpoint{0.000000in}{0.000000in}}{\pgfqpoint{6.000000in}{3.600000in}}%
\pgfusepath{clip}%
\pgfsetrectcap%
\pgfsetroundjoin%
\pgfsetlinewidth{3.011250pt}%
\definecolor{currentstroke}{rgb}{0.250980,0.250980,0.250980}%
\pgfsetstrokecolor{currentstroke}%
\pgfsetdash{}{0pt}%
\pgfpathmoveto{\pgfqpoint{2.250000in}{1.800000in}}%
\pgfpathlineto{\pgfqpoint{0.750000in}{2.250000in}}%
\pgfusepath{stroke}%
\end{pgfscope}%
\begin{pgfscope}%
\pgfpathrectangle{\pgfqpoint{0.000000in}{0.000000in}}{\pgfqpoint{6.000000in}{3.600000in}}%
\pgfusepath{clip}%
\pgfsetrectcap%
\pgfsetroundjoin%
\pgfsetlinewidth{3.011250pt}%
\definecolor{currentstroke}{rgb}{0.250980,0.250980,0.250980}%
\pgfsetstrokecolor{currentstroke}%
\pgfsetdash{}{0pt}%
\pgfpathmoveto{\pgfqpoint{3.750000in}{1.800000in}}%
\pgfpathlineto{\pgfqpoint{5.250000in}{2.250000in}}%
\pgfusepath{stroke}%
\end{pgfscope}%
\begin{pgfscope}%
\pgfpathrectangle{\pgfqpoint{0.000000in}{0.000000in}}{\pgfqpoint{6.000000in}{3.600000in}}%
\pgfusepath{clip}%
\pgfsetrectcap%
\pgfsetroundjoin%
\pgfsetlinewidth{3.011250pt}%
\definecolor{currentstroke}{rgb}{0.250980,0.250980,0.250980}%
\pgfsetstrokecolor{currentstroke}%
\pgfsetdash{}{0pt}%
\pgfpathmoveto{\pgfqpoint{3.750000in}{1.800000in}}%
\pgfpathlineto{\pgfqpoint{2.250000in}{1.800000in}}%
\pgfusepath{stroke}%
\end{pgfscope}%
\begin{pgfscope}%
\pgfpathrectangle{\pgfqpoint{0.000000in}{0.000000in}}{\pgfqpoint{6.000000in}{3.600000in}}%
\pgfusepath{clip}%
\pgfsetrectcap%
\pgfsetroundjoin%
\pgfsetlinewidth{3.011250pt}%
\definecolor{currentstroke}{rgb}{0.250980,0.250980,0.250980}%
\pgfsetstrokecolor{currentstroke}%
\pgfsetdash{}{0pt}%
\pgfpathmoveto{\pgfqpoint{0.750000in}{1.350000in}}%
\pgfpathlineto{\pgfqpoint{0.750000in}{2.250000in}}%
\pgfusepath{stroke}%
\end{pgfscope}%
\begin{pgfscope}%
\pgfpathrectangle{\pgfqpoint{0.000000in}{0.000000in}}{\pgfqpoint{6.000000in}{3.600000in}}%
\pgfusepath{clip}%
\pgfsetrectcap%
\pgfsetroundjoin%
\pgfsetlinewidth{3.011250pt}%
\definecolor{currentstroke}{rgb}{0.250980,0.250980,0.250980}%
\pgfsetstrokecolor{currentstroke}%
\pgfsetdash{}{0pt}%
\pgfpathmoveto{\pgfqpoint{2.250000in}{0.450000in}}%
\pgfpathlineto{\pgfqpoint{0.750000in}{1.350000in}}%
\pgfusepath{stroke}%
\end{pgfscope}%
\begin{pgfscope}%
\pgfpathrectangle{\pgfqpoint{0.000000in}{0.000000in}}{\pgfqpoint{6.000000in}{3.600000in}}%
\pgfusepath{clip}%
\pgfsetrectcap%
\pgfsetroundjoin%
\pgfsetlinewidth{3.011250pt}%
\definecolor{currentstroke}{rgb}{0.250980,0.250980,0.250980}%
\pgfsetstrokecolor{currentstroke}%
\pgfsetdash{}{0pt}%
\pgfpathmoveto{\pgfqpoint{3.750000in}{0.450000in}}%
\pgfpathlineto{\pgfqpoint{5.250000in}{1.350000in}}%
\pgfusepath{stroke}%
\end{pgfscope}%
\begin{pgfscope}%
\pgfpathrectangle{\pgfqpoint{0.000000in}{0.000000in}}{\pgfqpoint{6.000000in}{3.600000in}}%
\pgfusepath{clip}%
\pgfsetbuttcap%
\pgfsetroundjoin%
\definecolor{currentfill}{rgb}{0.800000,0.800000,0.800000}%
\pgfsetfillcolor{currentfill}%
\pgfsetlinewidth{1.003750pt}%
\definecolor{currentstroke}{rgb}{0.000000,0.000000,0.000000}%
\pgfsetstrokecolor{currentstroke}%
\pgfsetdash{}{0pt}%
\pgfsys@defobject{currentmarker}{\pgfqpoint{-0.104167in}{-0.104167in}}{\pgfqpoint{0.104167in}{0.104167in}}{%
\pgfpathmoveto{\pgfqpoint{0.000000in}{-0.104167in}}%
\pgfpathcurveto{\pgfqpoint{0.027625in}{-0.104167in}}{\pgfqpoint{0.054123in}{-0.093191in}}{\pgfqpoint{0.073657in}{-0.073657in}}%
\pgfpathcurveto{\pgfqpoint{0.093191in}{-0.054123in}}{\pgfqpoint{0.104167in}{-0.027625in}}{\pgfqpoint{0.104167in}{0.000000in}}%
\pgfpathcurveto{\pgfqpoint{0.104167in}{0.027625in}}{\pgfqpoint{0.093191in}{0.054123in}}{\pgfqpoint{0.073657in}{0.073657in}}%
\pgfpathcurveto{\pgfqpoint{0.054123in}{0.093191in}}{\pgfqpoint{0.027625in}{0.104167in}}{\pgfqpoint{0.000000in}{0.104167in}}%
\pgfpathcurveto{\pgfqpoint{-0.027625in}{0.104167in}}{\pgfqpoint{-0.054123in}{0.093191in}}{\pgfqpoint{-0.073657in}{0.073657in}}%
\pgfpathcurveto{\pgfqpoint{-0.093191in}{0.054123in}}{\pgfqpoint{-0.104167in}{0.027625in}}{\pgfqpoint{-0.104167in}{0.000000in}}%
\pgfpathcurveto{\pgfqpoint{-0.104167in}{-0.027625in}}{\pgfqpoint{-0.093191in}{-0.054123in}}{\pgfqpoint{-0.073657in}{-0.073657in}}%
\pgfpathcurveto{\pgfqpoint{-0.054123in}{-0.093191in}}{\pgfqpoint{-0.027625in}{-0.104167in}}{\pgfqpoint{0.000000in}{-0.104167in}}%
\pgfpathclose%
\pgfusepath{stroke,fill}%
}%
\begin{pgfscope}%
\pgfsys@transformshift{2.250000in}{3.150000in}%
\pgfsys@useobject{currentmarker}{}%
\end{pgfscope}%
\end{pgfscope}%
\begin{pgfscope}%
\pgfpathrectangle{\pgfqpoint{0.000000in}{0.000000in}}{\pgfqpoint{6.000000in}{3.600000in}}%
\pgfusepath{clip}%
\pgfsetbuttcap%
\pgfsetroundjoin%
\definecolor{currentfill}{rgb}{0.800000,0.800000,0.800000}%
\pgfsetfillcolor{currentfill}%
\pgfsetlinewidth{1.003750pt}%
\definecolor{currentstroke}{rgb}{0.000000,0.000000,0.000000}%
\pgfsetstrokecolor{currentstroke}%
\pgfsetdash{}{0pt}%
\pgfsys@defobject{currentmarker}{\pgfqpoint{-0.104167in}{-0.104167in}}{\pgfqpoint{0.104167in}{0.104167in}}{%
\pgfpathmoveto{\pgfqpoint{0.000000in}{-0.104167in}}%
\pgfpathcurveto{\pgfqpoint{0.027625in}{-0.104167in}}{\pgfqpoint{0.054123in}{-0.093191in}}{\pgfqpoint{0.073657in}{-0.073657in}}%
\pgfpathcurveto{\pgfqpoint{0.093191in}{-0.054123in}}{\pgfqpoint{0.104167in}{-0.027625in}}{\pgfqpoint{0.104167in}{0.000000in}}%
\pgfpathcurveto{\pgfqpoint{0.104167in}{0.027625in}}{\pgfqpoint{0.093191in}{0.054123in}}{\pgfqpoint{0.073657in}{0.073657in}}%
\pgfpathcurveto{\pgfqpoint{0.054123in}{0.093191in}}{\pgfqpoint{0.027625in}{0.104167in}}{\pgfqpoint{0.000000in}{0.104167in}}%
\pgfpathcurveto{\pgfqpoint{-0.027625in}{0.104167in}}{\pgfqpoint{-0.054123in}{0.093191in}}{\pgfqpoint{-0.073657in}{0.073657in}}%
\pgfpathcurveto{\pgfqpoint{-0.093191in}{0.054123in}}{\pgfqpoint{-0.104167in}{0.027625in}}{\pgfqpoint{-0.104167in}{0.000000in}}%
\pgfpathcurveto{\pgfqpoint{-0.104167in}{-0.027625in}}{\pgfqpoint{-0.093191in}{-0.054123in}}{\pgfqpoint{-0.073657in}{-0.073657in}}%
\pgfpathcurveto{\pgfqpoint{-0.054123in}{-0.093191in}}{\pgfqpoint{-0.027625in}{-0.104167in}}{\pgfqpoint{0.000000in}{-0.104167in}}%
\pgfpathclose%
\pgfusepath{stroke,fill}%
}%
\begin{pgfscope}%
\pgfsys@transformshift{3.750000in}{3.150000in}%
\pgfsys@useobject{currentmarker}{}%
\end{pgfscope}%
\end{pgfscope}%
\begin{pgfscope}%
\pgfpathrectangle{\pgfqpoint{0.000000in}{0.000000in}}{\pgfqpoint{6.000000in}{3.600000in}}%
\pgfusepath{clip}%
\pgfsetbuttcap%
\pgfsetroundjoin%
\definecolor{currentfill}{rgb}{0.800000,0.800000,0.800000}%
\pgfsetfillcolor{currentfill}%
\pgfsetlinewidth{1.003750pt}%
\definecolor{currentstroke}{rgb}{0.000000,0.000000,0.000000}%
\pgfsetstrokecolor{currentstroke}%
\pgfsetdash{}{0pt}%
\pgfsys@defobject{currentmarker}{\pgfqpoint{-0.104167in}{-0.104167in}}{\pgfqpoint{0.104167in}{0.104167in}}{%
\pgfpathmoveto{\pgfqpoint{0.000000in}{-0.104167in}}%
\pgfpathcurveto{\pgfqpoint{0.027625in}{-0.104167in}}{\pgfqpoint{0.054123in}{-0.093191in}}{\pgfqpoint{0.073657in}{-0.073657in}}%
\pgfpathcurveto{\pgfqpoint{0.093191in}{-0.054123in}}{\pgfqpoint{0.104167in}{-0.027625in}}{\pgfqpoint{0.104167in}{0.000000in}}%
\pgfpathcurveto{\pgfqpoint{0.104167in}{0.027625in}}{\pgfqpoint{0.093191in}{0.054123in}}{\pgfqpoint{0.073657in}{0.073657in}}%
\pgfpathcurveto{\pgfqpoint{0.054123in}{0.093191in}}{\pgfqpoint{0.027625in}{0.104167in}}{\pgfqpoint{0.000000in}{0.104167in}}%
\pgfpathcurveto{\pgfqpoint{-0.027625in}{0.104167in}}{\pgfqpoint{-0.054123in}{0.093191in}}{\pgfqpoint{-0.073657in}{0.073657in}}%
\pgfpathcurveto{\pgfqpoint{-0.093191in}{0.054123in}}{\pgfqpoint{-0.104167in}{0.027625in}}{\pgfqpoint{-0.104167in}{0.000000in}}%
\pgfpathcurveto{\pgfqpoint{-0.104167in}{-0.027625in}}{\pgfqpoint{-0.093191in}{-0.054123in}}{\pgfqpoint{-0.073657in}{-0.073657in}}%
\pgfpathcurveto{\pgfqpoint{-0.054123in}{-0.093191in}}{\pgfqpoint{-0.027625in}{-0.104167in}}{\pgfqpoint{0.000000in}{-0.104167in}}%
\pgfpathclose%
\pgfusepath{stroke,fill}%
}%
\begin{pgfscope}%
\pgfsys@transformshift{0.750000in}{2.250000in}%
\pgfsys@useobject{currentmarker}{}%
\end{pgfscope}%
\end{pgfscope}%
\begin{pgfscope}%
\pgfpathrectangle{\pgfqpoint{0.000000in}{0.000000in}}{\pgfqpoint{6.000000in}{3.600000in}}%
\pgfusepath{clip}%
\pgfsetbuttcap%
\pgfsetroundjoin%
\definecolor{currentfill}{rgb}{0.800000,0.800000,0.800000}%
\pgfsetfillcolor{currentfill}%
\pgfsetlinewidth{1.003750pt}%
\definecolor{currentstroke}{rgb}{0.000000,0.000000,0.000000}%
\pgfsetstrokecolor{currentstroke}%
\pgfsetdash{}{0pt}%
\pgfsys@defobject{currentmarker}{\pgfqpoint{-0.104167in}{-0.104167in}}{\pgfqpoint{0.104167in}{0.104167in}}{%
\pgfpathmoveto{\pgfqpoint{0.000000in}{-0.104167in}}%
\pgfpathcurveto{\pgfqpoint{0.027625in}{-0.104167in}}{\pgfqpoint{0.054123in}{-0.093191in}}{\pgfqpoint{0.073657in}{-0.073657in}}%
\pgfpathcurveto{\pgfqpoint{0.093191in}{-0.054123in}}{\pgfqpoint{0.104167in}{-0.027625in}}{\pgfqpoint{0.104167in}{0.000000in}}%
\pgfpathcurveto{\pgfqpoint{0.104167in}{0.027625in}}{\pgfqpoint{0.093191in}{0.054123in}}{\pgfqpoint{0.073657in}{0.073657in}}%
\pgfpathcurveto{\pgfqpoint{0.054123in}{0.093191in}}{\pgfqpoint{0.027625in}{0.104167in}}{\pgfqpoint{0.000000in}{0.104167in}}%
\pgfpathcurveto{\pgfqpoint{-0.027625in}{0.104167in}}{\pgfqpoint{-0.054123in}{0.093191in}}{\pgfqpoint{-0.073657in}{0.073657in}}%
\pgfpathcurveto{\pgfqpoint{-0.093191in}{0.054123in}}{\pgfqpoint{-0.104167in}{0.027625in}}{\pgfqpoint{-0.104167in}{0.000000in}}%
\pgfpathcurveto{\pgfqpoint{-0.104167in}{-0.027625in}}{\pgfqpoint{-0.093191in}{-0.054123in}}{\pgfqpoint{-0.073657in}{-0.073657in}}%
\pgfpathcurveto{\pgfqpoint{-0.054123in}{-0.093191in}}{\pgfqpoint{-0.027625in}{-0.104167in}}{\pgfqpoint{0.000000in}{-0.104167in}}%
\pgfpathclose%
\pgfusepath{stroke,fill}%
}%
\begin{pgfscope}%
\pgfsys@transformshift{5.250000in}{2.250000in}%
\pgfsys@useobject{currentmarker}{}%
\end{pgfscope}%
\end{pgfscope}%
\begin{pgfscope}%
\pgfpathrectangle{\pgfqpoint{0.000000in}{0.000000in}}{\pgfqpoint{6.000000in}{3.600000in}}%
\pgfusepath{clip}%
\pgfsetbuttcap%
\pgfsetroundjoin%
\definecolor{currentfill}{rgb}{0.800000,0.800000,0.800000}%
\pgfsetfillcolor{currentfill}%
\pgfsetlinewidth{1.003750pt}%
\definecolor{currentstroke}{rgb}{0.000000,0.000000,0.000000}%
\pgfsetstrokecolor{currentstroke}%
\pgfsetdash{}{0pt}%
\pgfsys@defobject{currentmarker}{\pgfqpoint{-0.104167in}{-0.104167in}}{\pgfqpoint{0.104167in}{0.104167in}}{%
\pgfpathmoveto{\pgfqpoint{0.000000in}{-0.104167in}}%
\pgfpathcurveto{\pgfqpoint{0.027625in}{-0.104167in}}{\pgfqpoint{0.054123in}{-0.093191in}}{\pgfqpoint{0.073657in}{-0.073657in}}%
\pgfpathcurveto{\pgfqpoint{0.093191in}{-0.054123in}}{\pgfqpoint{0.104167in}{-0.027625in}}{\pgfqpoint{0.104167in}{0.000000in}}%
\pgfpathcurveto{\pgfqpoint{0.104167in}{0.027625in}}{\pgfqpoint{0.093191in}{0.054123in}}{\pgfqpoint{0.073657in}{0.073657in}}%
\pgfpathcurveto{\pgfqpoint{0.054123in}{0.093191in}}{\pgfqpoint{0.027625in}{0.104167in}}{\pgfqpoint{0.000000in}{0.104167in}}%
\pgfpathcurveto{\pgfqpoint{-0.027625in}{0.104167in}}{\pgfqpoint{-0.054123in}{0.093191in}}{\pgfqpoint{-0.073657in}{0.073657in}}%
\pgfpathcurveto{\pgfqpoint{-0.093191in}{0.054123in}}{\pgfqpoint{-0.104167in}{0.027625in}}{\pgfqpoint{-0.104167in}{0.000000in}}%
\pgfpathcurveto{\pgfqpoint{-0.104167in}{-0.027625in}}{\pgfqpoint{-0.093191in}{-0.054123in}}{\pgfqpoint{-0.073657in}{-0.073657in}}%
\pgfpathcurveto{\pgfqpoint{-0.054123in}{-0.093191in}}{\pgfqpoint{-0.027625in}{-0.104167in}}{\pgfqpoint{0.000000in}{-0.104167in}}%
\pgfpathclose%
\pgfusepath{stroke,fill}%
}%
\begin{pgfscope}%
\pgfsys@transformshift{2.250000in}{1.800000in}%
\pgfsys@useobject{currentmarker}{}%
\end{pgfscope}%
\end{pgfscope}%
\begin{pgfscope}%
\pgfpathrectangle{\pgfqpoint{0.000000in}{0.000000in}}{\pgfqpoint{6.000000in}{3.600000in}}%
\pgfusepath{clip}%
\pgfsetbuttcap%
\pgfsetroundjoin%
\definecolor{currentfill}{rgb}{0.800000,0.800000,0.800000}%
\pgfsetfillcolor{currentfill}%
\pgfsetlinewidth{1.003750pt}%
\definecolor{currentstroke}{rgb}{0.000000,0.000000,0.000000}%
\pgfsetstrokecolor{currentstroke}%
\pgfsetdash{}{0pt}%
\pgfsys@defobject{currentmarker}{\pgfqpoint{-0.104167in}{-0.104167in}}{\pgfqpoint{0.104167in}{0.104167in}}{%
\pgfpathmoveto{\pgfqpoint{0.000000in}{-0.104167in}}%
\pgfpathcurveto{\pgfqpoint{0.027625in}{-0.104167in}}{\pgfqpoint{0.054123in}{-0.093191in}}{\pgfqpoint{0.073657in}{-0.073657in}}%
\pgfpathcurveto{\pgfqpoint{0.093191in}{-0.054123in}}{\pgfqpoint{0.104167in}{-0.027625in}}{\pgfqpoint{0.104167in}{0.000000in}}%
\pgfpathcurveto{\pgfqpoint{0.104167in}{0.027625in}}{\pgfqpoint{0.093191in}{0.054123in}}{\pgfqpoint{0.073657in}{0.073657in}}%
\pgfpathcurveto{\pgfqpoint{0.054123in}{0.093191in}}{\pgfqpoint{0.027625in}{0.104167in}}{\pgfqpoint{0.000000in}{0.104167in}}%
\pgfpathcurveto{\pgfqpoint{-0.027625in}{0.104167in}}{\pgfqpoint{-0.054123in}{0.093191in}}{\pgfqpoint{-0.073657in}{0.073657in}}%
\pgfpathcurveto{\pgfqpoint{-0.093191in}{0.054123in}}{\pgfqpoint{-0.104167in}{0.027625in}}{\pgfqpoint{-0.104167in}{0.000000in}}%
\pgfpathcurveto{\pgfqpoint{-0.104167in}{-0.027625in}}{\pgfqpoint{-0.093191in}{-0.054123in}}{\pgfqpoint{-0.073657in}{-0.073657in}}%
\pgfpathcurveto{\pgfqpoint{-0.054123in}{-0.093191in}}{\pgfqpoint{-0.027625in}{-0.104167in}}{\pgfqpoint{0.000000in}{-0.104167in}}%
\pgfpathclose%
\pgfusepath{stroke,fill}%
}%
\begin{pgfscope}%
\pgfsys@transformshift{3.750000in}{1.800000in}%
\pgfsys@useobject{currentmarker}{}%
\end{pgfscope}%
\end{pgfscope}%
\begin{pgfscope}%
\pgfpathrectangle{\pgfqpoint{0.000000in}{0.000000in}}{\pgfqpoint{6.000000in}{3.600000in}}%
\pgfusepath{clip}%
\pgfsetbuttcap%
\pgfsetroundjoin%
\definecolor{currentfill}{rgb}{0.800000,0.800000,0.800000}%
\pgfsetfillcolor{currentfill}%
\pgfsetlinewidth{1.003750pt}%
\definecolor{currentstroke}{rgb}{0.000000,0.000000,0.000000}%
\pgfsetstrokecolor{currentstroke}%
\pgfsetdash{}{0pt}%
\pgfsys@defobject{currentmarker}{\pgfqpoint{-0.104167in}{-0.104167in}}{\pgfqpoint{0.104167in}{0.104167in}}{%
\pgfpathmoveto{\pgfqpoint{0.000000in}{-0.104167in}}%
\pgfpathcurveto{\pgfqpoint{0.027625in}{-0.104167in}}{\pgfqpoint{0.054123in}{-0.093191in}}{\pgfqpoint{0.073657in}{-0.073657in}}%
\pgfpathcurveto{\pgfqpoint{0.093191in}{-0.054123in}}{\pgfqpoint{0.104167in}{-0.027625in}}{\pgfqpoint{0.104167in}{0.000000in}}%
\pgfpathcurveto{\pgfqpoint{0.104167in}{0.027625in}}{\pgfqpoint{0.093191in}{0.054123in}}{\pgfqpoint{0.073657in}{0.073657in}}%
\pgfpathcurveto{\pgfqpoint{0.054123in}{0.093191in}}{\pgfqpoint{0.027625in}{0.104167in}}{\pgfqpoint{0.000000in}{0.104167in}}%
\pgfpathcurveto{\pgfqpoint{-0.027625in}{0.104167in}}{\pgfqpoint{-0.054123in}{0.093191in}}{\pgfqpoint{-0.073657in}{0.073657in}}%
\pgfpathcurveto{\pgfqpoint{-0.093191in}{0.054123in}}{\pgfqpoint{-0.104167in}{0.027625in}}{\pgfqpoint{-0.104167in}{0.000000in}}%
\pgfpathcurveto{\pgfqpoint{-0.104167in}{-0.027625in}}{\pgfqpoint{-0.093191in}{-0.054123in}}{\pgfqpoint{-0.073657in}{-0.073657in}}%
\pgfpathcurveto{\pgfqpoint{-0.054123in}{-0.093191in}}{\pgfqpoint{-0.027625in}{-0.104167in}}{\pgfqpoint{0.000000in}{-0.104167in}}%
\pgfpathclose%
\pgfusepath{stroke,fill}%
}%
\begin{pgfscope}%
\pgfsys@transformshift{0.750000in}{1.350000in}%
\pgfsys@useobject{currentmarker}{}%
\end{pgfscope}%
\end{pgfscope}%
\begin{pgfscope}%
\pgfpathrectangle{\pgfqpoint{0.000000in}{0.000000in}}{\pgfqpoint{6.000000in}{3.600000in}}%
\pgfusepath{clip}%
\pgfsetbuttcap%
\pgfsetroundjoin%
\definecolor{currentfill}{rgb}{0.800000,0.800000,0.800000}%
\pgfsetfillcolor{currentfill}%
\pgfsetlinewidth{1.003750pt}%
\definecolor{currentstroke}{rgb}{0.000000,0.000000,0.000000}%
\pgfsetstrokecolor{currentstroke}%
\pgfsetdash{}{0pt}%
\pgfsys@defobject{currentmarker}{\pgfqpoint{-0.104167in}{-0.104167in}}{\pgfqpoint{0.104167in}{0.104167in}}{%
\pgfpathmoveto{\pgfqpoint{0.000000in}{-0.104167in}}%
\pgfpathcurveto{\pgfqpoint{0.027625in}{-0.104167in}}{\pgfqpoint{0.054123in}{-0.093191in}}{\pgfqpoint{0.073657in}{-0.073657in}}%
\pgfpathcurveto{\pgfqpoint{0.093191in}{-0.054123in}}{\pgfqpoint{0.104167in}{-0.027625in}}{\pgfqpoint{0.104167in}{0.000000in}}%
\pgfpathcurveto{\pgfqpoint{0.104167in}{0.027625in}}{\pgfqpoint{0.093191in}{0.054123in}}{\pgfqpoint{0.073657in}{0.073657in}}%
\pgfpathcurveto{\pgfqpoint{0.054123in}{0.093191in}}{\pgfqpoint{0.027625in}{0.104167in}}{\pgfqpoint{0.000000in}{0.104167in}}%
\pgfpathcurveto{\pgfqpoint{-0.027625in}{0.104167in}}{\pgfqpoint{-0.054123in}{0.093191in}}{\pgfqpoint{-0.073657in}{0.073657in}}%
\pgfpathcurveto{\pgfqpoint{-0.093191in}{0.054123in}}{\pgfqpoint{-0.104167in}{0.027625in}}{\pgfqpoint{-0.104167in}{0.000000in}}%
\pgfpathcurveto{\pgfqpoint{-0.104167in}{-0.027625in}}{\pgfqpoint{-0.093191in}{-0.054123in}}{\pgfqpoint{-0.073657in}{-0.073657in}}%
\pgfpathcurveto{\pgfqpoint{-0.054123in}{-0.093191in}}{\pgfqpoint{-0.027625in}{-0.104167in}}{\pgfqpoint{0.000000in}{-0.104167in}}%
\pgfpathclose%
\pgfusepath{stroke,fill}%
}%
\begin{pgfscope}%
\pgfsys@transformshift{5.250000in}{1.350000in}%
\pgfsys@useobject{currentmarker}{}%
\end{pgfscope}%
\end{pgfscope}%
\begin{pgfscope}%
\pgfpathrectangle{\pgfqpoint{0.000000in}{0.000000in}}{\pgfqpoint{6.000000in}{3.600000in}}%
\pgfusepath{clip}%
\pgfsetbuttcap%
\pgfsetroundjoin%
\definecolor{currentfill}{rgb}{0.800000,0.800000,0.800000}%
\pgfsetfillcolor{currentfill}%
\pgfsetlinewidth{1.003750pt}%
\definecolor{currentstroke}{rgb}{0.000000,0.000000,0.000000}%
\pgfsetstrokecolor{currentstroke}%
\pgfsetdash{}{0pt}%
\pgfsys@defobject{currentmarker}{\pgfqpoint{-0.104167in}{-0.104167in}}{\pgfqpoint{0.104167in}{0.104167in}}{%
\pgfpathmoveto{\pgfqpoint{0.000000in}{-0.104167in}}%
\pgfpathcurveto{\pgfqpoint{0.027625in}{-0.104167in}}{\pgfqpoint{0.054123in}{-0.093191in}}{\pgfqpoint{0.073657in}{-0.073657in}}%
\pgfpathcurveto{\pgfqpoint{0.093191in}{-0.054123in}}{\pgfqpoint{0.104167in}{-0.027625in}}{\pgfqpoint{0.104167in}{0.000000in}}%
\pgfpathcurveto{\pgfqpoint{0.104167in}{0.027625in}}{\pgfqpoint{0.093191in}{0.054123in}}{\pgfqpoint{0.073657in}{0.073657in}}%
\pgfpathcurveto{\pgfqpoint{0.054123in}{0.093191in}}{\pgfqpoint{0.027625in}{0.104167in}}{\pgfqpoint{0.000000in}{0.104167in}}%
\pgfpathcurveto{\pgfqpoint{-0.027625in}{0.104167in}}{\pgfqpoint{-0.054123in}{0.093191in}}{\pgfqpoint{-0.073657in}{0.073657in}}%
\pgfpathcurveto{\pgfqpoint{-0.093191in}{0.054123in}}{\pgfqpoint{-0.104167in}{0.027625in}}{\pgfqpoint{-0.104167in}{0.000000in}}%
\pgfpathcurveto{\pgfqpoint{-0.104167in}{-0.027625in}}{\pgfqpoint{-0.093191in}{-0.054123in}}{\pgfqpoint{-0.073657in}{-0.073657in}}%
\pgfpathcurveto{\pgfqpoint{-0.054123in}{-0.093191in}}{\pgfqpoint{-0.027625in}{-0.104167in}}{\pgfqpoint{0.000000in}{-0.104167in}}%
\pgfpathclose%
\pgfusepath{stroke,fill}%
}%
\begin{pgfscope}%
\pgfsys@transformshift{2.250000in}{0.450000in}%
\pgfsys@useobject{currentmarker}{}%
\end{pgfscope}%
\end{pgfscope}%
\begin{pgfscope}%
\pgfpathrectangle{\pgfqpoint{0.000000in}{0.000000in}}{\pgfqpoint{6.000000in}{3.600000in}}%
\pgfusepath{clip}%
\pgfsetbuttcap%
\pgfsetroundjoin%
\definecolor{currentfill}{rgb}{0.800000,0.800000,0.800000}%
\pgfsetfillcolor{currentfill}%
\pgfsetlinewidth{1.003750pt}%
\definecolor{currentstroke}{rgb}{0.000000,0.000000,0.000000}%
\pgfsetstrokecolor{currentstroke}%
\pgfsetdash{}{0pt}%
\pgfsys@defobject{currentmarker}{\pgfqpoint{-0.104167in}{-0.104167in}}{\pgfqpoint{0.104167in}{0.104167in}}{%
\pgfpathmoveto{\pgfqpoint{0.000000in}{-0.104167in}}%
\pgfpathcurveto{\pgfqpoint{0.027625in}{-0.104167in}}{\pgfqpoint{0.054123in}{-0.093191in}}{\pgfqpoint{0.073657in}{-0.073657in}}%
\pgfpathcurveto{\pgfqpoint{0.093191in}{-0.054123in}}{\pgfqpoint{0.104167in}{-0.027625in}}{\pgfqpoint{0.104167in}{0.000000in}}%
\pgfpathcurveto{\pgfqpoint{0.104167in}{0.027625in}}{\pgfqpoint{0.093191in}{0.054123in}}{\pgfqpoint{0.073657in}{0.073657in}}%
\pgfpathcurveto{\pgfqpoint{0.054123in}{0.093191in}}{\pgfqpoint{0.027625in}{0.104167in}}{\pgfqpoint{0.000000in}{0.104167in}}%
\pgfpathcurveto{\pgfqpoint{-0.027625in}{0.104167in}}{\pgfqpoint{-0.054123in}{0.093191in}}{\pgfqpoint{-0.073657in}{0.073657in}}%
\pgfpathcurveto{\pgfqpoint{-0.093191in}{0.054123in}}{\pgfqpoint{-0.104167in}{0.027625in}}{\pgfqpoint{-0.104167in}{0.000000in}}%
\pgfpathcurveto{\pgfqpoint{-0.104167in}{-0.027625in}}{\pgfqpoint{-0.093191in}{-0.054123in}}{\pgfqpoint{-0.073657in}{-0.073657in}}%
\pgfpathcurveto{\pgfqpoint{-0.054123in}{-0.093191in}}{\pgfqpoint{-0.027625in}{-0.104167in}}{\pgfqpoint{0.000000in}{-0.104167in}}%
\pgfpathclose%
\pgfusepath{stroke,fill}%
}%
\begin{pgfscope}%
\pgfsys@transformshift{3.750000in}{0.450000in}%
\pgfsys@useobject{currentmarker}{}%
\end{pgfscope}%
\end{pgfscope}%
\end{pgfpicture}%
\makeatother%
\endgroup%

%% file: fig/decomposition_H2.pgf
%% Creator: Matplotlib, PGF backend
%%
%% To include the figure in your LaTeX document, write
%%   \input{<filename>.pgf}
%%
%% Make sure the required packages are loaded in your preamble
%%   \usepackage{pgf}
%%
%% Figures using additional raster images can only be included by \input if
%% they are in the same directory as the main LaTeX file. For loading figures
%% from other directories you can use the `import` package
%%   \usepackage{import}
%%
%% and then include the figures with
%%   \import{<path to file>}{<filename>.pgf}
%%
%% Matplotlib used the following preamble
%%   \usepackage{fontspec}
%%   \setmainfont{DejaVuSerif.ttf}[Path=\detokenize{C:/Users/ccros/Anaconda3/Lib/site-packages/matplotlib/mpl-data/fonts/ttf/}]
%%   \setsansfont{DejaVuSans.ttf}[Path=\detokenize{C:/Users/ccros/Anaconda3/Lib/site-packages/matplotlib/mpl-data/fonts/ttf/}]
%%   \setmonofont{DejaVuSansMono.ttf}[Path=\detokenize{C:/Users/ccros/Anaconda3/Lib/site-packages/matplotlib/mpl-data/fonts/ttf/}]
%%
\begingroup%
\makeatletter%
\begin{pgfpicture}%
\pgfpathrectangle{\pgfpointorigin}{\pgfqpoint{6.000000in}{3.600000in}}%
\pgfusepath{use as bounding box, clip}%
\begin{pgfscope}%
\pgfsetbuttcap%
\pgfsetmiterjoin%
\definecolor{currentfill}{rgb}{1.000000,1.000000,1.000000}%
\pgfsetfillcolor{currentfill}%
\pgfsetlinewidth{0.000000pt}%
\definecolor{currentstroke}{rgb}{1.000000,1.000000,1.000000}%
\pgfsetstrokecolor{currentstroke}%
\pgfsetdash{}{0pt}%
\pgfpathmoveto{\pgfqpoint{0.000000in}{0.000000in}}%
\pgfpathlineto{\pgfqpoint{6.000000in}{0.000000in}}%
\pgfpathlineto{\pgfqpoint{6.000000in}{3.600000in}}%
\pgfpathlineto{\pgfqpoint{0.000000in}{3.600000in}}%
\pgfpathclose%
\pgfusepath{fill}%
\end{pgfscope}%
\begin{pgfscope}%
\pgfpathrectangle{\pgfqpoint{0.000000in}{0.000000in}}{\pgfqpoint{6.000000in}{3.600000in}}%
\pgfusepath{clip}%
\pgfsetrectcap%
\pgfsetroundjoin%
\pgfsetlinewidth{3.011250pt}%
\definecolor{currentstroke}{rgb}{0.250980,0.250980,0.250980}%
\pgfsetstrokecolor{currentstroke}%
\pgfsetdash{}{0pt}%
\pgfpathmoveto{\pgfqpoint{3.750000in}{3.150000in}}%
\pgfpathlineto{\pgfqpoint{2.250000in}{3.150000in}}%
\pgfusepath{stroke}%
\end{pgfscope}%
\begin{pgfscope}%
\pgfpathrectangle{\pgfqpoint{0.000000in}{0.000000in}}{\pgfqpoint{6.000000in}{3.600000in}}%
\pgfusepath{clip}%
\pgfsetrectcap%
\pgfsetroundjoin%
\pgfsetlinewidth{3.011250pt}%
\definecolor{currentstroke}{rgb}{0.250980,0.250980,0.250980}%
\pgfsetstrokecolor{currentstroke}%
\pgfsetdash{}{0pt}%
\pgfpathmoveto{\pgfqpoint{0.750000in}{2.250000in}}%
\pgfpathlineto{\pgfqpoint{2.250000in}{3.150000in}}%
\pgfusepath{stroke}%
\end{pgfscope}%
\begin{pgfscope}%
\pgfpathrectangle{\pgfqpoint{0.000000in}{0.000000in}}{\pgfqpoint{6.000000in}{3.600000in}}%
\pgfusepath{clip}%
\pgfsetrectcap%
\pgfsetroundjoin%
\pgfsetlinewidth{3.011250pt}%
\definecolor{currentstroke}{rgb}{0.250980,0.250980,0.250980}%
\pgfsetstrokecolor{currentstroke}%
\pgfsetdash{}{0pt}%
\pgfpathmoveto{\pgfqpoint{5.250000in}{2.250000in}}%
\pgfpathlineto{\pgfqpoint{3.750000in}{3.150000in}}%
\pgfusepath{stroke}%
\end{pgfscope}%
\begin{pgfscope}%
\pgfpathrectangle{\pgfqpoint{0.000000in}{0.000000in}}{\pgfqpoint{6.000000in}{3.600000in}}%
\pgfusepath{clip}%
\pgfsetrectcap%
\pgfsetroundjoin%
\pgfsetlinewidth{3.011250pt}%
\definecolor{currentstroke}{rgb}{0.250980,0.250980,0.250980}%
\pgfsetstrokecolor{currentstroke}%
\pgfsetdash{}{0pt}%
\pgfpathmoveto{\pgfqpoint{2.250000in}{1.800000in}}%
\pgfpathlineto{\pgfqpoint{0.750000in}{2.250000in}}%
\pgfusepath{stroke}%
\end{pgfscope}%
\begin{pgfscope}%
\pgfpathrectangle{\pgfqpoint{0.000000in}{0.000000in}}{\pgfqpoint{6.000000in}{3.600000in}}%
\pgfusepath{clip}%
\pgfsetrectcap%
\pgfsetroundjoin%
\pgfsetlinewidth{3.011250pt}%
\definecolor{currentstroke}{rgb}{0.250980,0.250980,0.250980}%
\pgfsetstrokecolor{currentstroke}%
\pgfsetdash{}{0pt}%
\pgfpathmoveto{\pgfqpoint{3.750000in}{1.800000in}}%
\pgfpathlineto{\pgfqpoint{5.250000in}{2.250000in}}%
\pgfusepath{stroke}%
\end{pgfscope}%
\begin{pgfscope}%
\pgfpathrectangle{\pgfqpoint{0.000000in}{0.000000in}}{\pgfqpoint{6.000000in}{3.600000in}}%
\pgfusepath{clip}%
\pgfsetrectcap%
\pgfsetroundjoin%
\pgfsetlinewidth{3.011250pt}%
\definecolor{currentstroke}{rgb}{0.250980,0.250980,0.250980}%
\pgfsetstrokecolor{currentstroke}%
\pgfsetdash{}{0pt}%
\pgfpathmoveto{\pgfqpoint{0.750000in}{1.350000in}}%
\pgfpathlineto{\pgfqpoint{0.750000in}{2.250000in}}%
\pgfusepath{stroke}%
\end{pgfscope}%
\begin{pgfscope}%
\pgfpathrectangle{\pgfqpoint{0.000000in}{0.000000in}}{\pgfqpoint{6.000000in}{3.600000in}}%
\pgfusepath{clip}%
\pgfsetrectcap%
\pgfsetroundjoin%
\pgfsetlinewidth{3.011250pt}%
\definecolor{currentstroke}{rgb}{0.250980,0.250980,0.250980}%
\pgfsetstrokecolor{currentstroke}%
\pgfsetdash{}{0pt}%
\pgfpathmoveto{\pgfqpoint{5.250000in}{1.350000in}}%
\pgfpathlineto{\pgfqpoint{5.250000in}{2.250000in}}%
\pgfusepath{stroke}%
\end{pgfscope}%
\begin{pgfscope}%
\pgfpathrectangle{\pgfqpoint{0.000000in}{0.000000in}}{\pgfqpoint{6.000000in}{3.600000in}}%
\pgfusepath{clip}%
\pgfsetrectcap%
\pgfsetroundjoin%
\pgfsetlinewidth{3.011250pt}%
\definecolor{currentstroke}{rgb}{0.250980,0.250980,0.250980}%
\pgfsetstrokecolor{currentstroke}%
\pgfsetdash{}{0pt}%
\pgfpathmoveto{\pgfqpoint{2.250000in}{0.450000in}}%
\pgfpathlineto{\pgfqpoint{0.750000in}{1.350000in}}%
\pgfusepath{stroke}%
\end{pgfscope}%
\begin{pgfscope}%
\pgfpathrectangle{\pgfqpoint{0.000000in}{0.000000in}}{\pgfqpoint{6.000000in}{3.600000in}}%
\pgfusepath{clip}%
\pgfsetrectcap%
\pgfsetroundjoin%
\pgfsetlinewidth{3.011250pt}%
\definecolor{currentstroke}{rgb}{0.250980,0.250980,0.250980}%
\pgfsetstrokecolor{currentstroke}%
\pgfsetdash{}{0pt}%
\pgfpathmoveto{\pgfqpoint{3.750000in}{0.450000in}}%
\pgfpathlineto{\pgfqpoint{5.250000in}{1.350000in}}%
\pgfusepath{stroke}%
\end{pgfscope}%
\begin{pgfscope}%
\pgfpathrectangle{\pgfqpoint{0.000000in}{0.000000in}}{\pgfqpoint{6.000000in}{3.600000in}}%
\pgfusepath{clip}%
\pgfsetbuttcap%
\pgfsetroundjoin%
\definecolor{currentfill}{rgb}{0.800000,0.800000,0.800000}%
\pgfsetfillcolor{currentfill}%
\pgfsetlinewidth{1.003750pt}%
\definecolor{currentstroke}{rgb}{0.000000,0.000000,0.000000}%
\pgfsetstrokecolor{currentstroke}%
\pgfsetdash{}{0pt}%
\pgfsys@defobject{currentmarker}{\pgfqpoint{-0.104167in}{-0.104167in}}{\pgfqpoint{0.104167in}{0.104167in}}{%
\pgfpathmoveto{\pgfqpoint{0.000000in}{-0.104167in}}%
\pgfpathcurveto{\pgfqpoint{0.027625in}{-0.104167in}}{\pgfqpoint{0.054123in}{-0.093191in}}{\pgfqpoint{0.073657in}{-0.073657in}}%
\pgfpathcurveto{\pgfqpoint{0.093191in}{-0.054123in}}{\pgfqpoint{0.104167in}{-0.027625in}}{\pgfqpoint{0.104167in}{0.000000in}}%
\pgfpathcurveto{\pgfqpoint{0.104167in}{0.027625in}}{\pgfqpoint{0.093191in}{0.054123in}}{\pgfqpoint{0.073657in}{0.073657in}}%
\pgfpathcurveto{\pgfqpoint{0.054123in}{0.093191in}}{\pgfqpoint{0.027625in}{0.104167in}}{\pgfqpoint{0.000000in}{0.104167in}}%
\pgfpathcurveto{\pgfqpoint{-0.027625in}{0.104167in}}{\pgfqpoint{-0.054123in}{0.093191in}}{\pgfqpoint{-0.073657in}{0.073657in}}%
\pgfpathcurveto{\pgfqpoint{-0.093191in}{0.054123in}}{\pgfqpoint{-0.104167in}{0.027625in}}{\pgfqpoint{-0.104167in}{0.000000in}}%
\pgfpathcurveto{\pgfqpoint{-0.104167in}{-0.027625in}}{\pgfqpoint{-0.093191in}{-0.054123in}}{\pgfqpoint{-0.073657in}{-0.073657in}}%
\pgfpathcurveto{\pgfqpoint{-0.054123in}{-0.093191in}}{\pgfqpoint{-0.027625in}{-0.104167in}}{\pgfqpoint{0.000000in}{-0.104167in}}%
\pgfpathclose%
\pgfusepath{stroke,fill}%
}%
\begin{pgfscope}%
\pgfsys@transformshift{2.250000in}{3.150000in}%
\pgfsys@useobject{currentmarker}{}%
\end{pgfscope}%
\end{pgfscope}%
\begin{pgfscope}%
\pgfpathrectangle{\pgfqpoint{0.000000in}{0.000000in}}{\pgfqpoint{6.000000in}{3.600000in}}%
\pgfusepath{clip}%
\pgfsetbuttcap%
\pgfsetroundjoin%
\definecolor{currentfill}{rgb}{0.800000,0.800000,0.800000}%
\pgfsetfillcolor{currentfill}%
\pgfsetlinewidth{1.003750pt}%
\definecolor{currentstroke}{rgb}{0.000000,0.000000,0.000000}%
\pgfsetstrokecolor{currentstroke}%
\pgfsetdash{}{0pt}%
\pgfsys@defobject{currentmarker}{\pgfqpoint{-0.104167in}{-0.104167in}}{\pgfqpoint{0.104167in}{0.104167in}}{%
\pgfpathmoveto{\pgfqpoint{0.000000in}{-0.104167in}}%
\pgfpathcurveto{\pgfqpoint{0.027625in}{-0.104167in}}{\pgfqpoint{0.054123in}{-0.093191in}}{\pgfqpoint{0.073657in}{-0.073657in}}%
\pgfpathcurveto{\pgfqpoint{0.093191in}{-0.054123in}}{\pgfqpoint{0.104167in}{-0.027625in}}{\pgfqpoint{0.104167in}{0.000000in}}%
\pgfpathcurveto{\pgfqpoint{0.104167in}{0.027625in}}{\pgfqpoint{0.093191in}{0.054123in}}{\pgfqpoint{0.073657in}{0.073657in}}%
\pgfpathcurveto{\pgfqpoint{0.054123in}{0.093191in}}{\pgfqpoint{0.027625in}{0.104167in}}{\pgfqpoint{0.000000in}{0.104167in}}%
\pgfpathcurveto{\pgfqpoint{-0.027625in}{0.104167in}}{\pgfqpoint{-0.054123in}{0.093191in}}{\pgfqpoint{-0.073657in}{0.073657in}}%
\pgfpathcurveto{\pgfqpoint{-0.093191in}{0.054123in}}{\pgfqpoint{-0.104167in}{0.027625in}}{\pgfqpoint{-0.104167in}{0.000000in}}%
\pgfpathcurveto{\pgfqpoint{-0.104167in}{-0.027625in}}{\pgfqpoint{-0.093191in}{-0.054123in}}{\pgfqpoint{-0.073657in}{-0.073657in}}%
\pgfpathcurveto{\pgfqpoint{-0.054123in}{-0.093191in}}{\pgfqpoint{-0.027625in}{-0.104167in}}{\pgfqpoint{0.000000in}{-0.104167in}}%
\pgfpathclose%
\pgfusepath{stroke,fill}%
}%
\begin{pgfscope}%
\pgfsys@transformshift{3.750000in}{3.150000in}%
\pgfsys@useobject{currentmarker}{}%
\end{pgfscope}%
\end{pgfscope}%
\begin{pgfscope}%
\pgfpathrectangle{\pgfqpoint{0.000000in}{0.000000in}}{\pgfqpoint{6.000000in}{3.600000in}}%
\pgfusepath{clip}%
\pgfsetbuttcap%
\pgfsetroundjoin%
\definecolor{currentfill}{rgb}{0.800000,0.800000,0.800000}%
\pgfsetfillcolor{currentfill}%
\pgfsetlinewidth{1.003750pt}%
\definecolor{currentstroke}{rgb}{0.000000,0.000000,0.000000}%
\pgfsetstrokecolor{currentstroke}%
\pgfsetdash{}{0pt}%
\pgfsys@defobject{currentmarker}{\pgfqpoint{-0.104167in}{-0.104167in}}{\pgfqpoint{0.104167in}{0.104167in}}{%
\pgfpathmoveto{\pgfqpoint{0.000000in}{-0.104167in}}%
\pgfpathcurveto{\pgfqpoint{0.027625in}{-0.104167in}}{\pgfqpoint{0.054123in}{-0.093191in}}{\pgfqpoint{0.073657in}{-0.073657in}}%
\pgfpathcurveto{\pgfqpoint{0.093191in}{-0.054123in}}{\pgfqpoint{0.104167in}{-0.027625in}}{\pgfqpoint{0.104167in}{0.000000in}}%
\pgfpathcurveto{\pgfqpoint{0.104167in}{0.027625in}}{\pgfqpoint{0.093191in}{0.054123in}}{\pgfqpoint{0.073657in}{0.073657in}}%
\pgfpathcurveto{\pgfqpoint{0.054123in}{0.093191in}}{\pgfqpoint{0.027625in}{0.104167in}}{\pgfqpoint{0.000000in}{0.104167in}}%
\pgfpathcurveto{\pgfqpoint{-0.027625in}{0.104167in}}{\pgfqpoint{-0.054123in}{0.093191in}}{\pgfqpoint{-0.073657in}{0.073657in}}%
\pgfpathcurveto{\pgfqpoint{-0.093191in}{0.054123in}}{\pgfqpoint{-0.104167in}{0.027625in}}{\pgfqpoint{-0.104167in}{0.000000in}}%
\pgfpathcurveto{\pgfqpoint{-0.104167in}{-0.027625in}}{\pgfqpoint{-0.093191in}{-0.054123in}}{\pgfqpoint{-0.073657in}{-0.073657in}}%
\pgfpathcurveto{\pgfqpoint{-0.054123in}{-0.093191in}}{\pgfqpoint{-0.027625in}{-0.104167in}}{\pgfqpoint{0.000000in}{-0.104167in}}%
\pgfpathclose%
\pgfusepath{stroke,fill}%
}%
\begin{pgfscope}%
\pgfsys@transformshift{0.750000in}{2.250000in}%
\pgfsys@useobject{currentmarker}{}%
\end{pgfscope}%
\end{pgfscope}%
\begin{pgfscope}%
\pgfpathrectangle{\pgfqpoint{0.000000in}{0.000000in}}{\pgfqpoint{6.000000in}{3.600000in}}%
\pgfusepath{clip}%
\pgfsetbuttcap%
\pgfsetroundjoin%
\definecolor{currentfill}{rgb}{0.800000,0.800000,0.800000}%
\pgfsetfillcolor{currentfill}%
\pgfsetlinewidth{1.003750pt}%
\definecolor{currentstroke}{rgb}{0.000000,0.000000,0.000000}%
\pgfsetstrokecolor{currentstroke}%
\pgfsetdash{}{0pt}%
\pgfsys@defobject{currentmarker}{\pgfqpoint{-0.104167in}{-0.104167in}}{\pgfqpoint{0.104167in}{0.104167in}}{%
\pgfpathmoveto{\pgfqpoint{0.000000in}{-0.104167in}}%
\pgfpathcurveto{\pgfqpoint{0.027625in}{-0.104167in}}{\pgfqpoint{0.054123in}{-0.093191in}}{\pgfqpoint{0.073657in}{-0.073657in}}%
\pgfpathcurveto{\pgfqpoint{0.093191in}{-0.054123in}}{\pgfqpoint{0.104167in}{-0.027625in}}{\pgfqpoint{0.104167in}{0.000000in}}%
\pgfpathcurveto{\pgfqpoint{0.104167in}{0.027625in}}{\pgfqpoint{0.093191in}{0.054123in}}{\pgfqpoint{0.073657in}{0.073657in}}%
\pgfpathcurveto{\pgfqpoint{0.054123in}{0.093191in}}{\pgfqpoint{0.027625in}{0.104167in}}{\pgfqpoint{0.000000in}{0.104167in}}%
\pgfpathcurveto{\pgfqpoint{-0.027625in}{0.104167in}}{\pgfqpoint{-0.054123in}{0.093191in}}{\pgfqpoint{-0.073657in}{0.073657in}}%
\pgfpathcurveto{\pgfqpoint{-0.093191in}{0.054123in}}{\pgfqpoint{-0.104167in}{0.027625in}}{\pgfqpoint{-0.104167in}{0.000000in}}%
\pgfpathcurveto{\pgfqpoint{-0.104167in}{-0.027625in}}{\pgfqpoint{-0.093191in}{-0.054123in}}{\pgfqpoint{-0.073657in}{-0.073657in}}%
\pgfpathcurveto{\pgfqpoint{-0.054123in}{-0.093191in}}{\pgfqpoint{-0.027625in}{-0.104167in}}{\pgfqpoint{0.000000in}{-0.104167in}}%
\pgfpathclose%
\pgfusepath{stroke,fill}%
}%
\begin{pgfscope}%
\pgfsys@transformshift{5.250000in}{2.250000in}%
\pgfsys@useobject{currentmarker}{}%
\end{pgfscope}%
\end{pgfscope}%
\begin{pgfscope}%
\pgfpathrectangle{\pgfqpoint{0.000000in}{0.000000in}}{\pgfqpoint{6.000000in}{3.600000in}}%
\pgfusepath{clip}%
\pgfsetbuttcap%
\pgfsetroundjoin%
\definecolor{currentfill}{rgb}{0.800000,0.800000,0.800000}%
\pgfsetfillcolor{currentfill}%
\pgfsetlinewidth{1.003750pt}%
\definecolor{currentstroke}{rgb}{0.000000,0.000000,0.000000}%
\pgfsetstrokecolor{currentstroke}%
\pgfsetdash{}{0pt}%
\pgfsys@defobject{currentmarker}{\pgfqpoint{-0.104167in}{-0.104167in}}{\pgfqpoint{0.104167in}{0.104167in}}{%
\pgfpathmoveto{\pgfqpoint{0.000000in}{-0.104167in}}%
\pgfpathcurveto{\pgfqpoint{0.027625in}{-0.104167in}}{\pgfqpoint{0.054123in}{-0.093191in}}{\pgfqpoint{0.073657in}{-0.073657in}}%
\pgfpathcurveto{\pgfqpoint{0.093191in}{-0.054123in}}{\pgfqpoint{0.104167in}{-0.027625in}}{\pgfqpoint{0.104167in}{0.000000in}}%
\pgfpathcurveto{\pgfqpoint{0.104167in}{0.027625in}}{\pgfqpoint{0.093191in}{0.054123in}}{\pgfqpoint{0.073657in}{0.073657in}}%
\pgfpathcurveto{\pgfqpoint{0.054123in}{0.093191in}}{\pgfqpoint{0.027625in}{0.104167in}}{\pgfqpoint{0.000000in}{0.104167in}}%
\pgfpathcurveto{\pgfqpoint{-0.027625in}{0.104167in}}{\pgfqpoint{-0.054123in}{0.093191in}}{\pgfqpoint{-0.073657in}{0.073657in}}%
\pgfpathcurveto{\pgfqpoint{-0.093191in}{0.054123in}}{\pgfqpoint{-0.104167in}{0.027625in}}{\pgfqpoint{-0.104167in}{0.000000in}}%
\pgfpathcurveto{\pgfqpoint{-0.104167in}{-0.027625in}}{\pgfqpoint{-0.093191in}{-0.054123in}}{\pgfqpoint{-0.073657in}{-0.073657in}}%
\pgfpathcurveto{\pgfqpoint{-0.054123in}{-0.093191in}}{\pgfqpoint{-0.027625in}{-0.104167in}}{\pgfqpoint{0.000000in}{-0.104167in}}%
\pgfpathclose%
\pgfusepath{stroke,fill}%
}%
\begin{pgfscope}%
\pgfsys@transformshift{2.250000in}{1.800000in}%
\pgfsys@useobject{currentmarker}{}%
\end{pgfscope}%
\end{pgfscope}%
\begin{pgfscope}%
\pgfpathrectangle{\pgfqpoint{0.000000in}{0.000000in}}{\pgfqpoint{6.000000in}{3.600000in}}%
\pgfusepath{clip}%
\pgfsetbuttcap%
\pgfsetroundjoin%
\definecolor{currentfill}{rgb}{0.800000,0.800000,0.800000}%
\pgfsetfillcolor{currentfill}%
\pgfsetlinewidth{1.003750pt}%
\definecolor{currentstroke}{rgb}{0.000000,0.000000,0.000000}%
\pgfsetstrokecolor{currentstroke}%
\pgfsetdash{}{0pt}%
\pgfsys@defobject{currentmarker}{\pgfqpoint{-0.104167in}{-0.104167in}}{\pgfqpoint{0.104167in}{0.104167in}}{%
\pgfpathmoveto{\pgfqpoint{0.000000in}{-0.104167in}}%
\pgfpathcurveto{\pgfqpoint{0.027625in}{-0.104167in}}{\pgfqpoint{0.054123in}{-0.093191in}}{\pgfqpoint{0.073657in}{-0.073657in}}%
\pgfpathcurveto{\pgfqpoint{0.093191in}{-0.054123in}}{\pgfqpoint{0.104167in}{-0.027625in}}{\pgfqpoint{0.104167in}{0.000000in}}%
\pgfpathcurveto{\pgfqpoint{0.104167in}{0.027625in}}{\pgfqpoint{0.093191in}{0.054123in}}{\pgfqpoint{0.073657in}{0.073657in}}%
\pgfpathcurveto{\pgfqpoint{0.054123in}{0.093191in}}{\pgfqpoint{0.027625in}{0.104167in}}{\pgfqpoint{0.000000in}{0.104167in}}%
\pgfpathcurveto{\pgfqpoint{-0.027625in}{0.104167in}}{\pgfqpoint{-0.054123in}{0.093191in}}{\pgfqpoint{-0.073657in}{0.073657in}}%
\pgfpathcurveto{\pgfqpoint{-0.093191in}{0.054123in}}{\pgfqpoint{-0.104167in}{0.027625in}}{\pgfqpoint{-0.104167in}{0.000000in}}%
\pgfpathcurveto{\pgfqpoint{-0.104167in}{-0.027625in}}{\pgfqpoint{-0.093191in}{-0.054123in}}{\pgfqpoint{-0.073657in}{-0.073657in}}%
\pgfpathcurveto{\pgfqpoint{-0.054123in}{-0.093191in}}{\pgfqpoint{-0.027625in}{-0.104167in}}{\pgfqpoint{0.000000in}{-0.104167in}}%
\pgfpathclose%
\pgfusepath{stroke,fill}%
}%
\begin{pgfscope}%
\pgfsys@transformshift{3.750000in}{1.800000in}%
\pgfsys@useobject{currentmarker}{}%
\end{pgfscope}%
\end{pgfscope}%
\begin{pgfscope}%
\pgfpathrectangle{\pgfqpoint{0.000000in}{0.000000in}}{\pgfqpoint{6.000000in}{3.600000in}}%
\pgfusepath{clip}%
\pgfsetbuttcap%
\pgfsetroundjoin%
\definecolor{currentfill}{rgb}{0.800000,0.800000,0.800000}%
\pgfsetfillcolor{currentfill}%
\pgfsetlinewidth{1.003750pt}%
\definecolor{currentstroke}{rgb}{0.000000,0.000000,0.000000}%
\pgfsetstrokecolor{currentstroke}%
\pgfsetdash{}{0pt}%
\pgfsys@defobject{currentmarker}{\pgfqpoint{-0.104167in}{-0.104167in}}{\pgfqpoint{0.104167in}{0.104167in}}{%
\pgfpathmoveto{\pgfqpoint{0.000000in}{-0.104167in}}%
\pgfpathcurveto{\pgfqpoint{0.027625in}{-0.104167in}}{\pgfqpoint{0.054123in}{-0.093191in}}{\pgfqpoint{0.073657in}{-0.073657in}}%
\pgfpathcurveto{\pgfqpoint{0.093191in}{-0.054123in}}{\pgfqpoint{0.104167in}{-0.027625in}}{\pgfqpoint{0.104167in}{0.000000in}}%
\pgfpathcurveto{\pgfqpoint{0.104167in}{0.027625in}}{\pgfqpoint{0.093191in}{0.054123in}}{\pgfqpoint{0.073657in}{0.073657in}}%
\pgfpathcurveto{\pgfqpoint{0.054123in}{0.093191in}}{\pgfqpoint{0.027625in}{0.104167in}}{\pgfqpoint{0.000000in}{0.104167in}}%
\pgfpathcurveto{\pgfqpoint{-0.027625in}{0.104167in}}{\pgfqpoint{-0.054123in}{0.093191in}}{\pgfqpoint{-0.073657in}{0.073657in}}%
\pgfpathcurveto{\pgfqpoint{-0.093191in}{0.054123in}}{\pgfqpoint{-0.104167in}{0.027625in}}{\pgfqpoint{-0.104167in}{0.000000in}}%
\pgfpathcurveto{\pgfqpoint{-0.104167in}{-0.027625in}}{\pgfqpoint{-0.093191in}{-0.054123in}}{\pgfqpoint{-0.073657in}{-0.073657in}}%
\pgfpathcurveto{\pgfqpoint{-0.054123in}{-0.093191in}}{\pgfqpoint{-0.027625in}{-0.104167in}}{\pgfqpoint{0.000000in}{-0.104167in}}%
\pgfpathclose%
\pgfusepath{stroke,fill}%
}%
\begin{pgfscope}%
\pgfsys@transformshift{0.750000in}{1.350000in}%
\pgfsys@useobject{currentmarker}{}%
\end{pgfscope}%
\end{pgfscope}%
\begin{pgfscope}%
\pgfpathrectangle{\pgfqpoint{0.000000in}{0.000000in}}{\pgfqpoint{6.000000in}{3.600000in}}%
\pgfusepath{clip}%
\pgfsetbuttcap%
\pgfsetroundjoin%
\definecolor{currentfill}{rgb}{0.800000,0.800000,0.800000}%
\pgfsetfillcolor{currentfill}%
\pgfsetlinewidth{1.003750pt}%
\definecolor{currentstroke}{rgb}{0.000000,0.000000,0.000000}%
\pgfsetstrokecolor{currentstroke}%
\pgfsetdash{}{0pt}%
\pgfsys@defobject{currentmarker}{\pgfqpoint{-0.104167in}{-0.104167in}}{\pgfqpoint{0.104167in}{0.104167in}}{%
\pgfpathmoveto{\pgfqpoint{0.000000in}{-0.104167in}}%
\pgfpathcurveto{\pgfqpoint{0.027625in}{-0.104167in}}{\pgfqpoint{0.054123in}{-0.093191in}}{\pgfqpoint{0.073657in}{-0.073657in}}%
\pgfpathcurveto{\pgfqpoint{0.093191in}{-0.054123in}}{\pgfqpoint{0.104167in}{-0.027625in}}{\pgfqpoint{0.104167in}{0.000000in}}%
\pgfpathcurveto{\pgfqpoint{0.104167in}{0.027625in}}{\pgfqpoint{0.093191in}{0.054123in}}{\pgfqpoint{0.073657in}{0.073657in}}%
\pgfpathcurveto{\pgfqpoint{0.054123in}{0.093191in}}{\pgfqpoint{0.027625in}{0.104167in}}{\pgfqpoint{0.000000in}{0.104167in}}%
\pgfpathcurveto{\pgfqpoint{-0.027625in}{0.104167in}}{\pgfqpoint{-0.054123in}{0.093191in}}{\pgfqpoint{-0.073657in}{0.073657in}}%
\pgfpathcurveto{\pgfqpoint{-0.093191in}{0.054123in}}{\pgfqpoint{-0.104167in}{0.027625in}}{\pgfqpoint{-0.104167in}{0.000000in}}%
\pgfpathcurveto{\pgfqpoint{-0.104167in}{-0.027625in}}{\pgfqpoint{-0.093191in}{-0.054123in}}{\pgfqpoint{-0.073657in}{-0.073657in}}%
\pgfpathcurveto{\pgfqpoint{-0.054123in}{-0.093191in}}{\pgfqpoint{-0.027625in}{-0.104167in}}{\pgfqpoint{0.000000in}{-0.104167in}}%
\pgfpathclose%
\pgfusepath{stroke,fill}%
}%
\begin{pgfscope}%
\pgfsys@transformshift{5.250000in}{1.350000in}%
\pgfsys@useobject{currentmarker}{}%
\end{pgfscope}%
\end{pgfscope}%
\begin{pgfscope}%
\pgfpathrectangle{\pgfqpoint{0.000000in}{0.000000in}}{\pgfqpoint{6.000000in}{3.600000in}}%
\pgfusepath{clip}%
\pgfsetbuttcap%
\pgfsetroundjoin%
\definecolor{currentfill}{rgb}{0.800000,0.800000,0.800000}%
\pgfsetfillcolor{currentfill}%
\pgfsetlinewidth{1.003750pt}%
\definecolor{currentstroke}{rgb}{0.000000,0.000000,0.000000}%
\pgfsetstrokecolor{currentstroke}%
\pgfsetdash{}{0pt}%
\pgfsys@defobject{currentmarker}{\pgfqpoint{-0.104167in}{-0.104167in}}{\pgfqpoint{0.104167in}{0.104167in}}{%
\pgfpathmoveto{\pgfqpoint{0.000000in}{-0.104167in}}%
\pgfpathcurveto{\pgfqpoint{0.027625in}{-0.104167in}}{\pgfqpoint{0.054123in}{-0.093191in}}{\pgfqpoint{0.073657in}{-0.073657in}}%
\pgfpathcurveto{\pgfqpoint{0.093191in}{-0.054123in}}{\pgfqpoint{0.104167in}{-0.027625in}}{\pgfqpoint{0.104167in}{0.000000in}}%
\pgfpathcurveto{\pgfqpoint{0.104167in}{0.027625in}}{\pgfqpoint{0.093191in}{0.054123in}}{\pgfqpoint{0.073657in}{0.073657in}}%
\pgfpathcurveto{\pgfqpoint{0.054123in}{0.093191in}}{\pgfqpoint{0.027625in}{0.104167in}}{\pgfqpoint{0.000000in}{0.104167in}}%
\pgfpathcurveto{\pgfqpoint{-0.027625in}{0.104167in}}{\pgfqpoint{-0.054123in}{0.093191in}}{\pgfqpoint{-0.073657in}{0.073657in}}%
\pgfpathcurveto{\pgfqpoint{-0.093191in}{0.054123in}}{\pgfqpoint{-0.104167in}{0.027625in}}{\pgfqpoint{-0.104167in}{0.000000in}}%
\pgfpathcurveto{\pgfqpoint{-0.104167in}{-0.027625in}}{\pgfqpoint{-0.093191in}{-0.054123in}}{\pgfqpoint{-0.073657in}{-0.073657in}}%
\pgfpathcurveto{\pgfqpoint{-0.054123in}{-0.093191in}}{\pgfqpoint{-0.027625in}{-0.104167in}}{\pgfqpoint{0.000000in}{-0.104167in}}%
\pgfpathclose%
\pgfusepath{stroke,fill}%
}%
\begin{pgfscope}%
\pgfsys@transformshift{2.250000in}{0.450000in}%
\pgfsys@useobject{currentmarker}{}%
\end{pgfscope}%
\end{pgfscope}%
\begin{pgfscope}%
\pgfpathrectangle{\pgfqpoint{0.000000in}{0.000000in}}{\pgfqpoint{6.000000in}{3.600000in}}%
\pgfusepath{clip}%
\pgfsetbuttcap%
\pgfsetroundjoin%
\definecolor{currentfill}{rgb}{0.800000,0.800000,0.800000}%
\pgfsetfillcolor{currentfill}%
\pgfsetlinewidth{1.003750pt}%
\definecolor{currentstroke}{rgb}{0.000000,0.000000,0.000000}%
\pgfsetstrokecolor{currentstroke}%
\pgfsetdash{}{0pt}%
\pgfsys@defobject{currentmarker}{\pgfqpoint{-0.104167in}{-0.104167in}}{\pgfqpoint{0.104167in}{0.104167in}}{%
\pgfpathmoveto{\pgfqpoint{0.000000in}{-0.104167in}}%
\pgfpathcurveto{\pgfqpoint{0.027625in}{-0.104167in}}{\pgfqpoint{0.054123in}{-0.093191in}}{\pgfqpoint{0.073657in}{-0.073657in}}%
\pgfpathcurveto{\pgfqpoint{0.093191in}{-0.054123in}}{\pgfqpoint{0.104167in}{-0.027625in}}{\pgfqpoint{0.104167in}{0.000000in}}%
\pgfpathcurveto{\pgfqpoint{0.104167in}{0.027625in}}{\pgfqpoint{0.093191in}{0.054123in}}{\pgfqpoint{0.073657in}{0.073657in}}%
\pgfpathcurveto{\pgfqpoint{0.054123in}{0.093191in}}{\pgfqpoint{0.027625in}{0.104167in}}{\pgfqpoint{0.000000in}{0.104167in}}%
\pgfpathcurveto{\pgfqpoint{-0.027625in}{0.104167in}}{\pgfqpoint{-0.054123in}{0.093191in}}{\pgfqpoint{-0.073657in}{0.073657in}}%
\pgfpathcurveto{\pgfqpoint{-0.093191in}{0.054123in}}{\pgfqpoint{-0.104167in}{0.027625in}}{\pgfqpoint{-0.104167in}{0.000000in}}%
\pgfpathcurveto{\pgfqpoint{-0.104167in}{-0.027625in}}{\pgfqpoint{-0.093191in}{-0.054123in}}{\pgfqpoint{-0.073657in}{-0.073657in}}%
\pgfpathcurveto{\pgfqpoint{-0.054123in}{-0.093191in}}{\pgfqpoint{-0.027625in}{-0.104167in}}{\pgfqpoint{0.000000in}{-0.104167in}}%
\pgfpathclose%
\pgfusepath{stroke,fill}%
}%
\begin{pgfscope}%
\pgfsys@transformshift{3.750000in}{0.450000in}%
\pgfsys@useobject{currentmarker}{}%
\end{pgfscope}%
\end{pgfscope}%
\end{pgfpicture}%
\makeatother%
\endgroup%

%% file: fig/decomposition_G0.pgf
%% Creator: Matplotlib, PGF backend
%%
%% To include the figure in your LaTeX document, write
%%   \input{<filename>.pgf}
%%
%% Make sure the required packages are loaded in your preamble
%%   \usepackage{pgf}
%%
%% Figures using additional raster images can only be included by \input if
%% they are in the same directory as the main LaTeX file. For loading figures
%% from other directories you can use the `import` package
%%   \usepackage{import}
%%
%% and then include the figures with
%%   \import{<path to file>}{<filename>.pgf}
%%
%% Matplotlib used the following preamble
%%   \usepackage{fontspec}
%%   \setmainfont{DejaVuSerif.ttf}[Path=\detokenize{C:/Users/ccros/Anaconda3/Lib/site-packages/matplotlib/mpl-data/fonts/ttf/}]
%%   \setsansfont{DejaVuSans.ttf}[Path=\detokenize{C:/Users/ccros/Anaconda3/Lib/site-packages/matplotlib/mpl-data/fonts/ttf/}]
%%   \setmonofont{DejaVuSansMono.ttf}[Path=\detokenize{C:/Users/ccros/Anaconda3/Lib/site-packages/matplotlib/mpl-data/fonts/ttf/}]
%%
\begingroup%
\makeatletter%
\begin{pgfpicture}%
\pgfpathrectangle{\pgfpointorigin}{\pgfqpoint{6.000000in}{3.600000in}}%
\pgfusepath{use as bounding box, clip}%
\begin{pgfscope}%
\pgfsetbuttcap%
\pgfsetmiterjoin%
\definecolor{currentfill}{rgb}{1.000000,1.000000,1.000000}%
\pgfsetfillcolor{currentfill}%
\pgfsetlinewidth{0.000000pt}%
\definecolor{currentstroke}{rgb}{1.000000,1.000000,1.000000}%
\pgfsetstrokecolor{currentstroke}%
\pgfsetdash{}{0pt}%
\pgfpathmoveto{\pgfqpoint{0.000000in}{0.000000in}}%
\pgfpathlineto{\pgfqpoint{6.000000in}{0.000000in}}%
\pgfpathlineto{\pgfqpoint{6.000000in}{3.600000in}}%
\pgfpathlineto{\pgfqpoint{0.000000in}{3.600000in}}%
\pgfpathclose%
\pgfusepath{fill}%
\end{pgfscope}%
\begin{pgfscope}%
\pgfpathrectangle{\pgfqpoint{0.000000in}{0.000000in}}{\pgfqpoint{6.000000in}{3.600000in}}%
\pgfusepath{clip}%
\pgfsetrectcap%
\pgfsetroundjoin%
\pgfsetlinewidth{3.011250pt}%
\definecolor{currentstroke}{rgb}{0.250980,0.250980,0.250980}%
\pgfsetstrokecolor{currentstroke}%
\pgfsetdash{}{0pt}%
\pgfpathmoveto{\pgfqpoint{3.750000in}{3.150000in}}%
\pgfpathlineto{\pgfqpoint{2.250000in}{3.150000in}}%
\pgfusepath{stroke}%
\end{pgfscope}%
\begin{pgfscope}%
\pgfpathrectangle{\pgfqpoint{0.000000in}{0.000000in}}{\pgfqpoint{6.000000in}{3.600000in}}%
\pgfusepath{clip}%
\pgfsetrectcap%
\pgfsetroundjoin%
\pgfsetlinewidth{3.011250pt}%
\definecolor{currentstroke}{rgb}{0.250980,0.250980,0.250980}%
\pgfsetstrokecolor{currentstroke}%
\pgfsetdash{}{0pt}%
\pgfpathmoveto{\pgfqpoint{0.750000in}{2.250000in}}%
\pgfpathlineto{\pgfqpoint{2.250000in}{3.150000in}}%
\pgfusepath{stroke}%
\end{pgfscope}%
\begin{pgfscope}%
\pgfpathrectangle{\pgfqpoint{0.000000in}{0.000000in}}{\pgfqpoint{6.000000in}{3.600000in}}%
\pgfusepath{clip}%
\pgfsetrectcap%
\pgfsetroundjoin%
\pgfsetlinewidth{3.011250pt}%
\definecolor{currentstroke}{rgb}{0.250980,0.250980,0.250980}%
\pgfsetstrokecolor{currentstroke}%
\pgfsetdash{}{0pt}%
\pgfpathmoveto{\pgfqpoint{5.250000in}{2.250000in}}%
\pgfpathlineto{\pgfqpoint{2.250000in}{3.150000in}}%
\pgfusepath{stroke}%
\end{pgfscope}%
\begin{pgfscope}%
\pgfpathrectangle{\pgfqpoint{0.000000in}{0.000000in}}{\pgfqpoint{6.000000in}{3.600000in}}%
\pgfusepath{clip}%
\pgfsetrectcap%
\pgfsetroundjoin%
\pgfsetlinewidth{3.011250pt}%
\definecolor{currentstroke}{rgb}{0.250980,0.250980,0.250980}%
\pgfsetstrokecolor{currentstroke}%
\pgfsetdash{}{0pt}%
\pgfpathmoveto{\pgfqpoint{5.250000in}{2.250000in}}%
\pgfpathlineto{\pgfqpoint{3.750000in}{3.150000in}}%
\pgfusepath{stroke}%
\end{pgfscope}%
\begin{pgfscope}%
\pgfpathrectangle{\pgfqpoint{0.000000in}{0.000000in}}{\pgfqpoint{6.000000in}{3.600000in}}%
\pgfusepath{clip}%
\pgfsetrectcap%
\pgfsetroundjoin%
\pgfsetlinewidth{3.011250pt}%
\definecolor{currentstroke}{rgb}{0.250980,0.250980,0.250980}%
\pgfsetstrokecolor{currentstroke}%
\pgfsetdash{}{0pt}%
\pgfpathmoveto{\pgfqpoint{2.250000in}{1.800000in}}%
\pgfpathlineto{\pgfqpoint{3.750000in}{3.150000in}}%
\pgfusepath{stroke}%
\end{pgfscope}%
\begin{pgfscope}%
\pgfpathrectangle{\pgfqpoint{0.000000in}{0.000000in}}{\pgfqpoint{6.000000in}{3.600000in}}%
\pgfusepath{clip}%
\pgfsetrectcap%
\pgfsetroundjoin%
\pgfsetlinewidth{3.011250pt}%
\definecolor{currentstroke}{rgb}{0.250980,0.250980,0.250980}%
\pgfsetstrokecolor{currentstroke}%
\pgfsetdash{}{0pt}%
\pgfpathmoveto{\pgfqpoint{2.250000in}{1.800000in}}%
\pgfpathlineto{\pgfqpoint{0.750000in}{2.250000in}}%
\pgfusepath{stroke}%
\end{pgfscope}%
\begin{pgfscope}%
\pgfpathrectangle{\pgfqpoint{0.000000in}{0.000000in}}{\pgfqpoint{6.000000in}{3.600000in}}%
\pgfusepath{clip}%
\pgfsetrectcap%
\pgfsetroundjoin%
\pgfsetlinewidth{3.011250pt}%
\definecolor{currentstroke}{rgb}{0.250980,0.250980,0.250980}%
\pgfsetstrokecolor{currentstroke}%
\pgfsetdash{}{0pt}%
\pgfpathmoveto{\pgfqpoint{2.250000in}{1.800000in}}%
\pgfpathlineto{\pgfqpoint{5.250000in}{2.250000in}}%
\pgfusepath{stroke}%
\end{pgfscope}%
\begin{pgfscope}%
\pgfpathrectangle{\pgfqpoint{0.000000in}{0.000000in}}{\pgfqpoint{6.000000in}{3.600000in}}%
\pgfusepath{clip}%
\pgfsetrectcap%
\pgfsetroundjoin%
\pgfsetlinewidth{3.011250pt}%
\definecolor{currentstroke}{rgb}{0.250980,0.250980,0.250980}%
\pgfsetstrokecolor{currentstroke}%
\pgfsetdash{}{0pt}%
\pgfpathmoveto{\pgfqpoint{3.750000in}{1.800000in}}%
\pgfpathlineto{\pgfqpoint{5.250000in}{2.250000in}}%
\pgfusepath{stroke}%
\end{pgfscope}%
\begin{pgfscope}%
\pgfpathrectangle{\pgfqpoint{0.000000in}{0.000000in}}{\pgfqpoint{6.000000in}{3.600000in}}%
\pgfusepath{clip}%
\pgfsetrectcap%
\pgfsetroundjoin%
\pgfsetlinewidth{3.011250pt}%
\definecolor{currentstroke}{rgb}{0.250980,0.250980,0.250980}%
\pgfsetstrokecolor{currentstroke}%
\pgfsetdash{}{0pt}%
\pgfpathmoveto{\pgfqpoint{0.750000in}{1.350000in}}%
\pgfpathlineto{\pgfqpoint{0.750000in}{2.250000in}}%
\pgfusepath{stroke}%
\end{pgfscope}%
\begin{pgfscope}%
\pgfpathrectangle{\pgfqpoint{0.000000in}{0.000000in}}{\pgfqpoint{6.000000in}{3.600000in}}%
\pgfusepath{clip}%
\pgfsetrectcap%
\pgfsetroundjoin%
\pgfsetlinewidth{3.011250pt}%
\definecolor{currentstroke}{rgb}{0.250980,0.250980,0.250980}%
\pgfsetstrokecolor{currentstroke}%
\pgfsetdash{}{0pt}%
\pgfpathmoveto{\pgfqpoint{0.750000in}{1.350000in}}%
\pgfpathlineto{\pgfqpoint{2.250000in}{1.800000in}}%
\pgfusepath{stroke}%
\end{pgfscope}%
\begin{pgfscope}%
\pgfpathrectangle{\pgfqpoint{0.000000in}{0.000000in}}{\pgfqpoint{6.000000in}{3.600000in}}%
\pgfusepath{clip}%
\pgfsetrectcap%
\pgfsetroundjoin%
\pgfsetlinewidth{3.011250pt}%
\definecolor{currentstroke}{rgb}{0.250980,0.250980,0.250980}%
\pgfsetstrokecolor{currentstroke}%
\pgfsetdash{}{0pt}%
\pgfpathmoveto{\pgfqpoint{5.250000in}{1.350000in}}%
\pgfpathlineto{\pgfqpoint{5.250000in}{2.250000in}}%
\pgfusepath{stroke}%
\end{pgfscope}%
\begin{pgfscope}%
\pgfpathrectangle{\pgfqpoint{0.000000in}{0.000000in}}{\pgfqpoint{6.000000in}{3.600000in}}%
\pgfusepath{clip}%
\pgfsetrectcap%
\pgfsetroundjoin%
\pgfsetlinewidth{3.011250pt}%
\definecolor{currentstroke}{rgb}{0.250980,0.250980,0.250980}%
\pgfsetstrokecolor{currentstroke}%
\pgfsetdash{}{0pt}%
\pgfpathmoveto{\pgfqpoint{5.250000in}{1.350000in}}%
\pgfpathlineto{\pgfqpoint{2.250000in}{1.800000in}}%
\pgfusepath{stroke}%
\end{pgfscope}%
\begin{pgfscope}%
\pgfpathrectangle{\pgfqpoint{0.000000in}{0.000000in}}{\pgfqpoint{6.000000in}{3.600000in}}%
\pgfusepath{clip}%
\pgfsetrectcap%
\pgfsetroundjoin%
\pgfsetlinewidth{3.011250pt}%
\definecolor{currentstroke}{rgb}{0.250980,0.250980,0.250980}%
\pgfsetstrokecolor{currentstroke}%
\pgfsetdash{}{0pt}%
\pgfpathmoveto{\pgfqpoint{5.250000in}{1.350000in}}%
\pgfpathlineto{\pgfqpoint{3.750000in}{1.800000in}}%
\pgfusepath{stroke}%
\end{pgfscope}%
\begin{pgfscope}%
\pgfpathrectangle{\pgfqpoint{0.000000in}{0.000000in}}{\pgfqpoint{6.000000in}{3.600000in}}%
\pgfusepath{clip}%
\pgfsetrectcap%
\pgfsetroundjoin%
\pgfsetlinewidth{3.011250pt}%
\definecolor{currentstroke}{rgb}{0.250980,0.250980,0.250980}%
\pgfsetstrokecolor{currentstroke}%
\pgfsetdash{}{0pt}%
\pgfpathmoveto{\pgfqpoint{2.250000in}{0.450000in}}%
\pgfpathlineto{\pgfqpoint{2.250000in}{1.800000in}}%
\pgfusepath{stroke}%
\end{pgfscope}%
\begin{pgfscope}%
\pgfpathrectangle{\pgfqpoint{0.000000in}{0.000000in}}{\pgfqpoint{6.000000in}{3.600000in}}%
\pgfusepath{clip}%
\pgfsetrectcap%
\pgfsetroundjoin%
\pgfsetlinewidth{3.011250pt}%
\definecolor{currentstroke}{rgb}{0.250980,0.250980,0.250980}%
\pgfsetstrokecolor{currentstroke}%
\pgfsetdash{}{0pt}%
\pgfpathmoveto{\pgfqpoint{3.750000in}{0.450000in}}%
\pgfpathlineto{\pgfqpoint{2.250000in}{1.800000in}}%
\pgfusepath{stroke}%
\end{pgfscope}%
\begin{pgfscope}%
\pgfpathrectangle{\pgfqpoint{0.000000in}{0.000000in}}{\pgfqpoint{6.000000in}{3.600000in}}%
\pgfusepath{clip}%
\pgfsetrectcap%
\pgfsetroundjoin%
\pgfsetlinewidth{3.011250pt}%
\definecolor{currentstroke}{rgb}{0.250980,0.250980,0.250980}%
\pgfsetstrokecolor{currentstroke}%
\pgfsetdash{}{0pt}%
\pgfpathmoveto{\pgfqpoint{3.750000in}{0.450000in}}%
\pgfpathlineto{\pgfqpoint{0.750000in}{1.350000in}}%
\pgfusepath{stroke}%
\end{pgfscope}%
\begin{pgfscope}%
\pgfpathrectangle{\pgfqpoint{0.000000in}{0.000000in}}{\pgfqpoint{6.000000in}{3.600000in}}%
\pgfusepath{clip}%
\pgfsetrectcap%
\pgfsetroundjoin%
\pgfsetlinewidth{3.011250pt}%
\definecolor{currentstroke}{rgb}{0.250980,0.250980,0.250980}%
\pgfsetstrokecolor{currentstroke}%
\pgfsetdash{}{0pt}%
\pgfpathmoveto{\pgfqpoint{3.750000in}{0.450000in}}%
\pgfpathlineto{\pgfqpoint{2.250000in}{0.450000in}}%
\pgfusepath{stroke}%
\end{pgfscope}%
\begin{pgfscope}%
\pgfpathrectangle{\pgfqpoint{0.000000in}{0.000000in}}{\pgfqpoint{6.000000in}{3.600000in}}%
\pgfusepath{clip}%
\pgfsetbuttcap%
\pgfsetroundjoin%
\definecolor{currentfill}{rgb}{0.800000,0.800000,0.800000}%
\pgfsetfillcolor{currentfill}%
\pgfsetlinewidth{1.003750pt}%
\definecolor{currentstroke}{rgb}{0.000000,0.000000,0.000000}%
\pgfsetstrokecolor{currentstroke}%
\pgfsetdash{}{0pt}%
\pgfsys@defobject{currentmarker}{\pgfqpoint{-0.104167in}{-0.104167in}}{\pgfqpoint{0.104167in}{0.104167in}}{%
\pgfpathmoveto{\pgfqpoint{0.000000in}{-0.104167in}}%
\pgfpathcurveto{\pgfqpoint{0.027625in}{-0.104167in}}{\pgfqpoint{0.054123in}{-0.093191in}}{\pgfqpoint{0.073657in}{-0.073657in}}%
\pgfpathcurveto{\pgfqpoint{0.093191in}{-0.054123in}}{\pgfqpoint{0.104167in}{-0.027625in}}{\pgfqpoint{0.104167in}{0.000000in}}%
\pgfpathcurveto{\pgfqpoint{0.104167in}{0.027625in}}{\pgfqpoint{0.093191in}{0.054123in}}{\pgfqpoint{0.073657in}{0.073657in}}%
\pgfpathcurveto{\pgfqpoint{0.054123in}{0.093191in}}{\pgfqpoint{0.027625in}{0.104167in}}{\pgfqpoint{0.000000in}{0.104167in}}%
\pgfpathcurveto{\pgfqpoint{-0.027625in}{0.104167in}}{\pgfqpoint{-0.054123in}{0.093191in}}{\pgfqpoint{-0.073657in}{0.073657in}}%
\pgfpathcurveto{\pgfqpoint{-0.093191in}{0.054123in}}{\pgfqpoint{-0.104167in}{0.027625in}}{\pgfqpoint{-0.104167in}{0.000000in}}%
\pgfpathcurveto{\pgfqpoint{-0.104167in}{-0.027625in}}{\pgfqpoint{-0.093191in}{-0.054123in}}{\pgfqpoint{-0.073657in}{-0.073657in}}%
\pgfpathcurveto{\pgfqpoint{-0.054123in}{-0.093191in}}{\pgfqpoint{-0.027625in}{-0.104167in}}{\pgfqpoint{0.000000in}{-0.104167in}}%
\pgfpathclose%
\pgfusepath{stroke,fill}%
}%
\begin{pgfscope}%
\pgfsys@transformshift{2.250000in}{3.150000in}%
\pgfsys@useobject{currentmarker}{}%
\end{pgfscope}%
\end{pgfscope}%
\begin{pgfscope}%
\pgfpathrectangle{\pgfqpoint{0.000000in}{0.000000in}}{\pgfqpoint{6.000000in}{3.600000in}}%
\pgfusepath{clip}%
\pgfsetbuttcap%
\pgfsetroundjoin%
\definecolor{currentfill}{rgb}{0.800000,0.800000,0.800000}%
\pgfsetfillcolor{currentfill}%
\pgfsetlinewidth{1.003750pt}%
\definecolor{currentstroke}{rgb}{0.000000,0.000000,0.000000}%
\pgfsetstrokecolor{currentstroke}%
\pgfsetdash{}{0pt}%
\pgfsys@defobject{currentmarker}{\pgfqpoint{-0.104167in}{-0.104167in}}{\pgfqpoint{0.104167in}{0.104167in}}{%
\pgfpathmoveto{\pgfqpoint{0.000000in}{-0.104167in}}%
\pgfpathcurveto{\pgfqpoint{0.027625in}{-0.104167in}}{\pgfqpoint{0.054123in}{-0.093191in}}{\pgfqpoint{0.073657in}{-0.073657in}}%
\pgfpathcurveto{\pgfqpoint{0.093191in}{-0.054123in}}{\pgfqpoint{0.104167in}{-0.027625in}}{\pgfqpoint{0.104167in}{0.000000in}}%
\pgfpathcurveto{\pgfqpoint{0.104167in}{0.027625in}}{\pgfqpoint{0.093191in}{0.054123in}}{\pgfqpoint{0.073657in}{0.073657in}}%
\pgfpathcurveto{\pgfqpoint{0.054123in}{0.093191in}}{\pgfqpoint{0.027625in}{0.104167in}}{\pgfqpoint{0.000000in}{0.104167in}}%
\pgfpathcurveto{\pgfqpoint{-0.027625in}{0.104167in}}{\pgfqpoint{-0.054123in}{0.093191in}}{\pgfqpoint{-0.073657in}{0.073657in}}%
\pgfpathcurveto{\pgfqpoint{-0.093191in}{0.054123in}}{\pgfqpoint{-0.104167in}{0.027625in}}{\pgfqpoint{-0.104167in}{0.000000in}}%
\pgfpathcurveto{\pgfqpoint{-0.104167in}{-0.027625in}}{\pgfqpoint{-0.093191in}{-0.054123in}}{\pgfqpoint{-0.073657in}{-0.073657in}}%
\pgfpathcurveto{\pgfqpoint{-0.054123in}{-0.093191in}}{\pgfqpoint{-0.027625in}{-0.104167in}}{\pgfqpoint{0.000000in}{-0.104167in}}%
\pgfpathclose%
\pgfusepath{stroke,fill}%
}%
\begin{pgfscope}%
\pgfsys@transformshift{3.750000in}{3.150000in}%
\pgfsys@useobject{currentmarker}{}%
\end{pgfscope}%
\end{pgfscope}%
\begin{pgfscope}%
\pgfpathrectangle{\pgfqpoint{0.000000in}{0.000000in}}{\pgfqpoint{6.000000in}{3.600000in}}%
\pgfusepath{clip}%
\pgfsetbuttcap%
\pgfsetroundjoin%
\definecolor{currentfill}{rgb}{0.800000,0.800000,0.800000}%
\pgfsetfillcolor{currentfill}%
\pgfsetlinewidth{1.003750pt}%
\definecolor{currentstroke}{rgb}{0.000000,0.000000,0.000000}%
\pgfsetstrokecolor{currentstroke}%
\pgfsetdash{}{0pt}%
\pgfsys@defobject{currentmarker}{\pgfqpoint{-0.104167in}{-0.104167in}}{\pgfqpoint{0.104167in}{0.104167in}}{%
\pgfpathmoveto{\pgfqpoint{0.000000in}{-0.104167in}}%
\pgfpathcurveto{\pgfqpoint{0.027625in}{-0.104167in}}{\pgfqpoint{0.054123in}{-0.093191in}}{\pgfqpoint{0.073657in}{-0.073657in}}%
\pgfpathcurveto{\pgfqpoint{0.093191in}{-0.054123in}}{\pgfqpoint{0.104167in}{-0.027625in}}{\pgfqpoint{0.104167in}{0.000000in}}%
\pgfpathcurveto{\pgfqpoint{0.104167in}{0.027625in}}{\pgfqpoint{0.093191in}{0.054123in}}{\pgfqpoint{0.073657in}{0.073657in}}%
\pgfpathcurveto{\pgfqpoint{0.054123in}{0.093191in}}{\pgfqpoint{0.027625in}{0.104167in}}{\pgfqpoint{0.000000in}{0.104167in}}%
\pgfpathcurveto{\pgfqpoint{-0.027625in}{0.104167in}}{\pgfqpoint{-0.054123in}{0.093191in}}{\pgfqpoint{-0.073657in}{0.073657in}}%
\pgfpathcurveto{\pgfqpoint{-0.093191in}{0.054123in}}{\pgfqpoint{-0.104167in}{0.027625in}}{\pgfqpoint{-0.104167in}{0.000000in}}%
\pgfpathcurveto{\pgfqpoint{-0.104167in}{-0.027625in}}{\pgfqpoint{-0.093191in}{-0.054123in}}{\pgfqpoint{-0.073657in}{-0.073657in}}%
\pgfpathcurveto{\pgfqpoint{-0.054123in}{-0.093191in}}{\pgfqpoint{-0.027625in}{-0.104167in}}{\pgfqpoint{0.000000in}{-0.104167in}}%
\pgfpathclose%
\pgfusepath{stroke,fill}%
}%
\begin{pgfscope}%
\pgfsys@transformshift{0.750000in}{2.250000in}%
\pgfsys@useobject{currentmarker}{}%
\end{pgfscope}%
\end{pgfscope}%
\begin{pgfscope}%
\pgfpathrectangle{\pgfqpoint{0.000000in}{0.000000in}}{\pgfqpoint{6.000000in}{3.600000in}}%
\pgfusepath{clip}%
\pgfsetbuttcap%
\pgfsetroundjoin%
\definecolor{currentfill}{rgb}{0.800000,0.800000,0.800000}%
\pgfsetfillcolor{currentfill}%
\pgfsetlinewidth{1.003750pt}%
\definecolor{currentstroke}{rgb}{0.000000,0.000000,0.000000}%
\pgfsetstrokecolor{currentstroke}%
\pgfsetdash{}{0pt}%
\pgfsys@defobject{currentmarker}{\pgfqpoint{-0.104167in}{-0.104167in}}{\pgfqpoint{0.104167in}{0.104167in}}{%
\pgfpathmoveto{\pgfqpoint{0.000000in}{-0.104167in}}%
\pgfpathcurveto{\pgfqpoint{0.027625in}{-0.104167in}}{\pgfqpoint{0.054123in}{-0.093191in}}{\pgfqpoint{0.073657in}{-0.073657in}}%
\pgfpathcurveto{\pgfqpoint{0.093191in}{-0.054123in}}{\pgfqpoint{0.104167in}{-0.027625in}}{\pgfqpoint{0.104167in}{0.000000in}}%
\pgfpathcurveto{\pgfqpoint{0.104167in}{0.027625in}}{\pgfqpoint{0.093191in}{0.054123in}}{\pgfqpoint{0.073657in}{0.073657in}}%
\pgfpathcurveto{\pgfqpoint{0.054123in}{0.093191in}}{\pgfqpoint{0.027625in}{0.104167in}}{\pgfqpoint{0.000000in}{0.104167in}}%
\pgfpathcurveto{\pgfqpoint{-0.027625in}{0.104167in}}{\pgfqpoint{-0.054123in}{0.093191in}}{\pgfqpoint{-0.073657in}{0.073657in}}%
\pgfpathcurveto{\pgfqpoint{-0.093191in}{0.054123in}}{\pgfqpoint{-0.104167in}{0.027625in}}{\pgfqpoint{-0.104167in}{0.000000in}}%
\pgfpathcurveto{\pgfqpoint{-0.104167in}{-0.027625in}}{\pgfqpoint{-0.093191in}{-0.054123in}}{\pgfqpoint{-0.073657in}{-0.073657in}}%
\pgfpathcurveto{\pgfqpoint{-0.054123in}{-0.093191in}}{\pgfqpoint{-0.027625in}{-0.104167in}}{\pgfqpoint{0.000000in}{-0.104167in}}%
\pgfpathclose%
\pgfusepath{stroke,fill}%
}%
\begin{pgfscope}%
\pgfsys@transformshift{5.250000in}{2.250000in}%
\pgfsys@useobject{currentmarker}{}%
\end{pgfscope}%
\end{pgfscope}%
\begin{pgfscope}%
\pgfpathrectangle{\pgfqpoint{0.000000in}{0.000000in}}{\pgfqpoint{6.000000in}{3.600000in}}%
\pgfusepath{clip}%
\pgfsetbuttcap%
\pgfsetroundjoin%
\definecolor{currentfill}{rgb}{0.800000,0.800000,0.800000}%
\pgfsetfillcolor{currentfill}%
\pgfsetlinewidth{1.003750pt}%
\definecolor{currentstroke}{rgb}{0.000000,0.000000,0.000000}%
\pgfsetstrokecolor{currentstroke}%
\pgfsetdash{}{0pt}%
\pgfsys@defobject{currentmarker}{\pgfqpoint{-0.104167in}{-0.104167in}}{\pgfqpoint{0.104167in}{0.104167in}}{%
\pgfpathmoveto{\pgfqpoint{0.000000in}{-0.104167in}}%
\pgfpathcurveto{\pgfqpoint{0.027625in}{-0.104167in}}{\pgfqpoint{0.054123in}{-0.093191in}}{\pgfqpoint{0.073657in}{-0.073657in}}%
\pgfpathcurveto{\pgfqpoint{0.093191in}{-0.054123in}}{\pgfqpoint{0.104167in}{-0.027625in}}{\pgfqpoint{0.104167in}{0.000000in}}%
\pgfpathcurveto{\pgfqpoint{0.104167in}{0.027625in}}{\pgfqpoint{0.093191in}{0.054123in}}{\pgfqpoint{0.073657in}{0.073657in}}%
\pgfpathcurveto{\pgfqpoint{0.054123in}{0.093191in}}{\pgfqpoint{0.027625in}{0.104167in}}{\pgfqpoint{0.000000in}{0.104167in}}%
\pgfpathcurveto{\pgfqpoint{-0.027625in}{0.104167in}}{\pgfqpoint{-0.054123in}{0.093191in}}{\pgfqpoint{-0.073657in}{0.073657in}}%
\pgfpathcurveto{\pgfqpoint{-0.093191in}{0.054123in}}{\pgfqpoint{-0.104167in}{0.027625in}}{\pgfqpoint{-0.104167in}{0.000000in}}%
\pgfpathcurveto{\pgfqpoint{-0.104167in}{-0.027625in}}{\pgfqpoint{-0.093191in}{-0.054123in}}{\pgfqpoint{-0.073657in}{-0.073657in}}%
\pgfpathcurveto{\pgfqpoint{-0.054123in}{-0.093191in}}{\pgfqpoint{-0.027625in}{-0.104167in}}{\pgfqpoint{0.000000in}{-0.104167in}}%
\pgfpathclose%
\pgfusepath{stroke,fill}%
}%
\begin{pgfscope}%
\pgfsys@transformshift{2.250000in}{1.800000in}%
\pgfsys@useobject{currentmarker}{}%
\end{pgfscope}%
\end{pgfscope}%
\begin{pgfscope}%
\pgfpathrectangle{\pgfqpoint{0.000000in}{0.000000in}}{\pgfqpoint{6.000000in}{3.600000in}}%
\pgfusepath{clip}%
\pgfsetbuttcap%
\pgfsetroundjoin%
\definecolor{currentfill}{rgb}{0.800000,0.800000,0.800000}%
\pgfsetfillcolor{currentfill}%
\pgfsetlinewidth{1.003750pt}%
\definecolor{currentstroke}{rgb}{0.000000,0.000000,0.000000}%
\pgfsetstrokecolor{currentstroke}%
\pgfsetdash{}{0pt}%
\pgfsys@defobject{currentmarker}{\pgfqpoint{-0.104167in}{-0.104167in}}{\pgfqpoint{0.104167in}{0.104167in}}{%
\pgfpathmoveto{\pgfqpoint{0.000000in}{-0.104167in}}%
\pgfpathcurveto{\pgfqpoint{0.027625in}{-0.104167in}}{\pgfqpoint{0.054123in}{-0.093191in}}{\pgfqpoint{0.073657in}{-0.073657in}}%
\pgfpathcurveto{\pgfqpoint{0.093191in}{-0.054123in}}{\pgfqpoint{0.104167in}{-0.027625in}}{\pgfqpoint{0.104167in}{0.000000in}}%
\pgfpathcurveto{\pgfqpoint{0.104167in}{0.027625in}}{\pgfqpoint{0.093191in}{0.054123in}}{\pgfqpoint{0.073657in}{0.073657in}}%
\pgfpathcurveto{\pgfqpoint{0.054123in}{0.093191in}}{\pgfqpoint{0.027625in}{0.104167in}}{\pgfqpoint{0.000000in}{0.104167in}}%
\pgfpathcurveto{\pgfqpoint{-0.027625in}{0.104167in}}{\pgfqpoint{-0.054123in}{0.093191in}}{\pgfqpoint{-0.073657in}{0.073657in}}%
\pgfpathcurveto{\pgfqpoint{-0.093191in}{0.054123in}}{\pgfqpoint{-0.104167in}{0.027625in}}{\pgfqpoint{-0.104167in}{0.000000in}}%
\pgfpathcurveto{\pgfqpoint{-0.104167in}{-0.027625in}}{\pgfqpoint{-0.093191in}{-0.054123in}}{\pgfqpoint{-0.073657in}{-0.073657in}}%
\pgfpathcurveto{\pgfqpoint{-0.054123in}{-0.093191in}}{\pgfqpoint{-0.027625in}{-0.104167in}}{\pgfqpoint{0.000000in}{-0.104167in}}%
\pgfpathclose%
\pgfusepath{stroke,fill}%
}%
\begin{pgfscope}%
\pgfsys@transformshift{3.750000in}{1.800000in}%
\pgfsys@useobject{currentmarker}{}%
\end{pgfscope}%
\end{pgfscope}%
\begin{pgfscope}%
\pgfpathrectangle{\pgfqpoint{0.000000in}{0.000000in}}{\pgfqpoint{6.000000in}{3.600000in}}%
\pgfusepath{clip}%
\pgfsetbuttcap%
\pgfsetroundjoin%
\definecolor{currentfill}{rgb}{0.800000,0.800000,0.800000}%
\pgfsetfillcolor{currentfill}%
\pgfsetlinewidth{1.003750pt}%
\definecolor{currentstroke}{rgb}{0.000000,0.000000,0.000000}%
\pgfsetstrokecolor{currentstroke}%
\pgfsetdash{}{0pt}%
\pgfsys@defobject{currentmarker}{\pgfqpoint{-0.104167in}{-0.104167in}}{\pgfqpoint{0.104167in}{0.104167in}}{%
\pgfpathmoveto{\pgfqpoint{0.000000in}{-0.104167in}}%
\pgfpathcurveto{\pgfqpoint{0.027625in}{-0.104167in}}{\pgfqpoint{0.054123in}{-0.093191in}}{\pgfqpoint{0.073657in}{-0.073657in}}%
\pgfpathcurveto{\pgfqpoint{0.093191in}{-0.054123in}}{\pgfqpoint{0.104167in}{-0.027625in}}{\pgfqpoint{0.104167in}{0.000000in}}%
\pgfpathcurveto{\pgfqpoint{0.104167in}{0.027625in}}{\pgfqpoint{0.093191in}{0.054123in}}{\pgfqpoint{0.073657in}{0.073657in}}%
\pgfpathcurveto{\pgfqpoint{0.054123in}{0.093191in}}{\pgfqpoint{0.027625in}{0.104167in}}{\pgfqpoint{0.000000in}{0.104167in}}%
\pgfpathcurveto{\pgfqpoint{-0.027625in}{0.104167in}}{\pgfqpoint{-0.054123in}{0.093191in}}{\pgfqpoint{-0.073657in}{0.073657in}}%
\pgfpathcurveto{\pgfqpoint{-0.093191in}{0.054123in}}{\pgfqpoint{-0.104167in}{0.027625in}}{\pgfqpoint{-0.104167in}{0.000000in}}%
\pgfpathcurveto{\pgfqpoint{-0.104167in}{-0.027625in}}{\pgfqpoint{-0.093191in}{-0.054123in}}{\pgfqpoint{-0.073657in}{-0.073657in}}%
\pgfpathcurveto{\pgfqpoint{-0.054123in}{-0.093191in}}{\pgfqpoint{-0.027625in}{-0.104167in}}{\pgfqpoint{0.000000in}{-0.104167in}}%
\pgfpathclose%
\pgfusepath{stroke,fill}%
}%
\begin{pgfscope}%
\pgfsys@transformshift{0.750000in}{1.350000in}%
\pgfsys@useobject{currentmarker}{}%
\end{pgfscope}%
\end{pgfscope}%
\begin{pgfscope}%
\pgfpathrectangle{\pgfqpoint{0.000000in}{0.000000in}}{\pgfqpoint{6.000000in}{3.600000in}}%
\pgfusepath{clip}%
\pgfsetbuttcap%
\pgfsetroundjoin%
\definecolor{currentfill}{rgb}{0.800000,0.800000,0.800000}%
\pgfsetfillcolor{currentfill}%
\pgfsetlinewidth{1.003750pt}%
\definecolor{currentstroke}{rgb}{0.000000,0.000000,0.000000}%
\pgfsetstrokecolor{currentstroke}%
\pgfsetdash{}{0pt}%
\pgfsys@defobject{currentmarker}{\pgfqpoint{-0.104167in}{-0.104167in}}{\pgfqpoint{0.104167in}{0.104167in}}{%
\pgfpathmoveto{\pgfqpoint{0.000000in}{-0.104167in}}%
\pgfpathcurveto{\pgfqpoint{0.027625in}{-0.104167in}}{\pgfqpoint{0.054123in}{-0.093191in}}{\pgfqpoint{0.073657in}{-0.073657in}}%
\pgfpathcurveto{\pgfqpoint{0.093191in}{-0.054123in}}{\pgfqpoint{0.104167in}{-0.027625in}}{\pgfqpoint{0.104167in}{0.000000in}}%
\pgfpathcurveto{\pgfqpoint{0.104167in}{0.027625in}}{\pgfqpoint{0.093191in}{0.054123in}}{\pgfqpoint{0.073657in}{0.073657in}}%
\pgfpathcurveto{\pgfqpoint{0.054123in}{0.093191in}}{\pgfqpoint{0.027625in}{0.104167in}}{\pgfqpoint{0.000000in}{0.104167in}}%
\pgfpathcurveto{\pgfqpoint{-0.027625in}{0.104167in}}{\pgfqpoint{-0.054123in}{0.093191in}}{\pgfqpoint{-0.073657in}{0.073657in}}%
\pgfpathcurveto{\pgfqpoint{-0.093191in}{0.054123in}}{\pgfqpoint{-0.104167in}{0.027625in}}{\pgfqpoint{-0.104167in}{0.000000in}}%
\pgfpathcurveto{\pgfqpoint{-0.104167in}{-0.027625in}}{\pgfqpoint{-0.093191in}{-0.054123in}}{\pgfqpoint{-0.073657in}{-0.073657in}}%
\pgfpathcurveto{\pgfqpoint{-0.054123in}{-0.093191in}}{\pgfqpoint{-0.027625in}{-0.104167in}}{\pgfqpoint{0.000000in}{-0.104167in}}%
\pgfpathclose%
\pgfusepath{stroke,fill}%
}%
\begin{pgfscope}%
\pgfsys@transformshift{5.250000in}{1.350000in}%
\pgfsys@useobject{currentmarker}{}%
\end{pgfscope}%
\end{pgfscope}%
\begin{pgfscope}%
\pgfpathrectangle{\pgfqpoint{0.000000in}{0.000000in}}{\pgfqpoint{6.000000in}{3.600000in}}%
\pgfusepath{clip}%
\pgfsetbuttcap%
\pgfsetroundjoin%
\definecolor{currentfill}{rgb}{0.800000,0.800000,0.800000}%
\pgfsetfillcolor{currentfill}%
\pgfsetlinewidth{1.003750pt}%
\definecolor{currentstroke}{rgb}{0.000000,0.000000,0.000000}%
\pgfsetstrokecolor{currentstroke}%
\pgfsetdash{}{0pt}%
\pgfsys@defobject{currentmarker}{\pgfqpoint{-0.104167in}{-0.104167in}}{\pgfqpoint{0.104167in}{0.104167in}}{%
\pgfpathmoveto{\pgfqpoint{0.000000in}{-0.104167in}}%
\pgfpathcurveto{\pgfqpoint{0.027625in}{-0.104167in}}{\pgfqpoint{0.054123in}{-0.093191in}}{\pgfqpoint{0.073657in}{-0.073657in}}%
\pgfpathcurveto{\pgfqpoint{0.093191in}{-0.054123in}}{\pgfqpoint{0.104167in}{-0.027625in}}{\pgfqpoint{0.104167in}{0.000000in}}%
\pgfpathcurveto{\pgfqpoint{0.104167in}{0.027625in}}{\pgfqpoint{0.093191in}{0.054123in}}{\pgfqpoint{0.073657in}{0.073657in}}%
\pgfpathcurveto{\pgfqpoint{0.054123in}{0.093191in}}{\pgfqpoint{0.027625in}{0.104167in}}{\pgfqpoint{0.000000in}{0.104167in}}%
\pgfpathcurveto{\pgfqpoint{-0.027625in}{0.104167in}}{\pgfqpoint{-0.054123in}{0.093191in}}{\pgfqpoint{-0.073657in}{0.073657in}}%
\pgfpathcurveto{\pgfqpoint{-0.093191in}{0.054123in}}{\pgfqpoint{-0.104167in}{0.027625in}}{\pgfqpoint{-0.104167in}{0.000000in}}%
\pgfpathcurveto{\pgfqpoint{-0.104167in}{-0.027625in}}{\pgfqpoint{-0.093191in}{-0.054123in}}{\pgfqpoint{-0.073657in}{-0.073657in}}%
\pgfpathcurveto{\pgfqpoint{-0.054123in}{-0.093191in}}{\pgfqpoint{-0.027625in}{-0.104167in}}{\pgfqpoint{0.000000in}{-0.104167in}}%
\pgfpathclose%
\pgfusepath{stroke,fill}%
}%
\begin{pgfscope}%
\pgfsys@transformshift{2.250000in}{0.450000in}%
\pgfsys@useobject{currentmarker}{}%
\end{pgfscope}%
\end{pgfscope}%
\begin{pgfscope}%
\pgfpathrectangle{\pgfqpoint{0.000000in}{0.000000in}}{\pgfqpoint{6.000000in}{3.600000in}}%
\pgfusepath{clip}%
\pgfsetbuttcap%
\pgfsetroundjoin%
\definecolor{currentfill}{rgb}{0.800000,0.800000,0.800000}%
\pgfsetfillcolor{currentfill}%
\pgfsetlinewidth{1.003750pt}%
\definecolor{currentstroke}{rgb}{0.000000,0.000000,0.000000}%
\pgfsetstrokecolor{currentstroke}%
\pgfsetdash{}{0pt}%
\pgfsys@defobject{currentmarker}{\pgfqpoint{-0.104167in}{-0.104167in}}{\pgfqpoint{0.104167in}{0.104167in}}{%
\pgfpathmoveto{\pgfqpoint{0.000000in}{-0.104167in}}%
\pgfpathcurveto{\pgfqpoint{0.027625in}{-0.104167in}}{\pgfqpoint{0.054123in}{-0.093191in}}{\pgfqpoint{0.073657in}{-0.073657in}}%
\pgfpathcurveto{\pgfqpoint{0.093191in}{-0.054123in}}{\pgfqpoint{0.104167in}{-0.027625in}}{\pgfqpoint{0.104167in}{0.000000in}}%
\pgfpathcurveto{\pgfqpoint{0.104167in}{0.027625in}}{\pgfqpoint{0.093191in}{0.054123in}}{\pgfqpoint{0.073657in}{0.073657in}}%
\pgfpathcurveto{\pgfqpoint{0.054123in}{0.093191in}}{\pgfqpoint{0.027625in}{0.104167in}}{\pgfqpoint{0.000000in}{0.104167in}}%
\pgfpathcurveto{\pgfqpoint{-0.027625in}{0.104167in}}{\pgfqpoint{-0.054123in}{0.093191in}}{\pgfqpoint{-0.073657in}{0.073657in}}%
\pgfpathcurveto{\pgfqpoint{-0.093191in}{0.054123in}}{\pgfqpoint{-0.104167in}{0.027625in}}{\pgfqpoint{-0.104167in}{0.000000in}}%
\pgfpathcurveto{\pgfqpoint{-0.104167in}{-0.027625in}}{\pgfqpoint{-0.093191in}{-0.054123in}}{\pgfqpoint{-0.073657in}{-0.073657in}}%
\pgfpathcurveto{\pgfqpoint{-0.054123in}{-0.093191in}}{\pgfqpoint{-0.027625in}{-0.104167in}}{\pgfqpoint{0.000000in}{-0.104167in}}%
\pgfpathclose%
\pgfusepath{stroke,fill}%
}%
\begin{pgfscope}%
\pgfsys@transformshift{3.750000in}{0.450000in}%
\pgfsys@useobject{currentmarker}{}%
\end{pgfscope}%
\end{pgfscope}%
\end{pgfpicture}%
\makeatother%
\endgroup%

%% file: fig/decomposition_G1.pgf
%% Creator: Matplotlib, PGF backend
%%
%% To include the figure in your LaTeX document, write
%%   \input{<filename>.pgf}
%%
%% Make sure the required packages are loaded in your preamble
%%   \usepackage{pgf}
%%
%% Figures using additional raster images can only be included by \input if
%% they are in the same directory as the main LaTeX file. For loading figures
%% from other directories you can use the `import` package
%%   \usepackage{import}
%%
%% and then include the figures with
%%   \import{<path to file>}{<filename>.pgf}
%%
%% Matplotlib used the following preamble
%%   \usepackage{fontspec}
%%   \setmainfont{DejaVuSerif.ttf}[Path=\detokenize{C:/Users/ccros/Anaconda3/Lib/site-packages/matplotlib/mpl-data/fonts/ttf/}]
%%   \setsansfont{DejaVuSans.ttf}[Path=\detokenize{C:/Users/ccros/Anaconda3/Lib/site-packages/matplotlib/mpl-data/fonts/ttf/}]
%%   \setmonofont{DejaVuSansMono.ttf}[Path=\detokenize{C:/Users/ccros/Anaconda3/Lib/site-packages/matplotlib/mpl-data/fonts/ttf/}]
%%
\begingroup%
\makeatletter%
\begin{pgfpicture}%
\pgfpathrectangle{\pgfpointorigin}{\pgfqpoint{6.000000in}{3.600000in}}%
\pgfusepath{use as bounding box, clip}%
\begin{pgfscope}%
\pgfsetbuttcap%
\pgfsetmiterjoin%
\definecolor{currentfill}{rgb}{1.000000,1.000000,1.000000}%
\pgfsetfillcolor{currentfill}%
\pgfsetlinewidth{0.000000pt}%
\definecolor{currentstroke}{rgb}{1.000000,1.000000,1.000000}%
\pgfsetstrokecolor{currentstroke}%
\pgfsetdash{}{0pt}%
\pgfpathmoveto{\pgfqpoint{0.000000in}{0.000000in}}%
\pgfpathlineto{\pgfqpoint{6.000000in}{0.000000in}}%
\pgfpathlineto{\pgfqpoint{6.000000in}{3.600000in}}%
\pgfpathlineto{\pgfqpoint{0.000000in}{3.600000in}}%
\pgfpathclose%
\pgfusepath{fill}%
\end{pgfscope}%
\begin{pgfscope}%
\pgfpathrectangle{\pgfqpoint{0.000000in}{0.000000in}}{\pgfqpoint{6.000000in}{3.600000in}}%
\pgfusepath{clip}%
\pgfsetrectcap%
\pgfsetroundjoin%
\pgfsetlinewidth{3.011250pt}%
\definecolor{currentstroke}{rgb}{0.250980,0.250980,0.250980}%
\pgfsetstrokecolor{currentstroke}%
\pgfsetdash{}{0pt}%
\pgfpathmoveto{\pgfqpoint{3.750000in}{3.150000in}}%
\pgfpathlineto{\pgfqpoint{2.250000in}{3.150000in}}%
\pgfusepath{stroke}%
\end{pgfscope}%
\begin{pgfscope}%
\pgfpathrectangle{\pgfqpoint{0.000000in}{0.000000in}}{\pgfqpoint{6.000000in}{3.600000in}}%
\pgfusepath{clip}%
\pgfsetrectcap%
\pgfsetroundjoin%
\pgfsetlinewidth{3.011250pt}%
\definecolor{currentstroke}{rgb}{0.250980,0.250980,0.250980}%
\pgfsetstrokecolor{currentstroke}%
\pgfsetdash{}{0pt}%
\pgfpathmoveto{\pgfqpoint{0.750000in}{2.250000in}}%
\pgfpathlineto{\pgfqpoint{2.250000in}{3.150000in}}%
\pgfusepath{stroke}%
\end{pgfscope}%
\begin{pgfscope}%
\pgfpathrectangle{\pgfqpoint{0.000000in}{0.000000in}}{\pgfqpoint{6.000000in}{3.600000in}}%
\pgfusepath{clip}%
\pgfsetrectcap%
\pgfsetroundjoin%
\pgfsetlinewidth{3.011250pt}%
\definecolor{currentstroke}{rgb}{0.250980,0.250980,0.250980}%
\pgfsetstrokecolor{currentstroke}%
\pgfsetdash{}{0pt}%
\pgfpathmoveto{\pgfqpoint{5.250000in}{2.250000in}}%
\pgfpathlineto{\pgfqpoint{2.250000in}{3.150000in}}%
\pgfusepath{stroke}%
\end{pgfscope}%
\begin{pgfscope}%
\pgfpathrectangle{\pgfqpoint{0.000000in}{0.000000in}}{\pgfqpoint{6.000000in}{3.600000in}}%
\pgfusepath{clip}%
\pgfsetrectcap%
\pgfsetroundjoin%
\pgfsetlinewidth{3.011250pt}%
\definecolor{currentstroke}{rgb}{0.250980,0.250980,0.250980}%
\pgfsetstrokecolor{currentstroke}%
\pgfsetdash{}{0pt}%
\pgfpathmoveto{\pgfqpoint{2.250000in}{1.800000in}}%
\pgfpathlineto{\pgfqpoint{2.250000in}{3.150000in}}%
\pgfusepath{stroke}%
\end{pgfscope}%
\begin{pgfscope}%
\pgfpathrectangle{\pgfqpoint{0.000000in}{0.000000in}}{\pgfqpoint{6.000000in}{3.600000in}}%
\pgfusepath{clip}%
\pgfsetrectcap%
\pgfsetroundjoin%
\pgfsetlinewidth{3.011250pt}%
\definecolor{currentstroke}{rgb}{0.250980,0.250980,0.250980}%
\pgfsetstrokecolor{currentstroke}%
\pgfsetdash{}{0pt}%
\pgfpathmoveto{\pgfqpoint{2.250000in}{1.800000in}}%
\pgfpathlineto{\pgfqpoint{3.750000in}{3.150000in}}%
\pgfusepath{stroke}%
\end{pgfscope}%
\begin{pgfscope}%
\pgfpathrectangle{\pgfqpoint{0.000000in}{0.000000in}}{\pgfqpoint{6.000000in}{3.600000in}}%
\pgfusepath{clip}%
\pgfsetrectcap%
\pgfsetroundjoin%
\pgfsetlinewidth{3.011250pt}%
\definecolor{currentstroke}{rgb}{0.250980,0.250980,0.250980}%
\pgfsetstrokecolor{currentstroke}%
\pgfsetdash{}{0pt}%
\pgfpathmoveto{\pgfqpoint{2.250000in}{1.800000in}}%
\pgfpathlineto{\pgfqpoint{0.750000in}{2.250000in}}%
\pgfusepath{stroke}%
\end{pgfscope}%
\begin{pgfscope}%
\pgfpathrectangle{\pgfqpoint{0.000000in}{0.000000in}}{\pgfqpoint{6.000000in}{3.600000in}}%
\pgfusepath{clip}%
\pgfsetrectcap%
\pgfsetroundjoin%
\pgfsetlinewidth{3.011250pt}%
\definecolor{currentstroke}{rgb}{0.250980,0.250980,0.250980}%
\pgfsetstrokecolor{currentstroke}%
\pgfsetdash{}{0pt}%
\pgfpathmoveto{\pgfqpoint{2.250000in}{1.800000in}}%
\pgfpathlineto{\pgfqpoint{5.250000in}{2.250000in}}%
\pgfusepath{stroke}%
\end{pgfscope}%
\begin{pgfscope}%
\pgfpathrectangle{\pgfqpoint{0.000000in}{0.000000in}}{\pgfqpoint{6.000000in}{3.600000in}}%
\pgfusepath{clip}%
\pgfsetrectcap%
\pgfsetroundjoin%
\pgfsetlinewidth{3.011250pt}%
\definecolor{currentstroke}{rgb}{0.250980,0.250980,0.250980}%
\pgfsetstrokecolor{currentstroke}%
\pgfsetdash{}{0pt}%
\pgfpathmoveto{\pgfqpoint{3.750000in}{1.800000in}}%
\pgfpathlineto{\pgfqpoint{5.250000in}{2.250000in}}%
\pgfusepath{stroke}%
\end{pgfscope}%
\begin{pgfscope}%
\pgfpathrectangle{\pgfqpoint{0.000000in}{0.000000in}}{\pgfqpoint{6.000000in}{3.600000in}}%
\pgfusepath{clip}%
\pgfsetrectcap%
\pgfsetroundjoin%
\pgfsetlinewidth{3.011250pt}%
\definecolor{currentstroke}{rgb}{0.250980,0.250980,0.250980}%
\pgfsetstrokecolor{currentstroke}%
\pgfsetdash{}{0pt}%
\pgfpathmoveto{\pgfqpoint{0.750000in}{1.350000in}}%
\pgfpathlineto{\pgfqpoint{0.750000in}{2.250000in}}%
\pgfusepath{stroke}%
\end{pgfscope}%
\begin{pgfscope}%
\pgfpathrectangle{\pgfqpoint{0.000000in}{0.000000in}}{\pgfqpoint{6.000000in}{3.600000in}}%
\pgfusepath{clip}%
\pgfsetrectcap%
\pgfsetroundjoin%
\pgfsetlinewidth{3.011250pt}%
\definecolor{currentstroke}{rgb}{0.250980,0.250980,0.250980}%
\pgfsetstrokecolor{currentstroke}%
\pgfsetdash{}{0pt}%
\pgfpathmoveto{\pgfqpoint{0.750000in}{1.350000in}}%
\pgfpathlineto{\pgfqpoint{2.250000in}{1.800000in}}%
\pgfusepath{stroke}%
\end{pgfscope}%
\begin{pgfscope}%
\pgfpathrectangle{\pgfqpoint{0.000000in}{0.000000in}}{\pgfqpoint{6.000000in}{3.600000in}}%
\pgfusepath{clip}%
\pgfsetrectcap%
\pgfsetroundjoin%
\pgfsetlinewidth{3.011250pt}%
\definecolor{currentstroke}{rgb}{0.250980,0.250980,0.250980}%
\pgfsetstrokecolor{currentstroke}%
\pgfsetdash{}{0pt}%
\pgfpathmoveto{\pgfqpoint{5.250000in}{1.350000in}}%
\pgfpathlineto{\pgfqpoint{5.250000in}{2.250000in}}%
\pgfusepath{stroke}%
\end{pgfscope}%
\begin{pgfscope}%
\pgfpathrectangle{\pgfqpoint{0.000000in}{0.000000in}}{\pgfqpoint{6.000000in}{3.600000in}}%
\pgfusepath{clip}%
\pgfsetrectcap%
\pgfsetroundjoin%
\pgfsetlinewidth{3.011250pt}%
\definecolor{currentstroke}{rgb}{0.250980,0.250980,0.250980}%
\pgfsetstrokecolor{currentstroke}%
\pgfsetdash{}{0pt}%
\pgfpathmoveto{\pgfqpoint{5.250000in}{1.350000in}}%
\pgfpathlineto{\pgfqpoint{2.250000in}{1.800000in}}%
\pgfusepath{stroke}%
\end{pgfscope}%
\begin{pgfscope}%
\pgfpathrectangle{\pgfqpoint{0.000000in}{0.000000in}}{\pgfqpoint{6.000000in}{3.600000in}}%
\pgfusepath{clip}%
\pgfsetrectcap%
\pgfsetroundjoin%
\pgfsetlinewidth{3.011250pt}%
\definecolor{currentstroke}{rgb}{0.250980,0.250980,0.250980}%
\pgfsetstrokecolor{currentstroke}%
\pgfsetdash{}{0pt}%
\pgfpathmoveto{\pgfqpoint{5.250000in}{1.350000in}}%
\pgfpathlineto{\pgfqpoint{3.750000in}{1.800000in}}%
\pgfusepath{stroke}%
\end{pgfscope}%
\begin{pgfscope}%
\pgfpathrectangle{\pgfqpoint{0.000000in}{0.000000in}}{\pgfqpoint{6.000000in}{3.600000in}}%
\pgfusepath{clip}%
\pgfsetrectcap%
\pgfsetroundjoin%
\pgfsetlinewidth{3.011250pt}%
\definecolor{currentstroke}{rgb}{0.250980,0.250980,0.250980}%
\pgfsetstrokecolor{currentstroke}%
\pgfsetdash{}{0pt}%
\pgfpathmoveto{\pgfqpoint{2.250000in}{0.450000in}}%
\pgfpathlineto{\pgfqpoint{2.250000in}{1.800000in}}%
\pgfusepath{stroke}%
\end{pgfscope}%
\begin{pgfscope}%
\pgfpathrectangle{\pgfqpoint{0.000000in}{0.000000in}}{\pgfqpoint{6.000000in}{3.600000in}}%
\pgfusepath{clip}%
\pgfsetrectcap%
\pgfsetroundjoin%
\pgfsetlinewidth{3.011250pt}%
\definecolor{currentstroke}{rgb}{0.250980,0.250980,0.250980}%
\pgfsetstrokecolor{currentstroke}%
\pgfsetdash{}{0pt}%
\pgfpathmoveto{\pgfqpoint{3.750000in}{0.450000in}}%
\pgfpathlineto{\pgfqpoint{2.250000in}{1.800000in}}%
\pgfusepath{stroke}%
\end{pgfscope}%
\begin{pgfscope}%
\pgfpathrectangle{\pgfqpoint{0.000000in}{0.000000in}}{\pgfqpoint{6.000000in}{3.600000in}}%
\pgfusepath{clip}%
\pgfsetrectcap%
\pgfsetroundjoin%
\pgfsetlinewidth{3.011250pt}%
\definecolor{currentstroke}{rgb}{0.250980,0.250980,0.250980}%
\pgfsetstrokecolor{currentstroke}%
\pgfsetdash{}{0pt}%
\pgfpathmoveto{\pgfqpoint{3.750000in}{0.450000in}}%
\pgfpathlineto{\pgfqpoint{0.750000in}{1.350000in}}%
\pgfusepath{stroke}%
\end{pgfscope}%
\begin{pgfscope}%
\pgfpathrectangle{\pgfqpoint{0.000000in}{0.000000in}}{\pgfqpoint{6.000000in}{3.600000in}}%
\pgfusepath{clip}%
\pgfsetrectcap%
\pgfsetroundjoin%
\pgfsetlinewidth{3.011250pt}%
\definecolor{currentstroke}{rgb}{0.250980,0.250980,0.250980}%
\pgfsetstrokecolor{currentstroke}%
\pgfsetdash{}{0pt}%
\pgfpathmoveto{\pgfqpoint{3.750000in}{0.450000in}}%
\pgfpathlineto{\pgfqpoint{2.250000in}{0.450000in}}%
\pgfusepath{stroke}%
\end{pgfscope}%
\begin{pgfscope}%
\pgfpathrectangle{\pgfqpoint{0.000000in}{0.000000in}}{\pgfqpoint{6.000000in}{3.600000in}}%
\pgfusepath{clip}%
\pgfsetbuttcap%
\pgfsetroundjoin%
\definecolor{currentfill}{rgb}{0.800000,0.800000,0.800000}%
\pgfsetfillcolor{currentfill}%
\pgfsetlinewidth{1.003750pt}%
\definecolor{currentstroke}{rgb}{0.000000,0.000000,0.000000}%
\pgfsetstrokecolor{currentstroke}%
\pgfsetdash{}{0pt}%
\pgfsys@defobject{currentmarker}{\pgfqpoint{-0.104167in}{-0.104167in}}{\pgfqpoint{0.104167in}{0.104167in}}{%
\pgfpathmoveto{\pgfqpoint{0.000000in}{-0.104167in}}%
\pgfpathcurveto{\pgfqpoint{0.027625in}{-0.104167in}}{\pgfqpoint{0.054123in}{-0.093191in}}{\pgfqpoint{0.073657in}{-0.073657in}}%
\pgfpathcurveto{\pgfqpoint{0.093191in}{-0.054123in}}{\pgfqpoint{0.104167in}{-0.027625in}}{\pgfqpoint{0.104167in}{0.000000in}}%
\pgfpathcurveto{\pgfqpoint{0.104167in}{0.027625in}}{\pgfqpoint{0.093191in}{0.054123in}}{\pgfqpoint{0.073657in}{0.073657in}}%
\pgfpathcurveto{\pgfqpoint{0.054123in}{0.093191in}}{\pgfqpoint{0.027625in}{0.104167in}}{\pgfqpoint{0.000000in}{0.104167in}}%
\pgfpathcurveto{\pgfqpoint{-0.027625in}{0.104167in}}{\pgfqpoint{-0.054123in}{0.093191in}}{\pgfqpoint{-0.073657in}{0.073657in}}%
\pgfpathcurveto{\pgfqpoint{-0.093191in}{0.054123in}}{\pgfqpoint{-0.104167in}{0.027625in}}{\pgfqpoint{-0.104167in}{0.000000in}}%
\pgfpathcurveto{\pgfqpoint{-0.104167in}{-0.027625in}}{\pgfqpoint{-0.093191in}{-0.054123in}}{\pgfqpoint{-0.073657in}{-0.073657in}}%
\pgfpathcurveto{\pgfqpoint{-0.054123in}{-0.093191in}}{\pgfqpoint{-0.027625in}{-0.104167in}}{\pgfqpoint{0.000000in}{-0.104167in}}%
\pgfpathclose%
\pgfusepath{stroke,fill}%
}%
\begin{pgfscope}%
\pgfsys@transformshift{2.250000in}{3.150000in}%
\pgfsys@useobject{currentmarker}{}%
\end{pgfscope}%
\end{pgfscope}%
\begin{pgfscope}%
\pgfpathrectangle{\pgfqpoint{0.000000in}{0.000000in}}{\pgfqpoint{6.000000in}{3.600000in}}%
\pgfusepath{clip}%
\pgfsetbuttcap%
\pgfsetroundjoin%
\definecolor{currentfill}{rgb}{0.800000,0.800000,0.800000}%
\pgfsetfillcolor{currentfill}%
\pgfsetlinewidth{1.003750pt}%
\definecolor{currentstroke}{rgb}{0.000000,0.000000,0.000000}%
\pgfsetstrokecolor{currentstroke}%
\pgfsetdash{}{0pt}%
\pgfsys@defobject{currentmarker}{\pgfqpoint{-0.104167in}{-0.104167in}}{\pgfqpoint{0.104167in}{0.104167in}}{%
\pgfpathmoveto{\pgfqpoint{0.000000in}{-0.104167in}}%
\pgfpathcurveto{\pgfqpoint{0.027625in}{-0.104167in}}{\pgfqpoint{0.054123in}{-0.093191in}}{\pgfqpoint{0.073657in}{-0.073657in}}%
\pgfpathcurveto{\pgfqpoint{0.093191in}{-0.054123in}}{\pgfqpoint{0.104167in}{-0.027625in}}{\pgfqpoint{0.104167in}{0.000000in}}%
\pgfpathcurveto{\pgfqpoint{0.104167in}{0.027625in}}{\pgfqpoint{0.093191in}{0.054123in}}{\pgfqpoint{0.073657in}{0.073657in}}%
\pgfpathcurveto{\pgfqpoint{0.054123in}{0.093191in}}{\pgfqpoint{0.027625in}{0.104167in}}{\pgfqpoint{0.000000in}{0.104167in}}%
\pgfpathcurveto{\pgfqpoint{-0.027625in}{0.104167in}}{\pgfqpoint{-0.054123in}{0.093191in}}{\pgfqpoint{-0.073657in}{0.073657in}}%
\pgfpathcurveto{\pgfqpoint{-0.093191in}{0.054123in}}{\pgfqpoint{-0.104167in}{0.027625in}}{\pgfqpoint{-0.104167in}{0.000000in}}%
\pgfpathcurveto{\pgfqpoint{-0.104167in}{-0.027625in}}{\pgfqpoint{-0.093191in}{-0.054123in}}{\pgfqpoint{-0.073657in}{-0.073657in}}%
\pgfpathcurveto{\pgfqpoint{-0.054123in}{-0.093191in}}{\pgfqpoint{-0.027625in}{-0.104167in}}{\pgfqpoint{0.000000in}{-0.104167in}}%
\pgfpathclose%
\pgfusepath{stroke,fill}%
}%
\begin{pgfscope}%
\pgfsys@transformshift{3.750000in}{3.150000in}%
\pgfsys@useobject{currentmarker}{}%
\end{pgfscope}%
\end{pgfscope}%
\begin{pgfscope}%
\pgfpathrectangle{\pgfqpoint{0.000000in}{0.000000in}}{\pgfqpoint{6.000000in}{3.600000in}}%
\pgfusepath{clip}%
\pgfsetbuttcap%
\pgfsetroundjoin%
\definecolor{currentfill}{rgb}{0.800000,0.800000,0.800000}%
\pgfsetfillcolor{currentfill}%
\pgfsetlinewidth{1.003750pt}%
\definecolor{currentstroke}{rgb}{0.000000,0.000000,0.000000}%
\pgfsetstrokecolor{currentstroke}%
\pgfsetdash{}{0pt}%
\pgfsys@defobject{currentmarker}{\pgfqpoint{-0.104167in}{-0.104167in}}{\pgfqpoint{0.104167in}{0.104167in}}{%
\pgfpathmoveto{\pgfqpoint{0.000000in}{-0.104167in}}%
\pgfpathcurveto{\pgfqpoint{0.027625in}{-0.104167in}}{\pgfqpoint{0.054123in}{-0.093191in}}{\pgfqpoint{0.073657in}{-0.073657in}}%
\pgfpathcurveto{\pgfqpoint{0.093191in}{-0.054123in}}{\pgfqpoint{0.104167in}{-0.027625in}}{\pgfqpoint{0.104167in}{0.000000in}}%
\pgfpathcurveto{\pgfqpoint{0.104167in}{0.027625in}}{\pgfqpoint{0.093191in}{0.054123in}}{\pgfqpoint{0.073657in}{0.073657in}}%
\pgfpathcurveto{\pgfqpoint{0.054123in}{0.093191in}}{\pgfqpoint{0.027625in}{0.104167in}}{\pgfqpoint{0.000000in}{0.104167in}}%
\pgfpathcurveto{\pgfqpoint{-0.027625in}{0.104167in}}{\pgfqpoint{-0.054123in}{0.093191in}}{\pgfqpoint{-0.073657in}{0.073657in}}%
\pgfpathcurveto{\pgfqpoint{-0.093191in}{0.054123in}}{\pgfqpoint{-0.104167in}{0.027625in}}{\pgfqpoint{-0.104167in}{0.000000in}}%
\pgfpathcurveto{\pgfqpoint{-0.104167in}{-0.027625in}}{\pgfqpoint{-0.093191in}{-0.054123in}}{\pgfqpoint{-0.073657in}{-0.073657in}}%
\pgfpathcurveto{\pgfqpoint{-0.054123in}{-0.093191in}}{\pgfqpoint{-0.027625in}{-0.104167in}}{\pgfqpoint{0.000000in}{-0.104167in}}%
\pgfpathclose%
\pgfusepath{stroke,fill}%
}%
\begin{pgfscope}%
\pgfsys@transformshift{0.750000in}{2.250000in}%
\pgfsys@useobject{currentmarker}{}%
\end{pgfscope}%
\end{pgfscope}%
\begin{pgfscope}%
\pgfpathrectangle{\pgfqpoint{0.000000in}{0.000000in}}{\pgfqpoint{6.000000in}{3.600000in}}%
\pgfusepath{clip}%
\pgfsetbuttcap%
\pgfsetroundjoin%
\definecolor{currentfill}{rgb}{0.800000,0.800000,0.800000}%
\pgfsetfillcolor{currentfill}%
\pgfsetlinewidth{1.003750pt}%
\definecolor{currentstroke}{rgb}{0.000000,0.000000,0.000000}%
\pgfsetstrokecolor{currentstroke}%
\pgfsetdash{}{0pt}%
\pgfsys@defobject{currentmarker}{\pgfqpoint{-0.104167in}{-0.104167in}}{\pgfqpoint{0.104167in}{0.104167in}}{%
\pgfpathmoveto{\pgfqpoint{0.000000in}{-0.104167in}}%
\pgfpathcurveto{\pgfqpoint{0.027625in}{-0.104167in}}{\pgfqpoint{0.054123in}{-0.093191in}}{\pgfqpoint{0.073657in}{-0.073657in}}%
\pgfpathcurveto{\pgfqpoint{0.093191in}{-0.054123in}}{\pgfqpoint{0.104167in}{-0.027625in}}{\pgfqpoint{0.104167in}{0.000000in}}%
\pgfpathcurveto{\pgfqpoint{0.104167in}{0.027625in}}{\pgfqpoint{0.093191in}{0.054123in}}{\pgfqpoint{0.073657in}{0.073657in}}%
\pgfpathcurveto{\pgfqpoint{0.054123in}{0.093191in}}{\pgfqpoint{0.027625in}{0.104167in}}{\pgfqpoint{0.000000in}{0.104167in}}%
\pgfpathcurveto{\pgfqpoint{-0.027625in}{0.104167in}}{\pgfqpoint{-0.054123in}{0.093191in}}{\pgfqpoint{-0.073657in}{0.073657in}}%
\pgfpathcurveto{\pgfqpoint{-0.093191in}{0.054123in}}{\pgfqpoint{-0.104167in}{0.027625in}}{\pgfqpoint{-0.104167in}{0.000000in}}%
\pgfpathcurveto{\pgfqpoint{-0.104167in}{-0.027625in}}{\pgfqpoint{-0.093191in}{-0.054123in}}{\pgfqpoint{-0.073657in}{-0.073657in}}%
\pgfpathcurveto{\pgfqpoint{-0.054123in}{-0.093191in}}{\pgfqpoint{-0.027625in}{-0.104167in}}{\pgfqpoint{0.000000in}{-0.104167in}}%
\pgfpathclose%
\pgfusepath{stroke,fill}%
}%
\begin{pgfscope}%
\pgfsys@transformshift{5.250000in}{2.250000in}%
\pgfsys@useobject{currentmarker}{}%
\end{pgfscope}%
\end{pgfscope}%
\begin{pgfscope}%
\pgfpathrectangle{\pgfqpoint{0.000000in}{0.000000in}}{\pgfqpoint{6.000000in}{3.600000in}}%
\pgfusepath{clip}%
\pgfsetbuttcap%
\pgfsetroundjoin%
\definecolor{currentfill}{rgb}{0.800000,0.800000,0.800000}%
\pgfsetfillcolor{currentfill}%
\pgfsetlinewidth{1.003750pt}%
\definecolor{currentstroke}{rgb}{0.000000,0.000000,0.000000}%
\pgfsetstrokecolor{currentstroke}%
\pgfsetdash{}{0pt}%
\pgfsys@defobject{currentmarker}{\pgfqpoint{-0.104167in}{-0.104167in}}{\pgfqpoint{0.104167in}{0.104167in}}{%
\pgfpathmoveto{\pgfqpoint{0.000000in}{-0.104167in}}%
\pgfpathcurveto{\pgfqpoint{0.027625in}{-0.104167in}}{\pgfqpoint{0.054123in}{-0.093191in}}{\pgfqpoint{0.073657in}{-0.073657in}}%
\pgfpathcurveto{\pgfqpoint{0.093191in}{-0.054123in}}{\pgfqpoint{0.104167in}{-0.027625in}}{\pgfqpoint{0.104167in}{0.000000in}}%
\pgfpathcurveto{\pgfqpoint{0.104167in}{0.027625in}}{\pgfqpoint{0.093191in}{0.054123in}}{\pgfqpoint{0.073657in}{0.073657in}}%
\pgfpathcurveto{\pgfqpoint{0.054123in}{0.093191in}}{\pgfqpoint{0.027625in}{0.104167in}}{\pgfqpoint{0.000000in}{0.104167in}}%
\pgfpathcurveto{\pgfqpoint{-0.027625in}{0.104167in}}{\pgfqpoint{-0.054123in}{0.093191in}}{\pgfqpoint{-0.073657in}{0.073657in}}%
\pgfpathcurveto{\pgfqpoint{-0.093191in}{0.054123in}}{\pgfqpoint{-0.104167in}{0.027625in}}{\pgfqpoint{-0.104167in}{0.000000in}}%
\pgfpathcurveto{\pgfqpoint{-0.104167in}{-0.027625in}}{\pgfqpoint{-0.093191in}{-0.054123in}}{\pgfqpoint{-0.073657in}{-0.073657in}}%
\pgfpathcurveto{\pgfqpoint{-0.054123in}{-0.093191in}}{\pgfqpoint{-0.027625in}{-0.104167in}}{\pgfqpoint{0.000000in}{-0.104167in}}%
\pgfpathclose%
\pgfusepath{stroke,fill}%
}%
\begin{pgfscope}%
\pgfsys@transformshift{2.250000in}{1.800000in}%
\pgfsys@useobject{currentmarker}{}%
\end{pgfscope}%
\end{pgfscope}%
\begin{pgfscope}%
\pgfpathrectangle{\pgfqpoint{0.000000in}{0.000000in}}{\pgfqpoint{6.000000in}{3.600000in}}%
\pgfusepath{clip}%
\pgfsetbuttcap%
\pgfsetroundjoin%
\definecolor{currentfill}{rgb}{0.800000,0.800000,0.800000}%
\pgfsetfillcolor{currentfill}%
\pgfsetlinewidth{1.003750pt}%
\definecolor{currentstroke}{rgb}{0.000000,0.000000,0.000000}%
\pgfsetstrokecolor{currentstroke}%
\pgfsetdash{}{0pt}%
\pgfsys@defobject{currentmarker}{\pgfqpoint{-0.104167in}{-0.104167in}}{\pgfqpoint{0.104167in}{0.104167in}}{%
\pgfpathmoveto{\pgfqpoint{0.000000in}{-0.104167in}}%
\pgfpathcurveto{\pgfqpoint{0.027625in}{-0.104167in}}{\pgfqpoint{0.054123in}{-0.093191in}}{\pgfqpoint{0.073657in}{-0.073657in}}%
\pgfpathcurveto{\pgfqpoint{0.093191in}{-0.054123in}}{\pgfqpoint{0.104167in}{-0.027625in}}{\pgfqpoint{0.104167in}{0.000000in}}%
\pgfpathcurveto{\pgfqpoint{0.104167in}{0.027625in}}{\pgfqpoint{0.093191in}{0.054123in}}{\pgfqpoint{0.073657in}{0.073657in}}%
\pgfpathcurveto{\pgfqpoint{0.054123in}{0.093191in}}{\pgfqpoint{0.027625in}{0.104167in}}{\pgfqpoint{0.000000in}{0.104167in}}%
\pgfpathcurveto{\pgfqpoint{-0.027625in}{0.104167in}}{\pgfqpoint{-0.054123in}{0.093191in}}{\pgfqpoint{-0.073657in}{0.073657in}}%
\pgfpathcurveto{\pgfqpoint{-0.093191in}{0.054123in}}{\pgfqpoint{-0.104167in}{0.027625in}}{\pgfqpoint{-0.104167in}{0.000000in}}%
\pgfpathcurveto{\pgfqpoint{-0.104167in}{-0.027625in}}{\pgfqpoint{-0.093191in}{-0.054123in}}{\pgfqpoint{-0.073657in}{-0.073657in}}%
\pgfpathcurveto{\pgfqpoint{-0.054123in}{-0.093191in}}{\pgfqpoint{-0.027625in}{-0.104167in}}{\pgfqpoint{0.000000in}{-0.104167in}}%
\pgfpathclose%
\pgfusepath{stroke,fill}%
}%
\begin{pgfscope}%
\pgfsys@transformshift{3.750000in}{1.800000in}%
\pgfsys@useobject{currentmarker}{}%
\end{pgfscope}%
\end{pgfscope}%
\begin{pgfscope}%
\pgfpathrectangle{\pgfqpoint{0.000000in}{0.000000in}}{\pgfqpoint{6.000000in}{3.600000in}}%
\pgfusepath{clip}%
\pgfsetbuttcap%
\pgfsetroundjoin%
\definecolor{currentfill}{rgb}{0.800000,0.800000,0.800000}%
\pgfsetfillcolor{currentfill}%
\pgfsetlinewidth{1.003750pt}%
\definecolor{currentstroke}{rgb}{0.000000,0.000000,0.000000}%
\pgfsetstrokecolor{currentstroke}%
\pgfsetdash{}{0pt}%
\pgfsys@defobject{currentmarker}{\pgfqpoint{-0.104167in}{-0.104167in}}{\pgfqpoint{0.104167in}{0.104167in}}{%
\pgfpathmoveto{\pgfqpoint{0.000000in}{-0.104167in}}%
\pgfpathcurveto{\pgfqpoint{0.027625in}{-0.104167in}}{\pgfqpoint{0.054123in}{-0.093191in}}{\pgfqpoint{0.073657in}{-0.073657in}}%
\pgfpathcurveto{\pgfqpoint{0.093191in}{-0.054123in}}{\pgfqpoint{0.104167in}{-0.027625in}}{\pgfqpoint{0.104167in}{0.000000in}}%
\pgfpathcurveto{\pgfqpoint{0.104167in}{0.027625in}}{\pgfqpoint{0.093191in}{0.054123in}}{\pgfqpoint{0.073657in}{0.073657in}}%
\pgfpathcurveto{\pgfqpoint{0.054123in}{0.093191in}}{\pgfqpoint{0.027625in}{0.104167in}}{\pgfqpoint{0.000000in}{0.104167in}}%
\pgfpathcurveto{\pgfqpoint{-0.027625in}{0.104167in}}{\pgfqpoint{-0.054123in}{0.093191in}}{\pgfqpoint{-0.073657in}{0.073657in}}%
\pgfpathcurveto{\pgfqpoint{-0.093191in}{0.054123in}}{\pgfqpoint{-0.104167in}{0.027625in}}{\pgfqpoint{-0.104167in}{0.000000in}}%
\pgfpathcurveto{\pgfqpoint{-0.104167in}{-0.027625in}}{\pgfqpoint{-0.093191in}{-0.054123in}}{\pgfqpoint{-0.073657in}{-0.073657in}}%
\pgfpathcurveto{\pgfqpoint{-0.054123in}{-0.093191in}}{\pgfqpoint{-0.027625in}{-0.104167in}}{\pgfqpoint{0.000000in}{-0.104167in}}%
\pgfpathclose%
\pgfusepath{stroke,fill}%
}%
\begin{pgfscope}%
\pgfsys@transformshift{0.750000in}{1.350000in}%
\pgfsys@useobject{currentmarker}{}%
\end{pgfscope}%
\end{pgfscope}%
\begin{pgfscope}%
\pgfpathrectangle{\pgfqpoint{0.000000in}{0.000000in}}{\pgfqpoint{6.000000in}{3.600000in}}%
\pgfusepath{clip}%
\pgfsetbuttcap%
\pgfsetroundjoin%
\definecolor{currentfill}{rgb}{0.800000,0.800000,0.800000}%
\pgfsetfillcolor{currentfill}%
\pgfsetlinewidth{1.003750pt}%
\definecolor{currentstroke}{rgb}{0.000000,0.000000,0.000000}%
\pgfsetstrokecolor{currentstroke}%
\pgfsetdash{}{0pt}%
\pgfsys@defobject{currentmarker}{\pgfqpoint{-0.104167in}{-0.104167in}}{\pgfqpoint{0.104167in}{0.104167in}}{%
\pgfpathmoveto{\pgfqpoint{0.000000in}{-0.104167in}}%
\pgfpathcurveto{\pgfqpoint{0.027625in}{-0.104167in}}{\pgfqpoint{0.054123in}{-0.093191in}}{\pgfqpoint{0.073657in}{-0.073657in}}%
\pgfpathcurveto{\pgfqpoint{0.093191in}{-0.054123in}}{\pgfqpoint{0.104167in}{-0.027625in}}{\pgfqpoint{0.104167in}{0.000000in}}%
\pgfpathcurveto{\pgfqpoint{0.104167in}{0.027625in}}{\pgfqpoint{0.093191in}{0.054123in}}{\pgfqpoint{0.073657in}{0.073657in}}%
\pgfpathcurveto{\pgfqpoint{0.054123in}{0.093191in}}{\pgfqpoint{0.027625in}{0.104167in}}{\pgfqpoint{0.000000in}{0.104167in}}%
\pgfpathcurveto{\pgfqpoint{-0.027625in}{0.104167in}}{\pgfqpoint{-0.054123in}{0.093191in}}{\pgfqpoint{-0.073657in}{0.073657in}}%
\pgfpathcurveto{\pgfqpoint{-0.093191in}{0.054123in}}{\pgfqpoint{-0.104167in}{0.027625in}}{\pgfqpoint{-0.104167in}{0.000000in}}%
\pgfpathcurveto{\pgfqpoint{-0.104167in}{-0.027625in}}{\pgfqpoint{-0.093191in}{-0.054123in}}{\pgfqpoint{-0.073657in}{-0.073657in}}%
\pgfpathcurveto{\pgfqpoint{-0.054123in}{-0.093191in}}{\pgfqpoint{-0.027625in}{-0.104167in}}{\pgfqpoint{0.000000in}{-0.104167in}}%
\pgfpathclose%
\pgfusepath{stroke,fill}%
}%
\begin{pgfscope}%
\pgfsys@transformshift{5.250000in}{1.350000in}%
\pgfsys@useobject{currentmarker}{}%
\end{pgfscope}%
\end{pgfscope}%
\begin{pgfscope}%
\pgfpathrectangle{\pgfqpoint{0.000000in}{0.000000in}}{\pgfqpoint{6.000000in}{3.600000in}}%
\pgfusepath{clip}%
\pgfsetbuttcap%
\pgfsetroundjoin%
\definecolor{currentfill}{rgb}{0.800000,0.800000,0.800000}%
\pgfsetfillcolor{currentfill}%
\pgfsetlinewidth{1.003750pt}%
\definecolor{currentstroke}{rgb}{0.000000,0.000000,0.000000}%
\pgfsetstrokecolor{currentstroke}%
\pgfsetdash{}{0pt}%
\pgfsys@defobject{currentmarker}{\pgfqpoint{-0.104167in}{-0.104167in}}{\pgfqpoint{0.104167in}{0.104167in}}{%
\pgfpathmoveto{\pgfqpoint{0.000000in}{-0.104167in}}%
\pgfpathcurveto{\pgfqpoint{0.027625in}{-0.104167in}}{\pgfqpoint{0.054123in}{-0.093191in}}{\pgfqpoint{0.073657in}{-0.073657in}}%
\pgfpathcurveto{\pgfqpoint{0.093191in}{-0.054123in}}{\pgfqpoint{0.104167in}{-0.027625in}}{\pgfqpoint{0.104167in}{0.000000in}}%
\pgfpathcurveto{\pgfqpoint{0.104167in}{0.027625in}}{\pgfqpoint{0.093191in}{0.054123in}}{\pgfqpoint{0.073657in}{0.073657in}}%
\pgfpathcurveto{\pgfqpoint{0.054123in}{0.093191in}}{\pgfqpoint{0.027625in}{0.104167in}}{\pgfqpoint{0.000000in}{0.104167in}}%
\pgfpathcurveto{\pgfqpoint{-0.027625in}{0.104167in}}{\pgfqpoint{-0.054123in}{0.093191in}}{\pgfqpoint{-0.073657in}{0.073657in}}%
\pgfpathcurveto{\pgfqpoint{-0.093191in}{0.054123in}}{\pgfqpoint{-0.104167in}{0.027625in}}{\pgfqpoint{-0.104167in}{0.000000in}}%
\pgfpathcurveto{\pgfqpoint{-0.104167in}{-0.027625in}}{\pgfqpoint{-0.093191in}{-0.054123in}}{\pgfqpoint{-0.073657in}{-0.073657in}}%
\pgfpathcurveto{\pgfqpoint{-0.054123in}{-0.093191in}}{\pgfqpoint{-0.027625in}{-0.104167in}}{\pgfqpoint{0.000000in}{-0.104167in}}%
\pgfpathclose%
\pgfusepath{stroke,fill}%
}%
\begin{pgfscope}%
\pgfsys@transformshift{2.250000in}{0.450000in}%
\pgfsys@useobject{currentmarker}{}%
\end{pgfscope}%
\end{pgfscope}%
\begin{pgfscope}%
\pgfpathrectangle{\pgfqpoint{0.000000in}{0.000000in}}{\pgfqpoint{6.000000in}{3.600000in}}%
\pgfusepath{clip}%
\pgfsetbuttcap%
\pgfsetroundjoin%
\definecolor{currentfill}{rgb}{0.800000,0.800000,0.800000}%
\pgfsetfillcolor{currentfill}%
\pgfsetlinewidth{1.003750pt}%
\definecolor{currentstroke}{rgb}{0.000000,0.000000,0.000000}%
\pgfsetstrokecolor{currentstroke}%
\pgfsetdash{}{0pt}%
\pgfsys@defobject{currentmarker}{\pgfqpoint{-0.104167in}{-0.104167in}}{\pgfqpoint{0.104167in}{0.104167in}}{%
\pgfpathmoveto{\pgfqpoint{0.000000in}{-0.104167in}}%
\pgfpathcurveto{\pgfqpoint{0.027625in}{-0.104167in}}{\pgfqpoint{0.054123in}{-0.093191in}}{\pgfqpoint{0.073657in}{-0.073657in}}%
\pgfpathcurveto{\pgfqpoint{0.093191in}{-0.054123in}}{\pgfqpoint{0.104167in}{-0.027625in}}{\pgfqpoint{0.104167in}{0.000000in}}%
\pgfpathcurveto{\pgfqpoint{0.104167in}{0.027625in}}{\pgfqpoint{0.093191in}{0.054123in}}{\pgfqpoint{0.073657in}{0.073657in}}%
\pgfpathcurveto{\pgfqpoint{0.054123in}{0.093191in}}{\pgfqpoint{0.027625in}{0.104167in}}{\pgfqpoint{0.000000in}{0.104167in}}%
\pgfpathcurveto{\pgfqpoint{-0.027625in}{0.104167in}}{\pgfqpoint{-0.054123in}{0.093191in}}{\pgfqpoint{-0.073657in}{0.073657in}}%
\pgfpathcurveto{\pgfqpoint{-0.093191in}{0.054123in}}{\pgfqpoint{-0.104167in}{0.027625in}}{\pgfqpoint{-0.104167in}{0.000000in}}%
\pgfpathcurveto{\pgfqpoint{-0.104167in}{-0.027625in}}{\pgfqpoint{-0.093191in}{-0.054123in}}{\pgfqpoint{-0.073657in}{-0.073657in}}%
\pgfpathcurveto{\pgfqpoint{-0.054123in}{-0.093191in}}{\pgfqpoint{-0.027625in}{-0.104167in}}{\pgfqpoint{0.000000in}{-0.104167in}}%
\pgfpathclose%
\pgfusepath{stroke,fill}%
}%
\begin{pgfscope}%
\pgfsys@transformshift{3.750000in}{0.450000in}%
\pgfsys@useobject{currentmarker}{}%
\end{pgfscope}%
\end{pgfscope}%
\end{pgfpicture}%
\makeatother%
\endgroup%

%% file: fig/decomposition_G2.pgf
%% Creator: Matplotlib, PGF backend
%%
%% To include the figure in your LaTeX document, write
%%   \input{<filename>.pgf}
%%
%% Make sure the required packages are loaded in your preamble
%%   \usepackage{pgf}
%%
%% Figures using additional raster images can only be included by \input if
%% they are in the same directory as the main LaTeX file. For loading figures
%% from other directories you can use the `import` package
%%   \usepackage{import}
%%
%% and then include the figures with
%%   \import{<path to file>}{<filename>.pgf}
%%
%% Matplotlib used the following preamble
%%   \usepackage{fontspec}
%%   \setmainfont{DejaVuSerif.ttf}[Path=\detokenize{C:/Users/ccros/Anaconda3/Lib/site-packages/matplotlib/mpl-data/fonts/ttf/}]
%%   \setsansfont{DejaVuSans.ttf}[Path=\detokenize{C:/Users/ccros/Anaconda3/Lib/site-packages/matplotlib/mpl-data/fonts/ttf/}]
%%   \setmonofont{DejaVuSansMono.ttf}[Path=\detokenize{C:/Users/ccros/Anaconda3/Lib/site-packages/matplotlib/mpl-data/fonts/ttf/}]
%%
\begingroup%
\makeatletter%
\begin{pgfpicture}%
\pgfpathrectangle{\pgfpointorigin}{\pgfqpoint{6.000000in}{3.600000in}}%
\pgfusepath{use as bounding box, clip}%
\begin{pgfscope}%
\pgfsetbuttcap%
\pgfsetmiterjoin%
\definecolor{currentfill}{rgb}{1.000000,1.000000,1.000000}%
\pgfsetfillcolor{currentfill}%
\pgfsetlinewidth{0.000000pt}%
\definecolor{currentstroke}{rgb}{1.000000,1.000000,1.000000}%
\pgfsetstrokecolor{currentstroke}%
\pgfsetdash{}{0pt}%
\pgfpathmoveto{\pgfqpoint{0.000000in}{0.000000in}}%
\pgfpathlineto{\pgfqpoint{6.000000in}{0.000000in}}%
\pgfpathlineto{\pgfqpoint{6.000000in}{3.600000in}}%
\pgfpathlineto{\pgfqpoint{0.000000in}{3.600000in}}%
\pgfpathclose%
\pgfusepath{fill}%
\end{pgfscope}%
\begin{pgfscope}%
\pgfpathrectangle{\pgfqpoint{0.000000in}{0.000000in}}{\pgfqpoint{6.000000in}{3.600000in}}%
\pgfusepath{clip}%
\pgfsetrectcap%
\pgfsetroundjoin%
\pgfsetlinewidth{3.011250pt}%
\definecolor{currentstroke}{rgb}{0.250980,0.250980,0.250980}%
\pgfsetstrokecolor{currentstroke}%
\pgfsetdash{}{0pt}%
\pgfpathmoveto{\pgfqpoint{3.750000in}{3.150000in}}%
\pgfpathlineto{\pgfqpoint{2.250000in}{3.150000in}}%
\pgfusepath{stroke}%
\end{pgfscope}%
\begin{pgfscope}%
\pgfpathrectangle{\pgfqpoint{0.000000in}{0.000000in}}{\pgfqpoint{6.000000in}{3.600000in}}%
\pgfusepath{clip}%
\pgfsetrectcap%
\pgfsetroundjoin%
\pgfsetlinewidth{3.011250pt}%
\definecolor{currentstroke}{rgb}{0.250980,0.250980,0.250980}%
\pgfsetstrokecolor{currentstroke}%
\pgfsetdash{}{0pt}%
\pgfpathmoveto{\pgfqpoint{0.750000in}{2.250000in}}%
\pgfpathlineto{\pgfqpoint{2.250000in}{3.150000in}}%
\pgfusepath{stroke}%
\end{pgfscope}%
\begin{pgfscope}%
\pgfpathrectangle{\pgfqpoint{0.000000in}{0.000000in}}{\pgfqpoint{6.000000in}{3.600000in}}%
\pgfusepath{clip}%
\pgfsetrectcap%
\pgfsetroundjoin%
\pgfsetlinewidth{3.011250pt}%
\definecolor{currentstroke}{rgb}{0.250980,0.250980,0.250980}%
\pgfsetstrokecolor{currentstroke}%
\pgfsetdash{}{0pt}%
\pgfpathmoveto{\pgfqpoint{5.250000in}{2.250000in}}%
\pgfpathlineto{\pgfqpoint{2.250000in}{3.150000in}}%
\pgfusepath{stroke}%
\end{pgfscope}%
\begin{pgfscope}%
\pgfpathrectangle{\pgfqpoint{0.000000in}{0.000000in}}{\pgfqpoint{6.000000in}{3.600000in}}%
\pgfusepath{clip}%
\pgfsetrectcap%
\pgfsetroundjoin%
\pgfsetlinewidth{3.011250pt}%
\definecolor{currentstroke}{rgb}{0.250980,0.250980,0.250980}%
\pgfsetstrokecolor{currentstroke}%
\pgfsetdash{}{0pt}%
\pgfpathmoveto{\pgfqpoint{2.250000in}{1.800000in}}%
\pgfpathlineto{\pgfqpoint{2.250000in}{3.150000in}}%
\pgfusepath{stroke}%
\end{pgfscope}%
\begin{pgfscope}%
\pgfpathrectangle{\pgfqpoint{0.000000in}{0.000000in}}{\pgfqpoint{6.000000in}{3.600000in}}%
\pgfusepath{clip}%
\pgfsetrectcap%
\pgfsetroundjoin%
\pgfsetlinewidth{3.011250pt}%
\definecolor{currentstroke}{rgb}{0.250980,0.250980,0.250980}%
\pgfsetstrokecolor{currentstroke}%
\pgfsetdash{}{0pt}%
\pgfpathmoveto{\pgfqpoint{2.250000in}{1.800000in}}%
\pgfpathlineto{\pgfqpoint{3.750000in}{3.150000in}}%
\pgfusepath{stroke}%
\end{pgfscope}%
\begin{pgfscope}%
\pgfpathrectangle{\pgfqpoint{0.000000in}{0.000000in}}{\pgfqpoint{6.000000in}{3.600000in}}%
\pgfusepath{clip}%
\pgfsetrectcap%
\pgfsetroundjoin%
\pgfsetlinewidth{3.011250pt}%
\definecolor{currentstroke}{rgb}{0.250980,0.250980,0.250980}%
\pgfsetstrokecolor{currentstroke}%
\pgfsetdash{}{0pt}%
\pgfpathmoveto{\pgfqpoint{2.250000in}{1.800000in}}%
\pgfpathlineto{\pgfqpoint{0.750000in}{2.250000in}}%
\pgfusepath{stroke}%
\end{pgfscope}%
\begin{pgfscope}%
\pgfpathrectangle{\pgfqpoint{0.000000in}{0.000000in}}{\pgfqpoint{6.000000in}{3.600000in}}%
\pgfusepath{clip}%
\pgfsetrectcap%
\pgfsetroundjoin%
\pgfsetlinewidth{3.011250pt}%
\definecolor{currentstroke}{rgb}{0.250980,0.250980,0.250980}%
\pgfsetstrokecolor{currentstroke}%
\pgfsetdash{}{0pt}%
\pgfpathmoveto{\pgfqpoint{2.250000in}{1.800000in}}%
\pgfpathlineto{\pgfqpoint{5.250000in}{2.250000in}}%
\pgfusepath{stroke}%
\end{pgfscope}%
\begin{pgfscope}%
\pgfpathrectangle{\pgfqpoint{0.000000in}{0.000000in}}{\pgfqpoint{6.000000in}{3.600000in}}%
\pgfusepath{clip}%
\pgfsetrectcap%
\pgfsetroundjoin%
\pgfsetlinewidth{3.011250pt}%
\definecolor{currentstroke}{rgb}{0.250980,0.250980,0.250980}%
\pgfsetstrokecolor{currentstroke}%
\pgfsetdash{}{0pt}%
\pgfpathmoveto{\pgfqpoint{3.750000in}{1.800000in}}%
\pgfpathlineto{\pgfqpoint{5.250000in}{2.250000in}}%
\pgfusepath{stroke}%
\end{pgfscope}%
\begin{pgfscope}%
\pgfpathrectangle{\pgfqpoint{0.000000in}{0.000000in}}{\pgfqpoint{6.000000in}{3.600000in}}%
\pgfusepath{clip}%
\pgfsetrectcap%
\pgfsetroundjoin%
\pgfsetlinewidth{3.011250pt}%
\definecolor{currentstroke}{rgb}{0.250980,0.250980,0.250980}%
\pgfsetstrokecolor{currentstroke}%
\pgfsetdash{}{0pt}%
\pgfpathmoveto{\pgfqpoint{3.750000in}{1.800000in}}%
\pgfpathlineto{\pgfqpoint{2.250000in}{1.800000in}}%
\pgfusepath{stroke}%
\end{pgfscope}%
\begin{pgfscope}%
\pgfpathrectangle{\pgfqpoint{0.000000in}{0.000000in}}{\pgfqpoint{6.000000in}{3.600000in}}%
\pgfusepath{clip}%
\pgfsetrectcap%
\pgfsetroundjoin%
\pgfsetlinewidth{3.011250pt}%
\definecolor{currentstroke}{rgb}{0.250980,0.250980,0.250980}%
\pgfsetstrokecolor{currentstroke}%
\pgfsetdash{}{0pt}%
\pgfpathmoveto{\pgfqpoint{0.750000in}{1.350000in}}%
\pgfpathlineto{\pgfqpoint{0.750000in}{2.250000in}}%
\pgfusepath{stroke}%
\end{pgfscope}%
\begin{pgfscope}%
\pgfpathrectangle{\pgfqpoint{0.000000in}{0.000000in}}{\pgfqpoint{6.000000in}{3.600000in}}%
\pgfusepath{clip}%
\pgfsetrectcap%
\pgfsetroundjoin%
\pgfsetlinewidth{3.011250pt}%
\definecolor{currentstroke}{rgb}{0.250980,0.250980,0.250980}%
\pgfsetstrokecolor{currentstroke}%
\pgfsetdash{}{0pt}%
\pgfpathmoveto{\pgfqpoint{0.750000in}{1.350000in}}%
\pgfpathlineto{\pgfqpoint{2.250000in}{1.800000in}}%
\pgfusepath{stroke}%
\end{pgfscope}%
\begin{pgfscope}%
\pgfpathrectangle{\pgfqpoint{0.000000in}{0.000000in}}{\pgfqpoint{6.000000in}{3.600000in}}%
\pgfusepath{clip}%
\pgfsetrectcap%
\pgfsetroundjoin%
\pgfsetlinewidth{3.011250pt}%
\definecolor{currentstroke}{rgb}{0.250980,0.250980,0.250980}%
\pgfsetstrokecolor{currentstroke}%
\pgfsetdash{}{0pt}%
\pgfpathmoveto{\pgfqpoint{5.250000in}{1.350000in}}%
\pgfpathlineto{\pgfqpoint{2.250000in}{1.800000in}}%
\pgfusepath{stroke}%
\end{pgfscope}%
\begin{pgfscope}%
\pgfpathrectangle{\pgfqpoint{0.000000in}{0.000000in}}{\pgfqpoint{6.000000in}{3.600000in}}%
\pgfusepath{clip}%
\pgfsetrectcap%
\pgfsetroundjoin%
\pgfsetlinewidth{3.011250pt}%
\definecolor{currentstroke}{rgb}{0.250980,0.250980,0.250980}%
\pgfsetstrokecolor{currentstroke}%
\pgfsetdash{}{0pt}%
\pgfpathmoveto{\pgfqpoint{5.250000in}{1.350000in}}%
\pgfpathlineto{\pgfqpoint{3.750000in}{1.800000in}}%
\pgfusepath{stroke}%
\end{pgfscope}%
\begin{pgfscope}%
\pgfpathrectangle{\pgfqpoint{0.000000in}{0.000000in}}{\pgfqpoint{6.000000in}{3.600000in}}%
\pgfusepath{clip}%
\pgfsetrectcap%
\pgfsetroundjoin%
\pgfsetlinewidth{3.011250pt}%
\definecolor{currentstroke}{rgb}{0.250980,0.250980,0.250980}%
\pgfsetstrokecolor{currentstroke}%
\pgfsetdash{}{0pt}%
\pgfpathmoveto{\pgfqpoint{2.250000in}{0.450000in}}%
\pgfpathlineto{\pgfqpoint{2.250000in}{1.800000in}}%
\pgfusepath{stroke}%
\end{pgfscope}%
\begin{pgfscope}%
\pgfpathrectangle{\pgfqpoint{0.000000in}{0.000000in}}{\pgfqpoint{6.000000in}{3.600000in}}%
\pgfusepath{clip}%
\pgfsetrectcap%
\pgfsetroundjoin%
\pgfsetlinewidth{3.011250pt}%
\definecolor{currentstroke}{rgb}{0.250980,0.250980,0.250980}%
\pgfsetstrokecolor{currentstroke}%
\pgfsetdash{}{0pt}%
\pgfpathmoveto{\pgfqpoint{3.750000in}{0.450000in}}%
\pgfpathlineto{\pgfqpoint{2.250000in}{1.800000in}}%
\pgfusepath{stroke}%
\end{pgfscope}%
\begin{pgfscope}%
\pgfpathrectangle{\pgfqpoint{0.000000in}{0.000000in}}{\pgfqpoint{6.000000in}{3.600000in}}%
\pgfusepath{clip}%
\pgfsetrectcap%
\pgfsetroundjoin%
\pgfsetlinewidth{3.011250pt}%
\definecolor{currentstroke}{rgb}{0.250980,0.250980,0.250980}%
\pgfsetstrokecolor{currentstroke}%
\pgfsetdash{}{0pt}%
\pgfpathmoveto{\pgfqpoint{3.750000in}{0.450000in}}%
\pgfpathlineto{\pgfqpoint{0.750000in}{1.350000in}}%
\pgfusepath{stroke}%
\end{pgfscope}%
\begin{pgfscope}%
\pgfpathrectangle{\pgfqpoint{0.000000in}{0.000000in}}{\pgfqpoint{6.000000in}{3.600000in}}%
\pgfusepath{clip}%
\pgfsetrectcap%
\pgfsetroundjoin%
\pgfsetlinewidth{3.011250pt}%
\definecolor{currentstroke}{rgb}{0.250980,0.250980,0.250980}%
\pgfsetstrokecolor{currentstroke}%
\pgfsetdash{}{0pt}%
\pgfpathmoveto{\pgfqpoint{3.750000in}{0.450000in}}%
\pgfpathlineto{\pgfqpoint{2.250000in}{0.450000in}}%
\pgfusepath{stroke}%
\end{pgfscope}%
\begin{pgfscope}%
\pgfpathrectangle{\pgfqpoint{0.000000in}{0.000000in}}{\pgfqpoint{6.000000in}{3.600000in}}%
\pgfusepath{clip}%
\pgfsetbuttcap%
\pgfsetroundjoin%
\definecolor{currentfill}{rgb}{0.800000,0.800000,0.800000}%
\pgfsetfillcolor{currentfill}%
\pgfsetlinewidth{1.003750pt}%
\definecolor{currentstroke}{rgb}{0.000000,0.000000,0.000000}%
\pgfsetstrokecolor{currentstroke}%
\pgfsetdash{}{0pt}%
\pgfsys@defobject{currentmarker}{\pgfqpoint{-0.104167in}{-0.104167in}}{\pgfqpoint{0.104167in}{0.104167in}}{%
\pgfpathmoveto{\pgfqpoint{0.000000in}{-0.104167in}}%
\pgfpathcurveto{\pgfqpoint{0.027625in}{-0.104167in}}{\pgfqpoint{0.054123in}{-0.093191in}}{\pgfqpoint{0.073657in}{-0.073657in}}%
\pgfpathcurveto{\pgfqpoint{0.093191in}{-0.054123in}}{\pgfqpoint{0.104167in}{-0.027625in}}{\pgfqpoint{0.104167in}{0.000000in}}%
\pgfpathcurveto{\pgfqpoint{0.104167in}{0.027625in}}{\pgfqpoint{0.093191in}{0.054123in}}{\pgfqpoint{0.073657in}{0.073657in}}%
\pgfpathcurveto{\pgfqpoint{0.054123in}{0.093191in}}{\pgfqpoint{0.027625in}{0.104167in}}{\pgfqpoint{0.000000in}{0.104167in}}%
\pgfpathcurveto{\pgfqpoint{-0.027625in}{0.104167in}}{\pgfqpoint{-0.054123in}{0.093191in}}{\pgfqpoint{-0.073657in}{0.073657in}}%
\pgfpathcurveto{\pgfqpoint{-0.093191in}{0.054123in}}{\pgfqpoint{-0.104167in}{0.027625in}}{\pgfqpoint{-0.104167in}{0.000000in}}%
\pgfpathcurveto{\pgfqpoint{-0.104167in}{-0.027625in}}{\pgfqpoint{-0.093191in}{-0.054123in}}{\pgfqpoint{-0.073657in}{-0.073657in}}%
\pgfpathcurveto{\pgfqpoint{-0.054123in}{-0.093191in}}{\pgfqpoint{-0.027625in}{-0.104167in}}{\pgfqpoint{0.000000in}{-0.104167in}}%
\pgfpathclose%
\pgfusepath{stroke,fill}%
}%
\begin{pgfscope}%
\pgfsys@transformshift{2.250000in}{3.150000in}%
\pgfsys@useobject{currentmarker}{}%
\end{pgfscope}%
\end{pgfscope}%
\begin{pgfscope}%
\pgfpathrectangle{\pgfqpoint{0.000000in}{0.000000in}}{\pgfqpoint{6.000000in}{3.600000in}}%
\pgfusepath{clip}%
\pgfsetbuttcap%
\pgfsetroundjoin%
\definecolor{currentfill}{rgb}{0.800000,0.800000,0.800000}%
\pgfsetfillcolor{currentfill}%
\pgfsetlinewidth{1.003750pt}%
\definecolor{currentstroke}{rgb}{0.000000,0.000000,0.000000}%
\pgfsetstrokecolor{currentstroke}%
\pgfsetdash{}{0pt}%
\pgfsys@defobject{currentmarker}{\pgfqpoint{-0.104167in}{-0.104167in}}{\pgfqpoint{0.104167in}{0.104167in}}{%
\pgfpathmoveto{\pgfqpoint{0.000000in}{-0.104167in}}%
\pgfpathcurveto{\pgfqpoint{0.027625in}{-0.104167in}}{\pgfqpoint{0.054123in}{-0.093191in}}{\pgfqpoint{0.073657in}{-0.073657in}}%
\pgfpathcurveto{\pgfqpoint{0.093191in}{-0.054123in}}{\pgfqpoint{0.104167in}{-0.027625in}}{\pgfqpoint{0.104167in}{0.000000in}}%
\pgfpathcurveto{\pgfqpoint{0.104167in}{0.027625in}}{\pgfqpoint{0.093191in}{0.054123in}}{\pgfqpoint{0.073657in}{0.073657in}}%
\pgfpathcurveto{\pgfqpoint{0.054123in}{0.093191in}}{\pgfqpoint{0.027625in}{0.104167in}}{\pgfqpoint{0.000000in}{0.104167in}}%
\pgfpathcurveto{\pgfqpoint{-0.027625in}{0.104167in}}{\pgfqpoint{-0.054123in}{0.093191in}}{\pgfqpoint{-0.073657in}{0.073657in}}%
\pgfpathcurveto{\pgfqpoint{-0.093191in}{0.054123in}}{\pgfqpoint{-0.104167in}{0.027625in}}{\pgfqpoint{-0.104167in}{0.000000in}}%
\pgfpathcurveto{\pgfqpoint{-0.104167in}{-0.027625in}}{\pgfqpoint{-0.093191in}{-0.054123in}}{\pgfqpoint{-0.073657in}{-0.073657in}}%
\pgfpathcurveto{\pgfqpoint{-0.054123in}{-0.093191in}}{\pgfqpoint{-0.027625in}{-0.104167in}}{\pgfqpoint{0.000000in}{-0.104167in}}%
\pgfpathclose%
\pgfusepath{stroke,fill}%
}%
\begin{pgfscope}%
\pgfsys@transformshift{3.750000in}{3.150000in}%
\pgfsys@useobject{currentmarker}{}%
\end{pgfscope}%
\end{pgfscope}%
\begin{pgfscope}%
\pgfpathrectangle{\pgfqpoint{0.000000in}{0.000000in}}{\pgfqpoint{6.000000in}{3.600000in}}%
\pgfusepath{clip}%
\pgfsetbuttcap%
\pgfsetroundjoin%
\definecolor{currentfill}{rgb}{0.800000,0.800000,0.800000}%
\pgfsetfillcolor{currentfill}%
\pgfsetlinewidth{1.003750pt}%
\definecolor{currentstroke}{rgb}{0.000000,0.000000,0.000000}%
\pgfsetstrokecolor{currentstroke}%
\pgfsetdash{}{0pt}%
\pgfsys@defobject{currentmarker}{\pgfqpoint{-0.104167in}{-0.104167in}}{\pgfqpoint{0.104167in}{0.104167in}}{%
\pgfpathmoveto{\pgfqpoint{0.000000in}{-0.104167in}}%
\pgfpathcurveto{\pgfqpoint{0.027625in}{-0.104167in}}{\pgfqpoint{0.054123in}{-0.093191in}}{\pgfqpoint{0.073657in}{-0.073657in}}%
\pgfpathcurveto{\pgfqpoint{0.093191in}{-0.054123in}}{\pgfqpoint{0.104167in}{-0.027625in}}{\pgfqpoint{0.104167in}{0.000000in}}%
\pgfpathcurveto{\pgfqpoint{0.104167in}{0.027625in}}{\pgfqpoint{0.093191in}{0.054123in}}{\pgfqpoint{0.073657in}{0.073657in}}%
\pgfpathcurveto{\pgfqpoint{0.054123in}{0.093191in}}{\pgfqpoint{0.027625in}{0.104167in}}{\pgfqpoint{0.000000in}{0.104167in}}%
\pgfpathcurveto{\pgfqpoint{-0.027625in}{0.104167in}}{\pgfqpoint{-0.054123in}{0.093191in}}{\pgfqpoint{-0.073657in}{0.073657in}}%
\pgfpathcurveto{\pgfqpoint{-0.093191in}{0.054123in}}{\pgfqpoint{-0.104167in}{0.027625in}}{\pgfqpoint{-0.104167in}{0.000000in}}%
\pgfpathcurveto{\pgfqpoint{-0.104167in}{-0.027625in}}{\pgfqpoint{-0.093191in}{-0.054123in}}{\pgfqpoint{-0.073657in}{-0.073657in}}%
\pgfpathcurveto{\pgfqpoint{-0.054123in}{-0.093191in}}{\pgfqpoint{-0.027625in}{-0.104167in}}{\pgfqpoint{0.000000in}{-0.104167in}}%
\pgfpathclose%
\pgfusepath{stroke,fill}%
}%
\begin{pgfscope}%
\pgfsys@transformshift{0.750000in}{2.250000in}%
\pgfsys@useobject{currentmarker}{}%
\end{pgfscope}%
\end{pgfscope}%
\begin{pgfscope}%
\pgfpathrectangle{\pgfqpoint{0.000000in}{0.000000in}}{\pgfqpoint{6.000000in}{3.600000in}}%
\pgfusepath{clip}%
\pgfsetbuttcap%
\pgfsetroundjoin%
\definecolor{currentfill}{rgb}{0.800000,0.800000,0.800000}%
\pgfsetfillcolor{currentfill}%
\pgfsetlinewidth{1.003750pt}%
\definecolor{currentstroke}{rgb}{0.000000,0.000000,0.000000}%
\pgfsetstrokecolor{currentstroke}%
\pgfsetdash{}{0pt}%
\pgfsys@defobject{currentmarker}{\pgfqpoint{-0.104167in}{-0.104167in}}{\pgfqpoint{0.104167in}{0.104167in}}{%
\pgfpathmoveto{\pgfqpoint{0.000000in}{-0.104167in}}%
\pgfpathcurveto{\pgfqpoint{0.027625in}{-0.104167in}}{\pgfqpoint{0.054123in}{-0.093191in}}{\pgfqpoint{0.073657in}{-0.073657in}}%
\pgfpathcurveto{\pgfqpoint{0.093191in}{-0.054123in}}{\pgfqpoint{0.104167in}{-0.027625in}}{\pgfqpoint{0.104167in}{0.000000in}}%
\pgfpathcurveto{\pgfqpoint{0.104167in}{0.027625in}}{\pgfqpoint{0.093191in}{0.054123in}}{\pgfqpoint{0.073657in}{0.073657in}}%
\pgfpathcurveto{\pgfqpoint{0.054123in}{0.093191in}}{\pgfqpoint{0.027625in}{0.104167in}}{\pgfqpoint{0.000000in}{0.104167in}}%
\pgfpathcurveto{\pgfqpoint{-0.027625in}{0.104167in}}{\pgfqpoint{-0.054123in}{0.093191in}}{\pgfqpoint{-0.073657in}{0.073657in}}%
\pgfpathcurveto{\pgfqpoint{-0.093191in}{0.054123in}}{\pgfqpoint{-0.104167in}{0.027625in}}{\pgfqpoint{-0.104167in}{0.000000in}}%
\pgfpathcurveto{\pgfqpoint{-0.104167in}{-0.027625in}}{\pgfqpoint{-0.093191in}{-0.054123in}}{\pgfqpoint{-0.073657in}{-0.073657in}}%
\pgfpathcurveto{\pgfqpoint{-0.054123in}{-0.093191in}}{\pgfqpoint{-0.027625in}{-0.104167in}}{\pgfqpoint{0.000000in}{-0.104167in}}%
\pgfpathclose%
\pgfusepath{stroke,fill}%
}%
\begin{pgfscope}%
\pgfsys@transformshift{5.250000in}{2.250000in}%
\pgfsys@useobject{currentmarker}{}%
\end{pgfscope}%
\end{pgfscope}%
\begin{pgfscope}%
\pgfpathrectangle{\pgfqpoint{0.000000in}{0.000000in}}{\pgfqpoint{6.000000in}{3.600000in}}%
\pgfusepath{clip}%
\pgfsetbuttcap%
\pgfsetroundjoin%
\definecolor{currentfill}{rgb}{0.800000,0.800000,0.800000}%
\pgfsetfillcolor{currentfill}%
\pgfsetlinewidth{1.003750pt}%
\definecolor{currentstroke}{rgb}{0.000000,0.000000,0.000000}%
\pgfsetstrokecolor{currentstroke}%
\pgfsetdash{}{0pt}%
\pgfsys@defobject{currentmarker}{\pgfqpoint{-0.104167in}{-0.104167in}}{\pgfqpoint{0.104167in}{0.104167in}}{%
\pgfpathmoveto{\pgfqpoint{0.000000in}{-0.104167in}}%
\pgfpathcurveto{\pgfqpoint{0.027625in}{-0.104167in}}{\pgfqpoint{0.054123in}{-0.093191in}}{\pgfqpoint{0.073657in}{-0.073657in}}%
\pgfpathcurveto{\pgfqpoint{0.093191in}{-0.054123in}}{\pgfqpoint{0.104167in}{-0.027625in}}{\pgfqpoint{0.104167in}{0.000000in}}%
\pgfpathcurveto{\pgfqpoint{0.104167in}{0.027625in}}{\pgfqpoint{0.093191in}{0.054123in}}{\pgfqpoint{0.073657in}{0.073657in}}%
\pgfpathcurveto{\pgfqpoint{0.054123in}{0.093191in}}{\pgfqpoint{0.027625in}{0.104167in}}{\pgfqpoint{0.000000in}{0.104167in}}%
\pgfpathcurveto{\pgfqpoint{-0.027625in}{0.104167in}}{\pgfqpoint{-0.054123in}{0.093191in}}{\pgfqpoint{-0.073657in}{0.073657in}}%
\pgfpathcurveto{\pgfqpoint{-0.093191in}{0.054123in}}{\pgfqpoint{-0.104167in}{0.027625in}}{\pgfqpoint{-0.104167in}{0.000000in}}%
\pgfpathcurveto{\pgfqpoint{-0.104167in}{-0.027625in}}{\pgfqpoint{-0.093191in}{-0.054123in}}{\pgfqpoint{-0.073657in}{-0.073657in}}%
\pgfpathcurveto{\pgfqpoint{-0.054123in}{-0.093191in}}{\pgfqpoint{-0.027625in}{-0.104167in}}{\pgfqpoint{0.000000in}{-0.104167in}}%
\pgfpathclose%
\pgfusepath{stroke,fill}%
}%
\begin{pgfscope}%
\pgfsys@transformshift{2.250000in}{1.800000in}%
\pgfsys@useobject{currentmarker}{}%
\end{pgfscope}%
\end{pgfscope}%
\begin{pgfscope}%
\pgfpathrectangle{\pgfqpoint{0.000000in}{0.000000in}}{\pgfqpoint{6.000000in}{3.600000in}}%
\pgfusepath{clip}%
\pgfsetbuttcap%
\pgfsetroundjoin%
\definecolor{currentfill}{rgb}{0.800000,0.800000,0.800000}%
\pgfsetfillcolor{currentfill}%
\pgfsetlinewidth{1.003750pt}%
\definecolor{currentstroke}{rgb}{0.000000,0.000000,0.000000}%
\pgfsetstrokecolor{currentstroke}%
\pgfsetdash{}{0pt}%
\pgfsys@defobject{currentmarker}{\pgfqpoint{-0.104167in}{-0.104167in}}{\pgfqpoint{0.104167in}{0.104167in}}{%
\pgfpathmoveto{\pgfqpoint{0.000000in}{-0.104167in}}%
\pgfpathcurveto{\pgfqpoint{0.027625in}{-0.104167in}}{\pgfqpoint{0.054123in}{-0.093191in}}{\pgfqpoint{0.073657in}{-0.073657in}}%
\pgfpathcurveto{\pgfqpoint{0.093191in}{-0.054123in}}{\pgfqpoint{0.104167in}{-0.027625in}}{\pgfqpoint{0.104167in}{0.000000in}}%
\pgfpathcurveto{\pgfqpoint{0.104167in}{0.027625in}}{\pgfqpoint{0.093191in}{0.054123in}}{\pgfqpoint{0.073657in}{0.073657in}}%
\pgfpathcurveto{\pgfqpoint{0.054123in}{0.093191in}}{\pgfqpoint{0.027625in}{0.104167in}}{\pgfqpoint{0.000000in}{0.104167in}}%
\pgfpathcurveto{\pgfqpoint{-0.027625in}{0.104167in}}{\pgfqpoint{-0.054123in}{0.093191in}}{\pgfqpoint{-0.073657in}{0.073657in}}%
\pgfpathcurveto{\pgfqpoint{-0.093191in}{0.054123in}}{\pgfqpoint{-0.104167in}{0.027625in}}{\pgfqpoint{-0.104167in}{0.000000in}}%
\pgfpathcurveto{\pgfqpoint{-0.104167in}{-0.027625in}}{\pgfqpoint{-0.093191in}{-0.054123in}}{\pgfqpoint{-0.073657in}{-0.073657in}}%
\pgfpathcurveto{\pgfqpoint{-0.054123in}{-0.093191in}}{\pgfqpoint{-0.027625in}{-0.104167in}}{\pgfqpoint{0.000000in}{-0.104167in}}%
\pgfpathclose%
\pgfusepath{stroke,fill}%
}%
\begin{pgfscope}%
\pgfsys@transformshift{3.750000in}{1.800000in}%
\pgfsys@useobject{currentmarker}{}%
\end{pgfscope}%
\end{pgfscope}%
\begin{pgfscope}%
\pgfpathrectangle{\pgfqpoint{0.000000in}{0.000000in}}{\pgfqpoint{6.000000in}{3.600000in}}%
\pgfusepath{clip}%
\pgfsetbuttcap%
\pgfsetroundjoin%
\definecolor{currentfill}{rgb}{0.800000,0.800000,0.800000}%
\pgfsetfillcolor{currentfill}%
\pgfsetlinewidth{1.003750pt}%
\definecolor{currentstroke}{rgb}{0.000000,0.000000,0.000000}%
\pgfsetstrokecolor{currentstroke}%
\pgfsetdash{}{0pt}%
\pgfsys@defobject{currentmarker}{\pgfqpoint{-0.104167in}{-0.104167in}}{\pgfqpoint{0.104167in}{0.104167in}}{%
\pgfpathmoveto{\pgfqpoint{0.000000in}{-0.104167in}}%
\pgfpathcurveto{\pgfqpoint{0.027625in}{-0.104167in}}{\pgfqpoint{0.054123in}{-0.093191in}}{\pgfqpoint{0.073657in}{-0.073657in}}%
\pgfpathcurveto{\pgfqpoint{0.093191in}{-0.054123in}}{\pgfqpoint{0.104167in}{-0.027625in}}{\pgfqpoint{0.104167in}{0.000000in}}%
\pgfpathcurveto{\pgfqpoint{0.104167in}{0.027625in}}{\pgfqpoint{0.093191in}{0.054123in}}{\pgfqpoint{0.073657in}{0.073657in}}%
\pgfpathcurveto{\pgfqpoint{0.054123in}{0.093191in}}{\pgfqpoint{0.027625in}{0.104167in}}{\pgfqpoint{0.000000in}{0.104167in}}%
\pgfpathcurveto{\pgfqpoint{-0.027625in}{0.104167in}}{\pgfqpoint{-0.054123in}{0.093191in}}{\pgfqpoint{-0.073657in}{0.073657in}}%
\pgfpathcurveto{\pgfqpoint{-0.093191in}{0.054123in}}{\pgfqpoint{-0.104167in}{0.027625in}}{\pgfqpoint{-0.104167in}{0.000000in}}%
\pgfpathcurveto{\pgfqpoint{-0.104167in}{-0.027625in}}{\pgfqpoint{-0.093191in}{-0.054123in}}{\pgfqpoint{-0.073657in}{-0.073657in}}%
\pgfpathcurveto{\pgfqpoint{-0.054123in}{-0.093191in}}{\pgfqpoint{-0.027625in}{-0.104167in}}{\pgfqpoint{0.000000in}{-0.104167in}}%
\pgfpathclose%
\pgfusepath{stroke,fill}%
}%
\begin{pgfscope}%
\pgfsys@transformshift{0.750000in}{1.350000in}%
\pgfsys@useobject{currentmarker}{}%
\end{pgfscope}%
\end{pgfscope}%
\begin{pgfscope}%
\pgfpathrectangle{\pgfqpoint{0.000000in}{0.000000in}}{\pgfqpoint{6.000000in}{3.600000in}}%
\pgfusepath{clip}%
\pgfsetbuttcap%
\pgfsetroundjoin%
\definecolor{currentfill}{rgb}{0.800000,0.800000,0.800000}%
\pgfsetfillcolor{currentfill}%
\pgfsetlinewidth{1.003750pt}%
\definecolor{currentstroke}{rgb}{0.000000,0.000000,0.000000}%
\pgfsetstrokecolor{currentstroke}%
\pgfsetdash{}{0pt}%
\pgfsys@defobject{currentmarker}{\pgfqpoint{-0.104167in}{-0.104167in}}{\pgfqpoint{0.104167in}{0.104167in}}{%
\pgfpathmoveto{\pgfqpoint{0.000000in}{-0.104167in}}%
\pgfpathcurveto{\pgfqpoint{0.027625in}{-0.104167in}}{\pgfqpoint{0.054123in}{-0.093191in}}{\pgfqpoint{0.073657in}{-0.073657in}}%
\pgfpathcurveto{\pgfqpoint{0.093191in}{-0.054123in}}{\pgfqpoint{0.104167in}{-0.027625in}}{\pgfqpoint{0.104167in}{0.000000in}}%
\pgfpathcurveto{\pgfqpoint{0.104167in}{0.027625in}}{\pgfqpoint{0.093191in}{0.054123in}}{\pgfqpoint{0.073657in}{0.073657in}}%
\pgfpathcurveto{\pgfqpoint{0.054123in}{0.093191in}}{\pgfqpoint{0.027625in}{0.104167in}}{\pgfqpoint{0.000000in}{0.104167in}}%
\pgfpathcurveto{\pgfqpoint{-0.027625in}{0.104167in}}{\pgfqpoint{-0.054123in}{0.093191in}}{\pgfqpoint{-0.073657in}{0.073657in}}%
\pgfpathcurveto{\pgfqpoint{-0.093191in}{0.054123in}}{\pgfqpoint{-0.104167in}{0.027625in}}{\pgfqpoint{-0.104167in}{0.000000in}}%
\pgfpathcurveto{\pgfqpoint{-0.104167in}{-0.027625in}}{\pgfqpoint{-0.093191in}{-0.054123in}}{\pgfqpoint{-0.073657in}{-0.073657in}}%
\pgfpathcurveto{\pgfqpoint{-0.054123in}{-0.093191in}}{\pgfqpoint{-0.027625in}{-0.104167in}}{\pgfqpoint{0.000000in}{-0.104167in}}%
\pgfpathclose%
\pgfusepath{stroke,fill}%
}%
\begin{pgfscope}%
\pgfsys@transformshift{5.250000in}{1.350000in}%
\pgfsys@useobject{currentmarker}{}%
\end{pgfscope}%
\end{pgfscope}%
\begin{pgfscope}%
\pgfpathrectangle{\pgfqpoint{0.000000in}{0.000000in}}{\pgfqpoint{6.000000in}{3.600000in}}%
\pgfusepath{clip}%
\pgfsetbuttcap%
\pgfsetroundjoin%
\definecolor{currentfill}{rgb}{0.800000,0.800000,0.800000}%
\pgfsetfillcolor{currentfill}%
\pgfsetlinewidth{1.003750pt}%
\definecolor{currentstroke}{rgb}{0.000000,0.000000,0.000000}%
\pgfsetstrokecolor{currentstroke}%
\pgfsetdash{}{0pt}%
\pgfsys@defobject{currentmarker}{\pgfqpoint{-0.104167in}{-0.104167in}}{\pgfqpoint{0.104167in}{0.104167in}}{%
\pgfpathmoveto{\pgfqpoint{0.000000in}{-0.104167in}}%
\pgfpathcurveto{\pgfqpoint{0.027625in}{-0.104167in}}{\pgfqpoint{0.054123in}{-0.093191in}}{\pgfqpoint{0.073657in}{-0.073657in}}%
\pgfpathcurveto{\pgfqpoint{0.093191in}{-0.054123in}}{\pgfqpoint{0.104167in}{-0.027625in}}{\pgfqpoint{0.104167in}{0.000000in}}%
\pgfpathcurveto{\pgfqpoint{0.104167in}{0.027625in}}{\pgfqpoint{0.093191in}{0.054123in}}{\pgfqpoint{0.073657in}{0.073657in}}%
\pgfpathcurveto{\pgfqpoint{0.054123in}{0.093191in}}{\pgfqpoint{0.027625in}{0.104167in}}{\pgfqpoint{0.000000in}{0.104167in}}%
\pgfpathcurveto{\pgfqpoint{-0.027625in}{0.104167in}}{\pgfqpoint{-0.054123in}{0.093191in}}{\pgfqpoint{-0.073657in}{0.073657in}}%
\pgfpathcurveto{\pgfqpoint{-0.093191in}{0.054123in}}{\pgfqpoint{-0.104167in}{0.027625in}}{\pgfqpoint{-0.104167in}{0.000000in}}%
\pgfpathcurveto{\pgfqpoint{-0.104167in}{-0.027625in}}{\pgfqpoint{-0.093191in}{-0.054123in}}{\pgfqpoint{-0.073657in}{-0.073657in}}%
\pgfpathcurveto{\pgfqpoint{-0.054123in}{-0.093191in}}{\pgfqpoint{-0.027625in}{-0.104167in}}{\pgfqpoint{0.000000in}{-0.104167in}}%
\pgfpathclose%
\pgfusepath{stroke,fill}%
}%
\begin{pgfscope}%
\pgfsys@transformshift{2.250000in}{0.450000in}%
\pgfsys@useobject{currentmarker}{}%
\end{pgfscope}%
\end{pgfscope}%
\begin{pgfscope}%
\pgfpathrectangle{\pgfqpoint{0.000000in}{0.000000in}}{\pgfqpoint{6.000000in}{3.600000in}}%
\pgfusepath{clip}%
\pgfsetbuttcap%
\pgfsetroundjoin%
\definecolor{currentfill}{rgb}{0.800000,0.800000,0.800000}%
\pgfsetfillcolor{currentfill}%
\pgfsetlinewidth{1.003750pt}%
\definecolor{currentstroke}{rgb}{0.000000,0.000000,0.000000}%
\pgfsetstrokecolor{currentstroke}%
\pgfsetdash{}{0pt}%
\pgfsys@defobject{currentmarker}{\pgfqpoint{-0.104167in}{-0.104167in}}{\pgfqpoint{0.104167in}{0.104167in}}{%
\pgfpathmoveto{\pgfqpoint{0.000000in}{-0.104167in}}%
\pgfpathcurveto{\pgfqpoint{0.027625in}{-0.104167in}}{\pgfqpoint{0.054123in}{-0.093191in}}{\pgfqpoint{0.073657in}{-0.073657in}}%
\pgfpathcurveto{\pgfqpoint{0.093191in}{-0.054123in}}{\pgfqpoint{0.104167in}{-0.027625in}}{\pgfqpoint{0.104167in}{0.000000in}}%
\pgfpathcurveto{\pgfqpoint{0.104167in}{0.027625in}}{\pgfqpoint{0.093191in}{0.054123in}}{\pgfqpoint{0.073657in}{0.073657in}}%
\pgfpathcurveto{\pgfqpoint{0.054123in}{0.093191in}}{\pgfqpoint{0.027625in}{0.104167in}}{\pgfqpoint{0.000000in}{0.104167in}}%
\pgfpathcurveto{\pgfqpoint{-0.027625in}{0.104167in}}{\pgfqpoint{-0.054123in}{0.093191in}}{\pgfqpoint{-0.073657in}{0.073657in}}%
\pgfpathcurveto{\pgfqpoint{-0.093191in}{0.054123in}}{\pgfqpoint{-0.104167in}{0.027625in}}{\pgfqpoint{-0.104167in}{0.000000in}}%
\pgfpathcurveto{\pgfqpoint{-0.104167in}{-0.027625in}}{\pgfqpoint{-0.093191in}{-0.054123in}}{\pgfqpoint{-0.073657in}{-0.073657in}}%
\pgfpathcurveto{\pgfqpoint{-0.054123in}{-0.093191in}}{\pgfqpoint{-0.027625in}{-0.104167in}}{\pgfqpoint{0.000000in}{-0.104167in}}%
\pgfpathclose%
\pgfusepath{stroke,fill}%
}%
\begin{pgfscope}%
\pgfsys@transformshift{3.750000in}{0.450000in}%
\pgfsys@useobject{currentmarker}{}%
\end{pgfscope}%
\end{pgfscope}%
\end{pgfpicture}%
\makeatother%
\endgroup%